\documentclass[]{amsart}

\usepackage[Symbol]{upgreek}

\usepackage{graphicx}

\usepackage{amssymb}
\usepackage{epic}
\usepackage{mathrsfs}
\usepackage{accents}
\usepackage{array}
\usepackage{stmaryrd}
\usepackage{enumitem}
\usepackage{longtable}

\usepackage{etoolbox} % for \patchcmd
\usepackage[hidelinks]{hyperref}
\usepackage{tikz}
\usetikzlibrary{shapes,arrows.meta}
\usetikzlibrary{positioning}
\usetikzlibrary {shapes.misc}

\input{xy}
\xyoption{matrix}
\xyoption{arrow}
\xyoption{curve}

\numberwithin{section}{part}
\numberwithin{equation}{section}

\newcounter{Hequation}

\makeatletter
\g@addto@macro\equation{\stepcounter{Hequation}}
\makeatother

\makeatletter
\renewcommand{\tocsection}[3]
 { \indentlabel{\@ifnotempty{#2}{\parbox{2.5em}{\ignorespaces#1 #2.}\quad}}#3}

\makeindex

\newcommand{\bbold}{\mathbb}
\newcommand{\cal}{\mathcal}

\def\R { {\bbold R} }
\def\Q { {\bbold Q} }
\def\Z { {\bbold Z} }
\def\C { {\bbold C} }
\def\N { {\bbold N} }
\def\T { {\bbold T} }
\def\E {{\mathcal E}}

\def\cc {\operatorname{c}}
\def\g {\operatorname{g}}
\def \I{\operatorname{I}}

\def\trig{\operatorname{trig}}
\def\tl{\operatorname{tl}}

\def \order{\operatorname{order}}
\def \val{\operatorname{mul}}
\newcommand\dval{\operatorname{dmul}}

\def \exc {{\mathscr E}}
\def \ex{\operatorname{e}}
\def \wr {\operatorname{wr}}

\def \Frac {\operatorname{Frac}}

\def \Univ {{\operatorname{U}}}

\renewcommand\epsilon{\varepsilon}

\def \d{\operatorname{d}}

\def \ev{\operatorname{e}}
\def \bar {\overline}
\def \<{\langle}
\def \>{\rangle}

\def \tilde {\widetilde}

\def \mult{\operatorname{mult}}

\def \hat {\widehat}

\def \supp {\operatorname{supp}}

\def \((  {(\!(}
\def \)) {)\!)}

\def \Hom{\operatorname{Hom}}

\def \res{\operatorname{res}}

\def \k {{{\boldsymbol{k}}}}

\DeclareMathSymbol{\precequ}{\mathrel}{symbols}{"16}
\DeclareMathSymbol{\succequ}{\mathrel}{symbols}{"17}

\def \nasymp{\not\asymp}

\renewcommand{\Re}{\operatorname{Re}}
\renewcommand{\Im}{\operatorname{Im}}

\newcommand{\claim}[2][\!\!]{\medskip\noindent {\bf Claim #1:} {\it #2}\medskip}

\newtheorem{theorem}{Theorem}[section]
 
\newtheorem*{normtheorem}{Normalization Theorem} 
\newcounter{tmptheorem}
\newcounter{tmpequation}
\newtheorem{lemma}[theorem]{Lemma}
\newtheorem{prop}[theorem]{Proposition}
\newtheorem{cor}[theorem]{Corollary}

\theoremstyle{definition}

\newtheorem{definition}[theorem]{Definition}

\theoremstyle{remark}
\newtheorem*{example}{Example}

\newtheorem{exampleNumbered}[theorem]{Example}

\newtheorem*{remark}{Remark}

\newtheorem{remarkNumbered}[theorem]{Remark}

\newcommand{\abs}[1]{\lvert#1\rvert}
\newcommand{\dabs}[1]{\lVert#1\rVert}

\def \fM {{\mathfrak M}}
\def \fd {{\mathfrak d}}
\def \fm {{\mathfrak m}}
\def \fp {{\mathfrak p}}
\def \fn {{\mathfrak n}}
\def \fv {{\mathfrak v}}
\def \fw {{\mathfrak w}}

\def \Ric{\operatorname{Ri}}

\def \id{\operatorname{id}}

\def \GL{\operatorname{GL}}

\let\oldi\i
\let\oldj\j
\renewcommand\i{\relax\ifmmode{\boldsymbol{i}}\else\oldi\fi}
\renewcommand\j{\relax\ifmmode{\boldsymbol{j}}\else\oldj\fi}

\renewcommand\leq{\leqslant}
\renewcommand\geq{\geqslant}
\renewcommand\preceq{\preccurlyeq}
\renewcommand\succeq{\succcurlyeq}
\renewcommand\le{\leq}
\renewcommand\ge{\geq}

\DeclareMathAlphabet{\mathbf}{OML}{cmm}{b}{it}

\DeclareFontFamily{U}{fsy}{}
\DeclareFontShape{U}{fsy}{m}{n}{<->s*[.9]psyr}{}
\DeclareSymbolFont{der@m}{U}{fsy}{m}{n}
\DeclareMathSymbol{\der}{\mathord}{der@m}{182}

\DeclareSymbolFont{der@m}{U}{fsy}{m}{n}
\DeclareMathSymbol{\derdelta}{\mathord}{der@m}{100}

\newcommand\dotprec{\mathrel{\dot\prec}}

\newcommand\wt{\operatorname{wt}}

\newcommand\bsigma{\boldsymbol{\sigma}}
\newcommand\m{\mathfrak m}

\newcommand\dwv{\operatorname{dwm}}
\newcommand\dwt{\operatorname{dwt}}

\newcommand\nwt{\operatorname{nwt}}
\newcommand\ndeg{\operatorname{ndeg}}
\newcommand\nval{\operatorname{nmul}}

\DeclareSymbolFont{imag@m}{OT1}{cmr}{m}{ui}
\DeclareMathSymbol{\imag}{\mathord}{imag@m}{105}

\DeclareFontFamily{OMS}{smallo}{}
\DeclareFontShape{OMS}{smallo}{m}{n}{<->s*[.65]cmsy10}{}
\DeclareSymbolFont{smallo@m}{OMS}{smallo}{m}{n}
\DeclareMathSymbol{\smallo}{\mathord}{smallo@m}{79}

\DeclareFontFamily{OMS}{largerdot}{}
\DeclareFontShape{OMS}{largerdot}{m}{n}{<->s*[.8]cmsy10}{}
\DeclareSymbolFont{largerdot@m}{OMS}{largerdot}{m}{n}
\DeclareMathSymbol{\largerdot}{\mathord}{largerdot@m}{15}

\newcommand\Aut{\operatorname{Aut}}

\DeclareMathSymbol{\llambda}{\mathord}{der@m}{108}
\DeclareMathSymbol{\rrho}{\mathord}{der@m}{114}

\def \upg{\upgamma}
\def \Upg{\Upgamma}
\def \upl{\uplambda}
\def \Upl{\Uplambda}
\def \upo{\upomega}

\def \Upd{\Updelta}

\newcommand{\equationqed}[1]{\[\pushQED{\qed}#1 \qedhere\popQED\]\let\qed\relax}
\newcommand{\alignqed}[1]{\begin{align*}\pushQED{\qed} #1 \qedhere\popQED\end{align*}\let\qed\relax}

\makeatletter
\newcommand{\dminus}{\mathbin{\text{\@dminus}}}

\newcommand{\@dminus}{%
  \ooalign{\hidewidth\raise1ex\hbox{\bf.}\hidewidth\cr$\m@th-$\cr}%
}
\makeatother

\def\ddeg{\operatorname{ddeg}}
\def\dwm{\operatorname{dwm}}

\def \cf{\operatorname{cf}}
\def \O{\mathcal{O}}

\makeatletter
\renewcommand\part{\@startsection{part}{0}%
  \z@{\linespacing\@plus\linespacing}{.5\linespacing}%
  {\normalfont\bfseries\centering}}
\makeatother

\makeatletter
\renewcommand\theindex{\@restonecoltrue\if@twocolumn\@restonecolfalse\fi
  \columnseprule\z@ \columnsep 35\p@
  \twocolumn[\@xp\part\@xp*\@xp{\bf Index}\bigskip]%
  \let\item\@idxitem
  \parindent\z@  \parskip\z@\@plus.3\p@\relax
  \small}  

\renewenvironment{thebibliography}[1]{%
  \@xp\part\@xp*\@xp{\refname}%
  \normalfont\small\labelsep .5em\relax
  \renewcommand\theenumiv{\arabic{enumiv}}\let\p@enumiv\@empty
  \list{\@biblabel{\theenumiv}}{\settowidth\labelwidth{\@biblabel{#1}}%
    \leftmargin\labelwidth \advance\leftmargin\labelsep
    \usecounter{enumiv}}%
  \sloppy \clubpenalty\@M \widowpenalty\clubpenalty
  \sfcode`\.=\@m
}{%
  \def\@noitemerr{\@latex@warning{Empty `thebibliography' environment}}%
  \endlist
}
 
\def\astr{$\,({}^{*})$}
 
\makeatletter
\newcommand{\smallbullet}{} % for safety
\DeclareRobustCommand\smallbullet{%
  \mathord{\mathpalette\smallbullet@{0.6}}%
}
\newcommand{\smallbullet@}[2]{%
  \vcenter{\hbox{\scalebox{#2}{$\m@th#1\bullet$}}}%
}
\makeatother
  
\begin{document}

\title{Normalizing    Asymptotic Differential Equations}
\author[Aschenbrenner]{Matthias Aschenbrenner}
\address{Kurt G\"odel Research Center for Mathematical Logic\\
Universit\"at Wien\\
1090 Wien\\ Austria}
\email{matthias.aschenbrenner@univie.ac.at}

\author[van den Dries]{Lou van den Dries}
\address{Department of Mathematics\\
University of Illinois at Urbana-Cham\-paign\\
Urbana, IL 61801\\
U.S.A.}
\email{vddries@illinois.edu}

\author[van der Hoeven]{Joris van der Hoeven}
\address{CNRS, LIX (UMR 7161)\\ 
Campus de l'\'Ecole Polytechnique\\  91120 Palaiseau \\ France}
\email{vdhoeven@lix.polytechnique.fr}

\date{November, 2026}

\begin{abstract}  We define the universal exponential extension of an algebraically closed differential field and investigate its properties in the presence of a nice valuation and in connection with linear differential equations. Next we prove normalization theorems for algebraic differential equations over $H$-fields, as a tool in solving such equations in suitable extensions. The results in this monograph are essential in our work on Hardy fields in \cite{ADH5}.
\end{abstract}

\maketitle
\thispagestyle{empty}
\pagestyle{empty}

\bigskip

\tableofcontents

\newpage
\vspace*{\fill}

\begin{quote}
Fearful of extending this paper beyond its due limits, I have abstained from introducing any researches not essential to the development of that general method in analysis which it was proposed to exhibit. It may however be remarked that the principles on which the method is founded have a much wider range. They may be applied to the solution of functional equations, to the theory of expansions, and, to a certain extent, to the integration of non-linear differential equations. The position which I am most anxious to establish is, that any great advance in the higher analysis must be sought for by an increased attention to the laws of the combinations of symbols. The value of this principle can scarcely be overrated~[\dots]\\

---\, George Boole, \textit{On a general method in analysis \cite{Boole}.}
\end{quote}
 \vspace*{\fill}

\begin{quote}
[The] field that is normally classified as algebra really consists of two quite separate fields. Let us call them algebra one and algebra two, for lack of a  better language. Algebra one is the algebra whose bottom lines are algebraic geometry or algebraic number theory. Algebra one has by far a  better pedigree than algebra two, and has reached a  high degree of sophistication and breadth. [\,\dots] Algebra two has had a more accidented history. It can be traced back to George Boole, who was the initiator of three well-known branches of algebra two, namely: in the first place, Boolean algebra, in the second place, the operational calculus that views the derivative as an operator $D$, on which Boole wrote two books of great beauty, and finally, invariant theory [\,\dots] G.~H.~Hardy subtly condemned algebra two in England in the latter half of the nineteenth century, with the exclamation `Too much $f(D)$!' G.~H.~Hardy must be turning in his grave now.\\

---\, Gian-Carlo~Rota, \textit{Combinatorics,  representation  theory and invariant theory: the story of a m\'enage \`a trois \cite{Rota}.}
\end{quote}
 \vspace*{\fill}

\newpage 
\pagestyle{plain}
\setcounter{page}{1}

%\begingroup
\setcounter{tmptheorem}{\value{theorem}}% store current value of theorem counter
\setcounter{theorem}{0} %assign desired value to theorem counter
\renewcommand\thetheorem{0.\arabic{theorem}}

\setcounter{tmpequation}{\value{equation}}% store current value of equation counter
\setcounter{equation}{0} %assign desired value to equation counter
\renewcommand\theequation{0.\arabic{equation}}

\addtocontents{toc}{\protect\setcounter{tocdepth}{0}}

\part*{Introduction}

\medskip
  
\noindent
This monograph follows up on our book [ADH]. That book contains a
model-theoretic analysis of the ordered differential field $\mathbb{T}$ of
transseries (introduced by {\'E}calle~\cite{Ecalle} in connection with his work on
Dulac's Problem), including an explicit complete axiomatization and a
quantifier elimination.

In 2021 we settled one of the main open problems left in [ADH] (see
also~{\cite{ADH2}}), by proving that all maximal Hardy fields are
elementarily equivalent, as ordered differential fields, to $\mathbb{T}$.
Here, a~{\it Hardy field}\/ is a differential field of germs at $+ \infty$ of
differentiable one-variable real-valued functions defined on intervals $(a, +
\infty)$. Hardy fields were introduced by Bourbaki~\cite{Bou}, revived in the 1980s
by Boshernitzan and Rosenlicht~\cite{Boshernitzan81, Ros}, and became important in the study of
{o-minimal} expansions of the real field~\cite{Miller}. A~Hardy field is {\it maximal}\/ if it has no proper
Hardy field extension. By Zorn, every Hardy field extends to a maximal one.
(There are in fact very many maximal Hardy fields, by \cite[Section 7]{ADHfgh}.)

The full solution of this problem takes a lot of space and is in the unwieldy 
manuscript~\cite{ADH4}. In the present monograph we have extracted the main
differential-algebraic (that is, non-analytic) parts, because we intend to use
this material also for other purposes. It 
 has
four chapters, each including its own introduction, with the following
dependencies: 
$$
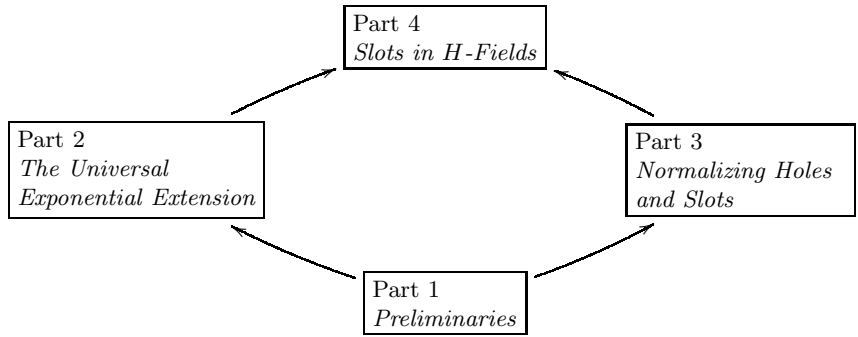
\xymatrix@R-1em{ & \fbox{\parbox{7em}{\small Chapter~\ref{part:dents in H-fields} \\ \it Slots in $H$-Fields}} &  \\ \fbox{\parbox{9em}{\small Chapter~\ref{part:universal exp ext}  \\ \it The Universal \newline Exponential Extension}} \ar@/^/[ru] & & \fbox{\parbox{8em}{\small Chapter~\ref{part:normalization} \\ \it Normalizing Holes \newline and Slots}} \ar@/_/[lu] \\ &  \ar@/^/[lu] \fbox{\parbox{5.5em}{\small Chapter~\ref{part:preliminaries}\\ \it Preliminaries}} \ar@/_/[ru] &}
$$
The application to Hardy fields
requires in addition a good dose of analysis and some model theory, and is
presented in~\cite{ADH5}.

To explain the overall aim of this monograph, 
consider an arbitrary algebraic differential equation with transseries coefficients, for example  
$$  \label{ode1} y^7 - \ex^{\ex^x} y'' y^3 + 7 \,\Gamma (x^x)\, y' y''' -
  \frac{\zeta (x)}{x + \log x + 1}\ =\ 0 \qquad (x>\R). $$
By the general theory from [ADH]  this equation happens to have a  
solution~$y \in \mathbb{T}$. (Indeed, this follows from the differential polynomial on the left-hand side of this equation having odd degree, see
 [ADH, pp.~17--18].) 
However, a concrete solution cannot be
derived in a transparent and direct way from the equation. It is even less
clear why there should also be a solution in a Hardy field.

Our main goal  is to reduce algebraic differential equations such as the one above   to
normal forms that can be solved more easily and explicitly in various
contexts. Some of this can be accomplished using   existing tools from [ADH], but for \cite{ADH5} we need
more precision.    After laying the groundwork in Chapters~\ref{part:preliminaries},~\ref{part:universal exp ext},~\ref{part:normalization}   we prove the main results of this monograph in Chapter~\ref{part:dents in H-fields}. Some ideas from Chapters~\ref{part:normalization} and~\ref{part:dents in H-fields}
also occur in~\cite{vdH:hfsol} in the more restricted context of grid-based transseries
solutions to algebraic differential equations over $\R$.  However, several shortcuts apply
in this case that are not available for~\cite{ADH5}.  

In the rest of this introduction we illustrate some key ideas informally by a few examples, before stating the main result in detail. We  make frequent reference to~[ADH];
the section \textit{Concepts and Results from~\textup{[ADH]}} at the end of this introduction should help in
consulting this source.  

\section*{Solving Quasilinear Equations}

\noindent
Our book~[ADH] contains a differential analogue of the
Newton diagram method to reduce arbitrary univariate algebraic  differential equations with asymptotic side conditions ({\it asymptotic equations}\/, for short) over certain valued differential fields such as $\T$ to
quasilinear equations, which can be considered as a first, crude kind of normal form for asymptotic equations.
An example of    a
 \textit{quasilinear}\/     equation  is
\begin{equation}\tag{E}
  \label{ode2} y' - y \ =\  x^{-1} yy''  -\ex^{-x} y^2 +x^{-1}, \qquad
  \quad y\, \prec\, 1.
\end{equation}
``Quasilinear'' means that the \textit{linear part}\/ of the equation (that is, the homogeneous part of degree~$1$, placed here on the left-hand side) dominates the rest in a certain sense. 
It is straightforward to compute a formal transseries solution of a quasilinear
equation such as \eqref{ode2} by transfinite recursion: we first determine its \textit{dominant term,}\/ in this
case~$f := - x^{- 1}$, and then substitute $f+y$ for $y$, with asymptotic side condition~${y \prec f}$, in the equation.  This yields a new quasilinear
equation
\begin{multline}\tag{E$_{+f}$}
  \label{refine-ode2} x^{-2}\,  y'' + y' +  ( - 1 +
  2x^{-4}  - 2x^{-1}\ex^{-x}   ) \, y\ =\  \\ x^{-1}\,y \,
  y''  -  \ex^{-x}y^2 - x^{-2} +
  2x^{-5} -  x^{-2} \ex^{-x}, \qquad y\, \prec\, x^{-1}.
\end{multline}
Such an asymptotic equation obtained by an additive change of variables together with the imposition of a possibly
stricter asymptotic constraint is called a \textit{refinement}\/ of~\eqref{ode2}. Applying this
method recursively, we obtain a particular solution
\[ y_{\ast}\ := \ - x^{-1} + x^{-2} - 2x^{-3} + 6x^{-4} -
    26x^{-5} + \cdots + x^{-2} \ex^{-x} - 2x^{-3} \ex^{-x} +
   \cdots  \]
to \eqref{ode2}.
This process is really transfinite: computing the first $\omega$
terms of~$y_{\ast}$ gives  
 %we substitute $f+y$ for $y$ in the originalcorresponding to \marginpar{...}  
 $$f\ :=\   - x^{-1} +
x^{-2} - 2x^{-3} + 6x^{-4} - 26x^{-5} + \cdots $$
and then we refine the original \eqref{ode2} by substituting $f+y$ for $y$, with the new side condition
$y\prec x^{\mathbb{N}}$. 
Likewise for other limit
ordinals.
In order to obtain the general solution to~\eqref{ode2}, we refine~\eqref{ode2} by
substituting $y_{\ast}+y$ for $y$, with asymptotic side condition $y\prec 1$, which~yields
\begin{multline}\tag{E$_*$}
  \label{ult-ode2}  ( x^{-2} - x^{-3} + 2x^{-4} -
  6x^{-5} +   \cdots  ) y'' + y' +  ( - 1 +
  2x^{-4} - 6x^{-5} + \cdots  ) y\ =\  \\ {x^{-1}\, {yy''}}  -
  \ex^{-x} {y^2}, \qquad y\, \prec\, 1.
\end{multline}
The first term of a non-zero solution to this asymptotic equation is necessarily the
first term of a zero of the   linear part of this equation.
Further computations show that such a dominant term is   of
the form $c\, \fm$ with~$c \in
\mathbb{R}^{\times}$ and~$\fm:=x^{- 5} \ex^{- x^3 / 3 - x^2 / 2}$. It follows that the general solution of~\eqref{ode2} has the form
$$%\begin{equation}\tag{G$_c$}
  \label{gen-sol} y_{c} \ =\  y_{\ast} + {c\, \fm} + \varepsilon_{c}\qquad \text{where $c \in \R$  and $\varepsilon_{c} \prec \fm$.}
$$%\end{equation}
However, this form of the general solution cannot be read off directly from the
original equation~\eqref{ode2}, and not even from~\eqref{refine-ode2}. This is
due to the fact that the shape of the linear part of the equation may drastically
change under refinement. Only after the computation of the first three terms
of $y_{\ast}$ do  we obtain a refinement that is sufficiently similar
to~\eqref{ult-ode2} to allow us to safely determine the form of the
general solution. Here we note that a particular solution $y_{\ast}$ is
generally not yet available during the resolution process. This prevents us from
working directly with~\eqref{ult-ode2}. We thus aim to always operate with
sufficiently good approximations to~$y_{\ast}$ and to~\eqref{ult-ode2} instead.
 
 Things become even worse if we introduce a  parameter $\varepsilon$ to perturb \eqref{ode2} to  
\begin{equation}\tag{E$_{\varepsilon}$}
  \label{param-ode2} y' - y \ =\  x^{-1}\,yy''  - \ex^{-x} y^2  +
  \varepsilon, \qquad y \prec 1,
\end{equation}
This
equation is quasilinear for every~$\varepsilon \prec 1$. For $\varepsilon = x^{- 1}$ it has
continuum many solutions in $\T$, as explained above. However, \eqref{param-ode2} has only one solution in $\T$
for~$\varepsilon = - x^{- 1}$. Indeed, in the latter case a
particular solution is given by
$$  y_{**}\ =\  x^{-1} - x^{-2}  + 2x^{-3}  -
  6x^{-4} + 22x^{-5} + \cdots + x^{-2} \ex^{-x} -
  2x^{-3} \ex^{-x} + \cdots,$$
 and  
 replacing $y$ by $y_{**} + y$ in \eqref{param-ode2}  leads to the
following counterpart of~\eqref{ult-ode2}:
\begin{multline*}  (-x^{-2} + x^{-3} - 2x^{-4} + 6x^{-5} -
   26x^{-6} + \cdots  ) y'' + y' +  ( - 1 - 2x^{-4} +
   6x^{-5} + \cdots  ) y\ =\ \\ x^{-1} yy''  - \ex^{-x}y^2,
   \qquad y \prec 1. 
\end{multline*}
One can show that $y=0$ is the only solution to this refinement of \eqref{param-ode2}, again by
examining its linear part. That  \textit{the general solution of a
quasilinear equation is not transparent from the equation}\/ and that \textit{the
resolution process is  not uniform under small perturbations}\/ are two major
reasons why it is desirable to normalize asymptotic equations beyond quasilinearity.

\section*{The Role of Factoring Linear Differential Operators}

\noindent
Here is an example of the kind of nicely normalized   equation that we are after:
\[ y' \ =\  \ex^{- x} + x^{-1} y^2, \qquad y\, \prec\, 1. \]
This equation can formally be solved by iterated integration:
\begin{align*}
  y &\, =\,  \textstyle\int \ex^{- x} + \int x^{-1} y^2\\
  &\, =\,  \textstyle\int \ex^{- x} + \int x^{-1}  \big( \int \ex^{- x} + \int
 x^{-1} y^2 \big){}^2\\[0.5em]
  &\, =\, \cdots\\[0.5em]
  &\, =\,  \textstyle\int \ex^{- x} + \int x^{-1}  \left( \int \ex^{- x} \right)\!{}^2 +
  2 \int x^{-1}  \left( \int \ex^{- x} \right)  \left( \int x^{-1}
  \left( \int \ex^{- x} \right)\!{}^2 \right) + \cdots .
\end{align*}
It turns out that this expansion converges in the formal transseries setting,
but also analytically to a germ that belongs to a Hardy field. In the formal
setting, this requires so-called  {\it distinguished integration}\/, where all integration constants are taken to be zero.  To obtain a Hardy field solution, we
systematically integrate from~$+ \infty$. For details we refer respectively to
\cite[Section~6.5]{JvdH} and~\cite{ADH5}.
 
Let us consider more generally a quasilinear equation  
\begin{equation}\tag{Q}
  L (y) \ = \ R (y), \qquad y \,\prec\, 1,  \label{qlin}
\end{equation}
where $L\in \T[\der]^{\neq}$ is a monic linear differential operator and $R\in \T\{Y\}$ is a  differential
polynomial. (For greater flexibility, $R$ is allowed to have a
non-zero linear part.) First of all, this requires the (distinguished) integration operator~$\der^{- 1}
:= \int$ to be replaced by a more general right inverse  $L^{- 1}$ to the $\R$-linear map $y\mapsto L(y)\colon\T\to\T$ defined by $L$. For
  $L = \der - f$ ($f\in\T$) of order~$1$, this is easy: identifying each $g\in\T$ with the $\R$-linear map $y\mapsto gy\colon\T\to\T$,
we may take
\[ L^{- 1}\ =\ (\der - f)^{- 1}\ :=\ \ex^{\int f} \circ\, {\der^{- 1}} \circ \ex^{{-}\int f}
   . \]
More generally, if $L$    splits over $\T$, that is, if
\[ L\  =\  (\der -f_1) \cdots (\der - f_r) \qquad \text{where $f_1,\dots,f_r\in\T$,} \]
then we may take
\[ L^{- 1} \ :=\  (\der - f_r)^{- 1} \circ \cdots \circ(\der -
   f_1)^{- 1} . \]
Now $L$ might not 
split over $\T$, but it does factor into a product of order~$1$ and
  order~$2$ operators over $\T$.  And if $\der^2+a\der + b\in
\mathbb{T} [\der]$ ($a,b\in\T$) is irreducible, it splits over the complexification~$\mathbb{T} [\imag]$ of~$\mathbb{T}$: there are $f,g\in\T$ with $g\ne 0$ and 
\begin{align*}
  \der^2 + a \der + b &\ = \  \big(\der - (f - g \imag +
  g^{\dag})\big)  \big(\der - (f + g \imag)\big), 
\nonumber \\  % \label{sec-order-fact}\\
  a &\ =\   - (2 f + g^{\dag}), \nonumber\\
  b &\ =\   f^2 + g^2 - f' + f g^{\dag} , \nonumber
\end{align*}
so that we can formally invert this second-order
operator as
\begin{align*}
  (\der^2 + a \der + b)^{- 1} &\,:=\,  (\ex^{\int (f+g\imag)} \circ\, \der^{- 1} 
  \ex^{-\int (f+g\imag)}) \circ  (\ex^{\int (f-g\imag)}  
    g \circ\,\der^{- 1}\circ \ex^{-\int (f-g\imag)}
  g^{- 1}) \nonumber\\
  &\ =\,  \ex^{\int (f+g\imag)}\circ\, \der^{- 1}\,\circ \ex^{-2\int g\imag}  g\circ  \ex^{-\int (f-g\imag)}
  g^{- 1}. % \label{second-order-inv}
\end{align*}
This discussion is to suggest    that factoring  linear differential
operators (not unlike in Boole's work \cite[pp.~129--132]{KY}) plays a key role in our normalization program, and that it will  involve
transseries in~$\mathbb{T} [\imag]$, and even
{\it oscillatory transseries}\/ in $\mathbb{T} [\imag] [\ex^{\mathbb{T}
\imag}]$ when inverting such operators. 

The
framework of asymptotic differential algebra in [ADH] was introduced in  
anticipation of this kind of developments, but the Hardy fields in \cite{ADH5} require a
very detailed and explicit treatment. Chapter~\ref{part:universal exp ext} of this monograph is dedicated to
this task in the abstract setting of $H$-fields as in~[ADH], while relying on
Chapter~\ref{part:preliminaries} for miscellaneous preliminary material. (An {\it $H$-field}\/ is an
ordered valued differential field subject to certain first-order laws. Ordered differential subfields of
$\mathbb{T}$ extending~$\R$ as well as Hardy field extensions~of~$\mathbb{R}$ are $H$-fields.)

In Chapter~\ref{part:universal exp ext} we introduce the {\it universal exponential extension}\/ of an
algebraically closed differential field and investigate its connection to
linear differential equations, especially in the presence of a compatible
valuation on the field. The universal exponential extension of $\mathbb{T}
[\imag]$ can be identified with $\mathbb{T} [\imag] [\ex^{\imag
\mathbb{T}}]$: see Section~\ref{sec:eigenvalues and splitting} as well as \cite[Sections~7.7 and~7.8]{JvdH} for this and the connection with factorization
of linear differential operators over $\mathbb{T} [\imag]$.
For a Liouville closed Hardy field $H \supseteq \mathbb{R}$ one can identify
the universal exponential extension of its algebraic closure $K = H [\imag]$
with~$K [\ex^{\imag H}]$, an integral domain of germs of
$\mathbb{C}$-valued functions. This is a key point in~\cite{ADH5} and is proved there.
(An $H$-field $H$ is said to be \textit{Liouville closed}\/ if $H$ is real closed and for
all $f, g \in H$, there exists $y \in H^{\times}$ with $y' + fy = g$.)  

\section*{Desiderata for a Normal Form}

\noindent
Returning to the quasilinear equation~\eqref{qlin},  
assume now that the
coefficients of $L$ and $R$ lie in an arbitrary $H$-field~$H$ (instead of $\T$).
Ideally,
when should we consider this equation to be in normal form? Three natural
requirements are:

\begin{itemize}
  \item[(R1)]  \textit{$L$ does not significantly change under
  refinement.} 
  
  \item[(R2)] \textit{$L$ splits   over $K := H  [\imag]$.}
  
  \item[(R3)] \textit{The equation \eqref{qlin} has a solution which is as unique as possible.}
\end{itemize}

\noindent
The first requirement amounts to $L$   \textit{strongly
dominating}\/    $R$.  What \textit{significantly\/} and 
\textit{strongly dominating\/} mean here is made precise in terms of the \textit{span} $\fv=\fv(L)$ of
$L$, which measures how far $L$ is from being ``regular singular''.
(For example, the~span of the operator on the left-hand side of~\eqref{refine-ode2} is equal
to the quotient of the coefficients of $y''$ and $y$, which is
asymptotic to~$x^{- 2}$.) It turns out that~(R1) holds if~$R$ is dominated by $\fv^m L$ for a large enough $m$, under some additional assumptions such as $\order R \leqslant \order L$ and $y \prec 1$ as the asymptotic side condition in~\eqref{qlin};
more general side conditions $y \prec \fm$ where $\fm\in H^>$ can be reduced to this case via a
multiplicative change of variables (replacing $y$ by~$y\fm$).

Concerning (R2), the condition on $H$ that {\it all}\/ monic linear operators in
$H [\der]^{\neq}$ split over~$K$ is very strong. Fortunately, for
our main application in~\cite{ADH5} and the particular operators $L$ needed there,
such factorizations come almost for free. To sketch this, consider an immediate (in the sense of valued fields)
differentially algebraic $H$-field extension $\hat{H}$ of $H$ and an $H$-field
embedding~$\iota$ of~$H$ into a Hardy field. Our main goal in \cite{ADH5}
is to extend~$\iota$ to an embedding   of~$\hat{H}$ into a Hardy
field. 
Suppose  $\hat{y} \in \hat{K} \setminus K$, where $\hat{K} := \hat{H}
[\imag]$, and we want to extend~$\iota$ to an embedding of~$H \langle\, \Re \hat{y}, \Im \hat{y}\, \rangle \subseteq
\hat{H}$  into
some Hardy field. Now the trick is to choose~$\hat H$ and~$\hat{y}$ such that the  minimal annihilator $P\in K\{Y\}$ of $\hat y$ over $K$ is of  \textit{minimal complexity}. This means in particular that~$P$
is of minimal order over $K$, say $r$, and so $K$ contains all zeros of differential polynomials over $K$ of order~$<r$,  as long as these zeros live in an   extension of $K$ of the form $E[\imag]$ for some immediate $H$-field extension $E$ of $H$.
Consequently, all $A\in K[\der]^{\neq}$ of order~$\leq r$ split  over $K$, since their associated Riccati polynomials are of order~$<r$.  

This minimality argument relies on working  
over $K$ instead of $H$. The interplay between $H$ and $K$ is subtle: we need
minimal elements to be taken in~$K$, but actual extensions to be done on the
level of $H$-fields. The length of this monograph is partly due to the fact
that some of the material applies to differential-valued fields, and
thus  to both~$H$ and~$K$, whereas other results need to be
developed separately for $H$ and~$K$, often with minor though crucial
differences.

In the present monograph we do not set ourselves
the task to actually solve any asymptotic   equations. Instead, we
wish to prepare them as much as possible in the purely algebraic and abstract
setting of $H$-fields. The resulting normalized asymptotic  
equations should then be easier to deal with in suitable contexts (Hardy
fields, various kinds of transseries, surreal numbers, etc.). 

On a technical level, this is implemented using the notion of a \textit{hole}\/ in $H$, that is,   an asymptotic  equation $P (y) = 0$, $y \prec \fm$ over $H$ that comes with
a~solution $\hat{y}$ in an immediate $H$-field extension $\hat{H}$ of
$H$, but outside $H$ itself; notation: $(P,\fm,\hat y)$. For our purposes, this asymptotic equation can  
be arranged to be quasilinear, and even non-singular in a certain sense. (This corresponds to the notion of \textit{deep}
holes.) Therefore, roughly speaking, a hole in $H$ is a witness that $H$ is not (yet) newtonian.   (In~[ADH] we defined $H$  to be {\it newtonian}\/ if every quasilinear equation~\eqref{qlin} has a solution in $H$; this is really the most fundamental first-order property of~$\T$, given by an axiom scheme.)
We then try to modify such a hole~$(P,\fm,\hat y)$ further to bring the associated asymptotic equation   into a normal form 
$P (f + y)
= 0$, $y \prec \fn$, for some choice of~$f \in H$ and~$\fn\preceq\fm$.
 This new  equation   should satisfy versions of~(R1),~(R2),~(R3) above, and hence be more amenable to solving
in the kind of extension we are really interested in, rather than in the  
abstractly given extension~$\hat{H}$ where a solution~$\hat{y} 
- f$ is already given. For example, if $H\supseteq\R$ is a Hardy field, we wish to solve this equation in an immediate Hardy field
extension of $H$.
 
Turning now to the last requirement (R3), we saw  
that in~$\T$, the equation~\eqref{param-ode2} has the unique solution~$y_{**}$ for $\varepsilon = - x^{- 1}$ and an infinite family~$(y_c)$ of
solutions if $\varepsilon = x^{- 1}$. There are cases in which
the existence of more than one solution is unavoidable. For instance, assume
that we start with an $H$-field $H =\R (x, \ex^x)$ with constant field $\R$ over
which the equation~\eqref{ode2} makes sense. Then $H$ contains no element
$f$ with $f - y_{\ast} \prec x^{-\N}$. (This follows from
the fact that the first $\omega$ terms of~$y_{\ast}$ form a~divergent power
series in $x^{-1}$, as can be checked using techniques from~\cite{Ecalle81,Ecalle, Hardy63}.) Now consider an immediate $H$-field extension $\hat{H}$ of $H$
that contains all solutions~$y_c$. Each $y_c$ realizes the same cut in the ordered set $H$   and also satisfies the same algebraic differential
equations over~$H$. 
%(Using techniques from differential
%algebra, one may verify that~\eqref{ode2} defines a prime differential ideal of $H\{Y\}$.)
 %this property holds automatically whenever~\eqref{qlin} is a minimal
%annihilator of our solution over $H$.) 
In model theoretic terms, this means that all
  $y_c$ realize exactly the same quantifier-free $1$-type over~$H$ (in the language of ordered valued differential fields). 
 Therefore no  quantifier-free first-order condition over~$H$ can distinguish between two distinct solutions and 
there is consequently no first-order   
normalization process over $H$ for which the normal form would 
have a particular~$y_c$ as its unique solution in $\hat H$.
This kind of indistinguishability does not necessarily survive under $H$-field
extensions. For instance, consider an $H$-field extension $H_1 := H\langle f,\ex^{- x^3 / 3 - x^2 / 2}\rangle$ of $H$ where $f
\prec 1$ satisfies $f' - f = x^{-1}ff''  +  x^{-1}$.   Then for $c_1\neq c_2$, the solutions $y_{c_1}$ and $y_{c_2}$  realize distinct  quantifier-free $1$-types over~$H_1$.

In view of the above discussion, a more precise formulation of (R3)
would be to require   all solutions of~\eqref{qlin} to realize the same quantifier-free $1$-type
over $H$. In this monograph, we introduce an even more stringent requirement.
%that will, in particular, imply this model-theoretic condition. 
The idea is to investigate
closely how distinguishable and indistinguishable solutions arise. The
linear part of the equation is again the key here, as we already saw when obtaining the general solution~$y_c$ from the linear part
of~\eqref{ult-ode2}. Under the assumption~(R1), the asymptotic behavior of
the non-zero solutions to the linear differential equation~$L (h) = 0$ does not change under refinement. Now
assume that $\hat{y} \in \hat{H} \setminus H$ satisfies~\eqref{qlin}. Roughly
speaking, (R3) holds if $\hat{y}$ and $\hat{y} + | h |$ realize the
same cut in $H$, for all $h\prec 1$ satisfying~$L (h) = 0$. Here $h$ typically belongs to the universal exponential extension of $H[\imag]$.

\section*{The Main Theorem}

\noindent
Before we state our main result  
we  introduce a few more
concepts and motivate some hypotheses of this theorem. Again, we only do this on an informal level;   the precise definitions are given at the appropriate places in the monograph.

First, in light of (R2), our $H$-field $H$ had better be sufficiently rich from the
outset, say  Liouville closed and~$\upo$-free. Here, the property of
{\it $\upo$-freeness}\/ is a certain first-order condition about linear
differential equations of order $2$, satisfied by $\T$, which plays a crucial role in~[ADH]. With an eye towards  factorizations over the complexification $K = H [\imag]$ of $H$, we  
also assume that $K$ is \textit{$1$-linearly newtonian}\/: each quasilinear
equation~$y' + fy = g$ ($f, g \in K$) has a
solution $y\preceq 1$ in $K$.  
Minor assumptions are that the constant field of $H$ is archimedean
and  the derivation of~$H$ is \textit{small} in the sense that $h \prec 1
\Rightarrow h' \prec 1$ for all $h \in H$. Both conditions hold for $H=\T$  and any Hardy field.

Since our aim is  to solve  
quasilinear equations that do not already have a~solution in $H$, we will  
further assume that   $H$ is \textit{not}\/ newtonian. This hypothesis yields a quasilinear hole~$(Q, 1, \hat b)$ 
in $H$, that is,  a quasilinear equation~$Q (y) = 0$, $y
\prec 1$  and a solution $\hat b \notin H$    thereof in some
immediate extension of $H$. We  call such a hole   \textit{minimal}
if $Q$ has minimal complexity (among all holes in $H$), and  we   say that it is \textit{special} if $\hat b$
is a pseudolimit of a pseudocauchy sequence in $H$ exhibiting a power-type rate of
pseudoconvergence. 

One key feature of the general theory of $H$-fields from [ADH] is that
many important asymptotic properties are \textit{eventual}. This means that
they hold upon replacing the derivation $\der$ of $H$ by a
small  derivation $\phi^{- 1} \der$ with suitable $\phi\in H$, ${\phi>0}$. 
This corresponds to replacing the underlying ``time'' variable~$x$ by faster times,
like $\ex^x$ or $\ex^{\ex^x}$. We say that $\phi\in H$ with $\phi>0$  is \textit{active} in $H$ if 
the derivation~$\phi^{- 1} \der$ remains small. For such $\phi$ the
ordered field $H$ equipped with the derivation $\phi^{-1}\der$ is again an $H$-field, denoted by~$H^{\phi}$, and
rewriting a differential polynomial $P$ over $H$ in terms of the new derivation $\phi^{-1}\der$, we obtain an equivalent differential polynomial~$P^\phi$ over~$H^\phi$. Thus
any hole $(P, \fm, \hat y)$ in~$H$
naturally gives rise to a hole~$(P^{\phi}, \fm, \hat y)$ in~$H^{\phi}$.
%(The process of passing from $H$ and~$(P, \fm, \hat y)$ to $H^\phi$ and~$(P^{\phi}, \fm, \hat y)$, for some active $\phi>0$ in $H$, is called {\it compositional conjugation.}\/)

The precise formalization of the requirements (R1), (R2), (R3) from above is done stepwise, via the introduction of
progressively stronger normalization properties of holes. We actually work
also with~\textit{slots}: a bit more general than holes and useful
intermediate stages in the normalization of holes. Likewise, instead of
minimal holes we often deal with somewhat more general \textit{$Z$-minimal}\/ holes
and~slots.

At the end of the day, the requirements (R1), (R2) and
(R3) are formalized through the concepts of strongly
repulsive-normal and ultimate holes. \textit{Normality} corresponds to the
requirement~(R1). The notion \textit{repulsive} takes care of
(R2) and part of (R3); the terminology is motivated by
analytic considerations in~\cite{ADH5}. The qualifier \textit{strongly}\/ and the property
\textit{ultimate}\/ indicate further contributions to (R3).

Let us now state our main result with all the necessary fine print:

\begin{normtheorem} \label{cor4543}
Let $H$ be an
$\upo$-free Liouville closed $H$-field with small derivation, archimedean  ordered constant field $C$, and $1$-linearly newtonian algebraic closure $H [\imag]$. Suppose $H$ is not newtonian. Then
for some $Z$-minimal special hole~$(Q, 1, \hat b)$ in~$H$ with $\order Q\geq 1$ and some active~$\phi > 0$ in~$H$ with $\phi \preceq 1$, the hole
$(Q^{\phi}, 1, \hat b)$ in $H^{\phi}$ is deep, strongly repulsive-normal,
and ultimate.
\end{normtheorem}

\noindent
The assumption on $H$ that $C$  is archimedean  is not  first-order  in the logical sense; it can perhaps be dropped. (Corollary~\ref{corfirstnewtchar} characterizes newtonianity in a less sharp way, but without this hypothesis on $C$.)
As already mentioned, this Normalization Theorem is an essential tool in~\cite{ADH5}. The example~\eqref{param-ode2} also illustrates that stronger
normal forms are mandatory for the uniform resolution of asymptotic
  equations depending on parameters. We expect this to play an important
role for a deeper understanding of definable functions and better effective
versions of our quantifier elimination result from [ADH].

\section*{Advice} 

\noindent
Readers may skip Chapter~\ref{part:preliminaries}, referring back to it when necessary. Section~\ref{sec:uplupo-freeness} discusses the important concepts of $\upl$-freeness and $\upo$-freeness. The results in this section are not needed later in the monograph, but are used in
\cite{ADH7}, which in turn is needed in~\cite{ADH5}. Sections~\ref{sec:splitting} and~\ref{sec:eigenvalues and splitting} explain the role played by the universal exponential extension in factoring linear differential operators, and hence motivate the definition of ``ultimate hole''. Otherwise these two sections are not used to obtain our normalization results, but they are needed in~\cite{ADH5}.  A few subsections and individual items are marked by an asterisk~{\astr} to indicate that they are not used for our normalizations or in connection with~\cite{ADH5} or other articles by the authors listed in the bibliography.
For a summary of   various shades of normality introduced and a list of other  important normalization results proved in this monograph see {\it Notions of Normality}\/
on p.~\pageref{section:flowchart}.

\section*{Concepts and Results from [ADH]} 

\noindent
This section includes notation and terminology used throughout this monograph. 
Thus $m$,~$n$ always range over the set~$\N=\{0,1,2,\dots\}$ of natural numbers.
Given an additively written abelian group $A$ we set $A^{\ne}:=A\setminus\{0\}$.
 Rings (usually, but not always, commutative) are associative with  identity~$1$. For a ring~$R$ we let~$R^\times$ be the multiplicative group of units of $R$ (consisting of the $a\in R$ such that~$ab=ba=1$ for some $b\in R$). 
 Let $S$ be a totally ordered set and $A\subseteq S$.  We set
 $$A^\downarrow:=\{s\in S:\text{$s\leq a$ for some $a\in A$}\}, \quad  A^\uparrow:=\{s\in S:\text{$s\geq a$ for some $a\in A$}\}$$
 where  $\leqslant$ is the ordering of $S$.
 We say that $A$ is {\it downward closed}\/ in~$S$ if $A=A^\downarrow$ and {\it upward closed}\/ in~$S$ if~$A=A^\uparrow$. For $a\in S$ we also let $S^{>a}:=\{s\in S:s>a\}$. Ordered abelian groups and ordered fields are totally ordered, by convention.
 Let~$\Gamma$ be an ordered abelian group, written additively. Then $\Gamma^>:=\Gamma^{>0}$, and likewise  with~$\geq$,~$<$, or $\leq$ in place of $>$. The ordered divisible hull of $\Gamma$ is denoted by $\Q\Gamma$.

\subsection*{Differential rings and fields}
Let $R$ be 
a {\em differential ring\/}, that is, a commutative ring $R$ containing
(an isomorphic copy of) $\Q$ as a subring and equipped with a derivation~$\der\colon R \to R$. When its derivation $\der$ is clear from the context, then for~$a\in R$ we denote~$\der(a),\der^2(a),\dots,\der^n(a),\dots$ by $a', a'',\dots, a^{(n)},\dots$. If $a\in R^\times$, then~${a^\dagger:=a'/a}$ denotes the {\em logarithmic derivative\/}\index{logarithmic derivative}\index{derivation!logarithmic}\label{p:adagger} of $a$, so $(ab)^\dagger=a^\dagger + b^\dagger$ for all~$a,b\in R^\times$.  
We have a subring $C_R:=\ker\der$  of $R$, called the ring of constants of $R$, with~$\Q\subseteq C_R$. 
 A {\em differential field\/}\index{differential field} is a differential ring 
$K$ whose underlying ring is a field. In this case~$C_K$ is a subfield of $K$, and if $K$ is understood
from the context we often write~$C$ instead of~$C_K$. Note that a differential field has characteristic $0$.

Often we are given a differential field $H$ in which $-1$ is not a square, and then~$H[\imag]$ is a differential field extension with
$\imag^2=-1$:  the derivation $\der$ on $H$ uniquely extends to a derivation on $H[\imag]$, and this extension has $\imag$ in its constant field. For~$z=a+b\imag\in H[\imag]$ ($a,b \in H$) we set $\Re z:=a$, $\Im z:=b$, and $\bar{z}:=a-b\imag$. Then
$z\mapsto \bar{z}$ is an automorphism of the differential field $H[\imag]$.
If there is also given a differential field extension $F$ of $H$ in which $-1$ is not a square, we always tacitly arrange $\imag$ to be such that $H[\imag]$ is a differential subfield of the differential field extension $F[\imag]$ of $F$.

\subsection*{Differential polynomials}
Let $R$ be a differential ring. 
We have the differential ring $R\{Y\}=R[Y, Y', Y'',\dots]$ of differential polynomials in a differential indeterminate~$Y$ over $R$. \label{p:RY}
Let $P=P(Y)\in R\{Y\}$. The {\em order\/} of $P$, denoted by~$\order(P)$, is the least $r\in\N$ such that~$P\in R[Y,Y',\dots, Y^{(r)}]$.\label{p:order P}\index{order!differential polynomial}\index{differential polynomial!order} Let $\order(P)\leqslant r$. Then $$P=\sum_{\i}P_{\i}Y^{\i}$$  with~$\i$ ranging over tuples~$(i_0,\dots,i_r)\in \N^{1+r}$, $Y^{\i}:= Y^{i_0}(Y')^{i_1}\cdots (Y^{(r)})^{i_r}$,     coefficients~$P_{\i}$   in $R$, and $P_{\i}\ne 0$ for only finitely many $\i$. 
For such $\i$ we set 
$$|\i|\ :=\ i_0+i_1+ \cdots + i_r, \qquad \|\i\|\ :=\ i_1+2i_2 + \dots + ri_r.$$
The {\it multiplicity of~$P$}\/ (at~$0$) \label{p:val P}\index{multiplicity!differential polynomial}\index{differential polynomial!multiplicity}
is 
$$\val P 
\ :=\ \min\big\{|\i|:\, P_{\i}\ne 0\big\}\in \N \text{ if $P\ne 0$,} \qquad \val P\ :=\ +\infty \text{ if $P=0$,}$$
the {\em degree\/} of $P$ is  \label{p:deg P}\index{degree!differential polynomial}\index{differential polynomial!degree}
$$\deg P\ :=\ \max\big\{|\i|:\, P_{\i}\ne 0\big\}\in \N \text{ if $P\ne 0$,} \qquad \deg P\ :=\ -\infty \text{ if $P=0$,}$$ 
and the {\em weight} of $P$ is \label{p:wt P}\index{weight!differential polynomial}\index{differential polynomial!weight}
$$\wt P\ :=\ \max\big\{\|\i\|:\, P_{\i}\ne 0\big\}\in \N \text{ if $P\ne 0$,} \qquad \wt P\ :=\ -\infty \text{ if $P=0$.}$$
For $d\in \N$ we set $P_d:=\sum_{|\i|=d} P_{\i}Y^{\i}$ (the {\it homogeneous part}\/ of degree $d$ of $P$), so~$P=\sum_d P_d$, and if $P\ne 0$, then\index{differential polynomial!homogeneous part}\label{p:Pd} 
$$\val P\ =\ \min\{d:\, P_d\ne 0\},\qquad \deg P\ =\ \max\{d:\, P_d\ne 0\}.$$ 
For~${a\in R}$ we let
$$P_{+a}\ :=\ P(a+Y)\quad\text{and}\quad P_{\times a}\ :=\ P(aY)$$ be
the {\it additive conjugate}\/ and the {\it multiplicative conjugate}\/ of $P$ by~$a$, respectively. \label{p:P+a} \label{p:Pxa}
For $\phi\in R^\times$ we  let~$R^{\phi}$ be the {\it compositional conjugate of $R$ by~$\phi$}\/: the differential ring
with the same underlying ring as~$R$ but with derivation~$\phi^{-1}\der$ (usually denoted by $\derdelta$) instead of $\der$.  \label{p:Pphi}
We have an $R$-algebra isomorphism~$P\mapsto P^\phi\colon R\{Y\}\to R^\phi\{Y\}$ such that $P^\phi(y)=P(y)$ for all $y\in R$;
see~[ADH, 5.7]. 

A {\it differentially algebraic\/} (for short: {\it $\d$-algebraic}) extension of a differential field~$K$ is a differential field extension $L$ of $K$ such that for all $y\in L$ we have $P(y)=0$ for some
differential polynomial $P\in K\{Y\}^{\ne}$. See [ADH, 4.1] for more on this.\index{differential field!differentially algebraic extension}\index{differential field!d-algebraic extension@$\d$-algebraic extension}\index{extension!differentially algebraic}\index{extension!d-algebraic@$\d$-algebraic} 

\subsection*{Complexity and the separant} 
%We recall some definitions and observations from~[ADH, 4.3].
Let $K$ be a differential field and $P\in K\{Y\}\setminus K$, and set 
$r=\order P$, $s=\deg_{Y^{(r)}} P$, and $t=\deg P$. Then the {\it complexity}\/ of~$P$ is the triple
$\cc(P)=(r,s,t)\in\N^3$;\index{complexity!differential polynomial}\label{p:complexity} \index{differential polynomial!complexity}
we order $\N^3$ lexicographically. Let $a\in K$. Then~$\cc(P_{+a})=\cc(P)$, and
$\cc(P_{\times a})=\cc(P)$ if $a\neq 0$.
The differential polynomial~$S_P:=\frac{\partial P}{\partial Y^{(r)}}$ is called the {\it separant}\/ of $P$; \index{separant} \index{differential polynomial!separant}\label{p:separant} thus
 $\cc(S_P)<\cc(P)$ (giving complexity~$(0,0,0)$ to elements of $K$), and $S_{aP}=aS_P$ if $a\neq 0$.  Moreover,
\begin{equation}\label{eq:separant fms} 
S_{P_{+a}}\  =\  (S_P)_{+a},\quad S_{P_{\times a}}=a\cdot(S_P)_{\times a},\quad 
S_{P^\phi}\ =\  \phi^r (S_P)^\phi\text{ for $\phi\in K^\times$.}
\end{equation}
(For $S_{P_{+a}}$ and $S_{P_{\times a}}$ this is from [ADH, p.~216]; 
for $S_{P^\phi}$, express $P$ as a polynomial in $Y^{(r)}$ and use $(Y^{(r)})^\phi= \phi^r Y^{(r)}+\text{lower order terms}$.)

\subsection*{Linear differential operators} 
Let $R$ be a differential ring. We associate to $R$ the ring $R[\der]$ of linear differential operators over $R$; see [ADH, 5.1]. \label{p:Rder}
This is the ring extension of $R$ generated over $R$
by an element $\der$: we use here the same symbol that denotes the derivation of $R$,  impose
$\der^m\ne \der^n$ for all $m\ne n$, require $R[\der]$ to be free as a left $R$-module with basis
$1=\der^0, \der=\der^1, \der^2, \der^3,\dots$, and impose
$\der a = a \der + a'$ (in $R[\der]$) for  $a\in R$.
Each $A\in R[\der]$ has accordingly the form 
\begin{equation}\label{eq:lin op}
A\ =\ a_0+a_1\der+\cdots+a_n\der^n \qquad
(a_0,\dots,a_n\in R),
\end{equation}
and for such $A$ and $y\in R$ we put 
$$A(y)\ :=\ a_0y+a_1y'+\cdots+a_ny^{(n)}\in R.$$
Then $(AB)(y)=A\big(B(y)\big)$ for all $A,B\in R[\der]$ and $y\in R$.
The kernel of $A\in R[\der]$ is the   $C_R$-submodule
$$\ker A \ =\  \big\{ y\in R:\, A(y)=0 \big\}$$
of $R$. If we want to stress the dependence on $R$, we also write $\ker_R A$ for $\ker A$. \label{p:kerA}
For~$A\in R[\der]^{\neq}$ there are unique elements $a_0,\dots,a_n$  of $R$ with $a_n\neq 0$ such that~\eqref{eq:lin op} holds. Then $\order(A):=n$ is the {\it order}\/ of $A$, and we say that $A$ is {\it monic}\/ if~$a_n=1$. \index{order!linear differential operator}\index{linear differential operator!order}\index{linear differential operator!monic}\index{linear differential operator!kernel}
Let  $u\in R^\times$. For $A\in R[\der]$ we set $A_{\ltimes u}:=u^{-1}Au\in R[\der]$, the {\it twist}\/ of~$A$ by $u$. If~$A$ is monic, then so
is $A_{\ltimes u}$, and $A\mapsto A_{\ltimes u}$ is an automorphism of the ring $R[\der]$ 
which is the identity on $R$ [ADH, p.~243]; its inverse is $B\mapsto B_{\ltimes u^{-1}}$.\index{linear differential operator!twist}\label{p:twist}  
Let $\phi\in R^\times$. Then we have the ring $R^\phi[\derdelta]$ of linear differential operators over the differential ring~$R^\phi$
(with derivation $\derdelta=\phi^{-1}\der$), and we have
a ring isomorphism~$A\mapsto A^\phi\colon R[\der]\to R^\phi[\derdelta]$; 
it is the identity on $R$, with $\der^\phi=\phi\derdelta$.

\medskip
\noindent
The {\it linear part}\/ of  $P\in R\{Y\}$  is the linear differential operator
$$L_P\ :=\ \sum_n \frac{\partial P}{\partial Y^{(n)}}(0)\der^n\in R[\der], \qquad \text{so }\ L_{P_{+a}}\ =\ \sum_n \frac{\partial P}{\partial Y^{(n)}}(a)\der^n\text{ for $a\in R$.}$$ 
We have $L_P(y)=P_1(y)$
for all $y\in R$.\index{linear differential operator!monic}\index{differential polynomial!linear part}\index{linear part!differential polynomial}\label{p:lin part}

\medskip
\noindent
Suppose now $K$ is a differential field. Then   $A\in K[\der]$ is said to {\it split}\/ over~$K$ if~$A=c(\der-f_1)\cdots(\der-f_r)$
for some $c\in K^\times$, $f_1,\dots,f_r\in K$; cf.~[ADH, 5.1]. If $A$ splits over $K$, then so does $aAb$ for $a,b\in K^\times$, and $A^\phi$  splits over $K^\phi$ for $\phi\in K^\times$.\index{linear differential operator!splits}  
In~[ADH, 5.2] we defined the functions 
\begin{equation}\label{eq:def omega}
\omega\colon K\to K, \qquad \omega(z)\ :=\ -(2z'+z^2)
\end{equation}
and 
\begin{equation}\label{eq:def sigma}
\sigma\colon K^\times\to K,\qquad \sigma(y)\ :=\ \omega(z)+y^2\quad\text{where $z:=-y^\dagger$.}
\end{equation}
Then for $A=4\der^2+f$ ($f\in K$) we have 
\begin{equation}\label{eq:A splits over K}
\text{$A$ splits over $K$}  \quad\Longleftrightarrow\quad f\in \omega(K)
\end{equation}
and if $-1$ is not a square in $K$, then 
\begin{equation}\label{eq:A splits over K[i]}
\text{$A$ splits over $K[\imag]$}  \quad\Longleftrightarrow\quad f\in \omega(K)\cup\sigma(K^\times).
\end{equation}
We say that $K$ is   {\it linearly closed}\/ if every $A\in K[\der]^{\neq}$ splits over $K$.
\index{closed!linearly}
\index{differential field!linearly closed}
By [ADH, 5.8.9] this holds if $K$ is
{\it weakly differentially closed}\/:   each $P\in K\{Y\}\setminus K$ has a zero in $K$ 
(hence  $K$ is also  {\it linearly surjective}, that is,  $A(K)=K$ for all $A\in K[\der]^{\neq}$). \index{linearly!closed}\index{differential field!linearly  closed}\index{closed!linearly}\index{differential field!weakly differentially closed}\index{closed!weakly differentially}\index{linearly!surjective}\index{differential field!linearly surjective}\index{surjective!linearly}

\subsection*{Valued fields}
For a field $K$ we have $K^\times=K^{\ne}$, and
a (Krull) valuation on $K$ is a surjective map 
$v\colon K^\times \to \Gamma$ onto an ordered abelian 
group $\Gamma$ (additively written) satisfying the usual laws, and extended to
$v\colon K \to \Gamma_{\infty}:=\Gamma\cup\{\infty\}$ by $v(0)=\infty$,
where the ordering on $\Gamma$ is extended to a total ordering
on $\Gamma_{\infty}$ by $\gamma<\infty$ for all~$\gamma\in \Gamma$. 
A {\em valued field\/} $K$ is a field (also denoted by $K$) together with a valuation ring~$\mathcal O$ of that field,
and the corresponding valuation $v\colon K^\times \to \Gamma$  on the underlying field is
such that $\mathcal O=\{a\in K:va\geq 0\}$ as explained in [ADH, 3.1].

\medskip
\noindent
Let $K$ be a valued field
with valuation ring $\mathcal O_K$ and valuation $v\colon K^\times \to \Gamma_K$. Then~$\mathcal O_K$ is a local ring  with maximal ideal $\smallo_K=\{a\in K:va>0\}$  and residue field $\res(K)=\mathcal{O}_K/\smallo_K$. 
If $\res(K)$ has characteristic zero, then $K$ is said to be of equicharacteristic zero.
When, as here,  we use the capital $K$ for the valued field under consideration, then
we denote $\Gamma_K$, $\mathcal O_K$, $\smallo_K$, by $\Gamma$, $\mathcal O$, $\smallo$, respectively.\label{p:valring}
A very handy alternative notation system in connection with the valuation is as follows.  With $a,b$ ranging over $K$, set 
\begin{align*} a\asymp b &\ :\Leftrightarrow\ va =vb, & a\preceq b&\ :\Leftrightarrow\ va\ge vb, & a\prec b &\ :\Leftrightarrow\  va>vb,\\
a\succeq b &\ :\Leftrightarrow\ b \preceq a, &
a\succ b &\ :\Leftrightarrow\ b\prec a, & a\sim b &\ :\Leftrightarrow\ a-b\prec a.
\end{align*}
It is easy to check that if $a\sim b$, then $a, b\ne 0$ and $a\asymp b$, and that
$\sim$ is an equivalence relation on $K^\times$.\label{p:asymprels} We set $K^{\succ 1}:=\{a\in K:\, a\succ 1\}$, in analogy 
with notation like~$S^{>a}$ for an ordered set $S$ and $a\in S$. Likewise with $K^{\prec 1}$, etc., so $K^{\prec 1}=\smallo$.  

% let $a^{\sim}$ be the equivalence class of an element $a\in E^\times$ with respect to~$\sim$. 
Given a valued field extension $L$ of $K$, we identify in the usual way~$\res(K)$ with a subfield of $\res(L)$, and $\Gamma$ with an ordered subgroup of~$\Gamma_L$. Such a valued field extension is called {\it immediate} if $\res(K)=\res(L)$ and $\Gamma=\Gamma_L$.\index{valuation!immediate extension}\index{extension!immediate} We use {\em pc-sequence\/} to abbreviate
{\em pseudocauchy sequence}, and~${a_\rho\leadsto a}$ indicates that~$(a_\rho)$ is a  pc-sequence pseudoconverging to~$a$;
here the $a_{\rho}$ and $a$ lie in a valued field understood from the context, see [ADH, 2.2,~3.2].\label{p:pc}

\medskip
\noindent
Next we summarize the complementary processes of {\it coarsening}\/ and {\it specialization}\/ of a valued field, which  play an important role in Chapters~\ref{part:normalization} and~\ref{part:dents in H-fields}. (For more details see~[ADH, 3.4].)  Let $K$ be a valued field and $\Delta$ a convex subgroup of
its value group $\Gamma$. Equip $\dot\Gamma:=\Gamma/\Delta$ with the unique ordering making it an ordered abelian group such that the residue morphism $\Gamma\to\dot\Gamma$ is increasing. Then the map~${\dot v=v_\Delta\colon K^\times\to\dot\Gamma}$ given by $\dot v f:=vf+\Delta$ is
a valuation on $K$, called the {\it coarsening}\/ of $v$ by~$\Delta$, or just the {\it $\Delta$-coarsening}\/ of $v$.
We denote the asymptotic relations associated with $\dot v$ by a subscript $\Delta$, so   $\preceq_\Delta$, $\prec_\Delta$, etc., or just by $\dot\preceq$, $\dot\prec$, etc., if $\Delta$ is clear from the context.\label{p:coarsening}
The valuation ring of~$\dot v$ is~${\dot{\mathcal O}:=\big\{a\in K:\, va\in\Delta^\uparrow \big\}}$
with maximal ideal~${\dot\smallo := \big\{a\in K:\, va > \Delta \big\}}$.
The valued field $(K,\dot{\mathcal O})$ is called the {\it coarsening} of the valued field $K$ by $\Delta$, or simply the {\it $\Delta$-coarsening}\/ of $K$.
Let~${\dot K:=\dot{\mathcal O}/\dot\smallo}$ be the residue field of $(K,\dot{\mathcal O})$, with residue morphism~${a\mapsto a+\dot\smallo\colon\dot{\mathcal O}\to \dot K}$.
Then for~${a\in\dot{\mathcal O}\setminus\dot\smallo}$, the value $va$ only depends on~$\dot a$, and we   obtain a valuation~$v\colon\dot K^\times\to\Delta$ on $\dot K$ with~$v\dot a:=va$ for $a\in\dot{\mathcal O}\setminus\dot\smallo$. The valuation ring of this valuation on $\dot K$ is~$\mathcal O_{\dot K}:=\{\dot a:a\in\mathcal O\}$. The valued field
$(\dot K,\mathcal O_{\dot K})$,   called the {\it $\Delta$-specialization}\/ of $K$, is also denoted by $\dot K$. 
The composed map
 $ \mathcal{O} \to \mathcal{O}_{\dot K} \to \res(\dot K)$ has kernel $\smallo$, and thus induces a field isomorphism $\res(K)\to \res(\dot K)$;  we use it to identify $\res(K)$ with $\res(\dot K)$. 
\index{valuation!coarsening}\index{valuation!specialization}

\subsection*{Valued differential fields}
As in~[ADH],  a {\em valued differential field\/} is a valued field of equicharacteristic zero together with a derivation, generally denoted by $\der$, on the underlying field.\index{valued differential field}\index{differential field!valued} (Unlike~\cite{VDF} we do not assume in this definition that $\der$ is continuous with respect to the valuation topology.) Let $K$ be a valued differential field and $K\<Y\>$ the fraction field of $K\{Y\}$. 
We extend the valuation $v\colon K \to \Gamma_{\infty}$ and the corresponding relations
$\prec$, $\preceq$, etc., first to $K\{Y\}$ by 
%to the so-called {\em gaussian valuation}\/~$v\colon K\<Y\>\to \Gamma_{\infty}$ on the fraction field $K\<Y\>$ of $K\{Y\}$ by requiring that,  in the notation used above,
$$v(P)\ :=\  \min\big\{v(P_{\i}):\, \i\in \N^{1+r}\big\}\quad\text{ for $P\in K\{Y\}$ of order $\le r$,}$$
and then (uniquely) to a valuation $v\colon K\<Y\>\to  \Gamma_{\infty}$ (the {\em gaussian extension}). \index{valuation!gaussian extension}\index{gaussian extension}\index{extension!gaussian}
We also define the {\em dominant degree\/} $\ddeg P$ for $P\in K\{Y\}$ by\index{differential polynomial!dominant degree}\index{degree!dominant}\label{p:ddeg P}
$$ \ddeg P\, :=\, \max\big\{d:\,v(P_d)=v(P)\big\}\in \N\ \text{ if $P\ne 0$,} \qquad \ddeg P\, :=\, -\infty \text{ if $P=0$.}$$
The 
{\em dominant weight~$\dwt P$\/}  of~$P$ is defined in the same way, and so is the  {\em dominant multiplicity\/} $\dval P$ of $P$ (at $0$), with $\dval P =+\infty$ if $P=0$.\index{differential polynomial!dominant multiplicity}\index{multiplicity!dominant}\index{differential polynomial!dominant weight}\index{weight!dominant} 

\medskip
\noindent
The derivation $\der$ of a valued differential field $K$ is said to be {\it small}\/ if~$\der(\smallo)\subseteq\smallo$; then~$\der$ is continuous with respect to the valuation topology of $K$, and~${\der(\mathcal O)\subseteq\mathcal O}$~[ADH, 4.4.2], so $\der$ induces a derivation on $\res(K)$ making the residue map ${\mathcal O\to\res(K)}$ into a morphism of differential rings.\index{valued differential field!small derivation}\index{small derivation}\index{derivation!small}   

We say that $K$ is {\em differential-henselian} ({\em $\d$-henselian\/} for short) if its derivation is small, the differential residue field $\res(K)$ is linearly surjective, and for every~${P\in \mathcal{O}\{Y\}}$ with $P_0\prec P_1\asymp 1$ there exists $y\prec 1$ in $K$ such that $P(y)=0$. See~[ADH, Chapter~7] for more, and for the weaker notions of {\it $r$-$\d$-henselian\/}, $r\in \N$.\index{dhenselian@$\d$-henselian}\index{r-d-henselian@$r$-$\d$-henselian}\index{valued differential field!r-d-henselian@$r$-$\d$-henselian}\index{valued differential field!d-henselian@$\d$-henselian}\index{valued differential field!differential-henselian}

\medskip
\noindent
Suppose now the derivation of $K$ is small,  and 
$A=\sum_i a_i\der^i\in  K[\der]$ (all $a_i\in K$). In~[ADH, 5.6] we defined
$v(A) := \min_i v(a_i) \in \Gamma_\infty$, extending accordingly the relations $\prec, \preceq$, etc.,  on $K$
to $K[\der]$.  For $A\neq 0$ we set
$$ \dwm A\, :=\, \min\big\{i:\,v(a_i)=v(A)\big\}\in \N,\qquad \dwt A\, :=\, \max\big\{i:\,v(a_i)=v(A)\big\}\in \N,$$
the {\it dominant weighted multiplicity}\/ of $A$ and the {\it dominant weight}\/ of $A$, respectively.\index{linear differential operator!dominant weight}\index{weight!dominant}\index{linear differential operator!dominant weighted multiplicity}\index{multiplicity!dominant weighted}\label{p:dwm A}
Let~$y$ range over $K^\times$. Then $v(Ay)$ only depends on~$vy$, not on $y$, and we can thus define $v_A(\gamma):=v(Ay)$ for $\gamma:=vy$.
The quantity $\dwt(Ay)$ also only depends on~$vy$, so we can define $\dwt_A(\gamma):=\dwt(Ay)$ for $\gamma=vy$.
Likewise, if~$\der(\mathcal O)\subseteq\smallo$, then~$\dwm(Ay)$   only depends on $vy$, and we define $\dwm_A(\gamma):=\dwm(Ay)$ for~$\gamma=vy$.
The set of {\it exceptional values}\/ of $A$ is $$\exc(A)\ =\ \exc_K(A)\ :=\ \big\{vy:\, \dwm(Ay)>0\big\}\ \subseteq\ \Gamma;$$ it contains~$v(\ker^{\neq} A)$, where
$\ker^{\neq}A:= (\ker A)^{\neq}$; cf.~[ADH, 5.6.7].\index{exceptional values}\index{linear differential operator!exceptional values}\index{values!exceptional}\label{p:exc A}

\subsection*{Ordered differential fields}
An  {\em ordered differential field\/} is a differential field $K$ with an ordering on $K$ making $K$ an ordered field.\index{ordered differential field}\index{differential field!ordered} Likewise, an  {\em ordered valued differential field\/} is a  valued differential field $K$ equipped with an ordering on $K$ making~$K$ an ordered field (no relation between derivation, valuation, or ordering
being assumed).\index{differential field!ordered valued}\index{ordered valued differential field} Let $K$ be an ordered differential field. We have the convex subring 
$$\mathcal{O}\ :=\ \big\{g\in K:\text{$\abs{g} \le c$ for some $c\in C$}\big\},$$ 
which is a valuation ring of $K$ and has maximal ideal
$$\smallo\ =\ \big\{g\in K:\text{$\abs{g} < c$ for all positive $c\in C$}\big\}.$$ We call $K$ an {\it $H$-field}\index{ordered differential field!$H$-field}\index{H-field@$H$-field} 
 if for all $f\in K$ with $f>C$ we have $f'>0$, and
  $\mathcal{O}=C+\smallo$. We view such an $H$-field $K$ as an ordered valued differential field with its valuation given by $\mathcal O$. 
{\it Pre-$H$-fields}\/\index{pre-H-field@pre-$H$-field}\index{ordered valued differential field!pre-$H$-field} are the ordered valued differential subfields of $H$-fields.
 See [ADH, 10.5] for basic facts about (pre-)$H$-fields. 
 An $H$-field $K$ is said to be {\it Liouville closed}\/
 if $K$ is real closed and for all $f,g\in K$ there exists~$y\in K^\times$ with~$y'+fy=g$.
 Every $H$-field extends to  a Liouville closed one; see [ADH, 10.6].\index{H-field@$H$-field!Liouville closed}\index{closed!Liouville}

\medskip
\noindent
Let $K$ be a pre-$H$-field. In [ADH, p.~520] we singled out  the following subsets: \label{p:special subsets} 
$$\Upg(K)\ :=\ (K^{\succ 1})^\dagger,\quad \Upl(K)\ :=\ -\big[(K^{\succ 1})^{\dagger\dagger}\big], \quad 
\Upd(K)\ :=\ -\big[(K^{\neq,\prec 1})^{\prime\dagger}\big].$$ 
 If    $K$ is   Liouville closed, then the restriction of $\omega$ to $\Upl(K)$ is strictly increasing with~$\omega\big(\Upl(K)\big)=\omega\big(\Upd(K)\big)$.  The restriction of
 $\sigma$ to $\Upg(K)$ is strictly increasing, and~$\omega\big(\Upl(K)\big)<\sigma\big(\Upg(K)\big)$; see [ADH, 11.8]. We call $K$ {\it Schwarz closed}\/ if
 $K$ is Liouville closed and $K=\omega\big(\Upl(K)\big)\cup\sigma\big(\Upg(K)\big)$
 % then $K$ is called {\it Schwarz closed}\/
  [ADH, 11.8.33].\index{closed!Schwarz}\index{H-field@$H$-field!Schwarz closed}

 \medskip\noindent
 We alert the reader that in a few places we refer to the Liouville closed $H$-field~$\T_{\g}$ of grid-based transseries from
 \cite{JvdH}, which is denoted there by $\T$. Here we  adopt the notation of [ADH] where $\T$ is the larger field of logarithmic-exponential series, which is also a Liouville closed $H$-field.\index{transseries}

\subsection*{Asymptotic fields}
In their capacity as valued differential fields, $H$-fields and pre-$H$-fields are among so-called {\em asymptotic fields}, which also include the algebraic closure of  any pre-$H$-field (where this algebraic closure is equipped with the unique derivation extending the derivation of the pre-$H$-field and any valuation extending the valuation of the pre-$H$-field).\index{asymptotic field}\index{valued differential field!asymptotic} That is one of the reasons (not the only one)  to consider this notion more closely:
An \textit{asymptotic field}\/ is a valued differential field~$K$ such that for all nonzero  $f,g\prec 1$ in $K$ we have: $f\prec g\Longleftrightarrow f'\prec g'$. Let~$K$ be an asymptotic field. Then $K^\phi$ with $\phi\in K^\times$ is also asymptotic. We associate to~$K$ its {\em asymptotic couple\/}\index{asymptotic couple}  $(\Gamma,\psi)$,
where $\psi\colon\Gamma^{\neq}\to\Gamma$ is given by 
$$\psi(vg)\ =\ v(g^\dagger)\ \text{ for $g\in K^\times$ with $vg\ne 0$}.$$ 
So the asymptotic couple of $K$ is the ordered abelian value group $\Gamma$ with a function on it that is induced by the derivation of $K$; it serves a similar purpose as the value group of a mere valued field. We put $\Psi:=\psi(\Gamma^{\neq})$, and if 
we want to stress the dependence on $K$ we also write $(\Gamma_K,\psi_K)$ and $\Psi_K$ instead of $(\Gamma,\psi)$
and $\Psi$, respectively.
See [ADH, 9.1, 9.2] for more on asymptotic couples, in particular the taxonomy of asymptotic fields introduced
via their asymptotic couples: having a {\it gap,}\/ being {\it grounded,}\/ having  {\it asymptotic integration,}\/ and having {\it rational asymptotic integration}.\index{asymptotic couple!gap}\index{asymptotic couple!grounded}\index{grounded!asymptotic couple}\index{asymptotic couple!asymptotic integration}\index{asymptotic couple!rational asymptotic integration}\index{gap}\index{grounded}\index{asymptotic integration}\index{asymptotic integration!rational} 
We now consider various conditions on our asymptotic field $K$, some mentioned in the last sentence. To define those conditions, let $f$, $g$ range over $K$, and $c$ over the constant field $C$ of $K$: 

\medskip
\noindent
\textit{$H$-asymptotic \textup{(or \textit{of $H$-type}\/)}}\/:     $0\neq f\prec g\prec 1\Rightarrow f^\dagger\succeq g^\dagger$. (Pre-$H$-fields and their algebraic closures are $H$-asymptotic. If $K$ is  $H$-asymptotic and $\phi\in K^\times$, then $K^\phi$ is $H$-asymptotic.)\index{asymptotic field!$H$-asymptotic}\index{H-asymptotic field@$H$-asymptotic field} 

\medskip
\noindent
\textit{Differential-valued \textup{(or \textit{$\d$-valued}\/)}}\/: for
 all $f\asymp 1$  
there exists $c$ with $f\sim c$. ($H$-fields and their algebraic closures are $\d$-valued.)\index{asymptotic field!$\d$-valued} 

\medskip
\noindent
{\it Pre-differential-valued \textup{(or {\it pre-$\d$-valued}\/)}}\/:  $f\preceq 1\ \&\ 0\neq g\prec 1\Rightarrow f'\prec g^\dagger$. (Pre-$H$-fields are pre-$\d$-valued. Every pre-$\d$-valued field has a canonical $\d$-valued extension, its
$\d$-valued hull $\operatorname{dv}(K)$, by [ADH, 10.3].)\index{asymptotic field!pre-$\d$-valued}\label{p:dv(K)} 

\medskip
\noindent
\textit{Grounded}\/: there is a nonzero $f\nasymp 1$ such that for all nonzero $g\nasymp 1$ we have
$g^\dagger \succeq f^\dagger$.\index{asymptotic field!grounded}\index{grounded!asymptotic field}

\medskip
\noindent
\textit{Asymptotic integration}\/:
for all $f\neq 0$ there exists $g\nasymp 1$ with $g'\asymp f$. (If $K$ has asymptotic integration, then $K$ is ungrounded with $\Gamma\ne \{0\}$. Liouville closed $H$-fields have asymptotic integration.)\index{asymptotic field!asymptotic integration}\index{asymptotic integration!asymptotic field}

%\medskip
%\noindent
%\textit{Asymptotically maximal}\/: $K$ has no proper immediate asymptotic field extension;
%p.~\pageref{p:asymptotically maximal}.

%\medskip
%\noindent
%\textit{Asymptotically $\d$-algebraically maximal}\/: $K$ has no proper immediate differential-al\-ge\-braic asymptotic field extension;
%p.~\pageref{p:asymptotically d-algebraically maximal}.

\medskip
\noindent
\textit{$\upl$-free}\/: $H$-asymptotic, ungrounded, and for all $f$ there exists $g\succ 1$ with $f-g^{\dagger\dagger}\succeq g^\dagger$.\index{asymptotic field!lambda-free@$\upl$-free}\index{lambda-free@$\upl$-free}

\medskip
\noindent
\textit{$\upo$-free}\/:
$H$-asymptotic, ungrounded, and for all $f$ there exists  $g\succ 1$ such that ${f-\omega(g^{\dagger\dagger})\succeq (g^\dagger)^2}$, where $\omega$ is as in \eqref{eq:def omega}.\index{asymptotic field!omega-free@$\upo$-free}\index{omega-free@$\upo$-free}\label{p:omega}  %, where $\omega(z):=-(2z'+z^2)$.

\medskip\noindent
Here we note that $\upo$-freeness is very robust, and powerful, and that $\T$ is $\upo$-free.  
For more on this, see [ADH, 13.6] and Section~\ref{sec:uplupo-freeness} below.

\subsection*{Flattening} This is from [ADH, 9.4].\index{H-asymptotic field@$H$-asymptotic field!flattening}\index{valuation!flattening} Suppose $K$ is $H$-asymptotic   
with asymptotic couple $(\Gamma,\psi)$. Then we have a convex subgroup $\Gamma^\flat:=\big\{\gamma\in\Gamma:\,\psi(\gamma)>0\big\}$   of
$\Gamma$, and  
the $\Gamma^\flat$-coarsening~$v^\flat\colon K^\times\to \Gamma/\Gamma^\flat$   of $v$    is called the {\it flattening}\/ of $v$. The
  differential field~$K$ together with the valuation ring of $v^\flat$ is $H$-asymptotic. We denote the     relations~$\asymp$,~$\sim$,~$\preceq$,~$\prec$ associated to $v^\flat$ by $\asymp^\flat$, $\sim^\flat$, $\preceq^\flat$, $\prec^\flat$, respectively.\label{p:flatten} For~$\phi\in K^\times$ we denote the flattened objects $v_\phi^\flat$, $\Gamma^\flat_\phi$, $\asymp^\flat_\phi$, $\sim^\flat_\phi$, $\preceq^\flat_\phi$, $\prec^\flat_\phi$ associated to~$K^\phi$ by a subscript~$\phi$. In particular, $\Gamma^\flat_\phi=\big\{\gamma\in\Gamma:\psi(\gamma)>v\phi\big\}$.
 
\subsection*{Newtonianity}
In order to define {\it newtonian\/}, assume the asymptotic field $K$ is ungrounded with ${\Gamma\ne \{0\}}$.  Then an element~$\phi\in K$ is said to be {\em active in $K$\/}\index{active}\index{element!active} if~$\phi\ne 0$ and
$\phi\succeq f^\dagger$ for some nonzero $f\nasymp 1$. If $\phi$ is active in $K$, then
the derivation~$\phi^{-1}\der$ of~$K^\phi$ is small.  Let $\phi$ range over the active elements of $K$.
A property $S(\phi)$ of (active) elements~$\phi$ is said to hold {\em eventually}\/\index{eventually} if there is an active $\phi_0$ in $K$ such that~$S(\phi)$ holds for all~$\phi\preceq \phi_0$; cf.~[ADH, p.~479].   The way to understanding Liouville closed $H$-fields such as $\T$ involves often eventual behavior of this kind. For example,  in~[ADH, 11.1] we showed that 
for $P\in K\{Y\}$, 
  $\dval P^\phi$, $\ddeg P^\phi$, and $\dwt P^\phi$ are eventually constant. The eventual values of these quantities are denoted by $\nval P$, $\ndeg P$, and $\nwt P$, respectively, and are
  called the {\it Newton multiplicity}\/ of $P$, the {\it Newton degree}\/ of $P$, and the {\it Newton weight of $P$.}\label{p:newton quants}\index{degree!Newton}\index{newton!degree}\index{differential polynomial!Newton degree}\index{multiplicity!Newton}\index{newton!multiplicity}\index{differential polynomial!Newton multiplicity}\index{weight!Newton}\index{newton!weight}\index{differential polynomial!Newton weight} 
We call $P$
 {\em quasilinear\/} if~$\ndeg P=1$.\index{differential polynomial!quasilinear}\index{quasilinear!differential polynomial} 

\begin{definition}
An asymptotic field $K$ is said to be {\em newtonian}\/ if $K$ is ungrounded of $H$-type with 
$\Gamma\ne \{0\}$, and every quasilinear $P\in K\{Y\}$ has a zero in $\mathcal{O}$.\index{asymptotic field!newtonian}\index{newtonian} 
\end{definition}

\noindent
We now list some properties of this notion from~[ADH] that we shall frequently use.
For this, assume as before that $K$ is an ungrounded $H$-asymptotic field with~${\Gamma\ne \{0\}}$. First a  consequence of~[ADH, 14.0.1 and the remarks after it]:
\begin{flalign}\label{eq:14.0.1}
&\quad\parbox[t]{24em}{ {\it If $K$ is $\upo$-free, then $K$ has an immediate $\d$-algebraic extension which is
newtonian and $\upo$-free.}\/} && \text{[ADH, 14.0.1]}
\end{flalign}
%\begin{list}{*}{\addtolength\itemindent{-3em}\addtolength\leftmargin{5em}}
%\item[{[ADH, 14.0.1]}] {\it If $K$ is $\upo$-free, then $K$ has an immediate $\d$-algebraic extension which is newtonian and $\upo$-free.}\/
%\end{list}
For $\d$-valued $K$ we often need the following:
\begin{flalign}\label{eq:14.5.3}
&\quad\parbox[t]{24em}{\it If $K$ is $\d$-valued, $\upo$-free,  newtonian, and algebraically closed, then $K$ is
weakly differentially closed.} 
%\textup{(}and thus linearly closed\textup{)}.} 
&& \text{[ADH, 14.5.3]}
\end{flalign}
For more on $\d$-valued $K$, see~[ADH, 14.5.4] and \cite[Theorem~B]{Nigel19}. 
%In the same vein:
\begin{flalign}\label{eq:14.4.2}
&\quad\parbox[t]{24em}{\it If $K$ is newtonian, then $K$ is linearly surjective.} && \text{[ADH, 14.4.2]}
\end{flalign}
Next, suppose $K$ is $\upo$-free and $\d$-valued with divisible value group, and $L$ is a valued differential field extension of $K$ and algebraic over $K$. Then:
\begin{flalign}
&\quad\parbox[t]{24em}{\it If $K$ is newtonian, then so is $L$.} && \text{[ADH, 14.5.7]}\label{eq:14.5.7}
\\
%\begin{flalign}
&\quad\parbox[t]{24em}{\it $K$ is newtonian if $L$ is newtonian and $L=K(C_L)$.} && \text{[ADH, 14.5.6]}\label{eq:14.5.6} 
\end{flalign}
For example, the valued differential field $\T$ is newtonian by [ADH,  15.0.2], hence its algebraic closure~$\T[\imag]$ is also newtonian by~\eqref{eq:14.5.7}, and  thus $\T[\imag]$ is linearly closed by \eqref{eq:14.5.3}
and   linearly surjective  by~\eqref{eq:14.4.2}.

\medskip
\noindent
See also [ADH, 14.2] for  {\it $r$-newtonian}\/ and {\it $r$-linearly newtonian}, useful weakenings of {\it newtonian\/} that allow for induction on $r\in \N$.\/\index{asymptotic field!r-newtonian@$r$-newtonian}\index{r-newtonian@$r$-newtonian}\index{asymptotic field!r-linearly newtonian@$r$-linearly newtonian}\index{r-linearly newtonian@$r$-linearly newtonian}

\subsection*{Closed $H$-fields} A {\it closed $H$-field}\/ (or {\it $H$-closed field}) is a Liouville closed, $\upo$-free, and newtonian $H$-field.\index{closed!H-closed@$H$-closed}\index{closed!H-field@$H$-field}\index{H-field@$H$-field!closed}\index{H-field@$H$-field!H-closed@$H$-closed}
A fundamental fact from [ADH] about the elementary theory
of $\T$ as an ordered valued differential field is that it is completely axiomatized by the requirements of being a closed $H$-field with
 small derivation. 
Moreover, 
the closed $H$-fields  are exactly the existentially closed models of the theory of $H$-fields.

\section*{Notions of Normality}\label{section:flowchart}

\noindent
In the diagram below we collect various normality conditions on slots and holes of order~$\ge 1$, together with
the section of first occurrence and a list of preservation results.
Concepts below the dotted line apply to any
  $H$-asymptotic field with small derivation and rational asymptotic integration,  while
 those above it apply to any real closed $H$-field with small derivation and asymptotic integration, except for ``ultimate'' where we assume also being Liouville closed. 
 Arrows indicate logical implications between the various conditions. 
 
\bigskip

\tikzstyle{narrowbox} = [rectangle, draw=black, thick, fill=white,
      text width=6.5em, text ragged,  inner sep=0.5em,  rounded corners]
\tikzstyle{box} = [rectangle, draw=black, thick, fill=white,
      text width=7em, text ragged,  inner sep=0.5em,  rounded corners]
\tikzstyle{verynarrowbox} = [rectangle, draw=black, thick, fill=white,
      text width=5em, text ragged, inner sep=0.5em,  rounded corners]
\tikzstyle{widebox} = [rectangle, draw=black, thick, fill=white,
      text width=9em, text ragged,  inner sep=0.5em,  rounded corners]

\noindent\begin{tikzpicture}[node distance=12em, scale=0.85, every node/.style={scale=0.85}]
  
\node [box, text width=8.5em] (stronglyrepulsivenormal) at (-2.25,0) {\textit{strongly repulsive-normal} (\ref{sec:repulsive-normal})\\[0.1em] \small \ref{stronglyrepnormalrefine}, \ref{pqr}, \ref{cor:strongly repulsive-normal compconj}};  
      
\node [box,text width=10em, right of=stronglyrepulsivenormal] (almoststronglyrepulsivenormal) {\textit{almost strongly repulsive-normal} (\ref{sec:repulsive-normal})\\[0.1em] \small \ref{stronglyrepnormalrefine}, \ref{pqr}, \ref{cor:strongly repulsive-normal compconj}}; 
    
\node [box, text width=9.75em, right of=almoststronglyrepulsivenormal] (repulsivenormal) {\textit{repulsive-normal} (\ref{sec:repulsive-normal})\\[0.1em]\small
\ref{lem:repulsive-normal comp conj}--\ref{lem:5.21 repulsive-normal} }; 
    
\node [box, text width=8.5em] (stronglysplitnormal)  at (-2.25,-2.25)  {\textit{strongly split-normal} (\ref{sec:split-normal holes})\\[0.1em]\small \ref{stronglysplitnormalrefine}, \ref{stronglysplitnormalrefine, q}, \ref{lem:strongly split-normal compconj}};   
    
\node [widebox,  text width=8.5em,  right of=stronglysplitnormal] (almoststronglysplitnormal) {\textit{almost strongly split-normal} (\ref{sec:split-normal holes})\\[0.1em]\small \ref{stronglysplitnormalrefine}, \ref{stronglysplitnormalrefine, q}, \ref{lem:strongly split-normal compconj}}; 
    
\node [box, text width=8em, right of=almoststronglysplitnormal] (splitnormal) {\textit{split-normal} (\ref{sec:split-normal holes})\\[0.1em] \small
    \ref{lem:split-normal comp conj}, \ref{splitnormalrefine}--\ref{splnq}    }; 

\node [widebox] (strictlynormal) at (-2.25,-4.75) {\textit{strictly  normal} (\ref{sec:normalization}) \\[0.1em] \small \ref{lem:normality comp conj, strong}--\ref{rem:strongly normal refine, 2} }; 

\node [box, text width=5.75em] at (6.2,-4.75) (normal) {\textit{normal} (\ref{sec:normalization}) \\[0.1em] \small\ref{lem:normality comp conj}, \ref{ufm}, \\\ref{normalrefine}--\ref{cor:normal for small q} }; 

\node [verynarrowbox, text width=4.6em] at (7,-6.5) (steep) {\textit{steep}  (\ref{sec:normalization})\\[0.1em] \small\ref{lem:steep1}  }; 

\node [box, text width=7.25em, left of=steep] (quasilinear) {\textit{quasilinear} (\ref{sec:holes}) \\[0.1em] \small \ref{cor:ref 2n} }; 

\node [verynarrowbox, text width=4.5em] at (4,-8.5) (deep) {\textit{deep}  (\ref{sec:normalization})\\[0.1em] \small\ref{lem:deep 1}--\ref{cor:deep 2, cracks}}; 

\node [narrowbox, text width=6.25em] at (9.5,-1.75) (ultimate) {\textit{ultimate}  (\ref{sec:ultimate})\\[0.1em]\small \ref{lem:ultimate refinement}, \ref{lem:ultmult} }; 

\node [narrowbox, text width=5.75em] at (9.5,-5) (isolated) {\textit{isolated}  (\ref{sec:isolated})\\[0.1em]\small \ref{lem:isolated refinement}, \ref{lem:isolated}};

\draw [dotted, thick] (-4,-3.4) -- (10.75,-3.4);
\draw [->, thick,  double distance=0.1em, arrows = -{Stealth[length=0.8em]}] (stronglyrepulsivenormal) -- (almoststronglyrepulsivenormal);
\draw [->, thick,  double distance=0.1em, arrows = -{Stealth[length=0.8em]}] (almoststronglyrepulsivenormal) -- (repulsivenormal);
\draw [->, thick,  double distance=0.1em, arrows = -{Stealth[length=0.8em]}] (repulsivenormal) -- (splitnormal);
\draw [->, thick,  double distance=0.1em, arrows = -{Stealth[length=0.8em]}] (stronglyrepulsivenormal) -- (stronglysplitnormal);
\draw [->, thick,  double distance=0.1em, arrows = -{Stealth[length=0.8em]}] (almoststronglyrepulsivenormal) -- (almoststronglysplitnormal);
\draw [->, thick,  double distance=0.1em, arrows = -{Stealth[length=0.8em]}] (stronglysplitnormal) -- (strictlynormal);
\draw [->, thick,  double distance=0.1em, arrows = -{Stealth[length=0.8em]}] (strictlynormal) -- (normal);
\draw [->, thick,  double distance=0.1em, arrows = -{Stealth[length=0.8em]}] (splitnormal) -- (normal);
\draw [->, thick,  double distance=0.1em, arrows = -{Stealth[length=0.8em]}] (normal) -- (quasilinear);
\draw [->, thick,  double distance=0.1em, arrows = -{Stealth[length=0.8em]}] (normal) -- (steep);
\draw [->, thick,  double distance=0.1em, arrows = -{Stealth[length=0.8em]}] (deep) -- (quasilinear);
\draw [->, thick,  double distance=0.1em, arrows = -{Stealth[length=0.8em]}] (deep) -- (steep);
\draw [->, thick,  double distance=0.1em, arrows = -{Stealth[length=0.8em]}] (ultimate) -- (isolated);
\draw [->, thick,  double distance=0.1em, arrows = -{Stealth[length=0.8em]}] (stronglysplitnormal) -- (almoststronglysplitnormal);
\draw [->, thick,  double distance=0.1em, arrows = -{Stealth[length=0.8em]}] (almoststronglysplitnormal) -- (splitnormal);

 \end{tikzpicture}     

\smallskip 
\noindent
We conclude with listing significant normalization results that we prove on the way to the main theorem stated earlier in this introduction. Throughout, $r\in \N^{\ge 1}$.

\medskip
\noindent
{\bf Theorem~\ref{mainthm}.} \textit{If $K$ is an $\upo$-free $r$-linearly newtonian $H$-asymptotic field with small derivation and 
divisible value group, then every $Z$-minimal slot in $K$
of order~$r$ has a refinement $(P,\fm,\hat a)$ such that~$(P^\phi,\fm,\hat a)$ is deep and normal, eventually.}
 
\medskip
\noindent
Below 
$H$ is an  $\upo$-free real closed $H$-field with small derivation, $\hat H$ is an immediate asymptotic extension of~$H$, and
 $K:=H[\imag]$,  $\hat K:= \hat H[\imag]$.
Moreover, $(P,\fm,\hat a)$ is a minimal hole in $K$ of order $r$, with  $\fm\in H^\times$ and~$\hat a=\hat b+\hat c\imag\in\hat K\setminus K$ ($\hat b,\hat c\in\hat H$).

\medskip

{\noindent
{\bf Theorem~\ref{thm:split-normal}.}  \textit{If $H$ is $1$-linearly newtonian, then one of the following holds:
\begin{list}{}{\leftmargin=2em}
\item[$\mathrm{(i)}$] $\hat b\notin H$  and some $Z$-minimal slot $(Q,\fm,\hat b)$ in $H$ has  a re\-fine\-ment~${(Q_{+b},\fn,\hat b-b)}$ such that $(Q^\phi_{+b},\fn,\hat b-b)$ is eventually deep and split-nor\-mal; 
\item[$\mathrm{(ii)}$] $\hat c\notin H$  and some $Z$-minimal slot $(R,\fm,\hat c)$ in $H$ has a refinement~${(R_{+c},\fn,\hat c-c)}$ such that $(R^\phi_{+c},\fn,\hat c-c)$ is eventually deep and split-normal. 
\end{list}}}

\medskip
{\noindent
{\bf Theorem~\ref{thm:strongly split-normal}.} \textit{If $H$ is $1$-linearly newtonian, then one of the following holds: 
\begin{list}{}{\leftmargin=2em}
\item[\textup{(i)}] $\hat b\notin H$  and some $Z$-minimal slot $(Q,\fm,\hat b)$ in $H$ has a refinement~${(Q_{+b},\fn,\hat b-b)}$ such that $(Q^\phi_{+b},\fn,\hat b-b)$ is eventually deep and almost strongly split-normal; 
\item[\textup{(ii)}] $\hat c\notin H$  and some $Z$-minimal slot $(R,\fm,\hat c)$ in $H$ has a refinement~${(R_{+c},\fn,\hat c-c)}$ such that $(R^\phi_{+c},\fn,\hat c-c)$ is eventually deep and almost strongly split-normal. 
\end{list}
Moreover, if $H$ is $1$-linearly newtonian and either  $\deg P>1$, or $\hat b\notin H$ and   $Z(H,\hat b)$ contains an element of order~$1$, or $\hat c\notin H$ and   $Z(H,\hat c)$ contains an element of order~$1$,
then \textup{(i)} holds with ``almost'' omitted, or  \textup{(ii)} holds with ``almost'' omitted. }}

\bigskip

{\noindent\sloppy {\bf Theorem~\ref{thm:repulsive-normal}.} \textit{Suppose the constant field $C$ of $H$ is archimedean and $\deg P>1$. Then one of the following conditions is satisfied:
\begin{list}{}{\leftmargin=2em}
\item[$\mathrm{(i)}$] $\hat b\notin H$  and some $Z$-minimal  slot $(Q,\fm,\hat b)$ in $H$ has a special re\-fine\-ment~${(Q_{+b},\fn,\hat b-b)}$ such that $(Q^\phi_{+b},\fn,\hat b-b)$ is eventually deep and repulsive-nor\-mal; 
\item[$\mathrm{(ii)}$] $\hat c\notin H$  and some  $Z$-minimal  slot $(R,\fm,\hat c)$ in $H$ has a special  re\-fine\-ment~${(R_{+c},\fn,\hat c-c)}$ such that   $(R^\phi_{+c},\fn,\hat c-c)$ is eventually deep and repulsive-nor\-mal. 
\end{list}}

 \medskip
\noindent
{\bf Theorem~\ref{thm:strongly repulsive-normal}.} \textit{Suppose $H$ is Liouville closed, $C$ is archimedean, $\I(K)\subseteq K^\dagger$,  and $\deg P>1$. Then one of the following conditions is satisfied:
\begin{list}{}{\leftmargin=2em}
\item[$\mathrm{(i)}$] $\hat b\notin H$  and some $Z$-minimal slot $(Q,\fm,\hat b)$ in $H$ has a special re\-fine\-ment~${(Q_{+b},\fn,\hat b-b)}$ such that $(Q^\phi_{+b},\fn,\hat b-b)$ is eventually deep, strongly re\-pul\-sive-normal, and ultimate; 
\item[$\mathrm{(ii)}$] $\hat c\notin H$  and some $Z$-minimal slot $(R,\fm,\hat c)$ in $H$ has a special re\-fine\-ment~${(R_{+c},\fn,\hat c-c)}$ such that $(R^\phi_{+c},\fn,\hat c-c)$ is eventually deep, strongly re\-pul\-sive-normal, and ultimate. 
\end{list}}}

\section*{Acknowledgements}

\noindent
Various institutions supported this work during its genesis:
Aschenbrenner and van den Dries thank the Institut Henri Poincar\'e, Paris, for its
hospitality  during
the special semester ``Model Theory, Combinatorics and Valued Fields'' in   2018.
Aschenbrenner   also   thanks the Institut f\"ur Ma\-the\-ma\-ti\-sche Logik und Grundlagenforschung, Universit\"at M\"unster, for     a M\"unster Research Fellowship during the spring of~2019, and he    acknowledges partial support from NSF Research Grant DMS-1700439.
All three authors   received support from the Fields Institute for   Research in Mathematical Sciences, Toronto, during its 
Thematic Program ``Tame Geometry, Transseries and Applications to Analysis and Geometry'', Spring~2022.
We also thank the anonymous referees for numerous suggestions that led to an improved exposition.

\newpage

\addtocontents{toc}{\protect\setcounter{tocdepth}{1}} 
\renewcommand\theequation{\arabic{part}.\arabic{section}.\arabic{equation}}
\renewcommand\thetheorem{\arabic{part}.\arabic{section}.\arabic{theorem}}

%\endgroup

\newpage

\part{Preliminaries}\label{part:preliminaries}

\medskip

\noindent
After generalities on linear differential operators and differential polynomials in
Section~\ref{sec:diff ops and diff polys}, we  investigate the group of
logarithmic derivatives in valued differential fields of various kinds (Section~\ref{sec:logder}). We also
assemble some basic preservation theorems for {\it $\upl$-freeness}\/ and {\it $\upo$-freeness}\/ (Section~\ref{sec:uplupo-freeness})
and continue the study of linear differential operators over $H$-asymptotic fields initiated in [ADH, 5.6, 14.2] (Section~\ref{sec:lindiff}).
In our analysis of the solutions of algebraic differential equations over $H$-asymptotic fields in Chapter~\ref{part:normalization},
special pc-sequences in the sense of~[ADH, 3.4] play an important role; Section~\ref{sec:special elements} explains why.
A cornerstone of~[ADH] is the concept of {\it newtonianity}\/,  an analogue of henselianity
appropriate for $H$-asymptotic fields with asymptotic integration [ADH, Chapter~14].
Related to this is  {\it differential-henselianity}\/~[ADH, Chapter~7], which makes sense for a broader class of
valued differential fields.
In Sections~\ref{sec:completion d-hens} and~\ref{sec:complements newton} we further explore these notions. Among other things, we study their persistence under taking the completion of a  valued differential field with small derivation (as defined in [ADH, 4.4]).

\section{Linear Differential Operators and Differential Polynomials}\label{sec:diff ops and diff polys}

\noindent
This section gathers miscellaneous facts of a general nature 
about linear differential operators and differential polynomials, sometimes in a valued differential setting. We   discuss splittings and least common left multiples of linear differential operators,  %then recall the complexity and the separant of differential polynomials, 
and then prove some useful estimates for derivatives
of exponential terms and for Riccati transforms.

\subsection*{Splittings}
In this subsection $K$ is a differential field. Let $A\in K[\der]^{\neq}$ be monic of order~$r\geq 1$.
A {\bf splitting of $A$ over $K$} is a tuple~$(g_1,\dots, g_r)\in K^r$ such that~$A=(\der-g_1)\cdots(\der-g_r)$. If $(g_1,\dots, g_r)$ is a   splitting of
$A$ over $K$ and $\fn\in K^\times$, then~$(g_1-\fn^\dagger,\dots, g_r-\fn^\dagger)$
is a   splitting of~$A_{\ltimes \fn}=\fn^{-1}A\fn$ over $K$.\index{splitting}\index{linear differential operator!splitting}

\medskip
\noindent
Suppose $A=A_1\cdots A_m$ where every $A_i\in K[\der]$ is monic of positive order $r_i$ (so~$r=r_1+\cdots + r_m$). Given any splittings 
$$(g_{11},\dots, g_{1r_1}),\ \dots,\ (g_{m1},\dots,g_{mr_m})$$ of $A_1,\dots, A_m$, respectively, we obtain a splitting 
$$\big(g_{11},\dots, g_{1r_1},\ \dots,\  g_{m1},\dots, g_{mr_m}\big)$$
of $A$ by concatenating the given splittings of $A_1,\dots, A_m$ in the order indicated, and call it a splitting of $A$ {\bf induced} by the factorization $A=A_1\cdots A_m$.  \index{splitting!induced by a factorization}
 For $B\in K[\der]$ of order $r\ge 1$ we have $B=bA$ with $b\in K^{\times}$ and monic
$A\in K[\der]$, and then a {\bf  splitting of $B$ over~$K$\/} is by definition a  splitting of $A$ over $K$.
A splitting of~$B$ over $K$ remains a splitting of $aB$ over $K$, for any $a\in K^\times$. Thus:

\begin{lemma}\label{lem:split and twist} 
If $B\in K[\der]$ has order $r\geq 1$, and $(g_1,\dots, g_r)$ is a splitting of
$B$ over $K$ and $\fn\in K^\times$, then~$(g_1-\fn^\dagger,\dots, g_r-\fn^\dagger)$
is a   splitting of $B_{\ltimes\fn}$ over $K$ and  a splitting of $B\fn$ over $K$.
\end{lemma}

\noindent
Let $A\in K[\der]^{\neq}$ and $r:=\order(A)$. 
From [ADH,~5.1,~5.7] we know that if $A$ splits over $K$, then for any~$\phi\in K^\times$ the
operator~$A^\phi\in K^\phi[\derdelta]$ splits over $K^\phi$; here is how a splitting of $A$ over $K$ transforms into
a splitting of $A^\phi$ over $K^\phi$: 

\begin{lemma}\label{lem:split and compconj}
Let $\phi\in K^\times$ and $$A\ =\ c(\der-a_1)\cdots(\der-a_r)\quad\text{ with $c\in K^\times$ and $a_1,\dots, a_r\in K$.}$$
Then  in $K^\phi[\derdelta]$ we have
$$A^\phi\ =\ c\phi^r(\derdelta-b_1)\cdots(\derdelta-b_r)\quad\text{ where
$b_j:=\phi^{-1}\big(a_j-(r-j)\phi^\dagger\big)$   \textup{(}$j=1,\dots,r$\textup{)}.}$$
\end{lemma}
\begin{proof}
Induction on $r$. The case $r=0$ being obvious, suppose $r\geq 1$, and set~$B:=(\der-a_2)\cdots(\der-a_r)$.  By inductive hypothesis 
$$B^\phi=\phi^{r-1}(\derdelta-b_2)\cdots(\derdelta-b_r)\quad\text{ where
$b_j:=\phi^{-1}\big(a_j-(r-j)\phi^\dagger\big)$  for $j=2,\dots,r$.}$$
Then
$$A^\phi\ =\ c\phi\,\big(\derdelta-(a_1/\phi)\big)\,B^\phi\ =\ c\phi^r \, \big(\derdelta-(a_1/\phi)\big)_{\ltimes \phi^{r-1}}\, (\derdelta-b_2)\cdots(\derdelta-b_r)$$
with $$\big(\derdelta-(a_1/\phi)\big)_{\ltimes \phi^{r-1}} = \derdelta-(a_1/\phi) + (r-1)\phi^\dagger/\phi$$ by [ADH, p.~243].
\end{proof}

\noindent
A different kind of factorization, see for example~\cite{Polya},
reduces the process of solving the differential equation~$A(y)=0$ to repeated multiplication and integration:

\begin{lemma}\label{lem:Polya fact} 
Let $A\in K[\der]^{\neq}$ be monic of order $r\geq 1$. If $b_1,\dots,b_r\in K^\times$ and
$$A\ =\ b_1\cdots b_{r-1}b_r (\der b_r^{-1})( \der b_{r-1}^{-1})\cdots (\der b_1^{-1}),$$ then $(a_r,\dots,a_1)$, where $a_j:=(b_1\cdots b_j)^\dagger$ for $j=1,\dots,r$,
is a splitting of $A$ over~$K$. Conversely, if $(a_r,\dots,a_1)$ is a splitting of $A$ over $K$ and $b_1,\dots,b_r\in K^\times$ are such that
$b_j^\dagger=a_j-a_{j-1}$ for $j=1,\dots,r$ with $a_0:=0$, then $A$ is as in the display.
\end{lemma}

\noindent
This follows easily by induction on $r$.

\subsection*{Real splittings} Let $H$ be a differential field in which
$-1$ is not a square. Then we let $\imag$ denote an element in a differential field extension of $H$ with $\imag^2=-1$, and consider the
differential field $K=H[\imag]$. 
Suppose $A\in H[\der]$ is monic of order $2$ and splits over $K$, so 
$$A\ =\ (\der-f)(\der-g),\qquad f,g\in K.$$ Then 
$$A\ =\ \der^2-(f+g)\der+fg-g',$$ and thus $f\in H$ iff $g\in H$. 
One checks easily that if $g\notin H$, then there are unique~$a,b\in H$ with $b\ne 0$ such that 
$$f\ =\  a-b\imag+b^\dagger, \qquad g\ =\ a+b\imag,$$
and thus $$A\  =\  \der^2-(2a+b^\dagger)\der + a^2+ b^2 -a'+ab^\dagger.$$
Conversely, if $a,b\in H$ and $b\ne 0$, then for $f:=a-b\imag+b^\dagger$ and $g:=a+b\imag$ we have~$(\der -f)(\der-g)\in H[\der]$. 

\medskip
\noindent
Let now $A\in H[\der]$ be monic of order $r\ge 1$.
Recall: $A$ is  {\it irreducible}\/ iff there are no monic $A_1,A_2\in K[\der]$ of positive order with $A=A_1A_2$; cf.~[ADH, p.~250].\index{irreducible}\index{linear differential operator!irreducible} 

\begin{lemma}\label{hkspl} Suppose $A$ splits over $K$. Then
  $A=A_1\cdots A_m$ for some  $A_1,\dots, A_m$ in $H[\der]$ that are monic and irreducible of order $1$ or $2$ and split over $K$. 
\end{lemma}
\begin{proof} By [ADH, 5.1.35], $A=A_1\cdots A_m$, where every $A_i\in H[\der]$ is monic and irreducible of order $1$ or $2$.
By [ADH, 5.1.22], such $A_i$ split over $K$.  
\end{proof}

\begin{definition}\label{def:real splitting}  \index{splitting!real}\index{linear differential operator!real splitting}
A {\bf real splitting of $A$} (over $K$) is a splitting of~$A$ over $K$ that is induced by a factorization $A=A_1\cdots A_m$ where every $A_i\in H[\der]$ is monic of order $1$ or $2$ and splits over $K$. (Note that we do not require the $A_i$ of order $2$ to be irreducible in $H[\der]$.)
\end{definition}

\noindent
Thus if $A$ splits over $K$, then  $A$ has a real splitting over $K$ by Lemma~\ref{hkspl}. Note that if~$(g_1,\dots, g_r)$ is a real splitting of
$A$ and $\fn\in H^\times$, then~$(g_1-\fn^\dagger,\dots, g_r-\fn^\dagger)$
is a real splitting of $A_{\ltimes \fn}$.  

\medskip
\noindent
It is convenient to extend the above slightly: for $B\in H[\der]$ of order $r\ge 1$ we have~$B=bA$ with $b\in H^{\times}$ and monic
$A\in H[\der]$, and then a {\bf real splitting of~$B$\/}~(over~$K$) is by definition a real splitting of~$A$ (over $K$).  

\medskip
\noindent
In later use, $H$ is a valued differential field with small derivation such that $-1$ is not a square in the differential residue field $\res(H)$. For such $H$, let $\mathcal{O}$ be the valuation ring of $H$. We make $K$ a valued differential field extension of $H$ with small derivation by taking 
$\mathcal{O}_K=\mathcal{O}+\mathcal{O}\imag$ as the valuation ring of $K$. We have the residue map
$a\mapsto \res a\colon \mathcal{O}_K\to \res(K)$, so
 $\res(K)=\res(H)[\imag]$, writing here $\imag$ for~$\res\imag$.
We extend this map to a ring morphism $B\mapsto\res B\colon\mathcal{O}_K[\der]\to\res(K)[\der]$ by sending
$\der\in \mathcal{O}[\der]$ to $\der\in \res(K)[\der]$.  

\begin{lemma}\label{lem:lift real splitting}
Suppose $(g_1,\dots,g_r)\in \res(K)^r$ is a real splitting of a monic
operator $D \in\res(H)[\der]$ of order $r\ge 1$.
Then there are $b_1,\dots,b_r\in\mathcal O_K$ such that 
$$B\ :=\ (\der-b_1)\cdots (\der-b_r)\in \mathcal O[\der],$$
$(b_1,\dots,b_r)$ is a real splitting of $B$, 
and $\res b_j=g_j$ for~$j=1,\dots,r$.
\end{lemma}
\begin{proof}
We can assume  
$r\in\{1,2\}$. The case $r=1$ is obvious, so let $r=2$. Then the case
where $g_1, g_2\in \res(H)$ is again obvious, so let $g_1=\res(a)-\res(b)\imag +(\res b)^\dagger$, $g_2=\res(a)+\res(b)\imag$ where
$a,b\in\mathcal O$, $\res b\neq 0$. Set $b_1:=a-b\imag+b^\dagger$, $b_2:=a+b\imag$. Then $b_1,b_2\in\mathcal O_K$ with~$\res b_1=g_1$,~$\res b_2=g_2$, and $B:=(\der-b_1)(\der-b_2)\in\mathcal O[\der]$ have the desired properties.
\end{proof}

\subsection*{Least common left multiples and complex conjugation} 
{\it In this subsection~$H$ is a differential field.}\/
Recall from [ADH,~5.1] the definition of the {\it least common left multiple}\/\index{least common left multiple}\index{linear differential operator!least common left multiple}
$\operatorname{lclm}(A_1,\dots,A_m)$ of operators $A_1,\dots,A_m\in H[\der]^{\neq}$, $m\geq 1$:
this is the monic operator $A\in H[\der]$ such that $H[\der]A_1\cap\cdots \cap H[\der]A_m=H[\der]A$.
For~${A,B\in H[\der]^{\neq}}$ we have:
$$\max\big\{ \!\order(A),\order(B) \big\} \ \leq\ \order\!\big(\!\operatorname{lclm}(A,B)\big)\ \leq\ \order(A)+\order(B).$$
For the inequality on the right, note that the natural $H[\der]$-module morphism 
$$H[\der]\ \to\ \big(H[\der]/H[\der]A\big)\times \big(H[\der]/H[\der]B\big)$$ has kernel $H[\der]\operatorname{lclm}(A,B)$, and
 for  $D\in H[\der]^{\ne}$,
the $H$-linear space $H[\der]/K[\der]D$ has dimension $\order D$. 

\medskip
\noindent
We now assume that $-1$ is not a square in $H$;
then we have a differential field extension $H[\imag]$ where $\imag^2=-1$.
The automorphism $a+b\imag\mapsto \bar{a+b\imag}:= a-b\imag$~(${a,b\in H}$) of the differential field $H[\imag]$ extends uniquely to an automorphism
$A\mapsto \bar{A}$ of the ring $H[\imag][\der]$ with $\bar{\der}=\der$. 
Let $A\in H[\imag][\der]$; then
$\bar{A}=A\Longleftrightarrow A\in H[\der]$. 
Hence if $A\neq 0$ is monic, then~$L:=\operatorname{lclm}(A,\overline{A})\in H[\der]$ and thus
$L= BA = \overline{B}\,\overline{A}$ where $B\in H[\imag][\der]$.

\begin{exampleNumbered}\label{ex:lclm compl conj}
Let $A=\der-a$ where $a\in H[\imag]$. If $a\in H$, then $\operatorname{lclm}(A,\overline{A})=A$, and if~$a\notin H$, then $\operatorname{lclm}(A,\overline{A})=(\der-b)(\der-a)=(\der-\overline{b})(\der-\overline{a})$ where $b\in H[\imag]\setminus H$.
\end{exampleNumbered}

\noindent
Let now $F$ be a differential field extension of $H$ in which $-1$ is not a square; we assume that $\imag$ is an element of a differential ring extension of $F$.  

{\sloppy
\begin{lemma}\label{lem:lclm compl conj}
Let $A\in H[\imag][\der]^{\neq}$ be monic, $b\in H[\imag]$,   $f\in F[\imag]$  such that~${A(f)=b}$.
Let $B\in H[\imag][\der]$ be such that $L:=\operatorname{lclm}(A,\overline{A})=BA$. Then~${L(f)=B(b)}$ and hence~$L\big(\!\Re(f)\big)=\Re\!\big(B(b)\big)$ and  $L\big(\!\Im(f)\big)=\Im\!\big(B(b)\big)$.
\end{lemma}
}

\noindent
In \cite{ADH5} we shall need a slight extension of this lemma:

\begin{remarkNumbered}\label{rem:lclm compl conj}
Let $F$ be a differential ring extension of $H$ in which $-1$ is not a square and let $\imag$ be an element of a commutative ring extension of $F$ such that~${\imag^2=-1}$ and the $F$-algebra $F[\imag]=F +  F\imag$ is a free
$F$-module with basis $1$,~$\imag$. For~$f=g+h\imag\in F[\imag]$ with $g,h\in F$ we set $\Re(f):= g$ and $\Im(f):= h$. We make $F[\imag]$ into a differential ring extension of~$F$ in the only way possible (which has $\imag'=0$). Then Lemma~\ref{lem:lclm compl conj} goes through. 
\end{remarkNumbered}

\subsection*{Estimates for derivatives} 
{\it In this subsection $K$ is an asymptotic differential field with small derivation.}\/
We  also fix
$\fm\in K^\times$ with~$\fm\prec 1$. 
%Recall from [ADH, 4.2]  that for $P\in K\{Y\}^{\neq}$ the {\it multiplicity of~$P$ at~$0$}\/ is $\val(P)=\min\{d\in\N:P_d\neq 0\}$, where $P_d$ denotes the homogeneous part of degree $d$ of $P$.
Here is a useful bound:

\begin{lemma} \label{lem:diff operator at small elt}  
Let $r\in\N$ and $y\in K$ satisfy $y\prec \fm^{r+m}\prec 1$.
Then $P(y) \prec \fm^{m\mu} P$ for all  $P\in K\{Y\}^{\neq}$   of order at most~$r$ with~$\mu=\val(P)\geq 1$.
\end{lemma}
\begin{proof} Note that $0\ne \fm \prec 1$ and $r+m\ge 1$. Hence 
$$y'\ \prec\ (\fm^{r+m})'\ =\ (r+m)\fm^{r+m-1}\fm'\ \prec\ \fm^{r-1+m},$$ so by induction $y^{(i)}\prec \fm^{r-i+m}$ for $i=0,\dots,r$.
Hence $y^{\i} \prec \fm^{(r+m)\abs{\i}-\dabs{\i}} \preceq \fm^{m\abs{\i}}$
for nonzero $\i=(i_0,\dots,i_r)\in\N^{1+r}$, which yields the lemma.
\end{proof}

\begin{cor}\label{cor:kth der of f}
If   $f\in K$ and $f\prec \fm^n$, then  $f^{(k)} \prec \fm^{n-k}$ for $k=0,\dots,n$.
\end{cor}
\begin{proof}
This is a special case of 
 Lemma~\ref{lem:diff operator at small elt}.
 \end{proof}
 
 \begin{cor}\label{cor:kth der of f, preceq}
Let   $f\in K^\times$ and $n\ge 1$ be such that $f\preceq \fm^n$. Then~${f^{(k)} \prec  \fm^{n-k}}$ for $k=1,\dots,n$.
\end{cor}
\begin{proof}
Note that   $\fm^n\neq 0$, so $f'\preceq (\fm^n)'=n\fm^{n-1}\fm'\prec\fm^{n-1}$ [ADH, 9.1.3].
Now apply Corollary~\ref{cor:kth der of f} with $f'$, $n-1$ in place of $f$, $n$.
\end{proof}

 %In the next two corollaries, $l\in \Z$, $\xi=\phi'$, and $\ex^\phi$ denotes a unit of a differential ring extension of $K$ with multiplicative inverse $\ex^{-\phi}$ and such that $(\ex^\phi)'=\phi'\ex^\phi$. 
 
%\begin{cor}\label{cor:xi 1}
%$(\xi^l\ex^\phi)^{(n)}\ =\ \xi^{l+n}(1+\varepsilon)\ex^\phi$ where $\varepsilon\in K$, $\varepsilon\prec 1$.
%\end{cor}
%\begin{proof}
%By Lemma~\ref{lem:xi zeta}(i) we have $l\zeta+\xi\sim\xi\succ 1$. Now use $(\xi^l\ex^\phi)^{(n)}/(\xi^l\ex^\phi)=R_n(l\zeta+\xi)$ for $R_n=\Ric(Y^{(n)})$ in combination with~[ADH, 11.1.5].
%\end{proof}

%\noindent
%Applying the corollary above with $\phi$, $\xi$ replaced by~$-\phi$,~$-\xi$, respectively, we obtain:

%\begin{cor}\label{cor:xi 2}
%$(\xi^l\ex^{-\phi})^{(n)}\ =\ (-1)^n\xi^{l+n}(1+\delta)\ex^{-\phi}$  where $\delta\in K$, $\delta\prec 1$.
%\end{cor}

\subsection*{Estimates for Riccati transforms}
{\it In this subsection $K$ is a valued differential field with small derivation.}\/
Recall from [ADH, 5.8] that for a homogeneous differential polynomial $P\in K\{Y\}$ of degree $d\in\N$ the {\it Riccati transform}\/ $\Ric(P)\in K\{Z\}$ of $P$
satisfies \label{p:Ric}\index{Riccati transform!differential polynomial}\index{differential polynomial!Riccati transform}  $$\Ric(P)(z)=P(y)/y^d\quad \text{ for $y\in K^\times$, $z=y^\dagger$.}$$
We put $R_n:=\Ric(Y^{(n)})\in\Q\{Z\}$, so 
$$R_0=1,\quad R_1=Z,\quad R_2=Z^2+Z',\quad\dots.$$
For $A=a_0+a_1\der+\cdots+a_n\der^n\in K[\der]$ ($a_0,\dots,a_n\in K$) we also let\index{linear differential operator!Riccati transform}\index{Riccati transform!linear differential operator}\label{p:Ric(A)} 
$$\Ric(A):=a_0R_0+a_1R_1+\cdots+a_nR_n\in K\{Z\}.$$ 
For later use we prove variants of~[ADH, 11.1.5].

\begin{lemma}\label{Riccatipower}
If $z\in K^{\succ 1}$, then $R_n(z)=z^n(1+ \epsilon)$ with 
$v\epsilon\ge v(z^{-1})+ o(vz)>0$.  
\end{lemma}
\begin{proof} This is clear for $n=0$ and $n=1$. Suppose $z\succ 1$, $n\ge 1$, and
$R_n(z)=z^n(1+ \epsilon)$ with $\epsilon$ as in the lemma. As in the proof of [ADH, 11.1.5],
$$ R_{n+1}(z)\ =\  z^{n+1}\left(1 + \epsilon +    n\frac{z^\dagger}{z}(1+\epsilon) +       
\frac{\epsilon'}{z}\right).$$
Now $v(z^\dagger)\ge o(vz)$: this is obvious if $z^\dagger\preceq 1$, and follows from~$\triangledown(\gamma)=o(\gamma)$ for~$\gamma\ne 0$ if $z^\dagger\succ 1$ [ADH, 6.4.1(iii)]. 
This gives the desired result in view of  $\epsilon'\prec 1$. 
\end{proof}

\begin{lemma}\label{Riccatipower+} Suppose $\der(\mathcal O)\subseteq\smallo$. If 
$z\in K^{\succeq 1}$, then $R_n(z)=z^n(1+ \epsilon)$ with~$\epsilon \prec 1$.
\end{lemma}
 
\begin{proof}
The case $z\succ 1$ follows from Lemma~\ref{Riccatipower}. For $z\asymp 1$, proceed as in the proof of that lemma, using $\der(\mathcal O)\subseteq\smallo$.
\end{proof}

\noindent
By [ADH, 9.1.3 (iv)] the condition $\der(\mathcal O)\subseteq\smallo$ is satisfied
if $K$ is $\d$-valued, or  asymptotic   with 
$\Psi\cap \Gamma^{>}\ne \emptyset$.
The following observation is not used later: 

\begin{lemma}\label{lem:Riccati bd flat}
Suppose $K$ is asymptotic, and $z\in K$ with $0\neq z\preceq z'\prec 1$. Then~$R_n(z)\sim z^{(n-1)}$ for $n\geq 1$.
\end{lemma}
\begin{proof}
Induction on $n$ gives $z\preceq z'\preceq\cdots \preceq z^{(n)}\prec 1$     
 for all $n$.
We now show~$R_n(z)\sim z^{(n-1)}$ for $n\geq 1$, also by induction. The case $n=1$ is clear from $R_1=Z$,
so suppose~${n\geq 1}$ and $R_n(z)\sim z^{(n-1)}$. Then
$$R_{n+1}(z)\ =\ zR_n(z)+R_n(z)'$$
where $R_n(z)'\sim z^{(n)}$ by [ADH, 9.1.4(ii)] and $zR_n(z)\asymp zz^{(n-1)} \prec z^{(n-1)}\preceq z^{(n)}$.
Hence $R_{n+1}(z)\sim z^{(n)}$.
\end{proof}

\section{The Group of Logarithmic Derivatives} \label{sec:logder}

\noindent
Let $K$ be a differential field.
The map $y\mapsto y^\dagger \colon K^\times\to K$ is a morphism 
from the multiplicative group of $K$ to the additive group of $K$, with kernel~$C^\times$.
Its image 
$$(K^\times)^\dagger\ =\ \big\{y^\dagger:\,y\in K^\times\big\}$$
is an additive  subgroup of  $K$, which we call the {\bf group of logarithmic derivatives}\index{group of logarithmic derivatives}\label{p:Kdagger} of $K$. The morphism $y\mapsto y^\dagger$ induces an isomorphism $K^\times/C^\times\to (K^\times)^\dagger$. To shorten notation, set $0^\dagger:=0$, so $K^\dagger=(K^\times)^\dagger$. 
For $\phi\in K^\times$ we have $\phi(K^\phi)^\dagger=  K^\dagger$.
The group $K^\times$ is divisible iff both $C^\times$ and $K^\dagger$ are divisible. 
If $K$ is algebraically
closed, then $K^\times$ and hence
$K^\dagger$ are divisible, making
$K^\dagger$ a  $\Q$-linear subspace of~$K$.
Likewise, if $K$ is real closed, then
the multiplicative subgroup~$K^{>}$ of $K^\times$ is divisible, so
$K^\dagger=(K^{>})^\dagger$ is a $\Q$-linear subspace of $K$.

\begin{lemma}\label{lem:Kdagger alg closure}
Suppose $K^\dagger$ is divisible, $L$ is a differential field extension of $K$ with~$L^\dagger\cap K=K^\dagger$, and
$M$ is a differential field extension of $L$ and algebraic over~$L$. Then $M^\dagger\cap K=K^\dagger$.
\end{lemma}
\begin{proof}
Let $f\in M^\times$ be such that $f^\dagger\in K$. Then $f^\dagger\in L$, so for $n:=[L(f):L]$, 
$$nf^\dagger\ =\ \operatorname{tr}_{L(f)|L}(f^\dagger)\ =\ \operatorname{N}_{L(f)|L}(f)^\dagger  \in L^\dagger$$ by 
an identity in [ADH, 4.4].
Hence $nf^\dagger\in K^\dagger$, and thus $f^\dagger\in K^\dagger$. 
\end{proof}

\noindent
In particular, if $K^\dagger$ is divisible and $M$ is  a differential field extension of $K$ and algebraic over~$K$, then
$M^\dagger\cap K=K^\dagger$.

\medskip
\noindent
In the next two lemmas~$a,b\in K$; distinguishing  whether or not $a\in K^\dagger$ helps to  describe 
the solutions to the differential equation $y'+ay=b$:

\begin{lemma}\label{0K} Suppose $\der(K) = K$, and let $L$ be differential field extension of $K$ with~$C_L=C$. 
Suppose also $a\in K^\dagger$. Then for some  $y_0\in K^\times$ and
$y_1\in K$,
$$\{y\in L:\, y'+ay=b\}\ =\ \{y\in K:\, y'+ay=b\}\ =\ Cy_0 + y_1.$$
\end{lemma}
\begin{proof} 
Take $y_0\in K^\times$ with $y_0^\dagger=-a$, so $y_0'+ay_0=0$. Twisting $\der+a\in K[\der]$ by~$y_0$ (see~[ADH, p.~243]) transforms the equation $y'+ay=b$ into $z'=y_0^{-1}b$. This gives $y_1\in K$ with $y_1'+ay_1=b$.
Using $C_L=C$, these  $y_0, y_1$ have the desired properties. 
\end{proof} 

\begin{lemma}\label{1K} Let $L$ be a differential field
extension of $K$ with $L^\dagger\cap K=K^\dagger$. Assume
$a\notin K^\dagger$. Then there is at most one $y\in L$ with $y'+ay=b$.
\end{lemma}
\begin{proof} If $y_1$, $y_2$ are distinct solutions in $L$ of the equation
$y'+ay=b$, then we have~$-a=(y_1-y_2)^\dagger\in L^\dagger\cap K=K^\dagger$, contradicting $a\notin K^\dagger$. 
\end{proof} 

\subsection*{Logarithmic derivatives under algebraic closure}
{\it In this subsection $K$ is a differential field.}\/
We describe for real closed $K$ how $K^\dagger$ changes if we pass from~$K$ to its algebraic closure. More generally, suppose the underlying field of $K$ is euclidean; in particular, $-1$ is not a square in~$K$. 
We equip~$K$ with the unique ordering making~$K$ an ordered field.
For $y=a+b\imag \in K[\imag]$ ($a,b\in K$) we let $\abs{y}\in K^{\geq}$ be such that $\abs{y}^2=a^2+b^2$.
Then $y\mapsto \abs{y}\colon K[\imag]\to K^{\geq}$ is an absolute value on $K[\imag]$, i.e., for all $x,y\in K[\imag]$,
$$\abs{x}\ =\ 0\ \Longleftrightarrow\ x=0, \qquad \abs{xy}\ =\ \abs{x}\abs{y}, \qquad \abs{x+y}\ \leq\ \abs{x}+\abs{y}.$$
For $a\in K$ we have $\abs{a}=\max\{a,-a\}$.
We have the subgroup
\[
S\ :=\ \big\{ y\in K[\imag]: \abs{y}=1\big\}\ =\ \big\{ a+b\imag : a,b\in K,\ a^2+b^2=1 \big\}
\]
of the multiplicative group $K[\imag]^\times$.
By an easy computation all elements of $K[\imag]$ are squares in $K[\imag]$; hence $K[\imag]^\dagger$ is $2$-divisible.
The next lemma describes $K[\imag]^\dagger$; it partly generalizes~[ADH, 10.7.8]. 
For $a,b\in K$,   put $\wr(a,b):=ab'-a'b$ [ADH, 4.1].\label{p:wr(a,b)}

\begin{lemma}\label{lem:logder}
We have $K[\imag]^\times = K^>\cdot S$ with $K^>\cap S=\{1\}$, and
$$K[\imag]^\dagger\ =\  K^\dagger \oplus S^\dagger\qquad\text{\textup{(}internal direct sum of subgroups of $K[\imag]^\dagger$\textup{)}.}$$
For $a,b\in K$ with $a+b\imag\in S$ we have
$(a+b\imag)^\dagger = \wr(a,b)\imag$. Thus $K[\imag]^\dagger\cap K=K^\dagger$. 
\end{lemma}
\begin{proof}
Let $y=a+b\imag\in K[\imag]^\times$ ($a,b\in K$), and take $r\in K^>$ with $r^2=a^2+b^2$; then~$y=r\cdot (y/r)$ with $y/r\in S$. Thus $K[\imag]^\times = K^>\cdot S$, and clearly $K^>\cap S=\{1\}$.
Hence~$K[\imag]^\dagger = K^\dagger + S^\dagger$.
Suppose $a\in K^\times$, $s\in S$, and $a^\dagger=s^\dagger$; then $a=cs$ with~${c\in C_{K[\imag]}}$, and $C_{K[\imag]}=C[\imag]$ by [ADH,~4.6.20] and hence $\max\{a,-a\}=\abs{a}=\abs{c}\in C$, so $a\in C$ and thus $a^\dagger=s^\dagger=0$; therefore the sum is direct.
Now if $a,b\in K$ and~$\abs{a+b\imag}=1$, then
\begin{align*}
(a+b\imag)^\dagger\	&=\ (a'+b'\imag)(a-b\imag) \\
					&=\ (aa'+bb')+(ab'-a'b)\imag \\ 
					&=\ \textstyle\frac{1}{2}\big( a^2+b^2 \big)' + (ab'-a'b)\imag\ =\ (ab'-a'b)\imag\ =\ \wr(a,b)\imag. \qedhere
\end{align*}				
\end{proof}

\begin{cor}\label{cor:logder abs value}
For $y\in K[\imag]^\times$ we have $\Re(y^\dagger)=\abs{y}^\dagger$, and the group morphism~$y\mapsto  \Re y^\dagger \colon K[\imag]^\times\to K$ has kernel $C^> S$.
\end{cor}

\noindent
If $K$ is real closed and $\mathcal{O}$ a convex valuation ring of $K$, 
then $\mathcal{O}[\imag]= \mathcal{O}+ \mathcal{O}\imag$ is the unique valuation ring of $K[\imag]$ that lies over $\mathcal{O}$, and so
$S\subseteq\mathcal O[\imag]^\times$, hence
$y\asymp\abs{y}$ for all $y\in K[\imag]^\times$.
Thus by [ADH, 10.5.2(i)] and Corollary~\ref{cor:logder abs value}:

\begin{cor}\label{cor:10.5.2 variant}
If $K$ is a real closed pre-$H$-field, then for all $y,z\in K[\imag]^\times$,
$$y\prec z \quad\Longrightarrow\quad \Re y^\dagger < \Re z^\dagger.$$
\end{cor}

\noindent
We also have a useful decomposition for $S$:

\begin{cor} \label{cor:decomp of S}
Suppose $K$ is a real closed $H$-field. Then $$S\ =\ S_C\cdot \big(S\cap(1+\smallo_{K[\imag]})\big)$$ where 
$S_C := S\cap C[\imag]^\times$ and $S\cap(1+\smallo_{K[\imag]})$ are subgroups of $\mathcal O[\imag]^\times$.  
\end{cor}  
\begin{proof} The inclusion $\supseteq$ is clear. For the reverse inclusion, let $a,b\in K$, $a^2+b^2=1$ and take the unique $c,d\in C$ with $a-c\prec 1$ and $b-d\prec 1$. Then $c^2+d^2=1$ and~$a+b\imag\sim c+d\imag$, and so $(a+b\imag)/(c+d\imag)\in S\cap (1+\smallo_{K[\imag]})$.
\end{proof}

\subsection*{Logarithmic derivatives in asymptotic fields}
{\em Let $K$ be an asymptotic field.}\/
If $K$  is   henselian and $\k:=\res K$, then by [ADH, remark before 3.3.33],
$K^\times$ is divisible iff the groups $\k^\times$ and $\Gamma$ are both divisible.
Recall that in~[ADH,~14.2] we defined the 
$\mathcal O$-submodule \label{p:I(K)}
$$\I(K)\ =\ \{y\in K:\, \text{$y\preceq f'$ for some $f\in\mathcal O$}\}$$ 
of $K$. We have  $\der(\mathcal O)\subseteq \I(K)$, hence $(1+\smallo)^\dagger\subseteq(\mathcal O^\times)^\dagger\subseteq\I(K)$.
One easily verifies:

\begin{lemma}\label{pldv}
Suppose $K$ is pre-$\d$-valued. If $\I(K)\subseteq \der(K)$, then
$\I(K)=\der(\mathcal O)$. If~$\I(K)\subseteq K^\dagger$, then 
$\I(K)= (\mathcal O^\times)^\dagger$, with $\I(K)=(1+\smallo)^\dagger$
 if $K$   is $\d$-valued.
\end{lemma}

\noindent
If $K$ is $\d$-valued or $K$ is pre-$\d$-valued without a gap, then
$$\I(K)\ =\ \{y\in K:\text{$y\preceq f'$ for some $f\in\smallo$}\}.$$ 
For $\phi\in K^\times$ we have $\phi \I(K^\phi)=\I(K)$.
If  $K$ has asymptotic integration and $L$ is an asymptotic extension of~$K$, then $\I(K)=\I(L)\cap K$.
The following is [ADH, 14.2.5]:
 
\begin{lemma}\label{lem:ADH 14.2.5}
If $K$ is $H$-asymptotic, has asymptotic integration, and is $1$-linearly newtonian, then it is $\d$-valued and $\der(\mathcal O) = \I(K) = (1 + \smallo)^\dagger$.
\end{lemma}

\noindent
We now turn our attention to the condition $\I(K)\subseteq K^\dagger$.
If $\I(K)\subseteq K^\dagger$, then also~$\I(K^\phi)\subseteq (K^\phi)^\dagger$ for $\phi\in K^\times$, where 
$$(K^\phi)^\dagger\ :=\ \{\phi^{-1}f'/f:\, f\in K^\times\}\ =\ \phi^{-1}K^\dagger.$$
By [ADH,  Section~9.5 and 10.4.3]: 

\begin{lemma}\label{lem:achieve I(K) subseteq Kdagger} 
Let $K$ be of $H$-type. If $K$ is $\d$-valued, or  pre-$\d$-valued without a gap, then $K$ has an immediate henselian asymp\-to\-tic extension~$L$  with $\I(L)\subseteq L^\dagger$.   
\end{lemma}

\begin{cor}\label{cor:no new LDs}
Suppose $K$ has asymptotic integration. Let $L$ be an  asymptotic field extension of $K$ such that $L^\times = K^\times C_L^\times (1+\smallo_L)$.
Then $L^\dagger=K^\dagger+(1+\smallo_L)^\dagger$, and
if $\I(K)\subseteq K^\dagger$, then  $L^\dagger\cap K=K^\dagger$. 
\end{cor}
\begin{proof}
Let $f\in L^\times$, and 
take $b\in K^\times$, $c\in C_L^\times$, $g\in\smallo_L$
with $f=bc(1+g)$; then~$f^\dagger=b^\dagger+(1+g)^\dagger$, showing $L^\dagger=K^\dagger+(1+\smallo_L)^\dagger$.
Next, suppose $\I(K)\subseteq K^\dagger$, let 
$b$, $c$, $f$, $g$ be as before, and assume $a:=f^\dagger\in K$;
then $$a-b^\dagger \in (1+\smallo_L)^\dagger \cap K\ \subseteq\ \I(L)\cap K\ =\  \I(K)\ \subseteq\ K^\dagger$$
and hence  $a\in K^\dagger$. This shows $L^\dagger\cap K=K^\dagger$.
\end{proof}

\noindent
Two cases where the assumption on $L$ in Corollary~\ref{cor:no new LDs} is satisfied: (1)  $L$ is an immediate asymptotic field extension of $K$,  because then  $L^\times = K^\times ({1+\smallo_L})$;  and (2)
$L$ is a $\d$-valued field extension of $K$ with $\Gamma=\Gamma_L$.

\medskip\noindent
If $F$ is a henselian valued field 
of residue characteristic $0$, then clearly the subgroup~$1+\smallo_F$ of $F^\times$ is divisible. 
Hence, if $K$ and $L$ are as in Corollary~\ref{cor:no new LDs} and in addition $K^\dagger$ is divisible and $L$ is henselian, then $L^\dagger$ is divisible.

\begin{exampleNumbered}\label{ex:Kdagger} 
Let $C$ be a field of characteristic $0$ and $Q$ be a subgroup of $\Q$ with~$1\in Q$. The Hahn field $C(\!( t^{Q} )\!)=C[[x^Q]]$, with $x=t^{-1}$, is given the natural derivation with $c'=0$ for all~$c\in C$ and~$x'=1$: this derivation is defined by 
$$\bigg(\sum_{q\in Q}c_qx^q\bigg)'\ :=\ \sum_{q\in Q} qc_qx^{q-1}\qquad \text{(all $c_q\in C$)}.$$
Then~$C(\!( t^{Q} )\!)$ has constant field $C$, and is $\d$-valued of $H$-type. Thus $K:=C(\!( t^{Q} )\!)$ satisfies $\I(K)\subseteq K^\dagger$
by Lemma~\ref{lem:achieve I(K) subseteq Kdagger}. Hence by Lemma~\ref{pldv},
$$\I(K)\ =\ (1+\smallo)^\dagger\ =\ \big\{f\in K:\, f\prec x^\dagger = t\big\}\ = \ \smallo\, t.$$ 
It follows easily that
$K^\dagger=Q t\oplus\I(K)$ (internal direct sum of subgroups of $K^\dagger$) and   
thus $(K^t)^\dagger=Q\oplus\smallo\subseteq\mathcal O$. In particular, if $Q=\Z$ (so $K=C(\!( t )\!)$), then~$(K^t)^\dagger=\Z\oplus tC[[t]]$. Moreover, if $L:=\operatorname{P}(C)\subseteq C(\!( t^{\Q} )\!)$ is the differential field of Puiseux series over~$C$, then
$(L^t)^\dagger=\Q\oplus\smallo_L$.
\end{exampleNumbered}

\subsection*{The real closed case} {\em In this subsection $H$ is a real closed asymptotic field whose valuation ring $\mathcal{O}$ is convex with respect to the ordering of $H$}. (In later use $H$ is often a Hardy field, which is why we use the letter $H$ here.) The valuation ring of the asymptotic field extension $K=H[\imag]$ of $H$ is then $\mathcal{O}_K=\mathcal{O}+\mathcal{O}\imag$, from which we obtain $\I(K)=\I(H) \oplus \I(H)\imag$. Let 
$$S\ :=\ \big\{y\in K:\, |y|=1\big\}, \qquad
W\ :=\ \big\{\!\wr(a,b):\, a,b\in H,\ a^2+b^2=1\big\},$$ 
so $S$ is a subgroup of $\mathcal{O}_K^\times$ with $S^\dagger = W\imag$ and 
$K^\dagger=H^\dagger\oplus W\imag$ by Lemma~\ref{lem:logder}.
Since~$\der(\mathcal O)\subseteq \I(H)$,  
we have $W\subseteq \I(H)$, and thus: $W = \I(H) \ \Longleftrightarrow\  \I(H)\imag\subseteq K^\dagger$.

\begin{lemma}\label{lem:W and I(F)} The following are equivalent: \begin{enumerate}
\item[\textup{(i)}] $\I(K)\subseteq K^\dagger$;
\item[\textup{(ii)}] $W=\I(H)\subseteq H^\dagger$.
\end{enumerate}
\end{lemma}
\begin{proof}
Assume (i). 
Then $\I(H)\,\imag\subseteq\I(K)\subseteq K^\dagger$, so $W=\I(H)$ by the equivalence preceding the lemma. Also
$\I(H)\subseteq  \I(K)$ and 
$K^\dagger\cap H=H^\dagger$ (by Lemma~\ref{lem:logder}), hence $\I(H)\subseteq H^\dagger$, so
(ii) holds. 
For the converse, assume (ii). Then \[ \I(K)\ =\ \I(H)\oplus\I(H)\imag\ \subseteq\ H^\dagger\oplus W\imag\ =\ K^\dagger.\qedhere \]
\end{proof}

\noindent
Applying now Lemma~\ref{lem:ADH 14.2.5} we obtain:

\begin{cor}\label{cor:logder}
If $H$ is $H$-asymptotic and has asymptotic integration, and~$K$ is $1$-linearly newtonian, then $K$ is $\d$-valued and $\I(K)\subseteq K^\dagger$; in particular, $W=\I(H)$.
\end{cor}

{\sloppy
\begin{cor}\label{cor:logderset ext} 
Suppose $H$  has asymptotic integration and  $W=\I(H)$. 
Let $F$ be a real closed asymptotic extension of $H$ whose valuation ring is convex.  
Then $$F[\imag]^\dagger\cap K\ =\ (F^\dagger\cap H)\oplus\I(H)\imag.$$  If in addition $H^\dagger=H$, then $F[\imag]^\dagger\cap K=H\oplus\I(H)\imag=K^\dagger$.
\end{cor}
\begin{proof} 
We have 
$$F^\dagger\cap H\subseteq F[\imag]^\dagger\cap K\quad\text{ and }\quad
\I(H)\imag = W\imag \subseteq K^\dagger\cap H\imag \subseteq F[\imag]^\dagger\cap K,$$
so~$(F^\dagger\cap H)\oplus\I(H)\imag \subseteq F[\imag]^\dagger\cap K$. For the reverse inclusion, 
$F[\imag]^\dagger=F^\dagger\oplus W_F\imag$, with 
$$W_F\ :=\ \big\{\!\operatorname{wr}(a,b):\, a,b\in F,\ a^2+b^2=1\big\}\ \subseteq\ \I(F),$$ 
hence 
\begin{align*}
F[\imag]^\dagger\cap K &\, =\, (F^\dagger\cap H)\oplus (W_F\cap H)\imag \\
&\, \subseteq\, (F^\dagger\cap H)\oplus \big(\!\I(F)\cap H\big)\imag  \, =\, (F^\dagger\cap H)\oplus\I(H)\imag,
\end{align*} 
using $\I(F)\cap H=\I(H)$, a consequence of $H$ having asymptotic integration. 
If~${H^\dagger = H}$ then clearly $F^\dagger\cap H=H$, hence $F[\imag]^\dagger\cap K=K^\dagger$.
\end{proof} }

\subsection*{Trigonometric closure\astr} {\it In this subsection $H$ is a real closed $H$-field.}\/ Let $\O$ be its valuation ring and $\smallo$ the maximal ideal of $\O$.
The algebraic closure $K= H[\imag]$ of $H$ is a $\d$-valued $H$-asymptotic extension with valuation ring $\O_K=\O+\O\imag$.
We have the ``complex conjugation'' automorphism $z=a+b\imag\mapsto \bar{z}=a-b\imag$ ($a,b\in H$) of the valued differential field $K$. For such $z$, $a$, $b$ we have 
 $$ |z|\ =\ \sqrt{z\bar{z}}\ =\ \sqrt{a^2+b^2}\ \in\  H^{\ge}.$$ 
 
 \begin{lemma}\label{lemtri1} Suppose $\theta\in H$ and $\theta'\imag\in K^\dagger$. Then $\theta'\in \der(\smallo)$, and there is a unique~$y\sim 1$ in $K$ such that $y^\dagger=\theta'\imag$. For this $y$ we have $|y|=1$, so $y^{-1}=\bar{y}$.
\end{lemma} 
\begin{proof} From $\theta'\imag \in K^\dagger$ we get $\theta'\in W\subseteq \I(H)$, so $\theta\preceq 1$, hence
$\theta'\in \der(\O)=\der(\smallo)$.  Let~$z\in K^\times$ and $z^\dagger=\theta'\imag$. Then  $\Re z^\dagger=0$, so by Corollaries~\ref{cor:logder abs value} and \ref{cor:decomp of S} we have~$z=cy$ with $c\in C_K^\times$ and $y\in S\cap(1+\smallo_K)$ where $S=\{a\in K:\ |a|=1\}$. Hence~$y\sim 1$, $|y|=1$, and $y^\dagger=\theta'\imag$. If also $y_1\in  K$ and $y_1\sim 1$, $y_1^\dagger=\theta'\imag$, then~$y_1=c_1y$ with~$c_1\in C_K^\times$, so $c_1=1$ in view of $y\sim y_1$. 
\end{proof} 

\noindent
By [ADH, 10.4.3], if $y$ in an $H$-asymptotic extension $L$ of $K$ satisfies
$y\sim 1$ and~${y^\dagger\in \der(\smallo_K)}$, then the asymptotic field $K(y)\subseteq L$ is an immediate extension of $K$, and so is any algebraic asymptotic extension of $K(y)$.

\medskip\noindent
Call $H$ {\bf trigonometrically closed} if
for all $\theta\prec 1$ in $H$ there is a (necessarily unique)  $y\in K$ such that $y\sim 1$ and $y^\dagger=\theta'\imag$.\index{trigonometric!closed}\index{closed!trigonometrically}\index{H-field@$H$-field!trigonometrically closed} (By convention ``trigonometrically closed'' includes ``real closed''.) For such $\theta$ and $y$ we think of
$y$ as $\ex^{\imag\theta}$ and accordingly of the elements $\frac{y+\bar{y}}{2}=\frac{y+y^{-1}}{2}$ and  $\frac{y-\bar{y}}{2\imag}=\frac{y-y^{-1}}{2\imag}$ of $H$ as $\cos \theta$ and $\sin \theta$; this explains the terminology. By Lemma~\ref{lemtri1} the restrictions $\theta\prec 1$ and $y\sim 1$ are harmless. Our aim in this subsection is to construct a canonical trigonometric closure
of $H$. 

Our interest in this notion comes from the condition $\I(K)\subseteq K^\dagger$, which   appears as a natural hypothesis
at many points in Chapter~\ref{part:dents in H-fields}, especially in Section~\ref{sec:ultimate}).
Note that if  $\I(K)\subseteq K^\dagger$, then $H$ is trigonometrically closed. As a partial converse, if~$\I(H)\subseteq H^\dagger\cap \der(H)$ and $H$ is trigonometrically closed, then $\I(K)\subseteq K^\dagger$; this is an
easy consequence of $\I(K)=\I(H)+\I(H)\imag$. Thus for Liouville closed $H$ we have:
$$H \text{ is trigonometrically closed}\ \Longleftrightarrow\  \I(K)\subseteq K^\dagger.$$ 
Note also that for trigonometrically closed $H$ there is no $y$ in any $H$-asymptotic extension of $K$ such that
$y\notin K$, $y\sim 1$, and $y^\dagger\in \der(\smallo)\imag$. 
If $H$ is Schwarz closed, then $H$ is trigonometrically closed by the next lemma:

\begin{lemma}\label{lem:sc=>tc}
Suppose $H$ is Liouville closed and $\omega(H)$ is downward closed. Then~$H$ is trigonometrically closed.
\end{lemma}
\begin{proof}
Let $0\ne \theta\prec 1$ in $H$. By Lemma~\ref{lemtri1} it suffices to show that then~${\theta'\imag\in K^\dagger}$.
Note that $h:=\theta'\in\I(H)^{\neq}$; we arrange $h>0$. Now
$$f\ :=\ \omega(-h^\dagger)+4h^2\ =\ \sigma(2h), \qquad 2h\in H^>\cap\I(H),$$ 
hence $2h\in H^>\setminus\Upg(H)$
by [ADH, 11.8.19]. So $f\in\omega(H)^\downarrow=\omega(H)$ by [ADH, 11.8.31],  and thus
$\dim_{C_H} \ker 4\der^2+f\ge 1$ by [ADH, p.~ 258].   
Put $A:=\der^2-h^\dagger \der+h^2\in H[\der]$.
The isomorphism $y\mapsto y\sqrt{h}\colon\ker({4\der^2+f})\to\ker A$ of $C_H$-linear spaces~[ADH, 5.1.13] then yields an element of $\ker^{\neq} A$ that for suggestiveness we denote by $\cos \theta$. 
Put~$\sin\theta:=-(\cos\theta)'/h$.
Then
\begin{align*}
(\sin\theta)'\	&=\  -(\cos\theta)''/h+(\cos\theta)'h^\dagger/h\\
	&=\  \big( {-h^\dagger(\cos\theta)'+h^2\cos\theta } \big)/h+(\cos\theta)'h^\dagger/h\ =\ h\cos\theta
\end{align*}
and thus $y^\dagger=\theta'\imag$ for $y:=\cos\theta+\imag\sin\theta\in K^\times$.
\end{proof}

\noindent
If $H$ is $H$-closed, then $H$ is Schwarz closed by [ADH, 14.2.20], and thus trigonometrically closed.  Using also Lemma~\ref{lem:W and I(F)}  and remarks preceding it  this yields:

\begin{cor}\label{cor:sc=>tc}
If $H$ is   $H$-closed,  then $\I(K) \subseteq K^\dagger = H\oplus\I(H)\imag$.
\end{cor}

\noindent
Suppose now that $H$ is {\it not}\/ trigonometrically closed; so we have $\theta\prec 1$ in $H$ with~$\theta'\imag\notin K^\dagger$. 
Then [ADH, 10.4.3] provides an immediate asymptotic extension~$K(y)$ of $K$ with $y\sim 1$ and
$y^\dagger=\theta'\imag$. To simplify notation and for suggestiveness we set
$$\cos \theta\ :=\ \frac{y+y^{-1}}{2}, \qquad \sin \theta\ :=\ \frac{y-y^{-1}}{2\imag},$$ so $y=\cos\theta+ \imag\sin \theta$ and
$(\cos\theta)^2+ (\sin\theta)^2=1$.  Moreover $(\cos \theta)'=-\theta'\sin \theta$ and~$(\sin \theta)'=\theta'\cos \theta$. It follows that
$H^+:=H(\cos\theta, \sin\theta)$ is a differential subfield of~$K(y)$ with $K(y)=H^+[\imag]$, and  thus $H^+$, as a valued differential
subfield of $H(y)$, is an asymptotic extension of $H$. 

\begin{lemma}\label{lemtri2}$H^+$ is an immediate extension of $H$.
\end{lemma}
\begin{proof} Since $(y^{-1})^\dagger=-\theta'\imag$, the uniqueness property stated in [ADH, 10.4.3] allows us to extend the complex conjugation automorphism
of $K$ (which is the identity on~$H$ and sends $\imag$ to $-\imag$) to 
 an automorphism $\sigma$ of the valued differential field $K(y)$ such that $\sigma(y)=y^{-1}$.
Then $\sigma(\cos \theta)=\cos\theta$ and $\sigma(\sin \theta)=\sin\theta$, so $H^+= \text{Fix}(\sigma)$.
Let~$\k$ be the residue field of $H$; so $\k[\res\imag]$ is the residue field of $K$ and of its immediate extension $K(y)$.
Now $\sigma(\O_{K(y)})=\O_{K(y)}$, so $\sigma$ induces an automorphism of this residue field $\k[\res\imag]$
which is the identity on $\k$ and sends $\res\imag$ to $-\res \imag$. Hence~$\res\imag$ does not lie in the residue field of $H^+$,   so this residue field is just~$\k$. 
\end{proof} 

\noindent
Equip $H^+$ with the unique field ordering making it an ordered field extension of $H$ in which $\O_{H^+}$ is convex; see 
[ADH, 10.5.8].  Then $H^+$ is an $H$-field, and its real closure is an immediate real closed $H$-field extension of $H$.

\begin{lemma}\label{lemtri3} The $H$-field $H^+$ embeds uniquely over $H$ into any trigonometrically
closed $H$-field extension of $H$. 
\end{lemma} 
\begin{proof} Let $H^*$ be a trigonometrically closed $H$-field extension of $H$. Take the unique $z\sim 1$ in $H^*$ such that $z^\dagger=\theta'\imag$. Then any $H$-field embedding $H^+\to H^*$ over $H$ extends to a valued differential field embedding $H^+[\imag]=K(y)\to H^*[\imag]$ sending $\imag\in K$ to $\imag\in H^*[\imag]$, and this extension
must send $y$ to $z$. Hence there is at most one $H$-field embedding $H^+\to H^*$ over $H$. 
For the existence of such an embedding, the uniqueness properties from [ADH, 10.4.3] yield a valued differential field embedding $K(y)\to H^*[\imag]$ over $H$ sending $\imag\in K$ to $\imag\in H^*[\imag]$ and $y$ to $z$. This embedding maps $H^+$ into $H^*$. The uniqueness property of the ordering
on $H^+$ shows that this embedding restricts to an $H$-field embedding $H^+\to H^*$. 
\end{proof}

\noindent
By iterating the extension step that leads from $H$ to $H^{+}$, alternating it with taking real closures, and taking unions at
limit stages we obtain:

\begin{prop}\label{protrig} $H$ has a trigonometrically closed $H$-field extension $H^{\trig}$
that embeds uniquely over $H$ into any trigonometrically closed $H$-field extension of $H$. 
\end{prop} 

\noindent
This is an easy consequence of Lemma~\ref{lemtri3}. Note that the universal property stated in Proposition~\ref{protrig} determines $H^{\trig}$ up-to-unique-isomorphism of $H$-fields over $H$. We refer to such $H^{\trig}$ as the {\bf trigonometric closure} of $H$.\index{trigonometric!closure}\index{closure!trigonometric}\index{H-field@$H$-field!trigonometric closure}\label{p:Htrig} Note that~$H^{\operatorname{trig}}$ is an immediate extension of $H$, by Lemma~\ref{lemtri2},
and that $H^{\operatorname{trig}}[\imag]$ is a Liouville extension of $K$ and thus of $H$.

\medskip\noindent
A {\bf trigonometric  extension\/}\index{extension!trigonometric}\index{trigonometric!extension} of $H$ is a real closed $H$-field extension $E$ of $H$ such that for all $a\in E$ there are
real closed $H$-subfields $H_0\subseteq H_1\subseteq \cdots \subseteq H_n$ of $E$ such that \begin{enumerate}
\item $H_0=H$ and $a\in H_n$;
\item for $j=0,\dots, n-1$ there are $\theta_j\in H_j$ and $y_j\in H_{j+1}[\imag]\subseteq E[\imag]$ such that~$y_j\sim 1$, $\theta_j'\imag=y_j^\dagger$, and $H_{j+1}[\imag]$ is algebraic over $H_j[\imag](y_j)$.
\end{enumerate}
If $E$ is a trigonometric extension of $H$, then $E$ is an immediate extension of $H$ and $E[\imag]$ is an immediate Liouville extension of $K$ and thus of $H$.
The next lemma states some further easy consequences of the definition above: 

\begin{lemma} If $E$ is a trigonometric extension of $H$, then $E$  is a trigonometric extension of any 
real closed $H$-subfield $F\supseteq H$ of $E$.
If $H$ is trigonometrically closed, then $H$ has no proper trigonometric extension. 
\end{lemma} 

\noindent
Induction on $m$ shows that if $E$ is a trigonometric extension of $H$, then for any $a_1,\dots, a_m\in E$ there are real closed $H$-subfields
$H_0\subseteq H_1\subseteq \cdots \subseteq H_n$ of $E$ such that $H_0=H$, $a_1,\dots, a_m\in H_n$ and (2) above holds. This helps in proving:

\begin{cor} A trigonometric extension of a trigonometric extension of $H$ is a trigonometric extension of $H$, and $H^{\trig}$ is a trigonometric extension of $H$. 
\end{cor}

\subsection*{Asymptotic fields of Hardy type} 
Let $(\Gamma,\psi)$ be an asymptotic couple, $\Psi:=\psi(\Gamma^{\neq})$, and let~$\gamma$,~$\delta$ range over $\Gamma$.
Recall that
$[\gamma]$
denotes the archimedean class of~$\gamma$ [ADH, 2.4].\label{p:[gamma]}
Following  \cite[Section~3]{Rosenlicht81} we say that   $(\Gamma,\psi)$ is of {\bf Hardy type} if 
for all $\gamma,\delta\neq 0$ we have
$[\gamma] \leq [\delta]  \Longleftrightarrow  \psi(\gamma)\geq\psi(\delta)$. \index{Hardy field!asymptotic couple}\index{Hardy type} \index{asymptotic couple!of Hardy type}
Note that then $(\Gamma,\psi)$ is of $H$-type, and~$\psi$ induces an order-reversing bijection~$[\Gamma^{\neq}]\to\Psi$.
If $\Gamma$ is archimedean, then $(\Gamma,\psi)$ is of Hardy type. If  $(\Gamma,\psi)$  is of Hardy type, then so is
 $(\Gamma,\psi+\delta)$ for each~$\delta$. 
We also say that an asymptotic field is of Hardy type if its asymptotic couple is. Every asymptotic subfield 
and every compositional conjugate of 
an asymptotic field of Hardy type is also of Hardy type. 
Moreover, every Hardy field is of Hardy type~[ADH, 9.1.11].
Let now~$\Delta$ be a convex  subgroup of $\Gamma$. Note that 
$\Delta$ contains the archimedean class~$[\delta]$ of each~$\delta\in\Delta$. Hence, if
 $\delta\in\Delta^{\neq}$ and $\gamma\notin\Delta$, then  $[\delta] < [\gamma]$ and thus:

\begin{lemma}\label{lem:Hardy type}
If $(\Gamma,\psi)$ is of Hardy type and $\gamma\notin\Delta$, $\delta\in\Delta^{\neq}$,  then
$\psi(\gamma)<\psi(\delta)$.
\end{lemma}

\begin{cor}\label{cor:psi(gamma)<0} 
Suppose $(\Gamma,\psi)$ is of Hardy type with small derivation, $\gamma,\delta\neq 0$, $\psi(\delta) \leq 0$, and $[\gamma']>[\delta]$.  Then $\psi(\gamma)<\psi(\delta)$. 
\end{cor}
\begin{proof}
Let $\Delta$ be the smallest convex subgroup of $\Gamma$ with $\delta\in\Delta$; then $\gamma'\notin\Delta$, and~$\psi(\delta)\in\Delta$
by [ADH, 9.2.10(iv)]. Thus $\gamma\notin\Delta$ by [ADH, 9.2.25].
\end{proof}

\noindent
In \cite[Section~7]{AvdD3} we say that an $H$-field $H$ is {\it closed under powers}\/ if for all $c\in C$ and~$f\in H^\times$ there is a $y\in H^\times$
with $y^\dagger=cf^\dagger$. (Think of $y$ as $f^c$.) \index{closed!under powers} Thus if $H$ is Liouville closed, then $H$ is closed under powers.
{\it In the rest of this subsection we let~$H$ be an $H$-field closed under powers, with asymptotic couple $(\Gamma,\psi)$ and constant field~$C$.} 
We recall some basic facts from \cite[Section~7]{AvdD3}.
First, we can make the value group $\Gamma$ into
an ordered vector space over the constant field $C$:

\begin{lemma}\label{OrderedVectorSpace}
For all $c\in C$ and $\gamma=vf$ with $f\in H^\times$ and each $y\in H^{\times}$
with~$y^\dagger = cf^\dagger$, the element $vy\in \Gamma$ only depends on 
$(c,\gamma)$ \textup{(}not on the choice of $f$ and $y$\textup{)}, and is denoted by $c\cdot \gamma$. The scalar multiplication
$(c,\gamma) \mapsto c \cdot\gamma  \colon C\times \Gamma \to \Gamma$
makes~$\Gamma$ into an ordered vector space over the ordered field $C$.
\end{lemma}

\noindent
Let $G$ be an ordered vector space over the ordered field $C$. From [ADH, 2.4] recall that
the $C$-archimedean class  of $a\in G$ is defined
as  $$[a]_C := \big\{ b\in G:\, 
\textstyle\frac{1}{c}|a|\leq |b|\leq c |a| \text{ for some }c\in C^{>}\big\}.$$
Thus if $C=\Q$, then $[a]_\Q$ is just the archimedean class $[a]$ of $a\in G$.
Moreover, if $C^*$ is an ordered subfield of $C$, then $[a]_{C^*}\subseteq [a]_C$ for each $a\in G$,
with equality if $C^*$ is cofinal in $C$.
Hence  if
$C$ is archimedean, then $[a]=[a]_C$ for all $a\in G$.
Put~$[G]_C:=\bigl\{[a]_C:a\in G\bigr\}$ and linearly order $[G]_C$ by
$$[a]_C<[b]_C \quad :\Longleftrightarrow\quad
\text{$[a]_C\neq [b]_C$ and $|a|<|b|$.}$$ 
Thus $[G]_C$ has smallest element $[0]_C=\{0\}$. We also set $[G^{\neq}]_C:=[G]_C\setminus\big\{[0]_C\big\}$.
From \cite[Proposition~7.5]{AvdD3} we have: 

\begin{prop}\label{prop:Hahn space property}
For all $\gamma,\delta\neq 0$ we have
$$[\gamma]_C \leq [\delta]_C \quad\Longleftrightarrow
\quad \psi(\gamma)\geq\psi(\delta).$$ 
Hence $\psi$ induces an order-reversing bijection $[\Gamma^{\neq}]_C\to\Psi=\psi(\Gamma^{\neq})$.
\end{prop}

\noindent
Proposition~\ref{prop:Hahn space property} yields:

\begin{cor}\label{cor:Hardy type arch C}
$H$ is of Hardy type $\Longleftrightarrow$  $[\gamma]=[\gamma]_C$ for all $\gamma$.
Hence if   $C$ is ar\-chi\-me\-dean, then $H$ is of Hardy type; if $\Gamma\neq\{0\}$, then the converse also holds.
\end{cor}

\section{\texorpdfstring{$\upl$}{Lambda}-freeness and \texorpdfstring{$\upo$}{omega}-freeness}\label{sec:uplupo-freeness} 

\noindent
This section contains  preservation results for  the important properties of {\it $\upl$-freeness}\/ and {\it $\upo$-freeness}\/ from [ADH].  Let  $K$ be an ungrounded $H$-asymptotic field such that~${\Gamma\ne \{0\}}$, and
as in [ADH, 11.5],  fix a logarithmic sequence~$(\ell_\rho)$ for $K$ and define the pc-sequences~$(\upl_\rho)=(-\ell_\rho^{\dagger\dagger})$ and~$(\upo_\rho)=\big(\omega(\upl_\rho)\big)$ in~$K$, where~$\omega(z):=-2z'-z^2$.  
Recall that~$K$ is  {\it $\upl$-free}\/ iff $(\upl_\rho)$ does not have a pseudolimit in $K$, and~$K$ is   {\it $\upo$-free}\/ iff~$(\upo_\rho)$ does not have a pseudolimit in~$K$.
 If $K$ is $\upo$-free, then~$K$ is $\upl$-free.
We refer to~[ADH, 11.6, 11.7] for this and other basic facts about $\upl$-freeness and $\upo$-freeness used below.  As in [ADH], $L$ being $\upl$-free or $\upo$-free includes $L$ being an ungrounded $H$-asymptotic field with $\Gamma_L\ne \{0\}$.

\subsection*{Preserving $\upl$-freeness and $\upo$-freeness}
{\it In this subsection $K$ is an ungrounded $H$-asymptotic field with $\Gamma\ne \{0\}$, and~$(\ell_\rho)$, $(\upl_\rho)$, 
$(\upo_\rho)$ are as above.}\/
If $K$ has  a $\upl$-free $H$-asymptotic field extension $L$ such that $\Gamma^<$ is cofinal in $\Gamma_L^<$, then $K$ is
$\upl$-free, and similarly with ``$\upo$-free'' in place of ``$\upl$-free'' [ADH, remarks after 11.6.4, 11.7.19].
The property of $\upo$-freeness is very robust; indeed, by~[ADH, 13.6.1]:

\begin{theorem}\label{thm:ADH 13.6.1}
If $K$ is $\upo$-free and $L$ is a pre-$\d$-valued $\d$-algebraic $H$-asymptotic  extension  of $K$, then $L$  is $\upo$-free and $\Gamma^<$ is cofinal in $\Gamma_L^<$.
\end{theorem}

\noindent
In contrast, $\upl$-freeness is more delicate:  Theorem~\ref{thm:ADH 13.6.1} fails with ``$\upl$-free'' in place of ``$\upo$-free'', as the next example shows.

\begin{exampleNumbered}\label{ex:rat as int and cofinality} 
The $H$-field $K=\R\langle\upo\rangle$ from [ADH, 13.9.1] is $\upl$-free, but its $H$-field extension
$L=\R\langle \upl\rangle$ is not, and this extension is $\d$-algebraic: $2\upl'+\upl^2+\upo=0$.
\end{exampleNumbered}

\noindent
In the rest of this subsection we consider cases where parts of Theorem~\ref{thm:ADH 13.6.1} do hold.
Recall from [ADH, 11.6.8]  that if~$K$ is $\upl$-free, then~$K$ has (rational) asymptotic integration, and $K$
is $\upl$-free iff its algebraic closure is $\upl$-free.
Moreover, $\upl$-freeness is preserved under adjunction of constants: 

\begin{prop}\label{prop:const field ext}
Suppose $K$ is $\upl$-free and $L=K(D)$ is an $H$-asymptotic extension of $K$ with  $D\supseteq C$ a subfield of $C_L$. 
Then $L$ is $\upl$-free with $\Gamma_L=\Gamma$.
\end{prop}

\noindent
We are going  to deduce this from the next three lemmas. 
Recall that $K$ is pre-$\d$-valued, by  [ADH, 10.1.3]. Let $\operatorname{dv}(K)$ be the $\d$-valued hull
of $K$ (see [ADH, 10.3]).

\begin{lemma}\label{lem:Gehret}
Suppose $K$ is $\upl$-free. Then  $L:=\operatorname{dv}(K)$ is $\upl$-free and $\Gamma_L=\Gamma$.
\end{lemma}
\begin{proof}
The first statement is  \cite[Theorem~10.2]{Gehret},
and the second statement follows from [ADH, 10.3.2(i)].
\end{proof}

\noindent
If $L=K(D)$ is a differential  field extension of $K$
with $D\supseteq C$ a subfield of $C_L$, then~$D=C_L$, and $K$ and $D$ are linearly disjoint over $C$ [ADH, 4.6.20].
If $K$ is $\d$-valued and $L=K(D)$ is an $H$-asymptotic extension of $K$ with $D\supseteq C$ a subfield of $C_L$,
then $L$ is $\d$-valued  and~$\Gamma_L=\Gamma$  [ADH, 10.5.15].

\begin{lemma}\label{lem:const field ext}
Suppose $K$ is $\d$-valued and $\upl$-free, and $L=K(D)$ is an $H$-asymptotic extension of $K$ with $D\supseteq C$ a subfield of $C_L$. Then $L$ is $\upl$-free.
\end{lemma}
\begin{proof} First, $(\upl_{\rho})$ is of transcendental type over $K$: otherwise, [ADH, 3.2.7] would give an algebraic extension of $K$ that is not $\upl$-free. Next, our logarithmic sequence~$(\ell_\rho)$ for $K$ remains a logarithmic sequence for $L$. 

Zorn and the $\forall\exists$-form of the $\upl$-freeness axiom~[ADH, 1.6.1(ii)] reduce us to the case~$D=C(d)$, $d\notin C$, $d$ transcendental over $K$, so $L=K(d)$. 
 Suppose~$L$ is not $\upl$-free. Then~$\upl_\rho\leadsto\upl\in L$, and such $\upl$ is transcendental over $K$ and gives an immediate extension $K(\upl)$ of $K$ by [ADH, 3.2.6].
Hence $L$ is algebraic over  $K(\upl)$,   so 
$\res L$ is algebraic over~$\res K(\upl)=\res K\cong C$ and thus $d$ is algebraic over $C$, a contradiction.
\end{proof}

\begin{lemma}\label{conpred}
Suppose $K$ is $\upl$-free and $L$ is an $H$-asymptotic extension of $K$, where~$L=K(d)$ with $d\in C_L$.  
Then $L$ is pre-$\d$-valued.
\end{lemma}
\begin{proof} 
Let $L^{\operatorname{a}}$ be an algebraic closure of the $H$-asymptotic field $L$, and let $K^{\operatorname{a}}$ be the algebraic closure~ of~$K$ inside $L^{\operatorname{a}}$. Then $K^{\operatorname{a}}$  is pre-$\d$-valued by [ADH, 10.1.22]. Replacing $K$, $L$ by $K^{\operatorname{a}}$, $K^{\operatorname{a}}(d)$ we arrange that
$K$ is algebraically closed. We may assume $d\notin C$, so $d$ is transcendental over $K$ by [ADH, 4.1.1, 4.1.2].

Suppose first that $\operatorname{res}(d)\in\operatorname{res}(K)\subseteq \operatorname{res}(L)$, and take~$b\in\mathcal O$ such that
$y:=b-d\prec 1$.
Then   $b'\notin\der(\smallo)$: otherwise $y'=b'=\delta'$ with $\delta\in\smallo$, so $y=\delta\in K$ and hence~$d\in K$,
a contradiction. 
Also~$vb'\in (\Gamma^>)'$: otherwise  $vb'<(\Gamma^>)'$, by [ADH, 9.2.14], and $vb'$ would be a gap in $K$, contradicting
$\upl$-freeness of $K$. Hence~$L=K(y)$ is pre-$\d$-valued by [ADH, 10.2.4, 10.2.5(iii)]  applied to $s:=b'$.  

If $\operatorname{res}(d)\notin\operatorname{res}(K)$, then $\res(d)$ is transcendental over $\res(K)$ by~[ADH, 3.1.17], hence
$\Gamma_L=\Gamma$ by [ADH, 3.1.11], and so $L$ has asymptotic integration and thus is pre-$\d$-valued by [ADH, 10.1.3].
\end{proof}

\begin{proof}[Proof of Proposition~\ref{prop:const field ext}] By Zorn we reduce to the case $L=K(d)$ with $d\in C_L$.
Then $L$ is pre-$\d$-valued by Lemma~\ref{conpred}.
By Lemma~\ref{lem:Gehret}, the $\d$-valued hull $K_1:=\operatorname{dv}(K)$ of $K$ is $\upl$-free with $\Gamma_{K_1}=\Gamma$, and by the universal property
of $\d$-valued hulls we may arrange that $K_1$ is a $\d$-valued subfield of $L_1:=\operatorname{dv}(L)$ [ADH, 10.3.1].
The proof of~[ADH, 10.3.1] gives $L_1=L(E)$ where $E=C_{L_1}$, and so
$L_1=K_1(E)$. Hence by Lemma~\ref{lem:const field ext} and the remarks preceding it, $L_1$ is $\upl$-free with $\Gamma_{L_1}=\Gamma_{K_1}=\Gamma$.
Thus $L$ is $\upl$-free with~$\Gamma_L=\Gamma$.
\end{proof}

\begin{lemma}\label{lemtrigupl} Let $H$ be a $\upl$-free real closed $H$-field. Then the trigonometric closure $H^{\trig} $ of $H$ is $\upl$-free.
\end{lemma} 
\begin{proof} We show that $H^+$ as in Lemma~\ref{lemtri2} is $\upl$-free. There $H^+[\imag]=K(y)$ where~$K$ is the $H$-asymptotic extension $H[\imag]$ of $H$ and $y\sim 1$, $y^\dagger\notin K^\dagger$,  $y^\dagger\in\imag\der(\smallo_H)$.  Then $K$ is $\upl$-free, so $K(y)$ is $\upl$-free by \cite[Proposition~7.2]{Gehret},   
hence $H^+$ is $\upl$-free. 
\end{proof}

\noindent
In Example~\ref{ex:rat as int and cofinality} we have a $\upl$-free $K$ and an $H$-asymptotic extension $L$ of $K$ that is not $\upl$-free, with $\operatorname{trdeg}(L|K) = 1$.
The next proposition shows that the second part of the conclusion of Theorem~\ref{thm:ADH 13.6.1} nevertheless
holds for such $K$, $L$. 

\begin{prop}\label{prop:rat as int and cofinality}
The following are equivalent:
\begin{enumerate}
\item[\textup{(i)}] $K$ has rational asymptotic integration;
\item[\textup{(ii)}] for every $H$-asymptotic extension $L$ of $K$ with $\operatorname{trdeg}(L|K)\le 1$ we have that
$\Gamma^{<}$  is cofinal in $\Gamma_L^{<}$.
\end{enumerate} 
\end{prop} 
\begin{proof} For  (i)~$\Rightarrow$~(ii), assume (i), and let  $L$ be an $H$-asymptotic extension of $K$ with~$\operatorname{trdeg}(L|K)\le 1$. Towards showing that $\Gamma^{<}$ is cofinal in $\Gamma_L^{<}$ we can arrange that $K$ and $L$ are algebraically closed. Suppose towards a contradiction that $\gamma\in \Gamma_L$ and
$\Gamma^{<} < \gamma < 0$. Then $\Psi < \gamma' < (\Gamma^{>})'$, and so $\Gamma$ is dense in
$\Gamma+\Q\gamma'$ by [ADH, 2.4.16, 2.4.17], in particular, $\gamma\notin \Gamma+\Q\gamma'$. Thus
$\gamma$, $\gamma'$ are $\Q$-linearly independent over $\Gamma$, which contradicts $\operatorname{trdeg}(L|K)\le 1$ by [ADH, 3.1.11]. 

As to (ii)~$\Rightarrow$~(i), we prove the contrapositive, so assume $K$ does not have rational asymptotic integration.
We arrange again that $K$ is algebraically closed. 
Then $K$ has a gap $vs$ with $s\in K^\times$, and so [ADH, 10.2.1 and its proof] gives an $H$-asymptotic extension $K(y)$ of $K$ with $y'=s$ and $0 < vy < \Gamma^{>}$. 
\end{proof}

\noindent
Recall from [ADH, 11.6] that Liouville closed $H$-fields are $\upl$-free. 
To prove the next result we also use Gehret's theorem~\cite[Theorem 12.1(1)]{Gehret}  
that an $H$-field $H$  has up to isomorphism over $H$ exactly one Liouville closure iff $H$ is grounded or $\upl$-free. Here {\it isomorphism\/} means of course {\it isomorphism of $H$-fields\/}, and likewise with the embeddings referred to in the next result:

\begin{prop}\label{proptl} Let $H$ be a grounded or $\upl$-free  $H$-field. Then $H$ has a trigonometrically closed 
and Liouville closed $H$-field extension $H^{\tl}$ that embeds over $H$ into any trigonometrically closed 
Liouville closed $H$-field extension of $H$.
\end{prop} 
\begin{proof} We build real closed $H$-fields $H_0\subseteq H_1\subseteq H_2\subseteq \cdots$ as follows: $H_0$ is a real closure of $H$, and, recursively,
  $H_{2n+1}$ is a Liouville closure of $H_{2n}$, and~$H_{2n+2}:= H_{2n+1}^{\trig}$ is the trigonometric closure of $H_{2n+1}$. 
Then $H^{*}:= \bigcup_n H_n$ is a trigonometrically closed Liouville closed $H$-field extension of $H$. 
Induction using Lemma~\ref{lemtrigupl} shows that all $H_n$ with $n\ge 1$ are $\upl$-free, and that 
$H_{2n}$ has for all~$n$ up to isomorphism over $H$ a unique Liouville closure. Given any trigonometrically closed Liouville closed $H$-field extension $E$ of $H$ we then use the embedding properties of {\it Liouville closure\/} and {\it trigonometric closure\/} to construct by a similar recursion embeddings $H_n\to E$ that extend to an embedding $H^{*}\to E$ over $H$.  
\end{proof} 

\noindent
For $H$ as in Proposition~\ref{proptl}, the $H^*$ constructed in its proof is minimal: Let~${E\supseteq H}$ be any trigonometrically closed Liouville closed $H$-subfield of $H^*$. Then induction on $n$ yields $H_n\subseteq~E$ for all $n$, so $E=H^*$.
It follows that any $H^{\tl}$ as in Proposition~\ref{proptl} is isomorphic over $H$ to $H^*$, and we refer to such
$H^{\tl}$ as a {\bf  trigonometric-Liouville closure\/} of $H$.\index{closure!trigonometric-Liouville}
\index{trigonometric!Liouville closure}\label{p:Htl} Here are some useful facts about $H^{\tl}$:

\begin{cor}\label{cortrig} Let $H$ be a $\upl$-free $H$-field. Then $C_{H^{\tl}}$ is a real closure of $C_H$, the $H$-asymptotic extension $K^{\tl}:= H^{\tl}[\imag]$ of $H^{\tl}$
is a Liouville extension of $H$ with~$\I(K^{\tl})\subseteq (K^{\tl})^\dagger$, and $\Gamma_H^{<}$ is cofinal in $\Gamma_{H^{\tl}}^{<}$. Moreover,
$$ \text{$H$ is $\upo$-free}\ \Longleftrightarrow\  \text{$H^{\tl}$ is $\upo$-free.}$$ 
\end{cor}
\begin{proof} The construction of $H^*$ in the proof of Proposition~\ref{proptl} gives that $C_{H^{*}}$ is a real closure of $C_H$,
and that the $H$-asymptotic extension $K^{*}:=H^*[\imag]$ of~$H^*$
is a Liouville extension of $H$ with $\I(K^{*})\subseteq (K^{*})^\dagger$.  Induction using Lemma~\ref{lemtrigupl} and Proposition~\ref{prop:rat as int and cofinality} shows that $H_n$ is $\upl$-free  and $\Gamma_H^{<}$ is cofinal in $\Gamma_{H_n}^{<}$, for all $n$, so  $\Gamma_H^<$ is cofinal in $\Gamma_{H^*}^{<}$. 

The final equivalence follows from Theorem~\ref{thm:ADH 13.6.1} and  a remark preceding it.  
\end{proof}

\noindent
Proposition~\ref{prop:rat as int and cofinality} and
[ADH, remarks after~11.6.4 and   after~11.7.19] yield:

\begin{cor}\label{cor:rat as int and cofinality}
Suppose $K$ has rational asymptotic integration, and let $L$ be an $H$-asymptotic extension of $K$ with $\operatorname{trdeg}(L|K)\leq 1$.
If $L$ is $\upl$-free, then so is $K$, and if $L$ is $\upo$-free, then so is $K$.
\end{cor}

\noindent
We also have a similar characterization of $\upl$-freeness:

\begin{prop}\label{prop:upl-free and as int} 
The following are equivalent:
\begin{enumerate}
\item[\textup{(i)}] $K$ is $\upl$-free;
\item[\textup{(ii)}] every $H$-asymptotic extension $L$ of $K$ with $\operatorname{trdeg}(L|K)\le 1$ has asymptotic integration.
\end{enumerate} 
\end{prop}
\begin{proof}
Assume $K$ is $\upl$-free, and let $L$ be an $H$-asymptotic extension of $K$ such that~$\operatorname{trdeg}(L|K)\le 1$.
By Proposition~\ref{prop:rat as int and cofinality},  $\Gamma^{<}$ is cofinal in $\Gamma_L^{<}$,  so $L$ is ungrounded.
Towards
a contradiction, suppose $vf$ ($f\in L^\times$) is a gap in $L$. 
Passing to algebraic closures we arrange that $K$ and $L$ are algebraically closed. 
 Set $\upl:=-f^\dagger$. Then for all active $a$ in $L$
we have $\upl+a^\dagger\prec a$ by~[ADH, 11.5.9] and hence $\upl_\rho\leadsto\upl$ by~[ADH, 11.5.6]. By $\upl$-freeness of $K$
and~[ADH, 3.2.6, 3.2.7], the valued field extension~$K(\upl)\supseteq K$ is  immediate of transcendence degree~$1$, so~$L\supseteq K(\upl)$ is algebraic and~$\Gamma=\Gamma_L$. Hence $vf$ is a gap in $K$, a contradiction. This shows~(i)~$\Rightarrow$~(ii).

To show the contrapositive of (ii)~$\Rightarrow$~(i),  
 suppose  $\upl\in K$ is a pseudolimit of~$(\upl_\rho)$. If the algebraic closure $K^{\operatorname{a}}$ of $K$ does not have asymptotic integration,
 then clearly~(ii) fails. If $K^{\operatorname{a}}$ has asymptotic integration,
 then $-\upl$ creates a gap over~$K$ by~[ADH, 11.5.14] applied to $K^{\operatorname{a}}$ in place of $K$, hence (ii) also fails.
\end{proof}

\noindent
The next two lemmas include converses to Lemmas~\ref{lem:Gehret} and~\ref{lem:const field ext}.

\begin{lemma}\label{notupl1} Let $E$ be a pre-$\d$-valued $H$-asymptotic field. Then:\begin{enumerate}
\item[\textup{(i)}] if $E$ is not $\upl$-free, then $\operatorname{dv}(E)$ is not $\upl$-free; 
\item[\textup{(ii)}] if $E$ is not $\upo$-free, then $\operatorname{dv}(E)$ is not $\upo$-free.
\end{enumerate}
\end{lemma}
\begin{proof} This is clear if $E$ does not have rational asymptotic integration, because then neither does $\operatorname{dv}(E)$, by [ADH, 10.3.2]. Assume $E$ has rational asymptotic integration. Then $\operatorname{dv}(E)$
is an immediate extension of $E$ by~[ADH, 10.3.2], and then (i) and (ii) follow from the characterizations of $\upl$-freeness and 
$\upo$-freeness in terms of nonexistence of certain pseudolimits. 
\end{proof} 

\begin{lemma}\label{notupl2} Let $E$ be a $\d$-valued $H$-asymptotic field and $F$ an $H$-asymptotic extension of $E$ such that
$F=E(C_F)$.
Then: \begin{enumerate}
\item[\textup{(i)}] if $E$ is not $\upl$-free, then $F$ is not $\upl$-free; 
\item[\textup{(ii)}] if $E$ is not $\upo$-free, then $F$ is not $\upo$-free.
\end{enumerate}
\end{lemma}
\begin{proof} By [ADH, 10.5.15] $E$ and $F$ have the same value group. The rest of the proof is like that for the previous lemma, with $F$ instead of $\operatorname{dv}(E)$. 
\end{proof}

\noindent
{\it In the rest of this subsection $K$ is in addition a pre-$H$-field and $L$ a pre-$H$-field extension of $K$.}\/
The following is shown in the proof of~\cite[Lemma~12.5]{Gehret}:

\begin{prop}[Gehret] \label{prop:Gehret} 
Suppose $K$  is a $\upl$-free $H$-field and $L$ is a Liou\-ville $H$-field extension of $K$.
Then $L$ is $\upl$-free and $\Gamma^<$ is cofinal in~$\Gamma_L^<$. 
\end{prop}

\begin{exampleNumbered}\label{ex:Gehret} 
Let $K=\R\langle\upo\rangle$ be the $\upl$-free but non-$\upo$-free $H$-field from [ADH, 13.9.1]. Then 
$K$ has a unique Liouville closure $L$, up to isomorphism over $K$, by~\cite[Theorem~12.1(1)]{Gehret}.
By Proposition~\ref{prop:Gehret}, $L$
is not $\upo$-free; another proof of this fact is in \cite{ADH3}. By [ADH, 13.9.5] we can take
here $K$ to be a Hardy field, and then $L$ is isomorphic  over $K$ to a Hardy field
extension of $K$ [ADH, 10.6.11].

Applying Corollary~\ref{cortrig} to $H:=\R\langle \upo\rangle$ yields a Liouville closed $H$-field~$H^{\tl}$ that is not
$\upo$-free but does satisfy $\I(K^{\tl})\subseteq (K^{\tl})^\dagger$ for $K^{\tl}:= H^{\tl}[\imag]$. 
\end{exampleNumbered}

%\noindent
%For a pre-$H$-field $H$ we singled out in [ADH, p.~520]  the following subsets: \label{p:special subsets}
%$$\Upg(H)\ :=\ (H^{\succ 1})^\dagger,\quad \Upl(H)\ :=\ -(H^{\succ 1})^{\dagger\dagger}, \quad  \Upd(H)\ :=\ -(H^{\neq,\prec 1})^{\prime\dagger}.$$
 
\begin{lemma}\label{lem:upl in D(H)}
Suppose $K$ is $\upl$-free, $\upl\in \Upl(L)^\downarrow$, $\upo:=\omega(\upl)\in K$, and suppose~$\omega\big(\Upl(K)\big) <  \upo   < \sigma\big(\Upg(K)\big)$. Then $\upl_\rho\leadsto \upl$, and the pre-$H$-sub\-field~$K\langle \upl\rangle=K(\upl)$ of~$L$   is an immediate extension of $K$ \textup{(}and so $K\langle \upl\rangle$ is not $\upl$-free\textup{)}.
\end{lemma}
\begin{proof} From $\Upl(L)<\Upd(L)$ [ADH, p.~522] and $\Upd(K)\subseteq\Upd(L)$ we
obtain $\upl<\Upd(K)$. 
The restriction of  $\omega$ to $\Upl(L)^\downarrow$ is strictly increasing~[ADH, p.~526] and $\Upl(K)\subseteq\Upl(L)$,  so~$\omega\big(\Upl(K)\big)<\upo=\omega(\upl)$ gives $\Upl(K) < \upl$.
Hence $\upl_\rho\leadsto \upl$ by [ADH, 11.8.16]. Also~$\upo_\rho\leadsto\upo$ by [ADH, 11.8.30]. Thus $K\langle \upl\rangle$ is an immediate   extension of~$K$ by~[ADH, 11.7.13].
\end{proof}

\subsection*{Achieving $\upo$-freeness for pre-$H$-fields}
{\it In the rest of this section $H$ is a pre-$H$-field and  $L$
is a Liouville closed $\d$-algebraic   $H$-field extension of $H$.}\/
Thus if~$H$ is $\upo$-free, then so is~$L$, by Theorem~\ref{thm:ADH 13.6.1}. 
The lemmas below give conditions guaranteeing that $L$ is $\upo$-free, while $H$ is not.
 
\begin{lemma}\label{lem:Li(H) upo-free} 
Suppose  $H$ is grounded or has a gap.  Then~$L$ is  $\upo$-free.
\end{lemma}
\begin{proof}
Suppose~$H$ is grounded. Let $H_{\upo}$ be the $\upo$-free pre-$H$-field extension of $H$ introduced in connection with~[ADH, 11.7.17] (where we use the letter $F$ instead of~$H$).  Identifying $H_{\upo}$ with its image in $L$ under an embedding $H_{\upo}\to L$ over $H$ of pre-$H$-fields, we apply Theorem~\ref{thm:ADH 13.6.1} to $K:=H_{\upo}$ to conclude that $L$ is $\upo$-free. 

  Next, suppose $H$ has a gap $\beta=vb$, $b\in H^\times$. Take~$a\in  L$
with~$a'=b$ and~$a\nasymp 1$. Then $\alpha:= va$ satisfies $\alpha'=\beta$, and so the pre-$H$-field
$H(a)\subseteq L$ is grounded, by~[ADH,  9.8.2 and remarks following its proof].
 Now apply the previous case to~$H(a)$ in place of $H$.
\end{proof}

\begin{lemma}\label{lem:D(H) upo-free, 2}
Suppose $H$ has asymptotic integration and divisible value group, and~$s\in H$ creates a gap over $H$. Then $L$ is $\upo$-free.
\end{lemma}
\begin{proof}
Take $f\in L^\times$ with $f^\dagger=s$.
Then by [ADH, remark after 11.5.14], $vf$ is a gap in~$H\langle f\rangle=H(f)$, so~$L$ is $\upo$-free by Lemma~\ref{lem:Li(H) upo-free} applied to $H\langle f\rangle$ in place of~$H$. 
\end{proof}

\begin{lemma}\label{lem:D(H) upo-free, 3}
Suppose $H$ is not $\upl$-free. Then $L$ is $\upo$-free.
\end{lemma}
\begin{proof}
By [ADH, 11.6.8], the real closure   $H^{\operatorname{rc}}$ of $H$ inside $L$
is not $\upl$-free, hence replacing $H$ by $H^{\operatorname{rc}}$  we arrange that $H$ is real closed.
If $H$ does not have asymptotic integration, then we are done 
by  Lemma~\ref{lem:Li(H) upo-free}. So suppose  $H$  has asymptotic integration. Then some
$s\in H$ creates a gap over $H$, by [ADH, 11.6.1], so $L$ is $\upo$-free by Lemma~\ref{lem:D(H) upo-free, 2}.
\end{proof}

\begin{cor}\label{cor:D(H) upo-free, 3}
Suppose $H$ is $\upl$-free and $\upl\in\Upl(L)^\downarrow$ is such that $\upo:=\omega(\upl)\in H$ and
$\omega\big(\Upl(H)\big) <  \upo   < \sigma\big(\Upg(H)\big)$. Then $L$ is $\upo$-free.
\end{cor}
\begin{proof}
By Lemma~\ref{lem:upl in D(H)}, the pre-$H$-subfield~$H\langle \upl\rangle=H(\upl)$ of $L$ is
an immediate non-$\upl$-free extension of $H$.  Now apply  Lemma~\ref{lem:D(H) upo-free, 3} to~$H\langle \upl\rangle$ in place of~$H$.
\end{proof}

\section{Complements to [ADH] on Linear Differential Operators}\label{sec:lindiff}

\noindent
In this section we tie up loose ends from the material on linear differential operators in [ADH,~14.2] and \cite[Section~8]{VDF}.
{\em Throughout $K$ is an ungrounded asymptotic field with $\Gamma=v(K^\times)\ne \{0\}$, $a$,~$b$,~$f$,~$g$,~$h$  range over arbitrary elements of $K$, and $\phi$ over those active in $K$, in particular,  $\phi\ne 0$}.
Recall   our use of the term ``eventually'': a property $S(\phi)$ of   
$\phi$ is said to hold {\em eventually\/} if for some active $\phi_0$ in $K$, $S(\phi)$ holds for all~$\phi\preceq \phi_0$. 

\medskip
\noindent
We shall consider linear differential operators $A\in K[\der]^{\ne}$ and  set $r:=\order(A)$. 
In~[ADH, 11.1] we showed that for each $\gamma\in\Gamma$ the quantity $\dwt_{A^\phi}(\gamma)$ is eventually constant; its eventual value is denoted by $\nwt_A(\gamma)$. %\marginpar{added   this sentence}
We also introduced the set
$$\exc^{\ev}(A)\ =\ \exc^{\ev}_K(A)\ :=\ \big\{ \gamma\in\Gamma:\,  \nwt_A(\gamma)\geq 1  \big\}\ =\ \bigcap_{\phi}\exc(A^\phi)$$
of {\em eventual exceptional values of $A$}. \label{p:excev}\index{linear differential operator!eventual exceptional values}\index{values!eventual exceptional}\index{exceptional values!eventual}For $a \neq 0$ we have $\exc^{\ev}(aA) = \exc^{\ev}(A)$ and $\exc^{\ev}(Aa) = \exc^{\ev}(A) - va$. An easy consequence of the definitions: $\exc^{\ev}(A^f)=\exc^{\ev}(A)$ for~$f\ne 0$. 
 A key fact about $\exc^{\ev}(A)$ is that if~$y\in K^\times$,
$vy\notin \exc^{\ev}(A)$, then $A(y)\asymp A^\phi y$, eventually.
Since $A^\phi y\ne 0$ for $y\in K^\times$, this gives
$v(\ker^{\neq} A)\subseteq \exc^{\ev}(A)$. 

\begin{lemma}\label{lemexc} If $L$ is an ungrounded asymptotic extension of $K$, then we have ${\exc^{\ev}_{L}(A)\cap\Gamma}\subseteq \exc^{\ev}(A)$, with equality if $\Psi$ is cofinal in $\Psi_L$. 
 \end{lemma}
 \begin{proof} For the inclusion, use that $\dwt(A^\phi)$ decreases as $v\phi$ strictly increases~[ADH, 11.1.12]. Thus its eventual value $\nwt(A)$, evaluated in $K$, cannot strictly increase when evaluated in an ungrounded asymptotic extension of $K$. 
 \end{proof} 

\noindent
{\em In the rest of this section we assume in addition that $K$ is $H$-asymptotic with asymptotic integration.}
Then by [ADH, 14.2.8]:

\begin{prop}\label{kerexc}
If $K$ is $r$-linearly newtonian, then $v(\ker^{\neq} A) = \exc^{\ev}(A)$.
\end{prop}

\begin{remarkNumbered}\label{rem:kerexc}
If  $K$ is $\d$-valued, then $\abs{v(\ker^{\neq} A)} = \dim_C \ker A \leq r$ by [ADH, 5.6.6], using a reduction to the case of ``small derivation'' by compositional conjugation. 
\end{remarkNumbered}

\begin{cor} \label{cor:nonexc sol} Suppose $K$ is $\d$-valued, $\exc^{\ev}(A)=v(\ker^{\neq} A)$, and $0\ne f\in A(K)$. Then $A(y)=f$ for some $y\in K$ with $vy\notin \exc^{\ev}(A)$.
\end{cor}
\begin{proof} Let $y\in K$, $A(y)=f$, with $vy$ maximal. Then $vy\notin \exc^{\ev}(A)$:
otherwise we have $z\in \ker A$ with $z\sim y$, so $A(y-z)=f$ and $v(y-z)>vy$.
\end{proof}

\begin{cor}\label{cor:sum of nwts}
Suppose $K$ is $\upo$-free. Then
$\sum_{\gamma\in\Gamma} \nwt_A(\gamma)=\abs{\exc^{\ev}(A)} \leq r$.
\end{cor}
\begin{proof} By \eqref{eq:14.0.1} we have
an immediate newtonian asymptotic extension~$L$ of~$K$.   %remarks after [ADH,~14.0.1] give  
Then $L$ is $\d$-valued 
by Lemma~\ref{lem:ADH 14.2.5},   hence~$\abs{\exc^{\ev}(A)} = \abs{\exc_L^{\ev}(A)}\leq r$ by Proposition~\ref{kerexc} and Remark~\ref{rem:kerexc}.
By~[ADH, 13.7.10]  we have $\nwt_{A}(\gamma)\leq 1$ for all $\gamma\in\Gamma$, thus~$\sum_{\gamma\in\Gamma} \nwt_A(\gamma)=\abs{\exc^{\ev}(A)}$.
\end{proof}

\noindent
In [ADH, 11.1] we defined ${v_A^{\ev}\colon\Gamma\to\Gamma}$ by requiring that for all $\gamma\in\Gamma$:\label{p:vAev}
\begin{equation}\label{eq:vAev}
v_{A^\phi}(\gamma)\ =\ v_A^{\ev}(\gamma)+\nwt_A(\gamma)v\phi,\qquad\text{ eventually.}
\end{equation}
We recall from that reference that for $a\neq 0$ and $\gamma\in\Gamma$ we have 
$$v_{aA}^{\ev}(\gamma)\ =\ va+v_A^{\ev}(\gamma),\qquad v_{Aa}^{\ev}(\gamma)=v_A^{\ev}(va+\gamma).$$
As an example from [ADH, p.~481], $v_{\der}^{\ev}(\gamma)=\gamma + \psi(\gamma)$ for $\gamma\in \Gamma\setminus \{0\}$ and $v_{\der}^{\ev}(0)=0$.
By [ADH, 14.2.7 and the remark preceding it] we have:

\begin{lemma}\label{lem:ADH 14.2.7}
The restriction of $v_A^{\ev}$ to a  function $\Gamma\setminus\exc^{\ev}(A)\to\Gamma$ is strictly increasing, and
$v\big(A(y)\big) = v_A^{\ev}(vy)$ for all $y\in K$ with $vy\in \Gamma\setminus\exc^{\ev}(A)$.
Moreover, if~$K$ is $\upo$-free, then $v_A^{\ev}\big(\Gamma\setminus \exc^{\ev}(A)\big)=\Gamma$.  
\end{lemma}

\noindent
A differential field $F$ is said to be {\it $r$-linearly surjective}\/ ($r\in \N$) if $A(F)=F$ for every  $A\in F[\der]^{\neq}$ of order at most~$r$.\index{r-linearly surjective@$r$-linearly surjective!differential field}\index{differential field!r-linearly surjective@$r$-linearly surjective}
The following is [ADH, 14.2.10] without the hypothesis of $\upo$-freeness:

\begin{cor}\label{cor:14.2.10, generalized} 
Suppose $K$ is $r$-linearly newtonian. Then for each $f\neq 0$ there exists $y\in K^\times$ such that $A(y)=f$,
$vy\notin\exc^{\ev}(A)$, and $v^{\ev}_A(vy)=vf$.
\end{cor}
\begin{proof} If $r=0$, then $\exc^{\ev}(A)=\emptyset$ and our claim is obviously valid. Suppose $r\ge 1$. Then
 $K$ is $\d$-valued by Lemma~\ref{lem:ADH 14.2.5}, and $v(\ker^{\neq} A) = \exc^{\ev}(A)$
by Proposition~\ref{kerexc}, Moreover, by [ADH, 14.2.2], $K$ is $r$-linearly surjective, hence $f\in A(K)$.
Now Corollary~\ref{cor:nonexc sol} yields $y\in K^\times$ with $A(y)=f$ and $vy\notin \exc^{\ev}(A)$.
By Lemma~\ref{lem:ADH 14.2.7} we have $v^{\ev}_A(vy)=v\big(A(y)\big)=vf$.
\end{proof}

\noindent
From the proof of [ADH, 14.2.10] we extract the following:

{\sloppy
\begin{cor}\label{cor:ADH 14.2.10 extract}
Suppose $K$ is $r$-linearly newtonian with small derivation, and~$A\in~\mathcal O[\der]$ with $a_0:=A(1)\asymp 1$,
and $f\asymp^\flat 1$. Then there is $y\in K^\times$ such that~${A(y)=f}$ and $y\sim f/a_0$. For any such $y$ we have $vy\notin\exc^{\ev}(A)$
and $v_A^{\ev}(vy)=vf$.
\end{cor}}
\begin{proof} The case $r=0$ is trivial. 
Assume $r\geq 1$, so $K$ is $\d$-valued by Lemma~\ref{lem:ADH 14.2.5}. Hence $f^\dagger\prec 1$, that is, $f'\prec f$,
so $f^{(n)}\prec f$ for all $n\ge 1$ by [ADH, 4.4.2].  Then $Af\preceq~f$ by [ADH, (5.1.3), (5.1.2)], and $A(f)\sim a_0f$, so 
$A_{\ltimes f}\in\mathcal O[\der]$ and~$A_{\ltimes f}(1)\sim~a_0$.  Thus we may replace $A$, $f$ by $A_{\ltimes f}$, $1$ to arrange~${f=1}$.
Now~${a_0\asymp 1}$ gives $\operatorname{dwm}(A)=0$, so $\operatorname{dwt}(A^\phi)=0$ eventually, by [ADH, 11.1.11(ii)],
that is, $\nwt(A)=0$. Also $A^\phi(1)=A(1)=a_0\asymp 1$,  so  $v^{\ev}(A)=0$.  Arguing as in the proof of [ADH, 14.2.10] we obtain
$y\in K^\times$ with $A(y)=1$ and~$y\sim 1/a_0$. It is clear that~$vy=0\notin \exc^{\ev}(A)$ and $v_A^{\ev}(vy)=v^{\ev}(A)=0=vf$
for any such $y$.
\end{proof}

\noindent
In the next few subsections below we consider more closely the case of order $r=1$, and in the last subsection the case of arbitrary order.

\subsection*{First-order operators}
{\em In this subsection $A=\der-g$}. By [ADH, p.~481],   
$$\exc^{\ev}(A)\ =\ \exc^{\ev}_K(A)\ =\ \big\{vy:\, y\in K^\times,\ v(g-y^\dagger)>\Psi\big\}$$ 
has at most one element. 
We also have
$\abs{v(\ker^{\neq}A)} = \dim_C\ker A \leq 1$ in view of~$C^\times\subseteq\mathcal O^\times$. Proposition~\ref{kerexc} holds under a weaker assumption on $K$ for $r=1$: 

\begin{lemma}\label{lem:v(ker)=exc, r=1}
Suppose $\I(K)\subseteq K^\dagger$. Then $v(\ker^{\neq} A)=\exc^{\ev}(A)$.
\end{lemma}
\begin{proof}
It remains to show ``$\supseteq$''.
Suppose $\exc^{\ev}(A)=\{0\}$. Then $g-y^\dagger\in \I(K)$ with~${y\asymp 1}$ in $K$, hence
$g\in \I(K)\subseteq K^\dagger$, so $g=h^\dagger$ with $h\asymp 1$, and thus~$0=vh\in v(\ker^{\neq} A)$. The general case reduces to the case $\exc^{\ev}(A)=\{0\}$ by twisting.
\end{proof}

\begin{lemma}\label{lemexc, order 1}
Suppose $L$ is an ungrounded $H$-asymptotic extension of $K$. Then  
$\exc^{\ev}_{L}(A) \cap \Gamma  = \exc^{\ev}(A)$.
\end{lemma}
\begin{proof}
Lemma~\ref{lemexc} gives
$\exc^{\ev}_{L}(A) \cap \Gamma \subseteq \exc^{\ev}(A)$.
Next, let $vy\in \exc^{\ev}(A)$, $y\in K^\times$.
Then $v(g-y^\dagger)>\Psi$ and so $v(g-y^\dagger)\in (\Gamma^>)'$ since $K$ has asymptotic integration. 
Hence $v(g-y^\dagger)>\Psi_L$ and thus $vy\in\exc^{\ev}_{L}(A)$, by [ADH, p.~481].
\end{proof}

\noindent
Recall also from [ADH,~9.7] that for an ordered abelian group $G$ and $U \subseteq G$, a function~$\eta\colon U \to G$ is said to be {\it slowly varying}\/ if $\eta(\alpha)-\eta(\beta) = o(\alpha-\beta)$ for all~$\alpha\neq\beta$ in~$U$;\index{slowly varying function} then the function
$\gamma\mapsto\gamma+\eta(\gamma)\colon U\to G$ is strictly increasing. The quintessential example of a slowly varying function
is $\psi\colon \Gamma^{\neq}\to\Gamma$   [ADH, 6.5.4(ii)]. 

\begin{prop}\label{prop:slow}
There is a unique slowly varying function $\psi_A\colon \Gamma\setminus\exc^{\ev}(A)\to\Gamma$
such that for all $y\in K^\times$ with $vy\notin \exc^{\ev}(A)$ we have $v\big(A(y)\big)=vy+\psi_A(vy)$.
\end{prop}
\begin{proof}
For $\d$-valued $K$, use \cite[8.4]{VDF}. In general, pass to the $\d$-valued hull $L:=\operatorname{dv}(K)$ of $K$ from 
[ADH,~10.3] and use $\Gamma_{L}=\Gamma$ [ADH, 10.3.2].
\end{proof}

\noindent
If $b\neq 0$, then $\exc^{\ev}(A_{\ltimes b})=\exc^{\ev}(A)-vb$ and $\psi_{A_{\ltimes b}}(\gamma)=\psi_A(\gamma+vb)$ for $\gamma\in\Gamma\setminus \exc^{\ev}(A_{\ltimes b})$.

\begin{example}
We have
$\exc^{\ev}(\der)=\{0\}$ and
$\psi_\der=\psi$.
More generally, if $g=b^\dagger$, $b\neq 0$, then~$A_{\ltimes b}=\der$ and so $\exc^{\ev}(A)=\{vb\}$ and
$\psi_A(\gamma)=\psi(\gamma-vb)$ for $\gamma\in\Gamma\setminus\{vb\}$. 
\end{example}

\noindent
If $\Gamma$ is divisible, then $\Gamma\setminus v\big(A(K)\big)$ has at most one element by [ADH,  11.6.16]. Also,
 $K$ is $\upl$-free iff $v\big(A(K)\big)=\Gamma_\infty$ for all $A=\der-g$ by [ADH, 11.6.17].

{\sloppy
\begin{lemma}\label{lem:slow}
Suppose $K$ is $\upl$-free and $f\neq 0$. Then for some $y\in K^\times$ we have~${A(y)\asymp f}$ and~${vy\notin\exc^{\ev}(A)}$.
\textup{(}Hence $\gamma\mapsto \gamma+\psi_A(\gamma)\colon \Gamma\setminus\exc^{\ev}(A)\to\Gamma$ is surjective.\textup{)}
\end{lemma}}
\begin{proof}{} [ADH, 11.6.17] gives $y\in K^\times$ with $A^\phi y\asymp f$ eventually. Now
$$A^\phi y\ =\ \phi y\derdelta-(g-y^\dagger)y\ \text{ in }K^\phi[\derdelta],\qquad \derdelta:=\phi^{-1}\der.$$ 
Since $v(A^{\phi} y)=vf$ eventually, this forces $g-y^\dagger\succ\phi$ eventually, so $vy\notin \exc^{\ev}(A)$.
\end{proof}

\noindent
Call $A$ {\bf steep} if $g\succ^\flat 1$, that is, $g\succ 1$ and $g^\dagger\succeq 1$. If $K$ has small derivation and
$A$ is steep, then $g^\dagger\prec g$ by [ADH, 9.2.10]. \index{steep!linear differential operator}\index{linear differential operator!steep}

\begin{lemma}\label{lem:order1, 2}
Suppose $K$ has small derivation,  $A$ is steep, and $y\in K^\times$ such that~$A(y)=f\ne 0$, $g \succ f^\dagger$, and $vy\notin \exc^{\ev}(A)$. Then $y\sim -f/g$.
\end{lemma}
\begin{proof}
We have $$(f/g)^\dagger-g=f^\dagger-g^\dagger-g\sim-g \succ g^\dagger,$$ hence $v(f/g)\notin \exc^{\ev}(A)$, and
$$A(f/g)\ =\ (f/g)'-(f/g)g\ =\  (f/g)\cdot \big( f^\dagger-g^\dagger-g \big)\ \sim\ (f/g)\cdot (-g)\ =\ -f.$$
Since $A(y)=f\sim A(-f/g)$ and $vy,v(f/g)\in\Gamma\setminus\exc^{\ev}(A)$, this gives $y = u\cdot f/g$ where $u\asymp 1$, by Proposition~\ref{prop:slow}.
Now $u^\dagger \prec 1\prec g$ and $(f/g)^\dagger=f^\dagger-g^\dagger\prec g$, 
hence~$y^\dagger\prec g$ and so $$f=A(y)=y\cdot(y^\dagger-g)\sim - y g.$$
Therefore $y\sim -f/g$.
\end{proof}

\begin{lemma}\label{prlemexc}
Suppose $K$ has small derivation and $y\in K^\times$ is such that $A(y)=f\ne 0$, $g-f^\dagger\succ^\flat 1$ and $vy\notin\exc^{\ev}(A)$.  Then
$y\sim f/(f^\dagger-g)$.
\end{lemma}
\begin{proof} From $g-f^\dagger\succ 1$ we get $vf\notin\exc^{\ev}(A)$. Now $A(y)=f\prec f(f^\dagger-g)=A(f)$, so $y\prec f$ by [ADH, 5.6.8], and
$v(y/f)\notin\exc^{\ev}(A_{\ltimes f})=\exc^{\ev}(A)-vf$. Since $A_{\ltimes f}=\der-(g-f^\dagger)$ is steep,
Lemma~\ref{lem:order1, 2} applies to $A_{\ltimes f}$, $y/f$, $1$ in the role of $A$, $y$, $f$.
\end{proof}

\noindent
Suppose $K$ is $\upl$-free and $f\ne 0$.  Then [ADH, 11.6.1] gives an active~$\phi_0$ in $K$  with ${f^\dagger-g-\phi^\dagger}\succeq \phi_0$ 
for all $\phi\prec\phi_0$. The convex subgroups $\Gamma^\flat_{\phi}$ of $\Gamma$ become arbitrarily small as we let $v\phi$ increase cofinally
in $\Psi^{\downarrow}$, so $\phi \prec^\flat_\phi \phi_0$ eventually, and hence~$f^\dagger-g-\phi^\dagger \succ^\flat_\phi \phi$ eventually, that is,
$\phi^{-1}(f/\phi)^\dagger-g/\phi\succ^\flat_\phi 1$ eventually. So replacing $K$ by $K^\phi$, $A$ by 
$\phi^{-1} A^\phi=\derdelta-(g/\phi)$ in $K^\phi[\derdelta]$,
and $f$ and $g$ by~$f/\phi$ and~$g/\phi$, for suitable $\phi$, we arrange $f^\dagger-g\succ^\flat 1$. 
Thus by Lemma~\ref{prlemexc}:

\begin{cor}\label{cor:prlemexc}
If $K$ is $\upl$-free, $y\in K^\times$, $A(y)=f\ne 0$, and $vy\notin\exc^{\ev}(A)$, then~$y\sim f/\big((f/\phi)^\dagger-g\big)$,  eventually.
\end{cor}

\begin{example}
If $K$ is $\upl$-free and $y\in K$, $y'=f\ne 0$ with $y\nasymp 1$, then $y\sim f/(f/\phi)^\dagger$, eventually.
\end{example}

\subsection*{From $K$ to $K[\imag]$} {\em In this subsection $K$ is a real closed $H$-field}. 
Then $K[\imag]$ (${\imag^2=-1}$) is an $H$-asymptotic extension of $K$, with $\Gamma_{K[\imag]}=\Gamma$.
Consider a linear differential operator $B=\der-(g+h\imag)$ over $K[\imag]$.
Note that $g+h\imag\in K[\imag]^\dagger$ iff $g\in K^\dagger$ and~${h\imag\in K[\imag]^\dagger}$, by
Lemma~\ref{lem:logder}.
Under further assumptions on $K$, the next two results give explicit descriptions of $\psi_B$ when $g\in K^\dagger$. 

\begin{prop}\label{prop:psiB} Suppose $K[\imag]$ is $1$-linearly newtonian and $g\in K^\dagger$. Then: \begin{enumerate}
\item[$\rm(i)$] if $h\imag\in K[\imag]^\dagger$, then for some $\beta\in \Gamma$ we have $$\exc^{\ev}(B)\ =\ \{\beta\}, \qquad \psi_B(\gamma)\ =\ \psi(\gamma-\beta)\ \text{ for all }\gamma\in\Gamma\setminus\{\beta\};$$
\item[$\rm(ii)$] if $h\imag\notin K[\imag]^\dagger$ and $g=b^\dagger$, 
$b\ne 0$, then 
$$\exc^{\ev}(B)\ =\ \emptyset, \qquad 
\psi_B(\gamma)\ =\ \min\!\big( \psi(\gamma-vb),vh\big)\  \text{ for all }\gamma\in\Gamma.$$
\end{enumerate}
\end{prop}
\begin{proof} As to (i), apply the example following Proposition~\ref{prop:slow} to $K[\imag]$, $B$, $g+h\imag$ in the roles
of $K$, $A$, $g$. For (ii), assume $h\imag\notin K[\imag]^\dagger$, $g=b^\dagger$, $b\ne 0$. 
Replacing $B$ by~$B_{\ltimes b}$ we arrange $g=0$, $b=1$, $B=\der-h\imag$.
Corollary~\ref{cor:logder} gives $K[\imag]^\dagger=K^\dagger\oplus \I(K)\imag$, so~$h\notin \I(K)$, and thus $vh\in \Psi^{\downarrow}$.
Let $y\in K[\imag]^\times$, and take  $z\in K^\times$ and $s\in\I(K)$ with $y^\dagger=z^\dagger+s\imag$. Then
$vh<vs$, hence
$$v(y^\dagger-h\imag)\ =\ \min\!\big( v(z^\dagger), v(s-h) \big)\ =\ \min\!\big( v(z^\dagger), vs, vh \big)\ =\
 \min\!\big( v(y^\dagger), vh \big),$$
 where the last equality uses $v(y^\dagger)=\min \big(v(z^\dagger), vs\big)$. Thus $v(y^\dagger-h\imag)\in \Psi^{\downarrow}$ and
$$v\big(B(y)\big)-vy\ =\ v(y^\dagger-h\imag)\ =\ \min\!\big(v(y^\dagger),vh\big)\ =\ \min\!\big(\psi(vy),vh\big),$$
which gives the desired result.
\end{proof}

\begin{cor}\label{cor:psiB} Suppose $K$ is $\upo$-free, $g\in K^\dagger$, $g=b^\dagger$, $b\ne 0$. Then either for some $\beta\in \Gamma$ we have $\exc^{\ev}(B) = \{\beta\}$ and $\psi_B(\gamma) = \psi(\gamma-\beta)$ for all $\gamma\in\Gamma\setminus\{\beta\}$, or~$\exc^{\ev}(B) = \emptyset$ and
$\psi_B(\gamma) =\min\!\big( \psi(\gamma-vb),vh\big)$  for all $\gamma\in\Gamma$.
\end{cor}
\begin{proof}By \eqref{eq:14.0.1}  %[ADH,~14.0.1 and remarks following it] 
we have an immediate newtonian extension $L$ of~$K$. Then $L$ is still a real closed $H$-field   [ADH, 10.5.8, 3.5.19], and $L[\imag]$ is newtonian by \eqref{eq:14.5.7}, so Proposition~\ref{prop:psiB} applies to $L$ in place of $K$. 
\end{proof}

\subsection*{Higher-order operators} We begin with the following observation:

\begin{lemma}\label{lem:exce product}
Let $B\in K[\der]^{\neq}$ and $\gamma\in\Gamma$. Then
$\nwt_{AB}(\gamma)\geq \nwt_B(\gamma)$, and
$$ \gamma\notin\exc^{\ev}(B)\ \Longrightarrow\ 
 \nwt_{AB}(\gamma)\ =\ \nwt_A\!\big( v^{\ev}_B(\gamma) \big) \text{ and }\ v^{\ev}_{AB}(\gamma)\ =\ v^{\ev}_A\big(v^{\ev}_B(\gamma)\big).$$
\end{lemma}
\begin{proof}
We have 
$\nwt_{AB}(\gamma)=\dwt_{(AB)^\phi}(\gamma)$ eventually, and 
$(AB)^\phi=A^\phi B^\phi$. Hence by [ADH, 5.6] and the definition of $v^{\ev}_B(\gamma)$ in  \eqref{eq:vAev}:
\begin{align*}
\nwt_{AB}(\gamma)\	&=\ \dwt_{A^\phi}\!\big(v_{B^\phi}(\gamma)\big)+\dwt_{B^\phi}(\gamma) \\
					&=\ \dwt_{A^\phi}\!\big(v_B^{\ev}(\gamma)+\nwt_B(\gamma)v\phi\big) + \nwt_B(\gamma), \text{ eventually}, 
\end{align*}
so $\nwt_{AB}(\gamma)\geq \nwt_B(\gamma)$. Now suppose  $\gamma\notin \exc^{\ev}(B)$. Then $\nwt_B(\gamma)=0$, so  
$$\nwt_{AB}(\gamma) = \dwt_{A^\phi}\!\big(v_B^{\ev}(\gamma)\big)   = \nwt_A\!\big(v_B^{\ev}(\gamma)\big), \qquad\text{eventually.}$$ 
Moreover, $v_{(AB)^\phi}= v_{A^\phi B^\phi} =v_{A^\phi}\circ v_{B^\phi}$, hence using   \eqref{eq:vAev}: 
$$v_{(AB)^\phi}(\gamma)\ =\ v_{A^\phi}\big(v_{B^\phi}(\gamma)\big)\ =\ v_{A^\phi}\big(v_{B}^{\ev}(\gamma)\big), \text{ eventually},$$
and thus eventually 
\begin{align*}
v^{\ev}_{AB}(\gamma)\ 	&=\  
     v_{(AB)^\phi}(\gamma)-\nwt_{AB}(\gamma)v\phi \\
						&=\  v_{A^\phi}\big(v_{B}^{\ev}(\gamma)\big) - \nwt_A\!\big(v_B^{\ev}(\gamma)\big)v\phi\ =\ 
v^{\ev}_A\big(v^{\ev}_B(\gamma)\big).\qedhere
\end{align*}
\end{proof} 

\noindent
Lemmas~\ref{lem:ADH 14.2.7} and~\ref{lem:exce product} yield: 

\begin{cor}\label{cor:exce product}
Let $B\in K[\der]^{\neq}$. Then
$$\exc^{\ev}(AB)\  =\  (v^{\ev}_B)^{-1}\big( \exc^{\ev}(A) \big) \cup \exc^{\ev}(B)$$
and hence
$\abs{\exc^{\ev}(AB)} \leq \abs{\exc^{\ev}(A)}+\abs{\exc^{\ev}(B)}$, 
with equality if $v^{\ev}_B\big(\Gamma\setminus\exc^{\ev}(B)\big)=\Gamma$.
\end{cor}

\noindent
As an easy consequence we have a variant of Corollary~\ref{cor:sum of nwts}:

\begin{cor}\label{cor:size of excev}
If $A$ splits over $K$, then $\abs{\exc^{\ev}(A)}\leq r$. 
\end{cor}

\noindent
To study $v^{\ev}_A$ in more detail we introduce the function 
$$ \psi_A\ \colon\ \Gamma\setminus\exc^{\ev}(A)\to\Gamma, \qquad \gamma\mapsto v_A^{\ev}(\gamma)-\gamma.$$
For monic $A$ of order $1$ this agrees with $\psi_A$ as defined in Proposition~\ref{prop:slow}.  
For~${A=a}$~($a\neq 0$) we have $\exc^{\ev}(A)=\emptyset$ and  $\psi_A(\gamma)=va$ for all $\gamma\in\Gamma$.

\begin{lemma}\label{lem:psiAB}
Let  $B\in K[\der]^{\neq}$ and $\gamma\in\Gamma\setminus\exc^{\ev}(AB)$. Then  
$$\psi_{AB}(\gamma)\ =\ \psi_A\big(v_B^{\ev}(\gamma)\big)+\psi_B(\gamma).$$
\end{lemma}
\begin{proof}
We have $\gamma\notin\exc^{\ev}(B)$ and $v_B^{\ev}(\gamma)\notin\exc^{\ev}(A)$ by Corollary~\ref{cor:exce product}, hence
$$\psi_{AB}(\gamma)\ =\ v_A^{\ev}\big(v_B^{\ev}(\gamma)\big)-\gamma\ =\  v_B^{\ev}(\gamma)+\psi_A\big(v_B^{\ev}(\gamma)\big) - \gamma\ =\ \psi_A\big(v_B^{\ev}(\gamma)\big)+\psi_B(\gamma)$$
by Lemma~\ref{lem:exce product}. 
\end{proof}

\noindent
Thus for $a\neq 0$ and $\gamma\in\Gamma$ we have
$$\psi_{aA}(\gamma)=va+\psi_A(\gamma)\text{ if $\gamma\notin\exc^{\ev}(A)$,}\quad \psi_{Aa}(\gamma)=\psi_A(va+\gamma)+va \text{ if 
$\gamma\notin\exc^{\ev}(A)-va$.}$$

\begin{example} Suppose $K$ has small derivation and $x\in K$, $x'\asymp 1$. Then $vx<0$ and~$\exc^{\ev}(\der^2)=\{vx,0\}$, and~$\psi_{\der^2}(\gamma)=\psi\big(\gamma+\psi(\gamma)\big)+\psi(\gamma)$ for $\gamma\in\Gamma\setminus\exc^{\ev}(\der^2)$.
\end{example}

\begin{lemma}\label{lem:v(A(y)) convex subgp}
Suppose $\psi_A$ is slowly varying. Let $\Delta$ be a convex subgroup of $\Gamma$ and let  $y,z\in K^\times$ be such that $vy,vz\notin\exc^{\ev}(A)$.
Then  $$v_\Delta (y) < v_\Delta (z) \ \Longleftrightarrow\ v_\Delta\big(A(y)\big) < v_\Delta\big(A(z)\big).$$ 
\end{lemma}
\begin{proof}
By Lemma~\ref{lem:ADH 14.2.7} we have 
$$v\big(A(y)\big)-v\big(A(z)\big)\ =\ v_A^{\ev}(vy)-v_A^{\ev}(vz)\ =\ vy-vz+\psi_A(vy)-\psi_A(vz)$$
and $\psi_A(vy)-\psi_A(vz)=o(vy-vz)$ if $vy\ne vz$.
\end{proof}

\noindent
Call $A$ {\bf asymptotically surjective} \index{asymptotically surjective} \index{linear differential operator!asymptotically surjective}\index{surjective!asymptotically}
if $v_A^{\ev}\big( \Gamma\setminus\exc^{\ev}(A) \big) = \Gamma$  and 
$\psi_A$ is slowly varying.
If $A$ is asymptotically surjective, then so are~$aA$ and $Aa$ for $a\neq 0$, and
if $A$ has order $0$, then $A$ is asymptotically surjective.
If $K$ is $\upl$-free and $A$ has order $1$, then~$A$ is asymptotically surjective, thanks to Proposition~\ref{prop:slow} and Lemma~\ref{lem:slow}.

\noindent
The next lemma has an obvious proof.

\begin{lemma}\label{lem:slowly varying}
Let $G$ be an ordered abelian group and $U, V\subseteq G$. If $\eta_1,\eta_2\colon U\to G$ are slowly varying,
then so is $\eta_1+\eta_2$. If $\eta\colon U\to G$ and $\zeta\colon V\to G$ 
are slowly varying and $\gamma+\zeta(\gamma)\in U$ for all $\gamma\in V$, then the function~$\gamma\mapsto \eta\big(\gamma+\zeta(\gamma)\big)\colon V\to G$ is also slowly varying.
\end{lemma}

\begin{lemma}
If $A$ and $B\in K[\der]^{\neq}$ are asymptotically surjective, then so is $AB$. 
\end{lemma}
\begin{proof}
Let $A$,~$B$ be asymptotically surjective and $\gamma\in \Gamma$. This gives $\alpha\in\Gamma\setminus\exc^{\ev}(A)$
with $v_A^{\ev}(\alpha)=\gamma$ and $\beta\in\Gamma\setminus\exc^{\ev}(B)$
with $v_B^{\ev}(\beta)=\alpha$. Then $\beta\notin\exc^{\ev}(AB)$
by Corollary~\ref{cor:exce product}, and $v_{AB}^{\ev}(\beta)=\gamma$ by Lemma~\ref{lem:exce product}. Moreover, $\psi_{AB}$ is slowly varying
by Lemmas~\ref{lem:psiAB} and \ref{lem:slowly varying}. 
\end{proof}

\noindent
A straightforward induction on $r$ using this lemma yields:

\begin{cor}\label{cor:well-behaved}
If $K$ is $\upl$-free and $A$ splits over $K$, then $A$ is asymptotically surjective. 
\end{cor}

\noindent
We can now add to Lemma~\ref{lem:ADH 14.2.7}:

\begin{cor}\label{cor1524}
Suppose $K$ is $\upo$-free. Then $A$ is asymptotically surjective.
\end{cor}
\begin{proof}
By the second part of Lemma~\ref{lem:ADH 14.2.7} it is enough to show that $\psi_A$ is slowly varying.
For this we may replace $K$ by any $\upo$-free extension $L$ of~$K$ with $\Psi$ cofinal in $\Psi_L$.
Thus we can arrange by~\eqref{eq:14.0.1} and \eqref{eq:14.5.7}  that~$K$ is newtonian, and by passing
to the algebraic closure,  algebraically closed. Then~$A$ splits over $K$ by  \eqref{eq:14.5.3} and Lemma~\ref{lem:ADH 14.2.5}, 
so $A$ is asymptotically surjective by Corollary~\ref{cor:well-behaved}.
\end{proof}

\section{Special Elements}\label{sec:special elements} 

\noindent
Let $K$ be a valued field and let $\hat a$ be an element of an immediate extension of~$K$ with~$\hat a\notin K$.
{\samepage Recall that $$v(\hat a-K)\ =\ \big\{v(\hat a-a):a\in K\big\}$$ is a nonempty downward closed subset of $\Gamma:= v(K^\times)$ without a largest element.}
Call~$\hat a$ {\it special}\/ over~$K$ if \index{special} \index{element!special}
some nontrivial  convex subgroup of~$\Gamma$ is cofinal in~$v({\hat a-K})$~[ADH, p.~167].
In this case $v(\hat a-K)\cap\Gamma^>\neq\emptyset$, and there is a unique such nontrivial convex subgroup $\Delta$ of $\Gamma$,
namely
$$\Delta\ =\ \big\{ \delta\in\Gamma:\, \abs{\delta}\in v(\hat a-K) \big\}.$$
We also call~$\hat a$ {\it almost special}\/ over $K$ if $\hat a/\fm$ is special over $K$ for some~$\fm\in K^\times$. \index{special!almost} \index{element!almost special}
If $\Gamma\neq\{0\}$ is archimedean, then $\hat a$ is special over $K$ iff $v(\hat a-K)=\Gamma$, iff $\hat a$ is the limit of a divergent c-sequence in $K$. (Recall that ``c-sequence'' abbreviates ``cauchy sequence'' [ADH, p.~82].)
In the next lemma $a$ ranges over $K$ and $\fm$, $\fn$ over $K^\times$.

\begin{lemma}\label{lem:special refinement}
Suppose $\hat a\prec \fm$ and $\hat a/\fm$ is special over $K$. Then for all $a$, $\fn$, if~$\hat a-a\prec\fn\preceq\fm$, then $(\hat a-a)/\fn$ is special over $K$.
\end{lemma}
\begin{proof}
Replacing $\hat a$, $a$, $\fm$, $\fn$ by $\hat a/\fm$, $a/\fm$, $1$, $\fn/\fm$, respectively, we arrange $\fm=1$. So let $\hat a$ be special over $K$ with $\hat a \prec 1$. It is enough to show: (1)~$\hat a-a$ is special over $K$, for all $a$; (2)~for all $\fn$, if $\hat a\prec\fn\preceq 1$, then $\hat a/\fn$ is special over $K$. Here~(1) follows from $v(\hat a-a-K)=v(\hat a-K)$. For (2), note that if $\hat a\prec\fn\preceq 1$, 
then $v\fn\in\Delta$ with $\Delta$ as above, and so $v(\hat a/\fn-K)=v(\hat a-K)-v\fn=v(\hat a-K)$.
\end{proof}

\noindent
The remainder of this section is devoted to showing that (almost) special elements   arise naturally in the analysis of
certain immediate $\d$-algebraic extensions of valued differential fields.  
We first treat the case of asymptotic fields with small derivation, and then 
focus on the linearly newtonian $H$-asymptotic case. 

We  recall some notation: for an ordered abelian group $\Gamma$ and  $\alpha\in \Gamma_{\infty}$, 
${\beta\in \Gamma}$, $\gamma\in \Gamma^{>}$
we mean by ``$\alpha \ge \beta+o(\gamma)$'' that
$\alpha\ge \beta-(1/n)\gamma$ for all $n\ge 1$, while~``${\alpha < \beta +o(\gamma)}$'' is its negation, that is,
$\alpha < \beta-(1/n)\gamma$ for some~$n\ge 1$; see [ADH, p.~312].  Here and later inequalities are
in the sense of the ordered divisible hull $\Q\Gamma$ of the relevant $\Gamma$.

\subsection*{A source of special elements}  {\it In this subsection~$K$ is an asymptotic field with small derivation, value group $\Gamma=v(K^\times)\ne \{0\}$, and differential residue field $\k$; we also let $r\in \N^{\ge 1}$.}\/ Below we use the notion {\em neatly surjective\/} from~[ADH,~5.6]: ${A\in K[\der]^{\ne}}$ is neatly surjective iff for all $b\in K^\times$ there exists $a\in K^\times$ with~$A(a)=b$ and $v_A(va)=vb$.\index{linear differential operator!neatly surjective}\index{neatly surjective}\index{surjective!neatly} 
For use in the next proof, recall from~[ADH, 7.1] the notion of a valued differential field being {\it $r$-differential-hen\-selian}\/ (or {\it $r$-$\d$-henselian,}\/ for short).\index{r-d-henselian@$r$-$\d$-henselian}\index{valued differential field!r-d-henselian@$r$-$\d$-henselian}
% if $\k$ is $r$-linearly surjective and
%  each~${P\in\mathcal O\{Y\}}$ of order at most $r$ and with $P_0 \prec 1$ and~$P_1 \asymp 1$ has a zero $y \prec 1$ in $K$.\index{asymptotic field!r-d-henselian@$r$-$\d$-henselian}\index{r-d-henselian@$r$-$\d$-henselian}
We often let $\hat f$ be an element in an immediate asymptotic extension~$\hat K$ of $K$, but in the statement of the next lemma we take $\hat f\in K$: 

\begin{lemma}\label{neat1} Assume $\k$ is $r$-linearly surjective, $A\in K[\der]^{\ne}$ of order $\le r$ is neatly surjective, $\gamma\in \Q\Gamma$, $\gamma>0$, $\hat f\in K^\times$, and $v\big(A(\hat f)\big)\ge v(A\hat f)+\gamma$. Then $A(f)=0$ and~$v(\hat f-f)\ge v(\hat f)+ \gamma+o(\gamma)$ for some $f\in K$. 
\end{lemma}
\begin{proof} Set $B:= g^{-1}A\hat f$, where we take $g\in K^\times$ such that $vg=v(A\hat f)$. Then~$B\asymp 1$, $B$ is still neatly surjective, and $B(1)=g^{-1}A(\hat f)$, $v\big(B(1)\big)\ge \gamma$. It suffices to find~$y\in K$ such that $B(y)=0$ and $v(y-1)\ge \gamma+o(\gamma)$, because then $f:=\hat fy$ has the desired property.
If $B(1)=0$, then $y=1$ works, so assume $B(1)\ne 0$. By~[ADH, 7.2.7] we have an immediate extension $\hat{K}$ of $K$ that is $r$-differential henselian. Then~$\hat{K}$ is asymptotic by [ADH, 9.4.2 and 9.4.5]. 
Set $R(Z):= \Ric(B)\in K\{Z\}$.
Then the proof of [ADH, 7.5.1] applied to $\hat{K}$ and~$B$ in the roles
of $K$ and $A$ yields~$z\prec 1$ in $\hat{K}$ with $R(z)=0$. Now $R(0)=B(1)$,  
hence by [ADH, 7.2.2] we can take such $z$ with $v(z)  \ge  \beta+o(\beta)$ where $\beta:=v\big(B(1)\big)\ge \gamma$. 
As in the proof of
[ADH, 7.5.1] we next take $y\in \hat{K}$ with $v(y-1)>0$ and
$y^\dagger=z$ to get~$B(y)=0$, and observe that then
$v(y-1)\ge \beta+o(\beta)$,  by [ADH, 9.2.10(iv)], hence $v(y-1)\ge \gamma+o(\gamma)$. It remains to note that $y\in K$ by [ADH, 7.5.7].  
\end{proof} 

\noindent
By a remark following the proof of [ADH, 7.5.1] the assumption that $\k$ is $r$-linearly surjective in the lemma above can be replaced for $r\ge 2$ by the assumption that $\k$ is $(r-1)$-linearly surjective. 

\medskip
\noindent
Next we establish a version of the above with $\hat f$ in an immediate asymptotic extension of $K$. Recall that an asymptotic extension of $K$ with the same value group as~$K$ has small derivation, by [ADH, 9.4.1].\index{r-linearly surjective@$r$-linearly surjective!valuation ring}

\begin{lemma}\label{neat2} Assume $\k$ is $r$-linearly surjective, $A\in K[\der]^{\ne}$ of order $\le r$ is neatly surjective, $\gamma\in \Q\Gamma$, $\gamma > 0$, $\hat{K}$ is an immediate asymptotic extension of $K$, $\hat f\in \hat K^\times$, and~$v\big(A(\hat f)\big)\ge v(A\hat f)+\gamma$. Then for some $f\in K$ we have
$$A(f)\ =\ 0, \qquad v(\hat f-f)\ \ge\ v(\hat f)+\gamma+o(\gamma).$$ 
\end{lemma} 
\begin{proof} By extending $\hat{K}$ we can arrange that 
$\hat{K}$ is $r$-differential henselian, 
so $A$ remains neatly surjective as an element of $\hat{K}[\der]$, by [ADH, 7.1.8]. Then by Lem\-ma~\ref{neat1} with $\hat{K}$ in the role of $K$ we get $f\in \hat{K}$ such that $A(f)=0$ and
$v(\hat f-f)\ge v(\hat f)+\gamma+o(\gamma)$. It remains to note that
$f\in K$ by [ADH, 7.5.7]. 
\end{proof}

\noindent
We actually need an inhomogeneous variant of the above: 

\begin{lemma}\label{neat3} Assume $\k$ is $r$-linearly surjective, $A\in K[\der]^{\ne}$ of order $\le r$ is neatly surjective, $b\in K$, $\gamma\in \Q\Gamma$, $\gamma>0$, $v(A)=o(\gamma)$, $v(b)\ge o(\gamma)$, $\hat{K}$ is an immediate asymptotic extension of $K$, $\hat f\in \hat K$,
$\hat f\preceq 1$, and $v\big(A(\hat f)-b\big)\ge \gamma+o(\gamma)$. 
Then 
$$A(f)\ =\ b,\qquad v(\hat f -f)\ \ge\ (1/2)\gamma +o(\gamma)$$
for some $f\in K$. 
\end{lemma}
\begin{proof} Take $y\in K$ with $A(y)=b$ and $v(y)\ge o(\gamma)$. Then $A(\hat g)=A(\hat f)-b$ for~$\hat g:= \hat f-y$, so $v\big(A(\hat g)\big)\ge \gamma+o(\gamma)$ and $v(\hat g)\ge o(\gamma)$. We distinguish two cases:

\medskip\noindent
(1) $v(\hat g)\ge (1/2)\gamma+o(\gamma)$. Then $v(\hat f - y)\ge (1/2)\gamma+o(\gamma)$, so $f:= y$ works. 

\medskip\noindent
(2) $v(\hat g) < (1/2)\gamma+o(\gamma)$. Then by [ADH, 6.1.3],
$$v(A\hat g)\ <\ (1/2)\gamma+o(\gamma), \qquad v\big(A(\hat g)\big)\ \ge\ \gamma+o(\gamma),$$ so
$v\big(A(\hat g)\big)\ge v(A\hat g) + (1/2)\gamma$. Then 
Lemma~\ref{neat2}
gives an element $g\in K$ such that~$A(g)=0$ and 
$v(\hat g-g)\ge (1/2)\gamma+o(\gamma)$. Hence $f:= y+g$ works. 
\end{proof}

\noindent
Recall from [ADH, 7.2] that $\O$ is said to be {\it $r$-linearly surjective} if for every $A$ in~$K[\der]^{\neq}$ of order $r$ with $v(A)=0$ there exists $y\in\O$ with $A(y)=1$.

\begin{prop}\label{propsp} Assume $\O$ is $r$-linearly surjective,
$P\in K\{Y\}$, $\order(P)\le r$, $\ddeg P=1$, 
and $P(\hat a)=0$, where $\hat a\preceq 1$ lies in an immediate asymptotic extension of $K$ and $\hat a\notin K$. Then $\hat a$ is special over $K$. 
\end{prop}
\begin{proof} The hypothesis on $\O$ yields: $\k$ is $r$-linearly surjective and all $A\in K[\der]^{\ne}$
of order $\le r$ are neatly surjective. Let $0 <\gamma\in v(\hat a -K)$; we claim that $v(\hat a -K)$ has an element $\ge (4/3)\gamma$. 
We arrange $P\asymp 1$. Take $a\in K$ with $v(\hat a -a)=\gamma$. Then~${P_{+a}\asymp 1}$, $\ddeg P_{+a}=1$, so 
$$P_{+a,1}\ \asymp\ 1, \quad P_{+a,>1}\ \prec\ 1, \quad P_{+a}\ =\ P(a) + P_{+a,1} + P_{+a,>1}$$
and
$$0\ =\ P(\hat a)\ =\ P_{+a}(\hat a -a)\ =\ P(a) + P_{+a,1}(\hat a -a) + P_{+a,>1}(\hat a -a),$$
with $$v\big(P_{+a,1}(\hat a -a)+P_{+a,>1}(\hat a -a)\big)\ge \gamma+o(\gamma),$$ and thus $v(P(a))\ge \gamma+o(\gamma)$.
Take $g\in K^\times$ with $vg=\gamma$ and set $Q:= g^{-1}P_{+a,\times g}$, so $Q=Q_0+ Q_1 + Q_{>1}$ with
$$Q_0\ =\ Q(0)\ =\ g^{-1}P(a),\quad Q_1\ =\ g^{-1}(P_{+a,1})_{\times g},\quad Q_{>1}\ =\ g^{-1}(P_{+a,>1})_{\times g}, $$
hence 
$$v(Q_0)\ge o(\gamma), \quad v(Q_1)=o(\gamma), \quad v(Q_{>1})\ge \gamma+o(\gamma).$$ 
We set $\hat f:= g^{-1}(\hat a -a)$, so $Q(\hat f)=0$ and $\hat f\asymp 1$, and $A:=L_Q\in K[\der]$. Then~$Q(\hat f)=0$ gives
$$Q_0+A(\hat f)\ =\ Q_0+Q_1(\hat f)\ =\ -Q_{>1}(\hat f),\ \text{ with }v\big(Q_{>1}(\hat f)\big)\ \ge\ \gamma+o(\gamma),$$ so $v\big(Q_0+A(\hat f)\big)\ge \gamma+o(\gamma)$. Since $v(A)=v(Q_1)=o(\gamma)$, Lemma~\ref{neat3} then gives~$f\in K$ with $v(\hat f - f)\ge (1/3)\gamma$.
In view of $\hat a-a=g\hat f$, this yields
$$v\big(\hat a -(a+gf)\big)\ =\ \gamma+v(\hat f - f)\ \ge\ (4/3)\gamma,$$ 
which proves our claim. It gives the desired result.    
\end{proof}

\subsection*{A source of almost special elements} {\it In this subsection
$K$, $\Gamma$, $\k$, and $r$ are as in the previous subsection, and we assume that $\O$ is
$r$-linearly surjective.}\/ (So $\k$ is $r$-linearly surjective, and $\sup \Psi=0$ by [ADH, 9.4.2].) 
Let 
$\hat a$ be an element in an immediate asymptotic extension of $K$ such that $\hat a\notin K$ and 
$K\<\hat a\>$ has transcendence degree $\le r$
over $K$. We shall use Proposition~\ref{propsp} to show:

\begin{prop} \label{propalsp}  
If $\Gamma$ is divisible, then $\hat a$ is almost special over $K$. 
\end{prop} 

\noindent
Towards the proof we first note that $\hat a$ has a minimal annihilator $P(Y)$\index{minimal!annihilator}\index{minimal!differential polynomial}\index{differential polynomial!minimal}
 over $K$ of order~$\le r$. We also fix a divergent pc-sequence $(a_{\rho})$ in $K$ such that $a_{\rho}\leadsto \hat a$. 
  (See [ADH, 4.1] for ``minimal annihilator'', and [ADH, 4.4] for ``minimal differential polynomial of $(a_\rho)$ over $K$''.)
 We next show how to improve $\hat a$ and $P$ (without assuming divisibility of $\Gamma$):

\begin{lemma}\label{replace} For some $\hat{b}$ in an immediate asymptotic extension of $K$ we have: \begin{enumerate}
\item[\textup{(i)}] $v(\hat a -K)=v(\hat b -K)$;
\item[\textup{(ii)}] $(a_{\rho})$ has a minimal differential polynomial $Q$ over $K$ of order~$\le r$ such that~$Q$
is also a minimal annihilator of $\hat b$ over $K$.
\end{enumerate}
\end{lemma}
\begin{proof} By [ADH, 6.8.1, 6.9.2], $(a_{\rho})$ is of $\d$-algebraic type over $K$ with a minimal differential polynomial $Q$ over $K$ such that $\order Q \le \order P\le r$. 
By [ADH, 6.9.3, 9.4.5] this gives an element $\hat b$ in an immediate asymptotic extension of $K$ such that~$Q$ is a minimal annihilator of $\hat b$ over $K$ and $a_{\rho}\leadsto \hat b$. Then $Q$ and $\hat b$ have the desired properties. 
\end{proof}

\begin{proof}[Proof of Proposition~\ref{propalsp}] Replace $\hat a$ and $P$
by $\hat b$ and $Q$ from Lemma~\ref{replace} (and rename) to arrange
that $P$ is a minimal differential polynomial of $(a_{\rho})$ over $K$.
Now assuming $\Gamma$ is divisible,~\cite[Proposition~3.1]{Nigel} gives $a\in K$ and $g\in K^\times$ such that~$\hat a -a\asymp g$ and $\ddeg P_{+a,\times g} = 1$. 

Set $F:= P_{+a,\times g}$ and
$\hat f:= (\hat a  -a)/g$. Then $\ddeg F=1$, $F(\hat f)=0$, and $\hat f\preceq 1$. Applying Proposition~\ref{propsp} to
$F$ and $\hat f$ in the role of $P$ and $\hat a$ yields a
nontrivial convex subgroup $\Delta$ of $\Gamma$ that is
cofinal in $v(\hat f -K)$. Setting $\alpha:= vg$, it follows that~$\alpha+\Delta$ is cofinal in $v\big((\hat a -a)-K\big)=v(\hat a -K)$.   
\end{proof}

\noindent
We can trade the divisibility assumption in Proposition~\ref{propalsp} against a stronger hypothesis on $K$, the proof  using \cite[3.3]{Nigel} instead of \cite[3.1]{Nigel}:

\begin{cor}
If $K$ is henselian and $\k$ is linearly surjective, then $\hat a$ is almost special over~$K$. 
\end{cor}

\subsection*{The linearly newtonian setting} {\em In this subsection $K$ is an $\upo$-free $r$-linearly newtonian $H$-asymptotic field,
$r\ge 1$.}\/ Thus $K$ is
$\d$-valued by Lemma~\ref{lem:ADH 14.2.5}. We let~$\phi$ range over the elements active in $K$.
We now mimick the material in the previous two subsections. Note that for $A\in K[\der]^{\ne}$ and any element $\hat f$ in an asymptotic extension of $K$ we have $A(\hat f)\preceq A^\phi \hat f$, since
$A(\hat f)=A^\phi(\hat f)$.

\begin{lemma}\label{neneat1} Assume that $A\in K[\der]^{\ne}$ has order $\le r$, $\gamma\in \Q\Gamma$, $\gamma>0$, ${\hat f\in K^\times}$, and $v\big(A(\hat f)\big)\ge v(A^\phi\hat f)+\gamma$, eventually. Then there exists an $f\in K$ with~${A(f)=0}$ and $v(\hat f-f)\ge v(\hat f) +\gamma+o(\gamma)$. 
\end{lemma}
\begin{proof} Take $\phi$ such that $v\phi\ge \gamma^\dagger$ and $v\big(A(\hat f)\big)\ge v(A^\phi\hat f)+\gamma$. 
Next, take $\beta\in \Gamma$ such that  $\beta\ge \gamma$ and $v\big(A(\hat f)\big)\ge v(A^\phi\hat f)+\beta$.
Then $v\phi\ge \beta^\dagger$, so $\beta> \Gamma_{\phi}^{\flat}$, hence the valuation ring of the flattening $(K^\phi, v_{\phi}^\flat)$ is $r$-linearly surjective, by [ADH, 14.2.1]. 
We now apply
Lemma~\ref{neat1} to $$(K^\phi,v_{\phi}^\flat),\quad
A^\phi,\quad \dot{\beta}:= \beta+\Gamma_{\phi}^{\flat}$$ in the role of $K$, $A$, $\gamma$ to give $f\in K$ with
$A(f)=0$ and $v_{\phi}^{\flat}(\hat f -f)\ge v_{\phi}^{\flat}(\hat f) + \dot{\beta} + o(\dot{\beta})$.
Then also $v(\hat f-f)\ge v(\hat f) +\beta+o(\beta)$, and thus $v(\hat f -f)\ge v(\hat f) + \gamma+o(\gamma)$.  
\end{proof}

\begin{lemma}\label{neneat2} Assume $A\in K[\der]^{\ne}$ has order $\le r$,  $\hat{K}$ is an immediate $\d$-algebraic asymptotic extension of $K$, $\gamma\in \Q\Gamma$, $\gamma>0$, $\hat f\in \hat K^\times$, and  $v\big(A(\hat f)\big)\ge v(A^\phi\hat f)+\gamma$ eventually. Then $A(f)=0$ and $v(\hat f-f)\ge v(\hat f) +\gamma+o(\gamma)$ for some $f\in K$.
\end{lemma} 
\begin{proof} Since $K$ is $\upo$-free, so is $\hat K$ by Theorem~\ref{thm:ADH 13.6.1}. By \eqref{eq:14.0.1}  %[ADH, 14.0.1 and subsequent remarks] 
we can extend $\hat{K}$ to arrange that 
$\hat{K}$ is also newtonian. Then by Lemma~\ref{neneat1} with $\hat{K}$ in the role of $K$ we get $f\in \hat{K}$ with $A(f)=0$ and~$v(\hat f-f)\ge v(\hat f)+\gamma+o(\gamma)$. Now use that
$f\in K$ by [ADH, line before 14.2.10]. 
\end{proof}

\begin{lemma}\label{neneat3} Assume $A\in K[\der]^{\ne}$ has order $\le r$, $b\in K$, $\gamma\in \Q\Gamma$, $\gamma>0$, $\hat{K}$ is an immediate $\d$-algebraic asymptotic extension of $K$, and $\hat f\in \hat K$, $v(\hat f)\ge o(\gamma)$. Assume also that
eventually $v(b)\ge v(A^\phi)+o(\gamma)$ and $v\big(A(\hat f)-b\big)\ge v(A^\phi)+\gamma+o(\gamma)$. 
Then for some~$f\in K$ we have $A(f)=b$ and
$v(\hat f -f)\ge (1/2)\gamma+o(\gamma)$.
\end{lemma}
{\sloppy\begin{proof} We take $y\in K$ with $A(y)=b$ as follows: If $b=0$, then $y:=0$. If~$b\ne 0$, then Corollary~\ref{cor:14.2.10, generalized} yields $y\in K^\times$ such that
$A(y)=b$, $vy\notin \exc^{\ev}(A)$, and~${v_A^{\ev}(vy)=vb}$. In any case, $vy\ge o(\gamma)$: when $b\ne 0$, the sentence preceding
[ADH, 14.2.7] gives~$v_{A^{\phi}}(vy)=vb$, eventually, 
to which we apply [ADH, 6.1.3].
 
 Now $A(\hat g)=A(\hat f)-b$ for $\hat g:= \hat f-y$, so
$v(\hat g)\ge o(\gamma)$, and eventually $v\big(A(\hat g)\big)\ge v(A^\phi)+\gamma+o(\gamma)$. We distinguish two cases:

\medskip\noindent
(1) $v(\hat g)\ge (1/2)\gamma+o(\gamma)$. Then $v(\hat f - y)\ge (1/2)\gamma+o(\gamma)$, so $f:= y$ works. 

\medskip\noindent
(2) $v(\hat g) < (1/2)\gamma+o(\gamma)$. Then by [ADH, 6.1.3] we have eventually
$$v(A^\phi\hat g)\ <\ v(A^\phi)+(1/2)\gamma+o(\gamma), \qquad v\big(A(\hat g)\big)\ \ge\ v(A^\phi)+\gamma+o(\gamma),$$ so
$v(A(\hat g))\ge v(A^\phi\hat g) + (1/2)\gamma$, eventually.   
Lemma~\ref{neneat2}
gives an ele\-ment~${g\in K}$ with $A(g)=0$ and 
$v(\hat g-g)\ge (1/2)\gamma+o(\gamma)$. Hence $f:= y+g$ works. 
\end{proof}}

{\sloppy
\begin{prop}\label{nepropsp}\label{prop:hata special} Suppose that
$P\in K\{Y\}$, $\order P\le r$, $\ndeg P=1$, 
and~$P(\hat a)=0$, where $\hat a\preceq 1$ lies in an immediate asymptotic extension of $K$ and~$\hat a\notin K$. Then $\hat a$ is special over $K$. 
\end{prop}}

\noindent
The proof is like that of Proposition~\ref{propsp}, but there are some differences that call for further details.

\begin{proof} Given $0 <\gamma\in v(\hat a -K)$, we claim that $v(\hat a -K)$ has an element $\ge (4/3)\gamma$. 
Take $a\in K$ with $v(\hat a -a)=\gamma$. Then $\ndeg P_{+a}=1$ by [ADH, 11.2.3(i)], so eventually we have
$$P(a)\ \preceq\ P^\phi_{+a,1}\ \succ\ P^\phi_{+a,>1}, \quad P^\phi_{+a}\ =\ P(a) + P^\phi_{+a,1} + P^\phi_{+a,>1}$$
and
\begin{align*} 
0\  =\ P(\hat a) &\ =\ P^\phi_{+a}(\hat a -a) \\ \ &\ =\ P(a) +  P^\phi_{+a,1}(\hat a -a) + P^\phi_{+a,>1}(\hat a -a),\\
&\phantom{=\ P(a)+} v\big(P^\phi_{+a,1}(\hat a -a)+P^\phi_{+a,>1}(\hat a -a)\big)\ \ge\ v(P^\phi_{+a,1})+\gamma+o(\gamma),
\end{align*}
and thus eventually $v\big(P(a)\big)\ \ge\ v(P^\phi_{+a,1})+\gamma+o(\gamma)$.
Take $g\in K^\times$ with $vg=\gamma$ and set $Q:= g^{-1}P_{+a,\times g}$, so $Q=Q_0+ Q_1 + Q_{>1}$ with
$$Q_0\ =\ Q(0)\ =\ g^{-1}P(a),\quad Q_1\ =\ g^{-1}(P_{+a,1})_{\times g},\quad Q_{>1}\ =\ g^{-1}(P_{+a,>1})_{\times g}. $$
Then $v(Q_0)=v\big(P(a)\big)-\gamma\ge v(P^\phi_{+a,1})+o(\gamma)$,
eventually. By [ADH, 6.1.3], $$v(Q_1^\phi)\ =\ v(P_{+a,1}^\phi)+o(\gamma),\qquad  v(Q_{>1}^\phi)\ \ge\ v(P_{+a,>1}^\phi)+\gamma+o(\gamma)$$
for all $\phi$. Since $P^\phi_{+a, >1}\preceq P^\phi_{+a,1}$, eventually, the last two displayed inequalities give~$v(Q^\phi_{>1})\ge v(Q^\phi_1)+\gamma+o(\gamma)$, eventually. We set $\hat f:= g^{-1}(\hat a -a)$, so $Q(\hat f)=0$ and~$\hat f\asymp 1$. Set $A:=L_Q\in K[\der]$. Then $Q(\hat f)=0$ gives
$$Q_0+A(\hat f)\ =\ Q_0+Q_1(\hat f)\ =\ -Q^\phi_{>1}(\hat f),$$
with $v\big(Q^\phi_{>1}(\hat f)\big)\ge v(Q^\phi_1)+\gamma+o(\gamma)$, eventually, so 
$$v\big(Q_0+A(\hat f)\big)\ \ge\ v(A^\phi)+\gamma+o(\gamma),\quad \text{eventually}.$$ 
Moreover, $v(Q_0)\ge v(A^\phi)+o(\gamma)$, eventually. Lemma~\ref{neneat3} then gives $f\in K$ with~$v(\hat f - f)\ge (1/3)\gamma$.
In view of $\hat a-a=g\hat f$, this yields
$$v\big(\hat a -(a+gf)\big)\ =\ \gamma+v(\hat f - f)\ \ge\ (4/3)\gamma,$$ 
which proves our claim.
\end{proof}

\noindent
{\it In the rest of this subsection we assume that $\hat a\notin K$ lies in an immediate asymptotic extension of $K$ and $K\<\hat a\>$ has transcendence degree $\le r$ over $K$.}\/

\begin{prop}\label{npropalsp}\label{prop:hata almostspecial}  If $\Gamma$ is divisible, then 
$\hat a$ is almost special over $K$.  
\end{prop} 

\noindent
Towards the proof, we fix  a minimal annihilator $P(Y)$ of $\hat a$
 over $K$, so $\order P\le r$. We also fix a divergent pc-sequence $(a_{\rho})$ in $K$ such that $a_{\rho}\leadsto \hat a$. We next show how to improve $\hat a$ and $P$ if necessary:

\begin{lemma}\label{nreplace} For some $\hat{b}$ in an immediate asymptotic extension of $K$ we have: \begin{enumerate}
\item[\textup{(i)}] $v(\hat a -a)=v(\hat b -a)$ for all $a\in K$;
\item[\textup{(ii)}] $(a_{\rho})$ has a minimal differential polynomial $Q$ over $K$ of order $\le r$ such that~$Q$
is also a minimal annihilator of $\hat b$ over $K$.
\end{enumerate}
\end{lemma}
\begin{proof} By the remarks following the proof of [ADH, 11.4.3] we have $P\in Z(K,\hat a)$.  Take $Q\in Z(K,\hat a)$ of minimal complexity.
Then $\order Q \le \order P\le r$, and $Q$ is a minimal differential polynomial of $(a_{\rho})$ over $K$ by [ADH, 11.4.13]. 
By [ADH, 11.4.8 and its proof] this gives an element $\hat b$ in an immediate asymptotic extension of $K$ such that (i) holds and  $Q$ is a minimal annihilator of $\hat b$ over $K$. Then $Q$ and $\hat b$ have the desired properties. 
\end{proof}

\begin{proof}[Proof of Proposition~\ref{npropalsp}] 
Assume $\Gamma$ is divisible. Replace $\hat a$, $P$
by $\hat b$, $Q$ from Lemma~\ref{nreplace} and rename to arrange
that $P$ is a minimal differential polynomial of $(a_{\rho})$ over $K$.
By [ADH, 14.5.1] we have $a\in K$ and $g\in K^\times$ such that $\hat a -a\asymp g$ and~$\ndeg P_{+a,\times g} = 1$.
Set $F:= P_{+a,\times g}$ and
$\hat f:= (\hat a  -a)/g$. Then $\ndeg F=1$, $F(\hat f)=0$, and $\hat f\preceq 1$. Applying Proposition~\ref{nepropsp} to
$F$ and $\hat f$ in the role of~$P$ and~$\hat a$ yields a
nontrivial convex subgroup $\Delta$ of $\Gamma$ that is
cofinal in $v(\hat f -K)$. Setting~$\alpha:= vg$, it follows that $\alpha+\Delta$ is cofinal in $v\big((\hat a -a)-K\big)=v(\hat a -K)$.   
\end{proof}

\begin{cor} 
If $K$ is henselian, then $\hat a$ is almost special over $K$.  
\end{cor}

\noindent
The proof is like that of Proposition~\ref{npropalsp}, using \cite[3.3]{Nigel19} instead of~[ADH, 14.5.1].

\subsection*{The case of order $1$\astr} 
We show here that Proposition~\ref{prop:hata special}  goes through in the case of order $1$ under weaker assumptions:
 {\it in this subsection $K$ is a  $1$-linearly newtonian $H$-asymptotic field with asymptotic integration.}\/ Then $K$ is $\d$-valued with~$\I(K)\subseteq K^\dagger$, by Lemma~\ref{lem:ADH 14.2.5}, and $\upl$-free, by [ADH, 14.2.3].
We let $\phi$ range over elements active in~$K$. In the next two lemmas $A\in K[\der]^{\neq}$ has order~$\leq 1$,   $\gamma\in \Q\Gamma$, $\gamma>0$, and $\hat{K}$ is an immediate  asymptotic extension of $K$.

\begin{lemma}\label{neneat2, r=1} Let $\hat f\in \hat K^\times$ be such that  $v\big(A(\hat f)\big)\ge v(A^\phi\hat f)+\gamma$ eventually. Then there exists~$f\in K$ such that
$A(f)=0$ and~$v(\hat f-f)\ge v(\hat f) +\gamma$.
\end{lemma} 
\begin{proof}
Note that $\order(A)=1$; 
we arrange $A=\der-g$ ($g\in K$).
If~$A(\hat f)=0$, then~$\hat f$ is in $K$ [ADH, line before 14.2.10], and $f:=\hat f$ works. 
Assume~$A(\hat f)\neq 0$.
Then $$v\big(A^\phi(\hat f)\big)=v\big(A(\hat f)\big)\ge v(A^\phi\hat f)+\gamma>v(A^\phi\hat f),\quad\text{ eventually,}$$ 
so
$v(\hat f)\in\exc^{\ev}(A)$, and Lemma~\ref{lem:v(ker)=exc, r=1} yields an $f\in K$ with $f\sim \hat f$ and~$A(f)=0$.
We claim that this $f$ has the desired property. 
Set $b:= A(\hat{f})$. By the remarks preceding Corollary~\ref{cor:prlemexc} we can
replace~$K$,~$\hat K$,~$A$,~$b$ by~$K^\phi$,~$\hat K^\phi$,~$\phi^{-1}A^\phi$,~$\phi^{-1}b$, respectively, for suitable $\phi$,
to arrange that $K$ has small derivation and~$b^\dagger-g\succ^\flat 1$.  Using the hypothesis of the lemma we also arrange $vb\geq v(A\hat f)+\gamma$.
It remains to show that for~$\hat g:=\hat f-f\neq 0$ we have $v(\hat g)\ge v(\hat f)+\gamma$. 
Now $A(\hat g)=b$ with~$v(\hat g)\notin\exc^{\ev}(A)$, hence~$\hat g\sim b/(b^\dagger-g) \prec^\flat b$ by Lemma~\ref{prlemexc}, and thus
$v(\hat g)>vb\geq v(A\hat f)+\gamma$, so it is enough to show $v(A\hat f)\ge v(\hat f)$. 
Now $b=A(\hat f)=\hat f(\hat{f}^\dagger-g)$ and  $A\hat f=\hat f\big(\der+{\hat f}^\dagger -g\big)$. As $vb\ge v(A\hat f) + \gamma> v(A\hat f)$, this yields $v(\hat{f}^\dagger-g) > 0$, so $v(A\hat f)=v(\hat f)$.
\end{proof}

\begin{lemma}\label{neneat3, r=1} Let $b\in K$ and $\hat f\in \hat K$ with $v(\hat f)\ge o(\gamma)$. Assume also that
eventually $v(b)\ge v(A^\phi)+o(\gamma)$ and $v\big(A(\hat f)-b\big)\ge v(A^\phi)+\gamma+o(\gamma)$. 
Then for some~$f\in K$ we have $A(f)=b$ and
$v(\hat f -f)\ge (1/2)\gamma+o(\gamma)$.
\end{lemma}

\noindent
The proof is like that of Lemma~\ref{neneat3}, using Lemma~\ref{neneat2, r=1} instead of Lemma~\ref{neneat2}.
In the same way Lemma~\ref{neneat3} gave Proposition~\ref{prop:hata special}, Lemma~\ref{neneat3, r=1} now yields:

\begin{prop}\label{nepropsp, r=1}\label{prop:hata special, r=1} If
$P\in K\{Y\}$, $\order P \le 1$, $\ndeg P=1$, 
and $P(\hat a)=0$, where~$\hat a\preceq 1$ lies in an immediate asymptotic extension of $K$ and $\hat a\notin K$, then~$\hat a$ is special over $K$. 
\end{prop}

\begin{remark} 
Proposition~\ref{prop:hata almostspecial} does not hold for $r=1$ under present assumptions.
To see this, let $K$ be a Liouville closed $H$-field which is not $\upo$-free, as in Example~\ref{ex:Gehret} or \cite{ADH3}.
Then $K$ is $1$-linearly newtonian by Corollary~\ref{cor:Liouville closed => 1-lin newt} below.
Consider the pc-sequences $(\upl_\rho)$ and $(\upo_\rho)$  in $K$ as in [ADH, 11.7],
 let $\upo\in K$ with~$\upo_\rho\leadsto\upo$, and~$P=2Y'+Y^2+\upo$. 
Then [ADH, 11.7.13] gives an element $\upl$ in an immediate asymptotic extension of $K$ but not in $K$
with $\upl_\rho\leadsto\upl$ and $P(\upl)=0$. However, $\upl$ is not almost special over $K$ [ADH, 3.4.13, 11.5.2].
\end{remark}

\subsection*{Relating $Z(K,\hat a)$ and $v(\hat a-K)$ for special $\hat a$} 
{\em In this subsection $K$ is a valued differential field with small derivation $\der\ne 0$
such that $\Gamma\ne \{0\}$ and $\Gamma^{>}$ has no least element.}\/ We recall from \cite{VDF} that a valued differential field extension $L$ of $K$ is said to be {\em strict\/}\index{extension!strict} if for all $\phi\in K^\times$,
$$\der(\smallo)\subseteq \phi \smallo\ \Rightarrow\ \der(\smallo_L)\subseteq \phi \smallo_L, \qquad \der(\O)\subseteq \phi \smallo\ \Rightarrow\ \der(\O_L)\subseteq \phi \smallo_L.$$
(If $K$ is asymptotic, then any immediate asymptotic extension of~$K$ is automatically strict, by \cite[1.11]{VDF}.)
Let $\hat a$ lie in an immediate strict extension of $K$ such that~${\hat a\preceq 1}$, $\hat a\notin K$, and~$\hat a$ is special over $K$.
 We adopt from \cite[Sections~2,~4]{VDF} the definitions of~$\ndeg P$ for $P\in K\{Y\}^{\ne}$ and of the set $Z(K,\hat a)\subseteq K\{Y\}^{\ne}$. 
Also recall that~$\Gamma(\der):=\{v\phi:\, \phi\in K^\times,\, \der(\smallo)\subseteq\phi\smallo\}$.

\begin{lemma}\label{Zp1} Let $P\in Z(K,\hat a)$ and $P\asymp 1$. Then
$v\big(P(\hat a)\big) >  v(\hat a-K)$. 
\end{lemma}
\begin{proof} 
Take a divergent pc-sequence $(a_{\rho})$ in $\mathcal O$ with $a_{\rho} \leadsto \hat a$, and  
as in [ADH, 11.2] let $\boldsymbol  a:=c_K(a_\rho)$. Then 
$\ndeg_{\boldsymbol a} P\geq 1$ by~\cite[4.7]{VDF}. We arrange $\gamma_{\rho}:= v(\hat a-a_{\rho})$ to be strictly increasing as a function of~$\rho$, with $0 < 2\gamma_{\rho} < \gamma_{s(\rho)}$  for all $\rho$. 
Take~$g_{\rho}\in \smallo$ with $g_{\rho} \asymp \hat a - a_{\rho}$; then $1 \leq d:= \ndeg_{\boldsymbol a} P = \ndeg P_{+a_{\rho}, \times g_{\rho}}$ for all sufficiently large~$\rho$, and we arrange that this holds for all $\rho$. 
We have~$\hat a = a_{\rho} + g_{\rho}y_{\rho}$ with $y_{\rho}\asymp 1$, and 
$$P(\hat a)\ =\ P_{+a_{\rho},\times g_{\rho}}(y_{\rho})\ =\ \sum_i (P_{+a_{\rho},\times g_{\rho}})_i(y_{\rho}).$$
Pick for every $\rho$ an element $\phi_{\rho}\in K^\times$ such that $0\le v(\phi_{\rho})\in \Gamma(\der)$ and
$(P^{\phi_{\rho}}_{+a_{\rho},\times g_{\rho}})_i\ \preceq\ (P^{\phi_{\rho}}_{+a_{\rho},\times g_{\rho}})_{d}$
for all $i$. Then for all $\rho$ and $i$, 
\begin{align*} (P_{+a_{\rho},\times g_{\rho}})_i(y_{\rho})\ & = (P^{\phi_{\rho}}_{+a_{\rho},\times g_{\rho}})_i(y_{\rho})\ \preceq\  (P^{\phi_{\rho}}_{+a_{\rho},\times g_{\rho}})_i\ \preceq\ (P^{\phi_{\rho}}_{+a_{\rho},\times g_{\rho}})_{d}\ 
\text{ with }\\
 v\big((P^{\phi_{\rho}}_{+a_{\rho},\times g_{\rho}})_{d}\big)\ &\ge\ d \gamma_{\rho} + o(\gamma_{\rho})\ \ge\ \gamma_{\rho} + o(\gamma_{\rho}),
 \end{align*}
where for the next to last inequality we use [ADH, 11.1.1, 5.7.1, 5.7.5, 6.1.3].  
Hence~$v\big(P(\hat a)\big) \ge \gamma_{\rho} + o(\gamma_{\rho})$ for all $\rho$, and thus $v\big(P(\hat a)\big)> v(\hat a -K)$.
\end{proof}

\noindent
We also have a converse under extra assumptions:

\begin{lemma} \label{lem:ZKhata} Assume $K$ is asymptotic and $\Psi\subseteq v(\hat a-K)$. Let $P\in K\{Y\}$ be such that $P\asymp 1$ and $v\big(P(\hat a)\big)> v(\hat a -K)$. Then $P\in Z(K,\hat a)$.
\end{lemma}
\begin{proof} 
Let $\Delta$ be the nontrivial convex subgroup of $\Gamma$ that is cofinal in $v(\hat a -K)$. Let $\kappa:=\cf(\Delta)$.
Take a divergent pc-sequence $(a_{\rho})_{\rho< \kappa}$ in $K$ such that $a_{\rho}\leadsto \hat a$.   
We arrange $\gamma_{\rho}:= v(\hat a -a_{\rho})$ is strictly increasing as a function of $\rho$, with $\gamma_{\rho}>0$ for all~$\rho$; thus
$a_{\rho}\preceq 1$ for all $\rho$. 
Consider the $\Delta$-coarsening $\dot v=v_{\Delta}$ of the valuation $v$ of~$K$; it has valuation ring $\dot{\O}$ with
differential residue field $\dot K$. Consider likewise the $\Delta$-coarsening of the valuation of
the immediate extension $L=K\<\hat a\>$ of $K$.  Let~$a^*$ be the image of $\hat a$ in the differential residue field
$\dot{L}$ of $(L,\dot v)$. Note that $\dot{L}$ is an immediate extension of $\dot{K}$. The pc-sequence $(a_{\rho})$ then
yields a sequence $(\dot{a_{\rho}})$  in~$\dot{K}$ with
$v(a^*-\dot{a_{\rho}})=\gamma_{\rho}$ for all $\rho$. Thus $(\dot a_{\rho})$ is a c-sequence in $\dot{K}$ with $\dot a_{\rho} \to a^*$, so
$\dot{P}(\dot a_{\rho})\to \dot{P}(a^*)$ by [ADH, 4.4.5]. From $v\big(P(\hat a)\big)>\Delta$ we obtain~$\dot{P}(a^*)=0$, and so $\dot{P}(\dot a_{\rho})\to 0$.
So far we have not used our assumption that $K$ is asymptotic and~$\Psi\subseteq v(\hat a-K)$. Using this now, we note that
for $\alpha\in \Delta^{>}$ we have 
$0 <\alpha'=\alpha+\alpha^\dagger$, so $\alpha' \in \Delta$, hence
the derivation of $\dot{K}$ is nontrivial. Thus we can apply [ADH, 4.4.10] to $\dot{K}$ and modify the $a_{\rho}$ 
without changing $\gamma_{\rho}=v(a^*-\dot a_{\rho})$ to arrange that in addition~$\dot{P}(\dot a_{\rho})\ne 0$
for all $\rho$.  Since $\kappa=\cf(\Delta)$, we can replace $(a_{\rho})$ by a cofinal subsequence so that $P(a_{\rho})\leadsto 0$, hence $P\in Z(K,\hat a)$ by \cite[4.6]{VDF}.
\end{proof}

\noindent
To elaborate on this, let $\Delta$ be a convex subgroup of $\Gamma$ and $\dot K$ the
valued differential residue field of the  $\Delta$-coarsening $v_\Delta$ of the valuation~$v$ of $K$. We view $\dot K$ in the usual way as a valued differential subfield of the valued differential residue field $\dot{\hat K}$ of the $\Delta$-coarsening of the valuation of $\hat K$ by $\Delta$; see [ADH, pp.~159--160 and~4.4.4].

\begin{cor}\label{cor:ZKhata} 
Suppose $K$ is asymptotic,  $\Psi\subseteq v(\hat a -K)$, and  $\Delta$  is cofinal in~$v({\hat a-K})$. 
 Let $P\in K\{Y\}$ with $P \asymp 1$. Then
$P\in Z(K,\hat a)$ if and only if~$\dot P(\dot{\hat a})=0$ in $\dot{\hat K}$. Also,
 $P$ is an element of $Z(K,\hat a)$ of minimal complexity if and only if $\dot P$ is a minimal
annihilator of $\dot{\hat a}$ over $\dot K$ and $\dot P$ has the same complexity as~$P$.
\end{cor}
\begin{proof}
The first statement is immediate from  Lemmas~\ref{Zp1} and~\ref{lem:ZKhata}.
For the rest use that 
for $R\in\dot{\mathcal O}\{Y\}$ we have $\cc(\dot R)\leq\cc(R)$ and
that for all $Q\in \dot K\{Y\}$ there is an $R\in\dot{\mathcal O}\{Y\}$
with  $Q=\dot R$
and $Q_\i\neq 0$ iff $R_\i\neq 0$ for all $\i$, so
$\cc(\dot R)=\cc(R)$.
\end{proof}

\section{Differential Henselianity of the Completion}\label{sec:completion d-hens}

\noindent
{\em Let $K$ be a valued differential field with small derivation.}\/ We let $\Gamma:= v(K^\times)$ be the value group of $K$ and 
$\k:=\res(K)$ be the differential residue field of $K$, and we let $r\in\N$.  
The following summarizes [ADH, 7.1.1, 7.2.1]:

\begin{lemma}
The valued differential field $K$ is $r$-$\d$-henselian iff  for each 
$P$ in~$K\{Y\}$ of order~$\leq r$ with $\ddeg P=1$  there is a $y\in\mathcal O$ with $P(y)=0$.
\end{lemma} 

\noindent
Recall that the derivation of $K$ being small, it is continuous [ADH, 4.4.6], and hence extends uniquely to a continuous derivation
on  the completion $K^{\cc}$ of the valued field $K$ [ADH, 4.4.11]. We equip  $K^{\cc}$ with this derivation, which remains small~[ADH, 4.4.12],  so $K^{\cc}$ is an immediate valued differential field extension of~$K$ with small derivation, in particular, $\k=\res(K^{\cc})$. \index{valued differential field!completion} \index{completion}

Below we  characterize in a first-order way when 
$K^{\cc}$ is $r$-$\d$-henselian. We shall use tacitly that for $P\in K\{Y\}$ 
we have $P(g)\preceq P_{\times g}$ for all $g\in K$;
to see this, replace~$P$ by $P_{\times g}$ to reduce to $g=1$, and observe that
$P(1)=\sum_{\dabs{\bsigma}=0} P_{[\bsigma]}\preceq P$.

\begin{lemma}\label{lem:Kc 1}
Let  $P\in K^{\cc}\{Y\}$, $b\in K^{\cc}$ with $b\preceq 1$ and $P(b)=0$, and 
$\gamma\in\Gamma^>$. Then there
is an $a\in\mathcal O$ with $v\big(P(a)\big)>\gamma$.
\end{lemma}
\begin{proof}
To find  an $a$ as claimed we take $f\in K$ satisfying $f\asymp P$ and replace~$P$,~$\gamma$ by $f^{-1}P$, $\gamma-vf$, respectively, to arrange $P\asymp 1$ and thus  $P_{+b}\asymp 1$. We also assume~$b\ne 0$.  
Since $K$ is dense in~$K^{\cc}$ we can take $a\in K$ such that $a\sim b$ (so~$a\in\mathcal O$) and
$v(a-b)>2\gamma$. Then  with $g:=a-b$,   using [ADH, 4.5.1(i) and 6.1.4] we conclude
$$v\big(P(a)\big)=v\big(P_{+b}(g)\big) \geq v\big((P_{+b})_{\times g}\big)
\geq v(P_{+b}) + vg + o(vg) = vg+o(vg) >\gamma$$
as required. 
\end{proof}

\noindent
Recall that if $K$ is asymptotic, then so is $K^{\cc}$ by [ADH, 9.1.6]. 

\begin{lemma}\label{lem:Kc 2}
Suppose $K$ is asymptotic, $\Gamma\ne \{0\}$, and for every $P\in K\{Y\}$ of order at most~$r$ with $\ddeg P=1$ and every $\gamma\in\Gamma^>$ there
is an $a\in\mathcal O$ with~${v\big(P(a)\big)>\gamma}$.
Then~$K^{\cc}$ is $r$-$\d$-henselian.
\end{lemma}
\begin{proof}
The hypothesis applied to $P\in\mathcal O\{Y\}$ of order~$\leq r$ with~${\ddeg P=\deg P=1}$
yields that $\k$ is $r$-linearly surjective.
Let now $P\in K^{\cc}\{Y\}$ be of order~$\leq r$ with~$\ddeg P=1$.   We need to show that there exists $b\in K^{\cc}$ with $b\preceq 1$ and~$P(b)=0$. First we arrange $P\asymp 1$.  
By [ADH, remarks after 9.4.11] we can take $b\preceq 1$ in an  immediate $\d$-henselian asymptotic field extension~$L$ 
of~$K^{\cc}$ with $P(b)=0$.  We prove below that~$b\in K^{\cc}$. We may assume $b\notin K$, so
$v(b-K)$ has no largest element, since~$L\supseteq K$ is immediate.
Note also that   $\ddeg P_{+b}=1$ by [ADH, 6.6.5(i)]; since~${P(b)=0}$ we thus have $\ddeg P_{+b,\times g}=1$
for all $g\preceq 1$ in $L^\times$ by [ADH, 6.6.7].

\claim{Let $\gamma\in\Gamma^>$ and $a\in K$ with $v(b-a)\geq 0$.  There is a $y\in\mathcal O$ such that~$v\big(P(y)\big)>\gamma$ and $v(b-y)\geq v(b-a)$.}

{\sloppy
\noindent
To prove this claim, take $g\in K^\times$ with $vg=v(b-a)$. Then by [ADH, 6.6.6] and the observation 
preceding the claim we have~${\ddeg P_{+a,\times g}=\ddeg P_{+b,\times g}=1}$. 
Thanks to density of $K$ in $K^{\cc}$ we may take $Q\in K\{Y\}$ of order~$\leq r$ with $P_{+a,\times g}\sim Q$ and $v(P_{+a,\times g}-Q)>\gamma$. Then 
$\ddeg Q=1$, so by the hypothesis of the lemma  we have~${z\in\mathcal O}$ with $v\big(Q(z)\big)>\gamma$.
Set $y:=a+gz\in\mathcal O$; then we have~$v\big(P(y)\big)=v\big(P_{+a,\times g}(z)\big)>\gamma$ and $v(b-y)=v(b-a-gz)\geq v(b-a)= vg$ as claimed.}

\medskip
\noindent
Let now $\gamma\in\Gamma^>$;  to show that $b\in K^{\cc}$,
it is enough by [ADH, 3.2.15, 3.2.16] to show that then $v(a-b)>\gamma$ for some $a\in K$.
Let   $A:=L_{P_{+b}}\in L[\der]$; then $A\asymp 1$. Since $\abs{\exc_L(A)}\leq r$   by [ADH, 7.5.3], the claim gives an $a\in\mathcal O$ with $v\big(P(a)\big)>2\gamma$ and $0<v(b-a)\notin\exc_L(A)$. Put $g:=a-b$ and $R:=(P_{+b})_{>1}$. Then
$R\prec 1$ and
$$P(a)\ =\ P_{+b}(g)\ =\ A(g) + R(g)$$
where by the definition of $\exc_L(A)$ and [ADH, 6.4.1(iii), 6.4.3] we have in $\Q\Gamma$:
$$v\big(A(g)\big)\ =\ v_A(vg)\ =\ vg+o(vg)\ <\ vR + (3/2)vg\   \leq\ v_R(vg)\ \leq\ v\big(R(g)\big)$$
and so $v\big(P(a)\big) = vg+o(vg) > 2\gamma$. Therefore $v(a-b)=vg>\gamma$ as required.
\end{proof}

\noindent
The last two lemmas yield an analogue of  [ADH, 3.3.7] for $r$-$\d$-hensel\-ianity and a partial generalization of [ADH, 7.2.15]:

\begin{cor}\label{cor:Kc d-henselian}
Suppose $K$ is asymptotic and $\Gamma\ne \{0\}$.
Then the following are equivalent:
\begin{enumerate}
\item[$\mathrm{(i)}$] $K^{\cc}$ is $r$-$\d$-henselian;
\item[$\mathrm{(ii)}$] for every $P\in K\{Y\}$ of order at most $r$ with $\ddeg P=1$ and every $\gamma\in\Gamma^>$ there
exists $a\in\mathcal O$ with $v\big(P(a)\big)>\gamma$.
\end{enumerate}
In particular, if $K$ is $r$-$\d$-henselian, then so is $K^{\cc}$.
\end{cor}

\section{Complements to [ADH] on Newtonianity}\label{sec:complements newton}

\noindent 
{\em In this section $K$ is an ungrounded $H$-asymptotic field with $\Gamma=v(K^\times)\neq\{0\}$}. Note that then $K^{\cc}$ is also $H$-asymptotic. We let $r$ range over $\N$ and $\phi$ over the active elements of $K$.
Our first aim is a newtonian analogue of Corollary~\ref{cor:Kc d-henselian}:

\begin{prop}\label{prop:Kc newtonian}
For $\d$-valued and $\upo$-free $K$, the  following are equivalent:
\begin{enumerate}
\item[$\mathrm{(i)}$] $K^{\cc}$ is $r$-newtonian;
\item[$\mathrm{(ii)}$]  for every $P\in K\{Y\}$ of order at most $r$ with $\ndeg P=1$ and every $\gamma\in\Gamma^>$ there
is an $a\in\mathcal O$ with $v\big(P(a)\big)>\gamma$.
\end{enumerate}
If $K$ is $\d$-valued, $\upo$-free, and $r$-newtonian, then so is $K^{\cc}$.
\end{prop} 

\noindent
The final statement in this proposition extends [ADH, 14.1.5]. Towards the proof we first
state a variant of [ADH, 13.2.2] which follows easily from [ADH, 11.1.4]:

\begin{lemma}\label{lem:same ndeg} \label{lem:ndeg of nearby diffpoly} 
Assume $K$ has small derivation and let $P,Q\in K\{Y\}^{\neq}$ and 
$\phi\preceq 1$. Then  $P^\phi \asymp^\flat P$, and so we have the three implications
$$P \preceq^\flat Q\ \Longrightarrow\ P^\phi\preceq^\flat Q^\phi,\quad 
P\prec^\flat Q\ \Longrightarrow\ 
P^\phi \prec^\flat Q^\phi,\quad P \sim^\flat Q\ \Longrightarrow\ P^\phi\sim^{\flat} Q^\phi.$$  
The last implication gives:  $P\sim^\flat Q\ \Longrightarrow\ \ndeg P=\ndeg  Q\text{ and }\nval P=\nval Q.$
\end{lemma}

\noindent
For $P^\phi\asymp^\flat P$ and the subsequent three implications 
in the lemma above we can drop the assumption that $K$ is ungrounded.

\begin{lemma}\label{lem:Kc 3}
Suppose $K$ is $\d$-valued, $\upo$-free, and for every $P\in K\{Y\}$ of order at most $r$ with $\ndeg P=1$ and every $\gamma\in\Gamma^>$ there
is an $a\in\mathcal O$ with $v\big(P(a)\big)>\gamma$. Then $K^{\cc}$ is $\d$-valued, $\upo$-free, and $r$-newtonian.
\end{lemma}

\begin{proof} By [ADH, 9.1.6 and 11.7.20], $K^{\cc}$ is $\d$-valued and $\upo$-free. 
Let $P\in K^{\cc}\{Y\}$ be of order~$\leq r$ with $\ndeg P=1$.   We need to show that $P(b)=0$ for some $b\preceq 1$ in $K^{\cc}$. 
To find $b$  we may replace $K$,~$P$ by $K^\phi$,~$P^\phi$; in particular we may assume that~$K$ has small derivation and $\Gamma^{\flat}\ne \Gamma$.
By \eqref{eq:14.0.1}  %[ADH, 14.0.1 and the remarks following it] 
we can take $b\preceq 1$ in 
an immediate newtonian extension $L$ of~$K^{\cc}$  such that $P(b)=0$. We claim that~$b\in K^{\cc}$.
To show this we may assume $b\notin K$, so  $v(b-K)$ does not have a largest element.
By [ADH, 11.2.3(i)] we have $\ndeg P_{+b}=1$ and so $\ndeg P_{+b,\times g}=1$
for all $g\preceq 1$ in~$L^\times$ by~[ADH, 11.2.5], in view of $P(b)=0$.

\claim{Let $\gamma\in\Gamma^>$ and $a\in K$ with $v(b-a)\geq 0$. There is a $y\in\mathcal O$ such that~$v\big(P(y)\big)>\gamma$ and $v(b-y)\geq v(b-a)$.}

\noindent
The proof is similar to that of the claim in the proof of Lemma~\ref{lem:Kc 2}:
Take $g\in K^\times$ with $vg=v(b-a)$.  Then $\ndeg P_{+a,\times g}=\ndeg P_{+b,\times g}=1$ by [ADH, 11.2.4] and the observation 
preceding the claim. 
Density of $K$ in $K^{\cc}$ yields $Q\in K\{Y\}$ of order~$\leq r$ with $v(P_{+a,\times g}-Q)>\gamma$ and $P_{+a,\times g}\sim^\flat Q$, the latter using $\Gamma^{\flat}\ne \Gamma$. Then 
$\ndeg Q=\ndeg P_{+a,\times g}=1$ by Lemma~\ref{lem:ndeg of nearby diffpoly}, so the hypothesis of the lemma gives $z\in\mathcal O$ with~$v\big(Q(z)\big)>\gamma$.
Setting $y:=a+gz\in\mathcal O$ we have $v\big(P(y)\big)=v\big(P_{+a,\times g}(z)\big)>\gamma$ and $v(b-y)=v(b-a-gz)\geq vg=v(b-a)$.

\medskip
\noindent
Let   $\gamma\in\Gamma^>$;  to get $b\in K^{\cc}$,
it is enough to show that then $v(a-b)>\gamma$
for some~$a\in K$.
Let   $A:=L_{P_{+b}}\in L[\der]$. Since $\abs{\exc^{\ev}_L(A)}\leq r$   by [ADH, 14.2.9], by the claim we can take
$a\in\mathcal O$ with $v\big(P(a)\big)>2\gamma$ and $0<v(b-a)\notin\exc^{\ev}_L(A)$. Now put $g:=a-b$ and
take $\phi$ with $vg\notin\exc_{L^\phi}(A^\phi)$; note that then $A^\phi=L_{P^\phi_{+b}}$.
Replacing~$K$,~$L$,~$P$ by $K^\phi$,~$L^\phi$,~$P^\phi$ we arrange $vg\notin \exc_L(A)$, and (changing $\phi$ if necessary) $\ddeg P_{+b}=1$. We also arrange
$P_{+b}\asymp 1$, and then  $(P_{+b})_{>1}\prec 1$. 
As in the proof of Lemma~\ref{lem:Kc 2} above we now derive   $v(a-b)=vg>\gamma$.
\end{proof}

\noindent
Combining Lemmas~\ref{lem:Kc 1} and \ref{lem:Kc 3} now yields  Proposition~\ref{prop:Kc newtonian}.  \qed

\medskip
\noindent
To show that newtonianity is preserved under specialization, we assume below that~$\Psi\cap \Gamma^{>}\ne \emptyset$, so $K$ has
small derivation.  Let $\Delta\neq\{0\}$ 
be a convex subgroup of~$\Gamma$ with $\psi(\Delta^{\neq})\subseteq\Delta$. Then $1\in \Delta$ where $1$ denotes the unique positive element of $\Gamma$ fixed by the function 
$\psi$: use that $\psi(\gamma)\ge 1$ for $0 < \gamma < 1$.
(Conversely, any convex subgroup $G$ of $\Gamma$ with $1\in G$ satisfies $\psi(G^{\ne})\subseteq G$.)  Let $\dot v$ be
the coarsening of the valuation $v$ of $K$ by $\Delta$, with valuation ring $\dot{\mathcal O}$,
maximal ideal $\dot\smallo$ of $\dot{\mathcal O}$, and residue field  $\dot K=\dot{\mathcal O}/\dot\smallo$.
The derivation of $K$ is small with respect to $\dot{v}$,
and $\dot K$ with the induced valuation $v\colon \dot K^\times\to\Delta$ and induced derivation as in [ADH, p.~405] is an asymptotic field with asymptotic couple $(\Delta,\psi|\Delta^{\neq})$, and so
is of $H$-type with small derivation. If $K$ is $\d$-valued, then so is
$\dot K$ by [ADH, 10.1.8], and if $K$ is $\upo$-free, then so is $\dot K$ by [ADH, 11.7.24].
The residue map
$a\mapsto\dot a:=a+\dot\smallo\colon \dot{\mathcal O}\to\dot K$ is a differential ring morphism, 
extends to a differential ring morphism $P\mapsto\dot P\colon \dot{\mathcal O}\{Y\}\to\dot K\{Y\}$ of differential rings
sending $Y$ to $Y$, and $\ddeg P=\ddeg\dot P$ for $P\in \dot{\mathcal O}\{Y\}$ with $\dot P \ne 0$. 

We now restrict $\phi$ to range over active elements of $\mathcal O$. Then $v\phi\le 1+1$, so~$v\phi\in \Delta$, and hence $\phi$ is a unit of $\dot{\mathcal O}$. 
It follows that $\dot\phi$ is active in $\dot K$, and every active element of $\dot K$ lying in its valuation ring is of this form. Moreover, the differential subrings~$\dot{\mathcal O}$  of $K$ and $\dot{\mathcal O}^\phi:=(\dot{\mathcal O})^\phi$ of $K^\phi$ have the same underlying ring, and the derivation of~$K^\phi$ is small with respect to $\dot{v}$. Thus the differential residue fields $\dot{K}=\dot{\mathcal O}/\dot{\smallo}$ and~$\dot{K}^\phi:=\dot{\mathcal O}^\phi/\dot{\smallo}$ have the same underlying field and are related as follows:
$$\dot{K}^\phi\ =\ (\dot{K})^{\dot\phi}.$$
For $P\in \dot{\mathcal O}\{Y\}$ we have $P^\phi\in \dot{\mathcal O}^\phi\{Y\}$, and the image of
$P^\phi$ under the residue map~$\dot{\mathcal O}^\phi\{Y\}\to\dot K^\phi\{Y\}$ equals
$\dot P^{\dot\phi}$; hence $\ndeg P=\ndeg \dot P$ for $P\in \dot{\mathcal O}\{Y\}$ satisfying~${\dot{P}\ne 0}$. 
These remarks imply:

\begin{lemma}\label{lem:dotK newtonian}
If $K$ is $r$-newtonian, then $\dot K$ is $r$-newtonian.
\end{lemma}

\noindent
Combining Proposition~\ref{prop:Kc newtonian} and Lemmas~\ref{lem:Kc 3} and~\ref{lem:dotK newtonian}  yields:

\begin{cor}\label{cor:Kc newtonian}
Suppose $K$ is $\d$-valued, $\upo$-free, and $r$-newtonian. Then $\dot K$ and its completion are
$\d$-valued, $\upo$-free, and $r$-newtonian.
\end{cor} 

\noindent
We finish with a newtonian analogue of [ADH, 7.1.7]:

\begin{lemma}\label{rdhrnewt}
Suppose $(K,\dot{\mathcal O})$ is $r$-$\d$-henselian and $\dot K$ is $r$-newtonian. Then $K$ is $r$-newtonian.
\end{lemma}
\begin{proof}
Let $P\in K\{Y\}$ be quasilinear of order~$\leq r$; we need to show the existence of~$b\in\mathcal O$ with $P(b)=0$. Replacing $K$, $P$ by $K^\phi$, $P^\phi$ for suitable~$\phi$ (and renaming) we arrange $\ddeg P =1$; this also uses [ADH,  7.3, subsection on compositional conjugation]. 
We can further assume that $P\asymp 1$. Put $Q:=\dot P\in\dot K\{Y\}$, so~$\ndeg Q=1$, and thus $r$-newtonianity of $\dot K$ yields an $a\in\mathcal O$ with $Q(\dot a)=0$. Then $P(a)\dotprec 1$, $P_{+a,1}\sim P_1\asymp 1$,
and $P_{+a,>1}\prec 1$. Since $(K,\dot{\mathcal O})$ is $r$-$\d$-henselian, this gives $y\in\dot\smallo$ with $P_{+a}(y)=0$,
and then $P(b)=0$ for $b:=a+y\in\mathcal O$.
\end{proof}

\noindent
Lemmas~\ref{lem:dotK newtonian},~\ref{rdhrnewt}, and  [ADH, 14.1.2] yield:

\begin{cor}
$K$ is $r$-newtonian iff $(K,\dot{\mathcal O})$ is $r$-$\d$-henselian and $\dot K$ is $r$-new\-to\-ni\-an.
\end{cor}

\subsection*{Invariance of Newton quantities} 
{\it In this subsection $P\in K\{Y\}^{\neq}$.}\/
In~[ADH, 11.1] we associated to $P$ its Newton weight $\nwt P$, Newton degree $\ndeg P$, and Newton multiplicity $\nval P$ at~$0$, all elements of $\N$,
as well as the element $v^{\ev}(P)$ of $\Gamma$;\label{p:vevP}
these quantities do not change when passing to an $H$-asymptotic
extension~$L$ of~$K$ with $\Psi$ cofinal in $\Psi_L$, cf.~[ADH, p.~480], where the assumptions on $K$, $L$ are slightly weaker. 
Thus by Theorem~\ref{thm:ADH 13.6.1}, these quantities do not change for $\upo$-free~$K$ in passing to an  $H$-asymptotic pre-$\d$-valued $\d$-algebraic extension of~$K$.
Below we improve on this in several ways. First,  for $\order P\leqslant1$ we can drop $\Psi$ being cofinal in $\Psi_L$ by a strengthening of [ADH, 11.2.13]:

\begin{lemma}\label{lem:11.2.13 invariant}
Suppose $K$ is $H$-asymptotic with rational asymptotic integration and   $P\in K[Y,Y']^{\ne}$.
Then there are $w\in \N$, $\alpha\in \Gamma^{>}$, 
$A\in K[Y]^{\ne}$, and an active~$\phi_0$ in~$K$ such that for every  asymptotic extension $L$ of $K$ and active~$f\preceq\phi_0$ in $L$,
$$P^{f}\ =\  f^w A(Y)(Y')^w + R_{f},\quad R_{f}\in L^{f}[Y, Y'],\quad v(R_{f})\ \geqslant\ v(P^{f}) + \alpha.$$
For such $w, A$ we have for any ungrounded $H$-asymptotic extension $L$ of $K$,
$$ \nwt_L P\ =\ w, \quad \ndeg_L P\ =\ \deg A+w,\quad
 \nval_L P\ =\ \val A+w,\quad\  v^{\ev}_L(P) = v(A).$$
\end{lemma}
\begin{proof} Let $w$ be as in the proof of [ADH, 11.2.13]. Using its notations, this proof yields an active 
$\phi_0$ in $K$ such that 
\begin{equation}\label{eq:11.2.13 invariant, 1} w\gamma+v(A_w)\  <\  j\gamma + v(A_j)\end{equation}
for all 
$\gamma\geqslant v(\phi_0)$ in $\Psi^{\downarrow}$ and  $j\in J\setminus \{w\}$. This gives
$\beta\in \Q\Gamma$  such that $\beta>\Psi$ and~\eqref{eq:11.2.13 invariant, 1} remains true for all  $\gamma\in \Gamma$ with $v(\phi_0)\leqslant \gamma < \beta$. Since $(\Q\Gamma, \psi)$ has asymptotic integration, $\beta$ is not a gap in  $(\Q\Gamma,\psi)$, so
$\beta>\beta_0> \Psi$ with $\beta_0\in \Q\Gamma$. This yields an element $\alpha\in (\Q\Gamma)^{>}$ such that
for all $\gamma\in \Q\Gamma$ with $v(\phi_0) \le \gamma \le \beta_0$ we have
\begin{equation}\label{eq:11.2.13 invariant, 2} w\gamma + v(A_w) +\alpha\ \leqslant\ j\gamma + v(A_j)\end{equation}
Since $\Gamma$ has no least positive element, we can decrease $\alpha$ to arrange $\alpha\in \Gamma^{>}$. 
Now~\eqref{eq:11.2.13 invariant, 2} remains true for all elements $\gamma$ 
of any divisible ordered abelian group extending $\Q\Gamma$ 
with $v(\phi_0)\leqslant \gamma \le \beta_0$.
Thus $w$, $\alpha$, $A=A_w$, and $\phi_0$ are as required. 

For any ungrounded $H$-asymptotic extension $L$ of $K$ we obtain for active $f\preceq \phi_0$ in $L$ that $v(P^f)=v(A)+wv(f)$,  so $v_L^{\ev}(P)=v(A)$ in view of  the  identity in  [ADH, 11.1.15]
defining $v_L^{\ev}(P)$. 
\end{proof} 

\noindent
For quasilinear~$P$ we have:

\begin{lemma}\label{lem:13.7.10}
Suppose $K$ is $\upl$-free and $\ndeg P=1$. Then there are active~$\phi_0$ in~$K$ and~$a,b\in K$ with $a\preceq b\ne 0$ 
 such that  either \textup{(i)} or \textup{(ii)} below holds:
\begin{enumerate}
\item[\textup{(i)}]
$P^{f}\, \sim_{\phi_0}^\flat\, a+bY$
for all active~$f\preceq\phi_0$ in all  $H$-asymptotic extensions of $K$;
\item[\textup{(ii)}]
$P^{f}\, \sim_{\phi_0}^\flat\, \frac{f}{\phi_0}b\,Y'$ for all active~$f\preceq\phi_0$ in all  $H$-asymptotic extensions of~$K$.
\end{enumerate}
In particular,  for each ungrounded $H$-asymptotic extension $L$  of~$K$,
$$\nwt_L P=\nwt P\leq 1, \quad \ndeg_L P=1,\quad \nval_L P=\nval P, \quad  v_L^{\ev}(P)=v^{\ev}(P).$$
\end{lemma} 
\begin{proof}
From [ADH, 13.7.10] we obtain an active $\phi_0$ in $K$ and $a,b\in K$ with $a\preceq b$ such that in $K^{\phi_0}\{Y\}$, 
either $P^{\phi_0} \sim^\flat_{\phi_0} a+bY$ or $P^{\phi_0}\ \sim^\flat_{\phi_0}\ b\,Y'$ (so $b\ne 0$). 
In the first case, set $A(Y):=a+bY$, $w:=0$; in the second case, set $A(Y):=bY'$, $w:=1$. Then
$P^{\phi_0}  = A + R$ where~$R\prec_{\phi_0}^\flat b\asymp P^{\phi_0}$ in $K^{\phi_0}\{Y\}$. 

Let $L$ be an $H$-asymptotic extension of $K$.
Then~$R\prec_{\phi_0}^\flat P^{\phi_0}$ remains
true in~$L^{\phi_0}\{Y\}$, and if $f\preceq \phi_0$ is active in $L$, then
$P^f = (P^{\phi_0})^{f/\phi_0}=(f/\phi_0)^w A+R^{f/\phi_0}$ where~$R^{f/\phi_0}\prec_{\phi_0}^\flat P^f$  by Lemma~\ref{lem:same ndeg} and the remark following its proof. As to~$v_L^{\ev}(P)=v^{\ev}(P)$ for ungrounded $L$,  the  identity from  [ADH, 11.1.15]
defining these quantities  shows that both are $vb$ in case  (i), and  $v(b)-v(\phi_0)$ in case (ii).
\end{proof}

\noindent
Lemma~\ref{lem:13.7.10}  has the following consequence, partly generalizing Corollary~\ref{cor:sum of nwts}:

\begin{cor} \label{cor:13.7.10}
Suppose $K$ is $\upl$-free, $A\in K[\der]^{\neq}$ and $L$ is an ungrounded $H$-asymptotic extension
of~$K$. Then for $\gamma\in\Gamma$ the quantities $\nwt_A(\gamma)\leq 1$ and $v^{\ev}_A(\gamma)$ do not change when passing from $K$ to $L$; in particular,
$$\exc^{\ev}(A)\ =\ \big\{\gamma\in\Gamma: \nwt_A(\gamma)=1\big\}\  =\  \exc^{\ev}_L(A)\cap\Gamma.$$
\end{cor}

\noindent
This leads to a variant of Corollary~\ref{cor:size of excev}:

\begin{cor}\label{cor:size of excev, strengthened} Suppose $K$ is $\upl$-free. Then $|\exc^{\ev}(A)|\leq \order A$ for all $A\in K[\der]^{\ne}$.
\end{cor}
\begin{proof} By [ADH, 10.1.3], $K$ is pre-$\d$-valued, hence by [ADH, 11.7.18] it has an $\upo$-free $H$-asymptotic extension. It remains to appeal to Corollaries~\ref{cor:sum of nwts} and~\ref{cor:13.7.10}.
\end{proof}

\noindent
For completeness we next state a version of Lemma~\ref{lem:13.7.10} for~$\ndeg P=0$;
the proof using [ADH, 13.7.9] is similar, but simpler, and hence omitted.

\begin{lemma}\label{lem: 13.7.9}
Suppose $K$ is $\upl$-free and $\ndeg P=0$. Then there are an active~$\phi_0$ in $K$ and~$a\in K^\times$ 
 such that   $P^f \sim_{\phi_0}^\flat a$ for all active~$f\preceq\phi_0$ in all $H$-asymptotic extensions of $K$.
\end{lemma}

\noindent
In particular, for $K$, $P$ as in Lemma~\ref{lem: 13.7.9}, no $H$-asymptotic extension of $K$ contains any $y\preceq 1$ such that $P(y)=0$. 

For general $P$ and $\upo$-free $K$ we can still do better than stated earlier:

\begin{lemma}\label{lem:13.6.12} 
Suppose $K$ is $\upo$-free. Then there are $w\in\N$, $A\in K[Y]^{\neq}$, and an active $\phi_0$ in $K$ such that for all active
$f\preceq \phi_0$ in all  $H$-asymptotic extensions of~$K$, 
$$P^{f}\ \sim^\flat_{\phi_0} \  (f/\phi_0)^w A(Y)(Y')^w. $$
For such $w$, $A$, $\phi_0$ we have for any ungrounded $H$-asymptotic extension $L$ of $K$,
\begin{align*} \nwt_L P\ &=\ w, &  \ndeg_L P&\ =\ \deg A+w,\\
 \nval_L P\ &=\ \val A+w, &   v^{\ev}_L(P)&\ =\ v(A)-wv(\phi_0).
 \end{align*} 
\end{lemma}
\begin{proof} By [ADH, 13.6.11] we have active $\phi_0$ in $K$ and $A\in K[Y]^{\ne}$ such that
$$P^{\phi_0}\ =\ A\cdot (Y')^w + R,\quad w:= \nwt P, \quad R \in K^{\phi_0}\{Y\},\ R\prec_{\phi_0}^\flat P^{\phi_0}.$$
(Here $\phi_0$ and $A$ are the $e$ and $aA$ in [ADH, 13.6.11].) The rest of the argument is just like in the second part of the proof of
Lemma~\ref{lem:13.7.10}. 
\end{proof}

\subsection*{Remarks on newton position}
For the next lemma we put ourselves in the setting of~[ADH, 14.3]: $K$ is $\upo$-free, $P\in K\{Y\}^{\neq}$,
and~$a$ ranges over~$K$. 
Recall that~$P$ is said to be in {\it newton position at $a$}\/ if $\nval P_{+a}=1$. \index{differential polynomial!newton position} \index{newton!position}

Suppose $P$ is in newton position at~$a$; then $A:=L_{P_{+a}}\in K[\der]^{\ne}$. 
Recall the definition of~$v^{\ev}(P,a)=v^{\ev}_K(P,a)\in \Gamma_\infty$:
if $P(a)=0$, then $v^{\ev}(P,a)=\infty$; if $P(a)\neq 0$, then $v^{\ev}(P,a)=vg$ where $g\in K^\times$
satisfies $P(a) \asymp (P_{+a})^\phi_{1,\times g}$ eventually, that is,
$v_{A^\phi}(vg)=v\big(P(a)\big)$ eventually. In the latter case
$\nwt_A(vg)=0$, that is, $vg\notin\exc^{\ev}(A)$, and 
$v^{\ev}_A(vg)=v\big(P(a)\big)$, since $v_{A^\phi}(vg)=v^{\ev}_A(vg)+ \nwt_A(vg)v\phi$ eventually. 
For any~$f\in K^\times$, $P^f$ is also in newton position at $a$, and  $v^{\ev}(P^f,a)=v^{\ev}(P,a)$.
Note also
that~$P_{+a}$ is in newton position
at $0$ and $v^{\ev}(P_{+a},0)=v^{\ev}(P,a)$.
Moreover, in passing from $K$ to  an $\upo$-free extension, $P$  remains in newton position at $a$ and~$v^{\ev}(P,a)$ does not change, by Lemma~\ref{lem:13.6.12}.

\medskip
\noindent
{\it In the rest of this subsection $P$ is in newton position at $a$, and~$\hat a$ is an element of an $H$-asymptotic extension $\hat K$ of $K$  such that $P(\hat a)=0$.}\/ (We allow~$\hat a\in K$.)
We first generalize part of [ADH, 14.3.1], with a similar proof:

\begin{lemma}\label{lem:14.3.1} 
$v^{\ev}(P,a)>0$ and  $v(\hat a-a) \leq v^{\ev}(P,a)$.
\end{lemma}
\begin{proof}
This is clear if $P(a)=0$. Assume $P(a)\neq 0$.
Replace~$P$,~$\hat a$,~$a$ by~$P_{+a}$,~$\hat a-a$,~$0$, respectively, to arrange $a=0$.
Recall that $K^\phi$ has small derivation.
Set~$\gamma:=v^{\ev}(P,0)\in\Gamma$ and take $g\in K$ with~$vg=\gamma$. 
Now $(P_1^\phi)_{\times g} \asymp P_0$, eventually, and~$\nval P=1$ gives $P(0) \prec P_1^\phi$, eventually, hence~$g\prec 1$.
Moreover, for~$j\ge 2$, $P_1^\phi\succeq P_j^\phi$, eventually, so~$(P_1^\phi)_{\times g}\succ (P_j^\phi)_{\times g}$, eventually, by [ADH, 6.1.3]. 
Thus for~${j\geq 1}$ we have~$(P_{\times g}^\phi)_j = (P_j^\phi)_{\times g} \preceq P(0)$, eventually;
in particular, there is no~$y\prec 1$ in any $H$-asymptotic  extension of $K$ with $P_{\times g}(y)=0$.
Since~$P(\hat a)=0$, this yields~$v(\hat a) \leq \gamma=v^{\ev}(P,0)$.
\end{proof}

\noindent
Here is a situation where $v(\hat a-a) = v^{\ev}(P,a)$:

\begin{lemma}\label{lem:14.3 complement} Suppose $\Psi$ is cofinal in $\Psi_{\hat K}$, $\hat a - a\prec 1$, and
 $v(\hat a-a)\notin\exc_{\hat K}^{\ev}(A)$ where~$A:=L_{P_{+a}}$.  Then $v(\hat a-a)=v^{\ev}(P,a)$. 
\end{lemma}
\begin{proof} Note that $\hat K$ is ungrounded, so $\exc_{\hat K}^{\ev}(A)$ is defined, and $\hat{K}$ is pre-$\d$-valued. 
As in the proof of Lemma~\ref{lem:14.3.1} we arrange $a=0$. As an asymptotic subfield  of~$\hat K$, $K\langle \hat a\rangle$
is pre-$\d$-valued. Hence $K\langle \hat a \rangle$ is $\upo$-free by Theorem~\ref{thm:ADH 13.6.1}.
 The remarks
preceding  Lemma~\ref{lem:14.3.1} then allow us to replace~$K$ by $K\langle \hat a\rangle$  to arrange $\hat a\in K$.
The case $\hat a=0$ is trivial, so assume $0\ne \hat a\prec 1$. Now $\nval P=1$ gives for $j\ge 2$ that $P_1^\phi\succeq P_j^\phi$, eventually, hence
$(P_1^\phi)_{\times \hat a}\succ (P_j^\phi)_{\times \hat a}$, eventually, by~[ADH, 6.1.3]. Moreover, $P_1(\hat a)=A(\hat a)= A^\phi(\hat a)\asymp A^\phi \hat a$, eventually, using $v(\hat a)\notin \exc^{\ev}_{\hat K}(A)$ in the last step, so for $j\ge 2$, eventually
$$P_1(\hat a)\ \asymp\ (P_1^\phi)_{\times \hat a}\ \succ\ (P_j^\phi)_{\times \hat a}\ \succeq P_j^\phi(\hat a)\ =\ P_j(\hat a). $$
Also $P_1(\hat a)\ne 0$, since $A^\phi\hat a \ne 0$. Then $P(\hat a)=0$
gives $P(0)\asymp P_1(\hat a)$. 
Thus $v\big(P(0)\big)=v_{A^\phi}\big(v(\hat a)\big)$, eventually, so $v^{\ev}(P,0)=v(\hat a)$ by the definition of $v^{\ev}(P,0)$. 
\end{proof}

\begin{cor}
Suppose $\hat K$  is ungrounded and equipped with an ordering making it a pre-$H$-field, and assume $\hat a-a \prec 1$ and $v(\hat a-a) \notin \exc^{\ev}_{\hat K}(A)$ where $A:=L_{P_{+a}}$. Then $v(\hat a-a)=v^{\ev}(P,a)$.
\end{cor}
\begin{proof}
In view of Lemma~\ref{lemexc} and using [ADH, 14.5.11] we can extend $\hat K$ to arrange that it is an $\upo$-free newtonian Liouville closed $H$-field. Next, let $H$ be the real closure of the $H$-field hull of $K\langle \hat a\rangle$, all inside~$\hat K$. Then $H$ is $\upo$-free, by Theorem~\ref{thm:ADH 13.6.1}, and hence has a Newton-Liouville closure $L$ inside~$\hat K$~[ADH, 14.5]. Since $L\preccurlyeq \hat K$ by [ADH, 16.2.5], we have~$v({\hat a-a}) \notin \exc^{\ev}_{L}(A)$. 
Now $L$ is $\d$-algebraic over $K$ by [ADH, 14.5.9], so $\Psi$ is cofinal in $\Psi_{L}$ by  Theorem~\ref{thm:ADH 13.6.1}. It remains to apply Lemma~\ref{lem:14.3 complement}. 
\end{proof}

\subsection*{Newton position in the order $1$ case\astr} 
{\it In this subsection $K$ is assumed to be $\upl$-free, $P\in K\{Y\}$ has order~$1$, and $a\in K$.}\/ We basically copy here a definition and two lemmas from~[ADH, 14.3] with the $\upo$-free assumption there replaced by the weaker
$\upl$-freeness, at the cost of restricting $P$ to have order $1$. 

{\sloppy
Suppose $\nval P=1$, $P_0\neq 0$. Then  [ADH, 11.6.17] yields $g\in K^\times$ such that~${P_0\asymp P^\phi_{1,\times g}}$, eventually. Since $P_0\prec P_1^\phi$, eventually, we have $g\prec 1$. 
Moreover, if~${i\geq 2}$, then $P_1^\phi\succeq P_i^\phi$, eventually, hence
$P^\phi_{1,\times g} \succ P^\phi_{i,\times g}$, eventually. Therefore~$\ndeg P_{\times g}=1$.}

Define $P$ to be in {\bf newton position at $a$} if $\nval P_{+a}=1$. \index{differential polynomial!newton position} \index{newton!position} Suppose $P$ is in newton position at $a$; set $Q:=P_{+a}$, so $Q(0)=P(a)$.
If $P(a)\neq 0$, then the above yields $g\in K^\times$ with
$P(a)=Q(0)\asymp Q^\phi_{1,\times g}$, eventually; as $vg$ does not depend on
the choice of such $g$, we set $v^{\ev}(P,a):=vg$. If $P(a)=0$, then we set~$v^{\ev}(P,a):=\infty\in\Gamma_\infty$.
 In passing from $K$ to  a $\upl$-free extension, $P$  remains in newton position at $a$ and~$v^{\ev}(P,a)$ does not change, by Lemma~\ref{lem:11.2.13 invariant}. {\em In the rest of this subsection we assume $P$ is in newton position at $a$}.

\begin{lemma}\label{newt4, order 1}
If $P(a)\ne 0$, then  there exists~$b\in K$ with the following properties: \begin{enumerate}
\item[\textup{(i)}] $P$ is in newton position at $b$, $v(a-b)= v^{\ev}(P,a)$, and $P(b)\prec P(a)$; 
\item[\textup{(ii)}] for all $b^*\in K$ with $v(a-b^*)\ge v^{\ev}(P,a)$: $\ P(b^*)\prec P(a)\Leftrightarrow a-b\sim a-b^*$;
\item[\textup{(iii)}] for all $b^*\in K$, if $a-b\sim a-b^*$, then $P$ is in newton position at $b^*$ and~$v^{\ev}(P,b^*)> v^{\ev}(P,a)$.
\end{enumerate}
\end{lemma} 

\noindent
This is shown as in [ADH, 14.3.2]. Next an analogue of [ADH, 14.3.3], 
with the same proof, but using Lemma~\ref{newt4, order 1} in place of [ADH, 14.3.2]:

\begin{lemma}\label{newt.imm, order 1} If there is no $b$ with~${P(b)=0}$ and 
$v(a-b) = v^{\ev}(P,a)$, then there is a divergent pc-sequence~$(a_\rho)_{\rho<\lambda}$ in~$K$, indexed by all ordinals $\rho$ smaller than 
some infinite limit ordinal $\lambda$, such that~$a_0=a$, $v(a_\rho-a_{\rho'})=v^{\ev}(P,a_\rho)$ for all $\rho<\rho'<\lambda$, and $P(a_\rho) \leadsto 0$.
\end{lemma}

\noindent
The next result is proved just like Lemma~\ref{lem:14.3.1}:

{\sloppy
\begin{lemma}\label{lem:14.3.1, order 1} If $P(\hat a)=0$ with~$\hat a$ in an $H$-asymptotic extension of $K$, then~$v^{\ev}(P,a)>0$ and $v(\hat a-a) \leq v^{\ev}(P,a)$.
\end{lemma}}

\noindent
Next an analogue of Lemma~\ref{lem:14.3 complement} using Propositions~\ref{prop:rat as int and cofinality} and~\ref{prop:upl-free and as int} in its proof: 

{\sloppy
\begin{lemma}\label{lem:14.3 complement, order 1} 
Suppose $\hat a$ in an ungrounded $H$-asymptotic extension  $\hat{K}$ of $K$ satisfies $P(\hat a)=0$, $\hat a - a\prec 1$, and
 $v(\hat a-a)\notin\exc^{\ev}_{\hat K}(A)$, where $A:=L_{P_{+a}}$.  Then~$v({\hat a-a})  =  v^{\ev}(P,a)$. 
\end{lemma}}
\begin{proof}
We arrange $a=0$ and assume $\hat a\neq 0$.
Then $L:=K\langle\hat a\rangle$ has asymptotic integration, by
Proposition~\ref{prop:upl-free and as int}, and $v(\hat a)\notin \exc^{\ev}_L(A)$ by
Lemma~\ref{lemexc, order 1} (applied with~$L$,~$\hat K$ in place of $K$, $L$).
Moreover, $\Psi$ is cofinal in $\Psi_L$
by Proposition~\ref{prop:rat as int and cofinality}.  As in the proof of Lemma~\ref{lem:14.3 complement} this leads to
$P_1(\hat a)=A(\hat a)=A^\phi(\hat a)\asymp A^{\phi}\hat a$, eventually, and then as in the rest of that proof we derive
$v^{\ev}(P,0)=v(\hat a)$. 
\end{proof}

\subsection*{Zeros of differential polynomials  of order and degree $1$} {\sloppy
{\it In this subsection~$K$ has asymptotic integration.}\/
We fix a differential polynomial
$$P(Y)\ =\ a(Y'+gY-u)\in K\{Y\}\qquad (a,g,u\in K,\  a\neq 0),$$
and set $A:=L_P=a(\der+g)\in K[\der]$. Section~\ref{sec:logder} gives  for $y\in K$ the equivalence~${y\in \I(K)\Leftrightarrow vy>\Psi}$, so by Section~\ref{sec:lindiff}, $\exc^{\ev}(A)=\emptyset\Leftrightarrow g\notin \I(K)+K^\dagger$,
and~$v(\ker_{\hat K}^{\neq} A)\subseteq \exc^{\ev}(A)$ for each immediate $H$-asymptotic field extension $\hat K$ of~$K$.
Thus:}

\begin{lemma}\label{lem:at most one zero}
If $g\notin \I(K) + K^\dagger$, then each immediate $H$-asymptotic extension of $K$ contains at most one~$y$ such that~$P(y)=0$.
\end{lemma}

\noindent
If  $\der(K)=K$ and $g\in K^\dagger$, then $P(y)=0$ for some $y\in K$, and if moreover $K$ is $\d$-valued, then any $y$ in any immediate $H$-asymptotic extension of $K$ with $P(y)=0$ lies in $K$. (Lemma~\ref{0K}.)
If $y\prec 1$ in an immediate $H$-asymptotic extension of~$K$ satisfies~$P(y)=0$, then by [ADH, 11.2.3(ii), 11.2.1] 
we have
$$ \nval P\ =\ \nval P_{+y}\ =\ \val P_{+y}\ =\ 1.$$
Lemma~\ref{newt.imm, order 1} 
 yields the following partial converse (a variant of \cite[Lemma~8.5]{VDF}):

\begin{cor}\label{cor:nmul=1}
Suppose $K$ is $\upl$-free and $\nval P=1$. Then there is a $y\prec 1$ in an immediate $H$-asymptotic extension of $K$ with $P(y)=0$.
\end{cor}
\begin{proof}
Replacing $K$ by its henselization and using [ADH, 11.6.7], we arrange that~$K$ is henselian.
Suppose that
$P$ has no zero in $\smallo$. Then $P$ is in newton position at~$0$,
and so   Lemma~\ref{newt.imm, order 1} yields a divergent pc-sequence $(a_{\rho})_{\rho<\lambda}$ in $K$, indexed by all ordinals~$\rho$ smaller than 
some infinite limit ordinal $\lambda$,   with $a_0=0$, $v(a_\rho-a_{\rho'})=v^{\ev}(P,a_\rho)$ for all $\rho<\rho'<\lambda$, and $P(a_{\rho}) \leadsto 0$. Since $\deg P=\order P=1$ and $K$ is henselian, $P$ is a minimal differential polynomial of~$(a_\rho)$ over $K$, and $v(a_\rho)=v^{\ev}(P,0)>0$ for all $\rho>0$.
Hence [ADH, 9.7.6] yields a pseudolimit~$y$ of~$(a_\rho)$ in an immediate asymptotic extension of $K$ with $P(y)=0$
and~$y\prec 1$, as required.
\end{proof}

\noindent
We say that $P$ is {\bf proper} if $u\neq 0$ and $g+u^\dagger\succ^\flat 1$. \index{proper}\index{differential polynomial!proper}\index{proper!differential polynomial} If $P$ is proper, then so is $bP$ for each $b\in K^\times$. For $\fm\in K^\times$ we have
$$P_{\times\fm}\ =\ a\fm\big(Y'+(g+\fm^\dagger)Y-u\fm^{-1}\big),$$
hence if $P$ is proper, then so is $P_{\times\fm}$.
If   $u\neq 0$, then
$P$ is proper iff~$a^{-1}A_{\ltimes u}=\der+(g+u^\dagger)$ is steep, as defined in Section~\ref{sec:lindiff}.
Note that
$$P^\phi\  =\ a\phi\big( Y'+(g/\phi)Y-(u/\phi) \big).$$

\begin{lemma}\label{lem:proper compconj}
Suppose $K$ has small derivation, and $P$ is proper. Then $P^\phi$ is proper $($with respect to  $K^\phi)$ for all $\phi\preceq 1$.
\end{lemma}
\begin{proof} Let $\phi\preceq 1$. Then we have
$\phi\asymp^\flat 1$ and hence $\phi^\dagger\asymp^\flat \phi' \preceq 1\prec^\flat g+u^\dagger$.
Thus
$$g +(u/\phi)^\dagger =(g+u^\dagger)-\phi^\dagger \sim^\flat g+u^\dagger\succ^\flat 1\succeq \phi,$$ hence
$(g/\phi)+\phi^{-1}(u/\phi)^\dagger\succ^\flat  1$ and so~${(g/\phi) + \phi^{-1}(u/\phi)^\dagger} \succ^\flat_\phi 1$.
Therefore $P^\phi$ is proper (with respect to $K^\phi$).
\end{proof}

\begin{lemma}\label{lem:proper evt}
Suppose $K$ is $\upl$-free and $u\neq 0$. Then there is an active $\phi_0$ in~$K$ such that for all $\phi\prec\phi_0$, $P^\phi$ is proper with
$g+(u/\phi)^\dagger\sim g+(u/\phi_0)^\dagger$.
\end{lemma}
{\sloppy
\begin{proof}
The argument before Corollary~\ref{cor:prlemexc} yields an active $\phi_0$ in $K$ such that~${u^\dagger +g-\phi^\dagger\succeq \phi_0}$ for all $\phi\prec \phi_0$. For   such~$\phi$ we have
$\phi^\dagger - \phi_0^\dagger\prec \phi_0$  as noted just before~[ADH, 11.5.3], and so $(u/\phi)^\dagger+g \sim (u/\phi_0)^\dagger+g$. 
The argument before Corollary~\ref{cor:prlemexc} also gives $\phi^{-1}(u/\phi)^\dagger+g/\phi\succ^\flat_\phi 1$ eventually,
and if  $\phi^{-1}(u/\phi)^\dagger+g/\phi\succ^\flat_\phi 1$, then~$P^\phi$ is proper.
\end{proof}}

\begin{lemma}\label{lem:proper nmul 1}
We have $\nval P=1$ iff $u\prec g$ or $u\in\I(K)$. Moreover, if $K$ is $\upl$-free, $\nval P=1$, and $u\neq 0$, then $u\prec^\flat_\phi g+(u/\phi)^\dagger$, eventually.
\end{lemma}
\begin{proof}
For the equivalence, note that the identity above for $P^\phi$ yields: $$\nval P=0\ \Longleftrightarrow\  u\succeq g, \text{ and $u/\phi\succeq 1$ eventually}.$$
Suppose $K$ is $\upl$-free, $\nval P=1$, and $u\neq 0$.
If $u\in\I(K)$, then~$u\prec \phi\prec^\flat_\phi g+(u/\phi)^\dagger$, eventually,
by Lemma~\ref{lem:proper evt}.
Suppose  $u\notin \I(K)$. Then $v(u)\in \Psi^{\downarrow}$ and $u\prec g$. Hence by [ADH, 9.2.11] we have~$(u/\phi)^\dagger \prec u\prec g$, eventually, and thus $u\prec g\sim g+(u/\phi)^\dagger$, eventually. Thus~$u\prec^\flat_\phi g+(u/\phi)^\dagger$, eventually.
\end{proof}

\noindent
Assume now $P(y)=0$ with $y$ in
an immediate $H$-asymptotic extension  of $K$; so~$A(y)=u$. Note: if~$vy\in\Gamma\setminus\exc^{\ev}(A)$, then $u\neq 0$. 
From Lemma~\ref{prlemexc} we get:

\begin{lemma}\label{lem:proper asymp}
If $K$ has small derivation, $P$ is proper, and $vy\in\Gamma\setminus\exc^{\ev}(A)$, then~$y\sim u/(g+u^\dagger)$.
\end{lemma}

\noindent
By Lemmas~\ref{lem:proper evt} and \ref{lem:proper asymp}, and using Lemma~\ref{lem:proper nmul 1} for the last part:

\begin{cor}\label{cor:proper}
If $K$ is $\upl$-free  and $vy\in\Gamma\setminus\exc^{\ev}(A)$, then  
$$y\sim  u/\big( g+(u/\phi)^\dagger\big) \quad\text{ eventually.}$$
If in addition $\nval P=1$, then $y\prec 1$. 
\end{cor}

\subsection*{A characterization of $1$-linear newtonianity} 
{\it In this subsection $K$ has asymptotic integration.}\/
We first expand~[ADH, 14.2.4]:

\begin{prop}\label{prop:char 1-linearly newt} 
The following are equivalent:
\begin{enumerate}
\item[\textup{(i)}] $K$ is $1$-linearly newtonian;
\item[\textup{(ii)}] every $P\in K\{Y\}$ with $\nval P=\deg P=1$ and $\order P\leq 1$ has a zero in~$\smallo$;
\item[\textup{(iii)}] $K$ is $\d$-valued, $\upl$-free, and $1$-linearly surjective, with $\I(K)\subseteq K^\dagger$.
\end{enumerate}
\end{prop}
\begin{proof} 
The equivalence of (i) and (ii) is [ADH, 14.2.4], and the implication (i)~$\Rightarrow$~(iii) follows from [ADH, 14.2.2, 14.2.3, 14.2.5].
To show (iii)~$\Rightarrow$~(ii), suppose (iii) holds, and
let $g,u\in K$ and $P=Y'+gY-u$ with $\nval P=1$. We need to find~$y\in \smallo$ such that $P(y)=0$.
Corollary~\ref{cor:nmul=1} gives an element~$y\prec 1$ in an immediate $H$-asymptotic extension $L$ of $K$ with $P(y)=0$.
It suffices to show that then
$y\in K$ (and thus~${y\in \smallo}$). 
If $g\notin K^\dagger$, then  this follows from Lemma~\ref{lem:at most one zero}, using $\I(K)\subseteq K^\dagger$ and $1$-linear surjectivity of $K$; if $g\in K^\dagger$, then this follows from Lemma~\ref{0K} and~$\der(K)=K$.
\end{proof}

\noindent
By the next corollary, each Liouville closed $H$-field is $1$-linearly newtonian:

\begin{cor} \label{cor:Liouville closed => 1-lin newt}
Suppose $K^\dagger=K$. Then the following are equivalent:
\begin{enumerate}
\item[$\mathrm{(i)}$] $K$ is $1$-linearly newtonian;
\item[$\mathrm{(ii)}$] $K$ is $\d$-valued and $1$-linearly surjective;
\item[$\mathrm{(iii)}$] $K$ is $\d$-valued and $\der(K)=K$.
\end{enumerate}
\end{cor}
\begin{proof}
Note that $K$ is $\upl$-free by [ADH, remarks following 11.6.2].
Hence the equivalence of (i) and (ii) follows from Proposition~\ref{prop:char 1-linearly newt}. For the equivalence of (ii) with~(iii), see [ADH, example following 5.5.22]. 
\end{proof}

\subsection*{Linear newtonianity descends} 
{\em In this subsection $H$ is $\d$-valued with valuation ring 
$\mathcal{O}$ and constant field $C$. Let~$r\in\N^{\ge 1}$}.
If $H$ is $\upo$-free, $\Gamma$ is divisible, and~$H$ has a newtonian algebraic extension $K=H(C_K)$, then
$H$ is also newtonian, by \eqref{eq:14.5.6}.  Here is an analogue of this for $r$-linear newtonianity:

\begin{lemma}\label{lem:descent r-linear newt}
Let $K=H(C_K)$ be an algebraic asymptotic extension of $H$ which is $r$-linearly newtonian. Then $H$ is $r$-linearly newtonian.
\end{lemma}
\begin{proof}
Take a basis $B$ of the $C$-linear space $C_K$ with $1\in B$, and let $b$ range over~$B$.
We have $H(C_K)=H[C_K]$, and $H$ is linearly disjoint from $C_K$ over $C$ [ADH, 4.6.16],
so $B$ is a basis of the $H$-linear space $H[C_K]$.
Let $P\in H\{Y\}$ with~$\deg P=1$ and~$\order(P)\leq r$ be quasilinear; then $P$ as element of~$K\{Y\}$ remains
quasilinear, since~${\Gamma_K=\Gamma}$ by  [ADH, 10.5.15].
Let $y\in\mathcal O_K$ be a zero of $P$. Take~$y_b\in H$ ($b\in B$)  with
$y_b=0$ for all but finitely many $b$ and $y=\sum_b y_b\,b$. Then $y_b\in\mathcal O$ for all~$b$, and
$$0\ =\ P(y)\ =\ P_0 + P_1(y)\ =\ P_0 + \sum_b P_1(y_b)b,$$
so $P(y_1) = P_0 + P_1(y_1) = 0$.
\end{proof}
 
\noindent
Thus if $H[\imag]$ with $\imag^2=-1$ is $r$-linearly newtonian, then $H$ is $r$-linearly newtonian. 

\subsection*{Cases of bounded order} 
{\it In the rest of this section $r\in\N^{\ge 1}$}.\index{H-asymptotic field@$H$-asymptotic field!strongly $r$-newtonian}\index{strongly!$r$-newtonian}\index{r-newtonian@$r$-newtonian!strongly}
We define~$K$ to be {\bf strongly $r$-newtonian\/} if $K$ is $r$-newtonian and for each divergent pc-sequence~$(a_{\rho})$ in $K$ with minimal differential polynomial $G(Y)$ over $K$ of order~$\le r$ we have~${\ndeg_{\boldsymbol a} G =1}$,
where ${\boldsymbol a}:=c_K(a_{\rho})$. Given $P\in K\{Y\}^{\ne}$, a {\bf $K$-external zero of $P$\/}\index{zero!K-external@$K$-external}\index{K-external zero@$K$-external zero}\index{differential polynomial!K-external zero@$K$-external zero}
is an element~$\hat{a}$ of some immediate asymptotic extension $\hat{K}$ of $K$ with~$P(\hat{a})=0$ and $\hat{a}\notin K$. Now~[ADH, 14.1.11] extends as follows with the same proof: 

\begin{lemma}\label{14.1.11.r} Suppose $K$ has rational asymptotic integration and $K$ is strongly $r$-newtonian. Then no $P\in K\{Y\}^{\ne}$ of order $\le r$ can have a
$K$-external zero.
\end{lemma} 

\noindent
The following is important in certain inductions on the order. 
A differential field~$F$ is {\it $r$-linearly closed}\/ ($r\in\N$) if every $A\in F[\der]^{\ne}$ of order $\leqslant r$ splits over $F$.
So $F$ is {\it linearly closed}\/ iff it is $r$-linearly closed for all $r\in\N$.\index{differential field!r-linearly closed@$r$-linearly closed}\index{r-linearly closed@$r$-linearly closed}\index{closed!r-linearly@$r$-linearly}

\begin{lemma}\label{rlcrln} Suppose $K$ has asymptotic integration, is $1$-linearly newtonian, and
$r$-linearly closed. Then $K$ is $r$-linearly newtonian.
\end{lemma}
\begin{proof} Note that $K$ is $\upl$-free and $\d$-valued by Proposition~\ref{prop:char 1-linearly newt}. 
Let~$P\in K\{Y\}$ be such that $\operatorname{nmul} P=\deg P=1$ and $\order P\le r$; by~[ADH, 14.2.6] it suffices to show that then $P$ has a zero in $\smallo$. By [ADH, proof of 13.7.10] we can compositionally conjugate, pass to an elementary extension, and multiply by an element of $K^\times$ to arrange that $K$ has small derivation, $P_0\prec^{\flat} 1$, and $P_1\asymp 1$. 
Let $A:= L_P$. The valuation ring of the flattening 
$(K, v^\flat)$ is $1$-linearly surjective by~[ADH, 14.2.1],
so all operators in $K[\der]$ of order $1$ are neatly
surjective in the sense of $(K, v^\flat)$. Since~$A$ splits over $K$,
we obtain from [ADH, 5.6.10(ii)] that~$A$ is neatly surjective in the sense of $(K,v^\flat)$. As $v^{\flat}(A)=0$ and $v^{\flat}(P_0)>0$, this gives~$y\in K$ with
$v^\flat(y)>0$ such that $P_0+A(y)=0$, that is, $P(y)=0$. 
\end{proof}  

\noindent
Using the terminology of $K$-external zeros,
we can add another item to the list of equivalent statements in Proposition~\ref{prop:char 1-linearly newt}:

\begin{lemma}\label{lem:char 1-linearly newt}
Suppose $K$ has asymptotic integration. Then we have:
\begin{align*} \text{$K$ is $1$-linearly newtonian}\ \Longleftrightarrow\ 
&\text{$K$ is $\upl$-free and no $P\in K\{Y\}$ with $\deg P=1$}\\
&\text{and $\order P=1$ has a $K$-external zero.}
\end{align*} 
\end{lemma}
\begin{proof}
Suppose $K$ is $1$-linearly newtonian. Then by (i)~$\Rightarrow$~(iii) in Proposition~\ref{prop:char 1-linearly newt}, $K$ is
$\upl$-free, $\d$-valued, $1$-linearly surjective, and~$\I(K)\subseteq K^\dagger$.
Let $P\in K\{Y\}$ where~$\deg P=\order P=1$ and  $y$ in an immediate asymptotic extension $L$ of~$K$ with~$P(y)=0$. Then [ADH, 9.1.2] and Corollary~\ref{cor:no new LDs} give $L^\dagger\cap K=K^\dagger$, so $y\in K$ by
Lemmas~\ref{0K} and~\ref{1K}. This gives the direction $\Rightarrow$.  
 The converse follows from Corollary~\ref{cor:nmul=1} and  
(ii)~$\Rightarrow$~(i) in Proposition~\ref{prop:char 1-linearly newt}.
\end{proof}

\noindent
Here is a higher-order version of Lemma~\ref{lem:char 1-linearly newt}:

\begin{lemma}\label{lem:char r-linearly newt}
Suppose $K$ is $\upo$-free. Then
\begin{align*} \text{$K$ is $r$-linearly newtonian}\quad \Longleftrightarrow\quad 
&\text{no $P\in K\{Y\}$ with $\deg P=1$ and $\order P\leq r$}\\
&\text{has a $K$-external zero.}
\end{align*}
\end{lemma}
\begin{proof}
Suppose $K$ is $r$-linearly newtonian. Then $K$ is $\d$-valued by Lemma~\ref{lem:ADH 14.2.5}.
Let~$P\in K\{Y\}$ be of degree~$1$ and order~$\leq r$, and let $y$ be in an immediate asymptotic extension $L$ of $K$
with $P(y)=0$. Then $A(y)=b$ for~${A:=L_P\in K[\der]}$, $b:=-P(0)\in K$. By~[ADH, 14.2.2] there is also a $z\in K$ with $A(z)=b$, hence~${y-z\in\ker_L A=\ker A}$ by   [ADH, remarks after 14.2.9] and so $y\in K$. 
This gives the direction~$\Rightarrow$. For the converse note that every quasilinear  $P\in K\{Y\}$   has a zero~${\hat a\preceq 1}$ in an immediate asymptotic extension of $K$ by \eqref{eq:14.0.1}.    %[ADH, 14.0.1 and subsequent remarks].
\end{proof}

\noindent
We also have the following $r$-version of \eqref{eq:14.0.1}:   %[ADH, 14.0.1]:

\begin{prop}\label{14.0.1r} If $K$ is $\upl$-free and no $P\in K\{Y\}^{\ne}$ of order $\le r$ has a $K$-external zero, then $K$ is
$\upo$-free and $r$-newtonian.
\end{prop}

\begin{proof}
The $\upo$-freeness follows as before from [ADH, 11.7.13]. 
The rest of the proof is as in [ADH, p.~653] with $P$ restricted to have order $\le r$.
\end{proof}

\subsection*{Application to solving asymptotic equations} {\it Here $K$ is $\d$-valued, $\upo$-free, with small derivation, and $\fM$ is a monomial group of $K$.}\/ See [ADH, 3.3] for ``monomial group'', and [ADH, 13.8] for ``asymptotic equation''.  We let $a$,~$b$,~$y$ range over~$K$. 
In addition we fix a $P\in K\{Y\}^{\ne}$ of order~$\le r$ and
a $\preceq$-closed set~$\E\subseteq K^\times$. (Recall that $r\ge 1$.) This gives the asymptotic equation
\begin{equation}\label{eq:asymp equ}\tag{E}
P(Y)=0,\qquad Y\in \E.
\end{equation} 
This gives the following
$r$-version of [ADH, 13.8.8], with basically the same proof: 

\begin{prop}\label{1388} Suppose $\Gamma$ is divisible, no $Q\in K\{Y\}^{\ne}$ of order $\le r$ has a $K$-external zero, $d:=\ndeg_{\E} P\ge 1$, and there is no $f\in \E\cup\{0\}$ with $\operatorname{mul} P_{+f}=d$. Then \eqref{eq:asymp equ} has an unraveler.
\end{prop}

\noindent
Here is an $r$-version of [ADH, 14.3.4] with the same proof:

\begin{lemma} Suppose $K$ is $r$-newtonian. Let $g\in K^\times$ be an approximate zero of~$P$ with $\ndeg P_{\times g}=1$.
Then there exists $y\sim g$ such that $P(y)=0$. 
\end{lemma}

\noindent
For the next three results we assume the following:

\medskip\noindent
{\em  $C$ is algebraically closed, $\Gamma$ is divisible, and 
no $Q\in K\{Y\}^{\ne}$ of order $\le r$ has a $K$-ex\-ter\-nal zero}.

\medskip\noindent
These three results are $r$-versions of [ADH, 14.3.5, 14.3.6, 14.3.7] with the same proofs, using Propositions~\ref{14.0.1r} and~\ref{1388} instead of \eqref{eq:14.0.1} and [ADH,  13.8.8]:

\begin{prop} If $\ndeg_{\E}P > \operatorname{mul}(P)=0$, then \eqref{eq:asymp equ} has a solution. 
\end{prop}

\begin{cor}\label{wrdc} $K$ is weakly $r$-differentially closed: for each $Q\in K\{Y\}\setminus K$ of or\-der~$\leq r$ there is a $y\in K$ with $Q(y)=0$.
\end{cor} 

\begin{cor} Suppose $g\in K^\times$ is an approximate zero of $P$. Then $P(y)=0$ for some $y\sim g$. 
\end{cor} 

\subsection*{A useful equivalence} {\it Suppose $K$ is $\upo$-free.}\/ 
(No small derivation or monomial group assumed.) Recall that $r\ge 1$. Here is an $r$-version of \cite[3.4]{Nigel19}: 

\begin{cor}\label{14.5.2.r} The following are equivalent: 
\begin{enumerate}
\item[\textup{(i)}] $K$ is $r$-newtonian;
\item[\textup{(ii)}] $K$ is strongly $r$-newtonian;
\item[\textup{(iii)}] no $P\in K\{Y\}^{\ne}$ of order $\le r$ has a $K$-external zero.
\end{enumerate}
\end{cor}
\begin{proof} Since $K$ is $\upo$-free  it has rational asymptotic integration [ADH, p.~515]. Also,
if $K$ is $1$-newtonian, then $K$ is henselian [ADH, p.~645] and $\d$-valued [ADH, 14.2.5].
For~(i)~$\Rightarrow$~(ii), use 
\cite[3.3]{Nigel19}, for (ii)~$\Rightarrow$~(iii), use Lemma~\ref{14.1.11.r}, and for (iii)~$\Rightarrow$~(i), use Proposition~\ref{14.0.1r}.
\end{proof}  

\noindent
Next an $r$-version of  \eqref{eq:14.5.3}:  %[ADH, 14.5.3]: 

\begin{cor}\label{14.5.3.r} Suppose $K$ is $r$-newtonian, $\Gamma$ is divisible, and $C$ is algebraically closed. Then $K$ is weakly $r$-differentially closed, so $K$ is $(r+1)$-linearly closed and thus~${(r+1)}$-linearly newtonian.  
\end{cor}
\begin{proof} To show that $K$ is weakly $r$-differentially closed we arrange by compositional conjugation and passing to a suitable elementary extension that $K$ has small derivation and $K$ has a monomial group. Then
$K$ is weakly $r$-differentially closed by 
Corollaries~\ref{wrdc} and~\ref{14.5.2.r}. The rest uses [ADH, 5.8.9] and
Lemma~\ref{rlcrln}. 
\end{proof}

\newpage  

\part{The Universal Exponential Extension}\label{part:universal exp ext}

\medskip

\noindent
Let $K$ be an algebraically closed differential field.
In Section~\ref{sec:univ exp ext} below we extend~$K$ in a canonical way to a differential integral  domain $\Univ=\Univ_K$
whose differential fraction field has the same constant field $C$ as $K$,
called the {\it universal exponential extension}\/ of~$K$. (The universal exponential extension of $\mathbb T[\imag]$ appeared in \cite{JvdH} in the guise of ``oscillating transseries''; we explain
the connection at the end of Section~\ref{sec:eigenvalues and splitting}.)
The underlying ring of $\Univ$ is a   group ring of a certain abelian group over~$K$, and we therefore first review some relevant basic facts about such group rings
in Section~\ref{sec:group rings}. The main feature of $\Univ$ is that if $K$ is $1$-linearly surjective, then   each $A\in K[\der]$ of order $r\in\N$ which splits over $K$ has~$r$  many $C$-linearly independent zeros in $\Univ$. This is explained in Section~\ref{sec:eigenvalues and splitting}, after some differential-algebraic preliminaries in Section~\ref{sec:splitting}, where we consider a novel kind of {\it spectrum}\/ of a linear differential operator over a differential field.
In Section~\ref{sec:valuniv} we introduce for $H$-asymptotic~$K$ with small derivation and asymptotic integration
the  {\it ultimate exceptional values}\/ of a given linear differential operator $A\in K[\der]^{\neq}$. These  
help to isolate the zeros of~$A$ in~$\Univ$ much like
the exceptional values of $A$ help to locate the zeros of~$A$ in immediate asymptotic extensions of $K$ as in~Section~\ref{sec:lindiff}.

Of this chapter,  the construction of $\Univ$ (Sections~\ref{sec:group rings} and~\ref{sec:univ exp ext})  and the set of ultimate exceptional values with its basic properties (Section~\ref{sec:valuniv}) are essential in Chapter 4, but the definition of ``ultimate exceptional values''  uses notation and terminology from Sections~\ref{sec:splitting} and~\ref{sec:eigenvalues and splitting}. 
 In \cite{ADH5} we discuss the analytic meaning of~$\Univ$ when $K$ is the algebraic closure of a Liouville closed Hardy field containing $\R$ as a subfield.
 In a follow-up paper~\cite{ADH6} we use the main theorem of~\cite{ADH5} together with the results from Sections~\ref{sec:splitting} and~\ref{sec:eigenvalues and splitting}   to study the solutions of linear differential equations over Hardy fields.

\section{Some Facts about Group Rings}\label{sec:group rings}

\noindent
{\it In this section $G$ is a torsion-free abelian group, written multiplicatively, $K$ is a field, and~$\gamma$,~$\delta$ range over $G$.}\/
For use in Section~\ref{sec:univ exp ext} below we recall some facts about the group ring $K[G]$: a commutative $K$-algebra with $1\ne 0$ that contains $G$ as a subgroup of its multiplicative group~$K[G]^\times$ and
which, as a $K$-linear space, decomposes as \index{group ring}
$$K[G]\ =\ \bigoplus_\gamma K\gamma\qquad\text{(internal direct sum)}.$$
Hence for any $f\in K[G]$ we have a unique family $(f_\gamma)$ of elements of $K$, with $f_\gamma=0$  for all but finitely many $\gamma$,
such that 
\begin{equation}\label{eq:elt of group ring}
f\ =\ \sum_\gamma f_\gamma \gamma.
\end{equation}
We define the support of $f\in K[G]$ as above by
$$\supp(f)\ :=\ \{\gamma:\, f_\gamma\neq 0\}\ \subseteq\ G.$$
{\em In the rest of this section~$f$,~$g$,~$h$ range over $K[G]$}.
For any $K$-algebra $R$, every group morphism $G\to R^\times$ extends uniquely to a $K$-algebra morphism $K[G]\to R$. 

\medskip
\noindent
Clearly $K[G]^\times \supseteq K^\times G$; in fact:

\begin{lemma}\label{lem:only trivial units}
The ring $K[G]$ is an integral domain and $K[G]^\times = K^\times G$.
\end{lemma}
\begin{proof} We take an ordering of~$G$ making $G$ into an ordered abelian group; see~[ADH, 2.4]. Let $f,g\neq 0$ and set 
$$\gamma^-:=\min\supp(f),\ \gamma^+:=\max\supp(f),\quad 
 \delta^-:=\min\supp(g),\ \delta^+:=\max\supp(g);$$
so  $\gamma^- \leq \gamma^+$ and $\delta^-\leq \delta^+$.
We have 
$(fg)_{\gamma^- \delta^-} = f_{\gamma^-}g_{\delta^-}\neq 0$, and likewise with~$\gamma^+$,~$\delta^+$ in place of $\gamma^-$, $\delta^-$.
In particular, 
$fg\neq 0$, showing that~$K[G]$ is an integral domain.
Now suppose $fg=1$. Then $\supp(fg)=\{1\}$, hence $\gamma^- \delta^-=1=\gamma^+\delta^+$, so $\gamma^-=\gamma^+$, and thus $f\in K^\times G$.
\end{proof}

\begin{lemma}\label{lem:Frac U not alg closed} 
Suppose $K$ has characteristic $0$ and $G\neq\{1\}$. Then the fraction field $\Omega$ of $K[G]$ is not algebraically closed.
\end{lemma}
\begin{proof} 
Let $\gamma\in G\setminus\{1\}$ and $n\geq 1$. We claim that there is no $y\in\Omega$ with~$y^2=1-\gamma^n$.
For this, first replace $G$ by its divisible hull to arrange that $G$ is divisible. 
Towards a contradiction, suppose $f,g\in K[G]^{\neq}$ and $f^2=g^2(1-\gamma^n)$.
Take a divisible subgroup~$H$ of  $G$ that is complementary to the smallest divisible subgroup $\gamma^{\Q}$ of $G$ containing $\gamma$, so $G=H\gamma^{\Q}$ and $G\cap \gamma^{\Q}=\{1\}$. Then
$K[G]\subseteq K(H)[\gamma^{\Q}]$ (inside $\Omega$), so we may
replace $K$, $G$ by $K(H)$, $\gamma^{\Q}$ to arrange~$G=\gamma^{\Q}$. 
For suitable~$m\geq 1$ we apply the $K$-algebra automorphism of~$K[G]$ given by $\gamma\mapsto \gamma^m$ to arrange
$f,g\in K[\gamma,\gamma^{-1}]$ (replacing $n$ by $mn$). 
Then replace~$f$,~$g$ by~$\gamma^m f$,~$\gamma^m g$ for suitable $m\ge 1$  
to arrange~$f,g\in K[\gamma]$. Now use that~$1-\gamma$ is a prime divisor of~$1-\gamma^n$ of multiplicity~$1$  in the UFD $K[\gamma]$ to get a contradiction.
\end{proof}

\noindent
The $K$-linear map 
$$f\mapsto \operatorname{tr}(f):=f_1\ \colon\  K[G]\to K$$ is called the
{\bf trace} of $K[G]$. \index{group ring!trace} \index{trace} Thus
$$\operatorname{tr}(fg)\ =\ \sum_\gamma f_\gamma g_{\gamma^{-1}}.$$
We claim that $\operatorname{tr}\circ\sigma=\operatorname{tr}$ for every automorphism $\sigma$ of the $K$-algebra $K[G]$. This invariance comes from an intrinsic description of $\operatorname{tr}(f)$ as follows: given $f$ we have a unique finite set $U\subseteq K[G]^\times = K^\times G$
such that $f=\sum_{u\in U} u$ and $u_1/u_2\notin K^\times$ for all distinct
$u_1, u_2\in U$; if $U\cap K^\times=\{c\}$, then~$\operatorname{tr}(f)=c$;  if $U\cap K^\times=\emptyset$, then~$\operatorname{tr}(f)=0$. 
If $G_0$ is a subgroup of $G$ and $K_0$ is a  subfield  of $K$, then~$K_0[G_0]$ is a subring of $K[G]$, and the trace of $K[G]$ extends the trace of $K_0[G_0]$.

\subsection*{The automorphisms of $K[G]$} For a commutative group $H$, 
written multiplicatively, $\Hom(G,H)$ denotes the set of group morphisms $G\to H$, made into a group by pointwise multiplication. Any $\chi\in\Hom(G,K^\times)$---sometimes called a {\em character}---gives  a $K$-algebra automorphism $f\mapsto f_\chi$ of $K[G]$ defined by
\begin{equation}\label{eq:fchi}
f_\chi\ :=\ \sum_\gamma f_\gamma \chi(\gamma)\gamma.
\end{equation}
This yields a group action of $\Hom(G,K^\times)$ on $K[G]$ by $K$-algebra automorphisms:
$$\Hom(G,K^\times)\times K[G]\to K[G], \qquad (\chi, f)\mapsto f_{\chi}.$$ 
Sending $\chi\in\Hom(G,K^\times)$ to  $f\mapsto f_\chi$ yields an   embedding of the group $\Hom(G,K^\times)$ into the group~$\Aut(K[G]|K)$ 
of automorphisms of the $K$-algebra $K[G]$; its 
 image is the (commutative) subgroup of $\Aut(K[G]|K)$ consisting of the $K$-algebra automorphisms~$\sigma$ of $K[G]$ such that $\sigma(\gamma)/\gamma\in K^\times$ for all $\gamma$. Identify $\Hom(G,K^\times)$ with its image under this embedding. From $K[G]^\times=K^\times G$ we obtain $\sigma(K^\times G)=K^\times G$ for all  $\sigma\in\Aut(K[G]|K)$, and using this one verifies easily that $\Hom(G,K^\times)$ is a normal subgroup of $\Aut(K[G]|K)$. 
 We also have the group embedding 
 $$\Aut(G)\ \to\ \Aut(K[G]|K)$$ assigning to each
$\sigma\in \Aut(G)$ the unique automorphism of the $K$-algebra $K[G]$ extending $\sigma$. Identifying $\Aut(G)$ with its image in $\Aut(K[G]|K)$ via this embedding we have $\Hom(G,K^\times)\cap\Aut(G)=\{\text{id}\}$ and 
$\Hom(G,K^\times)\cdot \Aut(G)=\Aut(K[G],| K)$ inside $\Aut(K[G]|K)$, and thus 
 $\Aut(K[G]|K)=\Hom(G,K^\times)\rtimes\Aut(G)$, an internal semidirect product of subgroups of $\Aut(K[G]|K)$.

\subsection*{The gaussian extension}
{\it In this subsection $v\colon K^\times\to\Gamma$ is a valuation on the field $K$.}\/
We extend $v$ to a map $v_{\g}\colon K[G]^{\neq}\to\Gamma$ by setting
\begin{equation}\label{eq:gaussian ext}
v_{\g}f\ :=\ \min_\gamma vf_\gamma\qquad \text{($f\in K[G]^{\neq}$ as in \eqref{eq:elt of group ring}).}
\end{equation}

\begin{prop} The map $v_{\g}\colon K[G]^{\neq}\to\Gamma$ is a valuation on the domain $K[G]$.
\end{prop}

{\sloppy
 \begin{proof} 
We can reduce to the case that $G$ is finitely generated,
since $K[G]$ is the directed union of its subrings $K[G_0]$ with $G_0$ a finitely generated subgroup of~$G$. 
We then have a group isomorphism~$G\to \Z^n$ inducing a $K$-al\-ge\-bra isomorphism~${K[G]\to K[X_1, X_1^{-1},\dots, X_n, X_n^{-1}]}$ (with distinct indeterminates~$X_1,\dots,X_n$) under which~$v_{\g}$ corresponds to the gaussian extension of the valuation of $K$ to $K(X_1,\dots,X_n)$ restricted to its subring
$K[X_1, X_1^{-1},\dots, X_n, X_n^{-1}]$; see~{[ADH,~3.1].} 
\end{proof}}

{\sloppy
\noindent
We call $v_{\g}$ the {\bf gaussian extension} of the valuation of $K$ to~$K[G]$.\index{group ring!gaussian extension of a valuation}\index{gaussian extension}\index{valuation!gaussian extension}\index{extension!gaussian} We
denote by~$\preceq_{\g}$ the dominance relation on $\Omega:=\Frac(K[G])$ associated to the extension of~$v_{\g}$ to a valuation on the field $\Omega$  [ADH, (3.1.1)], with corresponding asymptotic relations~$\asymp_{\g}$ and~$\prec_{\g}$. For the subring $\mathcal O[G]$ of $K[G]$ generated by $G$ over $\mathcal O$ we have
$$\mathcal O[G]\ =\ \{ f:\, f\preceq_{\g} 1\}.$$
The residue morphism~$\mathcal O\to\k:=\mathcal O/\smallo$
extends to a surjective ring mor\-phism~$\mathcal O[G]\to\k[G]$ with $\gamma\mapsto\gamma$ for all~$\gamma$ and whose kernel
is the ideal $$\smallo[G]:=\ \{f :\, f\prec_{\g} 1\}$$
of~$\mathcal O[G]$. Hence this ring morphism induces an isomorphism  $\mathcal O[G]/\smallo[G]\cong\k[G]$.
If~$G_0$ is subgroup of $G$ and $K_0$ is a valued subfield of $K$, then the restriction of~$v_{\g}$ to a valuation on $K_0[G_0]$ is the gaussian extension of the valuation of $K_0$ to $K_0[G_0]$.\label{p:vg} }

\subsection*{An inner product and two norms}
{\it In the rest of this section $H$ is a real closed subfield of $K$ such that $K=H[\imag]$ where $\imag^2=-1$.}\/ In later use $H$ will be a Hardy field, which is why we use the letter $H$ here. Note that the only nontrivial automorphism of the (algebraically closed) field $K$ over $H$ is {\em complex conjugation}\/: 
$$z=a+b\imag \mapsto \overline{z}:=a-b\imag\qquad (a,b\in H).$$ 
For $f$ as in \eqref{eq:elt of group ring} we set
$$ f^*\ :=\ \sum_\gamma \overline{f_\gamma} \gamma^{-1},$$
so $(f^*)^*=f$, and $f\mapsto f^*$ lies in $\Aut\!\big(K[G]|H\big)$. 
We define the function $$(f,g)\mapsto \langle f,g\rangle\ :\ K[G]\times K[G]\to K$$ by
$$\langle f,g\rangle\ :=\ \operatorname{tr}\!\big(f g^*\big)\ =\ \sum_\gamma f_\gamma  \overline{g_\gamma}.$$
One verifies easily that this  is a ``positive definite hermitian form'' on the $K$-linear space $K[G]$: it is additive on the left and on the right, and for all $f$,~$g$ and all~$\lambda\in K$: $\langle \lambda f, g \rangle=\lambda \langle f,g \rangle$, $\langle g,f\rangle=\overline{\langle f,g\rangle}$, 
$\langle f, f\rangle \in H^{\ge}$, and $\langle f, f\rangle=0\Leftrightarrow f=0$, and thus also $\langle f, \lambda g \rangle=\overline{\lambda}\langle f,g \rangle$. (Hermitian forms are usually defined only on $\mathbb{C}$-linear spaces and are $\mathbb{C}$-valued, which is why we used quotation marks, as we do below for {\em norm\/} and {\em orthonormal basis}; see \cite[Chapter~XV, \S{}5]{Lang} for
the more general case.)
Note:
$$\langle f,gh\rangle\ =\ \operatorname{tr}\!\big(f(g h)^*\big)\ =\ \big\langle f g^*,h\big\rangle.$$

\begin{lemma}\label{lem:inner prod intrinsic}
Let $u,w\in K[G]^\times$. If $u\notin K^\times w$, then $\langle u,w\rangle = 0$, and if $u\in K^\times w$, then $\langle u,w\rangle = uw^*$.
\end{lemma}
\begin{proof}
Take $a,b\in K^\times$ and $\gamma$, $\delta$ such that $u=a\gamma$, $w=b\delta$.
If $u\notin K^\times w$, then~$\gamma\neq\delta$, so~$\langle u,w\rangle = 0$.
If $u\in K^\times w$, then $\gamma=\delta$, hence $\langle u,w\rangle = a\overline{b} = uw^*$.
\end{proof}

\noindent
For $z\in K$ we set $\abs{z}:=\sqrt{z\overline{z}}\in H^{\ge}$, and then define $\dabs{\,\cdot\,}\colon K[G]\to H^{\geq}$ by
$$\dabs{f}^2\ =\ \langle f,f\rangle\ =\ \sum_\gamma\, \abs{f_\gamma}^2.$$
As in the case $H=\R$ and $K=\mathbb{C}$ one derives the Cauchy-Schwarz Inequality:
$$|\langle f,g\rangle|\ \le\ \dabs{f}\cdot\dabs{g}.$$
Thus $\dabs{\,\cdot\,}$ is a ``norm'' on the $K$-linear space $K[G]$: for all $f,g$ and all $\lambda\in K$,  
$$\dabs{f+g}\le \dabs{f} + \dabs{g}, \quad \dabs{\lambda f}=\abs{\lambda}\cdot\dabs{f}, \quad \dabs{f}=0\Leftrightarrow f=0.$$
Note that $G$ is an ``orthonormal basis'' of $K[G]$ with respect to $\langle\ ,\,\rangle$, and  $f_\gamma= \langle f,\gamma\rangle$. 
We also use the function $\dabs{\,\cdot\,}_1\colon K[G]\to H^{\geq}$ given by
$$\dabs{f}_1\ :=\ \sum_\gamma\, \abs{f_\gamma},$$
which is a ``norm'' on $K[G]$ in the sense of obeying the same laws as we
mentioned for $\dabs{\,\cdot\,}$.  \index{group ring!norms}
The two ``norms'' are in some sense equivalent:
\[
\dabs{f}\ \leq\ \dabs{f}_1\ \leq\ \sqrt{\abs{\supp(f)}}\cdot \dabs{f}.
\]
where the first inequality follows from the triangle inequality for $\dabs{\,\cdot\,}$ and the second is of Cauchy-Schwarz type. Moreover:

\begin{lemma}\label{lem:submult}
Let $u\in K[G]^\times$. Then
$\dabs{f u} = \dabs{f}\,\dabs{u}$ and $\dabs{f u}_1 = \dabs{f}_1\,\dabs{u}_1$.
\end{lemma}
\begin{proof} 
We have
$$\dabs{f\gamma}\ =\ \langle f\gamma, f\gamma\rangle\ =\ \big\langle f\gamma \gamma^*,f \big\rangle\ =\ 
\big\langle f ,f\big\rangle\ =\
\dabs{f} $$
using $\gamma^*=\gamma^{-1}$. 
Together with $K[G]^\times=K^\times G$ this yields
the first claim; the second claim follows easily from the definition of $\dabs{\,\cdot\,}_1$.
\end{proof}

\begin{cor}\label{cor:submult}
$\dabs{fg}\leq\dabs{f}\,\cdot\dabs{g}_1$ and $\dabs{fg}_1\leq\dabs{f}_1\,\cdot\dabs{g}_1$.
\end{cor}
\begin{proof}
By  the triangle inequality for $\dabs{\,\cdot\,}$ and  the previous lemma,
$$\dabs{fg}\  \leq\ \sum_\gamma\, \dabs{f g_\gamma\gamma}\ 
=\ \sum_\gamma\, \dabs{f}\, \dabs{g_\gamma\gamma}\ =\ \dabs{f} \sum_\gamma\, \abs{g_\gamma}\  =\ \dabs{f}\,\dabs{g}_1.$$
The inequality involving $\dabs{fg}_1$ follows likewise.
\end{proof}

\noindent
In the next lemma we let $\chi\in\Hom(G,K^\times)$; recall from \eqref{eq:fchi} the automorphism~$f\mapsto f_\chi$ of the $K$-algebra $K[G]$. 

\begin{lemma}\label{lem:commuting with f*}
 $(f_\chi)^*=(f^*)_\chi$ iff~$\abs{\chi(\gamma)}=1$ for all $\gamma\in\supp(f)$.
\end{lemma}
\begin{proof}
Let $a\in K$; then $\big((a\gamma)_\chi\big){}^*=\overline{a\chi(\gamma)}\gamma^{-1}$ and
$\big((a\gamma)^*\big){}_\chi = \overline{a}\chi(\gamma)^{-1}\gamma^{-1}$.
\end{proof}

\begin{cor}\label{cor:commuting with f*}
Let $\chi\in\Hom(G,K^\times)$ with $\abs{\chi(\gamma)}=1$ for all $\gamma$. Then~$\langle f_\chi,g_\chi \rangle=\langle f,g\rangle$ for all $f$, $g$, and hence~$\dabs{f_\chi}=\dabs{f}$ for all $f$.
\end{cor}
\begin{proof}
Since $\operatorname{tr}\circ\sigma=\operatorname{tr}$ for every automorphism $\sigma$ of the $K$-algebra $K[G]$, 
$$\langle f_\chi,g_\chi\rangle\ =\ \operatorname{tr}\!\big(f_\chi(g_\chi)^*\big)\ =\ \operatorname{tr}\!\big( (fg^*)_\chi \big)\ =\
\operatorname{tr}(fg^*)\ =\ \langle f,g\rangle,$$
where we use Lemma~\ref{lem:commuting with f*} for the second equality.
\end{proof}

\subsection*{Valuation and norm}
Let $v\colon H^\times\to\Gamma$ be a convex valuation on the ordered field~$H$, 
extended uniquely to a valuation~$v\colon K^\times\to\Gamma$ on the field $K=H[\imag]$, so~$a\asymp\abs{a}$ for~$a\in K$. 
 (See the remarks before Corollary~\ref{cor:10.5.2 variant}.) Let  $v_{\g}\colon K[G]^{\neq}\to\Gamma$ 
be the gaussian extension of $v$, given by \eqref{eq:gaussian ext}.   

\begin{lemma}\label{absval}
$\dabs{f}_1\preceq 1\Leftrightarrow f\preceq_{\g} 1$, and  $\dabs{f}_1\prec 1\Leftrightarrow f\prec_{\g} 1$.
\end{lemma}
\begin{proof} Using that the valuation ring of $H$ is convex we have
$$\dabs{f}_1=\sum_\gamma\, \abs{f_\gamma}\preceq 1\ \Longleftrightarrow\  \text{$\abs{f_\gamma}\preceq 1$ for
all $\gamma$} \ \Longleftrightarrow\ \text{$f_\gamma\preceq 1$ for
all $\gamma$} \ \Longleftrightarrow\ f\preceq_{\g} 1.$$
 Likewise one shows:
$\dabs{f}_1\prec 1\Leftrightarrow f\prec_{\g} 1$.
\end{proof}

\begin{cor}\label{cor:valuation and norm}
$\dabs{f}\asymp\dabs{f}_1\asymp_{\g} f$.
\end{cor}
\begin{proof} This is trivial for $f=0$, so assume $f\neq 0$. Take $a\in H^>$ with $a\asymp_{\g} f$, and replace~$f$ by $f/a$, to arrange $f\asymp_{\g} 1$. Then $\dabs{f}\asymp \dabs{f}_1\asymp_{\g} 1$ by Lemma~\ref{absval}. 
\end{proof}

\section{The Universal Exponential Extension}\label{sec:univ exp ext}

\noindent
As in [ADH, 5.9], given a differential ring $K$, a {\it differential $K$-algebra\/} is a
differential ring $R$ with a morphism $K \to R$ of differential rings. If $R$ is a differential ring extension of a differential ring $K$ we consider $R$ as a differential $K$-algebra via the in\-clu\-sion~${K \to R}$. \index{differential algebra}

\subsection*{Exponential extensions}
{\em In this subsection $R$ is a differential ring and $K$ is a differential subring of $R$}.
Call $a \in R$   {\bf exponential over~$K$} if $a' \in a K$. \index{element!exponential over}\index{exponential!element}
Note that if  
$a\in R$ is exponential over $K$, then~$K[a]$ is a differential subring of $R$.
If~$a\in R$ is exponential over $K$ and $\phi\in K^\times$, then $a$, as element of the differential ring extension~$R^\phi$ of~$K^\phi$, is exponential over~$K^\phi$.
Every   $c\in C_R$ is exponential over~$K$, and every
$u\in K^\times$ is exponential over $K$.  
If $a,b\in R$ are exponential over $K$, then so is $ab$,
and if $a\in R^\times$ is exponential over~$K$, then so is~$a^{-1}$. Hence the 
units of~$R$ that are exponential over~$K$ form a subgroup~$E$ of the  group~$R^\times$ of units of~$R$ with~$E\supseteq C_R^\times\cdot K^\times$;
if~${R=K[E]}$, then
we call~$R$   {\bf exponential over~$K$}.\index{extension!exponential}\index{exponential!extension}\index{differential algebra!exponential extension} An 
{\bf exponential extension of~$K$\/} is a differential ring extension of $K$ that is exponential over $K$. 
If $R=K[E]$ where $E$ is a set  of elements of $R^\times$ which are exponential over $K$, then $R$ is exponential over $K$.
If $R$ is an exponential extension of~$K$ and $\phi\in K^\times$, then $R^\phi$ is an exponential extension of~$K^\phi$. 
The following lemma is extracted from the proof of \cite[Theorem~1]{Rosenlicht75}:

\begin{lemma}[Rosenlicht]\label{lem:Rosenlicht lin indep} 
Suppose $K$ is a field and $R$ is an integral domain with
differential fraction field $F$. Let $I\neq R$ be a differential ideal of $R$, and
let~$u_1,\dots,u_n\in R^\times$ \textup{(}$n\geq 1$\textup{)} be exponential
over $K$ with $u_i\notin u_j C_F^\times K^\times$ for $i\neq j$. Then $\sum_i u_i\notin I$.
\end{lemma}
\begin{proof}
Suppose $u_1,\dots, u_n$ is a counterexample with minimal $n\geq 1$.
Then $n\geq 2$ and $\sum_i u_i'\in I$, so 
$$\sum_i u_i'-u_1^\dagger\sum_i u_i\ =\  \sum_{i>1} (u_i/u_1)^\dagger u_i\in I.$$ 
Hence
~$(u_i/u_1)^\dagger=0$ and thus $u_i/u_1\in  C_F^\times$, for all $i>1$, a contradiction.
\end{proof}

\begin{cor}\label{roscor}
Suppose $K$ is a field and~$F=K(E)$ is a differential field extension of $K$ with $C_F=C$, where $E$ is a subgroup of $F^\times$ whose elements are exponential over $K$. Then $\{y\in F^\times:\, \text{$y$ is exponential over}~$K$\}=K^\times E$.
\end{cor}
\begin{proof}
Let $y\in F^\times$ be exponential over~$K$. Take $K$-linearly independent $u_1,\dots,u_n$ in $E$
and $a_1,\dots,a_n,b_1,\dots,b_n\in K$ with $b_j\ne 0$ for some $j$, such that
$$y\ =\ \Big(\textstyle\sum_i a_i u_i\Big)\Big/\left(\textstyle\sum_j b_j u_j\right).$$
Then $\sum_j b_jyu_j-\sum_i a_iu_i=0$, and so Lemma~\ref{lem:Rosenlicht lin indep} applied with~$R=F$,~$I=\{0\}$ gives $b_jyu_j\in a_iu_iK^\times$
for some $i$, $j$ with~$a_i,b_j\neq 0$, and thus~$y\in K^\times E$.
\end{proof}

\begin{remark}
In the context of Corollary~\ref{roscor}, see~\cite[Theorem~1]{Rosenlicht75} for the structure of the group of elements
of $F^\times$ exponential over~$K$, for finitely generated $E$.
\end{remark}

\begin{lemma}\label{lem:exponential map}
Suppose~$C_R^\times$ is divisible and $E$ is a subgroup of $R^\times$ containing $C_R^\times$.
Then there is a group morphism $e\colon E^\dagger\to E$ such that $e(b)^\dagger=b$
for all $b\in E^\dagger$. 
\end{lemma}
\begin{proof}
We have a short exact sequence of commutative groups
$$1 \to C_R^\times \xrightarrow{\ \iota\ } E \xrightarrow{\ \ell\ } E^\dagger \to 0,$$
where $\iota$ is the natural inclusion and $\ell(a):=a^\dagger$ for~$a\in E$. Since $C_R^\times$ is divisible, this sequence splits, which is what we claimed. 
\end{proof}

\noindent
Let $E$, $e$, $R$ be as in the previous lemma. Then  $e$ is injective, and its image is a complement of $C_R^\times$ in $E$.
Moreover, given also a group morphism $\tilde{e}\colon E^\dagger\to E$ such that~$\tilde{e}(b)^\dagger=b$ for all $b\in E^\dagger$, 
the map $b\mapsto e(b) \tilde{e}(b)^{-1}$ is a  group morphism~$E^\dagger\to C_R^\times$.

\subsection*{The universal exponential extension} {\it In the rest of this section we assume $K$ is a differential field with algebraically closed constant field $C$ and divisible
group $K^\dagger$ of logarithmic derivatives.}\/ (These conditions are satisfied if $K$ is an algebraically closed differential field.) In this subsection we show that  up to isomorphism over $K$ there is a unique exponential  extension $R$ of $K$ satisfying~$C_R=C$ and~${(R^\times)^\dagger=K}$.  Lemma~\ref{lem:exponential map} then yields a group embedding~$e\colon K\to R^\times$ such that~$e(b)^\dagger=b$
for all~$b\in K$; this motivates the construction below.

 First we describe  a certain exponential extension of $K$:
Take a {\bf complement\/} $\Lambda$ of $K^\dagger$,\label{p:complement} that is, a $\Q$-linear subspace of $K$ such that $K=K^\dagger\oplus \Lambda$ (internal direct sum of $\Q$-linear subspaces of~$K$).  Below~$\lambda$ ranges over~$\Lambda$.\index{group of logarithmic derivatives!complement}
Let $\ex(\Lambda)$ be a multiplicatively written abelian group,
isomorphic to the additive subgroup $\Lambda$ of~$K$, 
with isomorphism $\lambda\mapsto \ex(\lambda)\colon \Lambda\to \ex(\Lambda)$. 
Put 
$$\Univ\ :=\ K\big[\!\ex(\Lambda)\big],$$
the group ring of $\ex(\Lambda)$ over $K$, an integral domain.   
As $K$-linear space, 
$$\Univ\ =\ \bigoplus_\lambda K\ex(\lambda)\qquad (\text{an internal direct sum of $K$-linear subspaces}).$$
For every $f\in \Univ$ we have a unique family $(f_\lambda)$ in $K$ such that
$$f\  =\ \sum_{\lambda} f_{\lambda}  \ex(\lambda),$$
with $f_\lambda=0$ for all but finitely many $\lambda$; we call $(f_\lambda)$ the {\bf spectral decomposition} of $f$ (with respect to $\Lambda$).\index{extension!universal exponential}\index{universal exponential extension}\index{universal exponential extension!spectral decomposition}\index{spectral decomposition!of an element}\index{element!spectral decomposition}\index{decomposition!spectral}
We turn $\Univ$ into a differential ring extension of $K$
by 
$$\ex(\lambda)'\ =\ \lambda\ex(\lambda)\qquad \text{for all $\lambda$.}$$
(Think of $\ex(\lambda)$ as $\exp(\int \lambda)$.) Thus for $f\in \Univ$ with spectral decomposition $(f_\lambda)$,
$$f'\ =\ \sum_{\lambda} \big(f_{\lambda}'+\lambda f_{\lambda}\big) \ex(\lambda),$$
so $f'$ has spectral decomposition $(f_\lambda'+\lambda f_\lambda)$.  
Note that $\Univ$ is exponential over $K$ by Lem\-ma~\ref{lem:only trivial units}:  $\Univ^\times=K^\times  \ex(\Lambda)$, so  $(\Univ^\times)^\dagger=K^\dagger+\Lambda=K$. 

\begin{exampleNumbered}\label{ex:Q}
Let $K=C(\!( t^{\Q} )\!)$ be as in Example~\ref{ex:Kdagger}, so $K^\dagger=(\Q \oplus \smallo) t$. Take a $\Q$-linear subspace $\Lambda_{\operatorname{c}}$ of~$C$ with $C=\Q\oplus \Lambda_{\operatorname{c}}$
(internal direct sum of $\Q$-linear subspaces of $C$), and let
$$K_\succ\ :=\ \big\{f\in K:\ \supp(f) \succ 1 \big\},$$ 
a $C$-linear subspace of $K$.
Then $\Lambda:= (K_\succ \oplus \Lambda_{\operatorname{c}})t$ is a complement to $K^\dagger$,  and
hence~$t^{-1}\Lambda= K_\succ \oplus \Lambda_{\operatorname{c}}$ is a complement to $(K^t)^\dagger$ in $K^t$.
Moreover, if $L:=\operatorname{P}(C)\subseteq K$ is the differential field of Puiseux series over $C$
and $L_\succ:=K_\succ\cap L$, then $L_\succ  \oplus \Lambda_{\operatorname{c}}$ is a complement to $(L^t)^\dagger$. 
\end{exampleNumbered}

\noindent
A subgroup $\Lambda_0$ of $\Lambda$ yields a differential subring  $K\big[\!\ex(\Lambda_0)\big]$ of~$\Univ$ that is exponential over $K$ as well.  These differential subrings have a useful property.
Recall from~[ADH, 4.6] that a differential ring is said to be {\it simple\/} if $\{0\}$ is its only proper differential ideal.

\begin{lemma}\label{lem:U simple} 
Let $\Lambda_0$ be a subgroup of $\Lambda$. Then the differential subring
$K\big[\!\ex(\Lambda_0)\big]$ of $\Univ$ is simple. In particular, the differential ring $\Univ$ is simple. 
\end{lemma}
\begin{proof} 
Let $I\ne R$ be a differential ideal of
$R:=K\big[\!\ex(\Lambda_0)\big]$. Let
$f_1,\dots, f_n\in K^\times$ and let $\lambda_1,\dots, \lambda_n\in \Lambda_0$ be distinct
such that $
f=\sum_{i=1}^n  f_i\ex(\lambda_i)\in I$.
If $n\geq 1$, then  Lemma~\ref{lem:Rosenlicht lin indep} yields $i\neq j$
with $\ex(\lambda_i)/\ex(\lambda_j)=cg$ for some constant $c$ in the differential  fraction field of $\Univ$ and
some $g\in K^\times$, so by taking logarithmic derivatives,  $\lambda_i-\lambda_j\in K^\dagger$ and thus $\lambda_i=\lambda_j$,
a contradiction.  Thus $f=0$.
\end{proof}

\begin{cor}\label{cor:U simple}
Any morphism
$K\big[\!\ex(\Lambda_0)\big]\to R$ of differential $K$-algebras, with~$\Lambda_0$   a subgroup of $\Lambda$ and 
$R$ a differential ring extension of $K$, is injective.
\end{cor}

\noindent
The differential ring $\Univ$ is the directed union of its differential subrings of the form~$\Univ_0=K\big[\!\ex(\Lambda_0)\big]$ where $\Lambda_0$ is a finitely generated subgroup of $\Lambda$. These $\Univ_0$ are simple by Lemma~\ref{lem:U simple} and finitely generated as a $K$-algebra, hence their differential fraction fields have constant field $C$ by [ADH, 4.6.12]. Thus the  differential fraction field of $\Univ$  has constant field $C$. 

\begin{lemma}\label{lem:U minimal}
Suppose $R$ is an exponential extension of $K$ and 
$R_0$ is a differential subring of $R$ with $C_R^\times\subseteq C_{R_0}$ and $K\subseteq (R_0^\times)^\dagger$.
Then $R_0=R$.
\end{lemma}
\begin{proof}
Let $E$ be the group of units of $R$ that are exponential over $K$; so $R=K[E]$.
Given $u\in E$ we have $u^\dagger\in K\subseteq (R_0^\times)^\dagger$, hence we have $u_0\in R_0^\times$ with $u^\dagger=u_0^\dagger$, so~$u=cu_0$  with $c\in C_R^\times\subseteq C_{R_0}$. Thus $E\subseteq R_0$ and so $R_0=R$.
\end{proof}

\begin{cor}
Every endomorphism of the differential $K$-algebra $\Univ$ is an au\-to\-mor\-phism.
\end{cor}
\begin{proof}
Injectivity holds by Corollary~\ref{cor:U simple},  and surjectivity by Lemma~\ref{lem:U minimal}.
\end{proof}

\noindent
Every  exponential  extension of $K$ with constant field $C$ embeds into $\Univ$, and hence is an integral domain. More precisely:   

\begin{lemma}\label{lem:embed into U}  
Let $R$ be an exponential extension of $K$ such that~$C_R^\times$ is divisible, and set  $\Lambda_0:=\Lambda\cap (R^\times)^\dagger$, a subgroup of $\Lambda$.  
Then there exists a mor\-phism $K\big[\!\ex(\Lambda_0)\big]\to R$ of differential $K$-algebras. Any such morphism is injective, and if~$C_R=C$, then any such
morphism is an isomorphism.
\end{lemma}
\begin{proof}
Let $E$ be as in the proof of Lemma~\ref{lem:U minimal}, and let $e_E\colon E^\dagger\to E$ be the map~$e$ from Lemma~\ref{lem:exponential map}.  Since $E^\dagger=K^\dagger+\Lambda_0$  
we have 
\begin{equation}\label{eq:embed into U} 
E\ =\ C_R^\times\,  e_E(E^\dagger)\ =\ C_R^\times \, e_E(K^\dagger)\, e_E(\Lambda_0)\ =\ C_R^\times\, K^\times \, e_E(\Lambda_0).
\end{equation}
The group morphism $\ex(\lambda_0)\mapsto e_E(\lambda_0)\colon \ex(\Lambda_0)\to E$ ($\lambda_0\in \Lambda_0$) extends uniquely to a $K$-algebra morphism
$\iota\colon K\big[\!\ex(\Lambda_0)\big]\to R=K[E]$. One verifies easily that~$\iota$ is a  differential ring morphism. The injectivity claim follows from Corollary~\ref{cor:U simple}. 
If~$C_R=C$, then $E=K^\times e_E(\Lambda_0)$ by \eqref{eq:embed into U}, whence surjectivity. 
\end{proof}

\noindent
Recall that $\Univ$ is an exponential extension of $K$ with $C_{\Univ}=C$
and $(\Univ^\times)^\dagger = K$. By Lemma~\ref{lem:embed into U}, this property characterizes $\Univ$ up to isomorphism:

\begin{cor}\label{corcharexp} If $U$ is an exponential extension of $K$ such that $C_U=C$
and~$K\subseteq (U^\times)^\dagger$, then $U$ is isomorphic to $\Univ$
as a differential $K$-algebra.
\end{cor}

\noindent
Now $\Univ$ is also an exponential extension of $K$ with $C_{\Univ}=C$
and with the property that every exponential extension $R$ of $K$ with $C_R=C$ embeds into $\Univ$ as a differential $K$-algebra. This property
determines $\Univ$ up to isomorphism as well:

\begin{cor}\label{cor:UnivK} Suppose $U$ is an exponential extension of $K$ with $C_U=C$ such that every exponential extension $R$ of $K$ with $C_R=C$ embeds into $U$ as a differential $K$-algebra. Then $U$ is isomorphic to $\Univ$
as a differential $K$-algebra.  
\end{cor}
\begin{proof} Any embedding $\Univ\to U$ of
differential $K$-algebras gives $K\subseteq (U^\times)^\dagger$.  
\end{proof}

\noindent
The results above show to what extent $\Univ$ is independent of the choice of $\Lambda$. We call~$\Univ$ the {\bf universal exponential extension of $K$}.  If we need to indicate the dependence of $\Univ$ on $K$ we denote it by $\Univ_K$.
By [ADH, 5.1.40] every~${y\in\Univ=K\{\ex(\Lambda)\}}$ satisfies a linear differential equation $A(y)=0$ where $A\in K[\der]^{\neq}$; in the next section we isolate conditions on $K$ which ensure that every $A\in K[\der]^{\neq}$ has a zero~$y\in\Univ^\times=K^\times\ex(\Lambda)$.\index{universal exponential extension}\index{differential algebra!universal exponential extension}\label{p:UK}

\medskip
\noindent
Corollary~\ref{corcharexp} gives for $\phi\in K^\times$ an isomorphism
$\Univ_{K^\phi}\cong(\Univ_K)^\phi$ of differential $K^\phi$-algebras.  Next we investigate how $\Univ_K$ behaves when passing from $K$ to a differential field extension. Therefore,
{\it in the rest of this subsection $L$ is a differential field extension of $K$ with algebraically closed constant field $C_L$, and $L^\dagger$ is divisible.}\/ The next lemma relates the universal exponential extension $\Univ_L$ of $L$ to $\Univ_K$:

\begin{lemma}\label{lem:Univ under d-field ext}
The inclusion $K\to L$ extends to an embedding $\iota\colon\Univ_K\to\Univ_L$ of differential rings. The image of any such embedding~$\iota$ is contained in $K[E]$
where~$E:=\{u\in \Univ_L^\times:u^\dagger\in K\}$, and if $C_L=C$, then $\iota(\Univ_K)=K[E]$. 
\end{lemma}
\begin{proof} The differential subring
 $R:=K[E]$ of~$\Univ_L$ is exponential over $K$ with ${(R^\times)^\dagger=K}$, hence Lemma~\ref{lem:embed into U} gives an embedding $\Univ_K\to R$ of differential $K$-algebras.
Let~$\iota\colon \Univ_K\to\Univ_L$ be any embedding of differential $K$-algebras. Then $\iota\big(\!\ex(\Lambda)\big)\subseteq E$, so~$\iota(\Univ_K)\subseteq R$; if $C_L=C$, then $\iota(\Univ_K) = R$
by Lemma~\ref{lem:U minimal}.
\end{proof}

\begin{cor}\label{cor:Univ under d-field ext}
If $L^\dagger\cap K=K^\dagger$ and $\iota\colon\Univ_K\to\Univ_L$ is an embedding of differential $K$-algebras,
then $L^\times\cap \iota(\Univ_K^\times) = K^\times$. 
\end{cor}
\begin{proof} Assume $L^\dagger\cap K=K^\dagger$ and identify $\Univ_K$ with a differential $K$-subalgebra of~$\Univ_L$ via an embedding $\Univ_K\to\Univ_L$ of differential $K$-algebras.
Let $a\in L^\times\cap \Univ_K^\times$; then~$a^\dagger\in L^\dagger\cap K=K^\dagger$, so $a = bc$ where $c\in C_L^\times$, $b\in K^\times$.
Now $c=a/b\in C_L^\times\cap\Univ_K^\times=C^\times$, since $\Univ_K$ has ring of constants~$C$. So $a\in K^\times$ as required. 
\end{proof}

\noindent
Suppose $L^\dagger\cap K=K^\dagger$. Then the subspace~$L^\dagger$ of the $\Q$-linear space $L$  has a complement
$\Lambda_L\supseteq \Lambda$.  We fix such $\Lambda_L$ and
extend~$\ex\colon\Lambda\to \ex(\Lambda)$ to a group isomorphism~$\Lambda_L\to\ex(\Lambda_L)$, also denoted by $\ex$, with~$\ex(\Lambda_L)$ a multiplicatively written commutative group  extending $\ex(\Lambda)$.
   Let $\Univ_L:=L\big[\!\ex(\Lambda_L)\big]$ be the corresponding universal exponential extension of~$L$.  
Then the natural inclusion~$\Univ_K\to\Univ_L$
is an embedding of differential $K$-algebras.

\subsection*{Automorphisms of $\Univ$} 
These are easy to describe: the beginning of Section~\ref{sec:group rings} gives a group embedding
$$\chi\mapsto \sigma_{\chi}\colon \Hom(\Lambda,K^\times)\to \Aut\!\big(K[\ex(\Lambda)]|K\big)$$ into the group of $K$-algebra automorphisms of $K\big[\!\ex(\Lambda)\big]$, given by
$$ \sigma_{\chi}(f)\ := f_{\chi}\ =\ \sum_\lambda f_\lambda \chi(\lambda)\ex(\lambda)\qquad (\chi\in \Hom(\Lambda,K^\times),\ f\in K[\ex(\Lambda)]).$$
It is easy to check that if $\chi\in \Hom(\Lambda,C^\times)\subseteq \Hom(\Lambda,K^\times)$, then $\sigma_\chi\in \Aut_\der(\Univ|K)$, that is, $\sigma_{\chi}$ is a differential $K$-algebra automorphism of $\Univ$. Moreover:

\begin{lemma}\label{autolem} The map $\chi\mapsto \sigma_\chi \colon \Hom(\Lambda,C^\times)\to\Aut_\der(\Univ|K)$ is a group isomorphism. Its inverse assigns to
any $\sigma\in \Aut_\der(\Univ|K)$ the function
$\chi\colon \Lambda \to C^\times$ given by~$\chi(\lambda):=\sigma\big(\!\ex(\lambda)\big)\ex(- \lambda)$. In particular, $\Aut_\der(\Univ|K)$ is commutative.
\end{lemma}
\begin{proof} Let $\sigma\in \Aut_{\der}(\Univ|K)$ and let $\chi\colon \Lambda\to \Univ^\times$ be given by $\chi(\lambda):=\sigma\big(\!\ex(\lambda)\big)\ex(- \lambda)$. 
Then $\chi(\lambda)^\dagger=0$ for all $\lambda$. It follows easily that $\chi\in \Hom(\Lambda,C^\times)$ and $\sigma_{\chi}=\sigma$.
\end{proof} 

\noindent
The proof of the next result uses that the additive group $\Q$ embeds into~$C^\times$.

\begin{cor}\label{cor:fixed field}
If $f\in\Univ$ and $\sigma(f)=f$ for all $\sigma\in\operatorname{Aut}_\der(\Univ|K)$, then $f\in K$.
\end{cor}
\begin{proof} Suppose $f\in U$ and $\sigma(f)=f$ for all $\sigma\in\operatorname{Aut}_\der(\Univ|K)$.
For $\chi\in\operatorname{Hom}(\Lambda,C^\times)$ we have $f_\chi=f$, that is, $f_\lambda\chi(\lambda)=f_\lambda$
for all $\lambda$, so  $\chi(\lambda)=1$ whenever $f_\lambda\neq 0$. Now use that for $\lambda\ne 0$ there exists
$\chi\in\operatorname{Hom}(\Lambda,C^\times)$ such that $\chi(\lambda)\ne 1$, so $f_{\lambda}=0$. 
\end{proof}

\begin{cor}\label{cor:extend autom}
Every automorphism of the differential field $K$ extends to an automorphism of the differential ring~$\Univ$.
\end{cor}
\begin{proof}
Lemma~\ref{lem:exponential map} yields a group morphism $\mu\colon K\to \Univ^\times$ such that $\mu(a)^\dagger=a$ for all $a\in K$.
Let $\sigma\in\operatorname{Aut}_\der(K)$. Then $\sigma$ extends to an endomorphism, denoted also by $\sigma$,
 of the ring $\Univ$, such that $\sigma\big(\!\ex(\lambda)\big)= \mu\big(\sigma(\lambda)\big)$ for all $\lambda$. 
Then  
$$\sigma\big(\!\ex(\lambda)'\big)\ =\ \sigma\big(\lambda\ex(\lambda)\big)\ =\ \sigma(\lambda)\mu\big(\sigma(\lambda)\big)\ =\ \mu\big(\sigma(\lambda)\big)'\ =\ \sigma\big(\!\ex(\lambda)\big)',$$
hence $\sigma$ is an endomorphism of the differential ring $\Univ$.  
By Lemma~\ref{lem:U simple}, $\sigma$ is injective, and by Lemma~\ref{lem:U minimal}, $\sigma$ is surjective.
\end{proof}

\subsection*{The real case} 
{\it In this subsection $K=H[\imag]$ where~$H$ is a real closed differential subfield of $K$ and $\imag^2=-1$.}\/
Set $S_C:=\big\{c\in C:\, |c|=1\big\}$, a subgroup of $C^\times$. Then by Lemmas~\ref{lem:commuting with f*} and~\ref{autolem}:

\begin{cor}\label{autolem, commuting with f*}
For $\sigma\in\Aut_\der(\Univ|K)$ we have the equivalence
$$\sigma(f^*)=\sigma(f)^*\text{ for all $f\in\Univ$}\quad\Longleftrightarrow\quad \sigma=\sigma_\chi\text{ for some $\chi\in\Hom(\Lambda,S_C)$.}$$
\end{cor}

\noindent
Corollaries~\ref{autolem, commuting with f*} and~\ref{cor:commuting with f*} together give:
 
\begin{cor}\label{cor:inner prod invariant}
Let  $\sigma\in\Aut_\der(\Univ|K)$ satisfy $\sigma(f^*)=\sigma(f)^*$   for all $f\in\Univ$. Then~$\big\langle\sigma(f),\sigma(g)\big\rangle=\langle f,g\rangle$ for all $f,g\in\Univ$, hence $\dabs{\sigma(f)}=\dabs{f}$ for all $f\in\Univ$.
\end{cor}

\noindent
Next we consider the subgroup 
$$S\ :=\ \{a+b\imag:\ a,b\in H,\ a^2+b^2=1\}$$ of $K^\times$, which is divisible,
hence so is the subgroup $S^\dagger$ of $K^\dagger$.
Lemma~\ref{lem:logder} yields~$K^\dagger = H^\dagger \oplus S^\dagger$
(internal direct sum of $\Q$-linear subspaces of $K$) and $S^\dagger\subseteq  H \imag$.
Thus we can (and do) take the complement $\Lambda$ of $K^\dagger$ in $K$ so that
$\Lambda= \Lambda_{\operatorname{r}}+\Lambda_{\operatorname{i}}\imag$ where $\Lambda_{\operatorname{r}}, \Lambda_{\operatorname{i}}$ are subspaces of the $\Q$-linear space $H$ with $\Lambda_{\operatorname{r}}$ a complement of $H^\dagger$ in $H$ and 
$\Lambda_{\operatorname{i}}\imag$
 a complement  of $S^\dagger$ in $H\imag$.
The automorphism $a+b\imag\mapsto \bar{a+b\imag}:= a-b\imag$~(${a,b\in H}$) of the differential field $K$ now satisfies in $\Univ=K[\ex(\Lambda)]$ the identity 
$$\ex(\bar{\lambda + \mu})\ =\ \ex(\bar{\lambda})\ex(\bar{\mu})\qquad(\lambda,\mu\in \Lambda),$$ so it extends to
an automorphism $f\mapsto\overline{f}$ of the ring $\Univ$ as follows: 
for $f\in \Univ$ with spectral decomposition~$(f_\lambda)$, set
$$\overline{f}\ :=\ \sum_\lambda \overline{f_\lambda}\ex(\overline{\lambda})\ =\ \sum_{\lambda}\overline{f_{\overline{\lambda}}}\ex(\lambda),  $$
so $\overline{\ex(\lambda)}=\ex(\overline{\lambda})$, and $\overline{f}$ has spectral decomposition  $(\overline{f_{\overline{\lambda}}})$.
We have $\overline{\overline{f}}=f$ for $f\in\Univ$,
and $f\mapsto\overline{f}$ lies in $\operatorname{Aut}_\der(\Univ|H)$. 
If~$H^\dagger=H$, then $\Lambda_r=\{0\}$ and hence~$\overline{f}=f^*$ for~$f\in\Univ$, where $f^*$ is as defined in Section~\ref{sec:group rings}.
For $f\in\Univ$ we set
$$\Re f\ :=\  \textstyle\frac{1}{2}(f+\overline{f}),\qquad  
\Im f\ :=\ \textstyle\frac{1}{2\imag}(f-\overline{f}).$$
(For $f\in K$ these agree with the usual real and imaginary parts of $f$ as an element of $H[\imag]$.)
Consider the differential $H$-subalgebra 
$$\Univ_{\operatorname{r}}\ :=\ \big\{f\in\Univ: \overline{f}=f\big\}$$ 
of~$\Univ$.
For $f\in\Univ$ with spectral decomposition $(f_\lambda)$ we have~$f\in \Univ_{\operatorname{r}}$ iff $f_{\overline{\lambda}}=\overline{f_\lambda}$ for all~$\lambda$;
in particular $\Univ_{\operatorname{r}}\cap K=H$.
For $f\in\Univ$ we have $f=(\Re f)+(\Im f)\imag$ with~$\Re f,\Im f\in \Univ_{\operatorname{r}}$, hence 
$$\Univ\ =\ \Univ_{\operatorname{r}}\oplus \Univ_{\operatorname{r}}\imag\quad\text{ (internal direct sum of $H$-linear subspaces).}$$
Let $D$ be a subfield of $H$ (not necessarily the constant field of $H$), so $D[\imag]$ is a  subfield  of $K$. Let 
$V$ be a $D[\imag]$-linear subspace of $\Univ$; then $V_{\operatorname{r}}:=V\cap \Univ_{\operatorname{r}}$ is a $D$-linear subspace of $V$. 
If~$\overline{V}=V$ (that is,~$V$~is closed under $f\mapsto \overline{f}$), then~$\Re f,\Im f\in V_{\operatorname{r}}$ for all $f\in V$, hence
 $V=V_{\operatorname{r}}\oplus V_{\operatorname{r}}\imag$ (internal direct sum of $D$-linear subspaces of $V$), so 
any basis of the $D$-linear space $V_{\operatorname{r}}$ is a basis of the $D[\imag]$-linear space $V$. 

\medskip
\noindent
Suppose now that $V=\bigoplus_\lambda V_\lambda$
(internal direct sum of subspaces of $V$) where $V_{\lambda}$ is for each $\lambda$
a $D[\imag]$-linear subspace of $K\ex(\lambda)$. Then~$\overline{V}=V$  iff
$V_{\overline{\lambda}}=\overline{V_\lambda}$ for all $\lambda$. Moreover:

\begin{lemma}\label{lem:real basis}
Assume $H=H^\dagger$, $V_0=\{0\}$,
and $\overline{V}=V$. Let $\mathcal V\subseteq \Univ^\times$ be a basis of the subspace~$\sum_{\Im\lambda>0} V_\lambda$ of $V$.
Then the maps $v\mapsto \Re v,\ v\mapsto \Im v\colon \mathcal{V} \to V_{\operatorname{r}}$ are injective, $\Re \mathcal{V}$ and $\Im \mathcal{V}$ are disjoint, and $\Re \mathcal{V}\cup \Im \mathcal{V}$ is  
a basis of $V_{\operatorname{r}}$.
\end{lemma}
\begin{proof} Note that $\Lambda=\Lambda_{\operatorname{i}}\imag$. Let $\mu$ range over $\Lambda_{\operatorname{i}}^{>}$ and set $\cal{V}_{\mu}=\cal{V}\cap K^\times \ex(\mu\imag)$, a basis of the $D[\imag]$-linear space $V_{\mu\imag}$. Then $\cal{V}=\bigcup_{\mu}\cal{V}_{\mu}$, a disjoint union. For~$v\in \cal{V}_{\mu}$ we have~$v=a\ex(\mu\imag)$ with $a=a_v\in K^\times$, so 
$$\Re v\ =\ \textstyle\frac{a}{2}\ex(\mu\imag) + \textstyle\frac{\bar{a}}{2}\ex(-\mu\imag), \qquad \Im v\ =\ \textstyle\frac{a}{2i}\ex(\mu\imag) - \textstyle\frac{\bar{a}}{2\imag}\ex(-\mu\imag),$$
from which it is clear that the two maps $\cal{V} \to V_{\operatorname{r}}$ in the statement of the lemma are injective. It is also easy to check that $\Re \mathcal{V}$ and $\Im \mathcal{V}$ are disjoint.  

As $\mathcal{V}$ is a basis of the $D[\imag]$-linear space $\sum_{\mu} V_{\mu \imag}=\sum_{\Im\lambda>0} V_\lambda$, its set of conjugates~$\overline{\mathcal V}$ is a basis of the $D[\imag]$-linear space $\sum_{\mu} \overline{V_{\mu \imag}}=\sum_{\mu}V_{-\mu\imag}=\sum_{\Im\lambda<0} V_\lambda$, and so~$\mathcal{V}\cup \overline{\mathcal{V}}$ (a disjoint union) is a basis of $V$. Thus~$\Re \mathcal{V}\cup \Im \mathcal{V}$ is a basis of $V$ as well. As~$\Re \mathcal{V}\cup \Im \mathcal{V}$ is contained in $V_{\operatorname{r}}$, it is a basis of the $D$-linear space $V_{\operatorname{r}}$. 
\end{proof}

\noindent
If $H=H^\dagger$, then $V:=\sum_{\lambda\neq 0}K\ex(\lambda)$ gives $\overline{V}=V$, so Lemma~\ref{lem:real basis} gives then for~$D:= H$ the basis of the $H$-linear space $V_{\operatorname{r}}$ consisting of the elements 
$$\Re\!\big(\!\ex(\lambda)\big)\ =\ \textstyle\frac{1}{2}\big(\!\ex(\lambda)+\ex(\overline{\lambda})\big),\qquad 
\Im\!\big(\!\ex(\lambda)\big)\ =\ \textstyle\frac{1}{2\imag}\big(\!\ex(\lambda)-\ex(\overline{\lambda})\big)\qquad (\Im\lambda>0).$$

\begin{cor}\label{cor:U_r} Suppose $H=H^\dagger$. Set $\operatorname{c}(\lambda):=\Re\!\big(\!\ex(\lambda)\big)$ and $\operatorname{s}(\lambda):=\Im\!\big(\!\ex(\lambda)\big)$,
for $\Im \lambda>0$. Then for $V:=\sum_{\lambda\neq 0}K\ex(\lambda)$ we have $\Univ_{\operatorname{r}}=H+V_{\operatorname{r}}$, so
$$\Univ_{\operatorname{r}}\ =\ H \oplus \bigoplus_{\Im\lambda>0} \big( H\operatorname{c}(\lambda)\oplus H\operatorname{s}(\lambda) \big) 
\quad \text{\textup{(}internal direct sum of $H$-linear subspaces\textup{)},}$$
and thus $\Univ_{\operatorname{r}} = H\big[ \operatorname{c}(\Lambda_{\operatorname{i}}^>\imag)\cup \operatorname{s}(\Lambda_{\operatorname{i}}^>\imag)\big]$.
\end{cor}

\section{The Spectrum of a Differential Operator}\label{sec:splitting}

\noindent
{\it In this section $K$ is a differential field, $a$, $b$ range over $K$, and $A$, $B$  over~$K[\der]$.}\/
This and the next two sections are mainly differential-algebraic in nature, and deal with splittings of linear differential operators.
In the present section we introduce the concept of {\it eigenvalue}\/ of $A$ and  the {\it spectrum}\/ of $A$ (the collection of its eigenvalues).
In Section~\ref{sec:eigenvalues and splitting} we show how the eigenvalues of $A$ relate to the behavior of $A$ over the universal exponential extension of $K$.

\subsection*{Twisting}
Let $L$ be a differential field extension  of $K$ with $L^\dagger\supseteq K$.
Let $u\in L^\times$ be such that $u^\dagger=a\in K$. Then the twist \index{twist}\index{linear differential operator!twist}
$A_{\ltimes u} = u^{-1} A u$
of $A$ by $u$ has the same order as $A$ and coefficients in $K$ [ADH,~5.8.8], and only depends on $a$, not on $u$ or~$L$;
in fact, $\Ric(A_{\ltimes u})=\Ric(A)_{+a}$ [ADH, 5.8.5].
Hence for  each $a$ we may define
$$A_a\ :=\ A_{\ltimes u}\ =\ u^{-1}Au\in K[\der]$$
where $u\in L^\times$ is arbitrary with $u^\dagger=a$. 
The map
$A\mapsto A_{\ltimes u}$ is an automorphism of the ring $K[\der]$ that is
the identity on $K$ (with inverse $B\mapsto B_{\ltimes u^{-1}}$);
so  $A\mapsto A_a$ is an automorphism of the ring $K[\der]$ that is the identity on $K$  (with inverse $B\mapsto B_{-a}$). 
Note that $\der_a=\der+a$, and that
$$(a,A)\mapsto A_a\ :\ K\times K[\der] \to K[\der]$$
is an action of the additive group of $K$ on the set $K[\der]$, in particular, $A_a=A$ for~$a=0$.  
For  $b\neq 0$ we have $(A_a)_{\ltimes b} = A_{a+b^\dagger}$.

\subsection*{Eigenvalues} 
{\it In the rest of this section $A\neq 0$ and $r:=\order(A)$.}\/
We call 
$$\mult_{a}(A)\ :=\ \dim_C \ker_K A_a\in \{0,\dots,r\}$$
the {\bf multiplicity} of $A$ at $a$. 
If $B\neq 0$, then  
$\mult_a(B) \leq \mult_a(AB)$, as well as
 \index{linear differential operator!multiplicity}\index{multiplicity!linear differential operator}\label{p:multa}
\begin{equation}\label{eq:mult(AB)}
\mult_a(AB)\ \leq\ \mult_a(A)+\mult_a(B),
\end{equation}
with equality if and only if $B_a(K)\supseteq  \ker_K A_a$; see [ADH, remarks before 5.1.12]. 
For $u\in K^\times$ we have an isomorphism 
$$y\mapsto yu\ \colon\ \ker_K A_{\ltimes u} \to \ker_K A$$ 
of $C$-linear spaces, hence 
$$\mult_{a}(A)\ =\  \mult_{b}(A)\qquad\text{whenever $a-b\in K^\dagger$.}$$
Thus we may  define the
{\bf multiplicity} of $A$ at the element $[a]:=a+K^\dagger$ of $K/K^\dagger$ 
as $\mult_{[a]}(A):=\mult_{a}(A)$. \label{p:multalpha}

\medskip
\noindent
{\it In the rest of this section $\alpha$ ranges over $K/K^\dagger$.}\/
We say that $\alpha$ is an {\bf eigenvalue} of~$A$ if~$\mult_{\alpha}(A)\geq 1$.  Thus for $B\ne 0$: if
$\alpha$ is an eigenvalue of~$B$ of multiplicity~$\mu$, then
$\alpha$ is an eigenvalue of $AB$ of multiplicity~$\ge \mu$;
if $\alpha$ is an eigenvalue of~$AB$, then it is an eigenvalue
of $A$ or of $B$; and if
$B_a(K)\supseteq\ker_K(A_a)$, then 
 $\alpha=[a]$ is an eigenvalue of $AB$ if and only if it is an eigenvalue
of~$A$ or of~$B$. 

\begin{exampleNumbered}\label{ex:ev order 1}
Suppose $A=\der-a$. 
Then for each element $u\neq 0$ in a differential field extension  of $K$ with $b:=u^\dagger\in K$ we have $A_b=A_{\ltimes u}=\der-(a-b)$,
so~$\mult_b(A)\geq 1$ iff $a-b\in K^\dagger$. Hence the only eigenvalue of $A$ is $[a]$.
\end{exampleNumbered}

\noindent
The {\bf spectrum} of $A$ is the set~$\Sigma(A)=\Sigma_K(A)$ of its eigenvalues. Thus $\Sigma(A)=\emptyset$ if $r=0$, and
for~$b\neq 0$ we have~$\mult_{a}(A)=\mult_{a}(bA)=\mult_{a}(A_{\ltimes b})$,
so $A$, $bA$, and~$Ab=bA_{\ltimes b}$  all have the same spectrum.\index{linear differential operator!eigenvalue}\index{eigenvalue!linear differential operator}\index{linear differential operator!spectrum}\index{spectrum}\label{p:SigmaA}
By [ADH, 5.1.21] we have
\begin{equation}\label{eq:spec A}
\Sigma(A) = \big\{ \alpha : A \in K[\der](\der-a) \text{ for some $a$ with $[a]=\alpha$} \big\}.
\end{equation}
Hence for irreducible $A$:   $\ \Sigma(A)\neq \emptyset\ \Leftrightarrow\ r=1$. 
From \eqref{eq:mult(AB)} we obtain:

{\samepage
\begin{lemma}\label{lem:spectrum fact}
Suppose~$B\neq 0$ and set $s:=\order B$. Then $$\mult_\alpha(B)\ \leq\ \mult_\alpha(AB)\ \leq\ 
\mult_\alpha(A)+\mult_\alpha(B),$$
where the second inequality is an equality if  $K$ is $s$-linearly surjective. Hence
$$\Sigma(B)\ \subseteq\ \Sigma(AB)\ \subseteq\ \Sigma(A)\cup\Sigma(B).$$
If $K$ is $s$-linearly surjective, 
then $\Sigma(AB) = \Sigma(A)\cup\Sigma(B)$.
\end{lemma}}

\begin{example}
For $n\geq 1$ we have  $\Sigma\big( (\der-a)^n \big)=\big\{[a]\big\}$. (By induction on $n$, using  Example~\ref{ex:ev order 1} and Lemma~\ref{lem:spectrum fact}.)
\end{example}

\noindent
It follows from Lemma~\ref{lem:spectrum fact} that $A$ has at most $r$ eigenvalues.  More precisely:
 
\begin{lemma}\label{lem:size of Sigma(A)}
We have $\sum_\alpha \mult_\alpha(A)\leq r$.
If $\sum_\alpha \mult_\alpha(A) = r$, then $A$ splits over~$K$; the converse holds if
$r=1$ or $K$  is $1$-linearly surjective. 
\end{lemma}
\begin{proof}
By induction on $r$. The case $r=0$ is obvious,
so suppose~$r>0$. We may also assume $\Sigma(A)\neq\emptyset$: otherwise  $\sum_\alpha \mult_\alpha(A)=0$ 
and $A$ does not split over~$K$. Now~\eqref{eq:spec A} gives~$a$,~$B$ with 
$A=B(\der-a)$. By Example~\ref{ex:ev order 1}  we have~$\Sigma(\der-a)=\big\{[a]\big\}$ and $\mult_a(\der-a)=1$.
By the inductive hypothesis applied to $B$ and the second inequality in Lemma~\ref{lem:spectrum fact} we thus get~$\sum_\alpha \mult_\alpha(A)\leq r$.

Suppose that $\sum_\alpha \mult_\alpha(A) = r$.
Then    $\sum_\alpha \mult_\alpha(B) = r-1$ by  
Lemma~\ref{lem:spectrum fact} and
the inductive hypothesis applied to $B$.
Therefore $B$ splits over $K$, again by the inductive hypothesis, and so does $A$.
Finally, if $K$ is $1$-linearly surjective and~$A$ splits over $K$, then we arrange that $B$ splits over $K$,
so $\sum_\alpha \mult_\alpha(B) = r-1$ by the inductive hypothesis,
hence $\sum_\alpha \mult_\alpha(A) = r$ by Lemma~\ref{lem:spectrum fact}.
\end{proof}

\noindent
Section~\ref{sec:eigenvalues and splitting} gives a more explicit proof of Lemma~\ref{lem:size of Sigma(A)},
under additional hypotheses on $K$.
Next, let $L$ be a differential field extension of $K$. Then 
$\mult_a(A)$ does not strictly decrease in passing from $K$ to $L$~[ADH, 4.1.13]. Hence the group morphism
$$a+K^\dagger\mapsto a+L^\dagger\colon K/K^\dagger\to L/L^\dagger$$ restricts to a map
$\Sigma_K(A)\to\Sigma_L(A)$; in particular, if $\Sigma_K(A)\neq\emptyset$, then $\Sigma_L(A)\neq\emptyset$.
If~$L^\dagger \cap K=K^\dagger$, then $\abs{\Sigma_K(A)}\leq\abs{\Sigma_L(A)}$, and
$\sum_\alpha \mult_\alpha(A)$ also 
does not strictly decrease if $K$ is replaced by $L$.

\begin{lemma}\label{lem:split evs}
Let $a_1,\dots, a_r\in K$ and 
$$A\ =\ (\der-a_r)\cdots(\der-a_1),\quad
\sum_{\alpha} \mult_{\alpha}(A)\ =\ r.$$
Then the spectrum of~$A$ is $\big\{[a_1],\dots,[a_r]\big\}$, and for all $\alpha$,
$$\mult_{\alpha}(A)\ =\ 
\big| \big\{ i\in\{1,\dots,r\}:\ \alpha=[a_i] \big\}  \big|.$$
\end{lemma}

\begin{proof}
Let  $i$ range over $\{1,\dots,r\}$.
By Lemma~\ref{lem:spectrum fact} and Example~\ref{ex:ev order 1},
$$\mult_\alpha(A)\ \leq\ \sum_i \mult_\alpha(\der-a_i)\ =\ 
\big| \big\{ i : \alpha=[a_i] \big\}  \big|$$
and hence
$$r\ =\ \sum_{\alpha} \mult_{\alpha}(A)\ \leq\ \sum_\alpha \big| \big\{ i : \alpha=[a_i] \big\}  \big|\ =\  r.$$
Thus for each $\alpha$ we have $\mult_\alpha(A) = 
\big| \big\{ i : \alpha=[a_i] \big\}  \big|$ as required.
\end{proof}

\noindent
Recall from [ADH, 5.1.8] that~$D^*\in K[\der]$ denotes the {\it adjoint}\/ of~$D\in K[\der]$, and that the map $D\mapsto D^*$ is an involution of the ring $K[\der]$ with  $a^*=a$ for all~$a$ and~$\der^*=-\der$.\index{linear differential operator!adjoint}\index{adjoint}\label{p:A*}
If $A$ splits over $K$, then so does $A^*$. Furthermore, $(A_a)^*=(A^*)_{-a}$ for all $a$. By
Lemmas~\ref{lem:size of Sigma(A)} and~\ref{lem:split evs}:

\begin{cor}\label{cor:spectrum, self-adjoint, 1}
Suppose $K$ is $1$-linearly surjective and $\sum_\alpha \mult_\alpha(A)=r$.  Then $\mult_\alpha(A)=\mult_{-\alpha}(A^*)$
for all $\alpha$. In particular,
the map~$\alpha\mapsto-\alpha$ restricts to a bijection~$\Sigma(A)\to\Sigma(A^*)$.
\end{cor}

\noindent
Let $\phi\in K^\times$. Then $(A^\phi)_a=(A_{\phi a})^\phi$ and hence
$$\mult_a(A^\phi)\  =\  \mult_{\phi a}(A),$$
so the group isomorphism 
\begin{equation}\label{eq:Kphidagger}
[a]\mapsto [\phi a]\ \colon\  K^\phi/\phi^{-1}K^\dagger\to K/K^\dagger
\end{equation}
maps $\Sigma(A^\phi)$ onto $\Sigma(A)$.

\medskip
\noindent
Note that $K[\der]/K\der]A$ as a $K$-linear space has dimension $r=\order A$. 
Recall from~[ADH, 5.1] that  $A$ and $B\neq 0$  are said to {\it have the same type}\/ if the (left) $K[\der]$-modules
$K[\der]/K[\der]A$ and $K[\der]/K[\der]B$ are isomorphic (and so $\order B=r$).  By [ADH, 5.1.19]:\index{linear differential operator!type}\index{type}

\begin{lemma}\label{lem:same type}
The operators $A$ and $B\neq 0$ have the same type iff $\order B=r$ and there is $R\in K[\der]$ of order~$<r$ with $1\in K[\der]R+K[\der]A$ and $BR\in K[\der]A$.
\end{lemma}

\noindent
Hence if $A$, $B$ have the same type, then they also have the same type as elements of~$L[\der]$, for any differential field
extension $L$ of $K$.   Since $B\mapsto B_a$ is an automorphism of the ring~$K[\der]$, Lemma~\ref{lem:same type} and~[ADH, 5.1.20] yield:

\begin{lemma}\label{lem:same type and eigenvalues}
If $A$ and $B\neq 0$ have the same type, then so do $A_a$,~$B_a$, for all $a$, and thus
$A$, $B$ have the same eigenvalues, with same multiplicity. \end{lemma}

\noindent
By this lemma the spectrum of $A$ depends only on the type of $A$, that is, on the isomorphism type of the $K[\der]$-module
$K[\der]/K[\der]A$, suggesting one might try to associate a spectrum to each differential module over $K$. 
(Recall from [ADH, 5.5] that  a differential module over $K$ is a $K[\der]$-module of finite dimension as $K$-linear space.) We do not develop this point of view further  in the present monograph, where
our focus is on linear differential operators. (There will be more on this in \cite{ADH6}.) But we remark here that  it  motivates the terminology of ``eigenvalues'' originating in the case of the
differential field of Puiseux series over~$\C$ treated in~\cite{vdPS}. 
%This point of view will be further developed in the projected second volume of [ADH]. 

\section{Eigenvalues and Splittings}\label{sec:eigenvalues and splitting}

\noindent
{\it In this section $K$ is a differential field such that $C$ is algebraically closed and $K^\dagger$ is divisible.}\/
We let $A$, $B$ range over $K[\der]$, and we assume $A\neq 0$ and set $r:=\order A$.

\subsection*{Spectral decomposition of  differential operators}
Fix a complement $\Lambda$ of the subspace $K^\dagger$ of the $\Q$-linear space $K$, let $\Univ:= K\big[\!\ex(\Lambda)\big]$ be the
universal exponential extension of $K$, let $\Omega$ be the differential fraction field of the differential $K$-algebra~$\Univ$, and let $\lambda$ range over $\Lambda$. Then
$$A_\lambda\ =\ A_{\ltimes\!\ex(\lambda)}\ =\ \ex({-\lambda})A\ex(\lambda)\in K[\der].$$ 
Moreover, for every $a\in K$ there is a unique $\lambda$ with $a-\lambda\in K^\dagger$, so 
$\mult_{[a]}(A)=\mult_{\lambda}(A)$. Call $\lambda$ an {\bf eigenvalue} of $A$ with respect 
to our complement~$\Lambda$ of $K^\dagger$ in~$K$ if~$[\lambda]$ is an eigenvalue of $A$; thus the group isomorphism $\lambda\mapsto [\lambda]\colon \Lambda \to K/K^\dagger$ maps the set of eigenvalues of~$A$ with respect to~$\Lambda$ onto the spectrum of $A$.  
For~$f\in\Univ$ with spectral decomposition~$(f_\lambda)$ we have \index{eigenvalue}
$$A(f)\ =\ \sum_\lambda A_\lambda(f_\lambda)\,\ex(\lambda),$$
so $A(\Univ^\times)\subseteq \Univ^\times\cup \{0\}$.
We call the family $(A_\lambda)$ the {\bf spectral decomposition} of $A$ (with respect to $\Lambda$).
Given a $C$-linear subspace $V$ of $\Univ$, we set $V_\lambda := V\cap K\ex(\lambda)$, a $C$-linear subspace  of $V$; the sum $\sum_\lambda V_\lambda$ is direct. For $V:=\Univ$ we have $\Univ_\lambda=K\ex(\lambda)$, and $\Univ=\bigoplus_\lambda \Univ_\lambda$
with $A(\Univ_\lambda)\subseteq\Univ_\lambda$ for all $\lambda$.
Taking $V:=\ker_{\Univ} A$, we obtain~$V_\lambda=(\ker_{K} A_\lambda)\ex(\lambda)$ 
and hence $\dim_C V_\lambda = \mult_\lambda(A)$,
and $V=\bigoplus_\lambda V_\lambda$. Thus  \index{spectral decomposition!of a differential operator}\index{decomposition!spectral} \index{linear differential operator!spectral decomposition}
\begin{equation}\label{eq:bd on sum mults}
\abs{\Sigma(A)}\ \leq\ \sum_\lambda \mult_\lambda(A)\ =\ \dim_C \ker_{\Univ} A\ \leq\ r.
\end{equation}
Moreover:

\begin{lemma}\label{newlembasis} The $C$-linear space $\ker_{\Univ} A$ has a basis contained in $\Univ^\times=K^\times\ex(\Lambda)$.
\end{lemma}

\begin{example} 
We have a $C$-algebra isomorphism $P(Y)\mapsto P(\der)\colon C[Y]\to C[\der]$.
Suppose~$A\in C[\der]\subseteq K[\der]$,  let $P(Y)\in C[Y]$, $P(\der)=A$, and let
$c_1,\dots,c_n\in C$ be the distinct zeros of $P$, of respective multiplicities $m_1,\dots,m_n\in\N^{\geq 1}$ (so $r=\deg P=m_1+\cdots+m_n$).
Suppose also $C\subseteq \Lambda$, and $x\in K$ satisfies $x'=1$. 
(This holds in Example~\ref{ex:Q}.)
Then the $x^i\ex(c_j)\in \Univ$ ($1\leq j\leq n$, $0\leq i<m_j$)
form a basis of the $C$-linear space $\ker_{\Univ} A$ by [ADH, 5.1.18].
So the eigenvalues of~$A$ with respect to $\Lambda$ are $c_1,\dots,c_n$, with respective multiplicities $m_1,\dots,m_n$.
\end{example}

\begin{cor}\label{cor:sum of evs}
Suppose $\dim_C \ker_{\Univ} A = r\ge 1$ and $A=\der^r+a_{r-1}\der^{r-1}+\cdots+a_0$ where~$a_0,\dots,a_{r-1}\in K$.
Then
$$\sum_{\lambda} \mult_{\lambda}(A)\lambda\  \equiv\  -a_{r-1} \mod K^\dagger.$$
In particular,  $\sum_{\lambda} \mult_{\lambda}(A)\lambda=0$ iff $a_{r-1}\in K^\dagger$.
\end{cor}
\begin{proof}
Take a basis $y_1,\dots,y_r$ of $\ker_{\Univ} A$ with $y_j=f_j\ex(\lambda_j)$, $f_j\in K^\times$, $\lambda_j\in\Lambda$. 
The Wronskian matrix $\operatorname{Wr}(y_1,\dots,y_r)$ of $(y_1,\dots,y_r)$ [ADH, p.~206] equals
$$\operatorname{Wr}(y_1,\dots,y_r)\ =\  M\begin{pmatrix} \ex(\lambda_1) & & \\ & \ddots & \\ & & \ex(\lambda_r) \end{pmatrix}\qquad\text{where $M\in\operatorname{GL}_n(K)$.}$$
Then $w:=\operatorname{wr}(y_1,\dots,y_r)=\det \operatorname{Wr}(y_1,\dots,y_r)\neq 0$ by [ADH, 4.1.13] and
 $$-a_{r-1}\ =\ w^\dagger\ =\ (\det M)^\dagger + \lambda_1+\cdots+\lambda_r$$ where we used [ADH, 4.1.17] for the first equality.
\end{proof}

\noindent
If $A$ splits over $K$, then so does $A_\lambda$.
Moreover, if $A_\lambda(K)=K$, then $A(\Univ_\lambda)=\Univ_\lambda$:
for   $f,g\in K$ with $A_\lambda(f)=g$ we have
$A\big(f\ex(\lambda)\big)=g\ex(\lambda)$.
Thus:

\begin{lemma}\label{lem:A(U)=U}
Suppose $K$ is $r$-linearly surjective, or $K$ is $1$-linearly surjective and~$A$ splits over $K$.  Then $A(\Univ_\lambda)=\Univ_\lambda$ for all $\lambda$ and hence $A(\Univ)=\Univ$.
\end{lemma}

\noindent
In the next subsection we  study the 
connection between splittings of $A$ and
bases of the $C$-linear space $\ker_{\Univ} A$ in more detail.

\subsection*{Constructing splittings and bases}
Recall that $\order A=r\in\N$. Set $\Univ=\Univ_K$, so $\Univ^\times =  K^\times\ex(\Lambda)$. Let 
$y_1,\dots,y_r\in \Univ^\times$. We construct a sequence~$A_0,\dots,A_n$  of monic operators in $K[\der]$ with~$n\leq r$ as follows.
First, set $A_0:=1$. Next, given~$A_0,\dots,A_{i-1}$ in $K[\der]^{\ne}$ ($1\leq i\leq r$), set $f_i:=A_{i-1}(y_i)$; if $f_i\ne 0$, then $f_i\in \Univ^\times$, so $f_i^\dagger\in K$, and the next term in the sequence
is
$$A_i\ :=\ (\der-a_i)A_{i-1}, \qquad a_i\ :=\ f_i^\dagger,$$ 
whereas if $f_i=0$, then $n:=i-1$ and the construction is finished.

\begin{lemma}\label{lem:distinguished splitting}
$\ker_{\Univ} A_i=Cy_1\oplus\cdots\oplus Cy_i$ \textup{(}internal direct sum\textup{)} for $i=0,\dots,n$.
\end{lemma}
\begin{proof}
By induction on $i\le n$. The case $i=0$ being trivial, suppose $1\le i\le n$ and the claim holds
for $i-1$ in place of $i$.  Then $A_{i-1}(y_i)=f_i\neq 0$, hence $y_i\notin \ker_{\Univ} A_{i-1}=Cy_1\oplus\cdots\oplus Cy_{i-1}$,
and $A_i=(\der-f_i^\dagger)A_{i-1}$, so by [ADH, 5.1.14(i)] we have
$\ker_{\Univ} A_i=\ker_{\Univ} A_{i-1}\oplus Cy_i=Cy_1\oplus\cdots\oplus Cy_i$.
\end{proof}

\noindent
We denote the tuple $(a_1,\dots,a_n)\in K^n$ just constructed by $\operatorname{split}(y_1,\dots,y_r)$,
so~$A_n=(\der-a_n)\cdots(\der-a_1)$. 
Suppose $r\geq 1$. Then $n\ge 1$, $a_1=y_1^\dagger$, $A_1=\der-a_1$,
$A_1(y_2),\dots, A_1(y_n)\in \Univ^\times$,  and we have
$$  (a_2,\dots,a_n)\ =\ \operatorname{split}\!\big(A_1(y_2),\dots,A_1(y_n)\big).$$
 By Lemma~\ref{lem:distinguished splitting},  $n\leq r$ is maximal such that $y_1,\dots,y_n$ are $C$-linearly independent. 
In particular, $y_1,\dots,y_r$ are $C$-linearly independent iff $n=r$.

 \begin{cor}\label{corbasissplit} 
If $A(y_i)=0$ for $i=1,\dots,n$, then $A\in K[\der]A_n$.
Thus if~$n=r$ and $A(y_i)=0$ for $i=1,\dots,r$, then  
$A=a(\der-a_r)\cdots(\der-a_1)$ where~$a\in K^\times$.
\end{cor}

\noindent
This follows from [ADH, 5.1.15(i)] and Lemma~\ref{lem:distinguished splitting}.

\medskip\noindent
Suppose that $H$ is a differential subfield of $K$ and $y_1^\dagger,\dots,y_r^\dagger\in H$.  Then we have~$\operatorname{split}(y_1,\dots,y_r)\in~H^n$: use that $y'\in Hy$ with $y\in U$ gives $y^{(m)}\in Hy$ for all $m$, so~$B(y)\in Hy$
for all $B\in H[\der]$, hence for such $B$, if $f:= B(y)\neq 0$, then~$f^\dagger\in H$.

\begin{cor}\label{corbasiseigenvalues} 
Suppose $\dim_C \ker_{\Univ} A = r$. Then $\ker_{\Univ} A = \ker_{\Omega} A$ and $A$ splits over $K$. If 
$A = (\der-a_r)\cdots(\der-a_1),\ a_1,\dots,a_r\in K$,
then the spectrum of~$A$ is~$\big\{[a_1],\dots,[a_r]\big\}$, and for all $\alpha\in K/K^\dagger$,
$$\mult_{\alpha}(A)\  =\ 
\big| \big\{ i\in\{1,\dots,r\}:\ \alpha=[a_i] \big\}  \big|.$$
\end{cor} 
\begin{proof} 
$A$ splits over $K$ by
 Lemma~\ref{newlembasis} and Corollary~\ref{corbasissplit}. 
The rest follows from Lemma~\ref{lem:split evs} in view of $\sum_\lambda \mult_\lambda(A)=\dim_C \ker_{\Univ} A$.
\end{proof}

\noindent
Conversely, we can associate to a given splitting of $A$ over $K$ a basis of $\ker_{\Univ} A$ consisting of $r$ elements
of~$\Univ^\times$, provided $K$ is $1$-linearly surjective when $r\ge 2$: 

\begin{lemma}\label{lem:basis of kerUA}
Assume $K$ is $1$-linearly surjective in case $r\ge 2$.  Let
$$A\ =\ (\der-a_r)\cdots (\der-a_1)\qquad\text{where $a_i=b_i^\dagger+\lambda_i$, $b_i\in K^\times$, $\lambda_i\in \Lambda$ \textup{(}$i=1,\dots,r$\textup{)}.}$$
Then there are $C$-linearly independent $y_1,\dots,y_r\in \ker_{\Univ} A$   
 such that $y_i\in K^\times\ex(\lambda_i)$ for~$i=1,\dots,r$ and $\operatorname{split}(y_1,\dots,y_r)=(a_1,\dots,a_r)$. 
\end{lemma}
{\sloppy\begin{proof}
By induction on $r$. The case $r=0$ is trivial, and for $r=1$ we can take~$y_1=b_1\ex(\lambda_1)$. Let $r\geq 2$ and suppose inductively that for
$B := {(\der-a_r)\cdots (\der-a_{2})}$
we have $C$-linearly independent $z_2,\dots,z_{r}\in\ker_{\Univ} B$ such that
$z_i \in K^\times\ex(\lambda_{i})$ for~${i=2,\dots,r}$ and~$\operatorname{split}(z_2,\dots,z_r)=(a_2,\dots,a_r)$.
For $i=2,\dots,r$, Lemma~\ref{lem:A(U)=U} gives~${y_i\in K^\times \ex(\lambda_i)}$ with $(\der-a_1)(y_i)=z_i$. Set $y_1:=b_1\ex(\lambda_1)$, so $\ker_{\Univ} (\der-a_1)=~Cy_1$.  Then~$y_1,\dots,y_r\in\ker_{\Univ} A$
are $C$-linearly independent such that $y_i\in K^\times\ex(\lambda_i)$ for~${i=1,\dots,r}$, and one
verifies easily that $\operatorname{split}(y_1,\dots,y_r)=(a_1,\dots,a_r)$.
\end{proof}}

\begin{cor}\label{cor:basis of kerUA}
Assume $K$ is $1$-linearly surjective when $r\ge 2$. Then
$$\text{$A$ splits over~$K$}\ \Longleftrightarrow\ \dim_C \ker_{\Univ} A\ =\ r.$$
\end{cor}

\begin{remark}
If $\dim_{C} \ker_{\Univ} A=r$ and $\lambda_1,\dots,\lambda_d$ are the eigenvalues of $A$ with respect to $\Lambda$, then the differential subring $K\big[\!\ex(\lambda_1),\ex(-\lambda_1),\dots,\ex(\lambda_d),\ex(-\lambda_d)\big]$ of $\Univ$ is the Picard-Vessiot ring for $A$ over $K$; see \cite[Section~1.3]{vdPS}.  
If $K$ is linearly closed  and linearly surjective,
then~$\Univ$ is by Corollary~\ref{cor:basis of kerUA} the universal Picard-Vessiot ring of the differential field~$K$ as defined in \cite[Chapter~10]{vdPS}.
Our construction of~$\Univ$ above is modeled on the description of the universal Picard-Vessiot ring 
of the algebraic closure of $C(\!( t)\!)$ given in    \cite[Chapter~3]{vdPS}. 
\end{remark}

\noindent
Recalling our convention that $r=\order A$, here is a complement to Lemma~\ref{newlembasis}: 

\begin{cor}\label{cor:newlembasis}
Let $V$ be a $C$-linear subspace of $\Univ$ with $r=\dim_C V$. Then 
there is at most one monic $A$ with $V=\ker_{\Univ} A$. Moreover,
the following are equivalent:
\begin{enumerate}
\item[\textup{(i)}] $V=\ker_{\Univ} A$ for some monic $A$  that splits over $K$;
\item[\textup{(ii)}] $V=\ker_{\Univ} B$ for some $B\neq 0$;
\item[\textup{(iii)}] $V=\sum_\lambda V_\lambda$;
\item[\textup{(iv)}] $V$ has a basis contained in $\Univ^\times$.
\end{enumerate}
\end{cor}
\begin{proof}
The first claim follows from [ADH, 5.1.15] applied to the differential fraction field of $\Univ$ in place of $K$.
The implication (i)~$\Rightarrow$~(ii) is clear,
(ii)~$\Rightarrow$~(iii)  was  noted
before Lem\-ma~\ref{newlembasis}, and (iii)~$\Rightarrow$~(iv) is obvious.
For   (iv)~$\Rightarrow$~(i),
 let~$y_1,\dots,y_r\in \Univ^\times$ be a basis of~$V$. Then $\operatorname{split}(y_1,\dots,y_r)=(a_1,\dots, a_r)\in K^r$,
 so  $V=\ker_{\Univ} A$ for~$A=(\der-a_r)\cdots(\der-a_1)$ by Lemma~\ref{lem:distinguished splitting}, so~(i) holds.
\end{proof}

\noindent
Let $y_1,\dots,y_r\in\Univ^\times$ and $(a_1,\dots, a_n):=\operatorname{split}(y_1,\dots,y_r)$. We finish this subsection with some remarks about~$(a_1,\dots,a_n)$ for use in \cite{ADH6}. 
Let~$A_0,\dots,A_n\in K[\der]$ be as above and
recall that $n\leq r$ is maximal such that~$y_1,\dots,y_n$ are $C$-linearly independent.

\begin{lemma}\label{lem:bases with same split} Assume $n=r$.  
Let $z_1,\dots,z_r\in\Univ^\times$. The following are equivalent:
\begin{enumerate}
\item[\textup{(i)}] $z_1,\dots,z_r$ are $C$-linearly independent and $(a_1,\dots, a_r)=\operatorname{split}(z_1,\dots,z_r)$;
\item[\textup{(ii)}] for $i=1,\dots, r$ there are $c_{ii}, c_{i,i-1},\dots, c_{i1}\in C$  such that
$$z_i\ =\ c_{ii}y_i + c_{i,i-1}y_{i-1}+\cdots+c_{i1}y_1\  \text{and}\  c_{ii}\neq 0.$$
\end{enumerate}
\end{lemma}
\begin{proof} The case $r=0$ is trivial. Let $r=1$. If (i) holds, then $y_1^\dagger=a_1=z_1^\dagger$, hence~$z_1\in C^\times\, y_1$, so (ii) holds. The converse is obvious. 
Let $r\ge 2$, and assume (i) holds. Put $\tilde{y}_i:=A_1(y_i)$ and $\tilde{z}_i:=A_1(z_i)$ for $i=2,\dots,r$. Then
$$\operatorname{split}(\tilde{y}_2,\dots,\tilde{y}_r)\ =\ (a_2,\dots,a_r)\ =\ 
\operatorname{split}(\tilde{z}_2,\dots,\tilde{z}_r) ,$$
so we can assume inductively to have $c_{ij}\in C$ ($2\leq j\le i\leq r$)   with 
$$\tilde{z}_i\ =\ c_{ii}\tilde{y}_i+c_{i,i-1}\tilde{y}_{i-1}+\cdots+c_{i2}\tilde{y}_2\quad\text{and}\quad c_{ii}\neq 0 \qquad (2\le i\le r).$$
Hence for $2\le i\le r$,
$$z_i \in c_{ii}y_i + c_{i,i-1}y_{i-1}+\cdots+c_{i2}y_2+\ker_{\Univ} A_1.$$
Now use $\ker_{\Univ} A_1=Cy_1$ to conclude (ii).  For the converse, let
$c_{ij}\in C$  be as in (ii).  Then clearly $z_1,\dots,z_r$ are $C$-linearly independent. 
Let $(b_1,\dots,b_r):=\operatorname{split}(z_1,\dots,z_r)$ and $B_{r-1}:=(\der-b_{r-1})\cdots (\der-b_1)$. 
Then $a_r=f_r^\dagger$ where $f_r=A_{r-1}(y_r)\neq 0$, and $b_r=g_r^\dagger$ where $g_r:=B_{r-1}(z_r)\neq 0$.
Now inductively we have~$a_j=b_j$ for $j=1,\dots,r-1$, so $A_{r-1}=B_{r-1}$, and $A_{r-1}(y_i)=0$ for $i=1,\dots,r-1$   by Lemma~\ref{lem:distinguished splitting}.
Hence $g_r=c_{rr}f_r$, and thus $a_r=b_r$.
\end{proof}

\begin{lemma}\label{lem:split mult conj}
Let   $z\in\Univ^\times$. Then 
$\operatorname{split}(y_1 z,\dots,y_r z) = (a_1+z^\dagger, \dots, a_n+z^\dagger)$.
\end{lemma}
\begin{proof}
Since for $m\leq r$, the units $y_1z,\dots,y_mz$ of $\Univ$ are $C$-linearly independent iff~$y_1,\dots,y_m$ are
$C$-linearly independent, we see that the tuples $\operatorname{split}(y_1 z,\dots,y_r z)$ and
$\operatorname{split}(y_1,\dots,y_r)$ have the same length $n$.
Let  $(b_1,\dots,b_n):=\operatorname{split}(y_1 z,\dots,y_r z)$; we show $(b_1,\dots,b_n)=(a_1+z^\dagger,\dots,a_n+z^\dagger)$
by induction on $n$. The case $n=0$ is obvious, so suppose $n\geq 1$.
Then $a_1=y_1^\dagger$ and~$b_1=(y_1z)^\dagger=a_1+z^\dagger$ as required.
By remarks following the proof of  Lemma~\ref{lem:distinguished splitting}  we have
$$  (a_2,\dots,a_n)\ =\ \operatorname{split}\!\big(A_1(y_2),\dots,A_1(y_n)\big)\qquad\text{where $A_1:=\der-a_1$.}$$ 
Now $B_1:=\der-b_1=(A_1)_{\ltimes z^{-1}}$, so likewise
$$ (b_2,\dots,b_n)\ =\ \operatorname{split}\!\big(B_1(y_2z),\dots,B_1(y_nz)\big) \ =\ 
\operatorname{split}\!\big( A_1(y_2)z,\dots,A_1(y_n)z\big).$$
Hence $b_2=a_2+z^\dagger,\dots,b_n=a_n+z^\dagger$ by our inductive hypothesis.
\end{proof}

\noindent
For $f\in \der(K)$ we let $\int f$ denote an element of $K$ such that~$(\int f)'=f$.  

\begin{lemma}\label{lem:Polya fact, 2}
Let $g_1,\dots,g_r\in K^\times$ and
$$A\  =\  g_1\cdots g_r (\der g_r^{-1}) (\der g_{r-1}^{-1})\cdots  (\der g_1^{-1}),$$
and suppose the integrals below can be chosen such that
$$y_1\ =\ g_1,\quad y_2\ =\ g_1\textstyle\int g_2,\quad \dots,\quad y_r\ =\ g_1\int (g_2\int g_3(\cdots(g_{r-1}\int g_r)\cdots)),$$
Then $y_1,\dots, y_r\in K^\times$, $n=r$, and  $a_i=(g_1\cdots g_i)^\dagger$ for $i=1,\dots,r$.
\end{lemma}
\begin{proof}
Let $b_i:=(g_1\cdots g_i)^\dagger$ for $i=1,\dots,r$. By induction on $i=0,\dots,r$ we show~$n\geq i$ and
$(a_1,\dots,a_i)=(b_1,\dots,b_i)$. This is clear for $i=0$, so suppose~${i\in\{1,\dots,r\}}$, $n\geq i-1$, and $(a_1,\dots,a_{i-1})=(b_1,\dots,b_{i-1})$. 
Then 
$$A_{i-1}=(\der-a_{i-1})\cdots (\der-a_1)=(\der-b_{i-1})\cdots(\der-b_1)=g_1\cdots g_{i-1} (\der g_{i-1}^{-1} )  \cdots   (\der g_1^{-1}),$$
using Lemma~\ref{lem:Polya fact} for the last equality. So $A_{i-1}(y_{i})=g_1\cdots g_{i}\neq 0$, and thus $n\geq i$ and
 $a_{i}=A_{i-1}(y_i)^\dagger=b_i$. 
\end{proof}

\subsection*{The case of real operators} 
We now continue the subsection {\it The real case}\/ of Section~\ref{sec:univ exp ext}.
Thus $K=H[\imag]$ where $H$ is a real closed differential subfield of~$K$ and~${\imag^2=-1}$, and $\Lambda=\Lambda_{\operatorname{r}}+\Lambda_{\operatorname{i}}\imag$ where $\Lambda_{\operatorname{r}}$, $\Lambda_{\operatorname{i}}$ are subspaces
of the $\Q$-linear space~$H$.
The complex conjugation automorphism $z\mapsto \bar{z}$ of the differential field~$K$ extends uniquely to an automorphism
$B\mapsto \bar{B}$ of the ring $K[\der]$ with~$\bar{\der}=\der$. We have~$\overline{A(f)}=\overline{A}(\overline{f})$ for $f\in\Univ$, from which it follows that $\dim_C \ker_K A=\dim_C \ker_K\overline{A}$,
$(\overline{A})_{\lambda}=\overline{(A_{\overline{\lambda}})}$, 
$\mult_\lambda \overline{A} = \mult_{\overline{\lambda}} A$, and $f\mapsto\overline{f}\colon\Univ\to \Univ$ restricts to a $C_H$-linear bijection~$\ker_{\Univ} A \to \ker_{\Univ} \overline{A}$.

\medskip
\noindent
{\it In the rest of this subsection we assume $H=H^\dagger$ $($so $\Lambda=\Lambda_{\operatorname{i}}\imag)$ and $A\in H[\der]$ $($and by earlier conventions, $A\ne 0$ and $r:=\order A)$.}\/
Then $A=\overline{A}$, hence for all $\lambda$ we have $A_{\overline{\lambda}}=\overline{A_\lambda}$ and $\mult_\lambda  A = \mult_{\overline{\lambda}} A$. Thus with $\mu$ ranging over $\Lambda_{\operatorname{i}}^{>}$:
$$\sum_\lambda \mult_\lambda(A)\ =\ \mult_0(A) + 2\sum_{\mu} \mult_{\mu\imag}(A).$$
Note that  $0$ is an eigenvalue of $A$ iff $\ker_H A\neq\{0\}$.

\medskip
\noindent
Let $V:=\ker_{\Univ} A$, a  subspace of the $C$-linear space $\Univ$ with $\overline{V}=V$ and $\dim_C V\leq r$.
Recall that we have the differential $H$-subalgebra $\Univ_{\operatorname{r}}=\{f\in\Univ:\overline{f}=f\}$ of $\Univ$
and the $C_H$-linear subspace $V_{\operatorname{r}} = \ker_{\Univ_{\operatorname{r}}} A$ of $\Univ_{\operatorname{r}}$.
Now
$V=V_{\operatorname{r}}\oplus V_{\operatorname{r}}\imag$ (internal direct sum of $C_H$-linear subspaces), so 
$\dim_C V = \dim_{C_H} V_{\operatorname{r}}$. 
Combining Lemma~\ref{newlembasis} and the remarks preceding it with Lemma~\ref{lem:real basis} and its proof yields:

\begin{cor}\label{cor:complex and real basis}
{\samepage The $C$-linear space $V$ has a basis 
$$a_1\ex(\mu_1\imag),\,\overline{a_1}\ex(-\mu_1\imag),\ \dots,\ a_m\ex(\mu_m\imag),\,\overline{a_m}\ex(-\mu_m\imag), \ h_1,\ \dots,\ h_n \qquad (2m+n\leq r),$$ 
where $a_1,\dots,a_m\in K^\times$, $\mu_1,\dots,\mu_m\in \Lambda_{\operatorname{i}}^{>}$, $h_1,\dots,h_n\in H^\times$.} For such a basis,
$$\Re\!\big(a_1\ex(\mu_1\imag)\big),\,\Im\!\big(a_1\ex(\mu_1\imag)\big),\ \dots,\ \Re\!\big(a_m\ex(\mu_m\imag)\big),\,\Im\!\big(a_m\ex(\mu_m\imag)\big), \ h_1,\ \dots,\ h_n$$
is a basis of the $C_H$-linear space $V_{\operatorname{r}}$, and $h_1,\dots,h_n$ is a basis of the $C_H$-linear subspace~$\ker_H A=V\cap H$ of $H$. 
\end{cor}

\noindent
Using $H=H^\dagger$, arguments as in the proof of Lemma~\ref{lem:basis of kerUA} show:
 
\begin{lemma} \label{lem:basis of kerUA, real}
Assume $H$ is $1$-linearly surjective when $r\ge 2$.  Let $a_1,\dots, a_r\in H$ be such that
$A =(\der-a_r)\cdots (\der-a_1)$. 
Then the $C_H$-linear space $\ker_H A$ has a basis~$y_1,\dots, y_r$
such that $\operatorname{split}(y_1,\dots,y_r)=(a_1,\dots, a_r)$. 
\end{lemma}

\noindent
Recall from Lemma~\ref{lem:size of Sigma(A)} that
if $r=1$ or   $K$ is   $1$-linearly surjective,   then
$$\text{$A$ splits over $K$}\quad\Longleftrightarrow\quad \sum_\lambda \mult_\lambda(A)=r.$$
Now $\mult_\lambda(A)=\mult_{\overline{\lambda}}(A)$ for all $\lambda$, so if
$\mult_\lambda(A)=r\ge 1$, then $\lambda=0$.
Also, for~$W:=V\cap K=\ker_K A$  and $W_{\operatorname{r}}:=W\cap\Univ_{\operatorname{r}}$ we have~$W_{\operatorname{r}}=\ker_H A$ and
$$W\ =\ W_{\operatorname{r}}\oplus W_{\operatorname{r}}\imag\quad\text{ (internal direct sum of $C_H$-linear subspaces)},$$ so $\mult_0(A)=\dim_C \ker_K A=\dim_{C_H} \ker_H A$. If $y_1,\dots,y_r$ is a basis of the $C_H$-linear space~$\ker_H A$, then  $\operatorname{split}(y_1,\dots,y_r)\in H^r$ in reversed order is a splitting of~$A$ over $H$ by Corollary~\ref{corbasissplit}.
These remarks and Lem\-ma~\ref{lem:basis of kerUA, real} now yield:

\begin{cor}\label{cor:unique eigenvalue, 1}
If 
$\mult_0(A)=r$, then $A$ splits over $H$. The converse
holds if~$H$ is $1$-linearly surjective or $r = 1$. 
\end{cor}

\begin{cor}\label{cor:unique eigenvalue, 2} 
Suppose $r\ge 1$, and $K$ is $1$-linearly surjective if $r\geq 2$.  Then
$$\text{$A$ splits over~$H$} \quad \Longleftrightarrow\quad  \mult_0(A)=r\quad \Longleftrightarrow\quad \abs{\Sigma(A)}=1.$$
\end{cor}

\noindent
We now focus on the order $2$ case:

\begin{lemma}\label{lem:order 2 eigenvalues}
Suppose  $r=2$ and $A$ splits over~$K$ but not over $H$. 
Then  $$\dim_{C} \ker_{\Univ} A\ =\ 2.$$
If $H$ is $1$-linearly surjective, 
then $A$ has  two distinct eigenvalues. 
\end{lemma}

{\sloppy
\begin{proof}
We can assume $A$ is monic, so $A=(\der-f)(\der-g)$ with $f,g\in K$ and~$g=a+b\imag$, $a,b\in H$, $b\ne 0$. Then  $g=d^\dagger+\mu\imag$ with $d\in K^\times$ and $\mu\in \Lambda_{\operatorname{i}}$, and so  $d\ex(\mu\imag)\in \ker_{\operatorname{U}}A$. From $A=\bar{A}$ we obtain $\bar{d}\ex(-\mu\imag)\in \ker_{\operatorname{U}}A$. These two elements of~$\ker_{\Univ} A$ are $C$-linearly independent, since 
$$d\ex(\mu\imag)/\bar{d}\ex(-\mu\imag)\ =\ (d/\bar{d})\ex(2\mu\imag)\notin C:$$
this is clear if $\mu\ne 0$, and if $\mu=0$, then $d^\dagger=g$, so~${(d/\bar{d})^\dagger=g-\bar{g}=2b\imag\ne 0}$, and hence $d/\bar{d}\notin C$.   Thus $\dim_{C} \ker_{\operatorname{U}} A\ =\ 2$, and $\mu\imag$, $-\mu \imag$ are eigenvalues of~$A$ with respect to $\Lambda$. Now assume $H$ is $1$-linearly surjective. Then we claim that~${\mu\neq 0}$. To see this note that [ADH, 5.1.21, 5.2.10] and the assumption that $A$ does not split over $H$ yield $ \dim_{C_H} \ker_H A = \dim_C \ker_K A =0$, hence $g\notin K^\dagger$ and thus~${\mu\imag=g-d^\dagger\neq 0}$.
\end{proof}}

\noindent
Combining Lemmas~\ref{lem:basis of kerUA, real} and~\ref{lem:order 2 eigenvalues}    yields:  

\begin{cor}\label{spldcr2} If $H$ is $1$-linearly surjective, $A$ has order~$2$, and $A$  splits over~$K$, then~$\dim_{C} \ker_{\operatorname{U}} A\ =\ 2$.
\end{cor} 

\noindent
{\it In the rest of this subsection $H$ is $1$-linearly surjective and $A=4\der^2+f$, $f\in H$.}\/
 Let~$\omega\colon H \to H$ and~$\sigma\colon H^\times \to H$ be  as in \eqref{eq:def omega} and \eqref{eq:def sigma}.  Then
 by~\eqref{eq:A splits over K} and \eqref{eq:A splits over K[i]}: 
\begin{align*}
\text{$A$ splits over $H$} &\quad\Longleftrightarrow\quad f\in\omega(H), \\
\text{$A$ splits over $K$} &\quad\Longleftrightarrow\quad f\in \sigma(H^\times)\cup\omega(H).
\end{align*}
If $A$ splits over $H$, then $\Sigma(A)=\{0\}$  and~$\operatorname{mult}_0(A)=2$, by Corollary~\ref{cor:unique eigenvalue, 2}.
Suppose~$A$ splits over $K$ but not over~$H$, and let~$y\in H^\times$ satisfy~$\sigma(y)=f\notin\omega(H)$.
Then by~[ADH, p.~262] we have~$A=4(\der+g)(\der-g)$ where~$g=\frac{1}{2}(-y^\dagger+y\imag)$.
Hence the two distinct eigenvalues of~$A$
are~$(y/2)\imag+K^\dagger$ and~$-(y/2)\imag+K^\dagger$. 

\subsection*{The case of oscillating transseries} 
We now apply the results above to the algebraically closed differential field $K=\T[\imag]$. Note that $\T[\imag]$ has constant field $\C$ and extends the (real closed) differential field $\T$ of transseries. \index{transseries!oscillating}   
After \eqref{eq:14.5.7} in the introduction, we already remarked:

\begin{lemma}\label{lem:T[imag] lc and ls}
$\T[\imag]$ is linearly closed and linearly surjective.
\end{lemma}
%\begin{proof}
%By [ADH,  15.0.2], $\T$ is newtonian, so $\T[\imag]$ is newtonian by \eqref{eq:14.5.7}. Hence $\T[\imag]$ is linearly closed by \eqref{eq:14.5.3}, and   linearly surjective  by [ADH,  14.2.2].
%\end{proof}

\noindent
Now  applying Corollary~\ref{cor:basis of kerUA} and Lemma~\ref{newlembasis} to $\T[\imag]$ gives:

\begin{cor}
For $K=\T[\imag]$, there are $\mathbb C$-linearly independent units $y_1,\dots,y_r$ of $\Univ_{\T[\imag]}$ with $A(y_1)=\cdots = A(y_r)=0$.
\end{cor}

\noindent
Next we describe another incarnation of $\Univ_{\T[\imag]}$, namely as a ring $\mathbb O$ of ``oscillating'' transseries. Towards this goal we first note that by  
 [ADH, 11.5.1, 11.8.2] we have
\begin{align*}
\I(\T)\	&=\ \big\{y\in\T:\ \text{$y\preceq f'$ for some $f\prec 1$ in $\T$} \big\}\\
		&=\ \big\{y\in \T:\ \text{$y\prec 1/(\ell_0\cdots\ell_n)$ for all $n$} \big\},
\end{align*}
so a complement $\Lambda_{\T}$ of $\I(\T)$ in $\T$ is given by
$$
\Lambda_{\T}\	:=\ \big\{y\in\T:\  \text{$\supp(y) \succ 1/(\ell_0\cdots\ell_{n-1}\ell_n^2)$ for all $n$} \big\}.$$
Since $\T^\dagger=\T$ and $\I\!\big(\T[\imag]\big)\subseteq \T[\imag]^\dagger$ we
have $\T[\imag]^\dagger=\T\oplus \I(\T)\imag$ by Lemmas~\ref{lem:logder} and~\ref{lem:W and I(F)}.
We now take $\Lambda=\Lambda_{\T}\imag$ as our complement $\Lambda$ of $\T[\imag]^\dagger$ in $\T[\imag]$ and 
explain how the universal exponential extension~$\Univ$ of $\T[\imag]$ for this~$\Lambda$ was introduced in \cite[Section~7.7]{JvdH} in a different way.
Let $$\T_{\succ}\ :=\ \{f\in\T:\supp f\succ 1\},$$ and similarly with $\prec$ in place of $\succ$; then $\T_{\prec}=\smallo_{\T}$ and~$\T_{\succ}$ are
$\R$-linear subspaces of $\T$, and~$\T$ decomposes as an internal direct sum 
\begin{equation}\label{eq:T decomp}
\T\ =\ \T_{\succ}\oplus \R \oplus \T_{\prec}
\end{equation}
of $\R$-linear subspaces of $\T$.
Let $\ex^{\imag \T_{\succ}}=\{\ex^{\imag f}:f\in\T_{\succ}\}$ be a multiplicative copy of the additive group~$\T_{\succ}$, with isomorphism $f\mapsto \ex^{\imag f}$. Then we have the group ring
$$\mathbb{O}\ :=\ K\big[\ex^{\imag\T_{\succ}}\big]$$
of $\ex^{\imag\T_{\succ}}$ over $K=\T[\imag]$. 
We make $\mathbb{O}$ into a differential ring extension of $K$ by
$$(\ex^{\imag f})'\ =\  \imag f' \ex^{\imag f}\qquad (f\in\T_{\succ}).$$
Hence $\mathbb{O}$ is an exponential extension of $K$.
The elements of $\mathbb{O}$ are called {\it oscillating trans\-series.}\/ 
For each $f\in\T$ there is a unique $g\in\T$, to be denoted by $\int f$,
such that~$g'=f$ and $g$ has constant term $g_1=0$.
The injective map $\int\colon \T\to\T$ is $\R$-linear; we use this map 
to show that $\Univ$ and $\mathbb{O}$ are disguised versions of each other:

\begin{prop}\label{prop:osc transseries}
There is a unique isomorphism $\Univ=K\big[\!\ex(\Lambda)\big] \to \mathbb{O}$ of differential $K$-algebras sending $\ex(h\imag)$ to $\ex^{\imag\int h}$ for all $h\in \Lambda_{\T}$.
\end{prop}

\noindent
This requires the next lemma. We assume familiarity with [ADH, Appendix~A], especially with the ordered group $G^{\operatorname{LE}}$ (a subgroup of $\T^\times$) of
logarithmic-exponential monomials and its subgroup
$G^{\operatorname{E}} = \bigcup_n G_n$ of exponential monomials.

\begin{lemma}\label{succP} If $\fm\in G^{\operatorname{LE}}$ and $\fm\succ 1$, then $\supp \fm'\ \subseteq\ \Lambda_{\T}$. 
\end{lemma}
\begin{proof} We first prove by induction on $n$ a fact about elements of $G^{\operatorname{E}}$: 
$$\text{ if $\fm\in G_n$, $\fm\succ 1$,
then $\supp \fm' \succ 1/x$}.$$ For $r\in\R^>$ we have $(x^r)'=rx^{r-1}\succ 1/x$,
so the claim holds for $n=0$. Suppose the claim holds for a certain $n$. Now
$G_{n+1}=G_n\exp(A_n)$, $G_n$ is a convex subgroup of $G_{n+1}$, and 
$$A_n\ =\ \big\{f\in \R[[G_n]]:\ \supp f\succ G_{n-1}\big\}\qquad\text{ (where $G_{-1}:=\{1\}$).}$$
Let $\fm=\fn\exp(a)\in G_{n+1}$ where $\fn\in G_n$, $a\in A_n$; then
$$\fm\succ 1\quad\Longleftrightarrow\quad\text{$a>0$, or $a=0$, $\fn\succ 1$.}$$
Suppose $\fm\succ 1$. If $a=0$, then $\fm=\fn$, and we are done by inductive hypothesis,
so assume $a>0$. Then $\fm' = (\fn'+\fn a')\exp(a)$ and 
$(\fn'+\fn a')\in \R[[G_n]]$, a differential subfield of $\T$, and
$\exp(a) > \R[[G_n]]$,  hence $\supp \fm' \succ 1\succ 1/x$ as required.

Next, suppose $\fm\in G^{\operatorname{LE}}$ and $\fm\succ 1$. Take $n\ge 1$ such that $\fm{\uparrow}^n\in G^{\operatorname{E}}$. 
We have~$(\fm{\uparrow}^n)'= (\fm'\cdot \ell_0\ell_1\cdots\ell_{n-1}){\uparrow}^n$. For $\fn\in\supp\fm'$ and using $\fm{\uparrow}^n\succ 1$
this gives
$$ (\fn\cdot \ell_0\ell_1\cdots\ell_{n-1} ){\uparrow}^n\ \succ\ 1/x$$
by what we proved for monomials in $G^{\operatorname{E}}$.  Applying ${\downarrow}_n$ this yields
$\fn\succ 1/(\ell_0\ell_1\cdots\ell_n)$,
hence $\fn\in \Lambda_{\T}$ as claimed.
\end{proof}

\begin{proof}[Proof of Proposition~\ref{prop:osc transseries}] Applying $\der$ to the decomposition~\eqref{eq:T decomp} gives $$\T\ =\ \der(\T_{\succ}) \oplus \der(\T_{\prec}).$$ Now $\der(\T_{\succ})\subseteq \Lambda_{\T}$ by Lemma~\ref{succP}, and $\der(\T_{\prec})\subseteq \I(\T)$, and so these two inclusions are equalities. Thus $\int \Lambda_{\T}= \T_{\succ}$, from which the proposition follows.
\end{proof}

\begin{prop} There is a unique group morphism $\exp \colon K=\T[\imag]\to \mathbb O^\times$
that extends the given exponential maps $\exp\colon \T\to \T^\times$ and
$\exp\colon \mathbb C \to \mathbb C^\times$, and such that~$\exp(\imag f)=\ex^{\imag f}$ for all $f\in \T_{\succ}$ and $\exp(\epsilon)=\sum_n \frac{\varepsilon^n}{n!}$ for all $\epsilon\in \smallo$. It is surjective, has kernel $2\pi\imag \Z\subseteq \mathbb C$, and satisfies $\exp(f)'=f'\exp(f)$ for all $f\in K$.   
\end{prop}
\begin{proof} The first statement follows easily from the decompositions
$$ K\ =\ \T \oplus \imag \T\ =\ \T \oplus \imag \T_{\succ} \oplus \imag \R \oplus \imag \smallo_{\T},\qquad \mathbb C\ =\ \R\oplus \imag \R, \qquad \smallo\ =\ \smallo_{\T}\oplus \imag \smallo_{\T}$$
of $K$, $\mathbb C$, and $\smallo=\smallo_K$ as internal direct sums of $\R$-linear subspaces. Next, 
$$\mathbb O^\times\ =\  K^\times \ex^{\imag \T_{\succ}}=\ \T^{>}\cdot S_{\mathbb C}\cdot (1+\smallo)\cdot\ex^{\imag\T_{\succ}}, \qquad S_{\mathbb C}\ :=\ \big\{z\in \mathbb{C}:\ |z|=1\big\},$$
by Lemmas~\ref{lem:only trivial units} and~\ref{lem:logder}, and Corollary~\ref{cor:decomp of S}.
Now $\T^{>}=\exp(\T)$ and $S_{\mathbb C}=\exp(\imag \R)$, so surjectivity follows from
$\exp(\smallo)=1+\smallo$, a consequence of
the well-known bijectivity of the map $\epsilon\mapsto \sum_n\frac{\varepsilon^n}{n!}\colon \smallo\to 1+\smallo$, whose inverse is given by
$$1+\delta\mapsto \log(1+\delta)  :=\sum_{n=1}^\infty \frac{(-1)^{n-1}}{n} \delta^n \qquad(\delta\in \smallo).$$
That the kernel is $2\pi\imag\Z$ follows from the initial decomposition of the additive group of $K$ as $\T \oplus \imag \T_{\succ} \oplus \imag \R \oplus \imag\smallo_{\T}$. The identity $\exp(f)'=f'\exp(f)$ for $f\in K$ follows from it being satisfied for $f\in\T$, $f\in\imag \T_{\succ}$, 
$f\in\mathbb C$, and $f\in\smallo$.  
\end{proof}  

\noindent
To integrate oscillating transseries, note first that the
$\R$-linear operator $\int\colon \T\to \T$ extends uniquely to a $\C$-linear operator $\int\colon \T[\imag] \to \T[\imag]$. This in turn extends  uniquely to a~$\C$-linear operator~$\int\colon \mathbb{O} \to \mathbb{O}$ such that~$(\int \Phi)'=\Phi$ for all $\Phi\in \mathbb{O}$ and
$\int \T[\imag] \ex^{\phi\imag}\subseteq \T[\imag]\ex^{\phi\imag}$ for all $\phi\in \T_{\succ}$: given~$\phi\in \T_{\succ}^{\ne}$ and~${g\in \T[\imag]}$, there is a unique
$f\in \T[\imag[$ such that $(f\ex^{\phi\imag})'=g\ex^{\phi\imag}$: existence holds because~${y'+y\phi'\imag=g}$ has a solution in $\T[\imag]$, the latter being linearly surjective, and uniqueness holds by Lemma~\ref{1K} applied to $K=L=\T[\imag]$, because $\phi'\imag\notin \T[\imag]^\dagger$ in view of remarks preceding Lemma~\ref{lem:W and I(F)}. 

\medskip
\noindent
The operator $\int$ is a right-inverse of
the linear differential operator~$\der$ on~$\mathbb{O}$. To extend this to other linear differential operators, make the subgroup~$G^{\mathbb{O}}:=G^{\operatorname{LE}}\ex^{\imag \T_{\succ}}$ of~$\mathbb{O}^\times$ into an ordered group so that the ordered subgroup $G^{\operatorname{LE}}$ of $\T^{>}$ is a convex ordered subgroup of $G^{\mathbb{O}}$ and 
$\ex^{\imag\phi} \succ G^{\operatorname{LE}}$ for $\phi>0$ in $\T_{\succ}$. (Possible in only one way.) Next, extend the natural inclusion 
$\T[\imag]\to \C[[G^{\operatorname{LE}}]]$ to a $\C$-algebra embedding~${\mathbb{O}\to \C[[G^{\mathbb{O}}]]}$ by sending
$\ex^{\imag\phi}\in \mathbb{O}$ to $\ex^{\imag\phi}\in G^{\mathbb{O}}\subseteq \C[[G^{\mathbb{O}}]]$. Identify $\mathbb{O}$
with a subalgebra of~$\C[[G^{\mathbb{O}}]] $ via this embedding, so $\supp f \subseteq G^{\mathbb{O}}$ for $f\in \mathbb{O}$. It makes the Hahn space~$\C[[G^{\mathbb{O}}]]$ over $\C$ an immediate extension of its valued subspace $\mathbb{O}$. The latter is in particular also a Hahn space over $\C$.

Let $A\in\T[\imag][\der]^{\neq}$. Then $A(\mathbb{O})=\mathbb{O}$ by Lemmas~\ref{lem:A(U)=U},~\ref{lem:T[imag] lc and ls}, and Proposition~\ref{prop:osc transseries}. 
The proof of [ADH, 2.3.22]  now  gives for each
$g\in \mathbb{O}$ a unique element~${f=:A^{-1}(g)\in\mathbb{O}}$ with $A(f)=g$ and
$\supp(f)\cap \fd\big(\!\ker_{\mathbb O}^{\neq} A\big)=\emptyset$.
This requirement on $\supp A^{-1}(g)$ yields a $\C$-linear   operator~$A^{-1}$ on $\mathbb O$  with~$A\circ A^{-1}=\id_{\mathbb O}$; we call it the {\bf distinguished} right-inverse of the operator $A$ on $\mathbb O$.  
With this definition~$\der^{-1}$ is the operator~$\int$ on $\mathbb{O}$ specified earlier.  

\medskip
\noindent
In the next section we explore various valuations on universal exponential extensions (such as $\mathbb{O}$)
with additional properties. 

\section{Valuations on the Universal Exponential Extension}\label{sec:valuniv}

{\sloppy
\noindent
{\it In this section $K$ is a valued differential field with algebraically closed constant field~${C\subseteq\mathcal O}$ and divisible group $K^\dagger$ of
logarithmic derivatives.}\/
Then $\Gamma=v(K^\times)$ is also divisible,
since we have a group isomorphism $$va\mapsto a^\dagger+(\mathcal O^\times)^\dagger\ :\ \Gamma\to K^\dagger/(\mathcal O^\times)^\dagger \qquad (a\in K^\times).$$ 
Let $\Lambda$ be a complement of the $\Q$-linear subspace~$K^\dagger$ of $K$, let $\lambda$ range over $\Lambda$, 
let~${\Univ=K\big[\!\ex(\Lambda)\big]}$ be the universal exponential extension of $K$ constructed in Section~\ref{sec:univ exp ext}
and set $\Omega:=\Frac(\Univ)$. Thus $\Omega$ is a differential field with
constant field~$C$.}

\subsection*{The gaussian extension}
We equip   $\Univ$
with the gaussian extension~$v_{\g}$ of the valuation of~$K$ as defined in Section~\ref{sec:group rings}; so for $f\in\Univ$ with spectral decomposition~$(f_\lambda)$: $$v_{\g}(f)\ =\ \min_\lambda v(f_\lambda),$$
and hence 
$$v_{\g}(f')\ =\ \min_\lambda v(f_\lambda'+\lambda f_\lambda).   $$
The field $\Omega$ with the valuation extending $v_{\g}$ is a valued differential field extension of $K$, but it can happen that 
$K$ has small derivation, whereas  $\Omega$  does not:

\begin{example} Let $K=C(\!( t^{\Q} )\!)$ and $\Lambda$ be as in Example~\ref{ex:Q}, so $t\prec 1\prec x=t^{-1}$ and~$t'=-t^2$. Then $K$ is $\d$-valued of $H$-type with small derivation, but in $\Omega$ with the above valuation, $$t\ex(x)\ \prec\ 1, \qquad \big(t\ex(x)\big){}'\ =\ -t^2\ex(x)+\ex(x)\ \sim\ \ex(x)\ \asymp\ 1.$$
To obtain an example where $K=H[\imag]$ for a Liouville closed $H$-field $H$ and $\imag^2=-1$, take $K:=\T[\imag]$ and $\Lambda:=\Lambda_{\T}\imag$ as at the end of Section~\ref{sec:eigenvalues and splitting}. Now $x\in \Lambda_{\T}$ and in $\Omega$ equipped with the above valuation we have for $t:=x^{-1}$:
$$t\ex(x\imag)\ \prec\ 1,\qquad \big(t\ex(x\imag)\big){}'\ =\ -t^2\ex(x\imag)+\imag\ex(x\imag)\ \sim\ \imag\ex(x\imag)\ \asymp\ 1,$$
so $\big(t\ex(x\imag)\big){}' \not\prec t^\dagger$, hence $\Omega$ is neither asymptotic nor has small derivation.
\end{example}

\noindent
However, we show next that under certain assumptions on $K$ with small derivation, $\Omega$ has also a valuation
which does  make $\Omega$ a valued differential field extension of~$K$ with small derivation.
For this we rely on results from [ADH, 10.4].
Although such a valuation is less canonical than $v_{\g}$, it is useful for  harnessing the finiteness statements about the set $\exc^{\ev}(A)$ of eventual exceptional values of $A\in K[\der]^{\neq}$
from  Section~\ref{sec:lindiff} to obtain similar facts about the set of {\it ultimate exceptional values}\/ of~$A$ introduced  later in this section.

\subsection*{Spectral extensions}
{\it In this subsection $K$ is $\d$-valued of $H$-type with $\Gamma\neq\{0\}$ and with small derivation.}\/
 
\begin{lemma}\label{specex}
The valuation of $K$ extends to a valuation on the field~$\Omega$ that makes~$\Omega$ a $\d$-valued extension of $K$ of $H$-type with small derivation.
\end{lemma}
\begin{proof} Applying [ADH, 10.4.7] to an algebraic closure of $K$ gives a $\d$-valued algebraically closed extension $L$ of $K$ of $H$-type with small derivation and $C_L=C$
such that~$L^\dagger\supseteq K$.
Let~$E:=\{y\in L^\times:\, y^\dagger\in K\}$, so $E$ is a subgroup of~$L^\times$, $E^\dagger=K$, and~$K[E]$ is an exponential extension of $K$ with $C_{K[E]}=C$. Then Corollary~\ref{corcharexp}
gives an embedding~${\Univ\to L}$ of differential $K$-algebras with image $K[E]$, which extends to an
embedding $\Omega\to L$ of differential fields. Using this embedding to transfer the valuation of $L$ to $\Omega$ gives a valuation as required.
\end{proof}

\noindent
A {\bf spectral extension} of the valuation of $K$ to $\Omega$ is a valuation on the field $\Omega$ with the properties stated in Lemma~\ref{specex}.\index{valuation!spectral extension}\index{spectral extension}\index{extension!spectral}
If $K$ is $\upo$-free, then so is $\Omega$ equipped with any spectral extension of the valuation of $K$,
by [ADH, 13.6] (and then $\Omega$ has rational asymptotic integration by [ADH, 11.7]).
We do not know whether this goes through with ``$\upl$-free'' instead of ``$\upo$-free''. Here is something weaker:

\begin{lemma}\label{lem:Omega as int} 
Suppose $K$ is algebraically closed and $\upl$-free. Then some spectral extension of the valuation of $K$ to $\Omega$  makes 
$\Omega$ a $\d$-valued field with divisible value group and asymptotic integration.
\end{lemma}
\begin{proof}
Take $L$, $E$ and an embedding $\Omega\to L$   as in the proof  of Lemma~\ref{specex}. Use this embedding  to identify~$\Omega$ with a differential subfield of~$L$, so $U=K[E]$ and $\Omega=K(E)$, and
equip $\Omega$ with the spectral extension of the valuation of $K$ obtained  by restricting the valuation of $L$ to $\Omega$.
Since $L$ is algebraically closed, $E$ is divisible, and $\Gamma_L=\Gamma+v(E)$ by [ADH, 10.4.7(iv)]. So~$\Gamma_\Omega=\Gamma_L$ is divisible.
Let~$a\in K^\times$, $y\in E$. Then $K(y)$  
has  asymptotic integration by Proposition~\ref{prop:upl-free and as int}, hence $v(ay)\in (\Gamma_{K(y)}^{\neq})'\subseteq (\Gamma_\Omega^{\neq})'$. Thus $\Omega$ has asymptotic integration.
\end{proof} 

\noindent
{\it In the rest of this subsection $\Omega$ is equipped with a spectral extension $v$ \textup{(}with value group $\Gamma_{\Omega}$\textup{)} of the valuation of~$K$.}\/ The proof of Lemma~\ref{specex} and [ADH, 10.4.7] show that we can choose $v$ so that $\Psi_\Omega\subseteq\Gamma$;
but under suitable hypotheses on $K$, this is automatic: 

{\sloppy
\begin{lemma}\label{lem:v(ex(Q))}
Suppose   $K$ has asymptotic integration and $\I(K)\subseteq K^\dagger$. 
Then~${\Psi_{\Omega}\subseteq \Gamma}$, the group morphism
\begin{equation}\label{eq:v(ex(q))}
\lambda\mapsto v\big(\!\ex(\lambda)\big)\ \colon\ \Lambda\to\Gamma_\Omega
\end{equation}
is injective, and $\Gamma_\Omega$ is divisible with $\Gamma_\Omega=\Gamma\oplus v\big(\!\ex(\Lambda)\big)$ \textup{(}internal direct sum of $\Q$-linear subspaces of $\Gamma_\Omega$\textup{)}. Moreover,
$\Psi_\Omega=\Psi^{\downarrow}$ in $\Gamma$. 
\end{lemma}}
\begin{proof}
For $a\in K^\times$ we have $\big(a\ex(\lambda)\big){}^\dagger=a^\dagger+\lambda\in K$, and 
if $a\ex(\lambda)\asymp 1$, then 
$$a^\dagger+\lambda\ =\ \big(a\ex(\lambda)\big){}^\dagger\in (\mathcal O^\times_\Omega)^\dagger\cap K\ \subseteq\ \I(\Omega)\cap K\ =\ \I(K),$$ 
so $\lambda\in \Lambda\cap\big(\I(K)+K^\dagger\big)=\Lambda\cap K^\dagger=\{0\}$ and  $a\asymp 1$.  Thus for $a_1,a_2\in K^\times$ and 
distinct $\lambda_1, \lambda_2\in \Lambda$ we have $a_1\ex(\lambda_1)\nasymp a_2\ex(\lambda_2)$, and so
for $f\in\Univ$ with spectral decomposition $(f_\lambda)$ we have
$vf=\min_\lambda v\big(f_\lambda\ex(\lambda)\big)$. Hence
$$\Psi_{\Omega}\ \subseteq 
\big\{v(a^\dagger+\lambda):\ a\in K^\times,\ \lambda\in \Lambda\big\}\ =\ v(K)\ =\ \Gamma_{\infty},$$ the map~\eqref{eq:v(ex(q))} is injective and 
$\Gamma\cap v\big(\!\ex(\Lambda)\big) = \{0\}$, and so 
$\Gamma_\Omega=\Gamma\oplus v\big(\!\ex(\Lambda)\big)$ (internal direct sum of subgroups
of $\Gamma_\Omega$). Since $\Gamma$ and $\Lambda$ are divisible, so is $\Gamma_\Omega$. 
Now $\Psi_\Omega=\Psi^{\downarrow}$  follows from $K=(\Univ^\times)^\dagger\subseteq \Omega^\dagger$
 and $K$ having asymptotic integration.
\end{proof}

\noindent
We can now improve on Lemma~\ref{newlembasis}: 

\begin{cor}\label{cor:spectral valuation basis}
Suppose $K$ has asymptotic integration and $\I(K)\subseteq K^\dagger$, and let~$A\in K[\der]^{\neq}$. Then the $C$-linear space $\ker_{\operatorname{U}} A$ has a basis
$\mathcal B\subseteq \operatorname{U}^\times$ such that $v$ is injective on $\mathcal B$ and
$v(\mathcal B)=v(\ker^{\neq}_{\operatorname{U}} A)$, and thus $\abs{v(\ker^{\neq}_{\operatorname{U}} A)} =\dim_C \ker_{\operatorname{U}} A$.  
\end{cor}
\begin{proof} By [ADH, 5.6.6] we have a basis $\mathcal B_\lambda$ of the $C$-linear space $\ker_K A_\lambda$ such that~$v$ is injective on $\mathcal{B}_{\lambda}$ and $v(\mathcal B_\lambda)=v(\ker^{\neq}_K A_\lambda)$.
Then $\mathcal B:=\bigcup_\lambda \mathcal B_\lambda\ex(\lambda)$ is a basis of $\ker_{\operatorname{U}} A$. It has the desired properties by Lemma~\ref{lem:v(ex(Q))}.
\end{proof}

\begin{cor}\label{cor:Omega as int}
Suppose $K$ is $\upl$-free and $\I(K)\subseteq K^\dagger$. Then  $\Omega$ has asymptotic integration. %, and so its $H$-asymptotic couple is closed by Lemma~\ref{lem:v(ex(Q))}. 
\end{cor}
\begin{proof}
By Lemma~\ref{lem:v(ex(Q))}, $\Gamma_\Omega= \Gamma+ v\big(\!\ex(\Lambda)\big)$. Using Proposition~\ref{prop:upl-free and as int} as in the proof of Lemma~\ref{lem:Omega as int}, with $\ex(\Lambda)$ in place of $E$, shows $\Omega$ has asymptotic integration.
\end{proof}

\subsection*{Ultimate exceptional values}
{\it In  this subsection $K$ is $H$-asymptotic with small derivation and asymptotic integration.}\/
Also  $A\in K[\der]^{\neq}$ and $r:=\order(A)$, and $\gamma$ ranges over $\Gamma=v(K^\times)$.
We have $v(\ker^{\neq} A_\lambda)\subseteq \exc^{\ev}(A_\lambda)$, so if
$\lambda$ is an eigenvalue of~$A$ with respect to $\lambda$, then $\exc^{\ev}(A_\lambda)\neq\emptyset$.  
We call the elements of the set
$$\exc^{\operatorname{u}}(A)\ =\ \exc^{\operatorname{u}}_{K}(A)\ :=\ \bigcup_\lambda\, \exc^{\ev}(A_\lambda)\ = \ \big\{ \gamma :\,  \text{$\nwt_{A_\lambda}(\gamma)\geq 1$ for some $\lambda$} \big\}$$
the {\bf ultimate exceptional values of $A$} with respect to $\Lambda$. \label{p:excu} The definition of $\exc^{\operatorname{u}}_{K}(A)$ involves our choice of $\Lambda$,  but we are leaving this implicit to avoid complicated notation.
In Section~\ref{sec:ultimate} we shall restrict $K$  and $\Lambda$ so that~$\exc^{\operatorname{u}}(A)$ does not depend any longer on the choice of $\Lambda$. There we shall use the following observation: \index{linear differential operator!ultimate exceptional values}\index{values!ultimate exceptional}\index{exceptional values!ultimate}\index{ultimate!exceptional values}

\begin{lemma}\label{lem:excu for different Q}
Let $a,b\in K$ be such that $a-b\in (\mathcal O^\times)^\dagger$. Then for all $\gamma$ we have~$\nwt_{A_a}(\gamma)=\nwt_{A_b}(\gamma)$; in particular, $\exc^{\ev}(A_a)=\exc^{\ev}(A_b)$.
\end{lemma} 
\begin{proof}
Use that if $u\in\mathcal O^\times$ and $a-b=u^\dagger$, then $A_a=(A_b)_{\ltimes u}$.
\end{proof}

\begin{cor}\label{cor:excu for different Q}
Let $\Lambda^*$ be a complement of the $\Q$-linear subspace $K^\dagger$ of $K$ and let~$\lambda\mapsto \lambda^*\colon \Lambda \to \Lambda^*$ be the group isomorphism with $\lambda- \lambda^*\in K^\dagger$ for all $\lambda$. If $\lambda-\lambda^*\in (\mathcal O^\times)^\dagger$ for all $\lambda$, then $\nwt_{A_\lambda}(\gamma)=\nwt_{A_{\lambda^*}}(\gamma)$ for all $\gamma$, so $\exc^{\operatorname{u}}(A)= \bigcup_\lambda\, \exc^{\ev}(A_{\lambda^*})$. 
\end{cor}

\begin{remarkNumbered}\label{rem:excu for different Q, 1}
For $a\in K^\times$ we have 
$\exc^{\operatorname{u}}(aA)=\exc^{\operatorname{u}}(A)$
and $\exc^{\operatorname{u}}(Aa)=\exc^{\operatorname{u}}(A)-va$.
Note also that~$\exc^{\ev}(A) = \exc^{\ev}(A_0) \subseteq \exc^{\operatorname{u}}(A)$.
Let $\phi\in K^\times$ be active in $K$, and set~$\lambda^\phi:=\phi^{-1}\lambda$.
Then~$\Lambda^\phi:=\phi^{-1}\Lambda$ is a complement of the $\Q$-linear subspace $(K^\phi)^\dagger=\phi^{-1}K^\dagger$  of~$K^\phi$,
and $(A^\phi)_{\lambda^\phi} = (A_\lambda)^\phi$. Hence $\exc^{\operatorname{u}}_{K}(A)$ agrees with   the set~$\exc^{\operatorname{u}}_{K^\phi}(A^\phi)$ of ultimate exceptional values of $A^\phi$ with respect to $\Lambda^\phi$. 
\end{remarkNumbered}

\begin{remarkNumbered}\label{rem:excu for different Q, 2}
Suppose $L$ is an $H$-asymptotic extension of $K$ with asymptotic integration and algebraically closed constant field $C_L$ such that 
$L^\dagger$ is divisible, and~$\Psi$ is cofinal in~$\Psi_L$ or $K$ is $\upl$-free.
Then~$\exc^{\ev}(A_\lambda)=\exc^{\ev}_L(A_\lambda)\cap\Gamma$, by Lemma~\ref{lemexc} and Corollary~\ref{cor:13.7.10}. Hence if $\Lambda_L\supseteq \Lambda$ is a complement of  the  subspace~$L^\dagger$ of the $\Q$-linear space $L$, and~$\exc^{\operatorname{u}}_{L}(A)$ is  the set of ultimate exceptional values of $A$ (viewed as an element of~$L[\der]$) with respect to~$\Lambda_L$, then $\exc^{\operatorname{u}}(A)\subseteq \exc^{\operatorname{u}}_{L}(A)$. (Note that such a complement~$\Lambda_L$ exists iff~$L^\dagger\cap K=K^\dagger$.)
\end{remarkNumbered}

\noindent
In the rest of this subsection  we equip $\Univ$ with the gaussian extension $v_{\g}$ of the valuation of~$K$.
Recall that we have a decomposition $\ker_{\Univ} A = \bigoplus_\lambda (\ker A_\lambda)\ex(\lambda)$ of the $C$-linear space~$\ker_{\Univ} A$
as an internal direct sum of subspaces, 
and hence
\begin{equation}\label{eq:vU(ker A)}
v_{\g}(\ker_{\Univ}^{\neq} A)\ =\ \bigcup_\lambda\, v(\ker^{\neq} A_\lambda)\ \subseteq\ \bigcup_\lambda\, \exc^{\ev}(A_\lambda)\ =\ \exc^{\operatorname{u}}(A).
\end{equation}
Here are some consequences:

\begin{lemma}\label{lem:excu 1}
Suppose $K$ is $r$-linearly newtonian. Then $v_{\g}(\ker_{\Univ}^{\neq} A)=\exc^{\operatorname{u}}(A)$.
\end{lemma}
\begin{proof}
By Proposition~\ref{kerexc} we have $v(\ker^{\neq} A_\lambda) = \exc^{\ev}(A_\lambda)$ for each $\lambda$. There\-fore~$v_{\g}(\ker_{\Univ}^{\neq} A)=\exc^{\operatorname{u}}(A)$ by \eqref{eq:vU(ker A)}.
\end{proof}

\begin{lemma}\label{lem:excu 2}
Suppose $K$ is $\d$-valued. Then  $\abs{v_{\g}(\ker_{\Univ}^{\neq} A)}\leq \dim_C \ker_{\Univ} A\leq r$.
\end{lemma}
{\samepage
\begin{proof}
By [ADH, 5.6.6(i)] applied to $A_\lambda$ in place of $A$ we have
$$\abs{v(\ker^{\neq} A_\lambda)}\ =\ \dim_C \ker A_\lambda\ =\ \mult_\lambda(A)\qquad\text{ for all~$\lambda$}$$ 
and thus by \eqref{eq:vU(ker A)},
$$\abs{v_{\g}(\ker_{\Univ}^{\neq} A)}\ \leq\ 
\sum_\lambda \,\abs{v(\ker^{\neq} A_\lambda)}\ =\ 
\sum_\lambda \,\mult_\lambda(A)\ =\ \dim_C \ker_{\Univ} A\ \leq\ r$$
as claimed.
\end{proof}
}

\begin{lemma}\label{lem:excu, r=1} 
Suppose  $\I(K)\subseteq K^\dagger$ and $r=1$. Then $$v_{\operatorname{g}}(\ker^{\neq}_{\Univ} A)\ =\ \exc^{\operatorname{u}}(A), \qquad \abs{\exc^{\operatorname{u}}(A)}\ =\ 1.$$
\end{lemma}
\begin{proof}
Arrange $A=\der-g$, $g\in K$, and take $f\in K^\times$ and $\lambda$ such that~$g=f^\dagger+\lambda$. Then $u:=f\ex(\lambda)\in\Univ^\times$
satisfies $A(u)=0$, hence $\ker^{\neq}_{\Univ}A=Cu$ and thus~$v_{\text{g}}(\ker^{\neq}_{\Univ} A) = \{ vf \}$.
By Lemma~\ref{lem:v(ker)=exc, r=1} we have $v(\ker^{\neq} A_\lambda)=\exc^{\ev}(A_\lambda)$ for all $\lambda$ and
hence $v_{\operatorname{g}}(\ker^{\neq}_{\Univ} A) = \exc^{\operatorname{u}}(A)$ by \eqref{eq:vU(ker A)}.
\end{proof}

\begin{cor}\label{corevisu} If $\I(K)\subseteq K^\dagger$ and $a\in K^\times$, then $\exc^{\ev}(\der-a^\dagger)=\exc^{\operatorname{u}}(\der-a^\dagger)=\{va\}$. 
\end{cor} 

\noindent
Proposition~\ref{prop:finiteness of excu(A)} below partly extends Lemma~\ref{lem:excu, r=1}.

\subsection*{Spectral extensions and ultimate exceptional values}
{\it In this subsection $K$ is $\d$-valued of $H$-type with small derivation, asymptotic integration, and $\I(K)\subseteq K^\dagger$.}\/ Also $A\in K[\der]^{\ne}$ has order $r$ and $\gamma$ ranges over $\Gamma$. 

{\sloppy
Suppose~$\Omega$ is equipped with a spectral extension $v$ of the valuation of $K$. Let~${g\in K^\times}$ with $vg=\gamma$.
The Newton weight of $A_\lambda g\in K[\der]$ does not change in passing from~$K$ to $\Omega$, since
$\Psi$ is cofinal in $\Psi_\Omega$ by Lemma~\ref{lem:v(ex(Q))}; see [ADH, 11.1]. 
Thus
$$\nwt_{A_\lambda}(\gamma)\ =\ \nwt(A_\lambda g)\ =\ \nwt\!\big(Ag\ex(\lambda)\big)\ =\ \nwt_A\!\big(v(g\ex(\lambda)\big)\ =\ \nwt_A\!\big(\gamma+v(\ex(\lambda)\big)\big).$$
In particular, using $\Gamma_{\Omega}=\Gamma\oplus v\big(\!\ex(\Lambda)\big)$,
\begin{equation}\label{eq:excevOmega}
\exc^{\ev}_\Omega(A)\ =\ \bigcup_\lambda \, \exc^{\ev}(A_\lambda)+v\big(\!\ex(\lambda)\big)\qquad\text{(a disjoint union)}.
\end{equation} 
Thus $\exc^{\operatorname{u}}(A)=\pi\big(\exc^{\ev}_\Omega(A)\big)$ where $\pi\colon \Gamma_{\Omega}\to \Gamma$ is given by
~$\pi\big({\gamma+v\big(\!\ex(\lambda)}\big)\big)=\gamma$. }

\begin{lemma}\label{lem:finiteness of excu(A)}
We have $\dim_C \ker_{\Univ} A \leq \sum_\lambda\abs{\exc^{\ev}(A_\lambda)}$, and
$$\dim_C \ker_{\Univ} A\ =\  \sum_\lambda\,\abs{\exc^{\ev}(A_\lambda)}
\ \Longleftrightarrow\  v(\ker^{\neq} A_\lambda)\ =\  \exc^{\ev}(A_\lambda)\text{ for all $\lambda$.}$$
Moreover, if  $\dim_C \ker_{\Univ} A =  \sum_\lambda\,\abs{\exc^{\ev}(A_\lambda)}$, then $v_{\g}(\ker_{\Univ}^{\neq} A) = \exc^{\operatorname{u}}(A)$.
\end{lemma}
\begin{proof}
Clearly, $\dim_C \ker_{\Univ} A \leq \dim_C \ker_\Omega A$.
Equip~$\Omega$ with a spectral extension of the valuation of $K$. 
Then $\dim_C\ker_\Omega A=\abs{v(\ker_\Omega^{\neq}A)}$ and $v(\ker_\Omega^{\neq} A)\subseteq \exc^{\ev}_\Omega(A)$ by~[ADH, 5.6.6(i)] and [ADH, p.~481], respectively, applied to $\Omega$ in the role of~$K$.
Also  $|\exc^{\ev}_{\Omega}(A)|=\sum_\lambda|\exc^{\ev}(A_\lambda)|$ (a sum of cardinals) by the remarks preceding the lemma. This yields the first claim of the lemma. 

Note that~$v(\ker^{\neq} A_\lambda)\subseteq\exc^{\ev}(A_\lambda)$ for all $\lambda$.
Hence from~\eqref{eq:excevOmega} and 
$$v(\ker_{\Univ}^{\neq} A)\ =\ \bigcup_\lambda v(\ker^{\neq}A_\lambda)+v\big(\!\ex(\lambda)\big)\qquad\text{(a disjoint union)}$$
we obtain: 
$$v(\ker_{\Univ}^{\neq} A)=\exc^{\ev}_\Omega(A)\quad\Longleftrightarrow\quad v(\ker^{\neq} A_\lambda)\ =\  \exc^{\ev}(A_\lambda) \text{ for all $\lambda$.}$$
As a valued vector space over $C$,  $\Omega$ is a Hahn space, so its finite-dimensional subspace~$\ker_{\Univ} A$ is as well. Hence
 $\abs{v(\ker_{\Univ}^{\neq} A)}=\dim_C \ker_{\Univ} A$ by~[ADH, 2.3.13], 
and $$v(\ker_{\Univ}^{\neq} A)\ \subseteq\ v(\ker_\Omega^{\neq} A)\ \subseteq\ \exc^{\ev}_\Omega(A), \qquad
|\exc^{\ev}_{\Omega}(A)|\ =\ \sum_\lambda|\exc^{\ev}(A_\lambda)|.$$  
This yields the displayed equivalence. 

Suppose $\dim_C \ker_{\Univ} A= \sum_\lambda\abs{\exc^{\ev}(A_\lambda)}$;
we need to show $v_{\g}(\ker_{\Univ}^{\neq} A)= \exc^{\operatorname{u}}(A)$.
We have~$\pi\big(\exc^{\ev}_\Omega(A)\big)=\exc^{\operatorname{u}}(A)$ for the above projection map $\pi\colon \Gamma_{\Omega} \to \Gamma$, so it is enough to show
$\pi\big(v(\ker_{\Univ}^{\neq} A)\big)=v_{\g}(\ker_{\Univ}^{\ne} A)$.
For that, note that for $\mathcal{B}\subseteq K^\times \ex(\Lambda)$ in Corollary~\ref{cor:spectral valuation basis} we have
$$\pi\big(v(\ker_{\Univ}^{\neq} A)\big)\ =\ \pi(v\mathcal B)\ =\ v_{\g}(\mathcal B)\ =\ v_{\g}(\ker_{\Univ}^{\ne} A),$$
using for the last equality the details in the proof of Corollary~\ref{cor:spectral valuation basis}.
\end{proof}

\begin{prop}\label{prop:finiteness of excu(A)}
Suppose $K$ is $\upo$-free. 
Then  $\nwt_{A_\lambda}(\gamma)=0$ for all but finitely many  pairs $(\gamma,\lambda)$  and  
$$\abs{\exc^{\operatorname{u}}(A)}\ \leq\ 
\sum_\lambda\, \abs{\exc^{\ev}(A_\lambda)}\ =\ 
\sum_{\gamma,\lambda} \nwt_{A_\lambda}(\gamma)\ \leq\ r.$$
If $\dim_C \ker_{\Univ}A =r$, then  $\sum_\lambda\, \abs{\exc^{\ev}(A_\lambda)}=r$ and $v_{\g}(\ker_{\Univ}^{\neq} A) = \exc^{\operatorname{u}}(A)$.
\end{prop}
\begin{proof}
Equip $\Omega$ with a spectral extension $v$ of the valuation of $K$. 
Then $\Omega$ is $\upo$-free, so $\sum_\lambda|\exc^{\ev}(A_\lambda)|=|\exc^{\ev}_{\Omega}(A)|\le r$ by the remarks preceding Lemma~\ref{lem:finiteness of excu(A)}  and Corollary~\ref{cor:sum of nwts} applied to $\Omega$ in place of $K$. These remarks also give 
$\nwt_{A_\lambda}(\gamma)=0$ for all but finitely many pairs $(\gamma,\lambda)$, and so
$$\sum_{\gamma,\lambda} \nwt_{A_\lambda}(\gamma)\ =\ \sum_{\gamma,\lambda}  \nwt_A\!\big(\gamma+v(\ex(\lambda)\big)\  =\  \abs{\exc^{\ev}_\Omega(A)}\ \leq\ r.$$
Corollary~\ref{cor:sum of nwts} applied to $A_\lambda$ in place of $A$ yields $\abs{\exc^{\ev}(A_\lambda)}=\sum_\gamma \nwt_{A_\lambda}(\gamma)$ and so~$\sum_\lambda\, \abs{\exc^{\ev}(A_\lambda)} = \sum_{\gamma,\lambda} \nwt_{A_\lambda}(\gamma)$. This proves the first part (including  the display). The rest follows from this and Lemma~\ref{lem:finiteness of excu(A)}.
 \end{proof} 

\noindent
In the next lemma (to be used in the proof of Proposition~\ref{prop:stability of excu, real}), % as well as in Corollary~\ref{cor:excev cap GammaOmega},  
$L$ is a $\d$-valued $H$-asymptotic extension of $K$  with 
algebraically closed constant field and asymptotic integration (so $L$ has small derivation), such that $L^\dagger$ is divisible, $L^\dagger\cap K=K^\dagger$, and $\I(L)\subseteq L^\dagger$ .  We also fix there a complement~$\Lambda_L$ of the $\Q$-linear subspace  $L^\dagger$ of $L$ with
$\Lambda\subseteq\Lambda_L$. Let $\Univ_L=L\big[\!\ex(\Lambda_L)\big]$ be
the corresponding universal exponential extension  
of $L$ containing $\Univ=K\big[\!\ex(\Lambda)\big]$ as a differential subring, as  
described in the remarks following Corollary~\ref{cor:Univ under d-field ext},
with differential fraction field $\Omega_L$.

\begin{lemma}\label{lem:excev cap GammaOmega} Assume $C_L=C$. Let $\Omega_L$ be equipped with a spectral extension of the valuation of $L$, and take $\Omega$ as a valued subfield of $\Omega_L$; so the valuation of $\Omega$ is a spectral extension of the valuation of $K$. Suppose $\Psi$ is cofinal in $\Psi_L$ or $K$ is $\upl$-free.
Then $\exc^{\ev}_{\Omega_L}(A)\cap \Gamma_\Omega\  =\  \exc^{\ev}_\Omega(A)$.
\end{lemma} 
\begin{proof}
Let $\mu$ range over $\Lambda_L$.
We have 
$$\Gamma_{\Omega_L}\ =\ \Gamma_L \oplus v\big(\!\ex(\Lambda_L)\big), \qquad \Gamma_\Omega\ =\ \Gamma\oplus v\big(\!\ex(\Lambda)\big)$$ by Lemma~\ref{lem:v(ex(Q))} and
$$\exc^{\ev}_{\Omega_L}\ =\ \bigcup_{\mu} \exc^{\ev}_L(A_\mu)+v\big(\!\ex(\mu)\big), \qquad 
\exc^{\ev}_\Omega\ =\  \bigcup_\lambda \exc^{\ev}(A_\lambda)+v\big(\!\ex(\lambda)\big)$$
by \eqref{eq:excevOmega}.  Hence
$$\exc^{\ev}_{\Omega_L}(A)\cap \Gamma_\Omega\ =\  \bigcup_{\lambda} \big( \exc^{\ev}_L(A_\lambda)\cap\Gamma \big)+v\big(\!\ex(\lambda)\big)\ =\  \bigcup_{\lambda}  \exc^{\ev}(A_\lambda) +v\big(\!\ex(\lambda)\big)\  =\  \exc^{\ev}_\Omega(A),$$
where we used the injectivity of $\mu\mapsto v\big(\!\ex(\mu)\big)$ 
for the first equality and  Remark~\ref{rem:excu for different Q, 2}  for the second.
\end{proof}

\subsection*{The real case}
{\it In this subsection $H$ is a real closed $H$-field with small derivation, asymptotic integration, and $H^\dagger=H$;
also $K=H[\imag]$, $\imag^2=-1$, for our valued
differential field $K$. We also assume $\I(H)\imag\subseteq K^\dagger$.}\/
Then $K$ is $\d$-valued of $H$-type with small derivation, asymptotic integration, $K^\dagger=H+\I(H)\imag$, 
and $\I(K)\subseteq K^\dagger$. Note that $H$ and thus $K$ is $\upl$-free by [ADH, remark after 11.6.2, and 11.6.8].
Let~$A$ in $K[\der]^{\neq}$ have order $r$ and let  $\gamma$ range over $\Gamma$.

\begin{lemma}\label{lem:LambdaL}
If the real closed $H$-asymptotic extension $F$ of~$H$ has asymptotic integration
and convex valuation ring, then~$L^\dagger\cap K=K^\dagger$ for the  algebraically closed
$H$-asymptotic field extension $L:=F[\imag]$ of $K$. 
\end{lemma}
\begin{proof} Use Corollary~\ref{cor:logderset ext} and earlier remarks in the same subsection. 
\end{proof}

\begin{cor}\label{cor:LambdaL}
The $H$-field $H$ has an $H$-closed  extension $F$ with $C_F=C_H$, and for any such~$F$, the algebraically closed
$\d$-valued field extension $L:=F[\imag]$ of $H$-type of $K$ is $\upo$-free with~$C_L=C$, $\I(L)\subseteq L^\dagger$, and
 $L^\dagger\cap K=K^\dagger$.
\end{cor}
 \begin{proof} Use [ADH, 16.4.1, 9.1.2] to extend $H$ to an $\upo$-free $H$-field with the same constant field as $H$, next use~[ADH, 11.7.23] to pass to its real closure, and then use~[ADH, 14.5.9]  to extend further to an $H$-closed $F$, still with the same constant field as $H$. 
 For any such~$F$, the $\d$-valued field
  $L:=F[\imag]$ of $H$-type is $\upo$-free by~[ADH,  11.7.23] and newtonian by~\eqref{eq:14.5.7}.  Hence~$\I(L)\subseteq L^\dagger$
  by Lemma~\ref{lem:ADH 14.2.5},
  and~$L^\dagger\cap K=K^\dagger$
  by Lemma~\ref{lem:LambdaL}.
 \end{proof}

\noindent
This leads to a variant of Proposition~\ref{prop:finiteness of excu(A)}:

\begin{prop}\label{prop:finiteness of excu(A), real}
The conclusion of Proposition~\ref{prop:finiteness of excu(A)} holds. 
%In particular $$  \dim_C\ker_{\Univ}A=r\ \Longrightarrow\ A \text{ is terminal}.$$
\end{prop}
\begin{proof} 
Corollary~\ref{cor:LambdaL} gives  an $H$-closed extension $F$ of~$H$ with $C_F=C_H$, so~$L:=F[\imag]$ is $\upo$-free,  $C_L=C$, $\I(L)\subseteq L^\dagger$, and $L^\dagger\cap K=K^\dagger$.   Take a complement~$\Lambda_L\supseteq \Lambda$ of the  subspace~$L^\dagger$ of the $\Q$-linear space $L$. By Remark~\ref{rem:excu for different Q, 2} we have  $\exc^{\ev}(A_\lambda)=\exc^{\ev}_L(A_\lambda)\cap\Gamma$. Hence Proposition~\ref{prop:finiteness of excu(A)} applied to~$K$,~$\Lambda$ replaced by~$L$,~$\Lambda_L$, respectively, and $A$ viewed as element of $L[\der]$, yields $\sum_\lambda\, \abs{\exc^{\ev}(A_\lambda)} \leq r$.  Corollary~\ref{cor:13.7.10} applied to~$A_\lambda$ in place of~$A$ gives~$\abs{\exc^{\ev}(A_\lambda)} = \sum_\gamma \nwt_{A_\lambda}(\gamma)$.  This yields  the conclusion of Proposition~\ref{prop:finiteness of excu(A)} as in the proof of that proposition. 
\end{proof}

%\noindent
%Let now $F$ be a Liouville closed $H$-field extension of $H$ and  suppose $\I(L)\subseteq L^\dagger$ where~$L:=F[\imag]$. Lemma~\ref{lem:LambdaL} yields $L^\dagger\cap K=K^\dagger$, so $L$ is as described just before Lemma~\ref{lem:excev cap GammaOmega}, and we have a complement~$\Lambda_L\supseteq\Lambda$ of the subspace $L^\dagger$ of the $\Q$-linear space $L$.  Note that if $A$ splits over $K$, then $A$ is terminal by Corollary~\ref{cor:ultimate prod, 2}.  

%\begin{cor}\label{cor:excev cap GammaOmega, real} 
%Suppose $A$ is terminal. Then, with respect to the complement~$\Lambda_L$ of $L^\dagger$ in $L$, the conclusions \text{\rm{(i)--(iv)}} of Corollary~\ref{cor:excev cap GammaOmega} hold.
%\end{cor}
%\begin{proof} By the remarks after Corollary~\ref{cor:excu for different Q} we have $\exc^{\ev}(A_\lambda)\subseteq\exc^{\ev}_L(A_\lambda)$ for all $\lambda$,  and so with $\mu$ ranging over $\Lambda_L$,  Proposition~\ref{prop:finiteness of excu(A), real} applied to $L$ in place of $K$, we have $r=\sum_\lambda\abs{\exc^{\ev}(A_\lambda)}\leq \sum_\mu \abs{\exc^{\ev}_L(A_\mu)}\leq r$. This yields (i)--(iv).
%\end{proof}

\newpage 

\part{Normalizing Holes and Slots}\label{part:normalization}

\medskip

\noindent
In this introduction $K$ is an $H$-asymptotic field with small derivation and rational asymptotic integration.
In Section~\ref{sec:holes} we introduce  {\it holes}\/ in $K$: A {\it hole in $K$\/} is a triple $(P,\fm,\hat a)$  with $P\in K\{Y\}\setminus K$, $\fm\in K^\times$, and $\hat a\in \hat{K}\setminus K$ for some immediate asymptotic extension
$\hat{K}$ of $K$, such that $\hat a \prec \fm$ and $P(\hat a)=0$.
The main goal of Chapter~\ref{part:normalization}  is a first normalization theorem, namely Theorem~\ref{mainthm},  that allows us to transform
under reasonable conditions a hole $(P,\fm,\hat a)$ in $K$ into a ``normal'' hole; this helps to pin down the location of
$\hat a$ relative to $K$. 
The notion of~$(P,\fm,\hat a)$ being {\em normal\/} involves the linear part of
the differential polynomial~$P_{\times\fm}$, in particular the {\it span} of this linear part. We introduce the span in the preliminary Section~\ref{sec:span}.
In Section~\ref{sec:isolated} we  study {\it isolated}\/ holes~$(P,\fm,\hat a)$ in~$K$, which under reasonable conditions ensure the uniqueness of the 
isomorphism type of $K\<\hat a\>$ as a
valued differential field over $K$; see Proposition~\ref{prop:2.12 isolated}. In Section~\ref{sec:holes of c=(1,1,1)} we   focus on holes~$(P,\fm,\hat a)$ in~$K$ where~$\order P=\deg P=1$. 
For technical reasons we actually work in Chapter~\ref{part:normalization} also with {\em slots\/} in $K$, which are a bit more general than holes in $K$. 

\medskip\noindent
First some notational conventions. Let $\Gamma$ be an ordered abelian group. For $\gamma, \delta\in \Gamma$ with 
$\gamma\ne 0$ the expression ``$\delta=o(\gamma)$'' means ``$n|\delta|< |\gamma|$ for all $n\ge 1$'' according to~[ADH, 2.4], but here we find it convenient to extend this to $\gamma=0$, in which case~``$\delta=o(\gamma)$" means ``$\delta=0$".  
Suppose $\Gamma=v(E^\times)$ is the value group
of a valued field~$E$ and $\fm\in E^\times$. Then we denote the archimedean class
$[v\fm ]\subseteq \Gamma$ of $v\fm\in \Gamma$ by just $[\fm]$. Suppose $\fm \nasymp 1$. Then we have a
proper convex subgroup\label{p:Delta}
$$\Delta(\fm)\ :=\ \big\{\gamma\in \Gamma:\, \gamma=o(v\fm)\big\}\ =\ \big\{\gamma\in \Gamma:\, [\gamma]<[\fm]\big\},$$ 
of~$\Gamma$. 
If  $\fm\asymp_{\Delta(\fm)}\fn\in E$, then~$0\ne \fn \nasymp 1$ and $\Delta(\fm)=\Delta(\fn)$. 
In particular,  if~$\fm\asymp\fn\in E$, then~$0\ne\fn \nasymp 1$ and $\Delta(\fm)=\Delta(\fn)$. Note that for $f,g\in E$ the meaning of ``$f\preceq_{\Delta(\fm)} g$'' does not change in passing to a valued field extension of~$E$, although~$\Delta(\fm)$ can increase as a subgroup of the value group of the extension.  

\section{The Span of a Linear Differential Operator} \label{sec:span}

\noindent
{\em In this section $K$ is a valued differential field with small derivation and   $\Gamma:= v(K^\times)$.  We let~$a$,~$b$, sometimes subscripted, range over $K$, and $\fm$,~$\fn$ over $K^\times$}.  Consider a linear differential operator
$$A\ =\ a_0+a_1\der+\cdots+a_r\der^r\in K[\der],\qquad   a_r\neq 0.$$
The ``span of $A$" defined below is unrelated to the linear span of a subset of a vector space. We shall use the quantities $\dwm(A)$ and $\dwt(A)$ defined in [ADH,  5.6].  We now introduce a measure $\fv(A)$ for the ``lopsidedness'' of $A$ as follows:\label{p:span}
$$\fv(A)\ :=\ a_r/a_m \in\ K^\times\qquad\text{where $m:=\dwt(A)$.}$$  
So $a_r\asymp \fv(A)A$ and $\fv(A)\preceq 1$, with
$$\fv(A)\asymp 1\quad\Longleftrightarrow\quad\dwt(A)=r\quad\Longleftrightarrow\quad\fv(A)=1.$$
Also note that $\fv(aA)=\fv(A)$ for $a\neq 0$. Moreover,  
$$\fv(A_{\ltimes\fn})A_{\ltimes\fn}\ \asymp\  a_r\ \asymp\ \fv(A)A$$ 
since $A_{\ltimes\fn}=a_r\der^r+\text{lower order terms in $\der$}$. 

\begin{example}
$\fv(a+\der)=1$ if  $a\preceq 1$, and $\fv(a+\der)=1/a$ if $a\succ 1$.
\end{example}

\noindent
We call $\fv(A)$ the {\bf span} of $A$.\index{span}\index{linear differential operator!span}
We are mainly interested in the valuation of $\fv(A)$. 
This is related to the gaussian valuation $v(A)$ of $A$: if~$A$ is monic, then $v\big(\fv(A)\big)=-v(A)$. 
An important property of the span of $A$ is that its valuation is not affected by small additive perturbations of $A$: 

\begin{lemma}\label{lem:fv of perturbed op}
Suppose $B\in K[\der]$, $\order(B)\leq r$ and $B\prec \fv(A) A$. Then: \begin{enumerate}
\item[\textup{(i)}] $A+B\sim A$, $\dwm(A+B)=\dwm(A)$, and $\dwt(A+B)=\dwt(A)$;
\item[\textup{(ii)}] $\order(A+B)=r$ and $\fv(A+B) \sim \fv(A)$. 
\end{enumerate}
\end{lemma}
\begin{proof}
From $B\prec \fv(A) A$ and $\fv(A)\preceq 1$ we obtain $B\prec A$, and thus (i).
Set $m:=\dwt(A)$, let~$i$ range over $\{0,\dots,r\}$, and let $B=b_0+b_1\der+\cdots+b_r\der^r$. 
Then $a_i\preceq a_m$ and $b_i\prec \fv(A)A\asymp a_r \preceq a_m$.
Therefore, if $a_i\asymp a_m$, then $a_i+b_i\sim a_i$, and if $a_i\prec a_m$, then $a_i+b_i\prec a_m$. Hence 
$\fv(A+B)=(a_r+b_r)/(a_m+b_m) \sim a_r/a_m=\fv(A)$.
\end{proof}

\noindent
For $b\neq 0$, the valuation of $\fv(Ab)$ only depends on $A$ and $vb$; it is enough to check this for~$b\asymp 1$. More generally:

\begin{lemma}\label{fvmult}
Let $B\in K[\der]^{\neq}$ and $b\asymp B$. Then $\fv(AB)\asymp\fv(Ab)\fv(B)$.
\end{lemma}
\begin{proof}
Let $B=b_0+b_1\der+\cdots+b_s\der^s$, $b_s\neq 0$. Then
$$AB\ =\ a_rb_s\der^{r+s}+\text{lower order terms in}\ \der,$$ so by [ADH, 5.6.1(ii)] for $\gamma=0$:
\alignqed{v\big(\fv(AB)\big)\ &=\ v(a_rb_s)-v(AB)\ =\ v(a_rb_s) - v(Ab)\\  &=\ v(a_rb)-v(Ab)+v(b_s)-v(B) \\ \ &=\ v\big(\fv(Ab)\fv(B)\big).} 
\end{proof}

\begin{cor}\label{cor:111} 
Let $B\in K[\der]^{\neq}$. If $\fv(AB)=1$, then $\fv(A)=\fv(B)=1$. The converse holds if $B$ is monic. 
\end{cor}

\noindent
This is clear from from Lemma~\ref{fvmult}, and in turn gives:

\begin{cor}\label{corAfv1}
Suppose $A=a(\der-b_1)\cdots(\der-b_r)$. Then
$$\fv(A) =  1 \quad\Longleftrightarrow\quad b_1,\dots, b_r\preceq 1. $$
\end{cor}

\begin{remark}
Suppose $K=C(\!(t)\!)$ with the $t$-adic valuation and derivation $\der=t\frac{d}{dt}$. In the literature, $A$ is called {\it regular singular}\/ if $\fv(A)=1$, and {\it irregular singular}\/ if~$\fv(A)\prec 1$; see \cite[Definition~3.14]{vdPS}.
\end{remark}

{\sloppy
\begin{lemma} 
Let $B\in K[\der]^{\neq}$.  Then $\fv(AB)\preceq\fv(B)$, and if  $B$ is monic, then~$\fv(AB)\preceq\fv(A)$.
\end{lemma}}
\begin{proof}
Lemma~\ref{fvmult} and $\fv(Ab)\preceq 1$ for $b\neq 0$  yields $\fv(AB)\preceq\fv(B)$.
Suppose~$B$ is monic, so $v(B)\leq 0$. To show $\fv(AB)\preceq\fv(A)$ we  arrange that $A$
is also monic. Then~$AB$ is monic, and $\fv(AB)\preceq\fv(A)$ is equivalent to $v(AB)\leq v(A)$. Now
$$v(AB)\ =\ v_{AB}(0)\  =\  v_A\big(v_B(0)\big)\ =\  v_A\big(v(B)\big)\  \leq\  v_A(0)\  =\  v(A)$$
by [ADH, 4.5.1(iii), 5.6.1(ii)].
\end{proof}

\begin{cor}\label{cor:bound on linear factors}
If $A=a(\der-b_1)\cdots(\der-b_r)$, then  
$b_1,\dots, b_r\ \preceq\ \fv(A)^{-1}.$
\end{cor}

\noindent
Let $\Delta$ be a convex subgroup of $\Gamma$, let $\dot{\mathcal O}$ be the valuation ring of the coarsening~$v_\Delta$ of the valuation~$v$ of $K$ by $\Delta$, with maximal ideal $\dot{\smallo}$,
and $\dot K=\dot{\mathcal O}/\dot{\smallo}$ be the
valued differential residue field of   $v_\Delta$.
The residue morphism $\dot{\mathcal O}\to\dot K$ extends to the ring morphism $\dot{\mathcal O}[\der]\to\dot K[\der]$ with $\der\mapsto\der$. If $A\in\dot{\mathcal O}[\der]$ and $\dot A\neq 0$,
then $\dwm(\dot A)=\dwm(A)$ and $\dwt(\dot A)=\dwt(A)$.
We set $\fv:=\fv(A)$.

\begin{lemma}\label{lem:dotfv}
If $A\in\dot{\mathcal O}[\der]$ and $\order(\dot A)=r$, then $\fv(\dot A)=\dot \fv$.
\end{lemma}

\subsection*{Behavior of the span under twisting} 
Recall that $o(\gamma):= 0\in \Gamma$ for $\gamma=0\in \Gamma$. 
With this convention, here is a consequence of [ADH, 6.1.3]: 

\begin{lemma}\label{lem:6.1.3 consequ} 
Let $B\in K[\der]^{\neq}$. Then $v(AB)=v(A)+v(B)+o\big(v(B)\big)$.
\end{lemma}
\begin{proof}
Take $b$ with $b\asymp B$. Then 
$$v(AB)=v_{AB}(0)=v_A\big(v_B(0)\big)=v_A(vb)=v(Ab)$$
by~[ADH, 5.6.1(ii)]. Moreover, $v(Ab)=v(A)+vb+o(vb)$, by [ADH, 6.1.3].
\end{proof}

\noindent
We have $\fv(A_{\ltimes\fn})=\fv(A\fn)$, so $v(A_{\ltimes\fn})=v(A)+o(v\fn)$
by Lemma~\ref{lem:6.1.3 consequ}. Moreover:

\begin{lemma}\label{lem:An} 
$v\big(\fv(A\fn)\big)= v\big(\fv(A)\big) + o(v\fn)$.
\end{lemma}
\begin{proof}
Replacing $A$ by $a_r^{-1} A$ we arrange $A$ is monic, so $A_{\ltimes\fn}$ is monic, and thus
$$v\big(\fv(A\fn)\big)\ =\ v\big(\fv(A_{\ltimes\fn})\big)=-v(A_{\ltimes\fn})=-v(A)+o(v\fn)=v\big(\fv(A)\big)+o(v\fn)$$
by remarks preceding the lemma.
\end{proof}

\noindent
Recall: we denote the archimedean class
$[v\fn]\subseteq \Gamma$ by $[\fn]$.   Lemma~\ref{lem:An} yields:

\begin{cor}\label{cor:An}
$\big[\fv(A)\big]<[\fn]\ \Longleftrightarrow\ \big[\fv(A\fn)\big]<[\fn]$.
\end{cor}

\noindent
Under suitable conditions on $K$ we can say more about the valuation of 
$\fv(A_{\ltimes\fn})$: Lemma~\ref{lem:Atwist} below. 

\begin{lemma}\label{lem:nepsilon, first part}
Let $\fn^\dagger\succeq 1$ and $\fm_0,\dots, \fm_r\in K^\times$ be such that
$$v(\fm_i) + v(A)\ =\ \min_{i\leq j\leq r} v(a_j)+(j-i)v(\fn^\dagger).$$
Then with $m:=\dwt(A)$ we have
$$\fm_0\ \succeq\ \cdots\ \succeq\ \fm_r\quad\text{and}\quad
(\fn^\dagger)^m\ \preceq\ \fm_0\ \preceq\ (\fn^\dagger)^r.$$
\textup{(}In particular, $[\fm_0] \leq [\fn^\dagger]$, with equality if~$m>0$.\textup{)}
\end{lemma}
\begin{proof}
From $v(\fn^\dagger)\leq 0$ we obtain $v(\fm_0)\leq \cdots\leq v(\fm_r)$.
We have $0\leq v(a_j/a_m)$ for~$j=0,\dots,r$ and so
$$rv(\fn^\dagger)\ \leq\ \min_{0\leq j\leq r}  v(a_j/a_m)+jv(\fn^\dagger)\ =\ v(\fm_0)\ \leq\ mv(\fn^\dagger)$$
as required.
\end{proof}

\begin{lemma}\label{lem:Atwist} 
Suppose $\der(\mathcal O)\subseteq\smallo$. Then
$$\fn^\dagger\preceq 1 \ \Longrightarrow\ v(A_{\ltimes \fn})=v(A),\qquad
\fn^\dagger\succ 1\ \Longrightarrow\ 
\abs{v(A_{\ltimes \fn})-v(A)} \leq -rv(\fn^\dagger).$$
\end{lemma}
\begin{proof}
Let $R:=\Ric A$. Then $v(A_{\ltimes\fn})=v(R_{+\fn^\dagger})$ by [ADH, 5.8.11].
If $\fn^\dagger\preceq 1$, then  $v(R_{+\fn^\dagger})=v(R)$
by [ADH, 4.5.1(i)], hence $v(A_{\ltimes\fn})=v(R)=v(A)$ by [ADH, 5.8.10].
Now suppose $\fn^\dagger\succ 1$. {\it Claim}\/: $v(A_{\ltimes\fn})-v(A) \geq rv(\fn^\dagger)$.
To prove this claim we
replace $A$ by $a^{-1}A$, where $a\asymp A$, to arrange~$A\asymp 1$.
Let~$i$,~$j$ range over~$\{0,\dots,r\}$. We   have 
$R_{+\fn^\dagger} = \sum_i b_i R_i$ where
$$b_i\ =\ \sum_{j\geq i} {j\choose i} a_j R_{j-i}(\fn^\dagger).$$
Take $\fm_i\in K^\times$ as in Lemma~\ref{lem:nepsilon, first part}. 
By Lemma~\ref{Riccatipower+} we have $R_n(\fn^\dagger)\sim (\fn^\dagger)^n$ for all~$n$;
hence $v(b_i)\geq v(\fm_i)$ for all $i$. Thus
$$v(A_{\ltimes\fn})-v(A)\ =\ v(A_{\ltimes\fn})\ =\ v(R_{+\fn^\dagger})\ \geq\ \min_i v(b_i)\ \geq\ v(\fm_0)\ \geq\ rv(\fn^\dagger)$$ by Lemma~\ref{lem:nepsilon, first part}, proving
our claim. 
Applying this claim with $A_{\ltimes \fn}$, $\fn^{-1}$ in place of~$A$,~$\fn$ also yields
$v(A_{\ltimes\fn})-v(A) \leq -rv(\fn^\dagger)$,   thus $\abs{v(A_{\ltimes \fn})-v(A)} \leq -rv(\fn^\dagger)$.
\end{proof}

\begin{remark} Suppose that $\der(\mathcal O)\subseteq\smallo$ and $\fn^\dagger \succ 1$. Then Lemma~\ref{lem:Atwist} improves on Lem\-ma~\ref{lem:An}, since $v(\fn^\dagger)=o(v\fn)$ by [ADH, 6.4.1(iii)].  
\end{remark}

\begin{lemma}\label{twistprec} Suppose $\der(\mathcal O)\subseteq\smallo$ and $\fn^\dagger\preceq\fv(A)^{-1}$. 
Let $B\in K[\der]$ and $s\in\N$  be such that $\order(B)\leq s$ and~$B\prec \fv(A)^{s+1}A$. Then $B_{\ltimes\fn} \prec \fv(A_{\ltimes\fn})A_{\ltimes\fn}$. 
\end{lemma}
\begin{proof}
We may assume $B\neq 0$ and $s=\order(B)$. It suffices to show $B_{\ltimes\fn} \prec \fv(A)A$.
If $\fn^\dagger\preceq 1$, then Lemma~\ref{lem:Atwist} applied to $B$ in place of $A$ yields
$B_{\ltimes\fn} \asymp B \prec\fv(A)A$.
Suppose $\fn^\dagger\succ 1$. Then Lemma~\ref{lem:Atwist} gives
$\abs{v(B_{\ltimes \fn})-v(B)} \leq -sv(\fn^\dagger)\leq sv(\fv(A))$ and hence
$B_{\ltimes\fn} \preceq \fv(A)^{-s} B \prec \fv(A) A$.
\end{proof}

\noindent
If   $\der(\mathcal{O})\subseteq \smallo$, then we have functions $\dwm_A, \dwt_A\colon \Gamma\to \N$ as defined in~[ADH, 5.6]. 
Combining Lemmas~\ref{lem:fv of perturbed op} and~\ref{twistprec} yields a variant of~[ADH, 6.1.7]: 

{\sloppy
\begin{cor}\label{cor:Atwist} Suppose $\der(\mathcal O)\subseteq\smallo$ and $\fn^\dagger\preceq\fv(A)^{-1}$. 
Let $B\in K[\der]$ be such that~$\order(B)\leq r$ and~$B\prec \fv(A)^{r+1}A$. 
Then $\dwm_{A+B}(v\fn)=\dwm_A(v\fn)$ and~$\dwt_{A+B}(v\fn)=\dwt_A(v\fn)$. In particular, $$v\fn\in\exc(A+B)\ \Longleftrightarrow\ 
v\fn \in \exc(A).$$
\end{cor}}

\subsection*{About $A(\fn^q)$ and $A\fn^q$} 
Suppose $\fm^l=\pm \fn^k$ where $k,l\in \Z$, $l\ne0$. Then 
$\fm^\dagger=q\fn^\dagger$ with $q=k/l\in \Q$. In particular, if  $K$ is real closed or algebraically closed, then for any $\fn$ and $q\in \Q$ we have $\fm^\dagger=q\fn^\dagger$ for some $\fm$. 

\medskip
\noindent
{\em Below in this subsection~$K$ is $\d$-valued and $\fn$ is such that for all 
$q\in \Q^{>}$ we are given an element of $K^\times$, denoted by~$\fn^q$ for suggestiveness, with $(\fn^q)^\dagger=q\fn^\dagger$.}

\medskip
\noindent
Let $q\in \Q^{>}$; then $v(\fn^q)=qv(\fn)$: to see this we may arrange that~$K$ is algebraically closed 
by [ADH, 10.1.23], and hence contains an~$\fm$ such that $v\fm=q\,v\fn$ and~$\fm^\dagger=q\fn^\dagger=(\fn^q)^\dagger$, and thus $v(\fn^q)=v\fm=q\,v\fn$.   

\begin{lemma}\label{qlA} Suppose $\fn^\dagger\succeq 1$. Then for all but finitely many $q\in \Q^{>}$,
$$v\big(A(\fn^q)\big)\ =\ v(\fn^q) + \min_j v(a_j)+jv(\fn^\dagger).$$ 
\end{lemma}
\begin{proof} Let $q\in \Q^{>}$ and take $b_0,\dots,b_r\in K$ with $A\fn^q=b_0+b_1\der+\cdots+b_r\der^r$.
Then 
$$b_0\ =\ A(\fn^q)\ =\ \fn^q \big(a_0 R_0(q\fn^\dagger)+a_1 R_1(q\fn^\dagger)+\cdots+
a_r R_r(q\fn^\dagger)\big).$$
Let $i$,~$j$ range over $\{0,\dots,r\}$.  
By Lemma~\ref{Riccatipower+}, $R_i(q\fn^\dagger) \sim q^i (\fn^\dagger)^i$ for all $i$. Take~$\fm$ (independent of $q$) such that
$v(\fm)=\min_j v(a_j)+jv(\fn^\dagger)$, and let $I$ be the nonempty set of
$i$ with $\fm\asymp a_i(\fn^\dagger)^i$. For $i\in I$ we take $c_i\in C^\times$ such that $a_i(\fn^\dagger)^i \sim c_i\fm$, and set $R:=\sum_{i\in I} c_i Y^i\in C[Y]^{\neq}$.
Therefore, if $R(q)\ne 0$, then 
$$\sum_{i\in I} a_iR_i(q\fn^\dagger)\ \sim\ \fm R(q).$$
Assume $R(q)\ne 0$ in what follows. Then
$$\sum_{i=0}^r a_iR_i(q\fn^\dagger)\ \sim\ \sum_{i\in I} a_iR_i(q\fn^\dagger)\ \sim\ \fm R(q)\ \asymp\ \fm,$$
hence $b_0\asymp \fm\fn^q$, in particular, $b_0\ne 0$.
\end{proof}

\begin{lemma}\label{lem:nepsilon} Assume $\fn^\dagger\succeq 1$ and $[\fv]<[\fn]$ for
$\fv:=\fv(A)$. Then 
$\big[\fv(A\fn^q)\big]<[\fn]$ for all $q\in \Q^>$, and for all but finitely many $q\in \Q^{>}$ 
we have 
$\fv(A\fn^q)\preceq \fv$, and thus~$[\fv]\le \big[\fv(A\fn^q)\big]$.
\end{lemma}
\begin{proof} Let $q\in \Q^>$. Then $[\fv]<[\fn]=[\fn^q]$, so $[\fv(A\fn^q)]<[\fn^q]=[\fn]$ by Corollary~\ref{cor:An}.
To show the second part, let $m=\dwt(A)$. 
Replacing $A$ by $a_m^{-1} A$ we arrange $a_m=1$,  so  $a_r = \fv$, $A\asymp 1$.
Take $b_0,\dots,b_r$ with $A\fn^q=b_0+b_1\der+\cdots+b_r\der^r$.
As in the proof of Lemma~\ref{qlA} we obtain an $\fm$ and a polynomial $R(Y)\in C[Y]^{\ne}$ (both independent of $q$) 
such that $v(\fm)=\min_j v(a_j)+jv(\fn^\dagger)$, and  $b_0\asymp \fm\fn^q$ if~$R(q)\ne 0$. 
Assume $R(q)\ne 0$ in what follows; we show that then $\fv(A\fn^q)\preceq \fv$.  For~$n:= \dwt(A\fn^q)$, 
$$ b_0\fv(A\fn^q)\ \preceq\ b_n\fv(A\fn^q)\ =\ b_r\ =\ \fn^q\fv,$$
hence
$\fv(A\fn^q)\preceq \fv/\fm$.
It remains to note that  
$\fm\succeq a_m(\fn^\dagger)^m=(\fn^\dagger)^m\succeq 1$.  
\end{proof}
 
\begin{lemma}\label{lem:nepsilon, refined} Assume $\fn^\dagger\succeq 1$ and $\fm$ satisfies 
$$v\fm+v(A)\ =\ \min_{0\leq j\leq r}  v(a_j)+jv(\fn^\dagger).$$
Then  $[\fm]\leq[\fn^\dagger]$, with equality if $\dwt(A)>0$, and for all but finitely many $q\in\Q^>$, 
$$A\fn^q\ \asymp\ \fm\, \fn^q\, A, \qquad \fv(A)/\fv(A\fn^q)\ \asymp\ \fm.$$
\end{lemma}
\begin{proof}
Replacing $A$ by $a_m^{-1}A$ where $m=\dwt(A)$ we arrange $a_m=1$, so $a_r=\fv:=\fv(A)$ and~$A\asymp 1$.
Let~$i$,~$j$ range over $\{0,\dots,r\}$. Let $q\in\Q^>$, and
take~$b_i\in K$ such that $A\fn^q = \sum_i b_i \der^i$.
By [ADH, (5.1.3)] we have
$$b_i\ =\   \frac{1}{i!} A^{(i)}(\fn^q)\  =\ \fn^q  \frac{1}{i!}\Ric(A^{(i)})(q\fn^\dagger)\ =\ \fn^q \sum_{j\geq i} {j\choose i} a_j R_{j-i}(q\fn^\dagger).$$
Take $\fm_i\in K^\times$ as in Lemma~\ref{lem:nepsilon, first part}. 
Then  $\fm_0\asymp\fm$ (so $[\fm]\leq[\fn^\dagger]$, with equality if~$m>0$), and $\fm_r\asymp \fv$.
Lemma~\ref{qlA} applied to $A^{(i)}/i!$ instead of $A$ gives that for all but finitely $q\in\Q^>$ we have
$b_i \asymp \fm_i \fn^q$ for all~$i$. Assume that~$q\in \Q^{>}$ has this property.
From  $v(\fm)=v(\fm_0)\leq \cdots\leq v(\fm_r)=v(\fv)$ we obtain
$$v(\fm)+qv(\fn)\ =\ v(b_0)\ \leq\ \cdots\ \leq\ v(b_r)\ =\ v(\fv)+q\,v(\fn).$$
With $n=\dwt(A\fn^q)$ this gives $v(b_0)=\cdots=v(b_n)=v(A\fn^q)$. Thus
$$\fv(A\fn^q)\ =\ b_r/b_n\ \asymp\ b_r/b_0\ \asymp\ 
(\fn^q \fv)/(\fn^q\fm)\ =\ \fv/\fm$$
as claimed.
\end{proof}

\noindent
The next lemma  (not used later) is a more precise version  of Lemma~\ref{lem:nepsilon, refined}, but 
with an additional hypothesis on $\fn^\dagger$:

\begin{lemma}[${}^*$]\label{lem:nepsilon, refined, 2}
Assume  $\fn^\dagger\succ\fv(A)^{-1}$. Then  
$$  A(\fn) \sim   A\fn   \sim   a_r\fn(\fn^\dagger)^r \sim a_r\fn^{(r)}, \qquad  \fv(A\fn)  \sim (\fn^\dagger)^{-r}.$$
\end{lemma}
\begin{proof}
%As in the proof of Lemma~\ref{lem:nepsilon, refined}   arrange  $a_r=\fv:=\fv(A)$ and $A\asymp 1$. 
Let~$i$,~$j$ range over $\{0,\dots,r\}$ and take $b_i\in K$ 
such that $A\fn = \sum_i b_i\der^i$, so~$b_i=\fn\sum_{j\geq i} {j\choose i}a_jR_{j-i}(\fn^\dagger)$.
By Lemma~\ref{Riccatipower}    we have $R_{j-i}(\fn^\dagger) \sim (\fn^\dagger)^{j-i}$  for~${i\leq j}$.
From  $\fn^\dagger\succ\fv^{-1}\succeq 1$ we get for $i\leq j<r$:
$$ a_r (\fn^\dagger)^{r-i}\  \succeq\  a_r \fn^\dagger (\fn^\dagger)^{j-i}\ 
\succ\ a_r \fv^{-1} (\fn^\dagger)^{j-i}\  \succeq\  a_j (\fn^\dagger)^{j-i}.$$
Therefore $b_i\sim \fn {r\choose i} a_r (\fn^\dagger)^{r-i}$, from which the first displayed equivalences follow.
Now~$\operatorname{dwt}(A\fn)=0$ and so $\fv(A\fn) = b_r/b_0 = (\fn a_r)/A(\fn)\sim (\fn^\dagger)^{-r}$ as claimed.
 \end{proof}

\noindent
Let $\fv\in K^\times$ with $\fv \nasymp 1$; so we have the proper convex subgroup of $\Gamma$ given by
$$\Delta(\fv)\ =\ \big\{\gamma\in \Gamma:\, \gamma=o(v\fv)\big\}\ =\ \big\{\gamma\in \Gamma:\, [\gamma]<[\fv]\big\}.$$ 
If $K$ is $H$-asymptotic, then we also have the convex subgroup
$$\Delta\ =\ \big\{\gamma\in\Gamma:\, \gamma^\dagger>v(\fv^\dagger) \big\}$$
of $\Gamma$ with $\Delta\subseteq\Delta(\fv)$.  If $K$ is $H$-asymptotic of  Hardy type (Section~\ref{sec:logder}), then we have ${\Delta = \Delta(\fv)}$,
and hence the relations $\preceq_{\Delta(\fv)}$, $\prec_{\Delta(\fv)}$, $\asymp_{\Delta(\fv)}$ agree 
with~$\preceq_\fv$,~$\prec_{\fv}$,~$\asymp_{\fv}$, respectively, from [ADH, p.~407]. 

\begin{cor}\label{cor:nepsilon} 
Suppose $\fn^\dagger\succeq 1$ and 
$[\fn^\dagger] < [\fv]$ where $\fv:=\fv(A)$ \textup{(}so $0\ne \fv\prec 1)$.  
Let $B\in K[\der]$ and $w\geq r$ be such that
$B \prec_{\Delta(\fv)} \fv^w A$. 
Then for all but finitely many~$q\in\Q^>$ we have $\fw:=\fv(A\fn^q)\asymp_{\Delta(\fv)}\fv$  and
$B\fn^q \prec_{\Delta(\fw)} \fw^w A\fn^q$.
\end{cor}
\begin{proof} The case $B=0$ is trivial, so
assume $B\neq 0$.  Take $\fm$ as in Lemma~\ref{lem:nepsilon, refined}, and take $\fm_B$ likewise with $B$ in place of $A$. By this lemma, $[\fm],[\fm_B]\leq[\fn^\dagger]< [\fv]$, 
hence~$\fm, \fm_B\asymp_{\Delta(\fv)} 1$. 
Moreover, for all but finitely many $q\in\Q^>$ we have $A\fn^q \asymp\fm\fn^q A$,
$B\fn^q \asymp\fm_B\fn^q B$, and $\fv/\fw\asymp\fm$ where $\fw:=\fv(A\fn^q)$; assume that $q\in \Q^{>}$ has these properties. Then 
$B \prec_{\Delta(\fv)} \fv^w A$ yields 
$$B\fn^q\ \asymp\ \fm_B \fn^q B\ 
\prec_{\Delta(\fv)} \fm \fn^q \fv^w  A\ \asymp\ \fv^w A\fn^q.$$
Now $\fm\asymp_{\Delta(\fv)} 1$ gives $\fv\asymp_{\Delta(\fv)} \fw$, hence 
$B\fn^q \prec_{\Delta(\fw)} \fw^w A\fn^q$. 
\end{proof}

\subsection*{The behavior of the span under compositional conjugation}  If $K$ is $H$-asymptotic with asymptotic integration, then $\Psi\cap \Gamma^{>}\ne \emptyset$, but it is convenient not to require ``asymptotic integration'' in some lemmas below. Instead: {\em In this subsection~$K$ is $H$-asymptotic and ungrounded  with $\Psi\cap \Gamma^{>}\ne \emptyset$.}\/ 
We let~$\phi$,~$\fv$  range over~$K^\times$. We say that
$\phi$ is {\em active\/} if $\phi$ is active in $K$. 
Recall from [ADH, pp.~290--292] that $\derdelta$ denotes the derivation $\phi^{-1}\der$ of~$K^\phi$, and that
\begin{equation}\label{eq:Aphi}
A^\phi\ =\ a_r\phi^r\derdelta^r+\text{lower order terms in $\derdelta$.}
\end{equation}

\begin{lemma}\label{lem:v(Aphi)}
Suppose $\fv:=\fv(A)\prec^\flat 1$ and $\phi\preceq 1$ is active. Then 
$$A\ \asymp_{\Delta(\fv)}\ A^\phi,\qquad \fv\ \asymp_{\Delta(\fv)}\ \fv(A^\phi)\ \prec^\flat\ 1, \qquad \fv, \fv(A^\phi)\ \prec_{\phi}^\flat\ 1.
$$
\end{lemma}
\begin{proof} From $\phi^\dagger\prec 1 \preceq \fv^\dagger$ we get $[\phi]<[\fv]$, so $\phi\asymp_{\Delta(\fv)} 1$. Hence $A^\phi \asymp_{\Delta(\fv)} A$
by~[ADH, 11.1.4]. For the rest we can arrange $A\asymp 1$, so $A^\phi\asymp_{\Delta(\fv)} 1$ and $\fv\asymp a_r$. In view of
\eqref{eq:Aphi}
this yields $\fv(A^\phi)\asymp_{\Delta(\fv)} a_r\phi^r\asymp_{\Delta(\fv)} \fv$. 
So $\fv(A^\phi)^\dagger\asymp \fv^\dagger\succeq 1$, which gives $\fv(A^\phi) \prec^{\flat} 1$, and also $\fv, \fv(A^\phi) \prec_{\phi}^\flat 1$.
\end{proof}

\begin{lemma}\label{lem:eventual value of fv} 
If $\nwt(A)=r$, then $\fv(A^\phi)=1$ eventually, and if $\nwt(A)<r$, then $\fv(A^\phi) \prec_\phi^\flat 1$ eventually.
\end{lemma}
\begin{proof}
Clearly, if $\nwt(A)=r$, then $\dwt(A^\phi)=r$ and so $\fv(A^\phi)=1$ eventually.
Suppose $\nwt(A)<r$. To show that $\fv(A^\phi) \prec_\phi^\flat 1$ eventually, we may replace $A$ by~$A^{\phi_0}$ for
suitable active $\phi_0$ and assume that $n:=\nwt(A) = \dwt(A^\phi) = \dwm(A^\phi)$ for all active $\phi\preceq 1$.
Thus $v(A^\phi)=v(A)+nv\phi$ for all active $\phi\preceq 1$ by [ADH, 11.1.11(i)]. Using~\eqref{eq:Aphi} 
we therefore obtain for active $\phi\preceq 1$:
$$\fv(A^\phi)\ \asymp\ a_r\phi^r/a_n\phi^n\ =\ \fv(A)\phi^{r-n}\ \preceq\ \phi^{r-n}\ \preceq\ \phi.$$
Take $x\in K^\times$ with $x\nasymp 1$ and $x'\asymp 1$; then $x\succ 1$, so $x^{-1}\asymp x^\dagger\prec 1$ is active. Hence for active $\phi\preceq x^{-1}$ we have $\phi \prec^\flat_\phi 1$ and thus $\fv(A^\phi)\prec^\flat_\phi 1$. 
\end{proof}

\begin{cor}\label{coruplnwteq}
The following conditions on $K$ are equivalent:
\begin{enumerate}
\item[\textup{(i)}] $K$ is $\upl$-free;
\item[\textup{(ii)}]  $\nwt(B)\leq 1$ for all $B\in K[\der]$ \textup{(}so~$\fv(B^\phi) \prec_\phi^\flat 1$ eventually\textup{)};
\item[\textup{(iii)}] $\nwt(B)\leq 1$ for all $B\in K[\der]$ of order $2$.
\end{enumerate}
\end{cor}
\begin{proof}
The implication (i)~$\Rightarrow$~(ii) follows from [ADH, 13.7.10] and Lemma~\ref{lem:eventual value of fv}, and
(ii)~$\Rightarrow$~(iii) is clear. Suppose  $K$ is not $\upl$-free. Take~$\upl\in K$   such that $\phi^\dagger+\upl\prec\phi$ for all active $\phi$   ([ADH, 11.6.1]);
set~$B:=(\der+\upl)\der=\der^2+\upl\der$. Then for active $\phi$ we have
$B^\phi=\phi^2\big(\derdelta^2+(\phi^\dagger+\upl)\phi^{-1}\derdelta\big)$, so
$\dwt(B^\phi)=2$. Thus~(iii)~$\Rightarrow$~(i). 
\end{proof}

\noindent
Lemma~\ref{lem:v(Aphi)} leads to an  ``eventual'' version of Corollary~\ref{cor:Atwist}:

\begin{lemma}\label{cor:excev stability} Suppose $K$ is $\upl$-free and $B\in K[\der]$ is such that $\order(B)\leq r$ and~$B\prec_{\Delta(\fv)} \fv^{r+1}A$, where $\fv:=\fv(A)\prec^\flat 1$.  Then 
$\exc^{\ev}(A+B) = \exc^{\ev}(A)$.
\end{lemma} 
\begin{proof} By [ADH, 10.1.3, 11.7.18] and Corollary~\ref{cor:13.7.10} we can pass to an extension to arrange that $K$ is
$\upo$-free. Next, by [ADH, 11.7.23] and \eqref{eq:14.0.1} we extend further to arrange that $K$ is algebraically closed and newtonian, and thus $\d$-valued by Lemma~\ref{lem:ADH 14.2.5}. Then $\exc^{\ev}(A)=v(\ker^{\neq} A)$ by Proposition~\ref{kerexc}, and $A$ splits over~$K$ by~\eqref{eq:14.5.3}.  It remains to show that $\exc^{\ev}(A)\subseteq \exc^{\ev}(A+B)$: the reverse inclusion then follows by
interchanging $A$ and $A+B$, using $\fv(A)\sim \fv(A+B)$. 
Let~${\gamma\in\exc^{\ev}(A)}$. Take~$\fn\in\ker^{\neq} A$ with $v\fn=\gamma$.  Then $A\in K[\der](\der-\fn^\dagger)$
by [ADH, 5.1.21] and so~$\fn^\dagger\preceq \fv^{-1}$,
by [ADH, 5.1.22] and  Corollary~\ref{cor:bound on linear factors}. Now
$\exc^{\ev}(A)\subseteq \exc(A)$, so~${\gamma=v\fn\in \exc(A+B)}$ by Corollary~\ref{cor:Atwist}.  Let $\phi\preceq 1$ be active; it remains to show that then~$\gamma\in \exc\big((A+B)^\phi\big)$.  By Lemma~\ref{lem:v(Aphi)}, $A^\phi \asymp_{\Delta(\fv)} A$; 
also   $B^\phi\preceq B$ by~[ADH, 11.1.4].
Lemma~\ref{lem:v(Aphi)} gives $\fv \asymp_{\Delta(\fv)} \fv(A^\phi)$, hence
$B^\phi \prec_{\Delta(\fv)} \fv(A^\phi)^{r+1}A^\phi$. Thus with $K^\phi$, $A^\phi$, $B^\phi$ in the role of $K$, $A$, $B$, the above argument leading to $\gamma\in \exc(A+B)$ gives $\gamma\in \exc(A^\phi+B^\phi)=\exc\big((A+B)^\phi\big)$.
\end{proof} 

\noindent
For $r=1$ we can weaken the hypothesis of  $\upl$-freeness:

\begin{cor}\label{cor:excev stability, r=1}
Suppose $K$ has asymptotic integration,  $r=1$,  and $B\in K[\der]$ of order~$\leq 1$ satisfies  $B\prec_{\Delta(\fv)} \fv^{2}A$, where $\fv:=\fv(A)\prec^\flat 1$.  Then $\exc^{\ev}(A+B) = \exc^{\ev}(A)$. 
\end{cor}
{\sloppy
\begin{proof}
Using Lemma~\ref{lem:achieve I(K) subseteq Kdagger} we replace $K$ by an immediate extension to arrange~$\I(K)\subseteq K^\dagger$.
Then  $\exc^{\ev}(A)=v(\ker^{\neq} A)$ by Lemma~\ref{lem:v(ker)=exc, r=1}. Now argue as in the proof of Lem\-ma~\ref{cor:excev stability}.
\end{proof}}

\noindent
{\em In the next proposition and its corollary $K$ is $\d$-valued  with algebraically closed constant field $C$ and divisible group~$K^\dagger$ of logarithmic derivatives}.
We choose a complement $\Lambda$ of the $\Q$-linear subspace $K^\dagger$ of $K$. Then we have the set $\exc^{\operatorname{u}}(A)$ of ultimate exceptional values of $A$ with respect to $\Lambda$. The following stability result  will be crucial in Section~\ref{sec:ultimate}:

\begin{prop}\label{prop:stability of excu} 
Suppose $K$ is $\upo$-free, $\I(K)\subseteq K^\dagger$, and
$B\in K[\der]$ of order $\le r$ satisfies  $B\prec_{\Delta(\fv)} \fv^{r+1}A$, where $\fv:=\fv(A)\prec^\flat 1$. 
Then $\exc^{\operatorname{u}}(A+B)=\exc^{\operatorname{u}}(A)$.
\end{prop}
\begin{proof}
Let $\Omega$ be the differential fraction field of the universal exponential extension $\Univ=K\big[\!\ex(\Lambda)\big]$ 
of $K$ from Section~\ref{sec:univ exp ext}.
Equip  $\Omega$ with a spectral extension of the valuation of $K$; 
see Section~\ref{sec:valuniv}. Apply Lem\-ma~\ref{cor:excev stability}
to $\Omega$ in place of $K$ to get~$\exc^{\ev}_\Omega(A+B)=\exc^{\ev}_\Omega(A)$. 
Hence $\exc^{\operatorname{u}}(A+B)=\exc^{\operatorname{u}}(A)$ by~\eqref{eq:excevOmega}.
\end{proof}

\noindent
In a similar manner we obtain an analogue of Corollary~\ref{cor:excev stability, r=1}:  

\begin{cor}[${}^*$]\label{cor:stability of excu} 
Suppose $K$ has asymptotic integration, $\I(K)\subseteq K^\dagger$, $r=1$, and~$B\in K[\der]$   satisfies~$\order(B)\leq 1$ and~$B\prec_{\Delta(\fv)} \fv^{2}A$, where $\fv:=\fv(A)\prec^\flat 1$. 
Then $\exc^{\operatorname{u}}(A+B)=\exc^{\operatorname{u}}(A)$.
\end{cor}
\begin{proof}
Let $\Omega$ be as in the proof of Proposition~\ref{prop:stability of excu}.
Then $\Omega$ is ungrounded by Lemma~\ref{lem:v(ex(Q))}, hence $\abs{\exc^{\ev}_\Omega(A)}\leq 1$ and $v(\ker^{\neq}_\Omega A)\subseteq\exc^{\ev}_\Omega(A)$ by [ADH, p.~481]. But~$\dim_C \ker_\Omega A=1$, so
$v(\ker_\Omega^{\neq} A)=\exc^{\ev}_\Omega(A)$. The proof of Lemma~\ref{cor:excev stability} with $\Omega$ in place of $K$ now gives $\exc^{\ev}_\Omega(A+B)=\exc^{\ev}_\Omega(A)$, so $\exc^{\operatorname{u}}(A+B)=\exc^{\operatorname{u}}(A)$ by~\eqref{eq:excevOmega}.
\end{proof}

\noindent
In the ``real''  case we have the following variant of Proposition~\ref{prop:stability of excu}: 

\begin{prop}\label{prop:stability of excu, real}  
Suppose $K=H[\imag]$, $\imag^2=-1$, where $H$ is a real closed $H$-field with asymptotic integration such that
$H^\dagger=H$ and $\I(H)\imag\subseteq K^\dagger$.
Let $B\in K[\der]$ of order~$\le r$ be such that~${B\prec_{\Delta(\fv)} \fv^{r+1}A}$ with $\fv:=\fv(A)\prec^\flat 1$. Let $\Lambda$ be a complement of
the subspace~$K^\dagger$ of the $\Q$-linear space $K$. 
Then $\exc^{\operatorname{u}}(A+B)=\exc^{\operatorname{u}}(A)$, where the ultimate exceptional values are with respect to $\Lambda$. 
\end{prop}

\begin{proof}
Take an $H$-closed extension $F$ of $H$ with $C_F=C_H$ as in Corollary~\ref{cor:LambdaL}. Then the algebraically closed $\d$-valued $H$-asymptotic extension $L:=F[\imag]$ of $K$ is $\upo$-free,  $C_L=C$, $\I(L)\subseteq L^\dagger$, and
$L^\dagger\cap K=K^\dagger$. 
Take  a complement $\Lambda_L\supseteq \Lambda$ of the subspace $L^\dagger$ of the $\Q$-linear space $L$.  Let $\Univ_L=L\big[\!\ex(\Lambda_L)\big]$ be
 the universal exponential extension  
of $L$ from Section~\ref{sec:univ exp ext}; it has the  universal exponential extension $\Univ:=K\big[\!\ex(\Lambda)\big]$ of $K$ as a differential subring. 
Let $\Omega$, $\Omega_L$ be the
differential fraction fields of $\Univ$, $\Univ_L$, respectively, and
equip  $\Omega_L$ with a spectral extension of the valuation of $L$; then the restriction of this valuation to
$\Omega$ is  a spectral extension of the valuation of $K$ (see remarks preceding Lemma~\ref{lem:excev cap GammaOmega}).  
Lem\-ma~\ref{cor:excev stability} applied
to $\Omega_L$ in place of $K$ yields~$\exc^{\ev}_{\Omega_L}(A+B)=\exc^{\ev}_{\Omega_L}(A)$,
hence $\exc^{\ev}_\Omega(A+B)=\exc^{\ev}_\Omega(A)$ by
Lemma~\ref{lem:excev cap GammaOmega} and thus $\exc^{\operatorname{u}}(A+B)=\exc^{\operatorname{u}}(A)$.
\end{proof}

\subsection*{The span of the linear part of a differential polynomial} 
{\it In this subsection~$P\in K\{Y\}^{\neq}$ has order $r$.}\/ Recall  that the {\it linear part}\/ of $P$ is the
differential operator %\index{differential polynomial!linear part}\index{linear part!differential polynomial} 
$$L_P\ :=\ \sum_n \frac{\partial P}{\partial Y^{(n)}}(0)\,\der^n \in K[\der]$$
of order~$\leq r$.
We have $L_{P_{\times\fm}}=L_P\fm$ [ADH, p.~242]; hence items~\ref{lem:An}, \ref{cor:An} and~\ref{lem:Atwist} above yield information about the span of $L_{P_{\times\fm}}$ (provided $L_P\neq 0$). We now want to similarly investigate the span of the linear part 
$$L_{P_{+a}}\ =\ \sum_n \frac{\partial P}{\partial Y^{(n)}}(a)\,\der^n$$  
of the additive conjugate~$P_{+a}$ of $P$ by some $a\prec 1$.
In the next two lemmas we assume $\order(L_P)=r$ \textup{(}in particular, $L_P\neq 0$\textup{)}, 
$\fv(L_P)\prec 1$, and $a\prec 1$, we set
$$L:=L_P,\quad L^+:=L_{P_{+a}},\quad \fv:=\fv(L),$$ 
and set $L_n:=\frac{\partial P}{\partial Y^{(n)}}(0)$ and $L_n^+:=\frac{\partial P}{\partial Y^{(n)}}(a)$, so $L=\sum_n L_n\der^n$, $L^+=\sum_n L_n^+\der^n$.
Recall from~[ADH, 4.2] the decomposition of $P$ into homogeneous parts: $P=\sum_d P_d$ where $P_d=\sum_{\abs{\i}=d} P_{\i}Y^{\i}$;
we set $P_{>1}:=\sum_{d>1} P_d$.\index{decomposition!into homogeneous parts}

\begin{lemma}\label{lem:linear part, new}
Suppose $P_{>1} \prec_{\Delta(\fv)} \fv P_1$ and  $n\le r$. Then 
\begin{enumerate}
\item[$\mathrm{(i)}$] $L_r^+ \sim_{\Delta(\fv)} L_r$, and thus $\order(L^+)=\order(L)=r$; 
\item[$\mathrm{(ii)}$] if $L_n\asymp_{\Delta(\fv)} L$, then $L_n^+ \sim_{\Delta(\fv)} L_n$, and so $v(L_n^+)=v(L_n)$;
\item[$\mathrm{(iii)}$] if $L_n\prec_{\Delta(\fv)} L$, then $L_n^+ \prec_{\Delta(\fv)} L$, and so $v(L_n^+)> v(L)$.
\end{enumerate}
In particular, $L^+\sim_{\Delta(\fv)} L$, $\dwt L^+=\dwt L$, 
and $\fv(L^+)\sim_{\Delta(\fv)}\fv$. 
\end{lemma}
\begin{proof}
Take $Q,R\in K\{Y\}$ with $\deg_{Y^{(n)}} Q\leq 0$ and  $R\in Y^{(n)}K\{Y\}$, such that
$$P\ =\ Q+(L_n+R)Y^{(n)},\qquad\text{so}\qquad  \frac{\partial P}{\partial Y^{(n)}}\  =\ \frac{\partial R}{\partial Y^{(n)}}Y^{(n)}+L_n+R.$$
Now $R\prec_{\Delta(\fv)}  \fv P_1$,
so $\frac{\partial P}{\partial Y^{(n)}}-L_n\prec_{\Delta(\fv)} \fv P_1$. In $K[\der]$ we thus have  
 $$ L_n^+-L_n\ =\ \frac{\partial P}{\partial Y^{(n)}}(a)-L_n\ \prec_{\Delta(\fv)}\   \fv L\ \asymp\ L_r.$$
So $L_n^+-L_n\prec_{\Delta(\fv)} L$ and (taking $r=n$)  
$L_r^+-L_r \prec_{\Delta(\fv)}  L_r$. This yields (i)--(iii). 
\end{proof}

\begin{lemma}\label{lem:linear part, split-normal, new}  
Suppose $P_{>1} \prec_{\Delta(\fv)} \fv^{m+1} P_1$, and let 
$A,B\in K[\der]$ be such that~$L=A+B$, $B\prec_{\Delta(\fv)} \fv^{m+1} L$.  
Then  
$$ L^+\ =\ A+B^+\ \text{ where $B^+\in K[\der]$, $B^+\ \prec_{\Delta(\fv)}\ \fv^{m+1} L^+$.}$$
In particular, $L-L^+\prec_{\Delta(\fv)} \fv^{m+1}L$. 
\end{lemma}

\begin{proof}  
Let $A_n,B_n\in K$ be such that $A=\sum_n A_n\der^n$ and $B=\sum_n B_n\der^n$, so $L_n=A_n+B_n$. Let any $n$ (possibly~$>r$) be given and
take $Q,R\in K\{Y\}$ as in the proof of Lemma~\ref{lem:linear part, new}. Then $R\prec_{\Delta(\fv)} \fv^{m+1} P_1$. 
Since $B\prec_{\Delta(\fv)} \fv^{m+1} L$, this yields
$$\frac{\partial P}{\partial Y^{(n)}} - A_n\ =\ \frac{\partial R}{\partial Y^{(n)}}Y^{(n)} + B_n + R\ \prec_{\Delta(\fv)}\ \fv^{m+1} P_1.$$
We have $L_n^+=\frac{\partial P}{\partial Y^{(n)}}(a)$, so
$$L_n^+-A_n\ =\ \frac{\partial P}{\partial Y^{(n)}}(a) - A_n\ \prec_{\Delta(\fv)}\ \fv^{m+1} L.$$
By Lemma~\ref{lem:linear part, new} we have 
$L^+\sim_{\Delta(\fv)} L$, hence $B^+=L^+-A \prec_{\Delta(\fv)} \fv^{m+1} L^+$. 
\end{proof}

\section{Holes and Slots} \label{sec:holes}

\noindent
{\em Throughout this section $K$ is an $H$-asymptotic field with small derivation and with rational asymptotic integration. We set $\Gamma:= v(K^\times)$}. 
So $K$ is pre-$\d$-valued, $\Gamma\ne \{0\}$ has no least positive element, and $\Psi\cap \Gamma^{>}\ne \emptyset$. We let  $a$, $b$, $f$, $g$ range over $K$, and~$\phi$,~$\fm$,~$\fn$,~$\fv$,~ $\fw$ (possibly decorated) over $K^\times$.   As at the end of the previous section we shorten ``active in $K$'' to ``active''. 

\subsection*{Holes}
A {\bf hole}\/ in $K$ is a triple $(P,\fm,\hat a)$ where 
$P\in K\{Y\}\setminus K$ and
$\hat a$ is an element of~$\hat K\setminus K$, for some
immediate asymptotic extension $\hat K$ of $K$,
such that $\hat a\prec\fm$ and~$P(\hat a)=0$. (The extension $\hat K$ may vary with $\hat a$.) 
The {\bf order}\/, {\bf degree}\/, and {\bf complexity}\/ of a hole
$(P,\fm,\hat a)$ in~$K$ are defined as the order, (total) degree, and complexity, respectively, of the differential
polynomial $P$. 
A hole~$(P,\fm, \hat a)$ in~$K$ is called {\bf minimal}\/ if no hole in $K$ has smaller complexity; then $P$ is a minimal annihilator of $\hat a$ over $K$.\index{hole}\index{hole!minimal}\index{hole!complexity}\index{complexity!hole}\index{minimal!hole}\label{p:hole}

\medskip
\noindent
If $(P,\fm,\hat a)$ is a hole in $K$, then $\hat a$ is a $K$-external zero of $P$, in the sense of Section~\ref{sec:complements newton}.
Conversely, every $K$-external zero $\hat a$ of a differential polynomial $P\in K\{Y\}^{\neq}$ 
gives for every $\fm\succ \hat a$ a hole $(P,\fm,\hat a)$ in $K$.
By Proposition~\ref{14.0.1r} and Corollary~\ref{14.5.2.r}:

\begin{lemma}\label{lem:no hole of order <=r} 
Let $r\in\N^{\geq 1}$, and suppose $K$ is $\upl$-free. Then 
$$\text{$K$ is $\upo$-free and $r$-newtonian}\quad\Longleftrightarrow\quad\text{$K$ has no hole of order~$\leq r$.}$$
\end{lemma}

\noindent
Thus for $\upo$-free $K$, being newtonian is equivalent to having no holes.   
Recall that~$K$ being henselian is equivalent to $K$ having no proper immediate algebraic valued field extension, and hence 
to $K$ having no hole of order $0$.

\medskip\noindent
Minimal holes are like the ``minimal counterexamples'' in certain combinatorial settings, and we need to understand such holes in a rather detailed way for later use in inductive arguments. Below we also consider the more general notion of {\em $Z$-minimal hole},
which has an important role to play as well.  We recall that $Z(K,\hat a)$ is the set of all $Q\in K\{Y\}^{\ne}$ that vanish at $(K,\hat a)$ as defined in  [ADH, 11.4].\label{p:Z(K,a)}

\begin{lemma}\label{lem:Z(K,hat a)}
Let $(P,\fm,\hat a)$ be a hole in $K$. Then $P\in Z(K,\hat a)$. 
If
$(P,\fm,\hat a)$ is minimal, then $P$ is an element of minimal complexity of $Z(K,\hat a)$.
\end{lemma}
\begin{proof}
Let $a$, $\fv$ with $\hat a-a\prec\fv$. Since $\hat a\notin K$ lies in an immediate extension of~$K$ we can take $\fn$ with $\fn\asymp \hat a-a$.
By [ADH, 11.2.1] we then have $\ndeg_{\prec\fv} P_{+a}\geq\ndeg P_{+a,\times\fn}\geq 1$.
Hence $P\in Z(K,\hat a)$. Suppose $P$ is not of minimal complexity 
in~$Z(K,\hat a)$. Take  $Q\in Z(K,\hat a)$ of minimal
complexity. Then [ADH, 11.4.8] yields a $K$-external zero $\hat b$ of $Q$,
and any $\fn\succ \hat b$ gives a hole $(Q,\fn,\hat b)$ in $K$ of smaller complexity than $(P, \fm, \hat a)$.
\end{proof}

\noindent
In connection with the next result, note that $K$ being $0$-newtonian just means that~$K$ is henselian as a valued field. 

\begin{cor}\label{minholenewt} Suppose $K$ is $\upl$-free and has a minimal hole of order $r\ge 1$. Then~$K$ is $(r-1)$-newtonian,
and $\upo$-free if $r\geq 2$.  
\end{cor}
\begin{proof}
This is clear for $r=1$ (and doesn't need $\upl$-freeness), and for $r\geq 2$ 
follows from Lemma~\ref{lem:no hole of order <=r}.
\end{proof}

\begin{cor}\label{corminholenewt} Suppose $K$ is $\upo$-free and has a minimal hole of order $r\ge 2$. Assume also that $C$ is algebraically closed and $\Gamma$ is divisible. Then 
$K$ is $\d$-valued, $r$-linearly closed, and $r$-linearly newtonian.
\end{cor} 
\begin{proof} This follows from Lemma~\ref{lem:ADH 14.2.5}, Corollary~\ref{14.5.3.r}, and Corollary~\ref{minholenewt}.
\end{proof}

\noindent
Here is a   linear version of Lemma~\ref{lem:no hole of order <=r}: 

\begin{lemma}\label{lem:no hole of order <=r, deg 1}
If $K$ is $\upl$-free, then
$$\text{$K$ is $1$-linearly newtonian}\ \Longleftrightarrow\ \text{$K$ has no hole of degree~$1$ and order~$1$.}$$
If $r\in\N^{\geq 1}$ and $K$ is $\upo$-free, then
$$\text{$K$ is $r$-linearly newtonian}\ \Longleftrightarrow\ \text{$K$ has no hole of degree~$1$ and order~$\leq r$.}$$
\end{lemma}
\begin{proof}
The first statement follows from Lemma~\ref{lem:char 1-linearly newt}, and the second statement from Lemma~\ref{lem:char r-linearly newt}.
\end{proof}

\begin{cor}\label{degmorethanone} If $K$ is $\upo$-free and has a minimal hole in $K$ of order~$r$ and degree~$>1$, then $K$ is $r$-linearly newtonian.
\end{cor}

\begin{lemma}
Suppose $K$ has a hole $(P,\fm,\hat a)$ of degree $1$, and $L_P\in K[\der]^{\neq}$  splits over $K$. Then $K$ has a hole of complexity $(1,1,1)$.
\end{lemma}
\begin{proof}
Let $(P,\fm,\hat a)$ as in the hypothesis have minimal order. Then~${\order P\ge 1}$, so~$\order P = \order L_P$. 
Take $A,B\in K[\der]$ such that $\order A=1$ and  $L_P=AB$.
If $\order B=0$, then $(P,\fm,\hat a)$ has complexity $(1,1,1)$.
Assume $\order B\ge 1$. Then~$B(\hat a)\notin K$: otherwise, taking $Q\in K\{Y\}$ of degree~$1$ with
$L_Q=B$ and~$Q(0)=-B(\hat a)$ yields a hole $(Q,\fm,\hat a)$ in $K$ where $\deg Q=1$ and $L_Q$ splits over $K$,
and~$(Q,\fm,\hat a)$ has smaller order than $(P,\fm,\hat a)$.
Set $\hat b:=B(\hat a)$ 
and take $R\in K\{Y\}$ of degree $1$ with $L_R=A$ and $R(0)=P(0)$. Then
$$R(\hat b)\ =\ R(0)+L_R(\hat b)\ =\ P(0)+L_P(\hat a)\ =\ 
P(\hat a)\ =\ 0,$$
hence for any $\fn\succ\hat b$, $(R,\fn,\hat b)$ is a hole in $K$ of complexity $(1,1,1)$.
\end{proof}

\begin{cor}\label{cor:minhole deg 1}
Suppose $K$ is $\upo$-free,  $C$ is algebraically closed, and $\Gamma$ is divisible.
Then every minimal hole in $K$ of degree~$1$ has order~$1$.
If in addition $K$ is $1$-linearly newtonian, then  every minimal hole in $K$ has degree~$>1$. 
\end{cor}
\begin{proof}
The first statement follows from Corollary~\ref{corminholenewt} and the preceding lemma.
For the second statement, use the first and Lemma~\ref{lem:no hole of order <=r, deg 1}.  
\end{proof}

\noindent
Let $(P,\fm, \hat a)$ be a hole in $K$. We say $(P,\fm, \hat a)$ is {\bf $Z$-minimal} if
$P$ has minimal complexity in $Z(K,\hat a)$. Thus if
$(P,\fm,\hat a)$ is minimal, then it is $Z$-minimal by Lem\-ma~\ref{lem:Z(K,hat a)}.
If $(P,\fm,\hat a)$ is $Z$-minimal, then by [ADH, remarks following 11.4.3],  the differential polynomial~$P$ is a minimal annihilator of $\hat a$ over $K$.
Note also that~$\ndeg P_{\times\fm} \geq 1$ by~[ADH, 11.2.1]. In more detail:\index{hole!Z-minimal@$Z$-minimal}\index{Z-minimal@$Z$-minimal!hole}

\begin{lemma}\label{lem:lower bd on ddeg}
Let $(P,\fm,\hat a)$ be a hole in $K$. Then for all $\fn$ with $\hat a \prec \fn \preceq \fm$,
$$1\ \le\ \dval P_{\times \fn}\ \le\ 
\ddeg P_{\times \fn}\ \le\ \ddeg P_{\times\fm}.$$ In particular, $\ddeg_{\prec \fm} P\ge 1$.
\end{lemma}
\begin{proof} Assume $\hat a \prec \fn \preceq \fm$. 
Then $\hat a = \fn \hat b$ with $\hat b\prec 1$; 
put $Q:=P_{\times\fn}\in K\{Y\}^{\neq}$. Then $Q(\hat b)=0$,  hence $D_Q(0)=0$ and so $\dval Q=\dval P_{\times \fn}\ge 1$. The rest follows from
[ADH, 6.6.5(ii), 6.6.7, 6.6.9] and $\Gamma^{>}$ having no least element.
\end{proof}

\noindent
In the next lemma, $(\upl_\rho)$, $(\upo_\rho)$ are pc-sequences in $K$ as in [ADH, 11.5, 11.7].
 
\begin{lemma}\label{lem:upl-free, not upo-free}
Suppose $K$ is $\upl$-free and 
$\upo\in K$ is such that $\upo_\rho\leadsto\upo$
\textup{(}so~$K$ is not $\upo$-free\textup{)}.  Then
we have a  hole $(P,\fm,\upl)$ in $K$ where $P= 2Y'+Y^2+\upo$ and~$\upl_\rho\leadsto\upl$, and
each such hole in $K$ is a $Z$-minimal hole in $K$.
\end{lemma}
\begin{proof}
From [ADH, 11.7.13] we obtain $\upl$ in  an immediate asymptotic extension
of~$K$ such that~$\upl_\rho\leadsto\upl$ and $P(\upl)=0$.
Taking any $\fm$ with $\upl\prec\fm$ then yields a hole $(P,\fm,\upl)$ in~$K$ with $\upl_\rho\leadsto\upl$, and each such hole in $K$ 
is a $Z$-minimal hole in~$K$ by~[ADH,  11.4.13, 11.7.12].
\end{proof}

\begin{cor}\label{cor:upl-free, not upo-free}
If $K$ is $\upl$-free but not $\upo$-free, then each minimal hole in $K$ of positive order has complexity~$(1,1,1)$ or complexity $(1,1,2)$. If $K$ is a Liouville closed $H$-field and not $\upo$-free, then  $(P, \fm, \upl)$ is a minimal hole of complexity $(1,1,2)$,
where  $\upo$, $P$, $\upl$, $\fm$ are as in Lemma~\ref{lem:upl-free, not upo-free}.
\end{cor}

\noindent
Here the second part uses Corollary~\ref{cor:Liouville closed => 1-lin newt} and Lemma~\ref{lem:no hole of order <=r, deg 1}.

\subsection*{Slots}
In some arguments the notion of a hole in $K$ turns out to be too stringent. Therefore we introduce a more
flexible version of it:

\begin{definition}
A {\bf slot} in $K$ is a triple $(P,\fm,\hat a)$ where $P\in K\{Y\}\setminus K$ and~$\hat a$ is an element of~$\hat K\setminus K$, for some immediate asymptotic extension $\hat K$ of $K$, such that $\hat a\prec\fm$ and $P\in Z(K,\hat a)$. The {\bf order}, {\bf degree}, and {\bf complexity} of such a slot  in $K$ are defined to be the order, degree, and complexity of the differential polynomial $P$, respectively. A slot in $K$ of degree $1$ is also called a {\bf linear} slot in $K$. 
A slot $(P,\fm,\hat a)$ in $K$ is {\bf $Z$-minimal} if $P$ is of minimal complexity among elements
 of~$Z(K,\hat a)$.\index{slot}\index{slot!complexity}\index{complexity!slot}\index{slot!linear}\index{slot!Z-minimal@$Z$-minimal}\index{Z-minimal@$Z$-minimal!slot}\label{p:slot}
\end{definition}

\noindent
Thus  by Lemma~\ref{lem:Z(K,hat a)}, holes in $K$ are slots in $K$, and a hole in $K$ is $Z$-minimal iff  
it is $Z$-minimal as a slot in $K$. 
From  [ADH, 11.4.13] we obtain:

\begin{cor}\label{mindivmin} Let $(P,\fm,\hat a)$ be a $Z$-minimal slot in $K$ and $(a_\rho)$ be a divergent pc-sequence in $K$ such that $a_\rho\leadsto \hat a$. Then $P$ is a minimal differential polynomial of $(a_\rho)$ over $K$.
\end{cor}

\noindent
We say that slots  $(P,\fm,\hat a)$ and $(Q,\fn,\hat b)$ in $K$ are {\bf equivalent}   if
$P=Q$, $\fm=\fn$, and $v(\hat a-a)=v(\hat b-a)$ for all $a$; note that then~$Z(K,\hat a)=Z(K,\hat b)$, so 
$(P,\fm,\hat a)$ is $Z$-minimal iff $(P,\fm,\hat b)$ is $Z$-minimal. Clearly this is an equivalence relation on the class of slots in $K$.
The following lemma often allows us to pass from a $Z$-minimal slot to
a $Z$-minimal hole:\index{slot!equivalence}

\begin{lemma}\label{lem:from cracks to holes}
Let $(P,\fm,\hat a)$ be a $Z$-minimal slot in $K$. Then  $(P,\fm, \hat a)$ is equivalent to a  $Z$-minimal hole  in $K$. 
\end{lemma}
\begin{proof}
By  [ADH, 11.4.8] we obtain $\hat b$ in an immediate asymptotic extension of $K$ with $P(\hat b)=0$
and $v(\hat a-a)=v(\hat b-a)$ for all $a$.
In particular  $\hat b\notin K$, $\hat b\prec\fm$, so~$(P,\fm,\hat b)$ is a hole in $K$ equivalent to $(P, \fm, \hat a)$.
\end{proof}

\noindent
By [ADH, 11.4.8] the extension below containing $\hat b$ is not required to be immediate:

\begin{cor}\label{corisomin}  
If  $(P,\fm, \hat{a})$ is a $Z$-minimal hole in $K$ and   $\hat b$ in an asymptotic extension of $K$ satisfies
$P(\hat b)=0$ and $v(\hat a-a)=v(\hat b-a)$ for all $a$, then there is an isomorphism 
$K\<\hat{a}\>\to K\<\hat{b}\>$ of valued differential fields over $K$ sending~$\hat{a}$ to $\hat{b}$. 
\end{cor} 

\noindent
In particular, equivalent $Z$-minimal holes $(P,\fm, \hat{a})$, $(P,\fm,\hat b)$  in $K$
yield an isomorphism $K\<\hat{a}\>\to K\<\hat{b}\>$ of valued differential fields over $K$ sending~$\hat{a}$ to $\hat{b}$.

\medskip
\noindent
From Lemmas~\ref{lem:no hole of order <=r} and \ref{lem:from cracks to holes} we obtain:

\begin{cor}\label{cor:no dent of order <=r} 
Let $r\in\N^{\geq 1}$, and suppose $K$ is $\upo$-free. Then 
$$\text{$K$ is $r$-newtonian}\quad\Longleftrightarrow\quad\text{$K$ has no  slot of order~$\leq r$.}$$
\end{cor}

\noindent
Let $(P,\fm,\hat a)$ be a slot in $K$. Then $(bP, \fm, \hat a)$ for $b\neq 0$ is a slot in $K$ of the same complexity as $(P,\fm,\hat a)$, and
if $(P,\fm, \hat a)$ is $Z$-minimal, then so is~$(bP, \fm, \hat a)$;
likewise with ``hole in $K$'' in place of ``slot in $K$''. 
For active $\phi$ we have the {\bf compositional conjugate}  $(P^\phi,\fm,\hat a)$ by $\phi$ of $(P,\fm,\hat a)$:\index{slot!compositional conjugate}\index{conjugate!compositional} it is a
slot in $K^\phi$ of the same complexity as $(P, \fm, \hat a)$, it is $Z$-minimal if~$(P, \fm, \hat a)$ is, and it is a hole (minimal hole) in $K^\phi$ if~$(P,\fm,\hat a)$ is a hole (minimal hole, respectively) in $K$.
If the slots~$(P,\fm,\hat a)$, $(Q,\fn,\hat b)$ in $K$ are equivalent, then so are 
$(bP,\fm,\hat a)$, $(bQ,\fn,\hat b)$ for $b\neq 0$, as well as the slots~$(P^\phi,\fm,\hat a)$, $(Q^\phi,\fn,\hat b)$ in $K^\phi$ for active $\phi$.

\subsection*{Refinements and multiplicative conjugates} {\em In the rest of this section  $r$ ranges over natural numbers $\ge 1$ and $(P,\fm,\hat a)$ denotes a slot in $K$ of order $r$, so $P\notin K[Y]$ has order $r$. We set $w:=\wt(P)$, so $w\geq r\ge  1$.} Thus by [ADH, 4.3.2, 5.7.5]: 
$$\wt(P_{+a})\ =\ \wt(P_{\times \fn})\ =\ \wt(P^\phi)=w.$$ 
For $a$,~$\fn$ such that $\hat a-a\prec\fn\preceq\fm$ 
we obtain a slot $(P_{+a},\fn,{\hat a-a})$ in~$K$
of the same complexity as $(P,\fm,\hat a)$ [ADH, 4.3, 11.4].
Slots of this form are said to {\bf refine~$(P, \fm, \hat a)$}\/ and are called {\bf refinements}\/ 
of~$(P,\fm,\hat a)$.\index{slot!refinement}\index{refinement} 
A refinement of a refinement of~$(P,\fm,\hat a)$ is itself a refinement of~$(P,\fm,\hat a)$.
If $(P, \fm, \hat a)$ is $Z$-minimal, then so is any refinement of~$(P, \fm, \hat a)$. 
If~$(P,\fm,\hat a)$ is a hole in $K$, then so is each of its refinements, and likewise with ``minimal hole'' in place of ``hole''. 
For active~$\phi$, $(P_{+a},\fn,{\hat a-a})$ refines $(P,\fm,\hat a)$ iff~$(P^\phi_{+a},\fn,\hat a-a)$ refines $(P^\phi,\fm,\hat a)$.
If~$(P,\fm,\hat a)$ and~$(P,\fm,\hat b)$ are equivalent slots in~$K$ and
$(P_{+a},\fn,{\hat a-a})$ refines $(P,\fm,\hat a)$, then
$(P_{+a},\fn,{\hat b-a})$ refines~$(P,\fm,\hat b)$, and the slots
$(P_{+a},\fn,{\hat a-a})$, $(P_{+a},\fn,{\hat b-a})$ in $K$ are equivalent. Conversely, if~$(P,\fm,\hat a)$ and $(P,\fm,\hat b)$ are slots in $K$ with equivalent refinements, then $(P,\fm,\hat a)$ and $(P,\fm,\hat b)$ are equivalent.

\begin{lemma}\label{lem:refinements linearly ordered} 
Let $(P_{+a},\fn,\hat a-a)$ be a slot in $K$. Then $(P_{+a},\fn,\hat a-a)$ re\-fines~$(P,\fm,\hat a)$, or
$(P,\fm,\hat a)$ refines $(P_{+a},\fn,\hat a-a)$.
\end{lemma}
\begin{proof}
If $\fn\preceq\fm$, then $\hat a-a\prec\fn\preceq\fm$, so $(P_{+a},\fn,\hat a-a)$ refines~$(P,\fm,\hat a)$,
whereas if~$\fm\prec\fn$, then $(\hat a-a)-(-a)=\hat a\prec\fm\preceq\fn$, so  
$$(P,\fm,\hat a) = \big( (P_{+a})_{+(-a)}, \fm, ({\hat a-a})-(-a) \big)$$ refines $(P_{+a},\fn,\hat a-a)$.
\end{proof}

\begin{lemma}\label{lem:notin Z(K,hata)} 
Let $Q\in K\{Y\}^{\neq}$ be such that $Q\notin Z(K,\hat a)$. Then there is a re\-fine\-ment~$(P_{+a},\fn,\hat a-a)$ of $(P,\fm,\hat a)$ such that $\ndeg  Q_{+a,\times\fn}=0$ and $\hat a-a\prec\fn\prec\hat a$.
\end{lemma}

\begin{proof}
Take~$b$, $\fv$ such that $\hat a-b\prec \fv$ and $\ndeg_{\prec\fv} Q_{+b}=0$. 
We shall find an $a$ such that $\ndeg_{\prec\fv} Q_{+a}=0$, $\hat a-a\preceq\hat a$, and $\hat a-a\prec\fv$:
if~$\hat a-b\preceq\hat a$, we take $a:=b$; if~$\hat a-b\succ\hat a$,
then $-b\sim\hat a-b$ and so $\ndeg_{\prec\fv} Q=\ndeg_{\prec\fv} Q_{+b}= 0$ by [ADH, 11.2.7], hence $a:=0$ works. We next arrange $\hat a-a\prec\hat a$: if $\hat a -a\asymp \hat a$, take~$a_1$ with $\hat a-a_1\prec \hat a$, so $a-a_1\prec\fv$, hence
$\ndeg_{\prec\fv} Q_{+a_1}=\ndeg_{\prec\fv} Q_{+a}=0$, and thus $a$ can be replaced by $a_1$.
Since $\Gamma^{>}$ has no least element, we can choose $\fn$ with $\hat a-a\prec\fn\prec\hat a,\fv$, and then $(P_{+a},\fn,\hat a-a)$ refines $(P,\fm,\hat a)$ as desired.   
\end{proof}

\noindent
If $(P_{+a},\fm,\hat a-a)$ refines~$(P,\fm,\hat a)$,
then 
$D_{P_{+a,\times\fm}}=D_{P_{\times\fm,+(a/\fm)}}=D_{P_{\times\fm}}$ by [ADH, 6.6.5(iii)], and thus 
$$\ddeg P_{+a,\times\fm}\ =\ \ddeg P_{\times\fm}, \qquad
\dval P_{+a,\times\fm}\ =\ \dval P_{\times\fm}.$$ 
In combination with Lemma~\ref{lem:lower bd on ddeg} this has some useful consequences:

\begin{cor}\label{cor:ref 1}
Suppose  $(P,\fm,\hat a)$ is a hole in $K$ such that $\ddeg P_{\times\fm}=1$.
Then~$\ddeg_{\prec\fm} P=1$, and for all $\fn$ with $\hat a \prec \fn\preceq \fm$, $(P,\fn,{\hat a})$ refines $(P, \fm, \hat a)$ with~$\ddeg P_{\times \fn}= \dval P_{\times\fn}=1$.  
\end{cor}

\begin{cor}\label{cor:ref 2}
Suppose $(P_{+a},\fn,{\hat a-a})$ refines
the hole $(P,\fm,\hat a)$ in $K$. Then  
$$\ddeg P_{\times\fm}\ =\ 1\ \Longrightarrow\ \ddeg P_{+a,\times\fn}\ =\ \dval P_{+a,\times\fn}\ =\ 1.$$
\end{cor}
\begin{proof} Use $$1\le \dval P_{+a,\times \fn}\le \ddeg P_{+a,\times \fn}\le \ddeg P_{+a,\times\fm}=\ddeg P_{\times \fm},$$ where the first inequality follows from Lemma~\ref{lem:lower bd on ddeg} applied to
$(P_{+a}, \fn, {\hat a -a})$. 
\end{proof}

\noindent
If $(P_{+a},\fm,\hat a-a)$ refines $(P,\fm,\hat a)$, then in analogy with $\ddeg$ and $\dval$, 
$$\ndeg P_{+a,\times\fm}\ =\ \ndeg P_{\times\fm}, \qquad
\nval P_{+a,\times\fm}\ =\ \nval P_{\times\fm}.$$ 
(Use compositional conjugation by active $\phi$.)
Lemma~\ref{lem:lower bd on ddeg} goes through for slots, provided we use
 $\ndeg$ and $\nval$ instead of 
$\ddeg$ and $\dval$: 

\begin{lemma}\label{lem:lower bd on ndeg}
Suppose $\hat a \prec \fn \preceq \fm$. Then
$$1\ \le\ \nval P_{\times \fn}\ \le\ 
\ndeg P_{\times \fn}\ \le\ \ndeg P_{\times\fm}.$$ 
\end{lemma}
\begin{proof}
By [ADH, 11.2.3(iii), 11.2.5] it is enough to show~$\nval P_{\times\fn}\geq 1$.
Replacing~$(P,\fm,\hat a)$ by its refinement $(P,\fn,\hat a)$ we arrange~$\fm=\fn$.
Now $\Gamma^>$ has no  smallest element, so by definition of~$Z(K,\hat a)$ and [ADH, p.~483] we have  
$$1\ \leq\ \ndeg_{\prec\fm} P\ =\ \max\big\{\!\nval P_{\times\fv} : \fv\prec\fm\big\}.$$
Thus by [ADH, 11.2.5] we can take $\fv$ with 
$\hat a\prec\fv\prec\fm$ with $\nval P_{\times\fv}\geq 1$, and hence $\nval P_{\times\fm}\geq 1$, again  by [ADH, 11.2.5].
\end{proof}

\noindent
Lemma~\ref{lem:lower bd on ndeg} yields results 
analogous to Corollaries~\ref{cor:ref 1} and~\ref{cor:ref 2} above:

\begin{cor}\label{cor:ref 1n}
If $\ndeg P_{\times\fm}=1$, then 
for all $\fn$ with $\hat a \prec \fn\preceq \fm$, $(P,\fn,{\hat a})$ re\-fines~$(P, \fm, \hat a)$ and $\ndeg P_{\times \fn}= \nval P_{\times\fn}=1$. 
\end{cor}

\begin{cor}\label{cor:ref 2n}
If $(P_{+a},\fn,{\hat a-a})$ refines  $(P,\fm,\hat a)$, then  
$$\ndeg P_{\times\fm}\ =\ 1\ \Longrightarrow\ \ndeg P_{+a,\times\fn}\ =\ \nval P_{+a,\times\fn}\ =\ 1.$$
\end{cor}

\noindent
Any triple $(P_{\times\fn},\fm/\fn,\hat a/\fn)$ is also
a slot in~$K$, with the same complexity as $(P,\fm, \hat a)$; it is called the {\bf multiplicative conjugate}\/ of $(P,\fm,\hat a)$
by~$\fn$.\index{slot!multiplicative conjugate}\index{conjugate!multiplicative} If $(P,\fm,\hat a)$ is $Z$-mi\-ni\-mal, then so is any multiplicative conjugate.
If $(P,\fm,\hat a)$ is a hole in $K$, then so is any  multiplicative conjugate; likewise with ``minimal hole'' in place
of ``hole''. If two slots in  $K$ are equivalent, then
so are their multiplicative conjugates by $\fn$.

\medskip
\noindent
Refinements and multiplicative conjugates interact in the following way: Suppose the slot~$(P_{+a},\fn,\hat a-a)$ refines $(P,\fm,\hat a)$.
Multiplicative conjugation~$(P_{+a},\fn,{\hat a-a})$ in $K$ by $\fv$
then results in the slot~$(P_{+a,\times\fv},\fn/\fv,(\hat a-a)/\fv)$ in~$K$.
On the other hand, first taking the multiplicative conjugate
$(P_{\times\fv},\fm/\fv,\hat a/\fv)$ of~$(P,\fm,\hat a)$ by $\fv$ and
then refining to
$(P_{\times\fv,+a/\fv},\fn/\fv,\hat a/\fv-a/\fv)$
results in the same slot in $K$, thanks to the identity $P_{+a,\times\fv} = P_{\times\fv,+a/\fv}$.

\subsection*{Quasilinear slots}
Note that $\ndeg P_{\times \fm}\ge 1$ by Lemma~\ref{lem:lower bd on ndeg}. 
We call~$(P,\fm,\hat a)$   {\bf qua\-si\-li\-near}\index{slot!quasilinear}\index{quasilinear!slot} if~$P_{\times\fm}$ is quasilinear, that is, $\ndeg P_{\times\fm}=1$. If $(P,\fm, \hat a)$ is quasilinear, then so is any slot in $K$ equivalent to $(P,\fm,\hat a)$, any multiplicative conjugate of~$(P,\fm,\hat a)$, as well as any refinement of $(P,\fm, \hat a)$,  by Corollary~\ref{cor:ref 2n}. 
If $(P,\fm, \hat a)$ is linear, then it is quasilinear by Lemma~\ref{lem:lower bd on ndeg}. 

\medskip\noindent
Let $(a_\rho)$ be a divergent pc-sequence in $K$ with $a_\rho\leadsto\hat a$ and for each index $\rho$, let $\rho+1$ be the immediate successor of $\rho$ in the well-ordered index set, and let~$\fm_\rho\in K^\times$ be such that $\fm_\rho\asymp\hat a-a_\rho$. 
Take an index $\rho_0$ such that $\fm_\sigma\prec\fm_\rho\prec \fm$ for all~$\sigma>\rho\geq\rho_0$, cf.~[ADH, 2.2].

 \begin{lemma}\label{lem:pc vs dent} Let $\sigma\geq\rho\geq \rho_0$. Then 
  \begin{enumerate}
  \item[\textup{(i)}]  $(P_{+a_{\rho+1}},\fm_{\rho},\hat a-a_{\rho+1})$ is a refinement of $(P,\fm,\hat a)$;  
  \item[\textup{(ii)}] if $(P_{+a},\fn,\hat a-a)$ is a refinement of $(P,\fm,\hat a)$, then    $\fm_\rho\preceq\fn$ for all sufficiently large $\rho$, and for such $\rho$, $(P_{+a_{\rho+1}},\fm_{\rho},{\hat a-a_{\rho+1}})$ refines~$(P_{+a},\fn,\hat a-a)$;  
    \item[\textup{(iii)}]     $(P_{+a_{\sigma+1}},\fm_{\sigma},\hat a-a_{\sigma+1})$  refines $(P_{+a_{\rho+1}},\fm_{\rho},\hat a-a_{\rho+1})$.
  \end{enumerate}
 \end{lemma}
\begin{proof}
Part (i) follows from $\hat a-a_{\rho+1}\asymp\fm_{\rho+1}\prec\fm_\rho\preceq \fm$. For (ii) 
let~$(P_{+a},\fn,{\hat a-a})$ be a refinement of $(P,\fm,\hat a)$. Since $\hat a-a\prec\fn$,  we have $\fm_\rho\preceq\fn$ for all sufficiently large $\rho$. For such $\rho$,    with~$b:=a_{\rho+1}-a$ we have
$$(P_{+a_{\rho+1}},\fm_{\rho},\hat a-a_{\rho+1})\  =\ \big( (P_{+a})_{+b}, \fm_\rho, (\hat a-a)-b\big)$$
and 
$$(\hat a-a)-b\ =\ \hat a-a_{\rho+1}\ \asymp\ \fm_{\rho+1}\ \prec\ \fm_\rho\ \preceq\ \fn.$$
Hence  $(P_{+a_{\rho+1}},\fm_{\rho},{\hat a-a_{\rho+1}})$ refines~$(P_{+a},\fn,\hat a-a)$. Part (iii) follows from~(i) and~(ii).
\end{proof}

\noindent
Let $\mathbf a=c_K(a_\rho)$ be the cut defined by $(a_\rho)$ in $K$ and $\ndeg_{\mathbf a} P$ be the Newton degree of $P$ in $\mathbf a$ as introduced in [ADH, 11.2].
Then $\ndeg_{\mathbf a} P$ is the eventual value of~$\ndeg P_{+a_\rho,\times\fm_\rho}$.
Increasing $\rho_0$ we arrange that additionally for all 
$\rho\geq\rho_0$ we have~$\ndeg P_{+a_\rho,\times\fm_\rho}=\ndeg_{\mathbf a} P$. 

\begin{cor}\label{cor:quasilinear refinement}
$(P,\fm,\hat a)$ has a quasilinear refinement iff $\ndeg_{\mathbf a} P=1$.
\end{cor}
\begin{proof}
By Lemma~\ref{lem:lower bd on ndeg} and [ADH, 11.2.8] we have
\begin{equation}\label{eq:quasilinear refinement}
1 \leq \ndeg P_{+a_{\rho+1},\times\fm_{\rho}} = \ndeg P_{+a_\rho,\times\fm_\rho}.
\end{equation}
Thus if $\ndeg_{\mathbf a} P=1$,  then   for $\rho\geq\rho_0$,
 the refinement $(P_{+a_{\rho+1}},\fm_{\rho},\hat a-a_{\rho+1})$ of~$(P,\fm,\hat a)$ is quasilinear. Conversely, if $(P_{+a},\fn,\hat a-a)$ is a quasilinear refinement of~$(P,\fm,\hat a)$, then
Lemma~\ref{lem:pc vs dent}(ii) yields a $\rho\geq\rho_0$ such that  $\fm_\rho\preceq\fn$, and
then~$(P_{+a_{\rho+1}},\fm_{\rho},{\hat a-a_{\rho+1}})$ in $K$ refines~$(P_{+a},\fn,\hat a-a)$ and hence  is also quasilinear, so
$\ndeg_{\mathbf a} P=\ndeg P_{+a_\rho,\times\fm_\rho}=1$ by \eqref{eq:quasilinear refinement}.  
\end{proof}

\begin{lemma}\label{lem:quasilinear refinement} Assume $K$ is $\d$-valued and $\upo$-free, and $\Gamma$ is divisible.
Then every $Z$-minimal slot in $K$ of positive order has a quasilinear refinement. 
\end{lemma}
\begin{proof}
Suppose $(P,\fm,\hat a)$ is $Z$-minimal. Take a divergent pc-sequence $(a_\rho)$ in $K$ such that 
$a_\rho\leadsto \hat a$.
Then $P$ is a minimal differential polynomial of $(a_\rho)$ over $K$,  by Corollary~\ref{mindivmin}. Hence 
$\ndeg_{\boldsymbol a} P=1$ by [ADH, 14.5.1], where ${\boldsymbol a}:=c_K(a_\rho)$. 
Now Corollary~\ref{cor:quasilinear refinement}  gives a quasilinear refinement of $(P,\fm,\hat a)$.
\end{proof}
 
\begin{remark}
Suppose $K$ is a real closed $H$-field that is  $\upl$-free but not $\upo$-free.
(For example, the real closure of the $H$-field $\R\langle\upo\rangle$
from [ADH, 13.9.1] satisfies these conditions, by~[ADH, 11.6.8, 11.7.23, 13.9.1].)
Take $(P,\fm,\upl)$  as in Lemma~\ref{lem:upl-free, not upo-free}. Then by Corollary~\ref{cor:quasilinear refinement} and
[ADH, 11.7.9],  $(P,\fm,\upl)$ has no quasilinear refinement. Thus Lemma~\ref{lem:quasilinear refinement}
 fails if 
``$\upo$-free'' is replaced by ``$\upl$-free''.
\end{remark}

\begin{lemma}\label{lem:zero of P}   
Let $L$ be an $r$-newtonian $H$-asymptotic extension of $K$ such that~$\Gamma^<$ is cofinal in $\Gamma_{L}^<$,  and suppose $(P,\fm,\hat a)$ is quasilinear. Then $P(\hat b)=0$ and $\hat b\prec\fm$ for some $\hat b\in L$.
\end{lemma}
\begin{proof}
Lemma~\ref{lem:lower bd on ndeg} and $\ndeg P_{\times\fm}=1$ gives $\fn\prec \fm$ with $\ndeg_{\times\fn} P=1$.
By [ADH, p.~480], $\ndeg P_{\times\fn}$ does not change in passing from $K$ to $L$. As $L$ is $r$-newtonian this yields $\hat b\preceq\fn$ in $L$   with $P(\hat b)=0$.
\end{proof}

\begin{cor}\label{cor:find zero of P} Let $K$ be  $\d$-valued and $\upo$-free, and $L$ a newtonian
$H$-asymptotic extension of $K$.
If $(P,\fm,\hat a)$ is quasilinear,  then $P(\hat b)=0$, $\hat b\prec\fm$  for some~$\hat b\in L$.
\end{cor}
\begin{proof} By \cite[Theorem~B]{Nigel19}, $K$ has a newtonization $K^*$ inside $L$. Such  $K^*$ is $\d$-algebraic over $K$ by 
\eqref{eq:14.0.1},  %[ADH, remarks after 14.0.1], 
 so $\Gamma^{<}$ is cofinal in $\Gamma_{K^*}^{<}$ by Theorem~\ref{thm:ADH 13.6.1}. Thus 
we can apply Lemma~\ref{lem:zero of P} to $K^*$ in the role of $L$. 
\end{proof}

\noindent
Here is a variant of Lemma~\ref{lem:from cracks to holes}:

\begin{cor}\label{cor:find zero of P, 2}  Let $K$ be  $\d$-valued and $\upo$-free, and $L$ a newtonian
$H$-asymptotic extension of $K$.
Suppose   $\Gamma$ is divisible and  
 $(P,\fm,\hat a)$ is  $Z$-minimal. Then there exists~$\hat b\in L$ such that $K\<\hat b\>$ is an immediate extension of $K$ and~$(P,\fm,\hat b)$ is a hole in $K$ equivalent to~$(P,\fm,\hat a)$. \textup{(}Thus   if  $(P,\fm,\hat a)$ is also a  hole in $K$,
then there is an embedding~$K\langle \hat a\rangle\to L$ of valued differential fields over $K$.\textup{)}
\end{cor}
\begin{proof}
By Lemma~\ref{lem:quasilinear refinement} we may refine $(P,\fm,\hat a)$ to arrange that~$(P,\fm,\hat a)$ is quasilinear.
Then [ADH, 11.4.8] gives $\hat b$ in an immediate $H$-asymptotic extension of~$K$ with~$P(\hat b)=0$ and $v(\hat a-a)=v(\hat b-a)$ for all $a$. So $(P,\fm,\hat b)$ is a hole in $K$
  equivalent to~$(P,\fm,\hat a)$. 
The immediate $\d$-algebraic extension $K\langle \hat b\rangle$ of $K$ is $\upo$-free by Theorem~\ref{thm:ADH 13.6.1}. Then \eqref{eq:14.0.1}
gives a newtonian $\d$-algebraic immediate extension $M$ of $K\langle \hat b\rangle$ and thus of $K$. Then $M$ is a newtonization of $K$ by [ADH, 14.5.4] and thus embeds over $K$ into $L$. The rest follows from Corollary~\ref{corisomin}.
\end{proof}

\begin{remark}
Lemma~\ref{lem:quasilinear refinement} and Corollary~\ref{cor:find zero of P, 2}   go through with the hypothesis ``$\Gamma$~is divisible'' replaced by ``$K$ is henselian''.
The proofs are the same, using \cite[3.3]{Nigel19} in place of [ADH, 14.5.1] in the proof of Lemma~\ref{lem:quasilinear refinement},  and \cite[3.5]{Nigel19} in place of~[ADH, 14.5.4] in the proof of Corollary~\ref{cor:find zero of P, 2}.
\end{remark}

\noindent
For $r=1$ we can   omit  the hypothesis of $\upo$-freeness in Corollary~\ref{cor:find zero of P, 2}:

\begin{cor}\label{cor:find zero of P, 3}
Suppose %$K$ is $\upl$-free and $\Gamma$ is divisible, and 
$(P,\fm,\hat a)$ is $Z$-minimal of order~$r=1$ with a quasilinear refinement.
Let
 $L$ be a newtonian $H$-asymptotic extension of $K$.
  Then there exists~$\hat b\in L$ such that $K\langle\hat b\rangle$ is an immediate extension of $K$ and
  $(P,\fm,\hat b)$ is a hole in $K$ equivalent to $(P,\fm,\hat a)$. \textup{(}So if 
  $(P,\fm,\hat a)$ is also a hole in $K$, then we have
  an embedding~$K\langle \hat a\rangle\to L$ of valued differential fields over $K$.\textup{)}
\end{cor}
\begin{proof}
Take a divergent pc-sequence $(a_\rho)$ in $K$ with~$a_\rho\leadsto\hat a$. Then $\ndeg_{\mathbf a}P=1$ for~$\mathbf a:=c_K(a_\rho)$, by
Corollary~\ref{cor:quasilinear refinement}, and~$P$ is a minimal differential polynomial of $(a_\rho)$ over~$K$, by [ADH, 11.4.13].
The equality $\ndeg_{\mathbf a}P=1$ remains valid when passing from $K$, $\mathbf a$ to $L$, $c_L(a_\rho)$, respectively,
by Lemma~\ref{lem:11.2.13 invariant}. Hence [ADH, 14.1.10] yields $\hat b\in L$ such that $P(\hat b)=0$ and~$a_\rho\leadsto\hat b$, so $v(\hat a-a)=v(\hat b-a)$ for all $a$.
Then~$K\langle\hat b\rangle$ is an immediate extension of $K$ by [ADH, 9.7.6], so~$(P,\fm,\hat b)$ is a hole in $K$ equivalent to $(P,\fm,\hat a)$.
For the rest use Corollary~\ref{corisomin}. 
\end{proof}

\subsection*{The linear part of a slot} 
We define the {\bf linear part}\/ of  $(P,\fm,\hat a)$ to be the linear part 
$L_{P_{\times\fm}}\in K[\der]$ of~$P_{\times\fm}$.\index{slot!linear part}\index{linear part!slot} Recall: $\order(P)=r$.  By [ADH, p.~242] and \eqref{eq:separant fms}, 
$$L_{P_{\times\fm}}\ =\ L_P\, \fm\ =\ \sum_{n=0}^r \frac{\partial P_{\times\fm}}{\partial Y^{(n)}}(0)\,\der^n\ =\ \fm S_{P}(0)\der^r +\text{lower order terms in $\der$}.$$
The slot $(P,\fm,\hat a)$ has the same linear part as each of its multiplicative conjugates.
The linear part of a refinement $(P_{+a},\fn,\hat a-a)$ of $(P,\fm,\hat a)$ is
given by
\begin{align*}
L_{P_{+a,\times\fn}}\	=\ L_{P_{+a}}\fn\ 
							&=\ \sum_{m=0}^r \left(\sum_{n=m}^r {n\choose m} \fn^{(n-m)} \frac{\partial P}{\partial Y^{(n)}}(a)\right)\der^m\\ 
							&=\ \fn\,  S_P(a)\,\der^r +\text{lower order terms in $\der$.}
\end{align*}
(See [ADH, (5.1.1)].) 
By [ADH, 5.7.5] we have $(P^\phi)_d=(P_d)^\phi$ for $d\in \N$; in particular $L_{P^\phi}=(L_P)^\phi$
and so $\order(L_{P^\phi})=\order(L_P)$.
A particularly favorable situation occurs when $L_P$ splits over a given differential field extension $E$ of $K$ (which includes requiring $L_P\ne 0$). Typically, $E$ is an algebraic closure of $K$.  In any case, $L_P$ splits over $E$ iff $L_{P_{\times\fn}}$ splits over $E$, iff $L_{P^\phi}$ splits over $E^\phi$. Thus:

\begin{lemma}\label{lem:deg 1 cracks splitting}  
Suppose $\deg P=1$ and $L_P$  splits over $E$. Then the linear part of any refinement of $(P, \fm, \hat a)$ and any
multiplicative conjugate of $(P, \fm, \hat a)$ also splits over $E$, and any compositional conjugate of $(P, \fm, \hat a)$ by an active $\phi$
 splits over~$E^\phi$.
\end{lemma}

\noindent
Let $\i=(i_0,\dots,i_r)$ range over $\N^{1+r}$. As in [ADH, 4.2] we set 
$$P_{(\i)}\ :=\ \frac{P^{(\i)}}{\i !} \qquad\text{where
$P^{(\i)}\ :=\ 
\frac{\partial^{|\i|}P}{\partial^{i_0}Y\cdots \partial^{i_r}Y^{(r)}}$.}$$
If $\abs{\i}=i_0+\cdots+i_r\geq 1$, then $\cc(P_{(\i)})<\cc(P)$. Note that for $\i=(0,\dots,0,1)$ we have $P_{(\i)}=S_P\ne 0$, since $\order P =r$. 
We now aim for Corollary~\ref{cor:order L=r}.

\begin{lemma}\label{lem:ndeg coeff stabilizes, 1}
Suppose that $(P,\fm,\hat a)$ is $Z$-minimal. Then $(P,\fm,\hat a)$ has a refinement $(P_{+a},\fn,{\hat a-a})$
such that for all $\i$ with $\abs{\i}\geq 1$ and $P_{(\i)}\neq 0$,
$${\ndeg\,(P_{(\i)})_{+a,\times\fn}\ =\ 0}.$$ 
\end{lemma}
\begin{proof}
Let $\i$ range over the (finitely many) elements of $\N^{1+r}$ satisfying $\abs{\i}\geq 1$ and~$P_{(\i)}\neq 0$.
Each $P_{(\i)}$ has smaller complexity than $P$, so~$P_{(\i)}\notin Z(K,\hat a)$.
Then~$Q:=\prod_{\i} P_{(\i)}\notin Z(K,\hat a)$ by [ADH, 11.4.4], so Lemma~\ref{lem:notin Z(K,hata)} gives a refinement~$(P_{+a},\fn,\hat a-a)$ of $(P,\fm,\hat a)$ with $\ndeg Q_{+a,\times\fn}=0$.
Then~$\ndeg\, (P_{(\i)})_{+a,\times\fn}=0$ for all $\i$, by [ADH, remarks before 11.2.6]. 
\end{proof}

\noindent
From [ADH, (4.3.3)] we   recall that $(P_{(\i)})_{+a}=(P_{+a})_{(\i)}$. Also recall that $(P_{+a})_{\i}=P_{(\i)}(a)$ by Tay\-lor expansion. In particular, if $P_{(\i)}=0$, then $(P_{+a})_{\i}=0$.  

\begin{lemma}\label{lem:ndeg coeff stabilizes, 2} 
Suppose $(P_{+a},\fn,\hat a-a)$ refines $(P,\fm,\hat a)$ and~$\i$ is such that~${\abs{\i}\geq 1}$, $P_{(\i)}\neq 0$, 
and $\ndeg\, (P_{(\i)})_{\times\fm}=0$.  Then
$$\ndeg\, (P_{(\i)})_{+a,\times\fn}\ =\ 0, \qquad {(P_{+a})_{\i}\ \sim\ P_{\i}}.$$
\end{lemma}
\begin{proof}
Using [ADH, 11.2.4, 11.2.3(iii), 11.2.5] we get
$$\ndeg\, (P_{(\i)})_{+a,\times\fn}\ =\ 
\ndeg\,(P_{(\i)})_{+\hat a,\times\fn}\ \leq\ 
\ndeg\, (P_{(\i)})_{+\hat a,\times\fm}\ =\ 
\ndeg\, (P_{(\i)})_{\times\fm}\ =\ 0,$$
so $\ndeg\,(P_{(\i)})_{+a,\times\fn}=0$. Thus $P_{(\i)}\notin Z(K,\hat a)$, hence $(P_{+a})_{\i}=P_{(\i)}(a)\sim P_{(\i)}(\hat a)$ by~[ADH, 11.4.3];
applying this to $a=0$, $\fn=\fm$ yields $P_{\i}=P_{(\i)}(0)\sim P_{(\i)}(\hat a)$. 
\end{proof}

\noindent
Combining Lemmas~\ref{lem:ndeg coeff stabilizes, 1} and~\ref{lem:ndeg coeff stabilizes, 2} gives:  

\begin{cor}\label{cor:order L=r}
Every $Z$-minimal slot in $K$ of order $r$ has a refinement $(P,\fm,\hat a)$ such that for all
refinements $(P_{+a},\fn,\hat a-a)$ of $(P,\fm,\hat a)$ and all $\i$ with $\abs{\i}\geq 1$ and~$P_{(\i)}\neq 0$ we have
$(P_{+a})_{\i}\sim P_{\i}$ \textup{(}and thus $\order L_{P_{+a}}=\order L_P=r$\textup{)}.
\end{cor}

\noindent
Here the condition ``of order $r$'' may seem irrelevant, but
is forced on us because refinements preserve order and by our convention that $P$ has order $r$.

\subsection*{Special slots} 
The slot $(P,\fm, \hat a)$ in $K$ is said to be {\bf special\/} if $\hat a/\fm$ is special over~$K$ in the sense of Section~\ref{sec:special elements}: some nontrivial convex subgroup $\Delta$ of $\Gamma$ is cofinal in~$v\big(\frac{\hat{a}}{\fm}-K\big)$.\index{slot!special}\index{special!slot}  
If $(P,\fm,\hat a)$ is special, then so are $(bP,\fm,\hat a)$ for $b\neq 0$, any multiplicative conjugate of $(P,\fm, \hat a)$, any compositional conjugate of $(P,\fm,\hat a)$, and any slot in $K$ equivalent to $(P,\fm,\hat a)$. Also, by Lemma~\ref{lem:special refinement}:

\begin{lemma}\label{speciallemma} If $(P,\fm, \hat a)$ is special, then so is any refinement.
\end{lemma}

\noindent
Here is our main source of special slots: 

{\sloppy
\begin{lemma}\label{lem:special dents} 
Let $K$ be $r$-linearly newtonian, and $\upo$-free if~$r>1$. Suppose~$(P,\fm, \hat a)$ is  quasilinear, and  $Z$-minimal or a hole in $K$. Then  $(P,\fm, \hat a)$ is special.
\end{lemma}}
\begin{proof} 
Use Lemma~\ref{lem:from cracks to holes} to arrange $(P,\fm,\hat a)$ is a hole in~$K$.
Next arrange $\fm=1$ by replacing  $(P,\fm,\hat a)$ with $(P_{\times\fm},1,\hat a/\fm)$. So $\ndeg P=1$, hence $\hat a$ is special over $K$ by Proposition~\ref{nepropsp} (if $r>1$) and~\ref{nepropsp, r=1} (if $r=1$).
\end{proof}

\noindent
Next an approximation result that will be needed in \cite{ADH5}:  %the proof of Corollary~\ref{mfhc} in Part~\ref{part:Hardy fields}:  

\begin{lemma}\label{lem:small P(a)}
Suppose $\fm=1$, $(P,1,\hat a)$ is special and $Z$-minimal, and  $\hat a-a\preceq\fn\prec 1$ for some $a$.
Then $\hat a-b\prec\fn^{r+1}$ for some $b$, and $P(b)\prec\fn P$ for any such~$b$.
\end{lemma}

\begin{proof}
Using Lemma~\ref{lem:from cracks to holes} we arrange $P(\hat a)=0$.
The differential po\-ly\-no\-mial~$Q:=\sum_{\abs{\i}\geq 1} P_{(\i)}(\hat a)Y^{\i}\in \hat K\{Y\}$ has order~$\leq r$  and $\val(Q)\geq 1$,
and Taylor expansion   yields,  for all $a$:
$$P(a)\ =\ P(\hat a) + \sum_{\abs{\i}\geq 1} P_{(\i)}(\hat a)(a-\hat a)^{\i}\ =\ Q(a-\hat a).$$
Since $\hat a$ is special over $K$, we have $b$ with $\hat a-b\prec\fn^{r+1}$, and then by Lemma~\ref{lem:diff operator at small elt} we have $Q(b-\hat a)\prec\fn Q\preceq\fn P$.
\end{proof}

\section{The First Normalization Theorems}\label{sec:normalization}

\noindent
{\em Throughout this section $K$ is an $H$-asymptotic field with small derivation and with rational asymptotic integration. We set $\Gamma:= v(K^\times)$}. 
The notational conventions introduced in the last section remain in force:
$a$,~$b$,~$f$,~$g$ range over $K$; $\phi$,~$\fm$,~$\fn$,~$\fv$,~$\fw$ over $K^\times$. As at the end of Section~\ref{sec:span} we shall frequently use for $\fv\prec 1$ the coarsening of $v$ by the convex subgroup $\Delta(\fv)=\big\{\gamma\in \Gamma:\, \gamma=o(v\fv)\big\}$ of $\Gamma$.

We fix a slot $(P,\fm,\hat a)$ in $K$ of order $r\geq 1$, and set $w:=\wt(P)$ (so $w\geq r\geq 1$).
In the next subsections we introduce various conditions on~$(P,\fm,\hat a)$. These conditions will be shown to be related as follows:
\[
\xymatrix{ \text{ strictly normal } \ar@{=>}[r] & \text{ normal } \ar@{=>}[r] \ar@{=>}[d] & \text{ steep } \\
& \text{ quasilinear } \ar@{<=}[r] &  \text{ deep }  \ar@{=>}[u]}
\] 
Thus ``deep + strictly normal'' yields the rest. The main results of this section are Theorem~\ref{mainthm} and its variants \ref{cor:mainthm}, \ref{varmainthm}, and \ref{cor:achieve strong normality, 2}.

\subsection*{Steep and deep slots}
In this subsection, if $\order (L_{P_{\times \fm}})=r$, then we set
$$\fv\ :=\ \fv(L_{P_{\times\fm}}).$$ 
The slot~$(P,\fm,\hat a)$ in $K$ is said to be {\bf steep}\index{steep!slot}\index{slot!steep} if $\order(L_{P_{\times \fm}})=r$ and $\fv\prec^\flat 1$. 
Thus 
$$(P,\fm, \hat a) \text{ is steep }\Longleftrightarrow\ (P_{\times\fn},\fm/\fn,\hat a/\fn) \text{ is steep }\Longleftrightarrow\ (bP,\fm,\hat a) \text{ is steep}$$
for $b\neq 0$.   If $(P,\fm,\hat a)$ is steep, then so is any slot in $K$ equivalent to $(P,\fm,\hat a)$. If $(P, \fm, \hat a)$ is steep, then so is any slot $(P^\phi,\fm,\hat a)$ in~$K^\phi$
for active~$\phi\preceq 1$, 
by  Lemma~\ref{lem:v(Aphi)}, and thus $\nwt(L_{P_{\times \fm}})<r$.  
Below we tacitly use that if~$(P,\fm,\hat a)$ is steep, then
$$\fn\asymp_{\Delta(\fv)}\fv\ \Longrightarrow\ [\fn]=[\fv], \qquad \fn\prec 1,\ [\fn]=[\fv]\ \Longrightarrow\ \fn \prec^\flat 1.$$
Note also that if $(P,\fm,\hat a)$ is steep, then $\fv^\dagger \asymp_{\Delta(\fv)} 1$ by [ADH, 9.2.10(iv)].

\begin{lemma}\label{lem:steep1}
Suppose $(P,\fm,\hat a)$ is steep, $\hat a\prec\fn\preceq \fm$ and $[\fn/\fm]\leq [\fv]$. Then 
$$\order(L_{P_{\times\fn}})\ =\ r, \qquad \fv(L_{P_{\times\fn}})\ \asymp_{\Delta(\fv)}\ \fv,$$ 
so $(P,\fn,\hat a)$ is a steep refinement of 
$(P,\fm,\hat a)$.
\end{lemma}
\begin{proof}
Replace $(P,\fm,\hat a)$, $\fn$ by $(P_{\times\fm},1,\hat a/\fm)$, $\fn/\fm$, respectively, to arrange~${\fm=1}$. Set~${L:= L_P}$ and $\tilde L:=L_{P_{\times\fn}}$.
Then $\tilde L=L\fn\asymp_{\Delta(\fv)} \fn L$ by [ADH, 6.1.3]. 
Hence 
$$\tilde L_r\ =\ \fn L_r\ \asymp\ \fn\fv L\ \asymp_{\Delta(\fv)}\ \fv\tilde L.$$
Since $\fv(\tilde L)\tilde L \asymp \tilde L_r$, this gives
$\fv(\tilde L)\tilde L\asymp_{\Delta(\fv)} \fv \tilde L$, and thus
$\fv(\tilde L)\asymp_{\Delta(\fv)} \fv$. 
\end{proof}

\noindent
If $(P,\fm,\hat a)$ is steep and linear, then $$L_{P_{+a,\times\fm}}=L_{P_{\times\fm,+(a/\fm)}}=L_{P_{\times\fm}},$$ so any refinement $(P_{+a},\fm,\hat a-a)$ of $(P,\fm,\hat a)$  is also steep and linear.

\begin{lemma}\label{lem:achieve steep} 
Suppose  $\order L_{P_{\times\fm}}=r$. Then $(P,\fm,\hat a)$ has a refinement $(P,\fn,\hat a)$ such that 
$\nwt L_{P_{\times\fn}}=0$, and $(P^\phi,\fn,\hat a)$
is steep, eventually.
\end{lemma}
\begin{proof}
Replacing $(P,\fm,\hat a)$ by $(P_{\times\fm},1,\hat a/\fm)$ we  arrange $\fm=1$. 
Take $\fn_1$ with $\hat a\prec \fn_1 \prec 1$. Then $\order\, (P_1)_{\times\fn_1}=\order P_1=\order L_P=r$, and 
thus~${(P_1)_{\times\fn_1}\neq 0}$. So~[ADH, 11.3.6] applied to~$(P_1)_{\times\fn_1}$ in place of $P$
yields an $\fn$ with~$\fn_1\prec\fn\prec 1$ and $\nwt\, (P_1)_{\times\fn}=0$, so $\nwt L_{P_{\times\fn}}=0$.
Hence by Lemma~\ref{lem:eventual value of fv},  $(P^\phi,\fn,\hat a)$
is steep, eventually.
\end{proof}

\noindent
Recall that  the separant $S_P=\partial P/\partial Y^{(r)}$ of $P$ has lower complexity than $P$.
Below we sometimes use the identity $S_{P_{\times\fm}^\phi}=\phi^r (S_{P_{\times\fm}})^\phi$ from~\eqref{eq:separant fms}.

\medskip\noindent
The slot $(P,\fm,\hat a)$ in $K$ is said to be {\bf deep} if it is steep and for all active $\phi\preceq 1$,\index{slot!deep}\index{deep}
{\samepage
\begin{itemize}
\item[(D1)] $\ddeg S_{P^\phi_{\times\fm}}=0$ (hence $\ndeg S_{P_{\times\fm}}=0$), and 
\item[(D2)] $\ddeg P^\phi_{\times\fm}=1$ (hence $\ndeg P_{\times \fm}=1$).
\end{itemize}}\noindent
If $\deg P=1$, then (D1) is automatic, for all active $\phi\preceq 1$. 
If~$(P,\fm,\hat a)$ is deep, then so are $(P_{\times\fn},\m/\fn,\hat a/\fn)$ and $(bP,\fm,\hat a)$ for $b\neq 0$, as well as every slot in $K$ equivalent to  $(P,\fm,\hat a)$ and the slot $(P^\phi,\fm,\hat a)$ in $K^\phi$ for active~${\phi\preceq 1}$. Every deep slot in $K$ is quasilinear, by (D2).
If  $\deg P=1$, then~$(P,\fm,\hat a)$ is quasilinear iff~$(P^\phi,\fm,\hat a)$ is deep for some active $\phi\preceq 1$. 
Moreover, if $(P,\fm,\hat a)$ is a deep hole in $K$, then~$\dval P^\phi_{\times \fm}=1$ for all active $\phi\preceq 1$, by
(D2) and Lemma~\ref{lem:lower bd on ddeg}.

\begin{exampleNumbered}\label{ex:order 1 linear steep} 
Suppose $P=Y'+gY-u$ where $g,u\in K$ and $\fm=1$, $r=1$. Set~$L:=L_P=\der+g$ and $\fv:=\fv(L)$. Then 
$\fv=1$ if $g\preceq 1$, and $\fv=1/g$ if $g\succ 1$. Thus 
$$\text{$(P,1,\hat a)$ is steep} \quad\Longleftrightarrow\quad g\succ^\flat 1  \quad\Longleftrightarrow\quad  
\text{$g\succ 1$ and $g^\dagger \succeq 1$.}$$
Note that $(P,1,\hat a)$ is steep iff $L$ is steep as defined in Section~\ref{sec:lindiff}. 
Also, 
$$\text{$(P,1,\hat a)$ is deep} \quad\Longleftrightarrow\quad \text{$(P,1,\hat a)$ is steep and $g\succeq u$.}$$
Hence if $u=0$, then $(P,1,\hat a)$ is deep iff it is steep. 
\end{exampleNumbered}

\begin{lemma}\label{lem:eventually deep}
For steep $(P,\fm,\hat a)$, the following are equivalent:
\begin{enumerate}
\item[$\mathrm{(i)}$]
 $(P^\phi,\fm,\hat a)$ is deep, eventually; 
\item[$\mathrm{(ii)}$] 
$\ndeg S_{P_{\times\fm}}=0$ and $\ndeg P_{\times\fm}=1$.
\end{enumerate}
\end{lemma}

\noindent
Note that if $\ddeg S_{P_{\times \fm}}=0$ or $\ndeg S_{P_{\times \fm}}=0$, then $S_{P_{\times\fm}}(0)\neq 0$, so $\order L_{P_{\times\fm}}=r$.

\begin{lemma}\label{lem:deep 1}
Suppose $(P_{+a},\fn,\hat a-a)$ refines the hole $(P,\fm,\hat a)$ in $K$. Then: 
\begin{enumerate}
\item[$\mathrm{(i)}$]  $\ddeg S_{P_{\times \fm}}=0\ \Longrightarrow\ \ddeg S_{P_{+a,\times\fn}}=0$;
\item[$\mathrm{(ii)}$] $\ddeg P_{\times \fm}=1\ \Longrightarrow\ \ddeg P_{+a,\times \fn}=1$;
\item[$\mathrm{(iii)}$] $\ndeg S_{P_{\times \fm}}=0\ \Longrightarrow\ S_P(a)\sim S_P(0)$.
\end{enumerate}
Thus if $(P,\fm, \hat a)$ is deep and $(P_{+a},\fn,\hat a-a)$ is steep, then $(P_{+a},\fn,\hat a-a)$ is deep. 
\end{lemma}
\begin{proof} 
Suppose $\ddeg S_{P_{\times \fm}}=0$. Then $\ddeg S_{P_{+a,\times \fn}}=0$ follows from $$\ddeg S_{P_{+a,\times\fn}}\ =\ \ddeg\,(S_P)_{+a,\times\fn}\ \text{ and }\
\ddeg (S_P)_{\times\fm}\ =\ \ddeg S_{P_{\times\fm}}$$
(consequences of~\eqref{eq:separant fms}), and
$$\ddeg\,(S_P)_{+a,\times\fn}\ =\
\ddeg\,(S_P)_{+\hat a,\times\fn}\ \leq\ 
\ddeg\,(S_P)_{+\hat a,\times\fm}\ =\ \ddeg\,(S_P)_{\times\fm}$$
which holds by [ADH, 6.6.7]. This proves (i). Corollary~\ref{cor:ref 2} yields (ii), and (iii) is contained in Lemma~\ref{lem:ndeg coeff stabilizes, 2}. 
\end{proof}

\noindent
Lemmas~\ref{lem:from cracks to holes} and~\ref{lem:deep 1} give:

\begin{cor}\label{cor:steep refinement}
If  $(P,\fm,\hat a)$ is $Z$-minimal and deep, then each steep refinement of~$(P,\fm,\hat a)$ is deep.
\end{cor}
 
\noindent
Here is another sufficient condition on refinements of deep holes to remain deep:

\begin{lemma}\label{lem:deep 2}
Suppose $(P,\fm,\hat a)$ is a deep hole in $K$, and $(P_{+a},\fn,\hat a-a)$ refines~$(P,\fm,\hat a)$ with $[\fn/\fm]\le[\fv]$. Then $(P_{+a},\fn,\hat a-a)$  is deep with $\fv(L_{P_{+a,\times\fn}})\asymp_{\Delta(\fv)}\fv$.
\end{lemma}
\begin{proof}
From $(P,\fm,\hat a)$  we pass to the hole
$(P_{+a}, \fm, \hat a -a)$ and then to $(P_{+a},\fn,\hat a-a)$.  
We first show that $\order L_{P_{+a,\times \fm}}=r$ and 
$\fv(L_{P_{+a, \times \fm}})\sim \fv$, from which it follows that $(P_{+a},\fm,\hat a-a)$ is steep, hence deep by Lemma~\ref{lem:deep 1}.
By Corollary~\ref{cor:ref 2}, 
$$\ddeg P_{+a,\times \fm}\ =\ \dval P_{+a, \times \fm}\ =\ 1,$$ so
$(P_{+a,\times \fm})_1 \sim P_{+a,\times \fm}$. Also
$$(P_{\times \fm})_1\sim P_{\times \fm}\sim P_{\times \fm,+(a/\fm)}= P_{+a,\times \fm},$$  by [ADH, 4.5.1(i)], and
thus $(P_{+a,\times \fm})_1\sim (P_{\times\fm})_1$.
By~\eqref{eq:separant fms} and Lemma~\ref{lem:deep 1}(iii),
$$S_{P_{+a,\times \fm}}(0)\ =\ \fm S_P(a)\ \sim\ \fm S_P(0)\ =\ S_{P_{\times \fm}}(0),$$
so $S_{P_{+a,\times \fm}}(0)\sim S_{P_{\times \fm}}(0)$. This gives
$\fv(L_{P_{+a,\times\fm}})\sim \fv$ as promised.

Next, Lemma~\ref{lem:steep1} applied to $(P_{+a}, \fm, \hat a -a)$
in the role of $(P,\fm, \hat a)$ gives that~$(P_{+a}, \fn, \hat a -a)$ is steep with $\fv(L_{P_{+a,\times \fn}})\asymp_{\Delta(\fv)} \fv$.
Now  Lemma~\ref{lem:deep 1} applied to $(P_{+a}, \fm, {\hat a -a})$ and  $(P_{+a}, \fn, {\hat a -a})$ 
in the role of  $(P,\fm, \hat a)$ and  $(P_{+a}, \fn, {\hat a -a})$, respectively, gives that  $(P_{+a}, \fn, \hat a -a)$ is deep. 
\end{proof}

\noindent
Lemmas~\ref{lem:from cracks to holes} and~\ref{lem:deep 2} give a version for $Z$-minimal slots:

\begin{cor}\label{cor:deep 2, cracks}
Suppose $(P,\fm,\hat a)$ is $Z$-minimal and deep, and $(P_{+a},\fn,{\hat a-a})$ refines $(P,\fm,\hat a)$ with $[\fn/\fm]\leq[\fv]$,
where $\fv:=\fv(L_{P_{\times\fm}})$.
Then~$(P_{+a},\fn,{\hat a-a})$  is deep with $\fv(L_{P_{+a,\times\fn}})\asymp_{\Delta(\fv)}\fv$.
\end{cor}

\noindent
Next we turn to the task of turning $Z$-minimal slots into deep ones.

\begin{lemma}\label{prop:normalize, q-linear} 
Every quasilinear $Z$-minimal slot in $K$ of order $r$ has a refinement~$(P,\fm,\hat a)$ such that:
\begin{enumerate}
\item[$\mathrm{(i)}$]
$\ndeg\,(P_{(\i)})_{\times\fm}=0$  for all $\i$ with $\abs{\i}\geq 1$ and $P_{(\i)}\neq 0$;
\item[$\mathrm{(ii)}$] $\ndeg P_{\times\fm}=\nval P_{\times\fm}=1$, and
\item[$\mathrm{(iii)}$] $\nwt L_{P_{\times\fm}}=0$. 
\end{enumerate}
\end{lemma}
\begin{proof}
By Corollary~\ref{cor:ref 1n}, any quasilinear $(P,\fm,\hat a)$  satisfies (ii). 
Any refinement  of a quasilinear $(P,\fm,\hat a)$ remains quasilinear, by Corollary~\ref{cor:ref 2n}. 
By Lemma~\ref{lem:ndeg coeff stabilizes, 1} and a subsequent remark any quasilinear $Z$-minimal slot in $K$ of order $r$ can be refined to a quasilinear $(P,\fm,\hat a)$ that satisfies (i), and by Lemma~\ref{lem:ndeg coeff stabilizes, 2}, any further refinement of such~$(P,\fm,\hat a)$ continues to satisfy (i). Thus to prove the lemma, assume we are given a
quasilinear $(P, \fm,\hat a)$ satisfying (i); it is enough to show that then $(P,\fm,\hat a)$ has a refinement
$(P,\fn,\hat a)$ satisfying (iii) with $\fn$ instead of $\fm$ (and thus also (i) and (ii) with $\fn$ instead of $\fm$).

 Take $\tilde\fm$ with $\hat a \prec\tilde\fm\prec\fm$. Then~$(P_{\times \tilde\fm})_1\ne 0$ by (ii), so 
[ADH, 11.3.6] applied to~$(P_1)_{\times\tilde\fm}$ in place of $P$ yields an~$\fn$ with $\tilde\fm\prec\fn\prec\fm$ and 
$\nwt\, (P_1)_{\times\fn}=0$. Hence the refinement~$(P,\fn,\hat a)$ of $(P,\fm,\hat a)$ 
satisfies (iii) with~$\fn$ instead of $\fm$.
\end{proof}

\begin{cor}\label{cor:deepening, q-linear}
Every quasilinear $Z$-minimal slot in $K$ of order $r$ has a re\-fine\-ment~$(P,\fm,\hat a)$ such that $\nwt L_{P_{\times\fm}}=0$, and~$(P^\phi,\fm,\hat a)$ is deep, eventually.
\end{cor}
\begin{proof} Given a quasilinear $Z$-minimal slot in $K$ of order $r$, we take a re\-fi\-ne\-ment~$(P,\fm,\hat a)$ as in Lemma~\ref{prop:normalize, q-linear}. Then $\ndeg S_{P_{\times \fm}}=0$ by (i) of that lemma, so $\order L_{P_{\times \fm}}=r$ by the remark that precedes Lemma~\ref{lem:deep 1}.
Then (iii) of Lem\-ma~\ref{prop:normalize, q-linear} and Lem\-ma~\ref{lem:eventual value of fv} give that
$(P^\phi, \fm,\hat a)$ is steep, eventually. Using now~$\ndeg S_{P_{\times \fm}}=0$ and~(ii) of Lemma~\ref{prop:normalize, q-linear} we obtain from Lemma~\ref{lem:eventually deep} that~$(P^\phi, \fm,\hat a)$ is deep, eventually.  
\end{proof}

\noindent
Lemma~\ref{lem:quasilinear refinement} and the previous lemma and its corollary now yield:

\begin{lemma}\label{prop:normalize} Suppose $K$ is $\d$-valued and $\upo$-free, and $\Gamma$ is divisible. Then
every $Z$-minimal slot in $K$ of order $r$ has a refinement $(P,\fm,\hat a)$ satisfying \textup{(i)}--\textup{(iii)} in Lem\-ma~\ref{prop:normalize, q-linear}.
\end{lemma}

\begin{cor}\label{cor:deepening} Suppose $K$ is $\d$-valued and $\upo$-free, and $\Gamma$ is divisible. Then
every $Z$-minimal slot in $K$ of order $r$ has a quasilinear refinement $(P,\fm,\hat a)$ such that $\nwt L_{P_{\times\fm}}=0$, and~$(P^\phi,\fm,\hat a)$ is deep, eventually.
\end{cor}

\subsection*{Approximating $Z$-minimal slots} In this subsection we set, as before, 
$$\fv\ :=\ \fv(L_{P_{\times\fm}}), $$ provided $L_{P_{\times \fm}}$ has order $r$. The next lemma is a key approximation result.  

\begin{lemma}\label{lem:good approx to hata}
Suppose  $(P,\fm,\hat a)$ is $Z$-minimal and steep,  and
$$\ddeg P_{\times\fm}\ =\ \ndeg P_{\times\fm}\ =\ 1,\qquad \ddeg S_{P_{\times_\fm}}\ =\ 0.$$ 
Then there exists an $a$ such that $\hat a-a \prec_{\Delta(\fv)} \fm$.
\end{lemma}
\begin{proof}
We can arrange $\fm=1$ and $P\asymp 1$. Then $\ddeg P=1$
gives $P_1\asymp 1$, so~$S_P(0)\asymp\fv$.
Take $Q,R_1,\dots,R_n\in K\{Y\}$ ($n\geq 1$) of order~$<r$ such that $$ P\ =\ Q+R_1Y^{(r)}+\cdots+R_n(Y^{(r)})^n,\qquad S_P\ =\ R_1+\cdots+nR_n(Y^{(r)})^{n-1}.$$
Then $R_1(0)=S_P(0)\asymp\fv$. As $\ddeg S_P=0$, this gives $S_P \sim R_1(0)$,
hence $$R\ :=\ P-Q\ \sim\ R_1(0)Y^{(r)}\ \asymp\ \fv\  \prec_{\Delta(\fv)}\ 1\ \asymp\ P,$$ so $P\sim_{\Delta(\fv)} Q$. Thus $Q\ne 0$, and $Q\notin Z(K,\hat a)$ because $\order Q < r$.
Now Lem\-ma~\ref{lem:notin Z(K,hata)} gives a refinement~$(P_{+a},\fn,\hat a-a)$ of $(P, 1, \hat a)$ such that~${\ndeg Q_{+a,\times \fn}=0}$ and $\fn\prec 1$.
We claim that then $\hat a -a\prec_{\Delta(\fv)} 1$. (Establishing this claim finishes the proof.)
Suppose the claim is false. Then $\hat a-a\asymp_{\Delta(\fv)} 1$, so $\fn \asymp_{\Delta(\fv)} 1$,
hence~$Q_{+a,\times\fn} \asymp_{\Delta(\fv)} Q_{+a}\asymp Q$ by [ADH, 4.5.1]. Likewise,
$R_{+a,\times\fn} \asymp_{\Delta(\fv)} R$.
Using~$P_{+a,\times\fn} =  Q_{+a,\times\fn} +  R_{+a,\times\fn}$ gives
$Q_{+a,\times\fn} \sim_{\Delta(\fv)}  P_{+a,\times\fn}$, so
$Q_{+a,\times\fn} \sim^\flat  P_{+a,\times\fn}$. 
Then $\ndeg  Q_{+a,\times\fn}=\ndeg P_{+a,\times\fn}=1$ by Lemma~\ref{lem:same ndeg} and Corollary~\ref{cor:ref 2n}, a contradiction.
\end{proof}

\noindent
Lemmas~\ref{lem:lower bd on ddeg} and~\ref{lem:good approx to hata}, and a remark following the definition of {\em deep\/}  give:
 
\begin{cor}\label{cor:good approx to hata, deg 1} 
If $(P,\fm,\hat a)$ is $Z$-minimal, steep, and linear, then there exists an $a$ such that $\hat a-a \prec_{\Delta(\fv)} \fm$.
\end{cor}

\begin{cor}\label{specialvariant}
Suppose $(P,\fm,\hat a)$ is $Z$-minimal, deep, and special. Then for all~$n\ge 1$ there is an~$a$ with $\hat a-a\prec \fv^n\fm$.
\end{cor} 
\begin{proof} We arrange $\fm=1$ in the usual way. Let $\Delta$ be the convex subgroup of $\Gamma$ that is cofinal in $v(\hat a - K)$. Lemma~\ref{lem:good approx to hata} gives an element 
$\gamma\in v(\hat a -K)$
with $\gamma\ge \delta/m$ for some~$m\ge 1$. Hence
$v(\hat a -K)$ contains for every $n\ge 1$ an element $>n\delta$.
\end{proof}

\noindent
Combining Lemma~\ref{lem:special dents} with Corollary~\ref{specialvariant} yields:

\begin{cor}\label{cor:closer to minimal holes} 
If $K$ is   $r$-linearly newtonian, $\upo$-free if $r>1$, and
$(P,\fm,\hat a)$ is $Z$-minimal and deep, then
for all $n\ge 1$ there is an~$a$ such that~$\hat a-a\prec \fv^n\fm$.
\end{cor}

\subsection*{Normal slots} 
We say that our slot~$(P,\fm,\hat a)$ in $K$, with linear part $L$, is {\bf normal}\/ if  $\order L=r$ and, with $\fv:=\fv(L)$ and $w:=\wt(P)$,\index{slot!normal}\index{normal!slot} 
\begin{itemize}
\item[(N1)] $\fv\prec^\flat 1$;
\item[(N2)] $(P_{\times\fm})_{> 1}\prec_{\Delta(\fv)} \fv^{w+1} (P_{\times\fm})_1$.  
\end{itemize}
Note that then $\fv\prec 1$, $\dwt(L)<r$, $(P, \fm, \hat a)$ is steep, and
\begin{equation}\label{eq:N3}
P_{\times\fm}\sim_{\Delta(\fv)} P(0)+(P_{\times\fm})_1\qquad\text{ (so $\ddeg P_{\times\fm} \le 1$).}
\end{equation}
If $\order L=r$, $\fv:=\fv(L)$, and $L$ is monic, then $(P_{\times\fm})_1\asymp\fv^{-1}$, so that~(N2) is then equivalent to:
$(P_{\times\fm})_{> 1}\prec_{\Delta(\fv)} \fv^{w}$. 
If $\deg P=1$, then $\order L=r$ and~(N2) automatically holds,  
hence $(P,\fm,\hat a)$ is normal iff it is steep.  
Thus by Lemma~\ref{lem:eventual value of fv}:

\begin{lemma}\label{lem:deg1 normal} 
If $\deg P=1$ and $\nwt(L)<r$, then $(P^\phi,\fm,\hat a)$ is normal, eventually.
\end{lemma}

\noindent
If $(P,\fm,\hat a)$ is normal, then so are
$(P_{\times\fn},\fm/\fn,\hat a/\fn)$ and
$(bP,\fm,\hat a)$ for  $b\neq 0$.
In particular, $(P,\fm,\hat a)$ is   normal iff $(P_{\times\fm},1,\hat a/\fm)$ is normal.  
If $(P,\fm,\hat a)$ is normal, then so is any equivalent slot. 
Hence by \eqref{eq:N3} and Lemmas~\ref{lem:lower bd on ddeg} and~\ref{lem:from cracks to holes}:

\begin{lemma}\label{lem:ddeg=dmul=1 normal} 
If $(P,\fm,\hat a)$ is normal, and $(P,\fm,\hat a)$ is $Z$-minimal or is a hole in~$K$, then $\ddeg P_{\times\fm} = \dval P_{\times\fm} = 1$.
\end{lemma}

\begin{example} 
Let $K\supseteq\R(\ex^x)$ be an $H$-subfield of $\T$, $\fm=1$, $r=2$. If $P=D+R$ where 
$$D\ =\ \ex^{-x}Y''-Y,\qquad R\ =\ f+\ex^{-4x}Y^5\quad (f\in K),$$
then $\fv=-\ex^{-x}\prec^\flat 1$, $P_1=D\sim -Y$, $w=2$, and
$P_{>1}=\ex^{-4x}Y^5 \prec_{\Delta(\fv)} \ex^{-3x}P_1$, so~$(P,1,\hat a)$ is normal.
However, if $P=D+S$ with $D$ as above and $S=f+\ex^{-3x}Y^5$ ($f\in K$), then
$P_{>1}=\ex^{-3x}Y^5 \succeq_{\Delta(\fv)} \ex^{-3x}P_1$, so $(P,1,\hat a)$ is not normal.
\end{example}

\begin{lemma}\label{lem:normal pos criterion}
Suppose $\order(L)=r$ and $\fv$ is such that \textup{(N1)} and~\textup{(N2)} hold, and $\fv(L)\asymp_{\Delta(\fv)} \fv$. Then 
$(P,\fm,\hat a)$ is   normal.
\end{lemma}
\begin{proof}
Put $\fw:=\fv(L)$. Then $[\fw]=[\fv]$, and so $\fv\prec^{\flat} 1$ gives $\fw\prec^{\flat} 1$. Also, 
$$(P_{\times\fm})_{> 1}\prec_{\Delta(\fv)} \fv^{w+1} (P_{\times\fm})_1\asymp_{\Delta(\fv)}\fw^{w+1}(P_{\times\fm})_1.$$ 
Hence (N1), (N2) hold with~$\fw$ in place of $\fv$. 
\end{proof}

\begin{lemma}\label{lem:normality comp conj} 
Suppose $(P,\fm,\hat a)$ is normal and $\phi\preceq 1$ is active. Then the slot~$(P^\phi,\fm,\hat a)$  in $K^\phi$ is normal.
\end{lemma}
\begin{proof}
We arrange $\fm=1$ and put $\fv:=\fv(L)$, $\fw:=\fv(L_{P^\phi})$.
Now $L_{P^\phi}=L^\phi$, so~$\fv\asymp_{\Delta(\fv)} \fw$ and $\fv\prec_{\phi}^\flat 1$ by Lemma~\ref{lem:v(Aphi)}. By [ADH, 11.1.1], $[\phi]<[\fv]$, and (N2) we have
$$(P^\phi)_{>1}\ =\ (P_{>1})^\phi\ \asymp_{\Delta(\fv)}\ P_{>1}\ \prec_{\Delta(\fv)}\ \fv^{w+1} P_1\ \asymp_{\Delta(\fv)}\ \fv^{w+1} P^\phi_1,$$
which by Lemma~\ref{lem:normal pos criterion} applied to 
$(P^\phi, 1, \hat a)$ in the role of $(P, \fm, \hat a)$
gives normality of~$(P^\phi, 1, \hat a)$.  
\end{proof}

\begin{cor}\label{cor:normal=>quasilinear}  
Suppose $(P,\fm,\hat a)$ is normal.  Then $(P,\fm,\hat a)$ is quasilinear.
\end{cor}
\begin{proof} Lemma~\ref{lem:lower bd on ndeg} gives $\ndeg P_{\times \fm}\ge 1$. The parenthetical remark after~\eqref{eq:N3} above and
Lemma~\ref{lem:normality comp conj} give $\ndeg P_{\times \fm}\le 1$.  
\end{proof}

\noindent
Combining Lemmas~\ref{lem:ddeg=dmul=1 normal} and \ref{lem:normality comp conj} yields:

\begin{cor} 
If $(P,\fm,\hat a)$ is normal and  linear, and $(P,\fm,\hat a)$ is $Z$-minimal or a hole in~$K$, then
$(P,\fm,\hat a)$ is deep.
\end{cor}

\noindent
There are a few occasions later where we need to change the ``monomial'' $\fm$ in~$(P,\fm,\hat a)$ while preserving key properties
of this slot. Here is what we need:  
{\sloppy
\begin{lemma}\label{ufm} Let $u\in K$, $u\asymp 1$. Then $(P,u\fm,\hat a)$ refines $(P,\fm,\hat a)$, and
if~${(P_{+a},\fn,\hat a -a)}$ refines $(P,\fm,\hat a)$, then so does~$(P_{+a},u\fn, \hat a -a)$. If $(P,\fm,\hat a)$ is qua\-si\-linear, respectively deep, respectively normal, then so is~$(P,u\fm,\hat a)$. 
\end{lemma}}

{\sloppy
\begin{proof} 
The refinement claims are clearly true, and quasilinearity is preserved since~$\ndeg P_{\times u\fm}=\ndeg P_{\times \fm}$ by [ADH, 11.2.3(iii)]. 
``Steep'' is preserved by Lemma~\ref{lem:steep1}, and hence ``deep''
is preserved using~\eqref{eq:separant fms}  and [ADH, 6.6.5(ii)]. Normality is preserved because steepness is, 
$$(P_{\times u\fm})_d\ =\ (P_d)_{\times u\fm}\ \asymp\ (P_d)_{\times \fm}\ =\ (P_{\times \fm})_d\quad\text{ for all~$d\in \N$}$$
by [ADH, 4.3, 4.5.1(ii)], and $\fv(L_{P_{\times u\fm}})\asymp \fv(L_{P_{\times\fm}})$ by Lemma~\ref{fvmult}. 
\end{proof} }

\noindent
Here is a useful invariance property of normal slots:

\begin{lemma}\label{lem:excev normal}
Suppose $(P,\fm,\hat a)$ is normal and $a\prec\fm$.  Then $L_P$ and $L_{P_{+a}}$ have or\-der~$r$. 
If  in addition   $K$ is $\upl$-free or $r=1$, then $\exc^{\ev}(L_{P})=\exc^{\ev}(L_{P_{+a}})$. 
\end{lemma}
\begin{proof} $L_{P_{\times \fm}}=L_P\fm$ (so $L_P$ has order $r$), and 
$L_{P_{+a,\times \fm}}=L_{P_{\times \fm, +a/\fm}}= L_{P_{+a}}\fm$. 
The slot $(P_{\times\fm},1,\hat a/\fm)$ in $K$ is normal and $a/\fm\prec 1$.  Thus we can apply Lemma~\ref{lem:linear part, new}(i)  to  $\hat K$, $P_{\times\fm}$, $a/\fm$ in place of $K$, $P$, $a$ to give $\order L_{P_{+a}} = r$. Next, applying
likewise Lemma~\ref{lem:linear part, split-normal, new}  with $L:=L_{P_{\times \fm}}$, $\fv:=\fv(L_{P_{\times \fm}})$, $m=r$, $B=0$,   gives 
$$L_P\fm - L_{P_{+a}}\fm\ =\ L_{P_{\times\fm}} - L_{P_{\times\fm,+a/\fm}}\ \prec_{\Delta(\fv)}\ \fv^{r+1}L_{P}\fm.$$
Hence, if $K$ is $\upl$-free, then $\exc^{\ev}(L_P\fm)\ =\ \exc^{\ev}(L_{P_{+a}}\fm)$ by Lemma~\ref{cor:excev stability}, so
$$ \exc^{\ev}(L_{P})\ =\ \exc^{\ev}(L_P\fm)+v(\fm)\ =\ \exc^{\ev}(L_{P_{+a}}\fm)+v(\fm)\ =\ \exc^{\ev}(L_{P_{+a}}).$$
If $r=1$ we obtain the same equality from Corollary~\ref{cor:excev stability, r=1}.
\end{proof}

\subsection*{Normality under refinements} 
In this subsection we study how normality behaves under more general refinements. This is not needed to prove the main result of this section, Theorem~\ref{mainthm}, but is included to
obtain useful variants of it.

\begin{prop}\label{normalrefine}
Suppose  $(P,\fm,\hat a)$ is normal. Let a refinement
$(P_{+a},\fm,\hat a-a)$ of $(P,\fm,\hat a)$ be given. Then this refinement is also normal. 
\end{prop}

{\sloppy
\begin{proof}
By the remarks following the definition of ``multiplicative conjugate'' in Sec\-tion~\ref{sec:holes} and after replacing the slots 
$(P,\fm,\hat a)$ and $(P_{+a},\fm,\hat a-a)$
in~$K$
by~$(P_{\times\fm},1,\hat a/\fm)$ and $\big(P_{\times\fm,+a/\fm},1,(\hat a-a)/\fm\big)$,
respectively, we arrange that $\fm=1$.  Let $\fv:=\fv(L_P)$.
By Lemma~\ref{lem:linear part, new} we have $\order(L_{P_{+a}})=r$, $\fv(L_{P_{+a}})\sim_{\Delta(\fv)} \fv$, and~$(P_{+a})_1\sim_{\Delta(\fv)} P_1$.
Using [ADH, 4.5.1(i)] we have for $d>1$ with~$P_d\ne 0$,
$$
(P_{+a})_{d}\ =\ \big((P_{\geq d})_{+a}\big){}_d\ \preceq\ (P_{\geq d})_{+a}\ \sim\  P_{\geq d}\ \preceq\ P_{>1},$$
and using (N2), this yields
$$ (P_{+a})_{>1}\ \preceq\ P_{>1}\ \prec_{\Delta(\fv)}\ \fv^{w+1}P_1\ \asymp\ \fv^{w+1}(P_{+a})_1.$$ 
Hence (N2) holds with $\fm=1$ and with $P$ replaced by $P_{+a}$. Thus $(P_{+a},1,\hat a-a)$
is normal,  by  Lemma~\ref{lem:normal pos criterion}.
\end{proof}}

\begin{prop}\label{easymultnormal}
Suppose   $(P,\fm,\hat a)$ is a normal hole in $K$,   $\hat{a}\prec \fn\preceq \fm$, and~$[\fn/\fm]\le\big[\fv(L_{P_{\times\fm}})\big]$. Then the refinement $(P,\fn,\hat a)$ of $(P,\fm,\hat a)$
is also  normal.
\end{prop}
\begin{proof} As in the proof of Lemma~\ref{lem:steep1}
we arrange $\fm=1$ and  
set $L:= L_P$, $\fv:=\fv(L)$, and $\tilde L:=L_{P_{\times\fn}}$, to obtain $[\fn]\le[\fv]$ and
$\fv(\tilde L)\asymp_{\Delta(\fv)} \fv$. 
Recall from [ADH, 4.3] that~$(P_{\times\fn})_{d}=(P_{d})_{\times\fn}$ for $d\in \N$.
For such $d$ we have by [ADH, 6.1.3],
$$
(P_{d})_{\times\fn}\ \asymp_{\Delta(\fv)}\ \fn^d P_d\ \preceq\ \fn^d P_{\geq d}.$$
In particular, $(P_{\times \fn})_1\asymp_{\Delta(\fv)} \fn P_1$.
By (N2) we also have, for $d>1$:
$$P_{\geq d}\ \preceq\ P_{>1}\ \prec_{\Delta(\fv)}\ \fv^{w+1} P_1.$$
By Lemma~\ref{lem:ddeg=dmul=1 normal} we have   $P \sim P_1$. For $d>1$ we have by [ADH, 6.1.3],
$$\fn^d P\ \asymp\ \fn^d P_1\ \asymp_{\Delta(\fv)}\ \fn^{d-1} (P_1)_{\times\fn}\ \preceq\ 
(P_1)_{\times\fn}\ =\ (P_{\times\fn})_1\ \preceq\ P_{\times\fn}$$
and thus 
$$(P_{\times\fn})_{d}\ =\ (P_d)_{\times\fn}\ \preceq_{\Delta(\fv)}\ \fn^d P_{\geq d}\ \prec_{\Delta(\fv)}\  \fv^{w+1} \fn^d P_1\ \preceq_{\Delta(\fv)}\ \fv^{w+1} (P_{\times\fn})_1.$$
Hence (N2) holds with $\fm$ replaced by $\fn$. Thus $(P,\fn,\hat a)$
is normal, using $\fv(\tilde L)\asymp_{\Delta(\fv)} \fv$ and Lem\-mas~\ref{lem:steep1} and~\ref{lem:normal pos criterion}.
\end{proof}

\noindent
From Lemma~\ref{lem:from cracks to holes} and Proposition~\ref{easymultnormal} we obtain:

\begin{cor}\label{corcorcor} Suppose $(P,\fm,\hat a)$ is normal and $Z$-minimal, $\hat a \prec \fn \preceq \fm$, and~$[\fn/\fm]\le \big[\fv(L_{P_{\times \fm}})\big]$. Then the refinement $(P,\fn, \hat a)$ of $(P,\fm, \hat a)$ is also normal.
\end{cor}

\noindent
{\it In the rest of this subsection $\fm=1$, $\hat a \prec \fn\prec 1$, $\order(L_P)=r$, and $[\fv]<[\fn]$ where~$\fv:=\fv(L_P)$.}\/ So $(P, \fn,\hat a)$ refines $(P,1,\hat a)$, $L_{P_{\times \fn}}=L_P\fn$, and
$\order L_{P_{\times \fn}}=r$. 

{\sloppy
\begin{lemma}\label{lem:normal from steep}
Suppose $(P,1,\hat a)$ is  steep, $\fv(L_{P_{\times\fn}})\preceq\fv$, and $P_{>1}\preceq P_1$.
Then~$(P,\fn,\hat a)$ is normal.
\end{lemma}}
\begin{proof}
Put $\fw:=\fv(L_{P_{\times\fn}})$. Then~$[\fw]<[\fn]$ by Corollary~\ref{cor:An}, and $\fw\preceq \fv \prec^{\flat} 1$ gives~$\fw\prec^\flat 1$.  It remains to show that
$(P_{\times\fn})_{>1} \prec_{\Delta(\fw)} \fw^{w+1} (P_{\times\fn})_1$.
Using~$[\fn]>[\fw]$ it is enough that
$(P_{\times\fn})_{>1} \prec_\Delta \fw^{w+1} (P_{\times\fn})_1$, 
where $\Delta:=\Delta(\fn)$.  Since
$\fw\asymp_\Delta 1$, it is even enough that $(P_{\times\fn})_{>1} \prec_{\Delta}  (P_{\times\fn})_1$, to
be derived below. 
Let~$d>1$. Then by~[ADH, 6.1.3] and $P_d\preceq P_{>1}\preceq P_1$ we have
$$(P_{\times\fn})_d\ =\ (P_d)_{\times\fn}\
\asymp_\Delta\  P_d \fn^d\ \preceq\ P_1\,\fn^d.$$
 In view of
$\fn\prec_{\Delta} 1$ and $d>1$ we have
$$ P_1\,\fn^d\ \prec_\Delta\ P_1\, \fn\ \asymp_\Delta\ (P_1)_{\times\fn}\ =\ (P_{\times\fn})_1,$$
using again [ADH, 6.1.3]. Thus $(P_{\times\fn})_d \prec_\Delta (P_{\times\fn})_1$, as promised.
\end{proof}

\begin{cor}\label{cor:normal for small q, prepZ}
If $(P, 1, \hat a)$ is normal and~${\fv(L_{P_{\times\fn}})\preceq\fv}$, then $(P,\fn,\hat a)$ is normal.
\end{cor}

\noindent
In the next lemma and its corollary $K$ is
$\d$-valued and for every $q\in \Q^{>}$ there is given an element $\fn^q$ of $K^\times$ such that $(\fn^q)^\dagger=q\fn^\dagger$;
the remark before Lemma~\ref{qlA} gives $v(\fn^q)=qv(\fn)$ for $q\in \Q^{>}$. 
Hence for $0 < q\leq 1$ in $\Q$ we have~$\hat a\prec\fn\preceq\fn^q\prec 1$, so $(P,\fn^q,\hat a)$ refines~$(P,1,\hat a)$. 

\begin{lemma}\label{lem:normal for small q}
Suppose $(P,1,\hat a)$ is  steep   and $P_{>1}\preceq P_1$.  Then $(P,\fn^q,\hat a)$ is normal,  for all but finitely many $q\in \Q$ with $0<q \le 1$. 
\end{lemma}

{\sloppy
\begin{proof} We have $\fn^\dagger\succeq 1$ by $\fn\prec \fv\prec 1$ and $\fv^\dagger\succeq 1$.   
Lemma~\ref{lem:nepsilon} gives~$\fv(L_{P_{\times\fn^q}})\preceq \fv$ for all but finitely many $q\in \Q^{>}$. Suppose $\fv(L_{P_{\times\fn^q}})\preceq \fv$, $0<q\le 1$ in $\Q$. Then~$(P,\fn^q,\hat a)$ is normal by Lemma~\ref{lem:normal from steep} applied with $\fn^q$ instead of $\fn$.
\end{proof}}

\begin{cor}\label{cor:normal for small q}
If $(P, 1, \hat a)$ is normal, 
then $(P,\fn^q,\hat a)$ is  normal for all but finitely many $q\in \Q$ with~$0<q\le 1$.
\end{cor}

\subsection*{Normalizing} {\it If in this subsection
$\order(L_{P_{\times \fm}})=r$, then $\fv:=\fv(L_{P_{\times\fm}})$.}\/
Towards proving that normality can always be achieved we first show:

\begin{lemma}\label{lem:normalizing}
Suppose $\Gamma$ is divisible, $(P,\fm,\hat a)$ is a deep hole in $K$, and $\hat a-a \prec \fv^{w+2}\fm$ for some $a$. 
Then $(P,\fm,\hat a)$  has a refinement that is deep and normal.
\end{lemma}
\begin{proof}
Replacing $(P,\fm,\hat a)$ by $(P_{\times \fm}, 1, \hat a/\fm)$ and renaming we arrange $\fm=1$. Take~$a$ such that
$\hat a -a\prec \fv^{w+2}$.   
For $e:=w+\frac{3}{2}$, let $\fv^e$ be an element of~$K^\times$ with~$v(\fv^e)=e\,v(\fv)$. {\em Claim}: the refinement
$(P_{+a},\fv^e,\hat a-a)$ of~$(P,1,\hat a)$  is deep and normal.
By Lemma~\ref{lem:deep 2}, $(P_{+a},\fv^e,\hat a-a)$  is deep, so
we do have $\order(L_{P_{+a,\times\fv^e}})=r$ 
and $\fv(L_{P_{+a,\times\fv^e}})\prec^\flat 1$.
Lemma~\ref{lem:deep 2} also yields $\fv(L_{P_{+a,\times\fv^e}}) \asymp_{\Delta(\fv)} \fv$. Since~$\ddeg P=\dval P=1$,
we can use Corollary~\ref{cor:ref 2} for $\fn=\fv^e$ and for $\fn=1$ to obtain 
$$\ddeg P_{+a,\times\fv^e}\ =\ \dval P_{+a,\times\fv^e}\ =\ \ddeg P_{+a}\ =\ \dval P_{+a}=1$$
and thus 
$(P_{+a,\times\fv^e})_1 \sim P_{+a,\times\fv^e}$; also
$P_1\sim P\sim P_{+a}\sim (P_{+a})_1$,  where $P\sim P_{+a}$ follows from
$a\prec 1$ and [ADH, 4.5.1(i)]. 
Now let $d>1$. Then 
\begin{align*}
(P_{+a,\times\fv^e})_d	&\ \asymp_{\Delta(\fv)}\ (\fv^{e})^d (P_{+a})_d 
						 \ \preceq\ (\fv^{e})^d P_{+a}
						 \ \sim\ (\fv^{e})^d(P_{+a})_1 \\
						&\ \asymp_{\Delta(\fv)}\ (\fv^{e})^{d-1} (P_{+a,\times\fv^e})_1\ \prec_{\Delta(\fv)} \fv^{w+1} (P_{+a,\times\fv^e})_1,
\end{align*}
using [ADH, 6.1.3] for $\asymp_{\Delta(\fv)}$. So $(P_{+a},\fv^e,\hat a-a)$ is normal by Lemma~\ref{lem:normal pos criterion}.
\end{proof}

\noindent
We can now finally show:

\begin{theorem}\label{mainthm} 
Suppose $K$ is $\upo$-free and $r$-linearly newtonian, and $\Gamma$ is divisible. Then every $Z$-minimal slot in $K$
of order $r$ has a refinement $(P,\fm,\hat a)$ such that~$(P^\phi,\fm,\hat a)$ is deep and normal, eventually. 
\end{theorem} 
\begin{proof}
By Lemma~\ref{lem:from cracks to holes} it is enough to show this for $Z$-minimal holes in $K$ of or\-der~$r$.
Given such a  hole in $K$, use Corollary~\ref{cor:deepening} to refine it to a hole $(P,\fm,\hat a)$ such that $(P^\phi,\fm,\hat a)$ is deep, eventually.  
Replacing $(P,\fm,\hat a)$ by~$(P^\phi,\fm,\hat a)$ for a suitable active $\phi\preceq 1$ we arrange that
$(P,\fm,\hat a)$ itself is deep.
Then an appeal to Corollary~\ref{cor:closer to minimal holes} followed by an application of Lemma~\ref{lem:normalizing} yields a deep and normal refinement of  $(P,\fm,\hat a)$. Now apply Lemma~\ref{lem:normality comp conj} to this refinement. 
\end{proof}

\noindent
Next we indicate some variants of Theorem~\ref{mainthm}:

\begin{cor}\label{cor:mainthm} Suppose $K$ is $\d$-valued
and $\upo$-free, and $\Gamma$ is divisible. Then every minimal hole in $K$
of order $r$ has a refinement $(P,\fm,\hat a)$ such that $(P^\phi,\fm,\hat a)$ is deep and normal, eventually.
\end{cor}
\begin{proof} Given a minimal hole in $K$ of order $r$, use Corollary~\ref{cor:deepening} to refine it to a hole~$(P,\fm,\hat a)$ in $K$ such that $\nwt L_{P_{\times \fm}}=0$ and $(P^\phi,\fm,\hat a)$ is deep, eventually.
If~${\deg P=1}$, then $(P^\phi,\fm,\hat a)$ is normal, eventually, by Lemma~\ref{lem:deg1 normal}. 
If~${\deg P>1}$, then~$K$ is $r$-linearly newtonian by
Corollary~\ref{degmorethanone}, so we can use Theorem~\ref{mainthm}.
\end{proof}

\noindent
For $r=1$ we can follow the proof of Theorem~\ref{mainthm},  using Corollary~\ref{cor:deepening, q-linear} in place of Corollary~\ref{cor:deepening}, to obtain:

\begin{cor}  \label{mainthm, r=1} 
If $K$ is  $1$-linearly newtonian and $\Gamma$ is divisible, then every quasilinear $Z$-minimal slot in $K$
of order $1$ has a refinement $(P,\fm,\hat a)$ such that~$(P^\phi,\fm,\hat a)$ is deep and normal, eventually. 
\end{cor}

\noindent
Here is another variant of Theorem~\ref{mainthm}:

\begin{prop}\label{varmainthm} If $K$ is $\d$-valued and $\upo$-free, and
$\Gamma$ is divisible, then every $Z$-minimal special slot
in $K$ of order $r$ has a refinement $(P,\fm,\hat a)$ such that $(P^\phi,\fm,\hat a)$ is deep and normal, eventually. 
\end{prop}

\noindent
To establish this proposition we follow the proof of Theorem~\ref{mainthm}, using Lem\-ma~\ref{speciallemma} to preserve specialness in the initial refining. Corollary~\ref{specialvariant} takes over the role of Corollary~\ref{cor:closer to minimal holes}
in that proof.   

\medskip
\noindent
For linear slots in $K$ we can   weaken the hypotheses of Theorem~\ref{mainthm}:

\begin{cor}\label{mainthm deg 1}
Suppose   $\deg P=1$.  Then $(P,\fm,\hat a)$ has a refinement~$(P,\fn,\hat a)$ such that $(P^\phi,\fn,\hat a)$ is deep and normal, eventually.
Moreover, if $K$ is $\upl$-free and~${r>1}$, then $(P^\phi,\fm,\hat a)$ is deep and normal, eventually.
\end{cor}
\begin{proof}
By the remarks before Lemma~\ref{lem:deg1 normal}, $(P,\fm,\hat a)$ is normal iff it is steep. Moreover,
if $(P,\fm,\hat a)$ is normal, then it is quasilinear by Corollary~\ref{cor:normal=>quasilinear}, and hence~$(P^\phi,\fm,\hat a)$ is deep and normal, eventually,
by  the remarks before Example~\ref{ex:order 1 linear steep} and Lemma~\ref{lem:normality comp conj}.
By Lemma~\ref{lem:achieve steep}, $(P,\fm,\hat a)$ has a refinement~$(P,\fn,\hat a)$ such that $(P^\phi,\fn,\hat a)$ is steep, eventually.
This yields the first part. The second part follows from Corollary~\ref{coruplnwteq} and Lemma~\ref{lem:deg1 normal}.
\end{proof}

\begin{cor}\label{mainthm order 1}
Suppose $K$ is $\upl$-free, $\Gamma$ is divisible, and $(P, \fm, \hat a)$ is a quasilinear minimal hole in $K$ of order $r=1$.
Then $(P, \fm, \hat a)$ has a refinement~$(Q,\fn,\hat b)$ such that $(Q^\phi,\fn,\hat b)$ is deep and normal, eventually.
\end{cor}
\begin{proof} The case $\deg P=1$ is part of Corollary~\ref{mainthm deg 1}.
If $\deg P>1$, 
then $K$ is $1$-linearly newtonian by Lemma~\ref{lem:no hole of order <=r, deg 1}, so we can use 
Corollary~\ref{mainthm, r=1}.
\end{proof}

\subsection*{Improving normality} 
{\it In this subsection $L:=L_{P_{\times\fm}}$.}\/ 
Note that if $(P,\fm,\hat a)$ is a normal hole in $K$, then $P_{\times\fm}\sim (P_{\times \fm})_1$ by Lemma~\ref{lem:ddeg=dmul=1 normal}.
We call our slot~$(P,\fm, \hat a)$ in $K$ {\bf strictly normal\/} if it is normal, but
with the condition~(N2) replaced by the stronger condition\index{slot!strictly normal}\index{strictly normal}\index{normal!strictly} 
\begin{enumerate}
\item[(N2s)] $(P_{\times\fm})_{\ne 1}\prec_{\Delta(\fv)} \fv^{w+1} (P_{\times \fm})_1$.
\end{enumerate}
Thus for normal $(P,\fm, \hat a)$ and $\fv=\fv(L)$ we have:
$$(P,\fm,\hat a) \text{ is strictly normal }\Longleftrightarrow\  P(0) \prec_{\Delta(\fv)} \fv^{w+1} (P_{\times \fm})_1.$$ 
So if $(P,\fm, \hat a)$ is  normal and $P(0)=0$, then $(P,\fm,\hat a)$ is strictly normal.
Note that if  $(P,\fm, \hat a)$ is strictly normal, then 
$$P_{\times \fm}\sim_{\Delta(\fv)} (P_{\times \fm})_1\qquad
\text{(and hence $\ddeg P_{\times\fm}=1$).}$$
If $(P,\fm,\hat a)$ is strictly normal, then so are
$(P_{\times\fn},\fm/\fn,\hat a/\fn)$ and
$(bP,\fm,\hat a)$ for  $b\neq 0$.
Thus $(P,\fm,\hat a)$ is strictly normal iff $(P_{\times\fm},1,\hat a/\fm)$ is strictly normal. 
If $(P,\fm,\hat a)$ is strictly normal, then so is every equivalent slot in~$K$.
The proof of Lemma~\ref{ufm} shows that if $(P,\fm,\hat a)$ is strictly normal and~${u\in K}$, $u\asymp 1$, then $(P,u\fm,\hat a)$ is also strictly normal. 
The analogue of
Lem\-ma~\ref{lem:normal pos criterion} goes through, with
$(P_{\times \fm})_{ \ne 1}$ instead of $(P_{\times \fm})_{>1}$
in the proof:

\begin{lemma}\label{lem:strongnormal pos criterion}
Suppose $\order(L)=r$ and $\fv$ are such that \textup{(N1)} and~\textup{(N2s)} hold, and $\fv(L)\asymp_{\Delta(\fv)} \fv$. Then 
$(P,\fm,\hat a)$ is strictly  normal.
\end{lemma}

\noindent
Lemma~\ref{lem:normality comp conj} goes likewise through with ``strictly normal'' instead of ``normal'':

\begin{lemma}\label{lem:normality comp conj, strong}
If $(P,\fm,\hat a)$ is strictly normal and $\phi\preceq 1$ is active, then the slot~$(P^\phi,\fm,\hat a)$  in $K^\phi$ is strictly normal.
\textup{(}Hence if $(P,\fm,\hat a)$ is strictly normal, then $(P,\fm,\hat a)$ is quasilinear, and if
in addition $(P,\fm,\hat a)$ is linear, then it is deep.\textup{)} 
\end{lemma}

\noindent
As to Proposition~\ref{normalrefine}, here is a weak version for strict normality:

\begin{lemma} \label{stronglynormalrefine}
Suppose $(P, \fm, \hat a)$ is a strictly normal hole in $K$ and $\hat a -a\prec_{\Delta(\fv)} \fv^{r+w+1}\fm$ where 
$\fv:=\fv(L)$. Then its refinement
$(P_{+a}, \fm, \hat a -a)$ is also strictly normal.
\end{lemma}

\begin{proof} 
As in the proof of Proposition~\ref{normalrefine} we 
arrange $\fm=1$.
We can also assume~$P_1\asymp 1$. From $P=P(0)+P_1+P_{>1}$ we get 
$$P(a)\ =\ P(0)+ P_1(a) + P_{>1}(a),$$
where $P(0)\prec_{\Delta(\fv)} \fv^{w+1}$ and $P_{>1}(a) \preceq P_{>1}\prec_{\Delta(\fv)}\fv^{w+1}$ by (N2s) and $a\prec 1$;
 we   show that also $P_1(a)\prec_{\Delta(\fv)}\fv^{w+1}$. To see this note that
$$0\ =\ P(\hat a)\ =\ P(0)+ P_1(\hat a)+ P_{>1}(\hat a),$$ 
where as before $P(0), P_{>1}(\hat a) \prec_{\Delta(\fv)} \fv^{w+1}$, so $P_1(\hat a)\prec_{\Delta(\fv)} \fv^{w+1}$. 
Lemma~\ref{lem:diff operator at small elt} applied to~$(\hat K, \preceq_{\Delta(\fv)}, P_1)$ in place of $(K, \preceq, P)$, with $m=w+1$, $y=a-\hat a$,  yields~$P_1(a-\hat a)\prec_{\Delta(\fv)} \fv^{w+1}$,
hence $$P_1(a)\ =\ P_1(a -\hat a) + P_1(\hat a)\ \prec_{\Delta(\fv)}\ \fv^{w+1}$$ as claimed.  It remains to
use $\fv(L_{P_{+a}})\asymp_{\Delta(\fv)} \fv$ and the normality of $(P_{+a}, 1, \hat a -a)$ obtained from
Proposition~\ref{normalrefine} and its proof. 
\end{proof}

\noindent
We also have a version of Lemma~\ref{stronglynormalrefine} for $Z$-minimal slots, obtained from that lemma via Lemma~\ref{lem:from cracks to holes}: 

\begin{lemma} \label{stronglynormalrefine, cracks}
Suppose $(P, \fm, \hat a)$ is $Z$-minimal and strictly normal. Set $\fv:=\fv(L)$,
and suppose $\hat a -a\prec_{\Delta(\fv)} \fv^{r+w+1}\fm$. Then the refinement
$(P_{+a}, \fm, \hat a -a)$ of~$(P, \fm, \hat a)$ is   strictly normal.
\end{lemma}

\noindent
Next two  versions of Proposition~\ref{easymultnormal}:

\begin{lemma}\label{lem:strongly normal refine, 1}
Suppose $(P,\fm,\hat a)$ is a strictly normal hole in $K$, $\hat a\prec\fn\preceq\fm$, and~$[\fn/\fm]< \big[\fv(L)\big]$. Then the refinement $(P,\fn,\hat a)$ of $(P,\fm,\hat a)$
 is strictly normal.
 \end{lemma}
\begin{proof}
As in the proof of Proposition~\ref{easymultnormal} we arrange $\fm=1$ and,
setting  $\fv:=\fv(L)$, $\tilde L:=L_{P_{\times\fn}}$, show
that $\order(\tilde L)=r$, $\fv(\tilde L)\asymp_{\Delta(\fv)} \fv$, and that~(N2) holds with~$\fm$ replaced by $\fn$. 
Now $[\fn]< [\fv]$ yields $\fn\asymp_{\Delta(\fv)} 1$; together with $(P_{\times\fn})_1\asymp_{\Delta(\fv)} \fn P_1$
this gives 
$P(0) \prec_{\Delta(\fv)} \fv^{w+1} P_1 \asymp_{\Delta(\fv)} \fv^{w+1} (P_{\times\fn})_1$. 
Hence (N2s) holds with $\fm$ replaced by~$\fn$. 
Lemma~\ref{lem:strongnormal pos criterion} now yields that  $(P,\fn,\hat a)$
is strictly normal.
\end{proof}

\begin{lemma}\label{lem:strongly normal refine, 2}  
Suppose $(P,\fm,\hat a)$ is a strictly normal hole in $K$ and $\hat a \prec_{\Delta(\fv)} \fm$ where $\fv:=\fv(L)$.
Assume also that for all $q\in\Q^>$ there is given an element $\fv^q$ of~$K^\times$ with $v(\fv^q)=q\,v(\fv)$.
Then for all sufficiently small $q\in\Q^>$ and~$\fn$ with~$\fn\asymp\fv^q\fm$ we have: $\hat a \prec \fn$
and the refinement $(P,\fn,\hat a)$ of $(P,\fm,\hat a)$ is strictly normal.
\end{lemma}
\begin{proof}
We arrange $\fm=1$ as usual, and
take $q_0\in\Q^>$ with $\hat a \prec \fv^{q_0}$ and~$P(0) \prec_{\Delta(\fv)} \fv^{w+1+q_0} P_1$.
Let  $q\in\Q$, $0<q\leq q_0$, and suppose $\fn\asymp\fv^q$. Then $(P,\fn,\hat a)$ is a refinement of $(P,1,\hat a)$, and the proof of Proposition~\ref{easymultnormal} gives: $\tilde L:=L_{P_{\times\fn}}$ has order~$r$ with $\fv(\tilde L)\asymp_{\Delta(\fv)} \fv$, $\fn P_1\asymp_{\Delta(\fv)}(P_{\times\fn})_1$, and (N2) holds with $\fm$ replaced by~$\fn$. 
Hence $$P(0) \prec_{\Delta(\fv)} \fv^{w+1+q_0} P_1 \preceq \fv^{w+1} \fn P_1  \asymp_{\Delta(\fv)} \fv^{w+1} (P_{\times\fn})_1.$$ 
As in the proof of the  last lemma we conclude that $(P,\fn,\hat a)$ is strictly normal.
\end{proof}

\begin{remarkNumbered} \label{rem:strongly normal refine, 2}
In Lemmas~\ref{lem:strongly normal refine, 1} and \ref{lem:strongly normal refine, 2} we assumed that $(P,\fm,\hat a)$ is a strictly normal hole in $K$. By Lemma~\ref{lem:from cracks to holes} these lemmas go through if this hypothesis is replaced by
``$(P,\fm,\hat a)$ is a strictly normal $Z$-minimal slot in $K$''.
\end{remarkNumbered}

\noindent
We now turn to refining  a given normal hole to a  strictly normal hole. We only do this under additional hypotheses, tailored so that we may employ Lemma~\ref{lem:nepsilon, refined}. Therefore we assume in the rest of this subsection: {\em  $K$ is $\d$-valued and for all $\fv$ and $q\in\Q^>$   we are given
an element $\fv^q$ of $K^\times$ with $(\fv^q)^\dagger=q\fv^\dagger$.} 
Note that then $v(\fv^q)=q\,v(\fv)$ for such $q$. (In particular, $\Gamma$ is divisible.)  
We also adopt the
convention that if
$\order L=r$, then $\fv:=\fv(L)$.

\begin{lemma}\label{lem:achieve strong normality}
Suppose  $(P,\fm,\hat a)$ is a normal hole in $K$ and   $\hat a-a\preceq  \fv^{w+2}\fm$. Then the refinement
$(P_{+a},\fm,\hat a-a)$  of
 $(P,\fm,\hat a)$  is strictly normal.
\end{lemma}
\begin{proof}
As usual we arrange that $\fm=1$.  By Proposition~\ref{normalrefine}, 
$(P_{+a},1,\hat a-a)$ is  normal; the proof of this proposition 
gives
$\order(L_{P_{+a}})=r$, $\fv(L_{P_{+a}})\sim_{\Delta(\fv)} \fv$, $(P_{+a})_1 \sim_{\Delta(\fv)} P_1$,
and (N2) holds with $\fm=1$ and~$P$ replaced by $P_{+a}$.
It  remains to show that
$P_{+a}(0) \prec_{\Delta(\fv)} \fv^{w+1} (P_{+a})_1$, equivalently, $P(a)\prec_{\Delta(\fv)} \fv^{w+1} P_1$.

Let $\hat L:=L_{P_{+\hat a}}\in\hat K[\der]$ and $R:=P_{>1}\in K\{Y\}$;
note that $P_{(\i)}=R_{(\i)}$ for $\abs{\i}>1$ and $R  \prec_{\Delta(\fv)} \fv^{w+1} P_1$. Hence
Taylor expansion and $P(\hat a)=0$ give
\begin{align*}
P(a)	&\ = \ P(\hat a)+\hat L(a-\hat a) + \sum_{\abs{\i} > 1} P_{(\i)}(\hat a)\cdot(a-\hat a)^\i \\
		&\ =\ \hat L(a-\hat a)+\sum_{\abs{\i} > 1} R_{(\i)}(\hat a)\cdot(a-\hat a)^\i \\
		& \qquad\qquad\qquad\text{where $R_{(\i)}(\hat a)\cdot(a-\hat a)^\i \ \prec_{\Delta(\fv)}\ \fv^{w+1} P_1$ for $\abs{\i} > 1$,}
\end{align*} 
so it is enough to show $\hat L(a-\hat a)\prec_{\Delta(\fv)}\fv^{w+1} P_1$.
Lemma~\ref{lem:linear part, new} applied to $(\hat K, \hat a)$ in place of $(K, a)$ gives
$\order \hat L = r$ and $\hat L\sim_{\Delta(\fv)} L$. Since $\hat K$ is $\d$-valued,
Lemma~\ref{lem:nepsilon, refined}
yields a $q\in\Q$ with $w+1<q\leq w+2$ and a $\fw$ such that $\hat L\fv^q \asymp \fw\, \fv^q\, \hat L$ where~$[\fw]\leq [\fv^\dagger]$ 
and 
hence $\fw \asymp_{\Delta(\fv)} 1$ (see the remark before Lemma~\ref{lem:steep1}).
With~$\fn\asymp a-\hat a$ we have $\fn\preceq\fv^{w+2}\preceq\fv^q\prec_{\Delta(\fv)} \fv^{w+1}$ and therefore
$$\hat L(a-\hat a)\ \preceq\ \hat L\fn\ \preceq\ \hat L\fv^q\ \asymp\ \fw\, \fv^q\, \hat L\ \asymp_{\Delta(\fv)}\ \fv^q \hat L\ \prec_{\Delta(\fv)}\ \fv^{w+1} \hat L.$$
Hence $\hat L(a-\hat a)\prec_{\Delta(\fv)}\fv^{w+1} P_1$ as required.
\end{proof}

\noindent
In particular, if  $(P,\fm,\hat a)$ is a normal hole in $K$ and   $\hat a\preceq  \fv^{w+2}\fm$, then  $(P,\fm,\hat a)$ is strictly  normal. 

\begin{cor}\label{cor:achieve strong normality, 1} 
Suppose $(P,\fm,\hat a)$ is $Z$-minimal, deep, and normal. If $(P,\fm,\hat a)$ is special, then
$(P,\fm,\hat a)$ has a deep and 
strictly normal refinement~$(P_{+a},\fm,{\hat a-a})$ where $\hat a-a\prec_{\Delta(\fv)}\fm$ and $\fv(L_{P_{+a,\times\fm}})\asymp_{\Delta(\fv)}\fv$. \textup{(}Note that
if $K$ is $r$-linearly newtonian, and  $\upo$-free if $r>1$, then $(P,\fm,\hat a)$ is special by Lemma~\ref{lem:special dents}.\textup{)} 
\end{cor}
\begin{proof} By Lemma~\ref{lem:from cracks to holes} we arrange that $(P,\fm,\hat a)$ is a hole in $K$.
If $(P,\fm,\hat a)$ is special, Corollary~\ref{specialvariant}
gives an  $a$ such that $\hat a-a\preceq  \fv^{w+2}\fm$,  and then the refinement
$(P_{+a},\fm,\hat a-a)$  of $(P,\fm,\hat a)$  is strictly normal by Lemma~\ref{lem:achieve strong normality}, and deep 
with $\fv(L_{P_{+a,\times\fm}})\asymp_{\Delta(\fv)}\fv$ by Lemma~\ref{lem:deep 2}.
\end{proof}

\noindent
This leads to a useful variant of Theorem~\ref{mainthm}:

\begin{cor}\label{cor:achieve strong normality, 2}
Suppose $K$ is $\upo$-free and $r$-linearly newtonian. Then every $Z$-minimal slot in $K$ of order~$r$ has a refinement
$(P,\fm,\hat a)$ such that~$(P^\phi,\fm,\hat a)$ is deep and strictly normal, eventually.
\end{cor}
\begin{proof}
Let a $Z$-minimal slot in $K$ of order~$r$ be given. Use Theorem~\ref{mainthm} to refine it to a slot $(P,\fm,\hat a)$ in $K$
with an active $\phi_0$ such that the slot 
$(P^{\phi_0},\fm,\hat a)$ in~$K^{\phi_0}$ is deep and normal.
Corollary~\ref{cor:achieve strong normality, 1} gives a deep and  strictly normal refinement~$(P^{{\phi}_0}_{+a},\fm,\hat a-a)$ of $(P^{\phi_0},\fm,\hat a)$.  By 
 Lemma~\ref{lem:normality comp conj, strong} the slot 
 $(P^\phi_{+a},\fm, {\hat a -a})$ in~$K^\phi$ is deep and strictly normal, for all active $\phi\preceq \phi_0$ (in $K$). Thus $(P_{+a},\fm, \hat a -a)$ refines the original $Z$-minimal slot in $K$ and has the desired property. 
\end{proof}

\noindent
Corollaries~\ref{degmorethanone} and~\ref{cor:achieve strong normality, 2} have the following consequence:

\begin{cor}\label{cor:achieve strong normality, 3} Suppose $K$ is $\upo$-free. Then
every minimal hole in $K$ of order $r$ and degree~$>1$ has a refinement  $(P,\fm,\hat a)$ such that $(P^\phi,\fm,\hat a)$ is deep and strictly normal, eventually. 
\end{cor}

\noindent
Corollary~\ref{cor:achieve strong normality, 1} also gives the following variant of Corollary~\ref{cor:achieve strong normality, 2}, where the role of  Theorem~\ref{mainthm}  in its proof is taken over by
Proposition~\ref{varmainthm}:

\begin{cor}\label{cor:achieve strong normality, 4} Suppose $K$ is $\upo$-free. Then every $Z$-minimal special  slot  in~$K$ of order~$r$ has a refinement~$(P,\fm,\hat a)$ such that $(P^\phi,\fm,\hat a)$
is deep and strictly normal, eventually.
\end{cor}

\section{Isolated Slots}\label{sec:isolated} 

\noindent
In this short section we study the concept of isolation, which plays well together with normality. 
{\em Throughout this section $K$ is an $H$-asymptotic field with small derivation and with rational asymptotic integration.}\/ 
We let $a$, $b$ range over $K$ and~$\phi$, $\fm$, $\fn$, $\fw$ over $K^\times$. We also let
$(P,\fm,\hat a)$ be a slot in $K$ of order $r\geq 1$.
Recall that~$v(\hat a-K)$ is a cut in $\Gamma$ without largest element.
Note that $v\big( (\hat a-a) - K \big) = v(\hat a-K)$
and~$v( \hat a \fn - K) = v(\hat a-K)+v\fn$. 

\begin{definition}\label{def:isolated}
We say that $(P,\fm,\hat a)$  is {\bf isolated}\index{slot!isolated} if for all $a\prec\fm$,
$$\order(L_{P_{+a}})=r\ \text{ and }\ \exc^{\ev}(L_{P_{+a}}) \cap v(\hat a-K)\ <\  v(\hat a-a);$$
equivalently, for all $a\prec \fm$:  $\order(L_{P_{+a}})=r$ and whenever
$\fw \preceq \hat a-a$ is such that~$v(\fw) \in \exc^{\ev}(L_{P_{+a}})$, then 
$\fw\prec \hat a-b$ for all $b$.
\end{definition}

\noindent
In particular,  if $(P,\fm,\hat a)$ is isolated, then $v(\hat a)\notin \exc^{\ev}(L_P)$. If $(P,\fm,\hat a)$ is isolated, then so is every equivalent slot in~$K$,
as well as $(bP,\fm,\hat a)$ for $b\neq 0$ and the slot~$(P^\phi,\fm,\hat a)$ in $K^\phi$ for active $\phi$ in $K$. Moreover:

\begin{lemma}\label{lem:isolated refinement}
If $(P,\fm,\hat a)$ is isolated, then so is any refinement~$(P_{+a},\fn,\hat a-a)$  of~it. 
\end{lemma}
\begin{proof} For the case $\fn=\fm$, use $v\big( (\hat a-a) - K \big) = v(\hat a-K)$. The case $a=0$ is clear. The general case reduces to these two special cases.
\end{proof}

{\sloppy
\begin{lemma}\label{lem:isolated}
Suppose  $(P,\fm,\hat a)$ is isolated. Then the multiplicative con\-ju\-gate~$(P_{\times\fn},\fm/\fn,\hat a/\fn)$ of 
$(P,\fm,\hat a)$ by $\fn$ is isolated.
\end{lemma}}
\begin{proof} Let $a\prec\fm/\fn$. Then $a\fn\prec \fm$, so
$\order(L_{P_{\times \fn, +a}})=\order(L_{P_{+a\fn,\times \fn}})=\order(L_{P_{+a\fn}})=r$. Suppose $\fw\preceq (\hat a/\fn)-a$ and
$ v(\fw) \in \exc^{\ev}\big(L_{P_{\times\fn,+a}}\big)$.
Now $L_{P_{\times\fn,+a}} = L_{P_{+a\fn,\times\fn}} = L_{P_{+a\fn}}\fn$  and  thus $\fw\fn \preceq \hat a-a\fn$, 
$v(\fw\fn)\in \exc^{\ev}\big(P_{+a\fn}\big)$.
But~$(P,\fm,\hat a)$ is isolated, so $v(\fw\fn)> v(\hat a-K)$ and hence~$v(\fw)> v\big((\hat a/\fn)-K\big)$.
Thus $(P_{\times\fn},\fm/\fn,\hat a/\fn)$  is isolated.
\end{proof}

\begin{lemma}\label{lem:isolated normal}
Suppose   $K$ is $\upl$-free or   $r=1$, and $(P,\fm,\hat a)$ is normal. Then  
$$ \text{$(P,\fm,\hat a)$  is isolated} \quad\Longleftrightarrow\quad
\exc^{\ev}(L_P) \cap v(\hat a-K)\ \leq\  v\fm.$$
\end{lemma}
\begin{proof} Use Lemma~\ref{lem:excev normal}; for the direction $\Rightarrow$, use also that  $\hat a-a\prec\fm$ iff $a\prec\fm$.
\end{proof}

\begin{lemma}\label{lem:isolated deg 1} 
Suppose  $\deg P=1$. Then
$$ \text{$(P,\fm,\hat a)$  is isolated} \quad\Longleftrightarrow\quad
\exc^{\ev}(L_P) \cap v(\hat a-K)\  \leq\  v\fm.$$
\end{lemma}
\begin{proof}
Use that $\order L_P=r$ and $L_{P_{+a}}=L_P$ for all $a$.
\end{proof}

{\sloppy
\begin{prop}\label{prop:achieve isolated}
Suppose~$K$ is $\upl$-free or $r=1$, and $(P,\fm,\hat a)$ is normal.  
Then~$(P,\fm,\hat a)$ has an isolated refinement.
\end{prop}}
\begin{proof}
Suppose $(P,\fm,\hat a)$  is not already isolated.
Then Lemma~\ref{lem:isolated normal} gives $\gamma$~with
$$\gamma\in\exc^{\ev}(L_P)\cap v(\hat a-K),\quad 
\gamma>v\fm.$$
We have $|\exc^{\operatorname{\ev}}(L_P)|\le r$,  by [ADH, p.~481] if $r=1$,  and Corollary~\ref{cor:size of excev, strengthened} and $\upl$-freeness of $K$ if $r>1$.
Hence we can take $\gamma:= \max\exc^{\ev}(L_P)\cap v(\hat a-K)$, and then~$\gamma > v\fm$.
Take~$a$ and~$\fn$ with~$v(\hat a-a)>\gamma=v(\fn)$;
then $(P_{+a},\fn,\hat a-a)$ is a refinement of~$(P,\fm,\hat a)$ and~$a\prec\fm$. 
Let $b\prec \fn$; then 
$a+b\prec\fm$, so by~Lemma~\ref{lem:excev normal},
$$\order(L_{(P_{+a})_{+b}})\ =\ r, \qquad
\exc^{\ev}(L_{(P_{+a})_{+b}})\ =\ 
\exc^{\ev}(L_P).$$
Also $v\big((\hat a-a)-b\big)>\gamma$, hence
$$\exc^{\ev}\big(L_{(P_{+a})_{+b}}\big)  \cap v\big((\hat a-a)-K\big)\  =\ 
\exc^{\ev}(L_P)\cap v(\hat a-K)\ \le\ \gamma\ <\ v\big((\hat a-a)-b\big).$$
Thus $(P_{+a},\fn,\hat a-a)$  is isolated. 
\end{proof}

\begin{remarkNumbered}\label{rem:achieve isolated}
Proposition~\ref{prop:achieve isolated}  goes through if instead of assuming that  $(P,\fm,\hat a)$ is normal, we assume that
 $(P,\fm,\hat a)$ is linear. (Same argument,  
using Lem\-ma~\ref{lem:isolated deg 1} in place of Lemma~\ref{lem:isolated normal} and 
$L_{(P_{+a})_{+b}}=L_P$ in place of Lemma~\ref{lem:excev normal}.) 
\end{remarkNumbered}

\begin{cor}\label{cor:isolated r=1}
Suppose $r=1$, and $(P,\fm,\hat a)$ is normal or linear.
If~$\exc^{\ev}(L_P)=\emptyset$, then $(P,\fm,\hat a)$ is isolated.
If $\exc^{\ev}(L_P)\neq\emptyset$, so $\exc^{\ev}(L_P)=\{v\mathfrak g\}$ where $\mathfrak g\in K^\times$, then~$(P,\fm,\hat a)$ is isolated iff $\fm\preceq\mathfrak g$ or $\hat a-K\succ\mathfrak g$.
\end{cor}

\noindent
This follows immediately from Lemmas~\ref{lem:isolated normal} and~\ref{lem:isolated deg 1}.
The results in the rest of this section are the {\em raison d'\^etre}\/ of isolated holes:

\begin{prop}\label{prop:2.12 isolated}  
Suppose $K$ is $\upo$-free and $(P,\fm,\hat a)$ is an isolated hole in~$K$ 
which is normal or linear.
Let $\hat b$ in an immediate
asymptotic extension of $K$ satisfy~${P(\hat b)=0}$ and $\hat b\prec\fm$. Then~$v(\hat a-a)=v(\hat b-a)$ for all~$a$, so~$\hat b\notin K$. \end{prop}
\begin{proof}
Replacing $(P,\fm,\hat a)$,~$\hat b$ by~$(P_{\times\fm},1,\hat a/\fm)$,~$\hat b/\fm$, we arrange $\fm=1$. Let~$a$ be given; we show $v(\hat a-a)=v(\hat b -a)$. This is clear if $a\succeq 1$, so assume~${a\prec 1}$. 
Corollary~\ref{cor:normal=>quasilinear} (if $(P,\fm,\hat a)$ is normal)
and Lemma~\ref{lem:lower bd on ndeg} (if $(P,\fm,\hat a)$ is linear) give~${\ndeg P=1}$. Thus $P$ is in newton position at $a$ by Corollary~\ref{cor:ref 2n}. Moreover~$v(\hat a-a) \notin \exc^{\ev}(L_{P_{+a}})$,
hence $v(\hat a-a)=v^{\ev}(P,a)$ by Lemma~\ref{lem:14.3 complement}. 
Likewise, if~$v(\hat b-a)\notin \exc^{\ev}(L_{P_{+a}})$, then 
$v(\hat b-a)=v^{\ev}(P,a)$ by Lemma~\ref{lem:14.3 complement}, so~$v({\hat a-a})=v(\hat b-a)$.

Thus to finish the proof it is enough to show that $\exc^{\ev}(L_{P_{+a}})\cap v(\hat b-K)\leq 0$. 
Now~$\abs{\exc^{\ev}(L_{P_{+a}})} \leq r$ by   Corollary~\ref{cor:sum of nwts}, so
we have $b\prec 1$ such that $$ \exc^{\ev}(L_{P_{+a}}) \cap v(\hat b-K)\ <\ v(\hat b-b),$$
in particular, $v(\hat b -b)\notin \exc^{\ev}(L_{P_{+a}})$. 
If $(P,\fm,\hat a)$ is normal, then
Lemma~\ref{lem:excev normal} gives
$$\exc^{\ev}(L_{P_{+a}})\ =\ \exc^{\ev}(L_P)\ =\ \exc^{\ev}(L_{P_{+b}}),$$
so by the above with $b$ instead of $a$ we have $v(\hat a-b)=v(\hat b-b)$. 
If $(P,\fm,\hat a)$ is linear, then~$L_{P_{+a}}=L_P=L_{P_{+b}}$, and we obtain likewise $v(\hat a-b)=v(\hat b-b)$. 
Hence
$$\exc^{\ev}(L_{P_{+a}})\cap v(\hat b-K)\ \subseteq\ 
\exc^{\ev}(L_{P_{+a}})\cap \Gamma^{< v(\hat a-b)}\ \subseteq\
\exc^{\ev}(L_P)\cap v(\hat a-K)\ \leq\ 0.$$
using Lemmas~\ref{lem:isolated normal} and~\ref{lem:isolated deg 1} for the last step.
\end{proof}

\noindent
Combining Proposition~\ref{prop:2.12 isolated} with Corollary~\ref{corisomin} yields:

\begin{cor}\label{cor:2.12 isolated}
Let $K$, $(P,\fm,\hat a)$, $\hat b$ be as in Proposition~\ref{prop:2.12 isolated}, and assume also that $(P,\fm,\hat a)$ is $Z$-minimal. Then there is an isomorphism $K\langle\hat a\rangle\to K\langle \hat b\rangle$ of valued differential fields
over $K$ sending~$\hat a$ to $\hat b$.
\end{cor}

\noindent
Using the first Normalization Theorem~\ref{mainthm}, we now obtain:

\begin{cor}\label{cor:2.12 isolated, 2}
Suppose $K$ is $\upo$-free and $\Gamma$ is divisible. Then every minimal hole in $K$ of order $r$ has an
isolated refinement 
$(P,\fm,\hat a)$ such that for any $\hat b$ in an immediate
asymptotic extension of $K$ with~${P(\hat b)=0}$ and $\hat b\prec\fm$ there is an isomorphism~$K\langle\hat a\rangle\to K\langle \hat b\rangle$ of valued differential fields
over $K$ sending~$\hat a$ to $\hat b$.
\end{cor}
\begin{proof}
Given a minimal linear hole in $K$ of order $r$, use Remark~\ref{rem:achieve isolated}  to refine it to an isolated minimal linear hole 
$(P,\fm,\hat a)$ in $K$ of order $r$, and use Corollary~\ref{cor:2.12 isolated}.
Suppose we are given a minimal nonlinear hole in $K$ of order $r$. Then $K$ is $r$-linearly newtonian by Corollary~\ref{degmorethanone}. Then Theorem~\ref{mainthm} yields a refinement~$(Q, \fw, \hat{d})$ of it and an active $\theta$
in $K$ such that  the minimal hole $(Q^\theta, \fw, \hat{d})$ in $K^\theta$ is normal. Proposition~\ref{prop:achieve isolated} gives an
isolated refinement $(Q^\theta_{+d}, \fv , \hat{d}-d)$ of $(Q^\theta, \fw, \hat{d})$. 
Suitably refining $(Q^\theta_{+d}, \fv , \hat{d}-d)$ further followed by compositionally conjugating with a suitable active element of $K^\theta$ yields by Theorem~\ref{mainthm} and Lemma~\ref{lem:isolated refinement} a
refinement~$(P, \fm, \hat a)$ of $(Q, \fw, \hat{d})$ (and thus of the originally given hole) and an active $\phi$ in $K$ such that
 $(P^\phi, \fm, \hat a)$ is both normal and isolated. Then $(P, \fm, \hat a)$ is isolated, and we can apply Corollary~\ref{cor:2.12 isolated} to $K^\phi$ and $(P^\phi, \fm, \hat a)$ in the role of $K$ and $(P,\fm,\hat a)$. 
 \end{proof}

\noindent
For $r=1$ we can replace ``$\upo$-free'' in  Proposition~\ref{prop:2.12 isolated} and Corollary~\ref{cor:2.12 isolated} by the weaker ``$\upl$-free'' (same proofs,  using Lemma~\ref{lem:14.3 complement, order 1} instead of Lem\-ma~\ref{lem:14.3 complement}):

\begin{prop}[${}^*$]\label{prop:2.12 isolated, r=1}
Suppose $K$ is $\upl$-free, $(P,\fm,\hat a)$ is an isolated   hole in~$K$ of order~$r=1$, and suppose $(P,\fm,\hat a)$ is normal or  linear. Let $\hat b$ in an immediate
asymptotic extension of $K$ satisfy~${P(\hat b)=0}$ and $\hat b\prec\fm$. Then~$v(\hat a-a)=v(\hat b-a)$ for all~$a$. \textup{(}Hence if~$(P,\fm,\hat a)$ is $Z$-minimal, then there is an isomorphism $K\langle\hat a\rangle\to K\langle \hat b\rangle$ of valued differential fields
over $K$ sending~$\hat a$ to $\hat b$.\textup{)}
\end{prop}

\noindent
This leads to an analogue of Corollary~\ref{cor:2.12 isolated, 2}: 

\begin{cor}[${}^*$]\label{cor:2.12 isolated, 2, r=1}
Suppose $K$ is $\upl$-free and $\Gamma$ is divisible.
 Then every quasilinear minimal  hole in $K$ of order~$r=1$ has an
isolated refinement 
$(P,\fm,\hat a)$ such that for any $\hat b$ in an immediate
asymptotic extension of $K$ with~${P(\hat b)=0}$ and $\hat b\prec\fm$ there is an isomorphism~$K\langle\hat a\rangle\to K\langle \hat b\rangle$ of valued differential fields
over $K$ sending~$\hat a$ to $\hat b$.
\end{cor}
\begin{proof}
Suppose we are given a quasilinear minimal  hole in $K$ of order~$r=1$.
Then Corollary~\ref{mainthm order 1} yields a refinement $(Q,\fw,\hat d)$ of it and an active $\theta$ in $K$
such that the quasilinear minimal hole $(Q^\theta,\fw,\hat d)$ in $K^\theta$ of order $1$ is normal. 
Proposition~\ref{prop:achieve isolated} gives an isolated refinement $(Q^{\theta}_{+d}, \fv, \hat d -d)$ of $(Q^\theta,\fw,\hat d)$,
and then Corollary~\ref{mainthm order 1} yields a refinement~$(P,\fm,\hat a)$ of $(Q,\fw,\hat d)$ and an active $\phi$ in $K$ such
that $(P^\phi,\fm,\hat a)$ is  normal and isolated.  Now apply Proposition~\ref{prop:2.12 isolated, r=1}
with $K^\phi$ and $(P^\phi, \fm, \hat a)$ in the role of $K$ and $(P,\fm,\hat a)$. 
\end{proof}

\noindent
Next a variant of Lemma~\ref{lem:no hole of order <=r}  for $r=1$ without assuming  $\upo$-freeness:

\begin{cor}[${}^*$]\label{cor:1-newt => no qlin Zmin dent order 1}
Suppose $K$ is  $1$-newtonian and $\Gamma$ is divisible. Then $K$ has no  quasilinear $Z$-minimal slot  of order $1$.
\end{cor}
\begin{proof}
By Proposition~\ref{prop:char 1-linearly newt}, $K$ is $\upl$-free.
 Towards a contradiction, let $(P,\fm,\hat a)$ be a quasilinear $Z$-minimal slot in~$K$ of order~$1$.
 By Lemma~\ref{lem:from cracks to holes} we arrange that~$(P,\fm,\hat a)$  is a hole in $H$. Using Corollary~\ref{mainthm, r=1}, 
  Lemma~\ref{lem:isolated refinement} and the remark before it,
  and Proposition~\ref{prop:achieve isolated}, we can refine further so that
  $(P^\phi,\fm,\hat a)$ is normal and isolated for some active $\phi$ in $K$.
   Then there is no $y\in K$ with~$P(y)=0$ and~$y\prec \fm$, by
  Proposition~\ref{prop:2.12 isolated, r=1}, contradicting Lemma~\ref{lem:zero of P} for~$L=K$.
\end{proof}

\noindent
Finally, for isolated linear holes, without additional hypotheses:

\begin{lemma}\label{lem:isolated, d=1}
Suppose $(P,\fm,\hat a)$ is an  isolated linear hole in $K$, and~${\hat a-a}\prec\fm$.  Then
$P(a)\neq 0$, and $\gamma=v(\hat a-a)$ is the unique element of $\Gamma\setminus\exc^{\ev}(L_P)$ such that~$v^{\ev}_{L_P}(\gamma)=v\big(P(a)\big)$.  
\end{lemma}
\begin{proof}
By Lemma~\ref{lem:isolated deg 1}, $\gamma:=v(\hat a-a)\in\Gamma\setminus\exc^{\ev}(L_P)$.
Since $\deg P =1$,    $$L_P(\hat a-a)\ =\ L_P(\hat a)-L_P(a)\ =\ -P(0)-L_P(a)\ =\ -P(a),$$ 
so $P(a)\neq 0$. By Lemma~\ref{lem:ADH 14.2.7},  
$v^{\ev}_{L_P}(\gamma)=v\big(L_P(\hat a-a)\big)=v\big(P(a)\big)$.
\end{proof}

\noindent
In \cite{ADHld} we shall prove a version of Proposition~\ref{prop:2.12 isolated} without the hypothesis that~$\hat b$ lies in an
immediate extension of $K$. 
In Section~\ref{sec:ultimate} below we consider, in a  more restricted setting, a variant of isolated slots, with 
ultimate exceptional values taking over the role played by exceptional values in Definition~\ref{def:isolated}.

\section{Holes of Order and Degree One}\label{sec:holes of c=(1,1,1)} 

\noindent
{\it In this section $K$ is a $\d$-valued field of $H$-type with small derivation and $\hat K$ is an immediate asymptotic extension of $K$.}\/ So $\hat{K}$ is also $\d$-valued of $H$-type with small derivation.
%After a key approximation result we focus here on slots of complexity $(1,1,1)$ in~$K$. 
The main result of this section is Corollary~\ref{cor:achieve strong normality, general},  a version of 
Corollary~\ref{cor:achieve strong normality, 3} for minimal holes in $K$ of arbitrary degree.
We  let~$k$ range over~$\N$ (in addition to $m$, $n$, as usual).

\subsection*{An approximation}
Suppose $\xi\in K$, $\xi\succ 1$, and~$\zeta:=\xi^\dagger\succeq 1$.

%\noindent
%In the next lemma we let $\xi\in K^\times$  and
%assume $\xi\succ 1$ and $\zeta:=\xi^\dagger\succeq 1$.   

\begin{lemma}\label{lem:xi zeta} 
{\samepage The elements $\xi$, $\zeta$ have the following asymptotic properties: 
\begin{enumerate}
\item[\textup{(i)}] $\zeta^n \prec \xi$ for all $n$;
\item[\textup{(ii)}] $\zeta^{(n)} \preceq \zeta^2$ for all $n$.
\end{enumerate}
Thus for each $P\in \mathcal O\{Z\}$ there is an $N\in\N$ with $P(\zeta)\preceq\zeta^N$, and hence $P(\zeta)\prec\xi$.}
\end{lemma}
\begin{proof}
Part (i)   follows from  [ADH, 9.2.10(iv)] for $\gamma=v(\xi)$.
As to (ii), if~$\zeta'\preceq\zeta$, then $\zeta^{(n)}\preceq\zeta$ by [ADH, 4.5.3], and we are done.
Suppose $\zeta'\succ\zeta$ and set~$\gamma:=v(\zeta)$. Then $\gamma,\gamma^\dagger < 0$, so $\gamma^\dagger=o(\gamma)$
by [ADH, 9.2.10(iv)] and hence $v(\zeta^{(n)})=\gamma+n\gamma^\dagger > 2\gamma=v(\zeta^2)$  by
[ADH, 6.4.1(iv)].
\end{proof} 

\noindent
Let now also $u\in K$ with~$u\preceq 1$, and
suppose $y\in \hat{K}$ satisfies 
$$y'+\xi y\,=\,u, \qquad y\,\preceq\, 1.$$ 
Then $y'\prec 1$, and $y'\xi^{-1}+y=u\xi^{-1}$, so $y-u\xi^{-1}=-y'\xi^{-1}\prec\xi^{-1}$, and thus $y\preceq \xi^{-1}$. Moreover:

\begin{lemma} \label{easybound} If $u\preceq \xi^{-n}$, $n\ge 1$,  then $y\prec \xi^{-n}$.
\end{lemma}
\begin{proof} Suppose $y\succeq \xi^{-n}$, $n\ge 1$. Since $\hat{K}$ is $H$-asymptotic, this gives $y^\dagger\preceq \big(\xi^{-n}\big)^\dagger\asymp \xi^\dagger\prec \xi$, and thus $u=y'+\xi y=y(y^\dagger +\xi) \asymp y\xi\succ \xi^{-n}$.
\end{proof} 

\noindent 
We  now use $u\xi^{-1}$ to start a sequence in $K$ approximating $y$.
Since $y$ is a fixed point of the map $$z\mapsto (u-z')\xi^{-1}\colon \hat{K}\to \hat{K},$$ this suggests approximating
$y$ by the sequence $(y_n)$ in $K$ where $$y_0\,:=\,u\xi^{-1}, \qquad y_{n+1}\,:=\,(u-y_n')\xi^{-1}.$$ This works (although we do not know how to specify a subset of $K$ containing $y_0$ that is closed under the above map and on which this map is contractive): 

\begin{prop}\label{cor:uaintfex, 1} For all $n$ we have $y-y_n\prec \xi^{-n}$. 
\end{prop} 
\begin{proof} For $g\in K$ we have $\left(\displaystyle\frac{g}{\xi^{k+1}}\right)'=\displaystyle\frac{g'-(k+1)\zeta g}{\xi^{k+1}}$, that is,
\begin{equation}\label{eq:gxik+1} \left(\frac{g}{\xi^{k+1}}\right)'+\xi\left(\frac{g}{\xi^{k+1}}\right)\ =\ \frac{g}{\xi^k}-  \frac{(k+1)\zeta g-g'}{\xi^{k+1}}.
\end{equation}
Define the differential polynomials $P_k\in \Q\{U,Z\}$ by
$$P_0\ :=\ U,\qquad P_{k+1}\ :=\ (k+1)ZP_k-P_k',$$
Induction on $n$ using the identity preceding \eqref{eq:gxik+1} with $g=P_k(u,\zeta)$ gives 
$$y_n\ =\ \sum_{k=0}^n \frac{P_k(u,\zeta)}{\xi^{k+1}}.$$ 
 For $g=P_k(u,\zeta)$, the   identity \eqref{eq:gxik+1} says
$$ \left(\frac{P_k(u,\zeta)}{\xi^{k+1}}\right)' +\xi\frac{P_k(\zeta,u)}{\xi^{k+1}}\ =\ \frac{P_k(u,\zeta)}{\xi^{k}}-\frac{P_{k+1}(u,\zeta)}{\xi^{k+1}},$$ 
Summing both sides for $k=0,\dots,n$ then yields
\begin{align*} y_n' + \xi y_n\ &=\ u-\frac{P_{n+1}(u,\zeta)}{\xi^{n+1}},\quad \text{ so in view of $y'+\xi y = u$:}\\
(y-y_n)'+\xi(y-y_n)\ &=\  \frac{P_{n+1}(u,\zeta)}{\xi^{n+1}}\ \prec \frac{1}{\xi^n} \qquad \text{using Lemma~\ref{lem:xi zeta} at the end}. 
\end{align*} 
Lemma~\ref{lem:xi zeta} also yields $y_n\preceq 1$. For $n=0$ we already know that $y-y_n\prec \xi^{-n}$. Let~$n\ge 1$. Then we apply Lemma~\ref{easybound} to
$\hat{K}$ in the role of both $K$ and $\hat{K}$, and with $y-y_n\preceq 1$ instead of $y$ and 
$ \frac{P_{n+1}(u,\zeta)}{\xi^{n+1}}$ instead of $u$ to obtain $y-y_n\prec \xi^{-n}$. 
\end{proof}

\subsection*{Slots of order and degree $1$}
{\it In this subsection we also assume that $K$ has rational asymptotic integration \textup{(}so slots in $K$ make sense\textup{)}, that $K$ is henselian, and that $(P,\fm,\hat f)$ is a slot in~$K$ with $\order P=\deg P=1$ and $\hat f\in \hat K\setminus K$.}\/ We  let~$f$ range over $K$, $\fn$ over $K^\times$, and $\phi$ over active elements of $K$.
Thus
\begin{align*}
P\				&=\  a(Y'+gY-u)\quad\text{where $a\in K^\times,\ g,u\in K$,}  \\
P_{\times\fn}\	&=\   a\fn\big(Y'+(g+\fn^\dagger)Y-\fn^{-1}u\big).
\end{align*} 
Since $K$ is henselian,   $(P,\fm,\hat f)$  is $Z$-minimal and thus equivalent to a hole in $K$, by Lem\-ma~\ref{lem:from cracks to holes}.
Also, $\nval P_{\times\fm}=\ndeg P_{\times\fm}=1$ by Lemma~\ref{lem:lower bd on ndeg}.
%We have~$L_P=a(\der+g)$, so 
%$$ g\in K^\dagger\ \Longleftrightarrow\ \ker L_P\neq\{0\},\quad\quad\qquad
%g\in \I(K)+K^\dagger\ \Longleftrightarrow\ \exc^{\ev}(L_P)\neq\emptyset, $$
%using for the second equivalence the remark on $\exc^{\ev}(A)$ preceding Lemma~\ref{lem:v(ker)=exc, r=1}. 
If~$(P,\fm,\hat f)$ is isolated, then $P(f)\neq 0$ for~$\hat f-f\prec\fm$ by Lemmas~\ref{lem:from cracks to holes} and~\ref{lem:isolated, d=1}, so, taking~$f=0$, we have $u\neq 0$.

\medskip\noindent
Set $\fv:=\fv(L_{P_{\times\fm}})$; thus $\fv=1$ if $g+\fm^\dagger\preceq 1$ and $\fv=1/(g+\fm^\dagger)$ otherwise. Hence from Example~\ref{ex:order 1 linear steep} and the remarks before Lemma~\ref{lem:deg1 normal} we obtain:
\begin{align*}
\text{$(P,\fm,\hat f)$ is normal}	&\quad\Longleftrightarrow\quad\text{$(P,\fm,\hat f)$ is steep}
\quad\Longleftrightarrow\quad \fv\prec^\flat 1, \\
\text{$(P,\fm,\hat f)$ is deep}		&\quad\Longleftrightarrow\quad \fv\prec^\flat 1 \text{ and }  u \preceq \fm/\fv.
\end{align*}
We have $P(0)=-au$, and if $\fv\prec 1$, then $(P_{\times\fm})_1 \sim (a\fm/\fv) Y$. Thus
$$\text{$(P,\fm,\hat f)$ is strictly normal}\quad\Longleftrightarrow\quad
\fv\prec^\flat 1 \text{ and } u \prec_{\Delta(\fv)} \fm\fv.$$
We say that $(P,\fm,\hat f)$ is {\bf balanced}\index{slot!balanced} if  $(P,\fm,\hat f)$ is steep and $P(0) \preceq S_{P_{\times\fm}}(0)$, equivalently, $(P,\fm,\hat f)$ is steep and $u \preceq \fm$.
Thus 
$$\text{$(P,\fm, \hat f)$ is strictly normal}\ \Longrightarrow\ \text{$(P,\fm, \hat f)$ is balanced}\ \Longrightarrow\ 
\text{$(P,\fm, \hat f)$ is deep},$$ 
and with $b\in K^\times$,
$$(P,\fm, \hat f) \text{ is balanced }\Longleftrightarrow (P_{\times\fn},\fm/\fn,\hat f/\fn) \text{ is balanced }\Longleftrightarrow (bP,\fm,\hat f) \text{ is balanced}.$$
If $(P,\fm,\hat f)$ is balanced, then so is any slot in $K$ equivalent to~$(P,\fm,\hat f)$.
Moreover, if~$(P,\fm,\hat f)$ is a hole in $K$, then $P(0)=-L_P(\hat f)$, so
$(P,\fm,\hat f)$ is balanced iff it is steep and $L_P(\hat f)\preceq S_{P_{\times\fm}}(0)$.
By Corollary~\ref{cor:good approx to hata, deg 1}, if $(P,\fm,\hat f)$ is steep, then  $\hat f-f\prec_{\Delta(\fv)}\fm$ for some $f$.
For balanced $(P,\fm,\hat f)$ we have a variant of this fact:

\begin{lemma}\label{lem:balanced good approx}
Suppose $(P,\fm,\hat f)$ is balanced. Then there is for all~$n$ an $f$ such that $\hat f-f\prec\fv^n\fm$.
\end{lemma}
\begin{proof}
Replacing $(P,\fm,\hat f)$ by an equivalent hole in $K$, we arrange that $(P,\fm,\hat f)$ is a hole in $K$, and replacing
$(P,\fm,\hat f)$  by $(P_{\times\fm},1,\hat f/\fm)$, that~$\fm=1$. Then~${\hat f}'+g\hat f=u$ with $g=1/\fv\succ^\flat 1$, and $u\preceq 1$. Now use Proposition~\ref{cor:uaintfex, 1}.
\end{proof}

\begin{cor}\label{cor:balanced -> strongly normal}
Suppose the subgroup $K^\dagger$ of $K$ is divisible and the slot $(P,\fm,\hat f)$ is balanced. Then $(P,\fm,\hat f)$ has a strictly normal refinement $(P_{+f},\fm,\hat f-f)$.
\end{cor}
\begin{proof}
First arrange that $(P,\fm,\hat f)$ is a hole in $K$.
The previous lemma yields an~$f$ such that $\hat f-f\preceq\fv^3\fm$.
Then $(P_{+f},\fm,\hat f-f)$ is a strictly normal refinement of~$(P,\fm,\hat f)$, by Lemma~\ref{lem:achieve strong normality} (where the latter uses divisibility of $K^\dagger$).
\end{proof}

\noindent
Since $K$ is henselian and $\d$-valued, the condition in Corollary~\ref{cor:balanced -> strongly normal} that $K^\dagger$ is divisible is satisfied if the groups~$C^\times$ and $\Gamma$ are divisible.

\begin{lemma}\label{lem:balanced refinement}
Suppose $(P,\fm,\hat f)$ is  balanced with  $v\hat f\notin \exc^{\ev}(L_P)$ and~$\hat f - f \preceq \hat f$. Then the refinement $(P_{+f},\fm,\hat f-f)$ of $(P,\fm,\hat f)$
is balanced.
\end{lemma}
\begin{proof} By Lemma~\ref{lem:from cracks to holes}  we arrange $(P,\fm,\hat f)$ is a hole. 
Replacing $(P,\fm,\hat f)$ and $f$ by $(P_{\times\fm},1,\hat f/\fm)$ and $f/\fm$ we arrange next that~$\fm=1$. 
By the remark preceding Lemma~\ref{lem:achieve steep},  $(P_{+f},1,\hat f-f)$ is steep.  
Take~$\phi$ such that $v\hat f\notin\exc\big((L_P)^\phi\big)$, and set~$\hat g:=\hat f-f$, so $0\neq\hat g\preceq\hat f$.
Recall from [ADH, 5.7.5] that~$L_{P^\phi}=(L_P)^\phi$ and hence
$L_{P^\phi}(\hat f)=L_P(\hat f)$ and $L_P(\hat g)=L_{P^\phi}(\hat g)$.
Thus
$$L_{P_{+f}}(\hat g)\ =\ L_P(\hat g)\  \preceq\  L_{P^\phi}\hat g\  \preceq\ L_{P^\phi}\hat f\ \asymp\ L_{P^\phi}(\hat f)\  =L_P(\hat f)\    \preceq\ S_P(0)\ =\ S_{P_{+f}}(0),$$
using [ADH, 4.5.1(iii)] to get the second $\preceq$ and $v\hat f \notin \exc(L_{P^\phi})$  to get   $\asymp$; the last $\preceq$ uses $(P,1,\hat f)$ being a hole.
Therefore~$(P_{+f},1,\hat g)$ is balanced.
\end{proof}

\noindent
Combining Lemmas~\ref{lem:isolated refinement} and~\ref{lem:balanced refinement} yields:

\begin{cor}\label{cor:balanced refinement}
If $(P,\fm,\hat f)$ is balanced and isolated, and $\hat f -f \preceq \hat f$, then the refinement~$(P_{+f},\fm,\hat f-f)$ of $(P,\fm,\hat f)$ is also balanced and isolated. 
\end{cor}
 
\noindent
We call  $(P,\fm,\hat f)$ {\bf proper}\index{slot!proper}\index{proper!slot} if the differential polynomial $P$ is proper as defined in Section~\ref{sec:complements newton} (that is, $u\neq 0$ and $g+u^\dagger\succ^\flat 1$). If~$(P,\fm,\hat f)$ is proper, then so are~$(bP,\fm,\hat f)$ for $b\neq 0$ and $(P_{\times\fn},\fm/\fn,\hat f/\fn)$, as well as each refinement $(P,\fn,\hat f)$ of~$(P,\fm,\hat f)$ and each slot in $K$ equivalent to $(P,\fm,\hat f)$. By Lemma~\ref{lem:proper compconj},
if $(P,\fm,\hat f)$ is proper, then so is~$(P^\phi,\fm,\hat f)$ for $\phi\preceq 1$.

\begin{lemma}\label{lem:proper => balanced}
Suppose $(P,\fm,\hat f)$ is proper and $\fm\asymp u$. Then $(P,\fm,\hat f)$ is balanced.
\end{lemma}
\begin{proof}
Replacing $(P,\fm,\hat f)$ by $(P_{\times\fm},1,\hat f/\fm)$, we arrange $\fm=1$.
Then $u\asymp 1$ and thus $(P,1,\hat f)$ is balanced.
\end{proof}

\begin{prop}\label{prop:balanced refinement}
Suppose $(P,\fm,\hat f)$ is proper and $v\hat f\notin\exc^{\ev}(L_P)$. Then~$(P,\fm,\hat f)$ has a balanced re\-fine\-ment.
\end{prop}

\begin{proof}
We arrange $\fm=1$ as usual.
By Lemmas~\ref{lem:proper asymp} and~\ref{lem:from cracks to holes} we have 
$$\hat f\  \sim\  u/(g+u^\dagger) \prec^\flat u.$$ Hence if $u\preceq 1$, then
$(P,u,\hat f)$ refines $(P,1,\hat f)$, and so $(P,u,\hat f)$ is balanced by Lem\-ma~\ref{lem:proper => balanced}.
Assume now that $u\succ 1$. Then $1\prec u\prec g$ by Lemma~\ref{lem:proper nmul 1} and~$\nval P=1$, and hence $u^\dagger\preceq g^\dagger \prec  g$. So~$g\sim g+u^\dagger \succ^\flat 1$, hence~$(P,1,\hat f)$ is steep, and~$\hat f\sim u/g$. Set $f:=u/g\prec 1$; then~$(P_{+f},1,\hat f-f)$ is a steep refinement of~$(P,1,\hat f)$. Moreover
$$P_{+f}(0) =P(f) = af'\prec a=S_{P_{+f}}(0),$$
hence  $(P_{+f},1,\hat f-f)$ is balanced.
\end{proof}

\begin{cor}\label{nonamecor}
Suppose $K$ is $\upl$-free.  Then there exists $\phi\preceq 1$ and  a re\-fine\-ment~$(P_{+f},\fn,{\hat f-f})$ of $(P,\fm,\hat f)$ such that $(P^\phi_{+f},\fn,\hat f-f)$ is balanced.
\end{cor}
\begin{proof}  
Using Remark~\ref{rem:achieve isolated} we can replace~$(P,\fm,\hat f)$ by a refinement to arrange that~$(P,\fm,\hat f)$ is isolated. Then $u\neq 0$
by a remark at the beginning of this subsection, so by Lemma~\ref{lem:proper evt} ,  $P^\phi$ is proper, eventually.
Now apply  Proposition~\ref{prop:balanced refinement} to a proper (and isolated) $(P^\phi,\fm,\hat f)$ with $\phi\preceq 1$.
\end{proof}

\begin{cor}\label{cor:strongly normal d=1}
Suppose $K$ is $\upl$-free and $K^\dagger$ is divisible. Then $(P,\fm,\hat f)$ has a
refinement~$(P_{+f},\fn,\hat f-f)$ with strictly normal $(P_{+f}^\phi,\fn,\hat f -f)$ for some $\phi\preceq 1$.
\end{cor}
\begin{proof}  Corollary~\ref{nonamecor} yields a refinement $(P_{+f_1}, \fn_1, \hat f-f_1)$ of $(P, \fm,\hat f)$ and a~$\phi\preceq 1$ such that
$(P^\phi_{+f_1},\fn_1,\hat f-f_1)$ is balanced. With $K^\phi$ 
and ${(P^\phi_{+f_1}, \fn_1, \hat f-f_1)}$
in the roles of~$K$  and $(P,\fm, \hat f)$, respectively, we can apply Corollary~\ref{cor:balanced -> strongly normal} to~$(P^\phi_{+f_1}, \fn_1, \hat f-f_1)$ to give a strictly normal refinement $(P^{\phi}_{f_1+f_2}, \fn, \hat f-f_1-f_2)$ of it.
Thus for $f:=f_1+f_2$ the refinement $(P_{+f},\fn, \hat f-f)$ of $(P,\fm,\hat f)$ has the property that $(P_{+f}^\phi,\fn,\hat f -f)$ is strictly normal. 
\end{proof}

\noindent
Combining this corollary with Corollaries~\ref{cor:minhole deg 1},~\ref{cor:achieve strong normality, 3}, and Lemma~\ref{lem:normality comp conj, strong} yields:

{\sloppy
\begin{cor}\label{cor:achieve strong normality, general}
If $K$ is $\upo$-free and algebraically closed, then every minimal hole in $K$ of order $\ge 1$ has
a re\-fine\-ment~$(Q,\fn,\hat g)$ such that~$(Q^\phi,\fn,\hat g)$ is deep and strictly normal, eventually.  
\end{cor} }

\newpage 

\part{Holes in $H$-Fields}\label{part:dents in H-fields}

\medskip

\noindent
Here we focus on holes in the algebraic closure $K$ of a Liouville closed $H$-field~$H$ with small derivation.
After the preliminary Sections~\ref{sec:aux} and~\ref{sec:approx linear diff ops} we come in Sections~\ref{sec:split-normal holes}--\ref{sec:repulsive-normal} to 
the technical heart of Chapter~\ref{part:dents in H-fields}.
  Section~\ref{sec:split-normal holes} shows that every minimal hole in~$K$ gives rise to a $Z$-minimal slot
  $(Q, \fn, \hat b)$ in $H$ such that the slot~$(Q^\phi,\fn,\hat b)$ in~$H^\phi$ is eventually {\it split-normal}, meaning {\em normal with its linear part ``asymptotically'' splitting over $K^\phi$}; see
Definition~\ref{SN} for the precise definition, and Theorem~\ref{thm:split-normal} for the main result of this section.
When $H$ is a Hardy field as in \cite{ADH5}, this asymptotic splitting will allow us to define 
a contractive operator on a space of real-valued functions; this operator then has a fixed point whose germ~$y$ satisfies 
$$Q(y)\ =\ 0,\qquad  y\prec\fn.$$ 
A main difficulty then lies in guaranteeing that  such germs $y$ have similar asymptotic properties as~$\hat b$.  
Sections~\ref{sec:ultimate} and~\ref{sec:repulsive-normal} prepare the ground for dealing with this:
In Section~\ref{sec:ultimate} we strengthen the concept of isolated slot  to {\it ultimate}\/ slot (in $H$, or in~$K$).
This relies on the ultimate exceptional values of linear differential operators over $K$ introduced in Chapter~\ref{part:universal exp ext}. 
In Section~\ref{sec:repulsive-normal} we single out among
split-normal slots those that are {\it repulsive-normal}, 
culminating in the proof of Theorem~\ref{thm:repulsive-normal}: an analogue of Theorem~\ref{thm:split-normal} producing from a minimal hole in $K$  and for small enough active $\phi>0$ in $H$  a deep
repulsive-normal ultimate slot in~$H^\phi$. This is further improved in Theorem~\ref{thm:strongly repulsive-normal}.

\section{Some Valuation-Theoretic Lemmas}\label{sec:aux}

\noindent
The present section contains preliminaries for the next section on approximating  splittings of linear differential operators; these facts
in turn are used in  Section~\ref{sec:split-normal holes} on split-normality. 
We shall often deal with real closed fields with extra structure, denoted usually by $H$, since the results in this section about such $H$ will be applied to $H$-fields (and to Hardy fields in \cite{ADH5}). 
We begin by summarizing some purely valuation-theoretic facts.

\subsection*{Completion and specialization of real closed valued fields}
Let $H$ be a real closed valued field whose valuation ring $\mathcal O$ is convex in $H$ (with respect to the unique ordering on $H$ making~$H$ an ordered field). Using [ADH, 3.5.15] we equip the algebraic closure $K=H[\imag]$ \textup{(}$\imag^2=-1$\textup{)} of~$H$ with its unique valuation ring lying over~$\mathcal O$, which is $\mathcal O + \mathcal O\imag$. We set $\Gamma:= v(H^\times)$, so $\Gamma_K=\Gamma$. 

\begin{lemma}\label{lem:Kc real closed}
The completion $H^{\operatorname{c}}$ of the valued field $H$ is real closed, its valuation ring is convex in $H^{\operatorname{c}}$, and there is a unique valued field embedding 
$H^{\operatorname{c}}\to K^{\operatorname{c}}$ over~$H$. Identifying $H^{\operatorname{c}}$ with its image under this embedding we have $H^{\operatorname{c}}[\imag]=K^{\operatorname{c}}$.
\end{lemma}
\begin{proof} For the first two claims, see [ADH, 3.5.20].  By [ADH, 3.2.20] we have a unique valued field embedding 
$H^{\operatorname{c}}\to K^{\operatorname{c}}$ over~$H$, and viewing $H^{\operatorname{c}}$ as a valued subfield of~$K^{\operatorname{c}}$ via this embedding we have $K^{\operatorname{c}}=H^{\operatorname{c}}K=H^{\operatorname{c}}[\imag]$ by [ADH, 3.2.29].
\end{proof}

\noindent
We identify $H^{\operatorname{c}}$ with its image in $K^{\operatorname{c}}$ as in the previous lemma.
Fix a  convex subgroup $\Delta$ of $\Gamma$. Let $\dot{\mathcal O}$ be the valuation ring of the coarsening
of $H$ by $\Delta$, with maximal ideal $\dot{\smallo}$. Then by [ADH, 3.5.11 and subsequent remarks] $\dot{\mathcal O}$ and $\dot{\smallo}$ are convex in $H$, the specialization $\dot H=\dot{\mathcal O}/\dot{\smallo}$ of~$H$ by $\Delta$ is naturally an ordered and valued field, and the valuation ring of $\dot{H}$ is convex in $\dot{H}$. Moreover,
$\dot{H}$ is even real closed by [ADH, 3.5.16]. 
 Likewise, the coarsening of $K$ by $\Delta$ has valuation ring 
$\dot{\mathcal O}_K$ with maximal ideal $\dot{\smallo}_K$
and valued residue field $\dot K$. Thus $\dot{\mathcal O}_K$ lies over $\dot{\mathcal O}$ by [ADH,~3.4, subsection {\it Coarsening and valued field extensions}\/], so
$(K, \dot{\mathcal O}_K)$ is a valued field extension of
$(H, \dot{\mathcal{O}})$. In addition:  

\begin{lemma}\label{lem:dotK real closed}
$\dot K$ is a valued field extension of $\dot H$ and an algebraic closure of $\dot H$.
\end{lemma}
\begin{proof} The second part follows by general valuation theory from $K$ being an algebraic closure of $H$. In fact, with the image of $\imag\in \mathcal{O}_K\subseteq \dot{\mathcal O}_K$
in $\dot{K}$ denoted by the same symbol, we have
$\dot{K}=\dot{H}[\imag]$.
\end{proof}

\noindent
Next, let $\hat H$ be an immediate valued field extension of $H$. We equip $\hat H$ with the unique field ordering making it an ordered field extension of $H$ in which $\mathcal O_{\hat H}$ is convex; see [ADH, 3.5.12].
Choose $\imag$ in a field extension of $\hat H$ with $\imag^2=-1$.  Equip~$\hat H[\imag]$ with the unique valuation ring of $\hat H[\imag]$ that lies over $\mathcal O_{\hat H}$, namely~${\mathcal O_{\hat H} + \mathcal{O}_{\hat H}\imag}$  [ADH, 3.5.15]. Let  $\hat a =\hat b+\hat c\,\imag\in \hat H[\imag]\setminus H[\imag]$ with  $\hat b,\hat c\in \hat H$, and let~$b$,~$c$ range over $H$.
Then
$$v\big(\hat a-(b+c\imag)\big)\ =\ \min\!\big\{ v(\hat b-b),v(\hat c-c) \big\}$$
and thus $v\big( \hat a - H[\imag] \big) \subseteq v(\hat b-H)$ and $v\big( \hat a - H[\imag] \big) \subseteq v(\hat c-H)$.

\begin{lemma}\label{lem:same width}
We have $v(\hat b-H)\subseteq v(\hat c-H)$ or $v(\hat c-H)\subseteq v(\hat b-H)$. Moreover, the following are equivalent:
\begin{enumerate}
\item[$\mathrm{(i)}$] $v(\hat b-H)\subseteq v(\hat c-H)$;
\item[$\mathrm{(ii)}$] for all $b$ there is a $c$ with $v\big(\hat a-(b+c\imag)\big)=v(\hat b-b)$;
\item[$\mathrm{(iii)}$]  $v\big( \hat a - H[\imag] \big) = v(\hat b-H)$.
\end{enumerate}
\end{lemma}
\begin{proof} For the first assertion, use that $v(\hat b-H), v(\hat c-H)\subseteq \Gamma_\infty$ are downward closed.  Suppose $v(\hat b-H)\subseteq v(\hat c-H)$, and let $b$ be given. 
If $\hat c\in H$, then for $c:=\hat c$ we have $v\big(\hat a-(b+c\imag)\big)=v(\hat b-b)$. Suppose $\hat c\notin H$. Then $v(\hat c-H)\subseteq\Gamma$ does not have a largest element and
$v(\hat b-b)\in v(\hat c-H)$, so we have~$c$ with $v(\hat b-b)<v(\hat c-c)$; thus 
$$v\big(\hat a-(b+c\imag)\big)=\min\!\big\{v(\hat b-b),v(\hat c-c)\big\}=v(\hat b-b).$$ 
This shows~(i)~$\Rightarrow$~(ii). Moreover,
(ii)~$\Rightarrow$~(iii) follows from $v\big( \hat a - H[\imag] \big) \subseteq v(\hat b-H)$, and (iii)~$\Rightarrow$~(i) from~$v\big( \hat a - H[\imag] \big) \subseteq v(\hat c-H)$. 
\end{proof} 

\noindent
So if $v(\hat b-H)\subseteq v(\hat c-H)$, then:  $\hat a$ is special over $H[\imag]\ \Longleftrightarrow\ \hat b$ is special over $H$. 
 
\medskip
\noindent
To apply Lemma~\ref{lem:same width} to $H$-fields we assume in the next lemma more generally that~$H$ is equipped with a derivation making it a $\d$-valued field and that $\hat H$ is equipped with a derivation~$\der$ making it an asymptotic field extension of $H$; then $\hat H$ is also $\d$-valued
with the same constant field as $H$ [ADH, 9.1.2].
 
\begin{lemma}\label{lem:same width, der}
Suppose $H$ is closed under integration. Then we have:
$$v(\hat b-H)\subseteq v(\hat c-H)\ \Longrightarrow\ v(\der(\hat b)-H)\subseteq v(\der(\hat c)-H).$$
\end{lemma}
\begin{proof} Assume $v(\hat b-H)\subseteq v(\hat c-H)$.
Let $b\in H$, and take $g\in H$ with $g'=b$;    adding a suitable constant to $g$ we   arrange $\hat b-g\nasymp 1$.
Next, take $h\in H$ with $\hat b-g\asymp \hat c-h$. Then 
$$\der(\hat b)-b\ =\ \der(\hat b-g)\ \asymp\ \der(\hat c-h)\ =\ \der(\hat c) -h',$$
so $v(\der(\hat b) - b)\in v(\der(\hat c) -H)$.
\end{proof}

\subsection*{Embedding into the completion}
{\it In this subsection $K$ is an asymptotic field, $\Gamma:= v(K^\times)\neq\{0\}$, and 
$L$ is an asymptotic field extension of $K$ such that $\Gamma$ is cofinal in $\Gamma_L$.}\/

\begin{lemma}\label{cdercon}
Let $a\in L$ and let $(a_\rho)$ be a c-sequence in $K$ with $a_\rho\to a$ in $L$. Then for each $n$, $(a_\rho^{(n)})$ is a c-sequence in $K$ with $a_\rho^{(n)}\to a^{(n)}$ in $L$. 
\end{lemma}
\begin{proof}
By induction on $n$ it suffices to treat the case $n=1$. Let
$\gamma\in\Gamma_L$; we need to show the existence of an index $\sigma$ such that $v(a'-a_\rho')>\gamma$ for all $\rho>\sigma$.
By [ADH, 9.2.6]   we have $f\in L^\times$ with $f\prec 1$ and $v(f') \geq \gamma$. 
Take $\sigma$ such that $v(a-a_\rho) > vf$ for all $\rho>\sigma$.
Then $v(a'-a_\rho') > v(f') \geq \gamma$ for $\rho>\sigma$. \end{proof}

\noindent
Let  $K^{\operatorname{c}}$ be the completion of the valued differential field $K$; then 
$K^{\operatorname{c}}$ is   asymptotic   by [ADH, 9.1.6].  Lemma~\ref{cdercon} and [ADH, 3.2.13 and 3.2.15] give:

\begin{cor}\label{klemb}
Let $(a_i)_{i\in I}$ be a family of elements of $L$ such that  $a_i$  is the limit in~$L$
of a c-sequence in $K$, for each $i\in I$. 
Then there is a unique embedding~$K\big\<(a_i)_{i\in I}\big\>\to K^{\operatorname{c}}$ of valued differential fields over $K$.
\end{cor}

\noindent
Next suppose that $H$ is a real closed asymptotic field whose
valuation ring $\mathcal{O}$ is convex in $H$ with $\mathcal{O}\ne H$, the asymptotic extension  $\hat H$ of $H$
is immediate, and~$\imag$ is an element of an asymptotic
extension of~$\hat H$ with $\imag^2=-1$.  Then $\imag\notin \hat H$, and we identify~$H^{\operatorname{c}}$ with a valued subfield of $H[\imag]^{\operatorname{c}}$
as in Lemma~\ref{lem:Kc real closed}, so that
$H^{\operatorname{c}}[\imag]=H[\imag]^{\operatorname{c}}$  as in that lemma.  Using also  Lemma~\ref{cdercon}  we see that
$H^{\operatorname{c}}$ is actually a valued {\em differential\/} subfield of the asymptotic field $H[\imag]^{\operatorname{c}}$, and so $H^{\operatorname{c}}[\imag]=H[\imag]^{\operatorname{c}}$ also as {\em asymptotic\/} fields.   Thus by Corollary~\ref{klemb} applied to $K:=H$ and
$L:=\hat{H}$:

\begin{cor}\label{cor:embed into Kc}
Let $a\in \hat{H}[\imag]$ be the limit in $\hat{H}[\imag]$ of a c-sequence in~$H[\imag]$. Then~$\Re a$,~$\Im a$ are limits in $\hat{H}$ of c-sequences in $H$, hence
there is a unique embedding
$H[\imag]\big\<\!\Re a,\Im a\big\>\to H^{\operatorname{c}}[\imag]$ of valued differential fields over $H[\imag]$.
\end{cor}

\section{Approximating Linear Differential Operators}\label{sec:approx linear diff ops}

\noindent
{\em In this section $K$ is a valued differential field with small derivation,}\/  $\Gamma:= v(K^\times)$.
For later use we prove here Corollaries~\ref{cor:approx LP+f} and~\ref{cor:approx LP+f, real, general} and consider {\em strong splitting}\/.\index{splitting!strong}\index{linear differential operator!strong splitting} 
Much of this section rests on the following basic estimate for linear differential operators which split over $K$:

\begin{lemma}\label{lem:approxB}
Let   $b_1,\dots,b_r\in K$ and $n$ be given.
Then there exists $\gamma_0\in\Gamma^{\geq}$ such that  for all $b_1^{\smallbullet},\dots,b_r^{\smallbullet}\in K$ and $\gamma\in\Gamma$ with $\gamma>\gamma_0$ and $v(b_i-b_i^{\smallbullet})\geq (n+r)\gamma$ for~$i=1,\dots,r$, we have $v(B-B^{\smallbullet})  \geq vB+n\gamma$, where
$$B\ :=\ (\der-b_1)\cdots(\der-b_r)\in K[\der],\quad B^{\smallbullet}\ :=\ (\der-b_1^{\smallbullet})\cdots(\der-b_r^{\smallbullet})\in K[\der].$$ 
\end{lemma}
\begin{proof}
By induction on $r\in\N$. The case $r=0$ is clear (any $\gamma_0\in\Gamma^{\geq}$ works).  Suppose the lemma holds for a certain $r$.
Let $b_1,\dots,b_{r+1}\in K$  and $n$ be given. Set~$\beta_i:=vb_i$ ($i=1,\dots,r+1$). Take~$\gamma_0$ as in the lemma applied to
$b_1,\dots,b_r$ and~$n+1$ in place of $n$, and
let $\gamma_1:=\gamma_0$ if $b_{r+1}=0$, $\gamma_1:=\max\big\{\gamma_0,\abs{\beta_{r+1}}\big\}$ otherwise. 
Let $b_1^{\smallbullet},\dots,b_{r+1}^{\smallbullet}\in K$ and $\gamma\in\Gamma$ with
$\gamma>\gamma_1$ and $v(b_i-b_i^{\smallbullet})\geq (n+r+1)\gamma$ for $i=1,\dots,r+1$.
Set 
$$B\ :=\ (\der-b_1)\cdots(\der-b_r),\qquad B^{\smallbullet}\ :=\ (\der-b_1^{\smallbullet})\cdots(\der-b_r^{\smallbullet}),\qquad E\ :=\ B-B^{\smallbullet}.$$
Then
$$B(\der-b_{r+1})\  =\ B^{\smallbullet}(\der-b^{\smallbullet}_{r+1})+ 
 B^{\smallbullet}(b^{\smallbullet}_{r+1}-b_{r+1})+E(\der-b_{r+1}).$$
Inductively we have  $vE \geq vB+(n+1)\gamma$. Suppose $E\ne 0$ and $0\neq b_{r+1}\nasymp 1$. Then by
[ADH, 6.1.5],  
\begin{align*}
v_E(\beta_{r+1})-v_B(\beta_{r+1})\ &=\ vE-vB+o(\beta_{r+1}) \\
												&\geq\ (n+1)\gamma + o(\beta_{r+1})\\
												&\geq\ n\gamma + |\beta_{r+1}|+o(\beta_{r+1})\ >\  n\gamma.
\end{align*}
Hence, using $E(\der-b_{r+1})=E\der - Eb_{r+1}$ and $v(E\der)=v(E)\ne v_E(\beta_{r+1})$,
\begin{align*}
v\big(E(\der-b_{r+1})\big)\ =\ \min\!\big\{vE,v_E(\beta_{r+1})\big\}\ &>\  \min\!\big\{ vB,v_B(\beta_{r+1})\big\}+n\gamma \\
																&=\ v\big(B(\der-b_{r+1})\big)+n\gamma,
\end{align*}
where for the last equality we use $vB\ne v_B(\beta_{r+1})$. Also,
$$v\big( B^{\smallbullet}(b^{\smallbullet}_{r+1}-b_{r+1}) \big) = v_{B^{\smallbullet}}\big(v(b^{\smallbullet}_{r+1}-b_{r+1})\big) \geq v_{B^{\smallbullet}}\big((n+r+1)\gamma\big)=v_B\big((n+r+1)\gamma\big)$$
where we use [ADH, 6.1.7] for the last equality. Moreover,    by [ADH, 6.1.4],
$$v_B\big((n+r+1)\gamma\big)-n\gamma\ \geq\ vB+(r+1)\gamma+o(\gamma)\ >\ vB\ 	\geq\   v\big(B(\der-b_{r+1})\big).$$
This yields the desired result for $E\ne 0$, $0\ne b_{r+1}\nasymp 1$. The cases $E\ne 0$, $b_{r+1}=0$ and $E=0$, $0\ne b_{r+1}\nasymp 1$ are simpler versions of the above, and so is the case~$E\ne 0$, $b_{r+1}\asymp 1$ using
[ADH, 5.6.1(i)]. The remaining cases, $E=0$, $b_{r+1}=0$ and $E=0$, $b_{r+1}\asymp 1$, are even simpler to handle. 
\end{proof}

\begin{cor}\label{cor:approx A}
Let   $a,b_1,\dots,b_r\in K$, $a\neq 0$.
Then there exists $\gamma_0\in\Gamma^{\geq}$ such that  for all $a^{\smallbullet},b_1^{\smallbullet},\dots,b_r^{\smallbullet}\in K$ and $\gamma\in\Gamma$ with $\gamma>\gamma_0$, $v(a-a^{\smallbullet})\geq va+\gamma$, and~$v(b_i-b_i^{\smallbullet})\geq (r+1)\gamma$ for $i=1,\dots,r$, we have
$v(A-A^{\smallbullet}) \geq vA+\gamma$, where
$$A\ :=\ a(\der-b_1)\cdots(\der-b_r)\in K[\der],\quad A^{\smallbullet}\ :=\ a^{\smallbullet}(\der-b_1^{\smallbullet})\cdots(\der-b_r^{\smallbullet})\in K[\der].$$ 
\end{cor}
\begin{proof}
Take $\gamma_0$ as in the previous lemma applied to $b_1,\dots,b_r$ and $n=1$, and let~$B=(\der-b_1)\cdots(\der-b_r)$, $A=aB$. Let $a^{\smallbullet},b_1^{\smallbullet},\dots,b_r^{\smallbullet}\in K$ and  $\gamma\in\Gamma$ be such that~$\gamma>\gamma_0$, $v(a-a^{\smallbullet})\geq va+\gamma$, and $v(b_i-b_i^{\smallbullet})\geq (r+1)\gamma$ for $i=1,\dots,r$. Set~$B^{\smallbullet}:=(\der-b_1^{\smallbullet})\cdots(\der-b_r^{\smallbullet})$, $A^{\smallbullet}:=a^{\smallbullet} B^{\smallbullet}$.
Then 
$$E\ :=\ A-A^{\smallbullet}\ =\ a(B-B^{\smallbullet})+(a-a^{\smallbullet})B^{\smallbullet}.$$
Lemma~\ref{lem:approxB} gives $vB^{\smallbullet}=vB$, and so {\samepage
$$v\big(a(B-B^{\smallbullet})\big)\ \geq\ va+vB+\gamma\ =\ vA+\gamma, \quad
v\big( (a-a^{\smallbullet})B^{\smallbullet} \big)\ =\ v(a-a^{\smallbullet})+vB\ \geq\ vA+\gamma,$$
so $vE\geq vA+\gamma$.}
\end{proof}

\noindent
{\em In the rest of this subsection we assume $P\in K\{Y\}\setminus K$,
set $r:=\order P$, and let~$\i$,~$\j$ range over $\N^{1+r}$}. 

\begin{lemma} For $\delta:= v\big(P-P(0)\big)$ and all $h\in \smallo$ we have $v\big(P_{+h} - P\big) \geq \delta+\textstyle\frac{1}{2}vh$.
\end{lemma}
\begin{proof} Note that $\delta\in\Gamma$ and $v(P_{\j})\ge \delta$ for all $\j$ with $\abs{\j}\geq 1$.
Let $h\in \smallo^{\ne}$ and~$\i$ be given; we claim that  $v\big( (P_{+h})_{\i} - P_{\i}\big) \geq \delta+\textstyle\frac{1}{2}vh$.
By [ADH, (4.3.1)] we have
$$(P_{+h})_{\i}\ 	=\  P_{\i}+   Q(h)\quad\text{where $Q(Y):=
\sum_{\abs{\j}\geq 1}  {\i+\j \choose \i} P_{\i+\j}\,Y^{\j}\in K\{Y\}$.}$$
From $Q(0)=0$ and [ADH, 6.1.4] we obtain
$$v(Q_{\times h})\ \geq\  v(Q)  + vh+o(vh)\ \geq\ \delta+\textstyle\frac{1}{2}vh.$$
Together with $v\big(Q(h)\big)\ge v(Q_{\times h})$ this yields the lemma.
\end{proof}

\begin{cor}\label{cor:approx P+f}
Let $f\in K$. Then there exists $\delta\in\Gamma$ such that for all $f^{\smallbullet}\in K$ with $f-f^{\smallbullet}\prec 1$  we have
$v\big(P_{+f^{\smallbullet}} - P_{+f}\big) \geq \delta+\textstyle\frac{1}{2}v(f^{\smallbullet}-f)$.
\end{cor}
\begin{proof}
Take  $\delta$   as
in the preceding lemma with  $P_{+f}$  in place of $P$ and $h=f^{\smallbullet}-f$.
\end{proof}

\begin{cor}\label{corcorappr}
Let $a,b_1,\dots,b_r,f\in K$ be such that
$$A\ :=\ L_{P_{+f}}\ =\ a(\der-b_1)\cdots(\der-b_r),\qquad a\ \neq\ 0.$$
Then there exists $\gamma_1\in\Gamma^{\geq}$ such that for all $a^{\smallbullet}, b_1^{\smallbullet},\dots,b_r^{\smallbullet},f^{\smallbullet}\in K$ and $\gamma\in\Gamma$, if 
$$\gamma>\gamma_1,\ v(a-a^{\smallbullet})\geq va +\gamma,\ v(b_i-b_i^{\smallbullet})\geq (r+1)\gamma\ (i=1,\dots,r), \text{ and }
v(f-f^{\smallbullet})\geq 4\gamma,$$ 
then
\begin{enumerate}
\item[$\mathrm{(i)}$] $v\big( P_{+f^{\smallbullet}}  - P_{+f}\big) \geq  vA+\gamma$; and
\item[$\mathrm{(ii)}$] $L_{P_{+f^{\smallbullet}}} = a^{\smallbullet}(\der-b_1^{\smallbullet})\cdots(\der-b_r^{\smallbullet}) + E$  where $vE\geq vA+\gamma$.
\end{enumerate}
\end{cor}
\begin{proof}
Take $\gamma_0$ as in Corollary~\ref{cor:approx A} applied to $a,b_1,\dots,b_r$, and
take  $\delta$   as
in  Corollary~\ref{cor:approx P+f}.
Then $\gamma_1:=\max\{\gamma_0,vA-\delta\}$  has the required property.
\end{proof}

\noindent
In the next result $L$ is a valued differential field extension of $K$ with small derivation such that~$\Gamma$ is cofinal in $\Gamma_L$.
Then the natural inclusion~$K\to L$ extends uniquely to an embedding~$K^{\operatorname{c}}\to L^{\operatorname{c}}$ 
of valued fields by [ADH, 3.2.20]. It is easy to check that this is even an embedding of valued {\em differential\/} fields; we identify~$K^{\operatorname{c}}$ with a valued differential subfield of $L^{\operatorname{c}}$ via this embedding.   

\begin{cor}\label{cor:approx LP+f}
Let $a,b_1,\dots,b_r\in L^{\operatorname{c}}$ and $f\in K^{\operatorname{c}}$ be such that in $L^{\operatorname{c}}[\der]$, 
$$A\ :=\ L_{P_{+f}}\ =\ a(\der-b_1)\cdots(\der-b_r),\qquad a,f\neq 0,\quad \fv:=\fv(A)\prec 1,$$
and let $w\in\N$. Then there  are $a^{\smallbullet},b_1^{\smallbullet},\dots,b_r^{\smallbullet}\in L$ and $f^{\smallbullet}\in K$ such that 
$$a^{\smallbullet}\ \sim\ a, \qquad f^{\smallbullet}\ \sim\ f, \qquad A^{\smallbullet}:=L_{P_{+f^{\smallbullet}}}\ \sim\ A,\qquad \order A^{\smallbullet}\ =\ r, \qquad
 \fv(A^{\smallbullet})\ \sim\ \fv,$$
and such that for $\Delta:=\big\{\alpha\in \Gamma_L:\, \alpha=o\big(v(\fv)\big)\big\}$ we have in $L[\der]$,
$$A^{\smallbullet}\ =\ a^{\smallbullet}(\der-b_1^{\smallbullet})\cdots(\der-b_r^{\smallbullet}) + E,\qquad  E\prec_{\Delta} \fv^{w+1} A.$$
\end{cor}
\begin{proof} 
Let $\gamma_1\in \Gamma_L^{\ge}$ be as in Corollary~\ref{corcorappr} applied to $L^{\operatorname{c}}$ in place of $K$, and take~$\gamma_2\in\Gamma$ such that $\gamma_2\geq \max\{\gamma_1,\frac{1}{4}vf\}+vA$ and
$\gamma_2\geq v\big((P_{+f})_{\i} \big)$ for all $\i$ with~$(P_{+f})_{\i}\neq 0$.
Let $\gamma\in\Gamma$ and $\gamma>\gamma_2$. Then 
$\gamma-vA>\gamma_1$. By the
density of~$K$,~$L$ in~$K^{\operatorname{c}}$,~$L^{\operatorname{c}}$, respectively, we can take
$a^{\smallbullet},b_1^{\smallbullet},\dots,b_r^{\smallbullet}\in L$ and $f^{\smallbullet}\in K$ such that
$$v(a-a^{\smallbullet})\ \ge\ va+(\gamma-vA),\qquad v(b_i-b_i^{\smallbullet})\geq (r+1)(\gamma-vA)\ \text{ for $i=1,\dots,r$,}$$ and 
$v(f-f^{\smallbullet})\geq 4(\gamma-vA) > vf$. Then $a^{\smallbullet}\sim a$, $f^{\smallbullet}\sim f$, and by Corollary~\ref{corcorappr},
$$v\big( P_{+f^{\smallbullet}}  - P_{+f}\big)\ \geq\ \gamma, \quad
A^{\smallbullet}:= L_{P_{+f^{\smallbullet}}}\ =\ a^{\smallbullet}(\der-b_1^{\smallbullet})\cdots(\der-b_r^{\smallbullet}) + E, \quad vE\geq \gamma.$$
Hence 
 $(P_{+f^{\smallbullet}})_{\i} \sim (P_{+f})_{\i}$ if $(P_{+f})_{\i}\neq 0$, and
 $v\big((P_{+f^{\smallbullet}})_{\i}\big)> \gamma_2\ge vA$ if $(P_{+f})_{\i}=0$, so~${A^{\smallbullet}\sim A}$, $\order A^{\smallbullet}=r$, and $\fv(A^{\smallbullet})\sim \fv$.
Choosing $\gamma$ so that also $\gamma > v(\fv^{w+1} A)+\Delta$  we achieve in addition that $E\prec_{\Delta} \fv^{w+1} A$. 
\end{proof}

\subsection*{Keeping it real} 
{\em In this subsection $H$ is a real closed $H$-asymptotic field with small derivation whose valuation ring is convex, with 
$\Gamma:= v(H^\times)\ne \{0\}$. In addition, $K$ is the asymptotic extension $H[\imag]$ of $H$ with $\imag^2=-1$.}\/ Then $H^{\operatorname{c}}$ is real closed and~${H^{\operatorname{c}}[\imag]=K^{\operatorname{c}}}$ as a valued field extension of $H$ according to Lemma~\ref{lem:Kc real closed}, and as an asymptotic field extension of $H$ by the discussion after Corollary~\ref{klemb}.
Using the real splittings from Definition~\ref{def:real splitting} we show here that we can ``preserve the reality of $A$'' in Corollary~\ref{cor:approx LP+f}.

\begin{lemma}\label{lem:approx real} Let $A\in H^{\operatorname{c}}[\der]$ be of order $r\ge 1$ and let $(g_1,\dots, g_r)\in H^{\operatorname{c}}[\imag]^r$
be a real splitting of $A$ over $H^{\operatorname{c}}[\imag]$. Then for every
$\gamma\in \Gamma$ there are $g_1^{\smallbullet},\dots,g_r^{\smallbullet}$
in~$H[\imag]$
such that $v(g_i-g_i^{\smallbullet}) > \gamma$ for $i=1,\dots,r$, 
$$A^{\smallbullet}\ :=\ (\der-g_1^{\smallbullet})\cdots(\der-g_r^{\smallbullet})\in H[\der],$$ 
and~$(g_1^{\smallbullet},\dots, g_r^{\smallbullet})$ is a real splitting of $A^{\smallbullet}$ over $H[\imag]$. 
\end{lemma}  
\begin{proof} We can reduce to the case where $r=1$ or $r=2$.
If $r=1$, then the lemma holds trivially, so suppose $r=2$. Then again the lemma holds trivially if~$g_1, g_2\in H^{\operatorname{c}}$, so we can assume instead that 
$$g_1\ =\ a-b\imag+b^\dagger, \quad g_2\ =\ a+b\imag, \qquad a\in H^{\operatorname{c}},\ b\in (H^{\operatorname{c}})^\times.$$
Let $\gamma\in\Gamma$ be given. The density of $H$ in $H^{\operatorname{c}}$ gives $a^{\smallbullet}\in H$ with $v(a-a^{\smallbullet})\ge\gamma$.
Next, choose $\gamma^{\smallbullet}\in\Gamma$ such that $\gamma^{\smallbullet}\ge \max\{\gamma,vb\}$ and $\alpha' >\gamma$ for all nonzero $\alpha>\gamma^{\smallbullet}-vb$ in~$\Gamma$,
and take $b^{\smallbullet}\in H$ with $v(b-b^{\smallbullet})>\gamma^{\smallbullet}$. Then $v(b-b^{\smallbullet})>\gamma$ and $b\sim b^{\smallbullet}$.
In fact, $b=b^{\smallbullet}(1+\varepsilon)$ where $v\varepsilon+vb=v(b-b^{\smallbullet})> \gamma^{\smallbullet}$ and so
$v\big( (b/b^{\smallbullet})^\dagger \big) = v(\varepsilon')>\gamma$. Set~$g_1^{\smallbullet}:= a^{\smallbullet}-b^{\smallbullet}\imag+{b^{\smallbullet}}^\dagger$ and $g_2^{\smallbullet}:= a^{\smallbullet}+b^{\smallbullet}\imag$. Then
\begin{align*} v(g_1 - g_1^{\smallbullet})\ &=\ 
v\big( a-a^{\smallbullet}+(b/b^{\smallbullet})^\dagger + (b^{\smallbullet}-b)\imag \big)\ >\ \gamma,\quad v(g_2-g_2^{\smallbullet})\ >\ \gamma,\\
( \der- g_1^{\smallbullet}) \cdot( \der-g_2^{\smallbullet})\ &=\  \der^2-\big(2a^{\smallbullet}+ b^{\smallbullet}{}^\dagger\big)\der + \big((-a^{\smallbullet})'+a^{\smallbullet}{}^2 + a^{\smallbullet} b^{\smallbullet}{}^\dagger + b^{\smallbullet}{}^2\big)\ \in\ H[\der].
\end{align*} 
Hence $(g_1^{\smallbullet}, g_2^{\smallbullet})$ is a real splitting of 
$A^{\smallbullet}:=(\der-g_1^{\smallbullet})(\der-g_2^{\smallbullet})\in H[\der]$.
\end{proof}

\noindent 
In the next two corollaries $a\in (H^{\operatorname{c}})^{\times}$ and
$b_1,\dots,b_r\in K^{\operatorname{c}}$ are such that 
$$A\ :=\ a(\der - b_1)\cdots(\der-b_r)\in H^{\operatorname{c}}[\der],$$
$(b_1,\dots,b_r)$ is a real splitting of $A$ over 
$K^{\operatorname{c}}$, and $\fv:=\fv(A)\prec 1$.  We set $\Delta:=\Delta(\fv)$.

\begin{cor}\label{cor:approx LP+f, real}  
Suppose  $A = L_{P_{+f}}$ with $P\in H\{Y\}$
of order $r\ge 1$ and $f$ in~$(H^{\operatorname{c}})^{\times}$. Let $\gamma\in \Gamma$ and $w\in \N$. Then there is $f^{\smallbullet}\in H^\times$ such that $v(f^{\smallbullet}- f) \ge \gamma$, 
\begin{equation}\label{eq:approx LP+f, real}\
f^{\smallbullet}\ \sim\ f, \quad A^{\smallbullet}\ :=\ L_{P_{+f^{\smallbullet}}}\ \sim\ A,\quad \order A^{\smallbullet}\ =\ r, \quad
 \fv(A^{\smallbullet})\ \sim\ \fv,
\end{equation}
and  we have $a^{\smallbullet}\in H^\times$, $b_1^{\smallbullet},\dots, b_r^{\smallbullet}\in K$, and $B^{\smallbullet},E^{\smallbullet}\in H[\der]$ with 
$A^{\smallbullet}= B^{\smallbullet} + E^{\smallbullet}$,  $E^{\smallbullet}\prec_{\Delta} \fv^{w+1} A$, such that
$$B^{\smallbullet}\ =\ a^{\smallbullet}(\der-b_1^{\smallbullet})\cdots(\der-b_r^{\smallbullet}),\qquad
v(a-a^{\smallbullet}),\
v(b_1-b_1^{\smallbullet}),\ \dots\ ,v(b_r-b_r^{\smallbullet})\ \geq\ \gamma,$$
and $(b_1^{\smallbullet},\dots, b_r^{\smallbullet})$ is a real splitting of $B^{\smallbullet}$ over $K$. 
\end{cor}
\begin{proof} We apply Corollary~\ref{cor:approx LP+f} with $H$,~$K$ in the role of $K$,~$L$, and take~$\gamma_1$,~$\gamma_2$ as in the proof of that corollary. We can assume $\gamma>\gamma_2$, so that $\gamma-vA>0$.
The density of $H$ in $H^{\operatorname{c}}$ gives $a^{\smallbullet}\in H$ such that $v(a-a^{\smallbullet})\geq \max\!\big\{va+(\gamma-vA), \gamma\big\}$ (so $a^{\smallbullet}\sim a$),
and Lemma~\ref{lem:approx real} gives $b_1^{\smallbullet},\dots,b_r^{\smallbullet}\in K$ such that
$v(b_i-b_i^{\smallbullet})\geq \max\!\big\{(r+1)(\gamma-vA), \gamma\big\}$ for $i=1,\dots,r$, and 
$(b_1^{\smallbullet},\dots, b_r^{\smallbullet})$ is a real splitting of $$B^{\smallbullet}\ :=\ a^{\smallbullet}(\der-b_1^{\smallbullet})\cdots(\der-b_r^{\smallbullet})\in H[\der]$$ over $K$.
Take $f^{\smallbullet}\in H$ with $v(f-f^{\smallbullet})\geq \max\!\big\{4(\gamma-vA),\gamma\big\}$.
Then \eqref{eq:approx LP+f, real} follows from the proof of Corollary~\ref{cor:approx LP+f}. We can increase $\gamma$ so that $\gamma> v(\fv^{w+1}A)+\Delta$, and then we have
$A^{\smallbullet}-B^{\smallbullet} \prec_\Delta \fv^{w+1} A$.
\end{proof}

\noindent
This result persists after multiplicative conjugation:

\begin{cor}\label{cor:approx LP+f, real, general}
Suppose  $A= L_{P_{+f,\times\fm}}$ with $P\in H\{Y\}$ of order $r\ge 1$, and $f$ in~$(H^{\operatorname{c}})^{\times}$,  $\fm\in H^\times$. Let $\gamma\in\Gamma$, $w\in\N$. Then there is $f^{\smallbullet}\in H^\times$ such that 
$$v(f^{\smallbullet}-f) \geq \gamma,\quad f^{\smallbullet}\ \sim\ f, \quad A^{\smallbullet}\ :=\ L_{P_{+f^{\smallbullet},\times\fm}}\ \sim\ A,\quad 
\order A^{\smallbullet}\ =\ r, \quad
 \fv(A^{\smallbullet})\ \sim\ \fv,$$
and we have $a^{\smallbullet}\in H^\times$, $b_1^{\smallbullet},\dots, b_r^{\smallbullet}\in K$, and $B^{\smallbullet},E^{\smallbullet}\in H[\der]$ with the properties stated in the previous corollary.
\end{cor}
\begin{proof}
Put $Q:=P_{\times\fm}\in H\{Y\}$, $g:=f/\fm\in H^{\operatorname{c}}$; then $Q_{+g}=P_{+f,\times\fm}$. Applying the previous corollary to $Q$, $g$ in place of $P$, $f$ yields
$g^{\smallbullet}\in H^\times$, $a^{\smallbullet}\in H^\times$, and~${b_1^{\smallbullet},\dots,b_r^{\smallbullet}\in K}$  such that $v(g^{\smallbullet}-g) \geq\ \gamma-v\fm$,
$$g^{\smallbullet}\ \sim\ g, \qquad A^{\smallbullet}\ :=\ L_{Q_{+g^{\smallbullet}}}\ \sim\ A,\qquad 
\order A^{\smallbullet}\ =\ r, \qquad
 \fv(A^{\smallbullet})\ \sim\ \fv$$
and
$A^{\smallbullet} = B^{\smallbullet} + E^{\smallbullet}$, with $B^{\smallbullet},E^{\smallbullet}\in H[\der]$, $E^{\smallbullet}\prec_{\Delta} \fv^{w+1} A$,
and   
$$B^{\smallbullet}\ =\ a^{\smallbullet}(\der-b_1^{\smallbullet})\cdots(\der-b_r^{\smallbullet}),\qquad
v(a-a^{\smallbullet}),\
v(b_1-b_1^{\smallbullet}),\ \dots\ ,v(b_r-b_r^{\smallbullet})\ \geq\ \gamma,$$
and $(b_1^{\smallbullet},\dots, b_r^{\smallbullet})$ is a real splitting of $B^{\smallbullet}$ over $K$. 
Therefore  $f^{\smallbullet}:=g^{\smallbullet}\fm\in H^\times$ and~$a^{\smallbullet},b_1^{\smallbullet},\dots,b_r^{\smallbullet}$ have the required properties.
\end{proof}

\subsection*{Strong splitting} 
{\em In this subsection $H$ is a real closed $H$-field with small derivation and asymptotic integration}. Thus $K:=H[\imag]$ is a $\d$-valued extension of $H$.
Let~$A\in K[\der]^{\neq}$ have order $r\ge 1$ and set $\fv:=\fv(A)$, and let  $f$, $g$, $h$ (possibly subscripted) range over $K$. 
Recall from Section~\ref{sec:diff ops and diff polys} that  a splitting of $A$ over $K$ is an $r$-tuple $(g_1,\dots,g_r)$ such that
$$A\ =\ f(\der-g_1)\cdots(\der-g_r)\quad\text{where $f\neq 0$.}$$
We call such a splitting $(g_1,\dots,g_r)$ of $A$ over $K$ {\bf strong} if 
$\Re g_j\succeq \fv^\dagger$ for $j=1,\dots,r$, and
we say that $A$ {\bf  splits strongly over~$K$} if there is a strong splitting of $A$ over~$K$.
This notion is mainly of interest for $\fv\prec 1$, since otherwise $\fv=1$, and then any splitting of 
$A$  over $K$ is a strong splitting of $A$ over $K$. 

\begin{lemma}\label{lem:Ah splits strongly} 
Let $(g_1,\dots,g_r)$ be a strong splitting of $A$ over $K$. If ${h\neq 0}$, then $(g_1,\dots,g_r)$ is a strong splitting of $hA$ over $K$. If  $h\asymp 1$, then ${(g_1-h^\dagger,\dots,g_r-h^\dagger)}$ is a strong splitting of $Ah$  over $K$.
\end{lemma}

{\sloppy
\begin{proof}
The first statement is clear, so suppose $h\asymp 1$. Now use Lemma~\ref{lem:split and twist}
 and the fact that   $\fv \prec 1$ implies~${\Re h^\dagger\preceq h^\dagger\prec \fv^\dagger}$. If $\fv=1$, then use that $\fv(Ah)=1$ by Corollary~\ref{cor:111}. 
\end{proof}}

\begin{lemma}\label{lem:order 1 splits strongly}
Suppose $g\asymp\Re g$. Then $A=\der-g$ splits strongly over $K$.
\end{lemma}
\begin{proof}
Assuming $\fv\prec 1$ gives $\fv'\prec 1$, so  $\fv^\dagger \prec 1/\fv\asymp g\asymp\Re g$.
\end{proof}

\noindent
In particular, every $A\in H[\der]^{\ne}$ of order~$1$ splits strongly over $K$.

\begin{lemma}\label{lem:split strongly compconj}
Suppose $(g_1,\dots,g_r)$ is a strong splitting of $A$  over $K$ and $\fv\prec^\flat 1$. Let $\phi\preceq 1$ be active in $H$ and set $h_j:=\phi^{-1}\big(g_j-(r-j)\phi^\dagger\big)$   for $j=1,\dots,r$. Then~$(h_1,\dots,h_r)$ is a strong splitting of $A^\phi$ over $K^\phi=H^\phi[\imag]$.
\end{lemma}
\begin{proof}
By Lemma~\ref{lem:split and compconj}, $(h_1,\dots,h_r)$ is a splitting of $A^\phi$ over $K^\phi$. 
We have~$\phi^\dagger\prec 1\preceq\fv^\dagger$, 
so $\Re h_j\sim \phi^{-1} \Re g_j\succeq \phi^{-1}\fv^\dagger$ for $j=1,\dots,r$.
Set $\fw:=\fv(A^\phi)$ and~$\derdelta:=\phi^{-1}\der$. 
Lem\-ma~\ref{lem:v(Aphi)} gives
$\fv^\dagger\asymp \fw^\dagger$, so 
$\phi^{-1} \fv^\dagger  \asymp \derdelta(\fw)/\fw$. 
\end{proof}

\noindent
{\em In the next two results we assume that for all $q\in \Q^{>}$ and
$\fn\in H^\times$  there is given an element
$\fn^q\in H^\times$ such that $(\fn^q)^\dagger=q\fn^\dagger$ \textup{(}and thus
$v(\fn^q)=q\,v(\fn)$\textup{)}}. 

\begin{lemma}\label{lem:split strongly multconj} 
Suppose $(g_1,\dots,g_r)$ is a splitting of $A$  over $K$, $\fv\prec 1$, $\fn\in H^\times$, and $[\fv]\le[\fn]$. Then for all $q\in \Q^{>}$ with at most~$r$ exceptions, $(g_1-q\fn^\dagger,\dots,g_r-q\fn^\dagger)$ is a strong splitting of $A\fn^q$   over $K$.  
\end{lemma}
\begin{proof} 
Let $q\in \Q^{>}$. Then
$(g_1-q\fn^\dagger,\dots,g_r-q\fn^\dagger)$ is a splitting of $A\fn^q$ over $K$, by Lemma~\ref{lem:split and twist}. Moreover, $\big[\fv(A\fn^q)\big]\le [\fn]$,  by Lemma~\ref{lem:An}, so
$\fv(A\fn^q)^\dagger \preceq \fn^\dagger$.
Thus if $\Re g_j\not\sim q\fn^\dagger$ for $j=1,\dots,r$, then
$(g_1-q\fn^\dagger,\dots,g_r-q\fn^\dagger)$ is a strong splitting of~$A\fn^q$   over $K$. 
\end{proof}

\begin{cor}\label{cor:split strongly multconj}
Let $(P,\fm,\hat a)$ be a steep slot in $K$ of order $r\ge 1$ whose linear part~$L:=L_{P_{\times\fm}}$ splits over $K$ and such that $\hat a\prec_{\Delta} \fm$ for $\Delta:=\Delta\big(\fv(L)\big)$.  Then for all sufficiently small $q\in \Q^{>}$, any $\fn\asymp|\fv(L)|^q\fm$ in $K^\times$ gives a  steep refinement~$\big(P,\fn,\hat a\big)$ of $(P,\fm,\hat a)$ whose linear part $L_{P_{\times\fn}}$ splits strongly over $K$.
\end{cor}
\begin{proof} Note that $|f|\asymp f$ for all $f$.  Lemma~\ref{lem:steep1} gives $q_0\in \Q^{>}$ such that
for all~$q\in \Q^{>}$ with $q\le q_0$ and any $\fn\asymp|\fv(L)|^q\fm$,  
$(P,\fn, \hat a)$ is a steep refinement of~$(P,\fm, \hat a)$. Now apply Lemma~\ref{lem:split strongly multconj} with $L$, $\fv(L)$, $|\fv(L)|$ in the respective roles
of~$A$,~$\fv$,~$\fn$, and use Lemma~\ref{lem:Ah splits strongly} and the fact that for $\fn\asymp |\fv(L)|^q\fm$ we have~$L_{P_{\times \fn}}=L\cdot\fn/\fm=L|\fv(L)|^qh$ with~$h\asymp 1$. 
 \end{proof} 

\noindent
We finish this section with a useful fact on slots in $K$. 
Given such a slot $(P,\fm,\hat a)$, the element $\hat a$ lies in an immediate asymptotic extension of~$K$ that might not be of the form 
$\hat H[\imag]$ with $\hat H$ an immediate $H$-field extension of $H$. 
By the next lemma we can nevertheless often reduce to this situation, and more: 

\begin{lemma}\label{lem:hole in hat K}
Suppose $H$ is $\upo$-free. Then every $Z$-minimal slot in $K$ of positive order  is equivalent to a hole
$(P,\fm,\hat b)$  in $K$ with $\hat b\in\hat K=\hat H[\imag]$ for some immediate $\upo$-free newtonian $H$-field extension $\hat H$ of $H$.
\end{lemma}
\begin{proof}
Let $(P,\fm,\hat a)$ be a $Z$-minimal slot in $K$ of order $\ge 1$.
Take an immediate $\upo$-free newtonian $H$-field extension $\hat H$ of $H$; such $\hat H$ exists by \eqref{eq:14.0.1}. 
 Then $\hat K=\hat H[\imag]$ is also newtonian by~\eqref{eq:14.5.7}. 
Now apply  Corollary~\ref{cor:find zero of P, 2} with $L:=\hat K$  to obtain~$\hat b\in\hat K$ such that~$(P,\fm,\hat b)$ is a hole in $K$ equivalent to  $(P,\fm,\hat a)$.
\end{proof}

\section{Split-Normal Slots}\label{sec:split-normal holes}

\noindent
{\em In this section $H$ is a real closed $H$-field with small derivation and  asymptotic integration. We let $\mathcal{O}:= \mathcal{O}_H$ be its valuation ring and  $C:= C_H$ its constant field.  
We fix an immediate asymptotic extension $\hat H$ of~$H$ with valuation ring $\hat{\mathcal{O}}$ and an element~$\imag$ of an asymptotic extension of $\hat H$ with $\imag^2=-1$}. 
Then $\hat H$ is also an $H$-field by [ADH, 10.5.8], $\imag\notin\hat H$ and $K:=H[\imag]$ is an algebraic closure of $H$. With~$\hat K:= \hat H[\imag]$ we have the inclusion diagram
$$\xymatrix{ {\hat H} \ar@{-}[r]  &  \hat K ={\hat H}[\imag]  \\ 
 H \ar@{-}[u] \ar@{-}[r] &    K = H[\imag] \ar@{-}[u] }$$
By [ADH, 3.5.15, 10.5.7],
$K$ and $\hat K$ are $\d$-valued with valuation rings $\mathcal{O}+\mathcal{O}\imag$ and~$\hat{\mathcal{O}}+\hat{\mathcal{O}}\imag$ and
with the same constant field~$C[\imag]$, and $\hat K$ is an immediate extension of~$K$. Thus $H$, $K$, $\hat H$, $\hat K$ have the same
$H$-asymptotic couple $(\Gamma, \psi)$.

\begin{lemma} 
Let  $\hat a\in\hat H\setminus H$. Then $Z(H,\hat a) = Z\big(K,\hat a\big)\cap H\{Y\}$.
\end{lemma}
\begin{proof}
The inclusion ``$\supseteq$'' is obvious since  the Newton degree of a differential polynomial $Q\in H\{Y\}^{\neq}$
does not change when $H$ is replaced by its algebraic closure; see~[ADH, 11.1]. 
Conversely, let $P\in Z(H,\hat a)$. Then for all $\fv\in H^\times$ and~$a\in H$ such that $a-\hat a\prec\fv$ we have $\ndeg_{\prec\fv} H_{+a}\geq 1$.
Let $\fv\in H^\times$ and $z\in K$ be such that $z-\hat a\prec\fv$.
Take $a,b\in H$ such that $z=a+b\imag$. Then $a-\hat a,b\imag\prec\fv$ and hence~$\ndeg_{\prec\fv} P_{+z}=\ndeg_{\prec\fv} P_{+a}\geq 1$, using [ADH, 11.2.7].
Thus $P\in Z\big(K,\hat a\big)$.
\end{proof}

\begin{cor} \label{cor:holes in K vs holes in K[i]} 
Let $(P,\fm,\hat a)$ be a slot in $H$ with $\hat a\in\hat H$. Then 
$(P,\fm,\hat a)$ is also a slot in $K$, and if
   $(P,\fm,\hat a)$  is $Z$-minimal as a slot in $K$, then 
$(P,\fm,\hat a)$  is $Z$-minimal as a slot in $H$.
Moreover, $(P,\fm,\hat a)$ is a hole in $H$ iff  $(P,\fm,\hat a)$ is a hole in $K$, and
if $(P,\fm,\hat a)$ is a minimal hole in $K$, then $(P,\fm,\hat a)$ is a minimal hole in $H$.
\end{cor}
\begin{proof}
The first three claims are obvious from $\hat K$ being an immediate extension of $K$ and the previous lemma.
Suppose $(P,\fm,\hat a)$ is minimal as a hole  in $K$.
Let~$(Q,\fn,\tilde{b})$ be a hole in~$H$;
thus $\tilde{b}\in \widetilde{H}$ where $\widetilde{H}$ is an immediate asymptotic extension of $H$.
By the first part of the corollary applied to $(Q,\fn,\tilde{b})$ and $\widetilde{H}$ in place of
$(P,\fm,\hat a)$ and~$\hat H$,  respectively, $(Q,\fn,\tilde{b})$ is also a hole  in $K$. Hence $\cc(P) \leq \cc(Q)$, proving the last claim. 
\end{proof}

\noindent
In the next subsection we define the notion of a {\it split-normal}\/ slot in $H$. Later in this section we employ the results of Sections~\ref{sec:normalization}--\ref{sec:approx linear diff ops} to show, under suitable hypotheses on $H$, that minimal holes in $K$ of  order~$\geq 1$ give rise to a split-normal $Z$-minimal slots in $H$. (Theorem~\ref{thm:split-normal}.)
We then investigate which kinds of refinements preserve split-normality, and also consider a strengthening of split-normality.

\subsection*{Defining split-normality} {\em In this subsection  $b$ ranges over $H$ and $\fm, \fn$ over $H^\times$. Also, $(P, \fm, \hat a)$ is a slot in $H$ of order $r\ge 1$ with $\hat a\in \hat H\setminus H$ and linear part~$L:=L_{P_{\times \fm}}$. Set $w:=\wt(P)$, so $w\ge r$; if $\order L=r$, we set $\fv:=\fv(L)$}.

\begin{definition}\label{SN}  
We say that $(P,\fm,\hat a)$ is {\bf split-normal} if $\order L=r$, and \index{slot!split-normal}\index{split-normal}
\begin{itemize}
\item[(SN1)] $\fv\prec^\flat 1$; 
\item[(SN2)]  $(P_{\times\fm})_{\geq 1}=Q+R$ where $Q, R\in H\{Y\}$, $Q$ is homogeneous of degree~$1$ and order~$r$,  $L_Q$ splits over $K$, and $R\prec_{\Delta(\fv)} \fv^{w+1} (P_{\times\fm})_1$. 
\end{itemize}
\end{definition}

\noindent
Note that in (SN2) we do not require that $Q=(P_{\times\fm})_1$. 

\begin{lemma}\label{splnormalnormal} Suppose $(P,\fm,\hat a)$ is split-normal.  Then $(P,\fm, \hat a)$ is normal, and with~$Q$,~$R$ as in \textup{(SN2)} we have $(P_{\times \fm})_1-Q\prec_{\Delta(\fv)} \fv^{w+1}(P_{\times \fm})_1$, so $(P_{\times \fm})_1\sim Q$. 
\end{lemma}
\begin{proof} We have $(P_{\times \fm})_1=Q+R_1$ and $R_1\preceq R \prec_{\Delta(\fv)} \fv^{w+1} (P_{\times\fm})_1$,     and thus $${(P_{\times \fm})_1-Q\,\prec_{\Delta(\fv)}}\, \fv^{w+1}(P_{\times \fm})_1.$$ Now $(P,\fm,\hat a)$ is normal because $(P_{\times \fm})_{>1}= R_{>1} \prec_{\Delta(\fv)} \fv^{w+1}(P_{\times \fm})_1$.
\end{proof}

\noindent
If $(P,\fm,\hat a)$ is normal and $(P_{\times\fm})_{1}=Q+R$  where $Q, R\in H\{Y\}$, $Q$ is homogeneous of degree~$1$ and order $r$,  $L_Q$ splits over $K$, and $R\prec_{\Delta(\fv)} \fv^{w+1} (P_{\times\fm})_1$, then~$(P,\fm, \hat a)$ is split-normal. 
Thus  if $(P,\fm,\hat a)$ is normal and $L$ splits over $K$, then $(P,\fm,\hat a)$ is split-normal; in particular, if
$(P,\fm,\hat a)$ is normal  of order $r=1$, then it is split-normal. 
If~$(P,\fm,\hat a)$ is  split-normal, then so are $(bP,\fm,\hat a)$ for $b\neq 0$ and~$(P_{\times\fn},\fm/\fn,\hat a/\fn)$. 
Note also that if  $(P,\fm,\hat a)$ is split-normal, 
then with $Q$ as in~(SN2) we have $\fv(L)\sim\fv(L_Q)$, by Lemma~\ref{lem:fv of perturbed op}. If $(P,\fm, \hat a)$ is split-normal and $H$ is $\upl$-free, then $\exc^{\ev}(L)=\exc^{\ev}(L_Q)$ with $Q$ as in (SN2),  by
 Lemmas~\ref{splnormalnormal} and~\ref{cor:excev stability}.

\begin{lemma}\label{lem:split-normal comp conj}
Suppose $(P,\fm,\hat a)$ is split-normal and  $\phi\preceq 1$ is active in $H$ and~$\phi>0$ \textup{(}so $H^\phi$ is still an $H$-field\textup{)}. Then the slot $(P^\phi,\fm,\hat a)$ in $H^\phi$ is split-normal.
\end{lemma}
\begin{proof}
We first arrange $\fm=1$. Note that $L_{P^\phi}=L^\phi$ has order $r$. Put $\fw:=\fv(L_{P^\phi})$, and take $Q$, $R$ as in (SN2). Then 
$\fv\asymp_{\Delta(\fv)}\fw\prec^\flat_\phi 1$ by   Lemma~\ref{lem:v(Aphi)}. Moreover, $L_{Q^\phi}=L_Q^\phi$ splits over $K^\phi$; see [ADH, p.~291] or Lemma~\ref{lem:split and compconj}. By
[ADH, 11.1.4],  
$$R^\phi\ \asymp_{\Delta(\fv)} R\ \prec_{\Delta(\fv)}\ \fv^{w+1} P_1\ \asymp_{\Delta(\fv)}\ \fw^{w+1} P_1^\phi,$$
so $(P^\phi,\fm,\hat a)$  is split-normal. 
\end{proof}

\noindent
Recall: ``$(P^\phi,\fm, \hat a)$ is    
split-normal, eventually'' means that there is an active $\phi_0$ in~$H$ such that $(P^\phi,\fm, \hat a)$ is split-normal for all active $\phi\preceq \phi_0$ in $H$.
Since we need to preserve $H$ being an $H$-field when compositionally conjugating, we say:  {\em $(P^\phi,\fm, \hat a)$ is eventually 
split-normal\/} if there exists an active $\phi_0$ in $H$ such that $(P^\phi,\fm, \hat a)$ is split-normal for all active $\phi\preceq \phi_0$ in $H$ with $\phi>0$. 
We use this terminology in a similar way with ``split-normal'' replaced by other properties of slots of  order~$r \geq 1$ in real closed $H$-fields with small derivation and asymptotic integration, such as ``deep'' and ``deep and split-normal''.

\subsection*{Achieving split-normality} {\em Assume $H$ is $\upo$-free and $(P,\fm,\hat a)$ is a minimal hole in $K=H[\imag]$ of order $r\ge 1$, with  $\fm\in H^\times$ and $\hat a\in\hat K\setminus K$.}\/ Note that then~$K$ is $\upo$-free by [ADH, 11.7.23],  $K$ is
$(r-1)$-newtonian by Corollary~\ref{minholenewt}, and~$K$ is $r$-linearly closed  by Corollary~\ref{corminholenewt}.
In particular, the linear part of $(P,\fm,\hat a)$ is~$0$ or splits over~$K$. 
If  $\deg P=1$, then $r=1$ by
Corollary~\ref{cor:minhole deg 1}. 
If $\deg P >1$, then
$K$ and~$H$ are $r$-linearly newtonian by Corollary~\ref{degmorethanone} and Lemma~\ref{lem:descent r-linear newt}.  
In particular, if~$H$ is $1$-linearly newtonian, then $H$ is $r$-linearly newtonian.  
{\em In this subsection we let $a$ range over $K$, $b$, $c$ over $H$, and $\fn$ over $H^\times$}.

{\sloppy
\begin{lemma}\label{lem:hole in K with same c as P} 
Let~$(Q,\fn,\hat b)$ be a hole in $H$ with $\cc(Q)\le \cc(P)$ and
$\hat b\in \hat H$. Then~$\cc(Q)=\cc(P)$, $(Q,\fn,\hat b)$ is minimal and remains a minimal hole in $K$. The linear part of~$(Q,\fn,\hat b)$
is $0$ or splits over~$K$, and $(Q,\fn,\hat b)$ has a refinement~$(Q_{+b},\fp,\hat b-b)$ \textup{(}in $H$\textup{)} such that $(Q^\phi_{+b},\fp,\hat b-b)$ is eventually deep and split-normal. 
\end{lemma} }
\begin{proof}
By Corollary~\ref{cor:holes in K vs holes in K[i]}, 
$(Q,\fn,\hat b)$ is a  hole in $K$, and this hole in $K$ is minimal with~$\cc(Q)=\cc(P)$, since $(P,\fm,\hat a)$ is minimal. By Corollary~\ref{cor:holes in K vs holes in K[i]} again, $(Q,\fn,\hat b)$  as a hole in $H$ is also minimal.  Since $K$ is $r$-linearly closed, the  linear part of~$(Q,\fn,\hat b)$ is~$0$ or  splits over $K$. 
Corollary~\ref{cor:mainthm} (with $H$ instead of $K$) gives a refinement~$(Q_{+b},\fp,\hat b-b)$ of the minimal hole $(Q,\fn,\hat b)$ in $H$
such that $(Q_{+b}^\phi,\fp,\hat b-b)$ is deep and normal, eventually. Thus the linear part of $(Q_{+b},\fp,\hat b-b)$ is not $0$, and
as $\cc(Q_{+b})=\cc(P)$, this linear part splits over $K$. Hence for active $\phi$ in $H$ the linear part of
$(Q_{+b}^\phi,\fp,\hat b-b)$ splits over $K^\phi=H^\phi[\imag]$. Thus $(Q^\phi_{+b},\fp,\hat b-b)$ is eventually split-normal.
\end{proof}

\noindent
Now $\hat a=\hat b + \hat c\, \imag$ with $\hat b, \hat c\in \hat H$, and $\hat b, \hat c \prec \fm$. Moreover, $\hat b\notin H$ or $\hat c\notin H$.  
Since $\hat a$  is differentially algebraic over~$H$, so is its conjugate  $\hat b - \hat c\, \imag$, and therefore  its 
real and imaginary parts~$\hat b$ and~$\hat c$ are differentially algebraic over~$H$; thus $Z(H,\hat b)\ne \emptyset$ for $\hat b\notin H$, and
$Z(H,\hat c)\ne \emptyset$ for $\hat c\notin H$. More precisely:

\begin{lemma}\label{kb} We have $\operatorname{trdeg}\!\big(H\<\hat b\>|H\big)\le 2r$. If $\hat b\notin H$, then $Z(H,\hat b)\cap H[Y]=\emptyset$, so
$1\le \order Q \le 2r$ for all $Q\in Z(H,\hat b)$ of minimal complexity. 
These statements also hold for $\hat c$ instead of $\hat b$. 
\end{lemma} 
\begin{proof} The first statement follows from $\hat b \in  H\<\hat b +\hat c\, \imag, \hat b - \hat c\, \imag\>$. Suppose $\hat b\notin H$. If~$Q\in Z(H,\hat b)$ has minimal complexity, then [ADH, 11.4.8] yields an element $f$ in a proper immediate asymptotic extension of $H$ with  $Q(f)=0$, so $Q\notin H[Y]$.
\end{proof}

\begin{lemma}\label{lem:r=deg P=1}
Suppose $\deg P=1$ and $\hat b\notin H$.  
Let $Q\in Z(H,\hat b)$ be of minimal complexity; then either $\order{Q}=1$, or $\order{Q}=2$, $\deg{Q}=1$.
Let $\hat{Q}\in H\{Y\}$ be a minimal annihilator of
$\hat b$ over $H$;
then either $\order{\hat{Q}}=1$, or $\order{\hat{Q}}=2$, $\deg{\hat{Q}}=1$, and $L_{\hat{Q}}\in H[\der]$ splits over $K$.
\end{lemma}
\begin{proof}
Recall that $r=1$ by Corollary~\ref{cor:minhole deg 1}. Example~\ref{ex:lclm compl conj} and
Lemma~\ref{lem:lclm compl conj}  
give a $\tilde Q\in H\{Y\}$ of degree $1$ and order~$1$ or $2$ such that $\tilde Q(\hat b)=0$ and~$L_{\tilde Q}$
splits over $K$. Then $\cc(\tilde Q)=(1,1,1)$ or $\cc(\tilde Q)=(2,1,1)$, which proves the claim about~$Q$, using also Lemma~\ref{kb}. 
Also, $\tilde{Q},\hat{Q}\in Z(H,\hat b)$, hence $\cc(Q)\leq \cc(\hat Q)\leq\cc(\tilde Q)$. If~$\cc(\hat{Q})=\cc(\tilde Q)$, then $\hat{Q}=a\tilde Q$ for some~$a\in H^\times$. The claim about $\hat Q$ now follows easily. 
\end{proof}

\noindent
By Corollary~\ref{cor:mainthm} and Lemma~\ref{ufm}, our minimal hole $(P,\fm,\hat a)$ in $K$ has a refinement $(P_{+a},\fn,\hat a-a)$ such that eventually $(P_{+a}^\phi,\fn,\hat a-a)$ is deep and normal. 
Moreover, as $K$ is $r$-linearly closed, the
linear part of $(P_{+a}^\phi,\fn,\hat a-a)$, for active~$\phi$ in~$H$,  splits over $K^\phi=H^\phi[\imag]$. Our main goal in this subsection is to prove  analogues
of these facts for  suitable $Z$-minimal slots $(Q,\fm,\hat b)$ or $(R,\fm,\hat c)$ in~$H$:

 \begin{theorem}\label{thm:split-normal} Recall from just before Lemma~\ref{kb} that $\hat a= \hat b + \hat c\, \imag$ with $\hat b, \hat c\in \hat H$. 
If $H$ is $1$-linearly newtonian, then one of the following holds:
\begin{list}{}{\leftmargin=2em  \labelwidth=2em}
\item[$\mathrm{(i)}$] $\hat b\notin H$ and some $Z$-minimal slot $(Q,\fm,\hat b)$ in $H$ has  a re\-fine\-ment~${(Q_{+b},\fn,\hat b-b)}$ such that $(Q^\phi_{+b},\fn,\hat b-b)$ is eventually deep and split-nor\-mal; 
\item[$\mathrm{(ii)}$] $\hat c\notin H$ and some $Z$-minimal slot $(R,\fm,\hat c)$ in $H$ has a refinement~${(R_{+c},\fn,\hat c-c)}$ such that $(R^\phi_{+c},\fn,\hat c-c)$ is eventually deep and split-normal. 
\end{list}
\end{theorem}

\noindent
Lemmas~\ref{lem:hat c in K}, \ref{lem:hat b in K} and Corollaries~\ref{cor:evsplitnormal, 1}--\ref{cor:r=deg P=1, 2}
below are more precise (only Corollary~\ref{cor:r=deg P=1, 1} has $H$ being $1$-linearly newtonian as a hypothesis) and together give Theorem~\ref{thm:split-normal}.  
We first deal with the case where $\hat b$ or $\hat c$ is in $H$:

\begin{lemma}\label{lem:hat c in K}
Suppose $\hat c\in H$. Then some hole $(Q,\fm,\hat b)$ in $H$ has the same complexity
as $(P,\fm,\hat a)$. Any such hole $(Q,\fm,\hat b)$ in $H$ is minimal and has a refinement~$(Q_{+b},\fn,\hat b-b)$ such that $(Q^\phi_{+b},\fn,\hat b-b)$ is eventually deep and split-normal. 
\end{lemma}
\begin{proof}
Let $A,B\in H\{Y\}$ be such that $P_{+\hat c\,\imag}(Y)=A(Y)+B(Y)\,\imag$.
Then $A(\hat b)=B(\hat b)=0$. If $A\ne 0$, then $\cc(A)\le \cc(P)$ gives that
$Q:=A$ has the desired property by Lemma~\ref{lem:hole in K with same c as P}. If $B\ne 0$, then likewise $Q:= B$ has the desired property.
The rest also follows from that lemma. \end{proof}

\noindent
Thus if $\hat{c}\in H$, we obtain a strong version of (i) in Theorem~\ref{thm:split-normal}. Likewise, the next lemma gives a strong version of (ii) in Theorem~\ref{thm:split-normal} if $\hat{b}\in H$.

\begin{lemma}\label{lem:hat b in K} 
Suppose $\hat b\in H$. Then there is a hole $(R,\fm,\hat c)$ in $H$ with the same complexity
as $(P,\fm,\hat a)$. Every such hole in $H$ is minimal and   has a refinement~$(R_{+c},\fn,\hat c-c)$ such that $(R^\phi_{+c},\fn,\hat c-c)$ is eventually deep and split-normal. 
\end{lemma}

\noindent
This follows by applying Lemma~\ref{lem:hat c in K}  with $(P,\fm,\hat a)$ replaced by  the minimal hole~$\big(P_{\times\imag},\fm,-\imag\hat a\big)$ in $K$, which has the same complexity as $(P,\fm,\hat a)$.

\medskip
\noindent
{\em We assume in the rest of this subsection that $\hat b, \hat c\notin H$ and that 
$Q\in Z(H,\hat b)$ has minimal complexity}.
Hence~$(Q,\fm,\hat b)$ is a $Z$-minimal slot in $H$, and so is every refinement of $(Q,\fm,\hat b)$.
If~$(P_{+a},\fn,\hat a-a)$ is a refinement of~$(P,\fm,\hat a)$ and $b=\Re a$, then
$(Q_{+b},\fn,\hat b-b)$ is a refinement of $(Q,\fm,\hat b)$.
Conversely,  if~$(Q_{+b},\fn,\hat b-b)$ is a refinement of~$(Q,\fm,\hat b)$ and $v\big({\hat b-H}\big)\subseteq v({\hat c-H})$, then
Lemma~\ref{lem:same width} yields a refinement $(P_{+a},\fn,\hat a-a)$ of~$(P,\fm,\hat a)$ with $\Re a=b$. Recall from that lemma
that~$v({\hat b-H})\subseteq v({\hat c-H})$ is equivalent to~$v({\hat a-K})=v({\hat b-H})$; in this case, 
 $(P,\fm,\hat a)$ is special iff $(Q,\fm,\hat b)$ is special. 
 Recall also that if $(Q,\fm,\hat b)$ is deep, then so is each of its refinements 
$(Q_{+b},\fm,\hat b-b)$, by Corollary~\ref{cor:deep 2, cracks}.

\medskip
\noindent 
Here is a key technical fact underlying Theorem~\ref{thm:split-normal}:

\begin{prop}\label{evsplitnormal}
Suppose  the hole $(P,\fm,\hat a)$ in $K$ is special, the slot~$(Q,\fm,\hat b)$ in~$H$ is normal,
and $v\big({\hat b-H}\big)\subseteq v({\hat c-H})$. Then some refinement $(Q_{+b},\fm,{\hat b-b})$ of $(Q,\fm,\hat b)$ has the property that $(Q^\phi_{+b},\fm,{\hat b-b})$ is eventually split-normal.
\end{prop}

\begin{proof}
Replacing $(P,\fm,\hat a)$, $(Q,\fm,\hat b)$ by $(P_{\times\fm},1,\hat a/\fm)$, $(Q_{\times\fm},1,\hat b/
\fm)$, respectively, we reduce to the case $\fm=1$; then $\hat a, \hat b\prec 1$.
Since $\hat a$ is special over~$K=H[\imag]$,
$$\Delta\ :=\ \big\{ \delta\in\Gamma:\  \abs{\delta}\in v(\hat a-K) \big\}$$
is a convex subgroup of~$\Gamma$ which is cofinal in $v(\hat a-K)$ and hence
in $v(\hat b-H)$, so $\hat b$ is special over $H$.  Compositionally conjugate $H$, $\hat H$, $K$, $\hat K$ 
by a suitable active~$\phi\preceq 1$ in $H^{>}$, and replace $P$, $Q$ by $P^\phi$, $Q^\phi$, 
to arrange  $\Gamma^\flat\subseteq\Delta$; in particular, $\Psi\subseteq v(\hat b-H)$ and  $\psi(\Delta^{\neq})\subseteq\Delta$.
Multiplying $P$, $Q$ by suitable elements of $H^\times$ we also arrange that~$P,Q\asymp 1$. 
By Lemma~\ref{lem:split-normal comp conj} it suffices to show that then~$(Q,1,\hat b)$ has a split-normal refinement
$(Q_{+b},1,\hat b-b)$, and this is what we shall do. 

Note that $H$,~$\hat H$,~$K$,~$\hat K$ have small derivation, so the specializations~$\dot H$, $\dot{\hat H}$, $\dot K$, $\dot{\hat K}$ of~$H$, $\hat H$, $K$, $\hat K$, respectively, by~$\Delta$, are valued differential fields with small derivation. These specializations are asymptotic with 
asymptotic couple~$(\Delta,\psi|\Delta^{\neq})$, and of $H$-type with asymptotic integration, by [ADH, 9.4.12]; in addition they are $\d$-valued, by [ADH, 10.1.8]. 
The natural inclusions $\dot{\mathcal O}\to\dot{\mathcal O}_K$,  
$\dot{\mathcal O}\to\dot{\mathcal O}_{\hat H}$, $\dot{\mathcal O}_{\hat H} \to \dot{\mathcal O}_{\hat K}$, and
$\dot{\mathcal O}_K\to\dot{\mathcal O}_{\hat K}$ 
induce valued differential field embeddings
$\dot H\to\dot K$, $\dot H\to\dot{\hat H}$, $\dot{\hat H}\to\dot{\hat K}$ and $\dot K\to\dot{\hat K}$, which we make into inclusions by
the usual identifications; see [ADH, pp.~405--406].
 By Lemma~\ref{lem:dotK real closed} and the remarks preceding it, $\dot H$ is real closed with convex valuation ring and $\dot K$ is an algebraic closure of $\dot H$. Moreover,~$\dot{\hat H}$ is an immediate
extension of $\dot H$ and $\dot{\hat K}$ is an immediate
extension of $\dot K$.
Denoting the image of $\imag$ under the residue morphism $\dot{{\mathcal O}}_{\hat K}\to \dot{\hat K}$ by the same symbol, we then have $\dot K=\dot H[\imag]$, $\dot{\hat K}=\dot{\hat H}[\imag]$, and $\imag\notin\dot{\hat H}$. This gives the following inclusion diagram:
$$\xymatrix{ \dot{\hat H} \ar@{-}[r]  & \dot{\hat K} =\dot{\hat H}[\imag]  \\ 
\dot H \ar@{-}[u] \ar@{-}[r] &   \dot K =\dot H[\imag] \ar@{-}[u] }$$
Now $\hat a\in\mathcal O_{\hat K}\subseteq\dot{\mathcal O}_{\hat K}$ and $\hat b,\hat c\in\mathcal O_{\hat H}\subseteq\dot{\mathcal O}_{\hat H}$, and 
$\dot{\hat a}=\dot{\hat b} + \dot{\hat c}\,\imag$, $\Re \dot{\hat a}=\dot{\hat b}$, $\Im\dot{\hat a}=\dot{\hat c}$. 
For all~$a\in\dot{\mathcal O}_K$ we have $v(\dot{\hat a}-\dot a)=v(\hat a-a)\in\Delta$, hence
$\dot{\hat a}\notin \dot K$; likewise~${v(\hat b -b)}\in \Delta$
for all $b\in \dot{\mathcal O}$, so $\dot{\hat b}\notin \dot H$. Moreover, for all $\delta\in\Delta$ there is an $a\in\dot{\mathcal O}_K$ with~${v(\dot{\hat a}-\dot a)}=\delta$;
hence~$\dot{\hat a}$ is the limit of a c-sequence in~$\dot K$. This leads us to consider the completions~$\dot{H}^{\operatorname{c}}$ and $\dot{K}^{\operatorname{c}}$ of $\dot{H}$ and $\dot{K}$.
By [ADH, 4.4.11] and Lemma~\ref{lem:Kc real closed}, these  yield an inclusion diagram of valued differential field extensions:
$$\xymatrix{ \dot{H}^{\operatorname{c}} \ar@{-}[r]  & \dot{K}^{\operatorname{c}}=\dot{H}^{\operatorname{c}}[\imag]  \\ 
\dot H \ar@{-}[u] \ar@{-}[r] &   \dot K =\dot H[\imag] \ar@{-}[u] }$$
where $\dot H^{\operatorname{c}}$ is real closed with algebraic closure $\dot K^{\operatorname{c}}=
\dot H^{\operatorname{c}}[\imag]$. These completions are $\d$-valued by [ADH, 9.1.6]. By Corollary~\ref{cor:Kc newtonian},
$\dot{K}$ and $\dot K^{\operatorname{c}}$ are $\upo$-free and ${(r-1)}$-newtonian;  thus~$\dot K^{\operatorname{c}}$ is $r$-linearly closed by Corollary~\ref{14.5.3.r}. We identify the valued differential subfield   $\dot K\big\<\!\Re\dot{\hat a},\Im\dot{\hat a}\big\>$ of $\dot{\hat K}$ with its image under
the embedding into~$\dot K^{\operatorname{c}}$ over $\dot K$ from
Corollary~\ref{cor:embed into Kc}; then $\dot{\hat a}\in \dot{K}^{\operatorname{c}}$ and
$\dot{\hat b}=\Re\dot{\hat a}\in \dot H^{\operatorname{c}}$.  This leads to the next inclusion diagram:
$$\xymatrix{ \dot{H}^{\operatorname{c}} \ar@{-}[r]  & \dot{K}^{\operatorname{c}}   \\ 
\dot H\<\dot{\hat b}\>\ar@{-}[u] & \dot K\<\dot{\hat a}\>\ar@{-}[u] \\
\dot H \ar@{-}[u] \ar@{->}[r] &   \dot K   \ar@{-}[u] }$$
By Corollary~\ref{cor:ZKhata}, 
$\dot P\in\dot K\{Y\}$ is a minimal annihilator of $\dot{\hat a}$ over $\dot K$ and has the same complexity as $P$. Likewise, 
$\dot Q\in\dot H\{Y\}$ is a minimal annihilator of $\dot{\hat b}$ over
$\dot H$ and has the same complexity as $Q$. 
Let $s:=\order Q=\order\dot Q$, so $1\le s\le 2r$ by Lemma~\ref{kb}, and    
the linear part~$A\in \dot{H}^{\operatorname{c}}[\der]$ of $\dot Q_{+\dot{\hat b}}$ has order $s$ as well. By [ADH, 5.1.37] applied to $\dot H^{\operatorname{c}}$, $\dot H$, $\dot P$, $\dot Q$, $\dot{\hat a}$ in the role of $K$, $F$, $P$, $S$, $f$, respectively, $A$ splits over~$\dot K^{\operatorname{c}}=\dot H^{\operatorname{c}}[\imag]$, so Lemma~\ref{hkspl} gives a real splitting $(g_1,\dots, g_s)$ of $A$ over $\dot{K}^{\operatorname{c}}$: 
$$ A\ =\ f(\der-g_1)\cdots(\der-g_s),\qquad
f,g_1,\dots,g_s\in \dot K^{\operatorname{c}},\ f\neq 0.$$
The slot $(Q,1,\hat b)$ in $H$ is normal, so $\fv(L_{Q_{+\hat b}}) \sim \fv(L_Q)\prec^\flat 1$ by Lemma~\ref{lem:linear part, new}, hence 
$\fv(A)\prec^\flat 1$ in $\dot K^{\operatorname{c}}$ by Lemma~\ref{lem:dotfv}. 
Then 
Corollary~\ref{cor:approx LP+f, real}  gives $a,b\in\dot{\mathcal O}$  and~$b_1,\dots,b_s\in\dot{\mathcal O}_K$ with $\dot a, \dot b\ne 0$ in
$\dot H$ such that for the linear part $\tilde{A}\in \dot{H}[\der]$ of~$\dot Q_{+\dot b}$,
$$\dot b\ \sim\ \dot{\hat b},\qquad \tilde{A}\ \sim\ A,\qquad \order \tilde{A}\ =\ s, \qquad
\fw\ :=\ \fv(\tilde{A})\ \sim\ \fv(A),$$
and such that for $w:=\wt(Q)$ and with $\Delta(\fw)\subseteq \Delta$: 
$$ \tilde{A} =  \tilde{B} +  \tilde{E},\ \tilde{B} = \dot a(\der-\dot b_1)\cdots(\der-\dot b_s)\in \dot{H}[\der],\quad 
\tilde{E}\in \dot{H}[\der],\quad  \tilde{E} \prec_{\Delta(\fw)}  \fw^{w+1} \tilde{A}, $$ 
and $(\dot{b}_1,\dots, \dot{b}_s)$ is a real splitting of $\tilde{B}$ over $\dot{K}$.  Lemma~\ref{lem:lift real splitting} shows that we can change $b_1,\dots, b_s$ if necessary, without changing $\dot{b}_1,\dots, \dot{b}_s$, to arrange that $B:=a(\der-b_1)\cdots (\der-b_s)$ lies in $\dot{\mathcal{O}}[\der]\subseteq H[\der]$ and
$(b_1,\dots, b_s)$ is a real splitting of $B$ over~$K$.  
Now $\hat b-b\prec \hat b\prec 1$, so
$(Q_{+b},1,\hat b-b)$ is a refinement of the normal slot~$(Q,1,\hat b)$. Hence
$(Q_{+b},1,\hat b-b)$ is normal by Proposition~\ref{normalrefine}, so
$\fv:=\fv(L_{Q_{+ b}})\prec^\flat 1$.
By Lem\-ma~\ref{lem:dotfv} we have $\dot\fv=\fw$, so $\Delta(\fv)=\Delta(\fw)\subseteq \Delta$. Hence in $H[\der]$:
$$L_{Q_{+b}}\ =\ B +   E, \quad E\in \dot{\mathcal O}[\der],\  E\prec_{\Delta(\fv)} \fv^{w+1} L_{Q_{+ b}}.$$
Thus $(Q_{+b},1,\hat b-b)$ is split-normal.
\end{proof}

\noindent
Recall from the beginning of this subsection that if~${\deg P>1}$, 
then $K=H[\imag]$ is $r$-linearly new\-tonian;  this allows
us to remove the assumptions that~$(P,\fm,\hat a)$ is special and~$(Q,\fm,\hat b)$ is normal in Proposition~\ref{evsplitnormal}, by reducing to that case:

\begin{cor}  \label{cor:evsplitnormal, 1}
Suppose  $\deg P>1$ and $v(\hat b -H)\subseteq v(\hat c - H)$.
Then~$(Q,\fm,\hat b)$  has a special refinement~$(Q_{+b},\fn,\hat b-b)$ such that $(Q^\phi_{+b},\fn,{\hat b-b})$ is eventually deep and split-normal. 
\end{cor}
\begin{proof}
By Lemmas~\ref{lem:quasilinear refinement} and~\ref{ufm}, the hole $(P,\fm,\hat a)$ in $K$ has a quasilinear refinement~${(P_{+a},\fn,\hat a-a)}$. (The use of Lemma~\ref{ufm} is because we require~${\fn\in H^\times}$.)
Let $b=\Re a$. Then, using Lemma~\ref{lem:same width} for the second equality,
$$v\big((\hat a-a)-K\big)\ =\ v(\hat a - K)\ =\ v(\hat b -H)\ =\ {v\big((\hat b-b)-H\big)},$$
and~${(Q_{+b},\fn,\hat b-b)}$ is a $Z$-minimal refinement of $(Q,\fm,\hat b)$.
We replace~$(P,\fm,\hat a)$ and~$(Q,\fm,\hat b)$ by
$(P_{+a},\fn,\hat a-a)$ and $(Q_{+b},\fn,\hat b-b)$, respectively, to arrange that the hole~$(P,\fm,\hat a)$ in $K$ is quasilinear.
Then
by Proposition~\ref{prop:hata special} and $K$ being $r$-linearly newtonian,  $(P,\fm,\hat a)$ is special. 
Hence~$(Q,\fm,\hat b)$ is also special, so Proposition~\ref{varmainthm} gives a refinement~$(Q_{+b},\fn,\hat b-b)$ of $(Q,\fm,\hat b)$ and an active~${\phi_0\in H^{>}}$ such that $(Q^{\phi_0}_{+b},\fn,{\hat b-b})$ is deep and normal. 
Refinements of~$(P,\fm,\hat a)$ remain quasilinear, by Corollary~\ref{cor:ref 2n}. Since $v(\hat b-H)\subseteq v(\hat c -H)$ we have
a refinement~$(P_{+a},\fn,\hat a-a)$ of~$(P,\fm,\hat a)$ with~$\Re a=b$.
Then by Lemma~\ref{speciallemma} the minimal hole $(P^{\phi_0}_{+a},\fn,{\hat a-a})$ in~$H^{\phi_0}[\imag]$ is special. Now apply Proposition~\ref{evsplitnormal} with~$H^{\phi_0}$, $(P^{\phi_0}_{+a},\fn,\hat a-a)$, $(Q^{\phi_0}_{+b},\fn,{\hat b-b})$
in place of $H$, $(P,\fm,\hat a)$,   $(Q,\fm,\hat b)$, respectively: it gives~${b_0\in H}$ and a 
refinement 
$$\big( (Q^{\phi_0}_{+b})_{+b_0},\fn, (\hat b-b)-b_0\big)\ =\ 
\big(Q^{\phi_0}_{+(b+b_0)},\fn, {\hat b - (b+b_0)}\big)$$ of 
$(Q^{\phi_0}_{+b},\fn,{\hat b-b})$, and thus a refinement~$\big(Q_{+(b+b_0)},\fn, {\hat b - (b+b_0)}\big)$ of $(Q_{+b},\fn,{\hat b-b})$, such that  $\big(Q^\phi_{+(b+b_0)}, \fn,{ \hat b - (b+b_0)}\big)$  is eventually split-normal. 
By the remark before Proposition~\ref{evsplitnormal}, $\big(Q^{\phi}_{+(b+b_0)},\fn, {\hat b - (b+b_0)}\big)$ is also eventually deep. 
\end{proof}

\noindent
Recall that $v(\hat b-H)\subseteq v(\hat c-H)$ or  $v(\hat c-H)\subseteq v(\hat b-H)$. The following corollary concerns the second case:

\begin{cor} \label{cor:evsplitnormal, 2}
If  $\deg P>1$, $v(\hat c-H)\subseteq v(\hat b-H)$, and $R\in Z(H,\hat c)$ has minimal complexity,  then the $Z$-minimal slot~$(R,\fm,\hat c)$ in $H$ has a special refinement~$(R_{+c},\fn,\hat c-c)$   such that $(R^\phi_{+c},\fn,{\hat c-c})$ is eventually deep and split-normal.
\end{cor}
\begin{proof}
Apply Corollary~\ref{cor:evsplitnormal, 1} to the minimal hole $(P_{\times \imag},\fm,-\imag\hat a)$ in~$H[\imag]$. 
\end{proof}

\noindent
In the next two corollaries we handle the case $\deg P=1$.
Recall from Lemma~\ref{lem:r=deg P=1} that then
$\order{Q}=1$ or $\order{Q}=2$, $\deg{Q}=1$. 
Theorem~\ref{mainthm} gives:

\begin{cor}\label{cor:r=deg P=1, 1}
Suppose $H$ is $1$-linearly newtonian and $\order{Q}=1$. 
Then the slot~$(Q,\fm,\hat b)$ in $H$  has a refinement~$(Q_{+b},\fn,\hat b-b)$   such that $(Q^\phi_{+b},\fn,{\hat b-b})$ is eventually deep and split-normal. 
\end{cor}

\begin{cor} \label{cor:r=deg P=1, 2}
Suppose $\deg P=1$ and $\order Q=2$, $\deg Q=1$.  
Let $\hat{Q}\in H\{Y\}$ be a minimal annihilator of $\hat b$ over~$H$.
Then $\big(\hat{Q},\fm,\hat b\big)$ is a $Z$-minimal hole in~$H$ and has a refinement~$\big({\hat{Q}_{+b},\fn,\hat b-b}\big)$  such that $\big(\hat{Q}^\phi_{+b},\fn,{\hat b-b}\big)$ is eventually deep and split-normal. 
\end{cor}
\begin{proof}
By the proof of Lemma~\ref{lem:r=deg P=1} we have $\cc(Q)=\cc(\hat{Q})$ 
(hence  $\big(\hat{Q},\fm,\hat b\big)$  is a $Z$-minimal hole in $H$)
and $L_{\hat{Q}}$ splits over $H[\imag]$. Corollary~\ref{cor:deepening} gives a refine\-ment~$\big({\hat{Q}_{+b},\fn,\hat b-b}\big)$ of $\big(\hat{Q},\fm,\hat b\big)$ whose linear part has Newton weight~$0$ and such that the slot $\big({\hat{Q}^\phi_{+b},\fn,\hat b-b}\big)$ in $H^\phi$ is deep, eventually.  Moreover, by Lemmas~\ref{lem:deg1 normal} and~\ref{lem:deg 1 cracks splitting}, $\big(\hat{Q}^\phi_{+b},\fn,{\hat b-b}\big)$ is normal and its linear part splits over $H^{\phi}[\imag]$, eventually. Thus $\big(\hat{Q}^\phi_{+b},\fn,{\hat b-b}\big)$ is eventually deep and split-normal. 
\end{proof}

\noindent
This concludes the proof of Theorem~\ref{thm:split-normal}.

\subsection*{Split-normality and refinements}
We now study the behavior of split-normality under refinements. {\it In this subsection $a$ ranges over $H$ and $\fm$, $\fn$, $\fv$ range over $H^\times$.
Let $(P,\fm,\hat a)$ be a slot in $H$ of order $r\geq 1$ with $\hat a\in\hat H\setminus H$, and $L:=L_{P_{\times\fm}}$, $w:=\wt(P)$.}\/
Here is the split-normal analogue of Lemma~\ref{lem:normal pos criterion}:

\begin{lemma}\label{lem:split-normal criterion}
Suppose $\order(L)=r$ and $\fv$ is such that \textup{(SN1)} and~\textup{(SN2)} hold, and $\fv(L)\asymp_{\Delta(\fv)} \fv$. Then $(P,\fm,\hat a)$ is split-normal.
\end{lemma}
\begin{proof} Same as that of~\ref{lem:normal pos criterion}, but with $R$ as in (SN2) instead of $(P_{\times \fm})_{>1}$.
\end{proof}

\noindent
Now split-normal analogues of Propositions~\ref{normalrefine} and~\ref{easymultnormal}:

\begin{lemma}\label{splitnormalrefine}
Suppose $(P,\fm,\hat a)$ is split-normal. Let a refinement $(P_{+a},\fm,{\hat a-a})$ of $(P,\fm,\hat a)$ be given. Then $(P_{+a},\fm,\hat a-a)$ is also split-normal.
\end{lemma}
\begin{proof}
As in the proof of Proposition~\ref{normalrefine} we arrange $\fm=1$ and show for $\fv:= \fv(L_P)$, using Lemmas~\ref{lem:linear part, new} and~\ref{splnormalnormal}, that $\order(L_{P_{+a}})=r$ and $$(P_{+a})_1\sim_{\Delta(\fv)} P_1, \quad \fv(L_{P_{+a}})\sim_{\Delta(\fv)}\fv, \quad (P_{+a})_{>1} \prec_{\Delta(\fv)} \fv^{w+1} (P_{+a})_1.$$
Now take $Q$, $R$ as in (SN2) for $\fm=1$. Then $P_1=Q+R_1$, and so by Lemma~\ref{lem:linear part, split-normal, new} for $A=L_Q$ we
obtain $(P_{+a})_1 - Q \prec_{\Delta(\fv)} \fv^{w+1} (P_{+a})_1$, and thus
$(P_{+a})_{\geq 1}-Q \prec_{\Delta(\fv)}\fv^{w+1} (P_{+a})_1$.
Hence (SN2) holds with $\fm=1$ and $P_{+a}$ instead of~$P$. Thus the slot~$(P_{+a},\fm,\hat a-a)$ in $H$ is split-normal by Lemma~\ref{lem:split-normal criterion}.
\end{proof}

{\sloppy

\begin{lemma}\label{easymultsplitnormal}
Suppose $(P,\fm,\hat a)$ is split-normal,  $\hat{a}\prec \fn\preceq \fm$, and $[\fn/\fm]\le[\fv]$, $\fv:=\fv(L)$.
Then the refinement $(P,\fn,\hat a)$ of $(P,\fm,\hat a)$
is split-normal: if~$\fm$,~$P$,~$Q$,~$\fv$ are as in \textup{(SN2)}, then \textup{(SN2)} holds with $\fn$, $Q_{\times \fn/\fm}$, $R_{\times \fn/\fm}$, $\fv(L_{P_{\times \fn}})$ in place of~$\fm$,~$Q$,~$R$,~$\fv$.
\end{lemma} }
\begin{proof} Set $\tilde{L}:= L_{P_{\times\fn}}$. Lemma~\ref{lem:steep1} gives $\order(\tilde{L})=r$ and $\fv(\tilde{L})\asymp_{\Delta(\fv)}\fv$. Thus~$(P_{\times\fn})_{>1} \prec_{\Delta(\fv)} \fv^{w+1} (P_{\times\fn})_1$ by Proposition~\ref{easymultnormal}. Now arrange $\fm=1$ in the usual way, and
take~$Q$,~$R$ as in (SN2) for $\fm=1$.  Then $$(P_{\times\fn})_1\ =\ (P_1)_{\times\fn}\ =\ Q_{\times\fn}+(R_1)_{\times\fn}, \qquad (P_{\times \fn})_{>1}\ =\ (R_{\times \fn})_{>1}\ =\ (R_{>1})_{\times \fn}$$ by~[ADH,~4.3], where
$Q_{\times\fn}$ is homogeneous of degree $1$ and order $r$, and $L_{Q_{\times\fn}}=L_Q\fn$ splits over~$K$.
Using [ADH, 4.3, 6.1.3] and $[\fn]\leq[\fv]$ we obtain
$$(R_1)_{\times\fn}\ \asymp_{\Delta(\fv)}\ \fn R_1\ \preceq\ \fn R\ \prec_{\Delta(\fv)}\ \fn\fv^{w+1}  P_1\ \asymp_{\Delta(\fv)} \fv^{w+1} (P_1)_{\times\fn}\ =\ \fv^{w+1} (P_{\times\fn})_1.$$
Hence (SN2) holds for $\fn, Q_{\times \fn}, R_{\times \fn}, \fv(\tilde{L})$ in place of $\fm, Q, R,\fv$. 
\end{proof}

\noindent 
Recall our standing assumption in this section that $H$ is a real closed $H$-field. Thus~$H$ is $\d$-valued, and
for all $\fn$ and $q\in \Q^{>}$ we have $\fn^q\in H^\times$ such that $(\fn^q)^\dagger=q\fn^\dagger$. {\em In the 
rest of this section we fix such an $\fn^q$ for all $\fn$ and $q\in\Q^{>}$.}\/ 
Now we upgrade Corollary~\ref{cor:normal for small q}  with ``split-normal'' instead of ``normal'':

\begin{lemma}\label{splnq}
Suppose $\fm=1$, $(P,1,\hat a)$ is split-normal, $\hat a \prec \fn\prec 1$, and for~$\fv:=\fv(L_P)$ we have $[\fn^\dagger]<[\fv]<[\fn]$. 
Then $(P,\fn^q,\hat a)$ is a split-normal refinement
of~$(P,1,\hat a)$ for all but finitely many $q\in \Q$ with $0 < q < 1$. 
\end{lemma}
\begin{proof}
Corollary~\ref{cor:normal for small q} gives that $(P,\fn^q,\hat a)$ is a normal refinement of $(P,1,\hat a)$ for all but finitely many $q\in \Q$ with $0<q<1$. 
Take~$Q$,~$R$  as in (SN2) for $\fm=1$. Then $L = L_Q + L_R$ where $L_Q$ splits over~$H[\imag]$ and $L_R \prec_{\Delta(\fv)} \fv^{w+1} L$,
for $\fv:=\fv(L)$.
Applying Corollary~\ref{cor:nepsilon} to $A:=L$, $B:=L_R$ we obtain: 
$L_R\fn^q \prec_{\Delta(\fw)} \fw^{w+1} L\fn^q$, $\fw:=\fv(L\fn^q)$, for all but finitely many $q\in \Q^>$. 

Let $q\in \Q$ be such that $0<q<1$, $(P,\fn^q,\hat a)$ is a normal refinement of $(P,1,\hat a)$, and $L_R\fn^q \prec_{\Delta(\fw)} \fw^{w+1} L\fn^q$, with $\fw$ as above. 
Then $(P_{\times\fn^q})_1 = Q_{\times\fn^q}+(R_1)_{\times\fn^q}$ where~$Q_{\times\fn^q}$ is homogeneous of degree~$1$ and order $r$,  $L_{Q_{\times\fn^q}}=L_Q\fn^q$ splits over~$H[\imag]$,
and $(R_1)_{\times\fn^q} \prec_{\Delta(\fw)} \fw^{w+1} (P_{\times\fn^q})_1$ for $\fw:=\fv(L_{P_{\times\fn^q}})$. Since  $(P,\fn^q, \hat a)$ is normal, we also have  
$(P_{\times\fn^q})_{>1}\prec_{\Delta(\fw)} \fw^{w+1} (P_{\times\fn^q})_1$. Thus $(P,\fn^q,\hat a)$ is split-normal. 
\end{proof}

\begin{remark} We do not know if in this last lemma we can drop the assumption $[\fn^\dagger]<[\fv]$. 
\end{remark}

\subsection*{Strengthening split-normality} 
{\em In this subsection $a, b$ range over $H$ and $\fm,\fn$ over $H^\times$,
and $(P,\fm,\hat a)$ is a slot in $H$ of order~$r\geq 1$ and weight $w:=\wt(P)$, so~${w\geq 1}$, and $L:=L_{P_{\times\fm}}$. If
$\order L=r$, we set $\fv:= \fv(L)$}. 

With an eye towards later use in connection with fixed point theorems over Hardy fields we strengthen here the concept of
split-normality; in the next subsection we  show how to improve Theorem~\ref{thm:split-normal} accordingly.
See the last subsection of Section~\ref{sec:approx linear diff ops} for the notion of strong splitting.

\begin{definition}\label{def:SSN}
{\samepage Call 
$(P,\fm,\hat a)$ {\bf almost strongly split-normal} \index{slot!almost strongly split-normal}\index{almost strongly!split-normal}\index{split-normal!almost strongly} if $\order L=r$, 
$\fv\prec^\flat 1$, and the following strengthening of (SN2) holds:
\begin{list}{*}{\addtolength\itemindent{-3.5em}\addtolength\leftmargin{0.5em}}
\item[(SN2as)]  $(P_{\times\fm})_{\geq 1}=Q+R$ where $Q, R\in H\{Y\}$, $Q$ is homogeneous of degree~$1$ and order~$r$,  $L_Q$ splits strongly over $K$, and $R\prec_{\Delta(\fv)} \fv^{w+1} (P_{\times\fm})_1$. 
\end{list}}\noindent
We say that $(P,\fm,\hat a)$ is {\bf strongly split-normal}
 if $\order L =r$, $\fv\prec^\flat 1$, and the following condition is satisfied:\index{slot!strongly split-normal}\index{strongly!split-normal}\index{split-normal!strongly}
\begin{enumerate}
\item[(SN2s)]  
$P_{\times\fm}=Q+R$ where $Q,R\in H\{Y\}$, $Q$ is homogeneous of degree~$1$ and order~$r$, 
$L_Q$   splits strongly over $K$, and $R\prec_{\Delta(\fv)} \fv^{w+1} (P_{\times\fm})_1$.
\end{enumerate}
\end{definition}

\noindent
To facilitate use of (SN2s) we observe:

\begin{lemma}\label{SN2suse} Suppose $(P,\fm,\hat a)$ is strongly split-normal and $P_{\times\fm}=Q+R$ as in {\rm(SN2s)}. Then $Q\sim (P_{\times \fm})_1$, 
$\fv_Q:= \fv(L_Q)\sim \fv$, so $R\prec_{\Delta(\fv)} \fv_Q^{w+1}Q$.
\end{lemma}
\begin{proof} We have $(P_{\times \fm})_1=Q+R_1$, so 
$Q=(P_{\times \fm})_1-R_1$ with $R_1\prec_{\Delta(\fv)} \fv^{w+1}(P_{\times \fm})_1$. Now apply Lemma~\ref{lem:fv of perturbed op} to $A:=L$ and $B:=-L_{R_1}$. 
\end{proof}

\noindent
If $(P,\fm,\hat a)$  is almost strongly split-normal, then
$(P,\fm,\hat a)$  is split-normal and hence normal by Lemma~\ref{splnormalnormal}. 
If $(P,\fm,\hat a)$ is normal and $L$ splits strongly over~$K$, then~$(P,\fm,\hat a)$ is  almost strongly split-normal;
in particular, if $(P,\fm,\hat a)$ is  normal of order $r=1$, then~$(P,\fm,\hat a)$ is almost strongly split-normal, by Lemma~\ref{lem:order 1 splits strongly}.
Moreover:

\begin{lemma}\label{lem:char strong split-norm}
The following are equivalent:
\begin{enumerate}
\item[\textup{(i)}] $(P,\fm,\hat a)$ is strongly split-normal;
\item[\textup{(ii)}] $(P,\fm,\hat a)$ is almost strongly split-normal and strictly normal;
\item[\textup{(iii)}] $(P,\fm,\hat a)$ is almost strongly split-normal and $P(0)\prec_{\Delta(\fv)} \fv^{w+1} (P_1)_{\times\fm}$.
\end{enumerate}
\end{lemma}
\begin{proof}
Suppose $(P,\fm,\hat a)$ is strongly split-normal, and
let $Q$, $R$ be as in (SN2s). Then~$(P_{\times\fm})_{\geq 1}=Q+R_{\geq 1}$, $L_Q$   splits strongly over $K$, and $R_{\geq 1} \prec_{\Delta(\fv)} \fv^{w+1} (P_{\times\fm})_1$.
Hence  $(P,\fm,\hat a)$  is almost strongly split-normal, and thus normal.
Also~$P(0)=R(0)\prec_{\Delta(\fv)} \fv^{w+1} (P_{\times\fm})_1$, so $(P,\fm,\hat a)$ is strictly normal.
This shows~(i)~$\Rightarrow$~(ii), and~(ii)~$\Rightarrow$~(iii) is clear.
For (iii)~$\Rightarrow$~(i) suppose~$(P,\fm,\hat a)$ is almost strongly split-normal and 
$P(0)\prec_{\Delta(\fv)} \fv^{w+1} (P_1)_{\times\fm}$. Take~$Q$,~$R$ as in~(SN2as). Then $P_{\times\fm} = Q+\tilde R$ where 
$\tilde R:=P(0)+R\prec_{\Delta(\fv)}\fv^{w+1} (P_1)_{\times\fm}$.
Thus $(P,\fm,\hat a)$ is strongly split-normal.
\end{proof}

\begin{cor}\label{cor:strongly splitting => strongly split-normal}
If $L$ splits strongly over $K$, then
$$\text{$(P,\fm,\hat a)$ is strongly split-normal }\ \Longleftrightarrow\   \text{$(P,\fm,\hat a)$ is strictly normal.}$$
\end{cor}

\noindent
The following diagram summarizes some implications between these variants of normality, for slots $(P,\fm,\hat a)$ in $H$
of order $r\ge 1$. (See also the diagram on p.~\pageref{section:flowchart}.)  
$$\xymatrix@L=6pt{	\text{strongly split-normal\ } \ar@{=>}[r] \ar@{=>}[d]&  \text{\ almost strongly split-normal} \ar@{=>}[r] &  \text{\ split-normal} \ar@{=>}[d] \\ 
\text{strictly normal\ } \ar@{=>}[rr] & & \text{\ normal}   }$$
If $(P,\fm,\hat a)$ is almost strongly split-normal, then  so are
$(bP,\fm,\hat a)$ for $b\neq 0$ and $(P_{\times\fn},\fm/\fn,\hat a/\fn)$, and likewise with ``strongly''
in place of ``almost strongly''.  

\medskip
\noindent
Here is a version of Lemma~\ref{splitnormalrefine} for (almost) strong split-normality:

\begin{lemma}\label{stronglysplitnormalrefine}
Suppose $(P_{+a},\fm,\hat a-a)$ refines $(P,\fm,\hat a)$.
If $(P,\fm,\hat a)$ is almost strongly split-normal, then so is $(P_{+a},\fm,\hat a-a)$.
If  $(P,\fm,\hat a)$ is   strongly split-normal, $Z$-minimal, and 
$\hat a - a \prec_{\Delta(\fv)} \fv^{r+w+1}\fm$, then $(P_{+a},\fm,\hat a-a)$ is strongly split-normal. 
\end{lemma}
\begin{proof}
The first part follows from Lemma~\ref{splitnormalrefine} and its proof.
In combination with Lemmas~\ref{stronglynormalrefine, cracks} and \ref{lem:char strong split-norm}, this also yields the second part.
\end{proof} 
 
\begin{lemma}\label{stronglysplitnormalrefine, q}
Suppose that $(P,\fm,\hat a)$ is split-normal and  $\hat{a}\prec_{\Delta(\fv)} \fm$.
Then for all sufficiently small  $q\in\Q^>$, any $\fn\asymp\fv^q\fm$ yields  an almost strongly split-normal refinement $(P,\fn,\hat a)$ of~$(P,\fm,\hat a)$.
\end{lemma}
\begin{proof}
We arrange $\fm=1$, so $\hat{a}\prec_{\Delta(\fv)} 1$. Take $Q$, $R$ as in~(SN2) with $\fm=1$, and take~$q_0\in\Q^>$ such that $\hat a\prec \fv^{q_0}\prec 1$. By Lemma~\ref{lem:split strongly multconj} we
can decrease~$q_0$ so that for all $q\in\Q$ with $0<q\leq q_0$ and any $\fn\asymp \fv^q$, $L_{Q_{\times\fn}}=L_Q\fn$ splits strongly over~$K$. 
Suppose $q\in \Q$, $0<q\leq q_0$, and $\fn\asymp\fv^q$.
Then  $(P,\fn,\hat a)$ is an almost strongly split-normal refinement of $(P, 1,\hat a)$, by Lemma~\ref{easymultsplitnormal}. 
\end{proof}

\begin{cor}\label{cor:deep and almost strongly split-normal}
Suppose that $(P,\fm,\hat a)$ is $Z$-minimal, deep, and split-normal. Then $(P,\fm,\hat a)$ has a refinement which is deep and
almost strongly split-normal.
\end{cor}
\begin{proof}  
Lemma~\ref{lem:good approx to hata} gives $a$ such that $\hat a - a \prec_{\Delta(\fv)} \fm$. By Corollary~\ref{cor:deep 2, cracks}, the re\-fine\-ment~$(P_{+a},\fm,\hat a-a)$ of
$(P,\fm,\hat a)$ is   deep  with $\fv(L_{P_{+a,\times \fm}})\asymp_{\Delta(\fv)} \fv$, and by Lemma~\ref{splitnormalrefine} it is also
split-normal. 
Now apply Lemma~\ref{stronglysplitnormalrefine, q} to $(P_{+a},\fm,\hat a-a)$ in place of~$(P,\fm,\hat a)$  and again use Corollary~\ref{cor:deep 2, cracks}.
\end{proof}

\noindent
We now turn to the behavior of these properties under compositional conjugation.

\begin{lemma}\label{lem:strongly split-normal compconj}
Let $\phi$ be active in $H$ with $0<\phi\preceq 1$.
If $(P,\fm,\hat a)$ is almost strongly split-normal, then so is 
the slot $(P^\phi,\fm,\hat a)$ in $H^\phi$. Likewise with ``strongly''
in place of ``almost strongly''.
\end{lemma}
\begin{proof}
We arrange $\fm=1$, 
assume $(P,\fm,\hat a)$ is almost strongly split-normal, and take~$Q$,~$R$ as in~(SN2as).
The proof of Lemma~\ref{lem:split-normal comp conj} shows that with $\fw:=\fv(L_{P^\phi})$
we have $\fw\prec^\flat_\phi 1$ and  $(P^\phi)_{\geq 1} = Q^\phi + R^\phi$ where
$Q^\phi\in H^\phi\{Y\}$ is homogeneous of degree~$1$ and order~$r$, $L_{Q^\phi}$ splits over~$H^\phi[\imag]$, and
$R^\phi \prec_{\Delta(\fw)} \fw^{w+1} (P^\phi)_1$. By Lemma~\ref{lem:split strongly compconj}, $L_{Q^\phi}=L_Q^\phi$ even  splits strongly over $H[\imag]$.
Hence~$(P^\phi,\fm,\hat a)$ is almost strongly split-normal.
The rest follows from Lemma~\ref{lem:char strong split-norm} and 
the fact that  if~$(P,\fm,\hat a)$ is strictly normal, then so is  $(P^\phi,\fm,\hat a)$.
\end{proof}

\noindent
If $H$ is $\upo$-free and $r$-linearly newtonian, then by Corollary~\ref{cor:achieve strong normality, 2}, every $Z$-minimal slot in~$H$ of order~$r$
has a refinement $(P,\fm,\hat a)$  such that the slot~$(P^\phi,\fm,\hat a)$ in $H^\phi$ is eventually deep and strictly normal.
Corollary~\ref{cor:5.30real} of the next lemma 
is a variant of this fact for strong split-normality.

\begin{lemma}\label{5.30real} 
Assume $H$ is $\upo$-free and $r$-linearly newtonian, and every~${A\in H[\der]}$ of order~$r$ splits over $K$.  
Suppose $(P,\fm,\hat a)$ is $Z$-minimal. Then there is a refinement $(P_{+a},\fn,\hat a-a)$  of $(P,\fm,\hat a)$ and an active $\phi$ in $H$
with $0<\phi\preceq 1$   such that~$(P^\phi_{+a},\fn,\hat a-a)$  is  deep and strictly normal, and its   linear part splits strongly over $K^\phi$ \textup{(}so $(P^\phi_{+a},\fn,\hat a-a)$  is strongly split-normal by Corollary~\ref{cor:strongly splitting => strongly split-normal}\textup{)}.
\end{lemma}
\begin{proof} 
For any active $\phi$ in $H$ with $0<\phi\preceq 1$ we may replace~$H$,~$(P,\fm,\hat a)$  by~$H^\phi$,~$(P^\phi,\fm,\hat a)$, respectively. We may
 also replace~$(P,\fm,\hat a)$ by any of its refinements. 
Now Theorem~\ref{mainthm} gives a refinement~$(P_{+a},\fn,\hat a-a)$ of $(P,\fm,\hat a)$ and an active~$\phi$ in $H$ such that~$0<\phi\preceq 1$ and $(P^\phi_{+a},\fn,\hat a-a)$ is deep and   normal.
Replacing~$H$,~$(P,\fm,\hat a)$ by~$H^\phi$,~$(P^\phi_{+a},\fn,{\hat a-a})$, respectively,  
we thus arrange that~$(P,\fm,\hat a)$ itself is deep and   normal. We show that then the lemma holds with~$\phi=1$. 
For this we first replace~$(P,\fm,\hat a)$  by a suitable refinement~$(P_{+a},\fm,{\hat a-a})$ to arrange by
Corollary~\ref{cor:achieve strong normality, 1}  that~$(P,\fm,\hat a)$ is strictly normal and $\hat a \prec_{\Delta(\fv)} \fm$. 
Now $L$ splits over~$K$, so by 
Corollary~\ref{cor:split strongly multconj},
for sufficiently small~${q\in\Q^>}$, any $\fn\asymp|\fv|^q\fm$ gives a  refinement~$(P,\fn,\hat a)$  of~$(P,\fm,\hat a)$ whose linear part~$L_{P_{\times\fn}}$ has order $r$ and splits strongly over $K$.
For each such $\fn$,  $(P,\fn,\hat a)$ is deep  by Corollary~\ref{cor:deep 2, cracks}, and for some such $\fn$,
 $(P,\fn,\hat a)$ is also strictly normal, by Remark~\ref{rem:strongly normal refine, 2}.
\end{proof}

\noindent
The previous lemma in combination with Lemma~\ref{lem:strongly split-normal compconj}  yields:
 
\begin{cor}\label{cor:5.30real}
With the same assumptions on $H$, $K$ as in Lemma~\ref{5.30real},  every $Z$-minimal slot in $H$ of order $r$ has a refinement $(P,\fm,\hat a)$ such that
$(P^\phi,\fm,\hat a)$ is eventually deep and strongly split-normal. 
\end{cor}

\noindent
For $r=1$ the splitting assumption is automatically satisfied (and this is the case most relevant later). 
We do not know whether ``every
$A\in H[\der]^{\ne}$ of order~$\le r$ splits over $K$'' is strictly weaker than ``$K$ is $r$-linearly closed''.

\subsection*{Achieving strong split-normality}
We make the same assumptions as in the subsection {\it Achieving split-normality}\/: {\em $H$ is $\upo$-free and $(P,\fm,\hat a)$ is a minimal hole in $K=H[\imag]$ of order $r\ge 1$, with  $\fm\in H^\times$ and $\hat a\in\hat K\setminus K$.}\/ Recall: $K$ is also
$\upo$-free by [ADH, 11.7.23], and if $H$ is $1$-linearly newtonian, then $H$ is $r$-linearly newtonian.  
We have  $$\hat a\ =\ \hat b + \hat c\, \imag,\qquad \hat b, \hat c\in \hat H.$$
We let  $a$ range over $K$, $b$, $c$  over $H$, and $\fn$ over~$H^\times$. 
In connection with the next two lemmas we note that given 
an active $\phi$ in~$H$ with~${0<\phi\preceq 1}$, 
if 
$(P,\fm,\hat a)$ is normal (strictly normal, respectively), then
so is~$(P^\phi,\fm,\hat a)$, by Lemma~\ref{lem:normality comp conj} (Lemma~\ref{lem:normality comp conj, strong}, respectively);
moreover, if  the linear part of $(P,\fm,\hat a)$ splits strongly over $K$, then the linear part of $(P^\phi,\fm,\hat a)$ splits strongly over $K^\phi=H^\phi[\imag]$,
by Lem\-ma~\ref{lem:split strongly compconj}. Here is a ``complex'' version of Lemma~\ref{5.30real}, with a similar proof:

\begin{lemma}\label{lem:achieve strong splitting}  
For some refinement $(P_{+a},\fn,\hat a-a)$  of $(P,\fm,\hat a)$ and active $\phi$ in~$H$ with $0<\phi\preceq 1$,
the hole $(P^\phi_{+a},\fn,\hat a-a)$ in $K^\phi$ is deep and  normal, its  linear part splits 
strongly over $K^\phi$, and it is moreover strictly normal if $\deg P>1$. 
\end{lemma}
\begin{proof} For any active $\phi$ in $H$ with $0<\phi\preceq 1$ 
we may replace 
 $H$ and $(P,\fm,\hat a)$  by~$H^\phi$ and the minimal hole $(P^\phi,\fm,\hat a)$ in $K^\phi$. We may
 also replace~$(P,\fm,\hat a)$   by any of its refinements $(P_{+a},\fn,\hat a-a)$. 
As noted before Theorem~\ref{thm:split-normal}, Corollary~\ref{cor:mainthm} and Lemma~\ref{ufm} give a refinement 
$(P_{+a},\fn,\hat a-a)$ of $(P,\fm,\hat a)$ and an active~${\phi}$ in $H$
with $0<\phi\preceq 1$ such that   $(P^\phi_{+a},\fn,\hat a-a)$ is deep and normal.
Replacing $H$, $(P,\fm,\hat a)$ by~$H^\phi$,~$(P^\phi_{+a},\fn,{\hat a-a})$, respectively,  
we thus arrange that~$(P,\fm,\hat a)$ itself is deep and normal. We show that then the lemma holds with $\phi=1$. 

Set $L:=L_{P_{\times\fm}}$ and $\fv:=\fv(L)$.
Lemma~\ref{lem:good approx to hata} gives $a$ with~${\hat a - a \prec_{\Delta(\fv)} \fm}$. 
If~${\deg P>1}$, then~$K$ is $r$-linearly newtonian and we use Corollary~\ref{cor:closer to minimal holes} to take~$a$ such that  
even $\hat a - a \preceq \fv^{w+2}\fm$.
Replacing~$(P,\fm,\hat a)$  by $(P_{+a},\fm,\hat a-a)$, we thus arrange by Lemma~\ref{lem:deep 2} and 
Proposition~\ref{normalrefine} that $\hat a \prec_{\Delta(\fv)} \fm$, and also by Lemma~\ref{lem:achieve strong normality} that~$(P,\fm,\hat a)$  is  strictly normal if $\deg P>1$. 
Now $L$ splits over~$K$, since~$K$ is $r$-linearly closed by Corollary~\ref{corminholenewt}.
Then by Corollary~\ref{cor:split strongly multconj}, for sufficiently small~${q\in\Q^>}$, any $\fn\asymp|\fv|^q\fm$ gives a refinement $(P,\fn,\hat a)$   of~$(P,\fm,\hat a)$ whose linear part~$L_{P_{\times\fn}}$ splits strongly over $K$. 
For such $\fn$,  $(P,\fn,\hat a)$ is deep by Lemma~\ref{lem:deep 2} and normal by Proposition~\ref{easymultnormal}. 
If~$(P,\fm,\hat a)$ is strictly normal, then for some such~$\fn$, 
 $(P,\fn,\hat a)$ is also strictly normal, thanks to Lemma~\ref{lem:strongly normal refine, 2}.  
\end{proof}

\noindent
We now remove the $\deg P>1$ condition in Lemma~\ref{lem:achieve strong splitting}: 

\begin{lemma}\label{lem:achieve strong splitting, d=1}
For some refinement~$(P_{+a},\fn,\hat a-a)$  of $(P,\fm,\hat a)$ and active $\phi$ in~$H$ with~$0<\phi\preceq 1$,
the hole~$(P^\phi_{+a},\fn,\hat a-a)$ in $K^\phi$ is deep and strictly normal, and its  linear part splits 
strongly over $K^\phi$. 
\end{lemma}
\begin{proof}  
Thanks to Lemma~\ref{lem:achieve strong splitting} we need only consider the case $\deg P=1$. Then we have~$r=1$ by Corollary~\ref{cor:minhole deg 1}. 
%(See now the remark following this proof.) 
As in the proof of Lem\-ma~\ref{lem:achieve strong splitting} we may replace~$H$ and $(P,\fm,\hat a)$ for any active $\phi\preceq 1$ in $H^{>}$ by $H^\phi$ and~$(P^\phi, \fm,\hat a)$, and
 also~$(P,\fm,\hat a)$   by any of its refinements $(P_{+a},\fn,\hat a-a)$. Recall here that  $\fn\in H^\times$.
 Hence using a remark preceding Lemma~\ref{lem:strongnormal pos criterion} and using also Corollary~\ref{cor:strongly normal d=1}   
 we   arrange that~$(P,\fm,\hat a)$ is strictly normal, and thus balanced and deep. We  show that then the lemma holds with $\phi=1$. 

 Set $L:=L_{P_{\times\fm}}$, $\fv:=\fv(L)$.
 Lemma~\ref{lem:balanced good approx} yields $a$ with~$\hat a-a \preceq \fv^4\fm$.
Replacing~$(P,\fm,\hat a)$  by $(P_{+a},\fm,\hat a-a)$ arranges that $\hat a \prec_{\Delta(\fv)} \fm$, by Lemmas~\ref{lem:deep 2} and~\ref{stronglynormalrefine}. As in the proof of Lemma~\ref{lem:achieve strong splitting}, for sufficiently small~${q\in\Q^>}$, any~$\fn\asymp|\fv|^q\fm$ now gives a strictly normal and deep refinement $(P,\fn,\hat a)$   of~$(P,\fm,\hat a)$ whose linear part splits strongly over $K$.
\end{proof}

\begin{remark}
Suppose we replace our standing assumption that $H$ is $\upo$-free and $(P, \fm, \hat a)$ is a minimal hole in $K$ by the assumption that $H$ is $\upl$-free and $(P, \fm, \hat a)$ is a slot in $K$ of order and degree $1$ (so $K$ is $\upl$-free by [ADH, 11.6.8] and $(P, \fm,\hat a)$ is $Z$-minimal). 
Then Lemma~\ref{lem:achieve strong splitting, d=1} goes through with ``hole" replaced by ``slot''. Its proof also goes through with the references to Lemmas~\ref{lem:deep 2} and~\ref{stronglynormalrefine} replaced by references to Corollary~\ref{cor:deep 2, cracks} and Lemma~\ref{stronglynormalrefine, cracks}. The end of that proof refers to the end of the proof of  Lemma~\ref{lem:achieve strong splitting}, and there one should replace Proposition~\ref{easymultnormal}  by Corollary~\ref{corcorcor}, and  Lemma~\ref{lem:strongly normal refine, 2} by  Remark~\ref{rem:strongly normal refine, 2}. 
\end{remark}

\noindent
In the remainder of this subsection we prove the following variant of Theorem~\ref{thm:split-normal}:

\begin{theorem}\label{thm:strongly split-normal} If $H$ is $1$-linearly newtonian, then one of the following holds: 
\begin{list}{}{\leftmargin=2em, \labelwidth=2em}
\item[\textup{(i)}] $\hat b\notin H$ and some $Z$-minimal slot $(Q,\fm,\hat b)$ in $H$ has a refinement~${(Q_{+b},\fn,\hat b-b)}$ such that $(Q^\phi_{+b},\fn,\hat b-b)$ is eventually deep and almost strongly split-normal; 
\item[\textup{(ii)}] $\hat c\notin H$ and some $Z$-minimal slot $(R,\fm,\hat c)$ in $H$ has a refinement~${(R_{+c},\fn,\hat c-c)}$ such that $(R^\phi_{+c},\fn,\hat c-c)$ is eventually deep and almost strongly split-normal. 
\end{list}
Moreover, if $H$ is $1$-linearly newtonian and either  $\deg P>1$, or $\hat b\notin H$ and   $Z(H,\hat b)$ contains an element of order~$1$, or $\hat c\notin H$ and   $Z(H,\hat c)$ contains an element of order~$1$,
then \textup{(i)} holds with ``almost'' omitted, or  \textup{(ii)} holds with ``almost'' omitted. 
\end{theorem}

\noindent
Towards the proof of this theorem we first show:

\begin{lemma}\label{lem:refine to almost strongly split-normal, Q}
Suppose $\hat b\notin H$ and $(Q,\fm,\hat b)$ is  a $Z$-minimal slot in $H$ with a refinement~${(Q_{+b},\fn,\hat b-b)}$  such that $(Q^\phi_{+b},\fn,\hat b-b)$ is eventually deep and split-normal.
Then $(Q,\fm,\hat b)$ has a refinement~${(Q_{+b},\fn,\hat b-b)}$  such that $(Q^\phi_{+b},\fn,\hat b-b)$ is eventually deep and almost strongly split-normal.
\end{lemma}
\begin{proof}
Let ${(Q_{+b},\fn,\hat b-b)}$ be a refinement of $(Q,\fm,\hat b)$  and
let $\phi_0$ be active in $H$ such that $0<\phi_0\preceq 1$
and $(Q^{\phi_0}_{+b},\fn,\hat b-b)$ is   deep and split-normal.
Then  Corollary~\ref{cor:deep and almost strongly split-normal}
yields a refinement $\big((Q^{\phi_0}_{+b})_{+b_0},\fn_0,(\hat b-b)-b_0\big)$ of $(Q^{\phi_0}_{+b},\fn,\hat b-b)$ which is deep and
almost strongly split-normal. Hence 
$$\big((Q_{+b})_{+b_0},\fn_0,(\hat b-b)-b_0\big)\ =\ \big( Q_{+(b+b_0)},\fn_0,\hat b - (b+b_0) \big)$$
is  a refinement of $(Q,\fm,\hat b)$, and $\big( Q^\phi_{+(b+b_0)},\fn_0,\hat b - (b+b_0) \big)$ is eventually deep and almost strongly split-normal by Lemma~\ref{lem:strongly split-normal compconj}.
\end{proof}

\noindent
Likewise:

\begin{lemma}\label{lem:refine to almost strongly split-normal, R}
Suppose $\hat c\notin H$, and $(R,\fm,\hat c)$ is  a $Z$-minimal slot in $H$ with   a refinement~${(R_{+c},\fn,\hat c-c)}$  such that $(R^\phi_{+c},\fn,\hat c-c)$ is eventually deep and split-normal.
Then $(R,\fm,\hat c)$  has a refinement~${(R_{+c},\fn,\hat c-c)}$  such that $(R^\phi_{+c},\fn,\hat c-c)$ is eventually deep and almost strongly split-normal.
\end{lemma}

\noindent
Theorem~\ref{thm:split-normal} and the two lemmas above give the first part of Theorem~\ref{thm:strongly split-normal}.  
We break up the proof of the ``moreover'' part  into several cases, along the lines of the proof of Theorem~\ref{thm:split-normal}. We begin with the case where $\hat b\in H$ or~$\hat c\in H$.

\begin{lemma}\label{lem:refine to almost strongly split-normal, order(Q)=r}
Suppose $H$ is $1$-linearly newtonian, $\hat b\notin H$, $(Q,\fm,\hat b)$ is a $Z$-minimal slot in $H$ of order $r$, and some 
refinement~${(Q_{+b},\fn,\hat b-b)}$ of $(Q,\fm, \hat b)$ is such that~$(Q^\phi_{+b},\fn,\hat b-b)$ is eventually deep and split-normal.
Then $(Q,\fm,\hat b)$ has a refinement~${(Q_{+b},\fn,\hat b-b)}$  with $(Q^\phi_{+b},\fn,\hat b-b)$ eventually deep and strongly split-normal.
\end{lemma}
\begin{proof}
Lemma~\ref{lem:refine to almost strongly split-normal, Q} gives a refinement 
$(Q_{+b},\fn,\hat b-b)$ of $(Q,\fm,\hat b)$ with $(Q^\phi_{+b},\fn,\hat b-b)$ eventually deep and almost strongly split-normal.
We  upgrade this to ``strongly split-normal'' as follows: Take active $\phi_0$ in $H$ with $0<\phi_0\preceq 1$
such that the slot~$(Q^{\phi_0}_{+b},\fn,\hat b-b)$ in $H^{\phi_0}$ is  deep and almost strongly split-normal.  
Now $H$ is $1$-linearly newtonian, hence $r$-linearly newtonian. Therefore
Corollary~\ref{cor:achieve strong normality, 1} yields a deep and strictly normal refinement
$\big( (Q^{\phi_0}_{+b})_{+b_0},\fn, (\hat b - b)-b_0 \big)$
of $\big( Q^{\phi_0}_{+b},\fn,\hat b - b \big)$. 
By
Lemma~\ref{stronglysplitnormalrefine}, this refinement is still almost  strongly split-normal,   thus strongly split-normal by
Lemma~\ref{lem:char strong split-norm}. 
Then by  Lemma~\ref{lem:strongly split-normal compconj},  
$\big( Q_{+(b+b_0)},\fn,{\hat b - (b+b_0)} \big)$
is  a refinement of $(Q,\fm,\hat b)$ such that $\big( Q^\phi_{+(b+b_0)},\fn,\hat b - (b+b_0) \big)$ is eventually deep and strongly split-normal.
\end{proof}

\noindent
Lemmas~\ref{lem:hat c in K} and~\ref{lem:refine to almost strongly split-normal, order(Q)=r} give the following:

\begin{cor}\label{cor:hat c in K, strong}
Suppose $H$ is $1$-linearly newtonian and $\hat c\in H$. Then there is a hole $(Q,\fm,\hat b)$ in $H$ of the same complexity
as $(P,\fm,\hat a)$. Every such hole $(Q,\fm,\hat b)$ in $H$ is minimal and has a refinement~$(Q_{+b},\fn,\hat b-b)$ such that $(Q^\phi_{+b},\fn,\hat b-b)$ is eventually deep and strongly split-normal.
\end{cor}

\noindent
Just as Lemma~\ref{lem:hat c in K} gave rise to Lemma~\ref{lem:hat b in K},  Corollary~\ref{cor:hat c in K, strong} leads to:

\begin{cor}\label{cor:hat b in K, strong}
Suppose $H$ is $1$-linearly newtonian and $\hat b\in H$. Then there is a hole~$(R,\fm,\hat c)$ in $H$ of the same complexity
as $(P,\fm,\hat a)$. Every such hole in $H$ is minimal and   has a refinement~$(R_{+c},\fn,\hat c-c)$ such that $(R^\phi_{+c},\fn,\hat c-c)$ is eventually deep and strongly split-normal. 
\end{cor}

\noindent
In the following two lemmas we assume that $\hat b,\hat c\notin H$. Let $Q\in Z(H,\hat b)$ be of minimal complexity, so $(Q,\fm,\hat b)$ is a $Z$-minimal slot in $H$, as is each of its refinements. 
The next lemma strengthens Corollary~\ref{cor:evsplitnormal, 1}:

{\sloppy
\begin{lemma}\label{lem:evstronlysplitnormal, 1}
Suppose $\deg P>1$ and $v(\hat b-H)\subseteq v(\hat c-H)$. Then
$(Q,\fm,\hat b)$ has a refinement $(Q_{+b},\fn,\hat b-b)$ such that $(Q^\phi_{+b},\fn,\hat b-b)$ is eventually deep and strongly split-normal.
\end{lemma}
\begin{proof}
Corollary~\ref{cor:evsplitnormal, 1} and Lemma~\ref{lem:refine to almost strongly split-normal, Q} give a refinement 
$(Q_{+b},\fn,\hat b-b)$ of~$(Q,\fm,\hat b)$ and an active $\phi_0$ in $H$ with $0<\phi_0\preceq 1$
such that the slot $(Q^{\phi_0}_{+b},\fn,{\hat b-b})$ in $H^{\phi_0}$ is  deep and almost strongly split-normal.  From $\deg P >1$ we obtain that~$H$ is $r$-linearly newtonian. Now argue as in the proof of Lemma~\ref{lem:refine to almost strongly split-normal, order(Q)=r}.
\end{proof}}

\noindent
Similarly we obtain a strengthening of Corollary~\ref{cor:evsplitnormal, 2}, using that corollary and Lemma~\ref{lem:refine to almost strongly split-normal, R} in place of Corollary~\ref{cor:evsplitnormal, 1} and Lemma~\ref{lem:refine to almost strongly split-normal, Q} in the proof:

\begin{lemma}\label{lem:evstronlysplitnormal, 2}
If $\deg P >1$, $v(\hat c-H)\subseteq v(\hat b-H)$, and $R\in Z(H,\hat c)$ has minimal complexity, then the $Z$-minimal slot~$(R,\fm,\hat c)$ in $H$ has a refinement~$(R_{+c},\fn,\hat c-c)$   such that $(R^\phi_{+c},\fn,{\hat c-c})$ is eventually deep and strongly split-normal. 
\end{lemma}

\medskip\noindent
{\em Proof  of Theorem~\ref{thm:strongly split-normal}}. As indicated after the statement of Lemma~\ref{lem:refine to almost strongly split-normal, R},
we only need to complete the proof of the ``moreover" part.
%We now prove the ``moreover'' part of Theorem~\ref{thm:strongly split-normal}.
Thus, suppose $H$ is $1$-linearly newtonian.
If~$\hat b\in H$,   then $\hat c\notin H$ and
Corollary~\ref{cor:hat b in K, strong} yields a strong version of~(ii)
with ``almost'' omitted.  Likewise, if $\hat c\in H$, then $\hat b\notin H$ and Corollary~\ref{cor:hat c in K, strong} yields a strong version of~(i),
with ``almost'' omitted.
In the rest of the proof we assume~${\hat b,\hat c\notin H}$.
By Lemma~\ref{lem:same width} we have $v({\hat b-H})\subseteq v({\hat c-H})$ or  $v(\hat c-H)\subseteq v({\hat b-H})$,
and thus Lemmas~\ref{lem:evstronlysplitnormal, 1} and \ref{lem:evstronlysplitnormal, 2} take care of the case~${\deg P >1}$. 
If~$Z(H,\hat b)$ contains an element of order~$1$, and $Q\in Z(H,\hat b)$ has minimal complexity, then $\order Q =1$ by Lemma~\ref{kb}, so Corollary~\ref{cor:5.30real} and the remark following it
yield (i) with ``almost'' omitted. Likewise, if $Z(H,\hat c)$ contains an element of order~$1$, then~(ii) holds with ``almost'' omitted.
\qed

\subsection*{Revisiting newtonianity} We now use our results about isolated holes and split-normality  to obtain
with Corollary~\ref{corfirstnewtchar} a sharper first-order characterization of newtonianity than provided by our definition of this notion in [ADH]. 

{\it Let $H$ be a real closed $H$-field with small derivation and asymptotic integration}.
Let $P\in H\{Y\}^{\ne}$ have order $r\geqslant 1$ and weight $w$. Just for the next corollary, call $P$
{\bf strongly split-normal} if the following conditions are satisfied:\index{differential polynomial!strongly split-normal}\index{strongly!split-normal}\index{split-normal!strongly}
\begin{enumerate}
\item  $L_P$ has order $r$ and $\fv:=\fv(L_{P})\prec^\flat 1$; and
\item 	$P=Q+R$ where   $Q\in H\{Y\}$  is homogeneous of degree $1$, $\order Q=r$,
$L_Q$ splits strongly over $K$, and $R\prec_{\Delta(\fv)} \fv^{w+1}P_1$.
\end{enumerate}
Call $P$ {\bf eventually deep and strongly split-normal}\/ if $\ndeg S_P=0$ and for all small enough active $\phi>0$ in $H$,
the differential polynomial $P^\phi\in H^\phi\{Y\}$ is strongly split-normal with respect to $H^\phi$.
Note: $\ndeg P=\nval P=1$ for such~$P$.
%A slot~$(P,\fm,\hat a)$  in~$H$ is strongly split-normal as in Definition~\ref{def:SSN} iff $P_{\times \fm}$ is strongly split-normal. Also note that $\fv\prec^\flat 1\Leftrightarrow \fv\prec 1\ \&\ \fv^\dagger\succeq 1$ and~$R\prec_{\Delta(\fv)} \fv^{w+1}P_1 \Leftrightarrow R\prec  \abs{\fv}^{w+1+1/n}P_1$ for some $n\geq 1$, so  (SN1s$'$), (SN2s$'$) are indeed expressible as (disjunctions of)
%first-order properties of the coefficients of $P$.\marginpar{added two comments}

{\samepage
\begin{cor}\label{corfirstnewtchar} 
Assume $H$ is $\upo$-free. Then the following are equivalent: \begin{enumerate}
\item[\rm(i)] $H$ is newtonian;
\item[\rm(ii)]  $H[\imag]$   is $1$-linearly newtonian and
    every eventually deep and strongly split-normal $P$ in $H\{Y\}^{\ne}$ of order $\geqslant 1$ has
  a zero~$y\prec 1$ in $H$. 
  \end{enumerate} 
\end{cor}}
\begin{proof}
The  direction (i)~$\Rightarrow$~(ii) is clear from \eqref{eq:14.5.7} and  [ADH, 14.2.11]. For (ii)~$\Rightarrow$~(i), 
suppose  $H$ is not newtonian and  $H[\imag]$   is $1$-linearly newtonian. 
 By Proposition~\ref{prop:2.12 isolated} it is enough to show that then $H$ has an isolated hole $(Q,\fm, \hat b)$ such that for all small enough active $\phi>0$ in $H$ the hole $(Q^\phi, \fm, \hat b)$ in $H^\phi$ is deep and
strongly split-normal. We set $K:= H[\imag]$ and let  $b$, $c$ range over $H$.  

Lemma~\ref{lem:no hole of order <=r} and subsequent remarks
give a minimal hole~$(P,\fm,\hat a)$ in $K$ of order~$r\geqslant 1$, where~$\fm\in H^\times$. Then~$\deg P>1$ by Corollary~\ref{cor:minhole deg 1}.  
By Lemma~\ref{lem:hole in hat K} we arrange that~${\hat a\in \hat K:= \hat H[\imag]}$ where~$\hat H$ is an immediate $\upo$-free newtonian $H$-field extension of~$H$, so 
$\hat a = \hat b + \hat c\imag$ with $\hat b, \hat c \in \hat H$.  
Then $v(\hat b-H)\subseteq v(\hat c-H)$ or~$v(\hat c-H)\subseteq v(\hat b-H)$ by Lemma~\ref{lem:same width}; we assume
$v(\hat b-H)\subseteq v(\hat c-H)$. (The other case is similar.)
The equivalence (i)~$\Leftrightarrow$~(iii) of that lemma then gives $\hat b\notin H$. 
 
Take $Q\in Z(H,\hat b)$ of minimal complexity. Then $(Q, \fm, \hat b)$ is a $Z$-minimal slot
in $H$, of positive order by Lemma~\ref{kb}. 
Given any refinement
  $(Q_{+b}, \fn, \hat b -b)$   of~$(Q,\fm, \hat b)$, Lemma~\ref{lem:same width} gives $c$ with $v(\hat a -a)=v(\hat b-b)$ for $a:=b+c\imag$, and we may then replace~$(P,\fm,\hat a)$ and $(Q,\fm,\hat b)$ by   $(P_{+a}, \fn, \hat a -a)$ and $(Q_{+b}, \fn, \hat b -b)$, respectively,
whenever convenient.
Likewise, for  any active $\phi$ in $H$ with $0<\phi\preccurlyeq 1$, we can also replace $H$, $K$, $(P, \fm, \hat a)$, $(Q, \fm, \hat b)$ by $H^\phi$, $K^\phi$, $(P^\phi, \fm, \hat a)$, $(Q^\phi, \fm, \hat b)$.

Suppose now that $\hat c\notin H$. 
Use Corollary~\ref{cor:evsplitnormal, 1} to arrange  that~$(Q,\fm,\hat b)$  is  normal.
Next, use Proposition~\ref{prop:achieve isolated} to arrange that  $(Q,\fm,\hat b)$ is  isolated, but possibly no longer normal.
Being isolated persists under refinement, so we can
use Lemma~\ref{lem:evstronlysplitnormal, 1}  to arrange that $(Q^\phi,\fm,\hat b)$ is eventually  deep and strongly split-normal. With Lemma~\ref{lem:from cracks to holes}, changing $\hat b$ if necessary, we arrange that
$(Q,\fm, \hat b)$ is an isolated hole in $H$, not just an isolated slot in $H$, thus achieving our goal.

Finally, suppose that $\hat c\in H$. Then  use Corollary~\ref{cor:hat c in K, strong} and Proposition~\ref{prop:achieve isolated} to choose 
$Q$ such that $(Q,\fm, \hat b)$ is a minimal  and isolated hole in $H$ with the property that
$(Q^\phi,\fm,\hat b)$ is eventually  deep and strongly split-normal. 
\end{proof}

\section{Ultimate Slots}\label{sec:ultimate}

\noindent 
{\em In this section $H$ is a Liouville closed $H$-field
with small derivation, $\hat H$ is an immediate asymptotic extension 
of~$H$, and $\imag$ is an element of an asymptotic extension of $\hat H$
with $\imag^2=-1$.} 
Then~$\hat H$ is an $H$-field, $\imag\notin\hat H$,
$K:=H[\imag]$  is an algebraic closure of $H$, and $\hat K:=\hat H[\imag]$ is an immediate $\d$-valued extension of $K$. (See the beginning of Section~\ref{sec:split-normal holes}.) Let $C$ be the constant field of $H$, let $\mathcal{O}$ denote the valuation ring of $H$ and $\Gamma$ its value group. Accordingly, the constant field of $K$ is $C_K=C[\imag]$ and the valuation ring of~$K$ is~$\mathcal{O}_K=\mathcal{O}+\mathcal{O}\imag$.  
Let $\fm$, $\fn$, $\fw$ range over $H^\times$ and $\phi$ over the  elements of $H^>$ which are active in~$H$ (and hence in $K$).

\medskip
\noindent
In Section~\ref{sec:logder} we introduced
$$W\ :=\ \big\{\!\wr(a,b):\ a,b\in H,\ a^2+b^2=1\big\}.$$
Note that $W$ is a subspace of the 
$\Q$-linear space $H$, because $W\imag=S^\dagger$ where 
$$S\ :=\ \{a+b\imag:\ a,b\in H,\ a^2+b^2=1\}$$ is a divisible subgroup of $K^\times$. We have $K^\dagger=H +W\imag$ by Lemma~\ref{lem:logder}. 
Thus there exists a complement $\Lambda$ of the subspace~$K^\dagger$ of $K$ such that $\Lambda\subseteq H\imag$, and in this section we fix such $\Lambda$ and let~$\lambda$ range over $\Lambda$. 
 Let $\Univ=K\big[\! \ex(\Lambda) \big]$ be the universal exponential extension of~$K$ defined in Section~\ref{sec:univ exp ext}.

\medskip
\noindent
For $A\in K[\der]^{\ne}$ we have its set~$\exc^{\operatorname{u}}(A)\subseteq \Gamma$ of ultimate exceptional values, 
which a priori might depend on our choice of~$\Lambda$. We now make good on a promise from Section~\ref{sec:valuniv} by showing under the mild assumption $\I(K)\subseteq K^\dagger$ and with our restriction~$\Lambda\subseteq H\imag$ there is no such dependence:

\begin{cor}\label{prop:excu independent of Q} Suppose $\I(K)\subseteq K^\dagger$. Then for $A\in K[\der]^{\neq}$,
% the status of $A$ being terminal does not depend on the choice of $\Lambda$, and  
the set $\exc^{\operatorname{u}}(A)$ of ultimate exceptional values of $A$    does not depend on this choice.
\end{cor} 
\begin{proof}
Let $\Lambda^*\subseteq H\imag$ also be a complement of $K^\dagger$. Let   $\lambda\mapsto \lambda^*$ be the $\Q$-linear bijection $\Lambda\to \Lambda^*$ with $\lambda- \lambda^*\in W\imag$ for all $\lambda$. Then by Lemmas~\ref{pldv} and~\ref{lem:W and I(F)},
$$\lambda- \lambda^*\in \I(H)\imag\ \subseteq\ \I(K)\ \subseteq\  (\mathcal O_K^\times)^\dagger$$  for all $\lambda$. Now use  Lemma~\ref{lem:excu for different Q} and Corollary~\ref{cor:excu for different Q}.
\end{proof}

\begin{cor}\label{cor:excu independent of Q} 
Suppose $\I(K)\subseteq K^\dagger$. Let $A=\der-g\in K[\der]$ where $g\in K$ and let~$\mathfrak g\in H^\times$ be such that $\mathfrak g^\dagger=\Re g$. Then
$$\exc^{\operatorname{u}}(A)\ =\ v_{\g}(\ker_{\Univ}^{\neq} A)\ =\ \{v\mathfrak g\}.$$
In particular, if $\Re g\in \I(H)$, then~$\exc^{\operatorname{u}}(A)=\{0\}$.
\end{cor}
\begin{proof}
Let $f\in K^\times$ and $\lambda$ be such that $g=f^\dagger+\lambda$. Then 
$$\exc^{\operatorname{u}}(A)\ =\ v_{\g}(\ker_{\Univ}^{\neq} A)\ =\ \{vf\}$$ by  Lemma~\ref{lem:excu, r=1} and its proof.
Recall that $K^\dagger = H+\I(H)\imag$ by Lemma~\ref{lem:W and I(F)} and remarks preceding it, so $g\in K^\dagger$ iff $\Im g\in\I(H)$.
Consider first the case $g\notin K^\dagger$. Then by Corollary~\ref{prop:excu independent of Q} we can change
$\Lambda$ if necessary to arrange $\lambda:=(\Im g)\imag\in\Lambda$ so that we can take $f:=\mathfrak g$ in the above.
Now suppose~$g\in K^\dagger$. Then $g=(\mathfrak g h)^\dagger$ where $h\in K^\times$, $h^\dagger=(\Im g)\imag$.  Then 
we can take $f:=\mathfrak g h$, $\lambda:=0$, and we have $h\asymp 1$ since~$h^\dagger\in \I(H)\imag\subseteq\I(K)$.
\end{proof}

\begin{cor}\label{cor:LambdaL, purely imag}
Suppose  $\I(K)\subseteq K^\dagger$, and let $F$ be a Liouville closed $H$-field extension of $H$, and $L:=F[\imag]$. 
Then the subspace~$L^\dagger$ of the  $\Q$-linear space $L$ has a
complement~$\Lambda_L$ with $\Lambda\subseteq\Lambda_L\subseteq F\imag$. For any such $\Lambda_L$ and 
  $A\in K[\der]^{\neq}$ we have~~$\exc^{\ev}(A_\lambda)=\exc^{\ev}_L(A_\lambda)\cap\Gamma$ for all $\lambda$, and thus  
  $\exc^{\operatorname{u}}(A) \subseteq \exc^{\operatorname{u}}_L(A)$, where $\exc^{\operatorname{u}}_L(A)$ is the
set of  ultimate exceptional values of $A\in L[\der]^{\ne}$  with respect to $\Lambda_L$.
\end{cor}
\begin{proof}
By the remarks at the beginning of this subsection applied to $F$, $L$ in place of~$H$,~$K$ we have 
$L^\dagger=F+W_F\imag$ where
$W_F$ is a subspace of the $\Q$-linear space~$F$. Also  $K^\dagger=H + \I(H)\imag$ by Lemma~\ref{lem:W and I(F)}, and $L^\dagger\cap K= K^\dagger$ by Lemma~\ref{lem:LambdaL}.  This yields a
complement~$\Lambda_L$ of $L^\dagger$ in $L$ with
$\Lambda\subseteq\Lambda_L\subseteq F\imag$. 
 Since~$H$ is Liouville closed and hence $\upl$-free by [ADH, 11.6.2], its algebraic closure~$K$ is $\upl$-free by [ADH, 11.6.8].   Now the rest follows from  remarks preceding Lemma~\ref{lem:excu 1}. 
\end{proof}

\noindent
Given $A\in K[\der]^{\neq}$, let $\exc^{\operatorname{u}}(A^\phi)$ be the set of ultimate exceptional values
of the linear differential operator $A^\phi\in K^\phi[\derdelta]$, $\derdelta=\phi^{-1}\der$, with respect to  $\Lambda^\phi=\phi^{-1}\Lambda$. 
We summarize some properties of ultimate exceptional values used later in this section:

\begin{lemma}\label{lem:excu properties}
Let $A\in K[\der]^{\neq}$ have order~$r$. Then for all $b\in K^\times$ and all $\phi$,
$$\exc^{\operatorname{u}}(bA)\ =\ \exc^{\operatorname{u}}(A),\quad  \exc^{\operatorname{u}}(Ab)\ =\ \exc^{\operatorname{u}}(A)-vb,
\quad \exc^{\operatorname{u}}(A^\phi)\ =\ \exc^{\operatorname{u}}(A).$$
Moreover, if  $\I(K)\subseteq K^\dagger$, then: 
\begin{enumerate}
\item[\rm{(i)}] $\abs{\exc^{\operatorname{u}}(A)} \leq r$;
\item[\rm{(ii)}] $\dim_{C[\imag]} \ker_{\Univ} A=r\ \Longrightarrow\ 
\exc^{\operatorname{u}}(A)=v_{\g}(\ker_{\Univ}^{\neq} A)$; 
\item[\rm{(iii)}] under the assumption that $\fv:=\fv(A)\prec^\flat 1$ and  $B\prec_{\Delta(\fv)} \fv^{r+1}A$ where~$B\in K[\der]$ has order~$\leq r$,
we have $\exc^{\operatorname{u}}(A+B)=\exc^{\operatorname{u}}(A)$;
\item[\rm(iv)] for $r=1$ we have $\abs{\exc^{\operatorname{u}}(A)} = 1$ and $ \exc^{\operatorname{u}}(A)=v_{\g}(\ker_{\Univ}^{\neq} A)$.
\end{enumerate}
\end{lemma}

\begin{proof}
For the displayed equalities, see Remark~\ref{rem:excu for different Q, 1}. Now assume~$\I(K)\subseteq K^\dagger$. Then~$K^\dagger=H+\I(H)\imag$, 
so (i) and (ii) follow from  Proposition~\ref{prop:finiteness of excu(A), real}
and (iii) from Proposition~\ref{prop:stability of excu, real}. Corollary~\ref{cor:excu independent of Q} yields (iv). 
\end{proof}

\noindent
Recall from Lemma~\ref{lem:ADH 14.2.5} that if $K$ is $1$-linearly newtonian, then $\I(K)\subseteq K^\dagger$. 

Suppose $\I(K)\subseteq K^\dagger$. Then
$K^\dagger=H+\I(H)\imag$, so our $\Lambda$ has the form~$\Lambda_H\imag$ with~$\Lambda_H$ a complement of $\I(H)$ in $H$. Conversely, any complement~$\Lambda_H$ of~$\I(H)$ in $H$ yields a complement $\Lambda=\Lambda_H\imag$ of $K^\dagger$ in $K$ with $\Lambda\subseteq H\imag$. Now $\I(H)$ is a $C$-linear subspace of $H$, so 
$\I(H)$ has a complement $\Lambda_H$ in $H$ that is a $C$-linear subspace of $H$, and then $\Lambda:=\Lambda_H\imag$ is also a $C$-linear subspace of $K$. 

\begin{lemma}\label{lladd}
Suppose $\I(K)\subseteq K^\dagger$ and $g\in K$, $g-\lambda\in K^\dagger$. Then 
$$\Im g \in \I(H)\ \Longleftrightarrow\  
\lambda=0, \qquad \Im g \notin \I(H)\ \Longrightarrow\ \lambda \sim (\Im g)\imag.$$
\end{lemma}
\begin{proof}
Recall that $\Lambda=\Lambda_H\imag$ where $\Lambda_H$ is a complement of $\I(H)$ in $H$, so $\lambda=\lambda_H\imag$ where $\lambda_H\in\Lambda_H$. Also, $K^\dagger=H\oplus\I(H)\imag$, hence $\Im(g)-\lambda_H\in \I(H)$; this proves the displayed equivalence. Suppose $\Im g \notin \I(H)$; since $\I(H)$  is an $\mathcal O_H$-submodule of $H$ and $\lambda_H\notin\I(H)$, we then have $\Im(g)-\lambda_H\prec\lambda_H$,
so $\lambda =\lambda_H\imag\sim \Im(g)\imag$.
\end{proof}

\begin{cor}\label{cor:evs v-small}
 Suppose $\I(K)\subseteq K^\dagger$,  $A\in K[\der]^{\neq}$ has order $r$, $\dim_{C[\imag]}\ker_{\Univ}A=r$, and
 $\lambda$ is an eigenvalue of $A$ with respect to $\Lambda$. Then $\lambda\preceq\fv(A)^{-1}$.
\end{cor}
\begin{proof}
Take $f\neq 0$ and $g_1,\dots,g_r$ in $K$ with   $A=f(\der-g_1)\cdots(\der-g_r)$.
By Corollary~\ref{cor:bound on linear factors} we have $g_1,\dots,g_r\preceq\fv(A)^{-1}$, and
so Corollary~\ref{corbasiseigenvalues} gives $j\in\{1,\dots,r\}$ with $g_j-\lambda\in K^\dagger$. Now use Lemma~\ref{lladd}.
\end{proof}

\subsection*{Ultimate slots in $H$}
{\it In this subsection  $a$, $b$ range over $H$.}\/
Also, $(P,\fm,\hat a)$ is a slot in $H$ of order $r\geq 1$, where $\hat a\in \hat H\setminus H$. Recall that~$L_{P_{\times\fm}}=L_P\fm$, so
if $(P,\fm,\hat a)$ is normal, then $L_P$ has order $r$. 

\begin{cor} \label{4.7} 
Suppose $\I(K)\subseteq K^\dagger$  and the slot $(P,\fm,\hat a)$ is split-normal with linear part $L:=L_{P_{\times\fm}}$.
Then with~$Q$ and~$R$ as in~\textup{(SN2)} 
we have $\exc^{\operatorname{u}}(L)=\exc^{\operatorname{u}}(L_Q)$.
\end{cor}

\noindent
This follows from Lemmas~\ref{splnormalnormal} and~\ref{lem:excu properties}(iii).  In a similar vein we have an analogue of Lemma~\ref{lem:excev normal}:

{\sloppy
\begin{lemma}\label{lem:excu normal}
Suppose $(P,\fm,\hat a)$ is normal and $a\prec\fm$.  Then $L_P$ and $L_{P_{+a}}$ have or\-der~$r$, and if $\I(K)\subseteq K^\dagger$, then $\exc^{\operatorname{u}}(L_{P})=\exc^{\operatorname{u}}(L_{P_{+a}})$. 
\end{lemma}
\begin{proof} 
We have $L_{P_{\times \fm}}=L_P\fm$ and  $L_{P_{+a,\times \fm}}=L_{P_{\times \fm, +a/\fm}}= L_{P_{+a}}\fm$.  The slot~$(P_{\times\fm},1,\hat a/\fm)$ in $H$ is normal and $a/\fm\prec 1$.   Lemma~\ref{lem:linear part, split-normal, new} applied to~$\hat H$,~$P_{\times\fm}$,~$\hat a/\fm$ in place of $K$, $P$, $a$, respectively,  gives: $L_P$ and $L_{P_{+a}}$ have order~$r$, and 
$$L_P\fm - L_{P_{+a}}\fm\ =\ L_{P_{\times\fm}} - L_{P_{\times\fm,+a/\fm}}\ \prec_{\Delta(\fv)}\ \fv^{r+1}L_{P}\fm$$ where $\fv:=\fv(L_{P}\fm)\prec^\flat 1$ by (N1). Suppose now that  $\I(K)\subseteq K^\dagger$. Then
$$\exc^{\operatorname{u}}(L_P)\ =\ \exc^{\operatorname{u}}(L_P\fm)+v(\fm)\ =\ \exc^{\operatorname{u}}(L_{P_{+a}}\fm) + v(\fm)\ =\  \exc^{\operatorname{u}}(L_{P_{+a}})$$
by Lemma~\ref{lem:excu properties}(iii).
\end{proof}}

\noindent
The notion introduced below is modeled on that of ``isolated slot'' (Definition~\ref{def:isolated}): 

\begin{definition}\label{def:ultimate}
Call $(P,\fm,\hat a)$   {\bf ultimate} if for all $a\prec\fm$,\index{ultimate!slot}\index{slot!ultimate}
$$\order(L_{P_{+a}})=r\ \text{ and }\ \exc^{\operatorname{u}}(L_{P_{+a}}) \cap v(\hat a-H)\ <\  v(\hat a-a);$$
equivalently, for all $a\prec \fm$:  $\order(L_{P_{+a}})=r$ and whenever
$\fw \preceq \hat a-a$ is such that~$v(\fw) \in \exc^{\operatorname{u}}(L_{P_{+a}})$, then 
$\fw\prec \hat a-b$ for all $b$. (Thus if $(P,\fm,\hat a)$ is ultimate, then it is isolated.) 
\end{definition}

\noindent
If $(P,\fm,\hat a)$ is ultimate, then so is every equivalent slot in $H$ and $(bP,\fm,\hat a)$ for~$b\neq 0$, as well as the slot~$(P^\phi,\fm,\hat a)$ in $H^\phi$ (by Lemma~\ref{lem:excu properties}). The proofs of the next two lemma are like those of their ``isolated'' versions, Lemmas~\ref{lem:isolated refinement} and~\ref{lem:isolated}:

\begin{lemma}\label{lem:ultimate refinement}
If $(P,\fm,\hat a)$ is ultimate, then so is any of its refinements. 
\end{lemma}

\begin{lemma}\label{lem:ultmult}
If  $(P,\fm,\hat a)$ is ultimate, then so is any of its multiplicative con\-ju\-gates.
\end{lemma}

\noindent
The ultimate condition is most useful in combination with other properties: 

\begin{lemma}\label{lem:ultimate normal}
If $\I(K)\subseteq K^\dagger$ and $(P,\fm,\hat a)$ is normal, then 
$$ \text{$(P,\fm,\hat a)$  is ultimate} \quad\Longleftrightarrow\quad
\exc^{\operatorname{u}}(L_P) \cap v(\hat a-H) \leq v\fm.$$
\end{lemma}
\begin{proof} Use Lemma~\ref{lem:excu normal} and the equivalence  $\hat a-a\prec\fm\Leftrightarrow a\prec\fm$.
\end{proof}

\noindent
The ``ultimate'' version of Lemma~\ref{lem:isolated deg 1} has the same proof: 

\begin{lemma}\label{lem:ultimate deg 1} 
If  $\deg P=1$, then
$$ \text{$(P,\fm,\hat a)$  is ultimate} \quad\Longleftrightarrow\quad
\exc^{\operatorname{u}}(L_P) \cap v(\hat a-H) \leq v\fm.$$
\end{lemma}

\noindent
The next proposition is the ``ultimate'' version of Proposition~\ref{prop:achieve isolated}:

\begin{prop}\label{prop:achieve ultimate}
Suppose $\I(K)\subseteq K^\dagger$, 
and~$(P,\fm,\hat a)$  is normal. 
Then   $(P,\fm,\hat a)$  has an ultimate refinement.
\end{prop}
\begin{proof}
Suppose $(P,\fm,\hat a)$ is not already ultimate.
Then Lem\-ma~\ref{lem:ultimate normal} gives $\gamma$ with
$$\gamma\in\exc^{\operatorname{u}}(L_P)\cap v(\hat a-H),\quad 
\gamma>v\fm.$$
Lemma~\ref{lem:excu properties}(i) gives $|\exc^{\operatorname{u}}(L_P)|\le r$, 
so we can take $$\gamma\ :=\  \max\exc^{\operatorname{u}}(L_P)\cap v(\hat a-H),$$ and then $\gamma > v\fm$.
Take~$a$ and $\fn$ with $v(\hat a-a)>\gamma=v(\fn)$;
then $(P_{+a},\fn,\hat a-a)$ is a refinement of~$(P,\fm,\hat a)$ and $a\prec\fm$. 
Let $b\prec \fn$; then 
$a+b\prec\fm$, so by~Lemma~\ref{lem:excu normal},
$$\order(L_{(P_{+a})_{+b}})\ =\ r, \qquad
\exc^{\operatorname{u}}(L_{(P_{+a})_{+b}})\ =\ 
\exc^{\operatorname{u}}(L_P).$$
Also $v\big((\hat a-a)-b\big)>\gamma$, hence
$$\exc^{\operatorname{u}}\big(L_{(P_{+a})_{+b}}\big)  \cap v\big((\hat a-a)-H\big)\  =\ 
\exc^{\operatorname{u}}(L_P)\cap v(\hat a-H)\ \le\ \gamma\ <\ v\big((\hat a-a)-b\big).$$
Thus $(P_{+a},\fn,\hat a-a)$  is ultimate.
\end{proof}

\begin{remarkNumbered}\label{rem:achieve ultimate}
Proposition~\ref{prop:achieve ultimate} goes through if instead of assuming that $(P,\fm,\hat a)$  is normal, we assume
that $(P,\fm,\hat a)$  is linear. (Same argument, using Lem\-ma~\ref{lem:ultimate deg 1} in place of Lemma~\ref{lem:ultimate normal}.)
\end{remarkNumbered}

\noindent
Finally, here is a consequence of Corollaries~\ref{corevisu},~\ref{cor:excu independent of Q}, and   Lem\-ma~\ref{lem:ultimate normal}, where we recall that $\order(L_{P_{\times \fm}})=\order(L_P\fm)=\order(L_P)$: 

\begin{cor}\label{cor:ultimate order=1}
Suppose  $\I(K)\subseteq K^\dagger$ and $(P,\fm,\hat a)$ is normal of order~$r=1$. 
Then~$L_{P}=f(\der-g)$ with $f\in H^\times$, $g\in H$, and for $\mathfrak g\in H^\times$ with $\mathfrak g^\dagger=g$ we have:
$$\text{$(P,\fm,\hat a)$  is ultimate} \quad\Longleftrightarrow\quad 
\text{$(P,\fm,\hat a)$  is isolated} \quad\Longleftrightarrow\quad 
 \text{$\mathfrak g \succeq \fm$ or $\mathfrak g \prec \hat a-H$.}$$ 
\textup{(}In particular, if~$g\in\I(H)$ and~$\fm\preceq 1$,
then~$(P,\fm,\hat a)$  is ultimate.\textup{)}
\end{cor}

\subsection*{Ultimate slots in $K$}
{\it In this subsection, $a$, $b$ range over $K=H[\imag]$.}\/
Also $(P,\fm,\hat a)$ is a slot in $K$ of order $r\geq 1$, where $\hat a\in \hat K\setminus K$.
Lemma~\ref{lem:excu normal} goes through in this setting, with $H$ in the proof replaced by $K$:

\begin{lemma}\label{lem:excu normal, K} 
Suppose $(P,\fm,\hat a)$ is normal, and $a\prec\fm$.  Then $L_P$ and $L_{P_{+a}}$ have order $r$, and if $\I(K)\subseteq K^\dagger$, then $\exc^{\operatorname{u}}(L_{P})=\exc^{\operatorname{u}}(L_{P_{+a}})$.
\end{lemma}

\noindent
We adapt Definition~\ref{def:ultimate} to slots in $K$:
call $(P,\fm,\hat a)$    {\bf ultimate} if for all $a\prec\fm$ we have $\order(L_{P_{+a}})=r$ and
$\exc^{\operatorname{u}}(L_{P_{+a}}) \cap v(\hat a-K) < v(\hat a-a)$.\index{ultimate!slot}\index{slot!ultimate} If $(P,\fm,\hat a)$ is ultimate, then it is isolated. 
Moreover, if $(P,\fm,\hat a)$ is ultimate, then so is~$(bP,\fm,\hat a)$ for $b\neq 0$  as well as
the slot~$(P^\phi,\fm,\hat a)$ in $K^\phi$. 
Lemmas~\ref{lem:ultimate refinement} and~\ref{lem:ultmult} go through in the present context,  and so do Lemmas~\ref{lem:ultimate normal} and~\ref{lem:ultimate deg 1} with $H$ replaced by $K$.   The analogue of Proposition~\ref{prop:achieve ultimate} 
follows likewise:

\begin{prop}\label{prop:achieve ultimate, K} 
If  $\I(K)\subseteq K^\dagger$ and~$(P,\fm,\hat a)$  is normal, then~$(P,\fm,\hat a)$  has an ultimate refinement.
\end{prop}

\begin{remarkNumbered}\label{rem:achieve ultimate, K} 
Proposition~\ref{prop:achieve ultimate, K} also holds if instead of assuming that $(P,\fm,\hat a)$  is normal, we assume
that $(P,\fm,\hat a)$  is linear. 
\end{remarkNumbered}

\noindent
Corollary~\ref{cor:excu independent of Q} and the $K$-versions of Lem\-mas~\ref{lem:ultimate normal} and~\ref{lem:ultimate deg 1} yield:

\begin{cor}\label{cor:ultimate order=1, K}
Suppose $\I(K)\subseteq K^\dagger$, $r=1$, and $(P,\fm,\hat a)$ is normal or linear. Then
$L_P=f(\der-g)$ with $f\in K^\times, g\in K$, and for $\mathfrak g\in H^\times$ with
$\mathfrak g^\dagger=\Re g$ we have:
$$\text{$(P,\fm,\hat a)$  is ultimate} \quad\Longleftrightarrow\quad  \text{$\mathfrak g\succeq \fm$ or 
$\mathfrak g \prec \hat a-K $.}$$ 
\textup{(}In particular, if~$\Re g\in\I(H)$ and~$\fm\preceq 1$, 
then~$(P,\fm,\hat a)$  is ultimate.\textup{)} 
\end{cor}

\subsection*{Using the norm to characterize being ultimate\astr} We employ here the  ``norm''~$\dabs{\,\cdot\,}$ on $\Univ$
and the gaussian extension $v_{\g}$ of the valuation of $K$ from Section~\ref{sec:group rings}. 

\begin{lemma}\label{uudag} For $u\in \Univ^\times$ we have $\dabs{u}^\dagger=\Re u^\dagger$.
\end{lemma}
\begin{proof} For $u=f\ex(\lambda)$, $f\in K^\times$ we have $\dabs{u}=|f|$ and
$u^\dagger=f^\dagger + \lambda$, so 
$$\dabs{u}^\dagger\ =\ |f|^\dagger\ =\ \Re f^\dagger\ =\ \Re u^\dagger,$$
using Corollary~\ref{cor:logder abs value} for the second equality. 
\end{proof}

\noindent
Using  Corollary~\ref{cor:valuation and norm}, Lemma~\ref{uudag}, and [ADH, 10.5.2(i)] we obtain:

\begin{lemma}\label{Wlem}
Let  $\mathfrak W\subseteq H^\times$ be $\preceq$-closed.
Then for all $u\in\Univ^\times$,
$$\dabs{u} \in \mathfrak W \quad \Longleftrightarrow\quad
v_{\g}u \in v(\mathfrak W) \quad \Longleftrightarrow\quad 
\text{$\Re u^\dagger < \fn^\dagger$ for all $\fn\notin \mathfrak W$.}$$
\end{lemma} 

\noindent
Let $(P,\fm,\hat a)$ be a slot in $H$ of order $r\ge 1$. 
Applying Lemma~\ref{Wlem} to the set~$\mathfrak W = \{ \fw:\ \fw\prec\hat a-H\}$---so $v(\mathfrak W)=\Gamma\setminus v(\hat a-H)$---we obtain a  reformulation of the condition ``$(P,\fm,\hat a)$ is ultimate'' 
in terms of the ``norm'' $\dabs{\,\cdot\,}$ on $\Univ$:

{\samepage
\begin{cor}
The following are equivalent \textup{(}with $a$ ranging over $H$\textup{)}:
\begin{enumerate}
\item[\rm{(i)}] $(P,\fm,\hat a)$  is ultimate;
\item[\rm{(ii)}] for  all $a\prec\fm$: $\order(L_{P_{+a}})=r$ and whenever~$u\in\Univ^\times$, $v_{\g}u \in \exc^{\operatorname{u}}(L_{P_{+a}})$, and
$\dabs{u}\preceq \hat a-a$, then  $\dabs{u}\prec \hat a-H$; 
\item[\rm{(iii)}] for  all $a\prec\fm$: $\order(L_{P_{+a}})=r$ and whenever~$u\in\Univ^\times$, $v_{\g}u \in \exc^{\operatorname{u}}(L_{P_{+a}})$, and 
$\dabs{u}\preceq \hat a-a$, then 
$\Re u^\dagger < \fn^\dagger$ for all $\fn$ with 
 $v(\fn)\in v(\hat a-H)$.
\end{enumerate}
\end{cor}}

\subsection*{A counterexample\astr} 
Suppose  $\I(K)\subseteq K^\dagger$ and $H$ is not  $\upo$-free. 
(Example~\ref{ex:Gehret} provides such $H$.) 
Let $(\upl_{\rho})$ and~$(\upo_{\rho})$ be as in Lemma~\ref{lem:upl-free, not upo-free} with
$H$ in the role of~$K$ there. That lemma yields a minimal hole $(P,\fm, \upl)$ in $H$ with $P= 2Y'+Y^2+\upo$ ($\upo\in H$). This is a good source of counterexamples: 

\begin{lemma}\label{lem:counterex}
The minimal hole $(P,\fm,\upl)$ in $H$ is ultimate,  and none of its refinements is quasilinear or normal.
\end{lemma}
\begin{proof} 
Let $a\in H$.  Then
$P_{+a} = 2Y'+2aY+Y^2 + P(a)$
and thus~$L_{P_{+a}}=2(\der+a)$, so for $b\in H^\times$ with $b^\dagger=-a$ we have
$\exc^{\operatorname{u}}(L_{P_{+a}})=\{vb\}$, by Corollary~\ref{cor:excu independent of Q}. 
Thus~$(P,\fm,\upl)$ is ultimate iff $\upl-a\prec b$ for all $a\prec\fm$ in $H$ and $b\in H^\times$ with~$b^\dagger=-a$
and $vb\in v(\upl-H)$; the latter holds by~[ADH, 11.5.6] since $v(\upl-H)=\Psi$.
Hence~$(P,\fm,\upl)$ is ultimate. No refinement of $(P,\fm,\upl)$ is    quasilinear  by Corollary~\ref{cor:quasilinear refinement} and
[ADH, 11.7.9], and so by Corollary~\ref{cor:normal=>quasilinear}, no refinement of $(P,\fm,\upl)$ is normal.
\end{proof}

\section{Repulsive-Normal Slots}\label{sec:repulsive-normal}

\noindent
{\it In this section   $H$ is a real closed $H$-field with
small derivation and asymptotic integration, with $\Gamma:= v(H^\times)$.
Also $K:= H[\imag]$ with $\imag^2=-1$ is an algebraic closure of $H$.}\/ We study here the concept of a repulsive-normal slot in $H$,  which strengthens that of a  split-normal slot in $H$.  
Despite their name, 
repulsive-normal slots will turn out to have attractive analytic properties in the realm of Hardy fields.

\subsection*{Attraction and repulsion}
In this subsection~$a$, $b$ range over $H$,  $\fm$,~$\fn$  over $H^\times$, $f$,~$g$,~$h$ (possibly with subscripts)   over $K$, and $\gamma$, $\delta$ over $\Gamma$.\index{element!repulsive}\index{repulsive!element}\index{element!attractive} 
We say that~$f$ is {\bf attractive} if $\Re f\succeq 1$ and $\Re f < 0$, and
{\bf repulsive} if $\Re f\succeq 1$ and~$\Re f > 0$.
If~$\Re f \sim \Re g$, then $f$ is attractive iff $g$ is attractive, and likewise with
``repulsive'' in place of ``attractive''.
Moreover, if $a>0$, $a\succeq 1$, and $f$ is attractive (repulsive), then~$af$ is attractive (repulsive, respectively). 

\begin{definition}
Let $\gamma>0$; we say $f$ is  {\bf $\gamma$-repulsive}\index{element!gamma-repulsive@$\gamma$-repulsive}\index{S-repulsive@$S$-repulsive!element} if $v(\Re f) < \gamma^\dagger$ or~${\Re f > 0}$.
Given $S\subseteq\Gamma$, we say $f$ is {\bf $S$-repulsive} if $f$ is
$\gamma$-repulsive for all $\gamma\in S\cap\Gamma^>$, equivalently,
$\Re f >0$, or $v(\Re f) < \gamma^\dagger$ for all $\gamma\in S\cap \Gamma^{>}$.
\end{definition}

\noindent
Note the following implications for $\gamma>0$:
\begin{align*} \text{$f$ is $\gamma$-repulsive}\quad &\Longrightarrow\quad   \Re f\neq 0,\\
 \text{$f$ is $\gamma$-repulsive, $\Re g\sim\Re f$}\quad &\Longrightarrow\quad  \text{$g$ is $\gamma$-repulsive}.
\end{align*} 
The following is easy to show:

\begin{lemma}\label{lem:repulsive 1}
Suppose $\gamma>0$ and $\Re f \succeq 1$. Then    $f$ is $\gamma$-repulsive iff $v(\Re f) < \gamma^\dagger$ or~$f$ is repulsive.
Hence, if $f$ is repulsive, then  $f$ is $\Gamma$-repulsive; the converse of this implication holds if $\Psi$ is not
bounded from below in $\Gamma$.
\end{lemma}

\noindent  
Let $\gamma, \delta>0$. If $f$ is $\gamma$-repulsive and $a>0$, $a\succeq 1$, then $af$ is $\gamma$-repulsive.
If $f$ is $\gamma$-repulsive and $\delta$-repulsive, then $f$ is $(\gamma+\delta)$-repulsive. If $f$ is $\gamma$-repulsive
and~$\gamma>\delta$, then  $f$ is $(\gamma-\delta)$-repulsive. Moreover:

\begin{lemma}\label{lem:repulsive 2}
Suppose $\gamma\geq\delta=v\fn>0$. Set $g:=f-\fn^\dagger$. Then:
$$\text{$f$ is $\gamma$-repulsive}\quad\Longleftrightarrow\quad\text{$f$ is $\delta$-repulsive and $g$ is $\gamma$-repulsive.}$$
\end{lemma} 
\begin{proof} Note that $\gamma\geq\delta>0$ gives $\gamma^\dagger \leq \delta^\dagger$.
Suppose $f$ is $\gamma$-repulsive;   by our remark, $f$ is $\delta$-repulsive. Now if $v(\Re f) < \gamma^\dagger$, then $\Re g\sim\Re f$,   whereas if~${\Re f>0}$, then~$\Re g=\Re f-\fn^\dagger > \Re f>0$; in both cases, $g$ is $\gamma$-repulsive.
Conversely, suppose $f$ is $\delta$-repulsive and $g$ is $\gamma$-repulsive.
If $\Re f>0$, then clearly~$f$ is $\gamma$-repulsive. Otherwise, $v(\Re f)<\delta^\dagger$, hence $\Re g\sim \Re f$, so $f$ is also $\gamma$-repulsive.
\end{proof}

\noindent
In a similar way we deduce a useful characterization of repulsiveness:

\begin{lemma}\label{lem:repulsive 3}
Suppose $\gamma=v\fm>0$. Set $g:=f-\fm^\dagger$. Then:
$$\text{$f$ is repulsive}\quad\Longleftrightarrow\quad\text{$\Re f\succeq 1$, $f$ is $\gamma$-repulsive, and
 $g$ is repulsive.}$$
\end{lemma}
\begin{proof}
Suppose $f$ is repulsive; then by Lemma~\ref{lem:repulsive 1}, $f$ is $\gamma$-repulsive. Moreover,
$\Re g=\Re f-\fm^\dagger > \Re f>0$, hence $\Re g\succeq 1$ and
$\Re g>0$, that is, $g$ is repulsive.
Conversely, suppose $\Re f\succeq 1$, $f$ is $\gamma$-repulsive, and $g$ is repulsive.
If $v(\Re f) < \gamma^\dagger$, then $\Re f\sim \Re g$;
otherwise $\Re f>0$. In both cases,   $f$ is repulsive.
\end{proof}

\begin{cor}\label{cor:repulsive}
Suppose $f$ is $\gamma$-repulsive where $\gamma=v\fm>0$, and $\Re f\succeq 1$. Then~$f$ is repulsive iff $f-\fm^\dagger$ is repulsive, and $f$ is attractive iff $f-\fm^\dagger$ is attractive.
\end{cor}
\begin{proof}
The first equivalence is immediate from Lemma~\ref{lem:repulsive 3};
this equivalence yields 
\begin{align*}
\text{$f$ is attractive} &\ \Longleftrightarrow\  \text{$f$ is not repulsive} 
\ \Longleftrightarrow\  \text{$f-\fm^\dagger$ is not repulsive} \\
&\ \Longleftrightarrow \ \text{$\Re f-\fm^\dagger\prec 1$ or $f-\fm^\dagger$ is attractive.}
\end{align*}
Thus if $f-\fm^\dagger $ is attractive, so is $f$. Now assume towards a contradiction that~$f$ is attractive and $f-\fm^\dagger$ is not. 
Then $\Re f <0$ and $\Re f -\fm^\dagger \prec 1$ by the above equivalence, so
 $\Re f\sim\fm^\dagger$ thanks to $\Re f\succeq 1$. But $f$ is $\gamma$-repulsive, that is, 
 $\Re f\succ \fm^\dagger$ or $\Re f >0$, a contradiction.
\end{proof}

\begin{lemma}\label{lem:make repulsive}
Suppose  $\gamma=v\fm>0$ and $v(\Re g)\geq\gamma^\dagger$. Then for all sufficiently large~$c\in C^{>}$ we have 
$\Re g-c\fm^\dagger > 0$ \textup{(}and hence $g-c\fm^\dagger$ is  $\Gamma$-repulsive\textup{)}.
\end{lemma}
\begin{proof}
If $v(\Re g)>\gamma^\dagger$, then $\Re g-c\fm^\dagger\sim-c\fm^\dagger>0$ for all $c\in C^>$.
Suppose~$v(\Re g)=\gamma^\dagger$. Take $c_0\in C^\times$ with $\Re g\sim c_0\fm^\dagger$;
then  $\Re g-c\fm^\dagger >0$ for~$c>c_0$.
\end{proof}

\noindent
{\em In the rest of this subsection we assume that $S\subseteq \Gamma$}. If $f$ is $S$-repulsive, then so is~$af$ for $a>0$, $a\succeq 1$.
If $S>0$, $\delta>0$, and $f$ is $S$-repulsive and $\delta$-repulsive, then $f$ is $(S+\delta)$-repulsive.

\begin{lemma}\label{lem:repulsive hata, 1}
Suppose $f$ is $S$-repulsive and $0<\delta=v\fn \in S$. Then
\begin{enumerate}
\item[\textup{(i)}] $f$ is $(S-\delta)$-repulsive;
\item[\textup{(ii)}]   $g:=f-\fn^\dagger$ is $S$-repulsive.
\end{enumerate}
\end{lemma}
\begin{proof}
Let $\gamma\in (S-\delta)$, $\gamma>0$.
Then $\gamma+\delta\in S$,  so
$f$ is $(\gamma+\delta)$-repulsive, hence $\gamma$-repulsive.
This shows (i).
For (ii), suppose   $\gamma\in S$, $\gamma>0$; we  need to show that  $g$ is $\gamma$-repulsive.
If $\gamma\geq\delta$, then $g$ is $\gamma$-repulsive
by Lemma~\ref{lem:repulsive 2}. Taking $\gamma=\delta$ we see that $g$ is $\delta$-repulsive, hence if $\gamma<\delta$, then
$g$ is also $\gamma$-repulsive.
\end{proof}

\noindent
Let $A\in K[\der]^{\neq}$ have order $r\ge 1$. An {\bf $S$-repulsive splitting}\index{linear differential operator!S-repulsive splitting@$S$-repulsive splitting}\index{splitting!S-repulsive@$S$-repulsive}\index{S-repulsive@$S$-repulsive!splitting} of $A$ over $K$ is a splitting $(g_1,\dots,g_r)$ of $A$ over $K$ where  $g_1,\dots,g_r$ are $S$-repulsive.
An $S$-repulsive splitting of $A$ over $K$ remains an
$S$-repulsive splitting of~$hA$ over $K$ for $h\neq 0$.
We say that  {\bf $A$ splits $S$-repulsively} over $K$ if
there is an $S$-repulsive splitting of $A$ over $K$.
From Lemmas~\ref{lem:split and twist} and~\ref{lem:repulsive hata, 1} we obtain:

{\sloppy
\begin{lemma}\label{lem:repulsive hata, 2}
Suppose $(g_1,\dots,g_r)$ is an $S$-repulsive splitting of $A$ over $K$ and~${0<\delta=v\fn \in S}$.
Then $(g_1,\dots,g_r)$ is  an $(S-\delta)$-repulsive splitting of $A$ over~$K$, and~$(h_1,\dots,h_r):={(g_1-\fn^\dagger,\dots,g_r-\fn^\dagger)}$ is an $S$-repulsive splitting of
$A\fn$ over $K$.  
\textup{(}Hence~${(h_1,\dots,h_r)}$  is also an  $(S-\delta)$-re\-pul\-sive splitting of $A\fn$  over $K$.\textup{)}
\end{lemma}}

\noindent
Note that if $\phi$ is active in $H$ with $0<\phi\preceq 1$, and $f$ is $\gamma$-repulsive (in~$K$), then~$\phi^{-1}f$ is $\gamma$-repulsive in $K^\phi=H^\phi[\imag]$.

{\sloppy
\begin{lemma}\label{lem:repulsive splitting compconj} 
Suppose $(g_1,\dots,g_r)$ is an $S$-repulsive splitting of $A$ over $K$  and~${S\cap \Gamma^{>}}\not\subseteq\Gamma^\flat$.
Let~$\phi$ be active in $H$ with $0<\phi\prec 1$, and set $h_j:=g_j-(r-j)\phi^\dagger$ for~$j=1,\dots,r$. 
Then $(\phi^{-1}h_1,\dots,\phi^{-1}h_r)$ is an $S$-repulsive splitting of $A^\phi$ over~$K^\phi$.
\end{lemma}}
\begin{proof}
By Lem\-ma~\ref{lem:split and compconj}, $(\phi^{-1}h_1,\dots,\phi^{-1}h_r)$ is   splitting of $A^\phi$ over~$K^\phi$. 
Let $j\in\{1,\dots,r\}$. If $\Re g_j>0$, then $\phi^\dagger < 0$ yields  $\Re h_j\geq \Re g_j>0$.
Otherwise, 
$v(\Re g_j)<\gamma^\dagger$ whenever $0<\gamma\in S$; in particular, $\Re g_j\succ 1\succ\phi^\dagger$,
so $\Re h_j\sim  \Re g_j$. In both cases $h_j$ is $S$-repulsive, so
$\phi^{-1}h_j$ is $S$-repulsive in $K^\phi$.
\end{proof}

\begin{prop}\label{prop:repulsive splitting}
Suppose $S\cap \Gamma^{>}\neq\emptyset$, $nS\subseteq S$ for all $n\geq 1$,   the ordered constant field~$C$ of~$H$ is ar\-chi\-me\-dean, and 
$(g_1,\dots,g_r)$ is a splitting of $A$ over $K$.  Then there exists~$\gamma\in S\cap \Gamma^{>}$
such that for any $\fm$ with $\gamma=v\fm$: $({g_1-n\fm^\dagger},\dots,{g_r-n\fm^\dagger})$ is an $S$-repulsive splitting of
$A\fm^n$ over $K$, for all big enough $n$.
\end{prop}
\begin{proof}
Let~$J$ be the set of $j\in\{1,\dots,r\}$ such that
$g_j$ is not $S$-repulsive. If  $\gamma >0$ and $g$ is not $\gamma$-repulsive, then~$g$ is not
$\delta$-repulsive, for all $\delta\geq\gamma$. 
Hence we can take $\gamma\in S\cap \Gamma^{>}$ such that $g_j$ is not $\gamma$-repulsive, for all $j\in J$.  
Suppose $\gamma=v\fm$.
Lemma~\ref{lem:make repulsive} yields $m\geq 1$ such that for all $n\geq m$, setting $\fn:=\fm^n$,
$g_j-\fn^\dagger$ is $\Gamma$-repulsive for all $j\in J$. For such $\fn$ we have $v\fn\in S$, so
by Lemma~\ref{lem:repulsive hata, 1}(ii), $g_j-\fn^\dagger$ is also $S$-repulsive for $j\notin J$. 
\end{proof}

\begin{cor}\label{cor:repulsive splitting}
If  $C$ is archimedean and~$(g_1,\dots,g_r)$ is a splitting of $A$ over $K$,
 then there exists $\gamma>0$ such that for all $\fm$ with $\gamma=v\fm$: ${(g_1-n\fm^\dagger,\dots,g_r-n\fm^\dagger)}$ is a $\Gamma$-repulsive splitting of
$A\fm^n$ over $K$, for all big enough $n$. If $\Gamma\neq\Gamma^\flat$ then we can choose such $\gamma>\Gamma^\flat$.
\end{cor}
\begin{proof} Taking $S=\Gamma$ this
follows from Proposition~\ref{prop:repulsive splitting} and its proof.
\end{proof}

\noindent
In logical jargon, the condition that $C$ is archimedean is not {\em first-order}. But it is satisfied when $H$ is a Hardy field, the case
where the results of this section will be applied. For other possible uses we indicate here a first-order variant of Proposition~\ref{prop:repulsive splitting} with essentially the same proof: 

\begin{cor} 
Suppose  $(g_1,\dots,g_r)$ is a splitting of $A$ over $K$. Then there exists $\fm\prec 1$ such that for all sufficiently large $c\in C^{>}$ and all~$\fn$, if $\fn^\dagger=c\fm^\dagger$, then~$({g_1-\fn^\dagger},\dots,{g_r-\fn^\dagger})$ is a $\Gamma$-repulsive splitting of
$A\fn$ over $K$.
\end{cor}

\noindent
In connection with this corollary we recall from~\cite[p.~105]{AvdD3}
that $H$ is said to be {\it closed under powers}\/ if for all $c\in C$ and~$\fm$ there is an $\fn$
with $c\fm^\dagger=\fn^\dagger$.  

\medskip
\noindent
{\it In the rest of this section $\hat H$ is an immediate asymptotic extension of~$H$ and $\imag$ with~$\imag^2=-1$ lies in an asymptotic extension of $\hat H$. Also $K:= H[\imag]$ and $\hat K:=\hat H[\imag]$.}

\medskip
\noindent
Let $\hat a\in\hat H\setminus H$, so 
$v({\hat a-H})$
is a downward closed subset of $\Gamma$. 
We say that~$f$ is~{\bf $\hat a$-repulsive} if $f$ is~$v({\hat a-H})$-repulsive;\index{element!a-repulsive@$\hat a$-repulsive} that is, $\Re f>0$, or $\Re f \succ \fm^\dagger$ for all~$a$,~$\fm$
with $\fm\asymp\hat a-a\prec 1$. 
(Of course, this notion is only interesting if $v({\hat a-H})\cap\Gamma^>\neq \emptyset$, since otherwise every $f$ is $\hat a$-repulsive.)
Various earlier results give:

{\samepage
\begin{lemma}\label{lem:hata-repulsive}
Suppose $f$ is $\hat a$-repulsive.  Then
\begin{enumerate}
\item[\textup{(i)}] $b>0,\ b\succeq 1\ \Longrightarrow\ bf$ is $\hat a$-repulsive; 
\item[\textup{(ii)}] $f$ is   $(\hat a-a)$-repulsive;  
\item[\textup{(iii)}] $\fm \asymp 1\ \Longrightarrow\ f$ is $\hat a\fm$-repulsive; 
\item[\textup{(iv)}] $\fn\asymp \hat a-a\prec 1\ \Longrightarrow\ f$ is $\hat a/\fn$ repulsive and $f-\fn^\dagger$ is $\hat a$-repulsive.
\end{enumerate}
\end{lemma}
}

\noindent
For (iv), use Lemma~\ref{lem:repulsive hata, 1}. 
An {\bf $\hat a$-repulsive splitting} of $A$ over $K$ is a
$v({\hat a-H})$-re\-pul\-sive splitting $(g_1,\dots,g_r)$ of $A$ over $K$:\index{linear differential operator!a-repulsive splitting@$\hat a$-repulsive splitting}\index{splitting!a-repulsive@$\hat a$-repulsive} 
$$A\ =\ f(\der-g_1)\cdots(\der-g_r)\qquad \text{where $f\neq 0$ and $g_1,\dots,g_r$  are $\hat a$-repulsive.}$$
We say
that {$A$ \bf splits $\hat a$-repulsively} over $K$ if it splits $v({\hat a-H})$-repulsively over~$K$.
Thus if $A$ splits $\hat a$-repulsively over $K$, then so does $hA$ ($h\neq 0$), and 
$A$ splits $(\hat a-a)$-repulsively over $K$, and splits $\hat a\fm$-repulsively over $K$ for $\fm\asymp 1$.
Moreover, from Lemma~\ref{lem:repulsive hata, 2} we obtain:

\begin{cor}\label{cor:repulsive hata} 
Suppose $(g_1,\dots,g_r)$ is an $\hat a$-repulsive splitting of $A$ over $K$ and~$\fn\asymp \hat a-a\prec 1$. 
Then $(g_1,\dots,g_r)$ is  an $\hat a/\fn$-repulsive splitting of $A$ over~$K$ and~${(g_1-\fn^\dagger,\dots,g_r-\fn^\dagger)}$ is an $\hat a$-repulsive splitting of
$A\fn$ over $K$.  
\end{cor}

\noindent
Taking $S:=v(\hat a -H)$ in Proposition~\ref{prop:repulsive splitting} we obtain:

\begin{cor}\label{cor:repulsive hata, 1} 
If $\hat a\preceq 1$ is special over $H$, $C$ is archimedean,  and $A$ splits over~$K$, then $A\fn$ splits $\hat a$-repulsively over $K$ for some~$a$ and $\fn\asymp\hat a-a\prec 1$.
\end{cor}

\noindent
Recall that in Section~\ref{sec:approx linear diff ops} we defined a splitting $(g_1,\dots,g_r)$ of $A$ over $K$ to be {\it strong}\/ if~$\Re g_j \succeq\fv(A)^\dagger$ for $j=1,\dots,r$.

\begin{lemma}\label{lem:hata-repulsive splitting compconj}
Suppose $\hat a-a\prec^\flat 1$ for some $a$. Let $(g_1,\dots,g_r)$ be an $\hat a$-repulsive splitting of $A$ over $K$, let $\phi$ be active in $H$ with $0<\phi\prec 1$, 
and set $$h_j\ :=\ \phi^{-1}\big( g_j - (r-j) \phi^\dagger \big) \qquad (j=1,\dots,r).$$ Then
$(h_1,\dots,h_r)$ is an $\hat a$-repulsive splitting of~$A^\phi$ over $K^\phi=H^\phi[\imag]$.
If $\fv(A)\prec^{\flat} 1$ and $(g_1,\dots,g_r)$ is strong, then $(h_1,\dots,h_r)$ is strong.
\end{lemma} 

\noindent
This follows from Lemmas~\ref{lem:split strongly compconj} and~\ref{lem:repulsive splitting compconj}. 

\begin{lemma}\label{lem:achieve strong repulsive splitting} 
Suppose $\fv:=\fv(A)\prec 1$ and $\hat a \prec_{\Delta(\fv)} 1$.
Let~$(g_1,\dots,g_r)$ be an $\hat a$-repulsive splitting of $A$ over $K$. Then for all sufficiently small $q\in\Q^>$ and any~$\fn\asymp\abs{\fv}^q$, $(g_1-\fn^\dagger,\dots,g_r-\fn^\dagger)$ is a strong $\hat a/\fn$-repulsive splitting of $A\fn$ over $K$.
\end{lemma}
\begin{proof}
Take $q_0\in\Q^>$ with $\hat a\prec |\fv|^{q_0}\prec 1$. Then for any $q\in\Q$ with $0<q\leq q_0$ and any $\fn\asymp\abs{\fv}^q$,
$(g_1-\fn^\dagger,\dots,g_r-\fn^\dagger)$ is an $\hat a/\fn$-repulsive splitting of $A\fn$ over~$K$, by Corollary~\ref{cor:repulsive hata}.
Using Lemmas~\ref{lem:split strongly multconj} and~\ref{lem:Ah splits strongly} (in that order) we can decrease~$q_0$ so that for all~$q\in\Q$ with~$0<q\leq q_0$ and $\fn\asymp\abs{\fv}^q$,
$(g_1-\fn^\dagger,\dots,g_r-\fn^\dagger)$ is also a strong splitting of~$A\fn$ over $K$.
\end{proof}

\noindent
{\it In the rest of this subsection we assume that $H$ is Liouville closed with~$\I(K)\subseteq K^\dagger$.}\/
 We choose a complement~$\Lambda\subseteq H\imag$ of $K^\dagger$ in $K$ as in Section~\ref{sec:ultimate} and set $\Univ := K\big[\!\ex(\Lambda)\big]$.
 We then have the set $\exc^{\operatorname{u}}(A)\subseteq\Gamma$ of ultimate exceptional values of $A$ 
 (which doesn't depend on $\Lambda$ by Corollary~\ref{prop:excu independent of Q}).
Recall from Corollary~\ref{cor:Hardy type arch C} that $H$ is of Hardy type iff $C$ is archimedean.
{\it  We now assume $r=1$ and~$\hat a\prec 1$ is special over $H$, and let 
$\Delta$ be the nontrivial convex subgroup of $\Gamma$ that is cofinal in $v(\hat a - H)$}. 

\begin{lemma}\label{lem:aAH}
Suppose $C$ is archimedean and 
$\exc^{\operatorname{u}}(A) \cap v(\hat a-H) < 0$. Then $A$ splits $\hat a$-repulsively over $K$.
\end{lemma}
\begin{proof}
We may arrange $A=\der-f$. 
Take $u\in\Univ^\times$  with $u^\dagger=f$, and   $b:=\dabs{u}\in H^>$. 
Then $\exc^{\operatorname{u}}(A)=\{vb\}$ by Lemma~\ref{lem:excu, r=1} and its proof, hence 
$$\exc^{\operatorname{u}}(A) \cap v(\hat a-H) < 0 \quad\Longleftrightarrow\quad \text{$b\succ 1$ or $vb>\Delta$,}$$ 
and $\Re f =  b^\dagger$ by Lemma~\ref{uudag}.
If $b\succ 1$, then $\Re f>0$, and if $vb>\Delta$, then for all $\delta\in\Delta^{\neq}$ we have $\psi(vb)<\psi(\delta)$ by Lemma~\ref{lem:Hardy type}, so $\Re f \succ \fm^\dagger$ for all $a$, $\fm$ with~$\hat a-a\asymp\fm\prec 1$.
In both cases $A$ splits $\hat a$-repulsively over $K$. 
\end{proof}

\begin{lemma}
Suppose $A\in H [\der]$ and $\fv(A)\prec 1$.
Then $0\notin\exc^{\operatorname{u}}(A)$, and if
$A$ splits $\hat a$-repulsively over $K$, then $\exc^{\operatorname{u}}(A) \cap v(\hat a-H) < 0$.
\end{lemma}
\begin{proof}
We again arrange $A=\der-f$ and
take $u$, $b$ as in the proof of Lemma~\ref{lem:aAH}. Then
$f\in H$ and $b^\dagger=f=-1/\fv(A)\succ 1$, so $b\nasymp 1$, and thus $0\notin \{vb\}=\exc^{\operatorname{u}}(A)$. Now suppose $A$ splits $\hat a$-repulsively over $K$, that is,
$f>0$ or $f\succ\fm^\dagger$ for all $a$, $\fm$ with $\hat a-a\asymp\fm\prec 1$.
In the first case $f=b^\dagger$ and $b\nasymp 1$ yield $b\succ 1$.
In the second case  $\psi(vb)=vf<\psi(\delta)$ for all $\delta\in\Delta^{\neq}$, hence $vb\notin\Delta$.
\end{proof}

\noindent
Combining Lemma~\ref{lem:order 1 splits strongly} with the previous two lemmas yields:

\begin{cor}\label{cor:split rep ultimate}
Suppose $A\in H [\der]$ and $\fv(A)\prec 1$, and $H$ is of Hardy type. Then~$A$ splits strongly over $K$, and we have the equivalence
$$\text{$A$ splits $\hat a$-repulsively over $K$} \ \Longleftrightarrow\ \exc^{\operatorname{u}}(A) \cap v(\hat a-H)\ \leq\ 0.$$
\end{cor}

\subsection*{Defining repulsive-normality}
{\it In this subsection
 $(P, \fm, \hat a)$ is a slot in $H$ of or\-der~$r\ge 1$ with $\hat a\in \hat H\setminus H$ and linear part $L:=L_{P_{\times \fm}}$.} 
 Set $w:=\wt(P)$;
 if $\order L=r$, set $\fv:=\fv(L)$.
We let $a$, $b$ range over $H$ and $\fn$ over $H^\times$.

\begin{definition}  
Call $(P,\fm,\hat a)$   {\bf repulsive-normal} if $\order L=r$, and\index{slot!repulsive-normal}\index{repulsive-normal}
\begin{itemize}
\item[(RN1)] $\fv\prec^\flat 1$; 
\item[(RN2)]  $(P_{\times\fm})_{\geq 1}=Q+R$ where $Q,R\in H\{Y\}$, $Q$ is homogeneous of degree~$1$ and order~$r$,  $L_Q$ splits $\hat a/\fm$-repulsively over $K$, and $R\prec_{\Delta(\fv)} \fv^{w+1} (P_{\times\fm})_1$. 
\end{itemize}
\end{definition}

\noindent
Compare this with ``split-normality'' from Definition~\ref{SN}: clearly repulsive-normal implies split-normal, and hence normal.
If $(P,\fm,\hat a)$ is normal and $L$ splits $\hat a/\fm$-repulsively over $K$, then~$(P,\fm,\hat a)$ is repulsive-normal. 
If~$(P,\fm,\hat a)$ is  repulsive-normal, then so are $(bP,\fm,\hat a)$ for $b\neq 0$ and $(P_{\times\fn},\fm/\fn,\hat a/\fn)$.

\begin{lemma}\label{lem:repulsive-normal comp conj}
Suppose $(P,\fm,\hat a)$ is repulsive-normal and  $\phi$ is active in $H$ such that~${0<\phi\prec 1}$, and $\hat a-a\prec^\flat \fm$
for some $a$. Then the slot $(P^\phi,\fm,\hat a)$ in $H^\phi$ is re\-pul\-sive-nor\-mal.
\end{lemma}
\begin{proof}
First arrange $\fm=1$, and let $Q$, $R$ be as in (RN2) for $\fm=1$. 
Now $(P^\phi,1,\hat a)$ is split-normal by Lemma~\ref{lem:split-normal comp conj}. In fact, 
$P^\phi_{\geq 1} = Q^\phi+R^\phi$, and
the proof
of this lemma shows that $R^\phi \prec_{\Delta(\fw)} \fw^{w+1} P_1^\phi$
where $\fw:=\fv(L_{P^\phi})$.
By Lemma~\ref{lem:hata-repulsive splitting compconj}, $L_{Q^\phi}=L_Q^\phi$ splits $\hat a$-repulsively over $K^\phi$.
So  $(P^\phi,1,\hat a)$ is repulsive-normal.
\end{proof}

\noindent
If $\order L=r$, $\fv\prec^\flat 1$, and $\hat a-a\prec_{\Delta(\fv)} \fm$, then $\hat a - a\prec^\flat \fm$. Thus we obtain from
Lemmas~\ref{lem:good approx to hata} and \ref{lem:repulsive-normal comp conj} the following result:

\begin{cor}\label{cor:repulsive-normal comp conj}  
Suppose  $(P,\fm,\hat a)$ is $Z$-minimal, deep, and repulsive-normal. Let~$\phi$ be active in $H$ with $0<\phi\prec 1$. Then the slot $(P^\phi,\fm,\hat a)$ in $H^\phi$ is repulsive-normal.
\end{cor}

\noindent
Before we turn to the task of obtaining repulsive-normal slots, we deal with the preservation of repulsive-normality under refinements.

\begin{lemma}\label{lemrnlem}
Suppose $(P,\fm,\hat a)$ is repulsive-normal, and let $Q$, $R$ be as in~\textup{(RN2)}.
Let $(P_{+a},\fn,\hat a-a)$ be a steep refinement of $(P,\fm,\hat a)$ where $\fn\prec\fm$ or
$\fn=\fm$. Suppose
$$(P_{+a,\times\fn})_{\geq 1} - Q_{\times\fn/\fm} \prec_{\Delta(\fw)} \fw^{w+1}(P_{+a,\times\fn})_1
\qquad\text{where $\fw:=\fv(L_{P_{+a,\times\fn}})$.}$$
Then  $(P_{+a},\fn,\hat a-a)$ is repulsive-normal.
\end{lemma}
\begin{proof}
By (RN2), $L_Q$ splits $\hat a/\fm$-repulsively over $K$, so $L_Q$ also splits
$(\hat a-a)/\fm$-repulsively over $K$. We have $(\hat a-a)/\fm\prec \fn/\fm \prec 1$ or
$(\hat a-a)/\fm\prec 1=\fn/\fm$, so~$L_Q$ splits $(\hat a-a)/\fn$-repulsively over $K$ by the first part of Corollary~\ref{cor:repulsive hata}, and hence~$L_{Q_{\times\fn/\fm}}=L_Q\cdot(\fn/\fm)$  splits $(\hat a-a)/\fn$-repulsively over $K$ by the second part of that
Corollary~\ref{cor:repulsive hata}. Thus $(P_{+a},\fn,{\hat a-a})$ is repulsive-normal.
\end{proof}

\noindent
The proofs of Lemmas~\ref{splitnormalrefine},~\ref{easymultsplitnormal},~\ref{splnq} give the following repulsive-normal analogues of these lemmas, using also Lemma~\ref{lemrnlem}; for Lemma~\ref{lem:5.21 repulsive-normal} below we adopt the notational conventions about $\fn^q$ ($q\in\Q^>$) stated
before Lemma~\ref{splnq}.

\begin{lemma}\label{lem:5.18 repulsive-normal}
If $(P,\fm,\hat a)$ is repulsive-normal and $(P_{+a},\fm,\hat a-a)$ is a refinement of $(P,\fm,\hat a)$, then $(P_{+a},\fm,\hat a-a)$ is also repulsive-normal.
\end{lemma}

\begin{lemma}\label{lem:5.19 repulsive-normal}
Suppose   $(P,\fm,\hat a)$ is repulsive-normal, $\hat a \prec \fn\prec\fm$, and~${[\fn/\fm]\le\big[\fv]}$. Then the refinement $(P,\fn,\hat a)$ of $(P,\fm,\hat a)$
is repulsive-normal: if $\fm$, $P$, $Q$, $\fv$ are as in \textup{(RN2)}, then \textup{(RN2)} holds with $\fn$, $Q_{\times \fn/\fm}$, $R_{\times \fn/\fm}$, $\fv(L_{P_{\times \fn}})$ in place of~$\fm$,~$Q$,~$R$,~$\fv$. 
\end{lemma}

\begin{lemma}\label{lem:5.21 repulsive-normal}
Suppose $\fm=1$, $(P,1,\hat a)$ is repulsive-normal, $\hat a \prec \fn\prec 1$, and for~$\fv:=\fv(L_P)$ 
we have $[\fn^\dagger]<[\fv]<[\fn]$;
then $(P,\fn^q,\hat a)$ is a repulsive-normal refinement
of~$(P,1,\hat a)$ for all but finitely many~$q\in \Q$ with $0<q<1$. 
\end{lemma}

\subsection*{Achieving repulsive-normality}
In this subsection we adopt the setting of the subsection {\it Achieving split-normality}\/ 
of Section~\ref{sec:split-normal holes}: {\it $H$ is $\upo$-free and $(P,\fm,\hat a)$ is a minimal hole in $K$ of order $r\geq 1$, $\fm\in H^\times$, and $\hat a\in \hat K\setminus K$, with
$\hat a = \hat b + \hat c \imag$, $\hat b, \hat c\in \hat H$.}\/ 
We let $a$ range over $K$, $b$, $c$ over $H$, and~$\fn$ over $H^\times$. We prove here the following variant of Theorem~\ref{thm:split-normal}:

{\sloppy
\begin{theorem}\label{thm:repulsive-normal}   
Suppose the constant field $C$ of $H$ is archimedean and $\deg P>1$. Then one of the following conditions is satisfied:
\begin{list}{}{\leftmargin=2em \labelwidth=2em}
\item[$\mathrm{(i)}$] $\hat b\notin H$ and some $Z$-minimal  slot $(Q,\fm,\hat b)$ in $H$ has a special refinement~${(Q_{+b},\fn,\hat b-b)}$ such that $(Q^\phi_{+b},\fn,\hat b-b)$ is eventually deep and repulsive-nor\-mal; 
\item[$\mathrm{(ii)}$] $\hat c\notin H$ and some $Z$-minimal  slot $(R,\fm,\hat c)$ in $H$ has a special  refinement~${(R_{+c},\fn,\hat c-c)}$ such that   $(R^\phi_{+c},\fn,\hat c-c)$ is eventually deep and repulsive-normal. 
\end{list}
\end{theorem}}

\noindent
To establish this theorem we need to take up the approximation arguments in the proof of Theorem~\ref{thm:split-normal} once again. While in that proof we  treated the cases~${\hat b\in H}$ and~${\hat c\in H}$ separately to obtain stronger results in those cases (Lem\-mas~\ref{lem:hat c in K}, \ref{lem:hat b in K}), 
here  we proceed  differently and first show a repulsive-normal version of  
Proposition~\ref{evsplitnormal} which also applies to those cases. {\it In the rest of
 this subsection we assume that  $C$ is archimedean.}\/

{\sloppy
\begin{prop}\label{evrepnormal}
Suppose  the hole $(P,\fm,\hat a)$ in $K$ is special and $v({\hat b-H})\subseteq v({\hat c-H})$ \textup{(}so $\hat b\notin H$\textup{)}. Let~$(Q,\fm,\hat b)$ be a $Z$-minimal deep normal slot in~$H$. Then~$(Q,\fm,\hat b)$ has a  repulsive-normal refinement.
\end{prop}}
\begin{proof}
As in the proof of Proposition~\ref{evsplitnormal} we first arrange $\fm=1$, and set
$$\Delta\ :=\ \big\{ \delta\in\Gamma:\  \abs{\delta}\in v\big(\hat a-K\big) \big\},$$
a convex subgroup of~$\Gamma$ which is cofinal in $v(\hat a-K) = v({\hat b-H})$,
so $\hat b$ is special over~$H$. 
Lemma~\ref{lem:good approx to hata} applied to $(Q,1,\hat b)$ and $\fv(L_{Q})\prec^\flat 1$ gives that 
$\Gamma^\flat$ is strictly contained in $\Delta$. To show that~$(Q,1,\hat b)$ has a repulsive-normal refinement, we follow the proof of Proposition~\ref{evsplitnormal}, skipping the initial compositional conjugation, and arranging first that $P,Q\asymp 1$. Recall from that proof that $\dot{\hat a}\in \dot{K}^{\operatorname{c}}=\dot{H}^{\operatorname{c}}[\imag]$ and~$\Re \dot{\hat a}=\dot{\hat b}\in \dot{H}^{\operatorname{c}}\setminus \dot{H}$, with $\dot{\hat b} \prec 1$, $\dot{Q}\in \dot{H}\{Y\}$, and so~$\dot{Q}_{+\dot{\hat b}}\in \dot{H}^{\operatorname{c}}\{Y\}$. Let~$A\in \dot{H}^{\operatorname{c}}[\der]$ be the linear part
of~$\dot Q_{+\dot{\hat b}}$. Recall from that proof that $1\le s:=\order Q = \order A \le 2r$ and that 
$A$ splits over $\dot{K}^{\operatorname{c}}$. Then Lemma~\ref{hkspl} gives a {\em real\/} splitting
 $(g_1,\dots,g_s)$ of~$A$ over $\dot K^{\operatorname{c}}$:
$$A\ =\ f (\der-g_1)\cdots (\der-g_s), \qquad 0\ne f\in \dot{H}^{\operatorname{c}},\  g_1,\dots, g_s\in \dot K^{\operatorname{c}}.$$
It follows easily from [ADH, 10.1.8] that the real closed $\d$-valued field $\dot{H}$ is an $H$-field, and so its completion $\dot{H}^{\operatorname{c}}$ is also a real closed $H$-field by [ADH, 10.5.9]. Recall also that $\Delta=v(\dot H^\times)$ is the value group  of $\dot H^{\operatorname{c}}$ and properly contains $\Gamma^\flat$. Thus we can apply Corollary~\ref{cor:repulsive splitting} with~$\dot{H}^{\operatorname{c}}$ in the role of $H$ to get $\fn\in\dot{\mathcal O}$ with~$0\neq \dot\fn\prec^\flat 1$ and~$m$ such that for all $n > m$,
$(h_1,\dots,h_s):=({g_1-n\dot\fn^\dagger},\dots,g_s-n\dot\fn^\dagger)$ is a $\Delta$-repulsive splitting of~$A\dot\fn^n$ over $\dot K^{\operatorname{c}}$, so $\Re h_1,\dots, \Re h_s\ne 0$. For any $n$, $A\dot\fn^n$ is the 
linear part  of  $\dot Q_{+\dot{\hat b},\times\dot\fn^n}\in \dot H^{\operatorname{c}}\{Y\}$,
and  $(h_1,\dots,h_s)$ is also a real splitting
of~$A\dot\fn^n$ over $\dot K^{\operatorname{c}}$:
$$A\dot\fn^n\ =\ \dot{\fn}^n f (\der-h_1)\cdots (\der-h_s). $$
By increasing $m$ we arrange that for all $n > m$  we have
$g_j\not\sim n\dot\fn^\dagger$  ($j=1,\dots,s$), and also $\fv(A\dot\fn^n) \preceq \fv(A)$
provided $\big[\fv(A)\big] < [\dot\fn]$;
for the latter part use Lemma~\ref{lem:nepsilon}.
Below we assume $n > m$. Then $\fv(A\dot\fn^n) \prec 1$: to see this use 
Corollary~\ref{corAfv1}, $\fv(A)\prec 1$, and $g_j \preceq h_j$ $(j=1,\dots,s$).  Note  that
$h_1,\dots, h_s\succeq 1$.
We now apply 
Co\-rol\-lary~\ref{cor:approx LP+f, real, general} to
$\dot H$, $\dot K$, $\dot Q$, $s$, $\dot\fn^n$, $\dot{\hat b}$, $\dot\fn^n f$, $h_1,\dots,h_s$
in place of $H$, $K$, $P$, $r$, $\fm$, $f$, $a$, $b_1,\dots,b_r$, respectively, and any $\gamma\in\Delta$
with $\gamma>v(\dot\fn^n), v(\Re h_1),\dots,v(\Re h_s)$.
This  gives $a,b\in\dot{\mathcal O}$  and $b_1,\dots,b_s\in\dot{\mathcal O}_K$ such that $\dot a, \dot b \ne 0$ in
$\dot H$ and such that for the linear part $\tilde{A}\in \dot{H}[\der]$ of $\dot Q_{+\dot b,\times\dot\fn^n}$ we have
$$\dot b\ - \dot{\hat b} \ \prec\ \dot\fn^n,\qquad \tilde{A}\ \sim\ A\dot{\fn}^n,\qquad \order \tilde{A}\ =\ s, \qquad
\fw\ :=\ \fv(\tilde{A})\ \sim\ \fv(A\dot\fn^n),$$
and such that for $w:=\wt(Q)$ and with $\Delta(\fw)\subseteq \Delta$: 
\begin{align*} \tilde{A}\ =\  \tilde{B} +  \tilde{E},& \quad \tilde{B}\ =\ \dot a(\der-\dot b_1)\cdots(\der-\dot b_s)\in \dot{H}[\der],\quad   \tilde{E}\in \dot{H}[\der],\\
   v(\dot b_1-  h_1),&\dots,v(\dot b_s- h_s)\ >\ \gamma,\quad   \tilde{E}\ \prec_{\Delta(\fw)}\  \fw^{w+1} \tilde{A},  
   \end{align*}
and $(\dot{b}_1,\dots, \dot{b}_s)$ is a real splitting of $\tilde{B}$ over $\dot{K}$.  This real splitting over $\dot{K}$ has a consequence that will be crucial at the end of the proof: by changing $b_1,\dots, b_s$ if necessary, without changing $\dot{b}_1,\dots, \dot{b}_s$ we arrange that $B:=a(\der-b_1)\cdots (\der-b_s)$ lies in $\dot{\mathcal{O}}[\der]\subseteq H[\der]$ and that
$(b_1,\dots, b_s)$ is a real splitting of $B$ over $K$.   (Lemma~\ref{lem:lift real splitting}.)

 Since $\Re\dot{b}_1\sim \Re h_1,\dots, \Re \dot{b}_s\sim \Re h_s$,
 the implication just before Lemma~\ref{lem:repulsive 1} gives that $(\dot{b}_1,\dots, \dot{b}_s)$ is a $\Delta$-repulsive splitting of $\tilde{B}$ 
over $\dot K$.   
Now $\hat b-b\prec\fn^n\prec 1$, so~$(Q_{+b},1,\hat b-b)$ is a refinement of the normal slot
$(Q,1,\hat b)$ in~$H$, hence $(Q_{+b},1,\hat b-b)$ is normal by Proposition~\ref{normalrefine}.
We claim  that the refinement $(Q_{+b},\fn^n,\hat b-b)$ of~$(Q_{+b},1,\hat b-b)$ is also normal.
If $[\fn] \leq \big[\fv(L_{Q_{+b}})\big]$, this claim holds by Corollary~\ref{corcorcor}.
From Lemma~\ref{lem:linear part, new} and~\ref{lem:dotfv} we obtain:
\begin{align*} \order L_{Q_{+b}}\ &=\ \order L_{Q}\ =\ \order L_{Q_{+\hat b}}\ =\ s,\\
 \fv(L_{Q_{+b}})\ \sim\ \fv(L_Q)\ &\sim\ \fv(L_{Q_{+\hat b}}),\qquad v\big(\fv(L_{Q_{+ \hat b}})\big)\ =\ v\big(\fv(A)\big),
 \end{align*} 
so $v\big(\fv(L_{Q_{+b}})\big)=v\big(\fv(A)\big)$.    Moreover,  by Lemma~\ref{lem:dotfv} and the facts about $\tilde{A}$, 
$$   v\big(\fv(L_{Q_{+ b,\times\fn^n}})\big)\ =\  v\big(\fv(\tilde{A})  \big)\ =\ v\big(\fv(A\dot\fn^n)\big)\ =\ v(\fw). $$
Suppose  $ \big[\fv(L_{Q_{+b}})\big] < [\fn]$. Then $[\fv(A)] < [\dot{\fn}]$, so $\fv(A\dot\fn^n)\preceq\fv(A)$ using $n>m$. Now the asymptotic relations among the various $\fv(\dots)$ above
give $$\fv(L_{Q_{+b,\times\fn^n}})\ \preceq\ \fv(L_{Q_{+b}}), $$ 
hence $(Q_{+b},\fn^n,\hat b-b)$  is normal by Corollary~\ref{cor:normal for small q, prepZ} applied to $H$ and the normal slot~$(Q_{+b},1,\hat b-b)$ in $H$ in the role of $K$ and $(P,1,\hat a)$, respectively. 
Put~$\fv:=\fv(L_{Q_{+b,\times\fn^n}})$, so $\fv\asymp\fw$. Note that $Q_{+b,\times \fn^n}\in \dot{\mathcal{O}}\{Y\}$, so the image of~$L_{Q_{+b,\times \fn^n}}\in \dot{\mathcal{O}}[\der]$ in 
$\dot{H}[\der]$ is $\tilde{A}$. Thus in $H[\der]$ we have:
$$L_{Q_{+b,\times\fn^n}}\ =\ B +   E \qquad \text{where $E\in \dot{\mathcal O}[\der]$, $\ E\prec_{\Delta(\fv)} \fv^{w+1} L_{Q_{+ b,\times\fn^n}}$.}$$
Now $\dot{b}_1,\dots, \dot{b}_s$ are $\Delta$-repulsive, so $b_1,\dots, b_s$ are $\Delta$-repulsive, hence $$B\ =\ a(\der-b_1)\cdots(\der-b_s)$$ splits $\Delta$-repulsively,  and thus $(\hat b-b)/\fn^n$-repulsively.
Therefore $(Q_{+b},\fn^n,\hat b-b)$ is re\-pul\-sive-normal.
\end{proof}

\noindent
Instead of assuming in the above proposition that $(P,\fm,\hat a)$ is special and $(Q,\fm,\hat b)$ is deep and normal, we can assume, as with  Corollary~\ref{cor:evsplitnormal, 1}, that~$\deg P>1$:

\begin{cor}\label{cor:evrepnormal, 1}
Suppose $\deg P>1$ and $v({\hat b}-H)\subseteq v({\hat c}-H)$. Let~$Q\in Z(H,\hat b)$ have minimal complexity.
Then the $Z$-minimal slot~$(Q,\fm,\hat b)$ in $H$ has a special  refinement~$(Q_{+b},\fn,\hat b-b)$ such that 
$(Q^\phi_{+b},\fn,{\hat b-b})$ is eventually deep and repulsive-normal. 
\end{cor}
\begin{proof}
The beginning of the subsection {\it Achieving split-normality}\/
of Section~\ref{sec:split-normal holes} and~$\deg P >1$ give that $K$ is $r$-linearly newtonian.
Lemmas~\ref{lem:quasilinear refinement} and~\ref{ufm} yield a quasilinear refinement~${(P_{+a},\fn,\hat a-a)}$ of 
our hole $(P,\fm,\hat a)$ in $K$.  
Set $b:=\Re a$. By  Lemma~\ref{lem:same width} we have
$$v\big((\hat a-a)-K\big)\ =\ v({\hat a} - K)\ =\ v\big({\hat b} -H\big)\ =\ {v\big((\hat b-b)-H\big)}.$$
Replacing~$(P,\fm,\hat a)$ and $(Q,\fm,\hat b)$ by
$(P_{+a},\fn,\hat a-a)$ and $(Q_{+b},\fn,\hat b-b)$, respectively, we arrange that~$(P,\fm,\hat a)$ is quasilinear.
Then
by Proposition~\ref{prop:hata special} and~$K$ being $r$-linearly newtonian,  $(P,\fm,\hat a)$ is special; 
hence so is~$(Q,\fm,\hat b)$. Proposition~\ref{varmainthm} gives a refinement~$(Q_{+b},\fn,\hat b-b)$ of $(Q,\fm,\hat b)$ and an active $\phi_0\in H^{>}$ such that~$(Q^{\phi_0}_{+b},\fn,{\hat b-b})$ is deep and normal. 
Refinements of $(P,\fm,\hat a)$ remain quasilinear by Corollary~\ref{cor:ref 2n}.  Since $v(\hat b-H)\subseteq v(\hat c -H)$, Lemma~\ref{lem:same width}(ii) gives 
a refinement~$(P_{+a},\fn,\hat a-a)$ of~$(P,\fm,\hat a)$ with $\Re a=b$.
By Lemma~\ref{speciallemma} the minimal hole~$(P^{\phi_0}_{+a},\fn,\hat a-a)$ in~$K^{\phi_0}$ is special.  Proposition~\ref{evrepnormal} applied to~$(P^{\phi_0}_{+a},\fn,{\hat a-a})$,  $(Q^{\phi_0}_{+b},\fn,\hat b-b)$ in place of~$(P,\fm,\hat a)$, $(Q,\fm,\hat b)$, respectively,   gives us~${b_0\in H}$, $\fn_0\in H^\times$ and a repulsive-normal
refinement $\big(Q^{\phi_0}_{+(b+b_0)},\fn_0, {\hat b - (b+b_0)}\big)$ of~$(Q^{\phi_0}_{+b},\fn,{\hat b-b})$. 
This refinement is steep and hence deep by Corollary~\ref{cor:steep refinement}, since~$(Q^{\phi_0}_{+b},\fn,{\hat b-b})$ is deep.  Thus 
by Corollary~\ref{cor:repulsive-normal comp conj},    $\big(Q_{+(b+b_0)},\fn_0, {\hat b - (b+b_0)}\big)$ is a refinement   
of~$(Q,\fm,\hat b)$ such that that
  $\big(Q^\phi_{+(b+b_0)}, \fn_0,{ \hat b - (b+b_0)}\big)$ is eventually deep and re\-pul\-sive-normal. As a refinement of $(Q, \fm, \hat b)$, it is special. 
\end{proof}

\noindent
In the same way that Corollary~\ref{cor:evsplitnormal, 1}   gave rise to Corollary~\ref{cor:evsplitnormal, 2},  
Corollary~\ref{cor:evrepnormal, 1} gives rise to the following:

\begin{cor} \label{cor:evrepnormal, 2}
If  $\deg P>1$, $v(\hat c-H)\subseteq v(\hat b-H)$, and $R\in Z(H,\hat c)$ has minimal complexity,  then the $Z$-minimal slot~$(R,\fm,\hat c)$ in $H$ has a special refinement~$(R_{+c},\fn,\hat c-c)$   such that $(R^\phi_{+c},\fn,{\hat c-c})$ is eventually deep and repulsive-normal.
\end{cor}

\medskip\noindent
{\em Completing the proof of Theorem~\ref{thm:repulsive-normal}}. 
By Lemma~\ref{lem:same width} we have  $$v(\hat b-H)\subseteq v(\hat c-H)\ \text{ or }\ v(\hat c-H)\subseteq v(\hat b-H),$$ hence
the last two corollaries yield Theorem~\ref{thm:repulsive-normal}.  \qed

\subsection*{Strengthening repulsive-normality}
In this subsection we adopt the setting of the subsection {\it Strengthening split-normality}\/ of Section~\ref{sec:split-normal holes}.
Thus~$(P,\fm,\hat a)$ is a slot in $H$ of order $r\geq 1$ and weight $w:=\wt(P)$, and $L:=L_{P_{\times\fm}}$. If $\order L=r$,
we set~$\fv:=\fv(L)$. We let $a$, $b$ range over $H$ and $\fm$, $\fn$ over~$H^\times$.

\begin{definition}\label{def:strongly repulsive-normal}
We say that 
$(P,\fm,\hat a)$ is {\bf almost strongly repulsive-normal} if $\order L=r$,\index{slot!almost strongly repulsive-normal}\index{almost strongly!repulsive-normal}\index{repulsive-normal!almost strongly}  
$\fv\prec^\flat 1$, and there are $Q, R\in H\{Y\}$ such that
\begin{list}{*}{\addtolength\itemindent{-3.5em}\addtolength\leftmargin{0.5em}}
\item[(RN2as)]  $(P_{\times\fm})_{\geq 1}=Q+R$, $Q$ is homogeneous of degree~$1$ and order~$r$,  $L_Q$ has a strong $\hat a/\fm$-repulsive splitting over $K$, and $R\prec_{\Delta(\fv)} \fv^{w+1} (P_{\times\fm})_1$. 
\end{list}
We say that $(P,\fm,\hat a)$ is {\bf strongly repulsive-normal}\index{slot!strongly repulsive-normal}\index{strongly!repulsive-normal} \index{repulsive-normal!strongly} 
 if $\order L =r$, $\fv\prec^\flat 1$, and there are $Q, R\in H\{Y\}$ such that:
\begin{enumerate}
\item[(RN2s)]  
$P_{\times\fm}=Q+R$, $Q$ is homogeneous of degree~$1$ and order~$r$, 
$L_Q$ has a strong $\hat a/\fm$-repulsive splitting over $K$, and $R\prec_{\Delta(\fv)} \fv^{w+1} (P_{\times\fm})_1$.
\end{enumerate}
\end{definition}

\noindent
If $(P,\fm,\hat a)$ is almost strongly repulsive-normal, then $(P,\fm,\hat a)$ is almost strongly split-normal; likewise without ``almost''. 
Thus we can augment our diagram from Sec\-tion~\ref{sec:split-normal holes} as follows, the implications holding for slots of order $\ge 1$
in real closed $H$-fields with small derivation and asymptotic integration: $$\xymatrix@L=6pt{	\parbox{4.5em}{strongly repulsive-normal} \ar@{=>}[r] \ar@{=>}[d]&  \ \parbox{4.5em}{almost strongly repulsive-normal} \ar@{=>}[r] \ar@{=>}[d] &  \ \parbox{4.5em}{repulsive-normal} \ar@{=>}[d] \\ 
\parbox{3.5em}{strongly split-normal} \ar@{=>}[r] \ar@{=>}[d]&  \ \parbox{3.5em}{almost strongly split-normal} \ar@{=>}[r] & \ \parbox{3em}{split-normal} \ar@{=>}[d] \\ 
\parbox{3.5em}{strictly normal} \ar@{=>}[rr] & & \ \parbox{3em}{normal}}  $$
\newline
Adapting the proof of Lemma~\ref{lem:char strong split-norm} gives:

\begin{lemma}\label{lem:char strong rep-norm}
The following are equivalent:
\begin{enumerate}
\item[\textup{(i)}] $(P,\fm,\hat a)$ is strongly repulsive-normal;
\item[\textup{(ii)}] $(P,\fm,\hat a)$ is almost strongly repulsive-normal and strictly normal;
\item[\textup{(iii)}] $(P,\fm,\hat a)$ is almost strongly repulsive-normal and $P(0)\prec_{\Delta(\fv)} \fv^{w+1} (P_1)_{\times\fm}$.
\end{enumerate}
\end{lemma}

\begin{cor}\label{cor:strongly rep splitting => strongly rep-normal} 
If $L$ has a strong $\hat a/\fm$-repulsive splitting  over $K$, then: 
\begin{align*}
\text{$(P,\fm,\hat a)$ is almost strongly repulsive-normal }&\ \Longleftrightarrow\   \text{$(P,\fm,\hat a)$ is  normal,} \\
\text{$(P,\fm,\hat a)$ is strongly repulsive-normal }&\ \Longleftrightarrow\   \text{$(P,\fm,\hat a)$ is strictly normal.}
\end{align*}
\end{cor}

\noindent
If $(P,\fm,\hat a)$ is almost strongly repulsive-normal, then  so are
$(bP,\fm,\hat a)$ for $b\neq 0$ and $(P_{\times\fn},\fm/\fn,\hat a/\fn)$, and likewise with ``strongly''
in place of ``almost strongly''.  
The proof of the next lemma is like that of Lemma~\ref{stronglysplitnormalrefine},
 using Lemmas~\ref{lem:5.18 repulsive-normal} and~\ref{lem:char strong rep-norm} in place of
 Lemmas~\ref{splitnormalrefine} and~\ref{lem:char strong split-norm}, respectively.

\begin{lemma}\label{stronglyrepnormalrefine}
Suppose $(P_{+a},\fm,\hat a-a)$ refines $(P,\fm,\hat a)$.
If $(P,\fm,\hat a)$ is almost strongly repulsive-normal, then so is $(P_{+a},\fm,\hat a-a)$.
If  $(P,\fm,\hat a)$ is   strongly repulsive-normal, $Z$-minimal, and 
$\hat a - a \prec_{\Delta(\fv)} \fv^{r+w+1}\fm$, then $(P_{+a},\fm,\hat a-a)$ is strongly repulsive-normal. 
\end{lemma}

\noindent
Here is the key to achieving almost strong repulsive-normality; its proof is similar to that of Lemma~\ref{stronglysplitnormalrefine, q}:

\begin{lemma}\label{stsprr}
Suppose that $(P,\fm,\hat a)$ is repulsive-normal and $\hat a \prec_{\Delta(\fv)} \fm$. Then for all sufficiently small $q\in\Q^>$,
any $\fn\asymp\fv^q\fm$ yields an almost strongly repulsive-normal refinement $(P,\fn,\hat a)$ of $(P,\fm,\hat a)$.
\end{lemma}
\begin{proof}
First arrange $\fm=1$. Take $Q$, $R$ as in (RN2) for $\fm=1$.  
Then Lemma~\ref{lem:achieve strong repulsive splitting}  gives $q_0\in \Q^{>}$ such that $\hat a\prec\fv^{q_0}$ and
for all $q\in\Q$ with~${0<q\leq q_0}$ and~$\fn\asymp\fv^q$,
$L_{Q_{\times\fn}}=L_Q\fn$ has a strong $\hat a/\fn$-repulsive splitting over $K$. 
Now  Lem\-ma~\ref{lem:5.19 repulsive-normal}  yields that~$(P,\fn,\hat a)$ is almost strongly repulsive-normal for such~$\fn$.
\end{proof}

\noindent
Using this lemma we now adapt the proof of Corollary~\ref{cor:deep and almost strongly split-normal} to obtain:

\begin{cor}\label{cor:deep and almost strongly repulsive-normal}
Suppose $(P,\fm,\hat a)$ is $Z$-minimal, deep, and repulsive-nor\-mal.
Then $(P,\fm,\hat a)$ has a deep and almost strongly repulsive-nor\-mal refinement.
\end{cor}
\begin{proof} Lemma~\ref{lem:good approx to hata} gives $a$ such that $\hat a - a \prec_{\Delta(\fv)} \fm$. By Corollary~\ref{cor:deep 2, cracks}, the refinement~$(P_{+a},\fm,{\hat a-a})$ of
$(P,\fm,\hat a)$ is deep  with $\fv(L_{P_{+a,\times \fm}})\asymp_{\Delta(\fv)} \fv$, and by Lem\-ma~\ref{lem:5.18 repulsive-normal} it is also
repulsive-normal.
Now apply Lemma~\ref{stsprr} to~${(P_{+a},\fm,\hat a-a)}$ in place of~$(P,\fm,\hat a)$ and again use Corollary~\ref{cor:deep 2, cracks} to preserve being deep.
\end{proof}

\noindent
Next we adapt the proof of Lemma~\ref{lem:strongly split-normal compconj}
to obtain a result about the behavior of (almost) repulsive-normality under compositional conjugation:

\begin{lemma}\label{pqr}
Suppose $\phi$ is active in $H$ with $0<\phi\prec 1$, and there exists $a$ with~$\hat a-a\prec^\flat\fm$. If $(P,\fm,\hat a)$ is almost strongly repulsive-normal, then so is the slot~$(P^\phi,\fm,\hat a)$ in $H^\phi$.  Likewise with ``strongly'' in place of ``almost strongly''.
\end{lemma}
\begin{proof}
We arrange $\fm=1$, 
assume $(P,\fm,\hat a)$ is almost strongly repulsive-normal, and take~$Q$,~$R$ as in~(RN2as).
The proof of Lemma~\ref{lem:split-normal comp conj} shows that with $\fw:=\fv(L_{P^\phi})$
we have $\fw\prec^\flat_\phi 1$ and  $(P^\phi)_{\geq 1} = Q^\phi + R^\phi$ where
$Q^\phi\in H^\phi\{Y\}$ is homogeneous of degree~$1$ and order~$r$, $L_{Q^\phi}$ splits over~$K^\phi$, and
$R^\phi \prec_{\Delta(\fw)} \fw^{w+1} (P^\phi)_1$. By Lemma~\ref{lem:hata-repulsive splitting compconj}, $L_{Q^\phi}=L_Q^\phi$ has even a strong  $\hat a$-repulsive splitting over~$K$.
Hence~$(P^\phi,\fm,\hat a)$ is almost strongly repulsive-normal.
For the rest we use Lem\-ma~\ref{lem:char strong rep-norm} and 
the fact that  if $(P,\fm,\hat a)$ is strictly normal, then so is  $(P^\phi,\fm,\hat a)$.
\end{proof}

\noindent
Lemma~\ref{lem:good approx to hata}, the remark preceding Corollary~\ref{cor:repulsive-normal comp conj},  and Lemma~\ref{pqr} yield:

\begin{cor}\label{cor:strongly repulsive-normal compconj}
Suppose $(P,\fm,\hat a)$ is 
$Z$-minimal and deep, and $\phi$ is active in~$H$ with $0<\phi\prec 1$.
If $(P, \fm, \hat a)$ is almost strongly repulsive-normal, then so
is the slot~$(P^\phi,\fm,\hat a)$ in $H^\phi$. Likewise with ``strongly'' in place of ``almost strongly''.
\end{cor}

\noindent
In the case $r=1$, ultimateness yields almost strong repulsive-normality, under suitable assumptions;
more precisely:  

\begin{lemma}\label{lem:rep-norm ultimate}
Suppose $H$ is Liouville closed and of Hardy type, and $\I(K)\subseteq K^\dagger$. 
Assume also that   $(P,\fm,\hat a)$ is normal and special, of order  $r=1$. Then~
$$\text{$(P,\fm,\hat a)$ is ultimate} \quad\Longleftrightarrow\quad\text{$L$ has a strong $\hat a/\fm$-repulsive splitting over $K$,}$$ in which case  $(P,\fm,\hat a)$ is almost strongly repulsive-normal.
\end{lemma}
\begin{proof}
By Lemma~\ref{lem:ultimate normal}, $(P,\fm,\hat a)$ is ultimate iff
$\exc^{\operatorname{u}}(L)\cap v\big((\hat a/\fm)-H\big)\leq 0$, and the latter is
equivalent to $L$ having a strong $\hat a/\fm$-repulsive splitting over $K$, by Corollary~\ref{cor:split rep ultimate}.
For the rest use  Corollary~\ref{cor:strongly rep splitting => strongly rep-normal}.
\end{proof}

\noindent
Liouville closed $H$-fields are $1$-linearly newtonian by Corollary~\ref{cor:Liouville closed => 1-lin newt}, so in view of Lemma~\ref{lem:special dents}  and Corollary~\ref{cor:normal=>quasilinear} we may replace the hypothesis ``$(P,\fm,\hat a)$ is special'' in the previous lemma
by ``$(P,\fm,\hat a)$ is $Z$-minimal or a hole in $H$''. 
This leads to repulsive-normal analogues of Lemma~\ref{5.30real} and Corollary~\ref{cor:5.30real} for $r=1$:

\begin{lemma}  \label{5.30real rep-norm} 
Assume $H$ is Liouville closed   and of Hardy type, and $\I(K)\subseteq K^\dagger$.  
Suppose $(P,\fm,\hat a)$ is $Z$-minimal and quasilinear of order $r=1$. Then there is a refinement $(P_{+a},\fn,\hat a-a)$  of $(P,\fm,\hat a)$ and an active $\phi$ in $H$
with $0<\phi\preceq 1$   such that~$(P^\phi_{+a},\fn,{\hat a-a})$  is  deep,  strictly normal, and ultimate  \textup{(}so $(P^\phi_{+a},\fn,\hat a-a)$  is strongly repulsive-normal by Lem\-mas~\ref{lem:rep-norm ultimate} and~\ref{lem:char strong rep-norm}\textup{)}.
\end{lemma}
\begin{proof}
For any active $\phi$ in $H$ with $0<\phi\preceq 1$ we may replace~$H$,~$(P,\fm,\hat a)$  by~$H^\phi$,~$(P^\phi,\fm,\hat a)$. We may
 also replace~$(P,\fm,\hat a)$ by any of its refinements. 
Since~$H$ is $1$-linearly newtonian,
Corollary~\ref{mainthm, r=1} gives a refinement~$(P_{+a},\fn,\hat a-a)$ of $(P,\fm,\hat a)$ and an active~$\phi$ in $H$ such that~$0<\phi\preceq 1$ and $(P^\phi_{+a},\fn,\hat a-a)$ is   normal.
Replacing~$H$,~$(P,\fm,\hat a)$ by~$H^\phi$,~$(P^\phi_{+a},\fn,{\hat a-a})$,  
we arrange that~$(P,\fm,\hat a)$ itself is normal. 
Then $(P,\fm,\hat a)$ has an ultimate refinement by Proposition~\ref{prop:achieve ultimate},  and 
applying Corollary~\ref{mainthm, r=1} to this refinement and using Lemma~\ref{lem:ultimate refinement}, we  obtain an ultimate refinement  $(P_{+a},\fn,\hat a-a)$ of~$(P,\fm,\hat a)$  and an active $\phi$ in $H$ with $0<\phi\preceq 1$
 such that the $Z$-minimal slot~$(P_{+a}^\phi,\fn,\hat a-a)$ in $H^\phi$ is deep, normal, and ultimate.
 Again replacing~$H$,~$(P,\fm,\hat a)$ by~$H^\phi$,~$(P^\phi_{+a},\fn,{\hat a-a})$,  
we   arrange that~$(P,\fm,\hat a)$  is deep, normal, and ultimate.
Corollary~\ref{cor:achieve strong normality, 1} yields a deep and strictly normal refinement~$(P_{+a},\fm,{\hat a-a})$ of
$(P,\fm,\hat a)$; this refinement is still ultimate by Lemma~\ref{lem:ultimate refinement}.
Hence $(P_{+a},\fm,{\hat a-a})$ is a refinement of $(P,\fm,\hat a)$ as required,  with~$\phi=1$. 
\end{proof}

\noindent
Combining Lemmas~\ref{lem:quasilinear refinement} and~\ref{5.30real rep-norm} with Corollary~\ref{cor:strongly repulsive-normal compconj} yields:

{\sloppy
\begin{cor}\label{cor:5.30real rep-norm}
Assume $H$ is Liouville closed, $\upo$-free, and of Hardy type, and~$\I(K)\subseteq K^\dagger$.  
Then  every $Z$-minimal slot in $H$ of order $r=1$ has a refinement~$(P,\fm,\hat a)$ such that $(P^\phi,\fm,\hat a)$ is eventually deep, ultimate, and strongly re\-pul\-sive-normal. 
\end{cor}}

\noindent
In the next subsection we show how minimal holes of degree $>1$ in $K$ give rise to deep, ultimate,  strongly repulsive-normal,   $Z$-minimal slots in $H$.

\subsection*{Achieving strong repulsive-normality}
Let $H$ be an $\upo$-free Liouville closed $H$-field with small derivation and constant field $C$, and $(P,\fm,\hat a)$ a minimal hole of order~$r\ge 1$ in~$K:=H[\imag]$. Other conventions are as in the subsection {\it Achieving repulsive-normality.}\/
Our goal is to prove a version of Theorem~\ref{thm:repulsive-normal} with ``repulsive-normal'' improved to 
``strongly repulsive-normal~+~ultimate'': 

{\samepage\sloppy
\begin{theorem}\label{thm:strongly repulsive-normal} 
Suppose  $C$ is archimedean, $\I(K)\subseteq K^\dagger$,  and $\deg P>1$. Then one of the following conditions is satisfied:
\begin{enumerate}
\item[$\mathrm{(i)}$] $\hat b\notin H$ and some $Z$-minimal slot $(Q,\fm,\hat b)$ in $H$ has a special refinement~${(Q_{+b},\fn,\hat b-b)}$ such that $(Q^\phi_{+b},\fn,\hat b-b)$ is eventually deep, strongly re\-pul\-sive-normal, and ultimate; 
\item[$\mathrm{(ii)}$] $\hat c\notin H$ and some $Z$-minimal slot $(R,\fm,\hat c)$ in $H$ has a special refinement~${(R_{+c},\fn,\hat c-c)}$ such that $(R^\phi_{+c},\fn,\hat c-c)$ is eventually deep, strongly re\-pul\-sive-normal, and ultimate. 
\end{enumerate}
\end{theorem}}

\noindent
The proof of this theorem rests on the following two lemmas, where the standing assumption that $H$ is Liouville closed can be dropped.

\begin{lemma}\label{lem:refine to almost strongly repulsive-normal, Q}
Suppose $\hat b\notin H$ and $(Q,\fm,\hat b)$ is  a $Z$-minimal slot in $H$ with a refinement~${(Q_{+b},\fn,\hat b-b)}$  such that $(Q^\phi_{+b},\fn,\hat b-b)$ is eventually deep and repulsive-normal.
Then $(Q,\fm,\hat b)$ has a refinement~${(Q_{+b},\fn,\hat b-b)}$  such that $(Q^\phi_{+b},\fn,\hat b-b)$ is eventually  deep  and almost strongly repulsive-normal.
\end{lemma}
\begin{proof}
We adapt the proof of Lemma~\ref{lem:refine to almost strongly split-normal, Q}.
 Let ${(Q_{+b},\fn,\hat b-b)}$ be a refinement of~$(Q,\fm,\hat b)$  and
let $\phi_0$ be active in $H$ such that $0<\phi_0\preceq 1$
and $(Q^{\phi_0}_{+b},\fn,\hat b-b)$ is   deep and repulsive-normal.
Then  Corollary~\ref{cor:deep and almost strongly repulsive-normal}
yields a refinement $$\big((Q^{\phi_0}_{+b})_{+b_0},\fn_0,(\hat b-b)-b_0\big)$$ of $(Q^{\phi_0}_{+b},\fn,\hat b-b)$ which is deep and
almost strongly repulsive-normal. Hence 
$$\big((Q_{+b})_{+b_0},\fn_0,(\hat b-b)-b_0\big)\ =\ \big( Q_{+(b+b_0)},\fn_0,\hat b - (b+b_0) \big)$$
is  a refinement of $(Q,\fm,\hat b)$, and $\big( Q^\phi_{+(b+b_0)},\fn_0,\hat b - (b+b_0) \big)$ is eventually deep and almost strongly repulsive-normal by Corollary~\ref{cor:strongly repulsive-normal compconj}. 
\end{proof}

\noindent
In the same way we obtain: 

\begin{lemma}\label{lem:refine to almost strongly repulsive-normal, R}
Suppose  $\hat c\notin H$ and $(R,\fm,\hat c)$ is  a $Z$-minimal slot in $H$ with a refinement~${(R_{+c},\fn,\hat c-c)}$  such that $(R^\phi_{+c},\fn,\hat c-c)$ is eventually deep and repulsive-normal.
Then $(R,\fm,\hat c)$  has a refinement~${(R_{+c},\fn,\hat c-c)}$  such that $(R^\phi_{+c},\fn,\hat c-c)$ is eventually deep and almost strongly repulsive-normal.
\end{lemma}

\noindent
Theorem~\ref{thm:repulsive-normal} and the two lemmas above give  Theorem~\ref{thm:repulsive-normal} with ``re\-pul\-sive-normal'' improved to ``almost strongly repulsive-normal''. 
We now upgrade this further to ``strongly repulsive-normal~+~ultimate'' (under an extra assumption).  

\medskip
\noindent
Recall from Lemma~\ref{lem:same width}  that 
$v(\hat b-H)\subseteq v(\hat c-H)$ or $v(\hat c-H)\subseteq v(\hat b-H)$.
Thus the next two lemmas finish the proof of Theorem~\ref{thm:strongly repulsive-normal}.

\begin{lemma}\label{Zdsrnu1}
Suppose   $C$ is archimedean,  $\I(K)\subseteq K^\dagger$, $\deg P>1$, and $$v(\hat b-H)\ \subseteq\ v({\hat c-H}).$$
Let~$Q\in Z(H,\hat b)$ have minimal complexity.
Then the $Z$-minimal slot~$(Q,\fm,\hat b)$ in~$H$ has a special refinement~$(Q_{+b},\fn,{\hat b-b})$ such that $(Q^\phi_{+b},\fn,{\hat b-b})$ is eventually  deep, strongly repulsive-normal, and ultimate.
\end{lemma}
\begin{proof}  Here are two ways of modifying $(Q, \fm, \hat b)$. First, let  $(Q_{+b}, \fn, \hat b -b)$ be a refinement of 
$(Q,\fm, \hat b)$. Lemma~\ref{lem:same width} gives $c\in H$ with $v(\hat a -a)=v(\hat b-b)$ with $a:=b+c\imag$, and so the minimal
hole $(P_{+a}, \fn, \hat a -a)$  in $K$ is a refinement of $(P,\fm,\hat a)$ that relates to $(Q_{+b}, \fn, \hat b -b)$
as $(P, \fm, \hat a)$ relates to $(Q, \fm, \hat b)$. So we can replace~$(P, \fm, \hat a)$ and $(Q, \fm, \hat b)$ by $(P_{+a}, \fn, \hat a -a)$
and $(Q_{+b}, \fn, \hat b -b)$, whenever convenient. Second, let $\phi$ be active in $H$ with $0<\phi\preceq 1$. 
Then we can likewise replace $H$, $K$, $(P, \fm, \hat a)$, $(Q, \fm, \hat b)$ by $H^\phi$, $K^\phi$, $(P^\phi, \fm, \hat a)$, $(Q^\phi, \fm, \hat b)$. 

In this way we first arrange as in the proof of Corollary~\ref{cor:evrepnormal, 1}  that~$(Q,\fm,\hat b)$  is special. Next, we use Proposition~\ref{varmainthm} likewise to arrange  that
$(Q, \fm, \hat b)$ is also normal.  By Propositions~\ref{prop:achieve ultimate} (where the assumption $\I(K)\subseteq K^\dagger$ comes into play) and~\ref{normalrefine}  we arrange that $(Q, \fm, \hat b)$ is
ultimate as well.  The properties ``special'' and ``ultimate''  persist under further refinements and compositional conjugations. 

Now Corollary~\ref{cor:evrepnormal, 1} and Lemma~\ref{lem:refine to almost strongly repulsive-normal, Q} 
give a refinement 
$(Q_{+b},\fn,{\hat b-b})$ of the slot~$(Q,\fm,\hat b)$ in $H$ and an active $\phi_0$ in $H$ with $0<\phi_0\preceq 1$
such that the slot~$(Q^{\phi_0}_{+b},\fn,{\hat b-b})$ in~$H^{\phi_0}$ is  deep and almost strongly repulsive-normal. 
Corollary~\ref{cor:achieve strong normality, 1} then yields a deep and strictly normal refinement
$$\big( (Q^{\phi_0}_{+b})_{+b_0},\fn, (\hat b - b)-b_0 \big)$$
of $\big( Q^{\phi_0}_{+b},\fn,{\hat b - b }\big)$.  
This refinement is still almost  strongly re\-pul\-sive-normal by
Lem\-ma~\ref{stronglyrepnormalrefine},   and therefore strongly repulsive-normal by
Lemma~\ref{lem:char strong rep-norm}. 
Co\-rol\-lary~\ref{cor:strongly repulsive-normal compconj} then gives that~$\big( Q_{+(b+b_0)},\fn,{\hat b - (b+b_0) }\big)$
is  a special refinement of our slot~$(Q,\fm,\hat b)$ such that~$\big( Q^\phi_{+(b+b_0)},\fn,{\hat b - (b+b_0) }\big)$ is eventually 
deep and strongly re\-pul\-sive-nor\-mal. 
\end{proof}

\noindent
Likewise: 

\begin{lemma}\label{adersu} 
Suppose $C$ is archimedean,  $\I(K)\subseteq K^\dagger$, $\deg P >1$, and 
$$v(\hat c-H)\ \subseteq\ v(\hat b-H).$$ Let~$R\in Z(H,\hat c)$ have minimal complexity. Then the $Z$-minimal slot~$(R,\fm,\hat c)$ in~$H$ has a special refinement~$(R_{+c},\fn,\hat c-c)$   such that $(R^\phi_{+c},\fn,{\hat c-c})$ is eventually  deep, strongly repulsive-normal, and ultimate. 
\end{lemma} 

\section{The Main Theorem}

\noindent
We prove here the Normalization Theorem from the introduction, as a corollary of Theorem~\ref{thm:strongly repulsive-normal}. It is accordingly less detailed than the latter, but more user-friendly.  It is what will get used at a key stage in \cite{ADH5}. 

\begin{cor} \label{c4543} Let $H$ be an
$\upo$-free Liouville closed $H$-field with small derivation, archimedean  ordered constant field $C$, and $1$-linearly newtonian algebraic closure~$H [\imag]$. Suppose $H$ is not newtonian. Then
for some $Z$-minimal special hole~$(Q, 1, \hat b)$ in~$H$ with $\order Q\geq 1$ and some active~$\phi > 0$ in~$H$ with $\phi \preceq 1$, the hole
$(Q^{\phi}, 1, \hat b)$ in $H^{\phi}$ is deep, strongly repulsive-normal,
and ultimate.
\end{cor}  
\begin{proof} By Proposition~\ref{prop:char 1-linearly newt}, $K:=H[\imag]$ is $1$-linearly surjective and $\I(K)\subseteq K^\dagger$. As~$H$ is not newtonian, 
neither is $K$, by \eqref{eq:14.5.6},  so Lemma~\ref{lem:no hole of order <=r} and subsequent remarks
give a minimal hole~$(P,\fm,\hat a)$ in $K$ of order $r\geqslant 1$, where~$\fm\in H^\times$. Then~$\deg P>1$ by Corollary~\ref{cor:minhole deg 1}.  
By Lemma~\ref{lem:hole in hat K} we arrange that~${\hat a\in \hat K:= \hat H[\imag]}$ where
$\hat H$ is an immediate $\upo$-free newtonian $H$-field extension of~$H$. 
%Then $\hat K$ is also newtonian by [ADH, 14.5.7]. 
Now~$\hat a = \hat b + \hat c\imag$ with $\hat b, \hat c \in \hat H$.  
By Theorem~\ref{thm:strongly repulsive-normal}  there are two cases: 

{\sloppy
\begin{enumerate}
\item  $\hat b\notin H$ and some $Z$-minimal slot $(Q,\fm,\hat b)$ in $H$ has a special refinement~${(Q_{+b},\fn,\hat b-b)}$ such that $(Q^\phi_{+b},\fn,\hat b-b)$ is eventually deep, strongly re\-pul\-sive-normal, and ultimate; 
\item  $\hat c\notin H$ and some $Z$-minimal slot $(R,\fm,\hat c)$ in $H$ has a special refinement~${(R_{+c},\fn,\hat c-c)}$ such that $(R^\phi_{+c},\fn,\hat c-c)$ is eventually deep, strongly re\-pul\-sive-normal, and ultimate. 
\end{enumerate}}

\noindent
Assume $Q$, $\fm$, $\hat b$ are as in Case (1). (Case (2) goes the same way.) Lemma~\ref{kb} gives~$1\leqslant \order Q \leqslant 2r$. 
Multiplicatively conjugating by $\fn$ and renaming $Q_{+b,\times \fn} $ and $\frac{\hat b -b}{\fn}$ as $Q$ and $\hat b$ we arrange that $(Q, 1, \hat b)$ is a $Z$-minimal special slot such that
$(Q^\phi, 1, \hat b)$ is eventually deep, strongly repulsive-normal, and ultimate (using Lemma~\ref{lem:ultmult} to preserve  {\em ultimate}\/). 
With Lemma~\ref{lem:from cracks to holes}, changing $\hat b$ if necessary, we arrange that $(Q, 1, \hat b)$ is a hole in $H$, not just a slot in $H$.  
\end{proof}

\printindex

\part*{List of Symbols}

\medskip

{\small
\begin{longtable}{m{8em} m{30em} m{4em}}

& {\bf Asymptotic differential algebra} \\[0.6em]

$a^\dagger$   &  logarithmic derivative $a^\dagger=a'/a$ of nonzero $a$ \dotfill & \pageref{p:adagger} \\[0.4em]

$\wr(a,b)$   &  the Wronskian $ab'-a'b$ of $(a,b)$  \dotfill & \pageref{p:wr(a,b)} \\[0.4em]

$\omega$   &  the function $z\mapsto -(2z'+z^2)$ \dotfill & \pageref{p:omega} \\[0.4em]

$\sigma$   &  the function $y\mapsto \omega(-y^\dagger)+y^2$ \dotfill & \pageref{p:omega} \\[0.4em]

$\mathcal O_K$, $\smallo_K$   &  valuation ring of a valued field $K$ and  its maximal ideal \dotfill & \pageref{p:valring} \\[0.4em]

$\Gamma_K$, $\res(K)$  &  value group and residue field of a valued field $K$  \dotfill & \pageref{p:valring} \\[0.4em]

$\preceq$, $\prec$, $\asymp$, $\sim$      &  asymptotic relations on  a valued field  \dotfill & \pageref{p:asymprels} \\[0.4em]

$v_{\Delta}$, $\dot v$ &  $\Delta$-coarsening  of the  valuation $v$ \dotfill & \pageref{p:coarsening} \\[0.4em]

\parbox[c]{8em}{$\preceq_\Delta$, $\prec_\Delta$, $\asymp_\Delta$, $\sim_\Delta$, $\dot\preceq$, $\dot\prec$, $\dot\asymp$, $\dot\sim$}      &  asymptotic relations on  a   valued field coarsened by $\Delta$ \dotfill & \pageref{p:coarsening} \\[1em]

$\dot K$      &  specialization of a valued field $K$ \dotfill & \pageref{p:coarsening} \\[0.4em]

$a_\rho\leadsto a$      & the sequence $(a_\rho)$ pseudoconverges to $a$  \dotfill & \pageref{p:pc} \\[0.4em]

$[\gamma]$   &  archimedean class of $\gamma$  \dotfill & \pageref{p:[gamma]} \\[0.4em]

$[\fm]$   &  archimedean class of $v\fm$  \dotfill & \pageref{p:Delta} \\[0.4em]

$\preceq^\flat$, $\prec^\flat$, $\asymp^\flat$, $\sim^\flat$       &  flattened asymptotic relations on  an $H$-asymptotic field  \dotfill & \pageref{p:flatten} \\[0.4em]

$\preceq^\flat_\phi$, $\prec^\flat_\phi$, $\asymp^\flat_\phi$, $\sim^\flat_\phi$       &  flattened asymptotic relations on  the compositional conjugate   by $\phi$  \dotfill & \pageref{p:flatten} \\[0.4em]

$\operatorname{dv}(K)$  &  $\d$-valued hull of the pre-$\d$-valued field $K$  \dotfill & \pageref{p:dv(K)} \\[0.4em]

$\I(K)$   & special definable $\mathcal O$-submodule of the asymptotic field~$K$  \dotfill &  \pageref{p:I(K)} \\[0.4em]

$\Upg(H)$, $\Upl(H)$, $\Upd(H)$   & special definable subsets of the   pre-$H$-field $H$  \dotfill &  \pageref{p:special subsets} \\[0.4em]

$H^{\trig}$, $H^{\tl}$   &  trigonometric closure of  $H$, trigonometric-Liouville closure of $H$ \dotfill &  \pageref{p:Htrig}, \pageref{p:Htl} \\[0.6em]

& {\bf Linear differential operators} \\[0.6em]

$R[\der]$   &  ring of linear differential operators over the differential ring $R$   \dotfill &  \pageref{p:Rder} \\[0.4em]

$\ker A$   & kernel of $A$ \dotfill & \pageref{p:kerA} \\[0.4em]

$A_{\ltimes u}$   &  twist of  $A$ by $u$ \dotfill & \pageref{p:twist} \\[0.4em]

$A^*$   &  adjoint of the linear differential operator  $A$ \dotfill & \pageref{p:A*} \\[0.4em]

$\Ric(A)$   &  Riccati transform of   $A$   \dotfill &  \pageref{p:Ric(A)} \\[0.4em]

$\mult_a(A)$   & multiplicity of     $A$ at $a\in K$  \dotfill & \pageref{p:multa} \\[0.4em]

$\mult_\alpha(A)$   & multiplicity of    $A$ at~${\alpha\in K/K^\dagger}$  \dotfill & \pageref{p:multalpha} \\[0.4em]

$\Sigma(A)$   &  spectrum of   $A$ \dotfill & \pageref{p:SigmaA} \\[0.4em]

$\dwm(A)$    &  dominant weighted multiplicity of $A$  \dotfill & \pageref{p:dwm A} \\[0.4em]

 $\dwt(A)$   &  dominant weight of $A$  \dotfill & \pageref{p:dwm A} \\[0.4em]

$\dwm_A(\gamma)$   &  dominant weighted multiplicity of $Ay$ where $\gamma=vy$  \dotfill & \pageref{p:dwm A} \\[0.4em]

$\dwt_A(\gamma)$   &  dominant weight of $Ay$ where $\gamma=vy$  \dotfill & \pageref{p:dwm A} \\[0.4em]

$\nwt_A(\gamma)$   &  eventual value of $\dwt_{A^\phi}(\gamma)$  \dotfill & \pageref{p:dwm A} \\[0.4em]

$v_A^{\ev}(\gamma)$  &  eventual value of $v_{A^\phi}(\gamma) - \nwt_A(\gamma)v\phi$   \dotfill & \pageref{p:vAev} \\[0.4em]

$\exc(A)$   &  set of   exceptional values of   $A$ \dotfill & \pageref{p:exc A} \\[0.4em]

$\exc^{\operatorname{e}}(A)$   &  set of eventual exceptional values of   $A$ \dotfill & \pageref{p:excev} \\[0.4em]

$\exc^{\operatorname{u}}(A)$   &  set of ultimate exceptional values of   $A$ \dotfill & \pageref{p:excu} \\[0.4em]

$\fv(A)$   &  span of  $A$ \dotfill & \pageref{p:span} \\[0.6em]

%$\mult_a(M)$   & multiplicity of the   differential module   $M$ at $a\in K$  \dotfill & \pageref{p:multaM} \\[0.4em]

%$\mult_\alpha(M)$   & multiplicity of    $M$ at $\alpha\in K/K^\dagger$  \dotfill &  \pageref{p:multalphaM} \\[0.4em]

%$M^*$   &  adjoint of the   differential module $M$ \dotfill & \pageref{p:M*} \\[0.4em]

%$\Sigma(M)$   &  spectrum of the   differential module  $M$   \dotfill & \pageref{p:SigmaM} \\[0.4em]

& {\bf Differential polynomials} \\[0.6em]

$R\{Y\}$   &  ring of differential polynomials over the differential ring $R$   \dotfill &  \pageref{p:RY} \\[0.4em]

$P_{+a}$   &  additive conjugate of $P$ by $a$  \dotfill &  \pageref{p:P+a} \\[0.4em]

$P_{\times a}$   &  multiplicative conjugate of $P$ by $a$  \dotfill &  \pageref{p:Pxa} \\[0.4em]

$P^\phi$   &  compositional conjugate of $P$ by $\phi$  \dotfill &  \pageref{p:Pphi} \\[0.4em]

$\order P$   &  order of $P$   \dotfill &  \pageref{p:order P} \\[0.4em]

$\deg P$   &  (total) degree of $P$   \dotfill &  \pageref{p:deg P} \\[0.4em]

$\val P$   &  multiplicity of $P$  (at $0$) \dotfill &  \pageref{p:val P} \\[0.4em]

$\wt P$   &  weight of $P$   \dotfill &  \pageref{p:wt P} \\[0.4em]

$\ddeg P$   &  dominant degree of $P$   \dotfill &  \pageref{p:ddeg P} \\[0.4em]

$\dval P$   &  dominant multiplicity of $P$ (at $0$)  \dotfill &  \pageref{p:ddeg P} \\[0.4em]

$\dwt P$   &  dominant weight of $P$   \dotfill &  \pageref{p:ddeg P} \\[0.4em]

$\ndeg P$   &  Newton degree of $P$   \dotfill &    \pageref{p:newton quants} \\[0.4em]

$\nval P$   &  Newton multiplicity of $P$  (at $0$)  \dotfill &  \pageref{p:newton quants} \\[0.4em]

$\nwt P$   &  Newton weight of $P$   \dotfill &  \pageref{p:newton quants} \\[0.4em]

$v^{\ev}(P)$   &  eventual value of $v(P^\phi)-\nwt(P)v\phi$   \dotfill &  \pageref{p:vevP} \\[0.4em]

$P_d$   &  homogeneous part of degree $d$ of $P$   \dotfill &  \pageref{p:Pd} \\[0.4em]

$S_P$   &  separant of the differential polynomial $P$   \dotfill &  \pageref{p:separant} \\[0.4em]

$\cc(P)$   &  complexity of  $P$   \dotfill &  \pageref{p:complexity} \\[0.4em]

$L_P$   & linear part of  $P$   \dotfill &  \pageref{p:lin part} \\[0.4em]

$\Ric(P)$   &  Riccati transform of   $P$   \dotfill &  \pageref{p:Ric} \\[0.4em]

$Z(K,\hat a)$  &   set of all $P\in K\{Y\}^{\ne}$ that vanish at $(K,\hat a)$  \dotfill &  \pageref{p:Z(K,a)} \\[0.6em]

& {\bf Universal exponential extension} \\[0.6em]

$K^\dagger$   &  group of logarithmic derivatives of $K$ \dotfill &  \pageref{p:Kdagger} \\[0.4em]

$\Lambda$   &  complement of $K^\dagger$ \dotfill &  \pageref{p:complement} \\[0.4em]

$\Univ_K$   & universal exponential extension of    $K$  \dotfill &  \pageref{p:UK} \\[0.4em]

$v_{\g}$   & gaussian extension of the valuation of $K$ to $K[G]$  \dotfill &  \pageref{p:vg} \\[0.4em]

$\preceq_{\g}$, $\prec_{\g}$, $\asymp_{\g}$   & dominance relations associated to $v_{\g}$  \dotfill &  \pageref{p:vg} \\[0.6em]

% $\Ex(H)$ & perfect hull of $H$  \dotfill &  \pageref{p:E(H)} \\[0.4em]

% $\Dx(H)$ & $\d$-perfect hull of $H$  \dotfill &  \pageref{p:D(H)} \\[0.4em]

% $\Ex^r(H)$ & $\c^r$-perfect hull of $H$  \dotfill &  \pageref{p:Er(H)} \\[0.4em]

%$\Li(H)$   &  Hardy-Liouville closure of  $H$  \dotfill &  \pageref{p:HL}  \\[0.4em]

%$\bar{\omega}(H)$   &  set of $f\in H$ such that $f/4$ does not generate oscillations  \dotfill &  \pageref{p:baromega} \\[0.4em]

%$\sim_H$, $\approx_H$ & asymptotic similarity over $H$ \dotfill &  \pageref{p:simH} \\[0.4em]

%$\sim_K$, $\approx_K$ & asymptotic similarity over $K$ \dotfill &  \pageref{p:simK} \\[0.4em]

& {\bf Slots} \\[0.6em]

$(P,\fm,\hat a)$   &  slot in $K$  \dotfill &  \pageref{p:hole}, \pageref{p:slot} \\[0.4em]

$\Delta(\fm)$   & convex subgroup of all $\gamma\in\Gamma$ with $[\gamma]<[\fm]$  \dotfill &  \pageref{p:Delta} \\[0.4em]

\end{longtable}}

\newpage
 
\part*{Errata and Comments to [ADH]}

\medskip

\noindent
The changes below apply to the edition published by Princeton University Press, and are already reflected in the versions posted on the  arXiv  and on our personal web pages (as of September 2025). We thank Allen Gehret for pointing out most of the errors left in that edition. Linguistic slips like missing commas or articles are not listed below unless they might mislead. Citations are to the bibliography of~[ADH].

\bigskip

\noindent
{\bf Acknowledgments:}

\smallskip\noindent 
The date of September 2015 on p.~xiv indicates when the manuscript 
was first submitted to Princeton University Press. The published version incorporates some changes and additions made since then. 

\medskip\noindent
{\bf Dramatis Personae:}

\smallskip\noindent
In the item for ``$\upo$-free'' under the heading ``Asymptotic Fields'', 
$f-\omega(g^{\dagger\dagger})\succeq g^\dagger$ should be $f-\omega(g^{\dagger\dagger})\succeq (g^\dagger)^2$.

\medskip\noindent
{\bf Introduction and Overview:}
\begin{enumerate}
\item In the subsection {\bf The special cuts $\upg$, $\upl$ and $\upo$} the definition of $\upo_\rho$ should have $\upl_{\rho}$ instead of $\upl_n$.
\end{enumerate}

\medskip\noindent
{\bf Chapter 1:} 
\begin{enumerate}
\item The first sentence of the subsection {\bf Irreducibility} in Section~1.1 should be:   {\em Let $X$ and $Y$ be topological spaces}.

\item In the second line of Section~1.2, ``$R$-modules'' should be ``left $R$-modules''. 

\item In the subsection {\bf Localization of modules} in Section~1.4 the formula for addition should have $s_2x_1+s_1x_2$ in the numerator.

\item In the subsection {\bf Tensor products} in Section~1.7, the $H$ in the 4th line should be a $B$, and $M\otimes N$ in the fifth line and at the end of the second display after that should be $M\otimes_R N$.

\item In the subsection {\bf Rational rank} in Section~1.7, in the line following the display: $\Q\otimes_{\Z}N$ should be $\Q\otimes_{\Z}M$.    

\item In the 4th line of the proof of Lemma~1.8.12, the second ``$:=$'' should be~``$=$''.

\item In the 6th line of the proof of Lemma~1.8.13, ``$(a,b)\!\to$'' should be ``$(a,b)\mapsto$''.

\item In Corollary 1.9.6 one should add the assumption that $L$ is separably generated over $K$, that is, $L$ is separably algebraic over an intermediate field~$K(B)$ with~$B\subseteq L$ algebraically independent over $K$. This assumption is satisfied if $\operatorname{char} K=0$. 
Corollary 1.9.7 is still correct as stated, but its proof requires for positive characteristic a variant of Corollary 1.9.6, namely: {\em $L$ is separably algebraic over $K$ iff every derivation on $L$ extending the trivial derivation on~$K$ is trivial}. (This variant with a proof, as in [249, pp.~370--371] is now included in the arXiv version.)   
Lemma~1.9.8 should be restricted to the case~$\operatorname{char} K=0$.
\end{enumerate}   

\medskip\noindent
{\bf Chapter 2:}

\begin{enumerate}
%\item In the 6th line before Lemma~2.2.23  and in the 4th line following the proof of Lemma~2.2.24, insert a comma before ``then''. 
\item The display in the statement of lemma 2.2.21 should have $v$ instead of $\nu$. 
\item In the 4th paragraph of Section~2.3, replace ``valued subgroup of $(G,S,v)$'' by ``valued subgroup of $(G',S',v')$''. 
\end{enumerate} 

{\samepage
\medskip\noindent
{\bf Chapter 3:}

\begin{enumerate}
\item In the second sentence of the proof of Proposition~3.1.21, one can omit
``with $\mathfrak q\cap A=\mathfrak q'\cap A=\mathfrak m$'' since this condition is automatically satisfied.
\item The proof of 3.2.11 can be simplified by replacing the part ``By the Taylor identity \dots'' with the sentence: ``By Proposition~3.2.1 we have~$P(a_\rho) \leadsto P(a)=0$, hence $(a_\rho)$ is of algebraic type over $K$.''

\item F.-V.~Kuhlmann pointed out that in the ``Notes and comments'' to Section~3.2 we misattribute Corollary~3.2.26 to Krull~[229].  An
early source for a  result of this kind is Theorem~11 in O. Schilling's book,%It seems that this fact only first appeared in Chapter~2, Theorem~11 of   
\begin{quote}
 {\it The Theory of Valuations,}
Mathematical Surveys, no. 4, American Mathematical Society, New York,  1950.
\end{quote}
This book refers for this theorem to I. Kaplansky's unpublished Ph.D.~thesis
\begin{quote}
 {\it Maximal Fields with Valuations,}  Harvard University, 1941.
\end{quote}

\item Replace ``theorem'' by ``proposition'' in the sentence following the statement of Proposition~3.4.22. 

\item Marcus Tressl alerted us to an error in the proof of Theorem 3.6.11: replace the condition $\bf{K}\preceq \bf{F}$ in the first sentence of the proof by $\bf{K}\subseteq \bf{F}$, so that Zorn's lemma can be applied as indicated in the next sentence.

\item Right after Lemma~3.7.6, replace 

    ``open ball of the form $\{y: v(y-f)>vf\}$ where $f\in K^\times$'' by 

``open ball of the form $\{y: v(y-f)>vg\}$ where $f,g\in K^\times$, $f\succeq g$''. 
\end{enumerate}}

\medskip\noindent
{\bf Chapter 4:}

\begin{enumerate}
\item In the first sentence of the proof of 4.1.10, omit {\it be.}

\item The last three sentences of the proof of 4.6.12 can be shortened to: 
{\it Then by Lemma~1.3.10, $a$ is algebraic over $K$, so $a$ is algebraic over $C$ by Lem\-ma~4.1.2.}
\end{enumerate}

\medskip\noindent
{\bf Chapter 5:}

\begin{enumerate}

\item In line 5 of Section 5.5, replace $K[\der]$ by $R[\der]$.

\item In the third line of the proof of Lemma 5.5.14, replace ``$F\in \GL_n(K)$'' by~``$F\in \GL_n(R)$''.

\item In Lemma 5.7.3, replace $$\text{``$\Q[\phi,\dots, \der^n(\phi)]=\Q[\phi,\dots, \derdelta^n(\phi)]$''}$$ by  
   $$\text{``$\Q[\phi,\dots, \der^n(\phi),\phi^{-1}]=\Q[\phi,\dots, \derdelta^n(\phi), \phi^{-1}]$''.}$$
\end{enumerate} 

\medskip\noindent
{\bf Chapter 6:}

\begin{enumerate}
\item In the second to last line of the proof of Lemma 6.1.9, replace $C$ by $D_0$.

\item In the second line before the first display in the proof of Theorem 6.3.2 there is a misplaced parenthesis in $K[Y,\dots, Y^{(r-1)}]$.

\item In the last line of the proof of Lemma 6.6.5, replace (ii) by (iii). 

\end{enumerate}

\medskip\noindent
{\bf Chapter 7:}

\begin{enumerate}
\item In the third line of the proof of Proposition~7.5.6, replace $E$ by $E^\times$.
\end{enumerate}

\medskip\noindent
{\bf Chapter 8:}

\begin{enumerate}
\item In the proof of Corollary 8.3.2, $(E, \Gamma, \k_E)$ should be $(E, \k_E, \Gamma)$.
\item A few lines before Corollary 8.3.3, the formula $\theta_{\operatorname{v}}(v_1,\dots,v_k,y)$ should be~$\theta_{\operatorname{v}}(v_1,\dots, v_k,z)$. 
\item In the proof of Proposition~8.4.12, third line from the bottom,  ``$\Gamma_{K_3}=\Gamma_{K_3}$'' should be 
``$\Gamma_{K_2}=\Gamma_{K_3}$''. 
\end{enumerate}

\medskip\noindent
{\bf Chapter 9:}

\begin{enumerate}
\item Two lines before Corollary~9.1.10, (3) should be (2). 

\item Replace ``Lemma'' in the last line of the proof of Lemma 9.2.17 by ``Corollary''.  

\item The correction following Lemma~3.7.6 leads to a corresponding correction in describing the  condition $z\in G_i$ when $s_i\ne 0$, in
the proof of Lemma~9.7.3. 
  
\item Verifying (AC3) in proof of~Lemma~9.8.2 
can be shortened using 
$$\max\big\{\psi^\alpha(\gamma+k\alpha):\ \gamma\in \Gamma,\ k\in \Z,\ \gamma+k\alpha\ne 0\big\}\ =\ \beta-\alpha.$$

\item In proof of Lemma~9.9.3, insert right after ``$v$-slow on the right'' the phrase\newline
``, where $v$ is the standard valuation of $\Gamma$''. 
\end{enumerate}

\medskip\noindent
{\bf Chapter 10:}

\begin{enumerate}

\item In Lemma~10.5.12, add ``If $K$ is an $H$-field, then so is $K(y)$ with that ordering, and $C_{K(y)}=C$'' and
 in its proof refer to the remarks after Lemma~10.2.3. 

\item In the last sentence of the third paragraph in the
``Notes and comments'' to Section~10.6, ``not not'' should be ``not''.
\end{enumerate}

\medskip\noindent
{\bf Chapter 11:}

\begin{enumerate}
\item In the last display before Lemma~11.1.4, the expression $\{\gamma:\ \gamma< (\Gamma^{>})'\}$ should be  replaced by $\{\gamma\in \Gamma:\ \gamma< (\Gamma^{>})'\}$.
\item In Lemma 11.2.3(ii), complete to ``$\nval P=\nval P_{+a}$'' at the end. 
\item In proof of Lemma~11.6.3, replace 
$v(s-a^{\dagger})\in (\Gamma_F^{>})'$ by 
$v(s-a^{\dagger})\in (\Gamma_F^{>})'\cup \{\infty\}$.

\item In last sentence of proof of Lemma~11.6.14, replace $\sim sf$ by $\sim -sf$.

\item In proof of Proposition~11.6.17, end of the fourth paragraph, replace 
$\lambda$ by~$\upl$.

\item In last display before 11.7.16, $f_n^\dagger$ should be $v(f_n^\dagger)$.

\item In second part of Lemma~11.8.5, omit the assumption that $K$ has asymptotic integration and replace $=$ at end of proof by $\subseteq$.

\item Omit the proof of Corollary~11.8.13; it has an 
erroneous forward reference.  

\end{enumerate}

\medskip\noindent
{\bf Chapter 12:}

\begin{enumerate}
\item In the statement of Lemma~12.6.3, the last part should be $[g]'=[g']$.
\end{enumerate}

\medskip\noindent
{\bf Chapter 13:}

\begin{enumerate}
\item In the {\it Notes and comments}\/ to~13.3, replace ``$n_0=2\dwv(P)$'' by ``$n_0=\dwv(P)+m+2$ where $m$
is such that $P{\uparrow^m}\in\T_{\exp}\{Y\}$''.
(We thank Julian Ziegler-Hunts for pointing this out.)
\end{enumerate}

\medskip\noindent
{\bf Chapter 14:}

\begin{enumerate}
\item In the line following the statement of Theorem~14.0.1, it would be better to refer to Corollary~11.7.13 than to Corollary~11.7.10. 
\item In the third line of the proof of  Lemma~14.1.8, replace $K$ by  $K^\times$.

\item In the last line of the last display preceding Proposition 14.2.18, replace $Y''$ by $Y''Y$. 

\item In Lemma 14.3.2 (iii), replace {\it~at newton position} by {\it in newton position}.

\item In the proof of Lemma 14.3.2, after
``$\nval P_{+b}=\nval P_{+a}=1$''  add ``by Lemma 11.2.3'' (referring to the addition to 11.2.3(ii) made above).
In the next to last display in that proof, $(gy)^{\j}$ is to be taken in the sense of $K^\phi$: 
$(gy)^{\j}=(gy)^{j_0}(\derdelta(gy))^{j_1}\cdots$ with $\derdelta=\phi^{-1}\der$. 

\item After Corollary~14.5.11, replace {\it~In Section 16.1}\/ by {\it In Section~16.2}.  
\end{enumerate}

\medskip\noindent
{\bf Appendix A:}

\begin{enumerate}
\item In the sixth line before the subsection ``Representing $\T$ ...'' on p.~719, replace~$[v(\ell_{n-1}]$ by $[v(\ell_{n-1})]$.
\end{enumerate} 

{\samepage
\medskip\noindent
{\bf Appendix B:}

\begin{enumerate}
\item In the example following B.5.15, replace ``Then $V\setminus W$ is infinite \dots'' by ``If~${V\neq W}$, then $V\setminus W$ is infinite \dots''
\item In Example~B.6.1(4) add the axiom $\forall x\forall y( x\leq y \vee y \leq x)$ to $\operatorname{Or}$.
\item In the remark following the definition of ``proper filter on $\Lambda$'' in B.7 omit ``either''.
\item In the displayed equivalences in the proof of B.7.7 replace $\mathcal F$ by $\mathcal U$.
%\item In Corollary~B.10.4(ii) replace ``some $\kappa>|N|$'' by ``some infinite $\kappa>|N|$'', and in the proof of that corollary
%replace ``$|N|$-saturated'' by ``$\kappa$-saturated where~$\kappa>|N|$ is infinite''.
\item In B.10.15 replace ``abelian groups'' by ``torsion-free abelian groups''.
\item In B.11.11(ii) replace ``elementary extension $\mathbf M^*$ of $\mathbf M$'' by
 ``model~$\mathbf M^*$ of~$\Sigma$ extending $\mathbf M$'',
 and replace the label $\preceq$ in the accompanying figure by $\subseteq$.  In B.11.12 replace ``elementary extension of $M$'' by ``model of $\Sigma$ extending $M$''.  (Thanks to Cezar Port for noting this.)
%\item Replace the last sentence in the remark before B.12.14 by
%``Then $\operatorname{RCF}'$ is model complete, but
%does not have~QE:  $\{a\in \R:\ a\ge 0\}$ is neither finite nor cofinite
%and hence is not definable in the field~$\R$ by a quantifier-free $\mathcal{L}_{\operatorname{R}}$-formula.''  
\item In B.12.15, replace ``singletons'' by ``singletons $\{a\}$ where $a\in K$''.
\item Add to the ``Notes and comments'' of Section~B.12 that  Corollaries~B.12.9 and B.12.11, with different proofs, are from:
\begin{quote}
A. H. Lightstone, A. Robinson, {\it On the representation of Herbrand functions in algebraically closed fields,}
J. Symb. Lo\-gic~{\bf 22} (1957), 187--204. 
\end{quote}
\end{enumerate}}

\newpage

\newlength\templinewidth
\setlength{\templinewidth}{\textwidth}
\addtolength{\templinewidth}{-2.25em}

\patchcmd{\thebibliography}{\list}{\printremarkbeforebib\list}{}{}

\let\oldaddcontentsline\addcontentsline% Store \addcontentsline
\renewcommand{\addcontentsline}[3]{\oldaddcontentsline{toc}{part}{References}}

\def\printremarkbeforebib{\bigskip\hskip1em The citation [ADH] refers to our book \\

\hskip1em\parbox{\templinewidth}{
M. Aschenbrenner, L. van den Dries, J. van der Hoeven,
\textit{Asymptotic Differential Algebra and Model Theory of Transseries,} Annals of Mathematics Studies, vol.~195, Princeton University Press, Princeton, NJ, 2017.
}

\bigskip

} 

\bibliographystyle{amsplain}

\begin{thebibliography}{999}




\bibitem{AvdD3} M. Aschenbrenner, L. van den Dries,
{\em Liouville closed $H$-fields,} J. Pure Appl. Algebra 
{\bf 197} (2005), 83--139.

\bibitem{ADH3} M. Aschenbrenner,  L. van den Dries, J. van der Hoeven,
{\em Differentially algebraic gaps,} Selecta Math. {\bf 11} (2005), 247--280.

 
\bibitem{VDF}
M. Aschenbrenner,  L. van den Dries, J. van der Hoeven, {\em Maximal immediate extensions of valued differential fields,}
Proc. London Math. Soc. {\bf 117} (2018), no.~2, 376--406.

\bibitem{ADH2} M. Aschenbrenner,  L. van den Dries, J. van der Hoeven, \emph{On numbers, germs, and transseries}, 
in:
B. Sirakov et al. (eds.): {\it Proceedings of the
International Congress of
Mathematicians, Rio de Janeiro 2018}, vol.~2, pp.~19--42, World Scientific Publishing Co., Singapore, 2018.

%\bibitem{ADHip} M. Aschenbrenner,  L. van den Dries, J. van der Hoeven,  {\em On a differential intermediate value property,} Rev. Un. Mat. Argentina {\bf 64} (2022), no. 1, 1--10.  

\bibitem{ADH4}  M. Aschenbrenner,  L. van den Dries, J. van der Hoeven,
{\em Maximal Hardy Fields,} manuscript, {\tt 	arXiv:2304.10846}, 2023.




\bibitem{ADH5} M. Aschenbrenner,  L. van den Dries, J. van der Hoeven,
{\em The theory of maximal Hardy fields,} preprint, {\tt arXiv:2408.05232}, 2024.

\bibitem{ADHfgh}
M. Aschenbrenner, L. van den Dries, J. van der Hoeven,  
 {\it Filling gaps in Hardy fields}\/,  J. Reine Angew. Math. {\bf 815} (2024), 107--172.  


\bibitem{ADH7} M. Aschenbrenner,  L. van den Dries, J. van der Hoeven,
{\em Constructing $\upo$-free Hardy fields,} J. Anal. Math. {\bf 159} (2026), 513--579.

\bibitem{ADH6} M. Aschenbrenner,  L. van den Dries, J. van der Hoeven,
{\em  Revisiting second-order linear differential equations over   Hardy fields,} preprint,  {\tt arXiv:2603.02013}, 2026.

\bibitem{ADHld}
M. Aschenbrenner,  L. van den Dries, J. van der Hoeven, \textit{Linear differential equations over Hardy fields,}  in preparation.

\bibitem{Boole}
G. Boole, \textit{On a general method in analysis,}
Philos. Trans. Roy. Soc. {\bf 134} (1844), 225--282.


\bibitem{Boshernitzan81}
M. Boshernitzan, {\it An extension of Hardy's class $L$ of ``orders of infinity'',} J. Analyse Math. {\bf 39} (1981), 235--255.

\bibitem{Bou} N. Bourbaki,  {\em Fonctions d'une Variable R\'eelle,} Chapitre V, {\em \'Etude Locale des Fonctions,} Her\-mann, Paris, 1951.

\bibitem{Ecalle81}
J. \'Ecalle, \textit{Les Fonctions R\'esurgentes, Tome I.   
Les Alg\`ebres de Fonctions R\'e\-sur\-gentes,}  Publications Math\'ematiques d'Orsay 81, vol. 5, Universit\'e de Paris-Sud, D\'e\-par\-te\-ment de Math\'ematiques, Orsay, 1981.

\bibitem{Ecalle} J. \'Ecalle, {\em Introduction aux Fonctions Analysables et Preuve Constructive de la Conjecture de Dulac,} Actualit\'es Math\'ematiques, Hermann, Paris, 1992. 

\bibitem{Gehret}
A. Gehret, \textit{A tale of two Liouville closures,}
Pacific J. Math. {\bf 290} (2017), no. 1, 41--76.

%\bibitem{Hardy16} G.~H.~Hardy, \textit{The Integration of Functions of a Single Variable,} 2nd ed.,  Cambridge Tracts in Mathematics and Mathematical Physics, vol. 2, Cambridge Univ. Press, Cambridge.

\bibitem{Hardy63}
G. H. Hardy, \textit{Divergent Series,} 3rd ed., Oxford, at the Clarendon Press, 1963.

%\bibitem{JvdH-IVP}
%J. van der Hoeven, \textit{A differential intermediate value theorem,} in: B. L. J. Braaksma et al. (eds.), \textit{Differential Equations and the Stokes Phenomenon,} pp.~147--170, World Scientific Publishing Co., Inc., River Edge, NJ, 2002.

\bibitem{JvdH}
J. van der Hoeven, {\em Transseries and Real Differential Algebra,}
Lecture Notes in Math., vol.~1888, Sprin\-ger-Verlag, New York, 2006. 

\bibitem{vdH:hfsol} J. van der Hoeven, {\em Transserial Hardy fields},
Ast\'erisque {\bf 323} (2009), 453--487.

\bibitem{KY}
A. N. Kolmogorov, A. P. Yushkevich (eds.),
\textit{Mathematics of the 19th Century.
Function Theory According to Chebyshev, Ordinary Differential Equations, Calculus of Variations, Theory of Finite Differences,} Birkh\"auser Verlag, Basel, 1998.

\bibitem{Lang}
S. Lang, {\em Algebra,} 3rd ed.,
Addison-Wesley Publishing Company, Reading, MA, 1993.

\bibitem{Miller}
C.~Miller, \textit{Basics of o-minimality and Hardy fields,}
in:  C. Miller et al. (eds.), \textit{Lecture Notes on O-minimal Structures and 
Real Analytic Geometry,} pp.~43--69,  
Fields Institute Communications, vol.~62, Springer, New York, 2012.

\bibitem{Polya}
G. P\'olya, {\it On the mean-value theorem corresponding to a given linear homogeneous differential equation,}  Trans. Amer. Math. Soc. {\bf 24} (1922), no. 4,  312--324.



\bibitem{vdPS} M. van der Put, M. Singer, {\em Galois Theory of Linear Differential Equations,} Grund\-leh\-ren Math. Wiss., vol.~328, Springer-Verlag, Berlin, 2003.

\bibitem{Nigel19}
N. Pynn-Coates, {\em Newtonian valued differential fields with arbitrary value group}, Comm. Algebra {\bf 47} (2019), no.~7, 2766--2776.

\bibitem{Nigel}
N. Pynn-Coates, {\em Differential-henselianity and maximality of asymptotic valued differential fields,} 
 Pacific J. Math. {\bf 308} (2020), no.~1, 161--205.
 

 \bibitem{Rosenlicht75} M. Rosenlicht, {\em Differential extension fields of exponential type,} Pacific J. Math. {\bf 57} (1975), 289--300.
 
 \bibitem{Rosenlicht81} M. Rosenlicht, {\em On the value group of a
differential valuation,} II, Amer. J. Math. {\bf 103} (1981), 977--996.

 \bibitem{Ros} M. Rosenlicht, {\em Hardy fields,} J. Math. Analysis and Appl. {\bf 93} (1983), 297--311.

\bibitem{Rota}
G.-C. Rota, \textit{Combinatorics, representation theory and invariant theory: the story of a m\'enage \`a trois,}
Discrete Math. {\bf 193} (1998), no. 1--3, 5--16. 
 

\end{thebibliography}

\end{document}